\documentclass[10pt, a4paper]{amsart}

\usepackage{fancybox} 
\usepackage{pifont}

\usepackage[utf8]{inputenc}
\usepackage[T1]{fontenc}
\usepackage{lmodern}
\usepackage{newunicodechar}
\usepackage{amsmath, amsthm, amsfonts,amssymb}
\usepackage{xcolor,mhsetup}
\usepackage[extra,safe]{tipa}
\usepackage{mathtools}
\usepackage{geometry}
\usepackage{mathrsfs}
\usepackage{hyperref}
\usepackage{url}
\usepackage{fancyhdr}
\theoremstyle{definition}
\newtheorem{theorem}{Theorem}[section]
\newtheorem{prop}{Proposition}[section]
\newtheorem{definition}[theorem]{Definition}
\newtheorem{example}[theorem]{Example}
\newtheorem{lem}{Lemma}[section]

\theoremstyle{remark}
\newtheorem{remark}[theorem]{Remark}
\theoremstyle{cor}
\newtheorem{cor}[theorem]{Corollary}

\numberwithin{equation}{section}

\def \HR{(\textbf{HR})}

\def \HRp{(\textbf{HR$_+$})}
\def \HST{(\textbf{HST})}

\def \HE{(\textbf{HE})}

\def \tr{\mbox{trace}}

\def \mH{\mathcal{H}}

\def\my_c{c_\infty}

\newcommand{\mynewtheorem}[2]{
  \newaliascnt{#1}{dummy}
  \newtheorem{#1}[#1]{#2}
  \aliascntresetthe{#1}
  \expandafter\def\csname #1autorefname\endcsname{#2}
}

\newcommand{\be}{\begin{equation}}
\newcommand{\ee}{\end{equation}}
\newcommand{\bde}{\begin{displaymath}}
\newcommand{\ede}{\end{displaymath}}
\newcommand{\beq}{\begin{eqnarray*}}
\newcommand{\eeq}{\end{eqnarray*}}
\newcommand{\beqa}{\begin{eqnarray}}
\newcommand{\eeqa}{\end{eqnarray}}
\newcommand{\bel }{\left\{\begin{array}{ll}}
\newcommand{\eel}{\cr \end{array} \right.}

\newcommand{\seq}[1]{{\lbrace #1 \rbrace}}

\newcommand{\dcb}{\begin{array}{lll}}
\newcommand{\dce}{\end{array}}
\newcommand{\ebe}{\begin{enumerate}\setlength{\baselineskip}{13pt}\setlength{\parskip}{0pt}}
\newcommand{\dbe}{\end{enumerate}}

\newcommand{\E}{\mathcal{E}}

\def\pp{{\cal P}}

\def\rr{{\mathbb R}}

\def\P{{\mathbb P}}

\def\E{\mathbb{E}}
\def \pp{{\mathcal{P}_2(\mathbb{R}^d)}}
\def\I{\mathsf{1}}

\def \supp{\mbox{Supp} }

\def\0{{\mathbf{0}}}

\def\d{{{\rm det}}}

\begin{document}

\setcounter{tocdepth}{1}

\title{Well-posedness for some non-linear SDEs and related PDE on the Wasserstein space}

\author{Paul-Eric Chaudru de Raynal}
  \address{Paul-\'Eric Chaudru de Raynal, Universit\'e de Savoie Mont Blanc, CNRS, LAMA, F-73000 Chamb\'ery, France}
\email{pe.deraynal@univ-smb.fr}

\author{Noufel Frikha}
\address{Noufel Frikha, Universit\'e de Paris, Laboratoire de Probabilit\'es, Statistique et Mod\'elisation (LPSM), F-75013 Paris, France}
\email{frikha@math.univ-paris-diderot.fr}




\subjclass[2000]{Primary 60H10, 93E03; Secondary 60H30, 35K40}

\date{\today}

\keywords{McKean-Vlasov SDEs; weak uniqueness; martingale problem; parametrix expansion; density estimates}

\begin{abstract} In this paper, we investigate the well-posedness of the martingale problem associated to non-linear stochastic differential equations (SDEs) in the sense of McKean-Vlasov under mild assumptions on the coefficients as well as classical solutions for a class of associated linear partial differential equations (PDEs) defined on $[0,T] \times \rr^d \times \mathcal{P}_2(\rr^d)$, for any $T>0$, $\mathcal{P}_2(\rr^d)$ being the Wasserstein space (\emph{i.e.} the space of probability measures on $\rr^d$ with a finite second-order moment). In this case, the derivative of a map along a probability measure is understood in the Lions' sense. The martingale problem is addressed by a fixed point argument on a suitable complete metric space, under some mild regularity assumptions on the coefficients that covers a large class of interaction. Also, new well-posedness results in the strong sense are obtained from the previous analysis. Under additional assumptions, we then prove the existence of the associated density and investigate its smoothness property. In particular, we establish some Gaussian type bounds for its derivatives. We eventually address the existence and uniqueness for the related linear Cauchy problem with irregular terminal condition and source term.

\end{abstract}

\maketitle

\tableofcontents

\section{Introduction}\label{introduction}
In this work, we are interested in some non-linear Stochastic Differential Equations (SDEs for short):
\begin{equation}
\label{SDE:MCKEAN}
X^{\xi}_t = \xi + \int_0^t b(s, X^{\xi}_s, [X^{\xi}_s]) ds + \int_0^t \sigma(s, X^{\xi}_s, [X^{\xi}_s]) dW_s, \quad [\xi] = \mu \in \mathcal{P}(\mathbb{R}^d),
\end{equation}

\noindent driven by a $q$-dimensional $W=(W^1,\cdots, W^q)$ Brownian motion with coefficients $b: \mathbb{R}_+ \times \mathbb{R}^d \times \mathcal{P}(\mathbb{R}^d)\rightarrow \mathbb{R}^d$ and $\sigma: \mathbb{R}_+ \times \mathbb{R}^d \times \mathcal{P}(\mathbb{R}^d) \rightarrow \mathbb{R}^d \otimes \mathbb{R}^q$. Here and throughout, we denote by $[\theta]$ the law of the random variable $\theta$. This kind of dynamics are also referred to as distribution dependent SDEs or mean-field or McKean-Vlasov SDEs as it describes the limiting behaviour of an individual particle evolving within a large system of particles interacting through its empirical measure, as the size of the population grows to infinity. More generally, the behaviour of the particle system is ruled by the so-called propagation of chaos phenomenon as originally studied by McKean \cite{mckean1967propagation} and then investigated by Sznitman \cite{Sznitman}. Roughly speaking, it says that if the initial conditions of a finite subset of the original system of particles become independent of each other, as the size of the whole system grows to infinity, then the dynamics of the particles of the finite subset synchronize and also become independent.

Since the original works of Kac \cite{kac1956} in kinetic theory and of McKean \cite{McKean:1966} in non-linear parabolic partial differential equations (PDEs for short), many authors have investigated theoretical and numerical aspects of McKean-Vlasov SDEs under various settings such as: the well-posedness of related martingale problem, the propagation of chaos phenomenom and other limit theorems, probabilistic representations to non-linear parabolic PDEs and their numerical approximation schemes. We refer to Tanaka \cite{tanaka1978probabilistic}, Gärtner \cite{gartner}, \cite{Sznitman} among others.\\

\noindent\textbf{On the well posedness of \eqref{SDE:MCKEAN}.} Well-posedness in the weak or strong sense of McKean-Vlasov SDEs have been intensively investigated under various settings by many authors during the last decades, see e.g. Funaki \cite{Funaki1984}, Oelschl\"ager \cite{oelschlager1984}, \cite{gartner}, \cite{Sznitman}, Jourdain \cite{jourdain:1997}, and more recently, Li and Min \cite{Li:min:2}, Chaudru de Raynal \cite{CHAUDRUDERAYNAL2019}, Mishura and Veretenikov \cite{mishura:veretenikov}, Lacker \cite{la:18} and Hammersley et al. \cite{HSSzpruch:18} for a short sample. 

Classical well-posedness results usually rely on the Cauchy-Lipschitz theory when both coefficients $b$ and $\sigma$ are Lipschitz continuous on $\rr^d \times \mathcal{P}_p(\rr^d)$ equipped with the product metric, the distance on $\mathcal{P}_p(\rr^d)$ being the Wasserstein distance of order $p$, see \emph{e.g.} \cite{Sznitman}. 

It actually turns out to be a challenging question to go beyond the aforementioned framework. Indeed, as it has been highlighted by Scheutzow in \cite{scheutzow_1987}, uniqueness may fail for a simple version of \eqref{SDE:MCKEAN}: when $p=q=1$, $\sigma \equiv 0$, for all $(t, x, m)$ in $\rr_+\times \rr^d\times \mathcal P(\rr^d),\, b(t, x, m ) = \int \overline b(y) dm(y)$, for some bounded and locally Lipschitz function $\overline b:\rr \to \rr$, the SDE \eqref{SDE:MCKEAN} with random initial condition have several solutions. Note that in this case the drift, seen as a function of the law, is only  Lipschitz with respect to the total variation distance. Nevertheless, still in this setting, it has been shown by Shiga and Tanaka in \cite{Shiga1985} that pathwise uniqueness holds when $\sigma \equiv 1$. In that case, one may also relax the local Lispchitz assumption of the function $\overline b$ and only assume that it is bounded and measurable. Such a result has been extended by Jourdain \cite{jourdain:1997} where uniqueness is shown to hold for more  general measurable and bounded drift $b$ satisfying only a Lipschitz assumption with respect to the total variation distance and diffusion coefficient $\sigma$ independent of the measure argument. These results have been recently revisited and extended to other non degenerate frameworks (allowing the diffusion coefficient to depend on the time and space variables) in Mishura and Veretenikov \cite{mishura:veretenikov}, Lacker \cite{la:18} and to possibly singular interaction of first order type by R\"ockner and Zhang in \cite{2018arXiv180902216R}. We importantly emphasize that in all the aforementioned works, the diffusion coefficient only depends on the time and space variables and that the Lipschitz assumption of the drift coefficient with respect to the total variation distance as well as the non-degeneracy of the noise play a crucial role.

We start our work by revisiting the problem of the unique solvability of the SDE \eqref{SDE:MCKEAN} by tackling the corresponding formulation of the martingale problem. Our main idea consists in a fixed point argument applied on a suitable complete metric space. To do so, we rely on a mild formulation of the transition density of the unique weak solution to the SDE \eqref{SDE:MCKEAN} with coefficients frozen with respect to the measure argument. This formulation may be seen as the first step of a perturbation method for Markov semigroups, known as the parametrix technique, such as exposed in Friedman \cite{friedman:64}, McKean and Singer \cite{mcke:sing:67}. We also refer to Konakov and Mammen \cite{kona:mamm:00}, for the expansion in infinite series of a transition density and Delarue and Menozzi \cite{dela:meno:10} or Frikha and Li \cite{frikha:li} for some extensions of this technique in other directions. 

Compared to the aforementioned results, our approach allows to deal with coefficients satisfying mild regularity assumption with respect to the space and measure variables. In particular, the diffusion coefficient may not be Lipschitz with respect to the Wasserstein distance which, to the best of our knowledge, appear to be new. Let us however mention the recent work \cite{CHAUDRUDERAYNAL2019} of the first author where such a framework is handled for a particular class of interaction (of scalar type) and under stronger regularity assumptions on the coefficients. Then, by adding a Lipschitz continuity assumption in space on the diffusion coefficient, we derive through usual strong uniqueness results on linear SDE the well-posedness in the strong sense of the SDE \eqref{SDE:MCKEAN}.

\vspace*{.3cm}

\noindent\textbf{Existence of a density for \eqref{SDE:MCKEAN} and associated Cauchy problem on the Wasserstein space.}

The well-posedness of the martingale problem then allows us to investigate in turn the regularity properties of the transition density associated to equation \eqref{SDE:MCKEAN} and to establish some Gaussian type estimates for its derivatives. Some partial results related to the smoothing properties of McKean-Vlasov SDEs have been obtained by Chaudru de Raynal \cite{CHAUDRUDERAYNAL2019}, Ba\~nos \cite{banos2018}, Crisan and McMurray \cite{Crisan2017}. In \cite{CHAUDRUDERAYNAL2019}, such type of bounds have been obtained in a regularized framework for McKean-Vlasov SDE (uniformly on the regularization procedure) with scalar interaction only. In \cite{banos2018}, a Bismut-Elworthy-Li formula is proved for a similar equation (with scalar type interaction) under the assumption that both the drift and the diffusion matrix are continuously differentiable with bounded Lipschitz derivatives in both variables and the diffusion matrix is uniformly elliptic. In \cite{Crisan2017}, in the uniform elliptic setting, using Malliavin calculus techniques, the authors proved several integration by parts formulae for the decoupled dynamics associated to the equation \eqref{SDE:MCKEAN} from which stem several estimates on the associated density and its derivatives when the coefficients $b, \, \sigma$ are smooth and when the initial law in \eqref{SDE:MCKEAN} is a Dirac mass.

Here, we will investigate the regularity properties of the density of both random variables $X^{\xi}_t$ and $X^{x, [\xi]}_t$ (given by the unique weak solution of the associated decoupled flow once the well-posedness for \eqref{SDE:MCKEAN} has been established) under mild assumptions on the coefficients, namely $b$ and $a=\sigma \sigma^{*}$ are assumed to be continuous, bounded and H\"older continuous in space and $a$ is uniformly elliptic. In this case, both the drift and diffusion coefficients are also assumed to have two bounded and H\"older continuous linear functional (or flat) derivatives with respect to its measure argument. We briefly present this notion of differentiation in Section \ref{wasserstein:differentiation} and refer to Carmona and Delarue \cite{carmona2018probabilistic} and Cardaliaguet \& al. \cite{cardaliaguet:delarue:lasry:lions} for more details. Within this framework, we are able to take advantage of the smoothing property of the underlying heat kernel and to bring to light the regularity properties of the density with respect to its measure argument for a coarser topology. Namely, the coefficients admit two linear functional derivatives but the density admits two derivatives in the sense of Lions (see section 2.1 for definitions), which appears to be a stronger notion of differentiation. As a consequence, we recover an ad hoc version of the theory investigated in the linear case in the monograph of Friedman \cite{friedman:64}, \cite{Friedman2011}. In particular, we establish some Gaussian type estimates for both densities and their derivatives with respect to the time, space and measure arguments.\\

Finally, the previous smoothing properties of the densities enable us to investigate classical solutions for a class of linear parabolic PDEs on the Wasserstein space, namely
\begin{align}
\begin{cases}
(\partial_t + \mathcal{L}_t) U(t, x, \mu)  = f(t, x, \mu) & \quad \mbox{ for } (t, x , \mu) \in [0,T) \times \rr^d \times \pp,  \\
U(T, x, \mu)   = h(x, \mu) & \quad \mbox{ for } (x , \mu) \in \rr^d \times \pp, 
\end{cases}
 \label{pde:wasserstein:space}
\end{align}

\noindent where the source term $f : \rr_+ \times \rr^d \times \pp \rightarrow \rr$ and the terminal condition $h: \rr^d \times \pp \rightarrow \rr$ are some given functions and the operator $\mathcal{L}_t$ is defined by
\begin{align}
\mathcal{L}_t g(x,  \mu) & = \sum_{i=1}^d b_i(t, x, \mu) \partial_{x_i}g(x, \mu) + \frac12 \sum_{i, j=1}^d a_{i, j}(t, x, \mu) \partial^2_{x_i, x_j}g( x, \mu) \nonumber \\
& \quad + \int \left\{ \sum_{i=1}^{d} b_i(t, z, \mu) [\partial_{\mu}g(x,\mu)(z)]_{i} + \frac12 \sum_{i, j=1}^d a_{i, j}(t, z, \mu) \partial_{z_i} [\partial_{\mu} g(x, \mu)(z)]_j  \right\} \mu(dz) \label{inf:generator:mckean:vlasov}
\end{align}

\noindent and acts on sufficiently smooth test functions $g: \rr^d \times \pp \rightarrow \rr$ and $a = \sigma \sigma^{*}$ is uniformly elliptic. Though the first part of the operator appearing in the right-hand side of \eqref{inf:generator:mckean:vlasov} is quite standard, the second part is new and involves the Lions' derivative of the test function with respect to the measure variable $\mu$, as introduced by P.-L. Lions in his seminal lectures at \emph{the Coll\`ege de France}, see \cite{lecture:lions:college}. We briefly present this notion of differentiation on the Wasserstein space in Section \ref{wasserstein:differentiation} together with the chain rule formula established in Chassagneux et al. \cite{chassagneux:crisan:delarue}, see also Carmona and Delarue \cite{carmona2018probabilistic}, for the flow of measures generated by the law of an It\^o process. Classical solutions for PDEs of the form \eqref{pde:wasserstein:space} have already been investigated in the literature using different methods and under various settings, e.g. Buckdhan et al. \cite{buckdahn2017} (for $f\equiv0$), \cite{chassagneux:crisan:delarue} and very recently \cite{Crisan2017} (for $f\equiv 0$). We also refer the reader to the pedagogical paper Bensoussan et al. \cite{BENSOUSSAN20172093} for a discussion of the different point of views in order to derive PDEs on the Wasserstein space and their applications. 

In the classical diffusion setting, provided the coefficients $b$ and $\sigma$ and the terminal condition $h$ are smooth enough (with bounded derivatives), it is now well-known that the solution to the related linear Kolmogorov PDE is smooth (see e. g. Krylov \cite{Krylov1999}). In \cite{buckdahn2017}, the authors proved a similar result in the case of the linear PDE \eqref{pde:wasserstein:space} (with $f\equiv0$) and Chassagneux et al. \cite{chassagneux:crisan:delarue} reached the same conclusion for a non-linear version also known as the Master equation. In this sense, the solution of the considered PDE preserves the regularity of the terminal condition. Still in the standard diffusion setting, it is known that one can weaken the regularity assumption on $h$ if one can benefit from the smoothness of the underlying transition density. Indeed in this case, $u(t, x) = \int h(y) \, p(t,T, x, y) \ dy$, $y\mapsto p(t, T, x, y)$ being the density of the (standard) SDE taken at time $T$ and starting from $x$ at time $t$. However, in order to benefit from this regularizing property, one has to assume that the associated operator $\mathcal{L}$ satisfies some non-degeneracy assumption. When the coefficients $b$, $a=\sigma \sigma^{*}$ are bounded measurable and H\"older continuous in space (uniformly in time) and if $a$ is unformly elliptic, it is known (see e.g. \cite{friedman:64}) that the linear Kolmogorov PDE admits a fundamental solution so that the unique classical solution exists when the terminal condition $h$ is not differentiable but only continuous. In the seminal paper \cite{horm:67}, H\"ormander gave a sufficient condition for a second order linear Kolmogorov PDE with smooth coefficients to be hypoelliptic. Thus, if H\"ormander's condition is satisfied then the unique classical solution exists even if the terminal condition is not smooth. Note that this condition is known to be nearly necessary since in the non-hypoelliptic regime, even in the case of smooth coefficients, there exists counterexample to the regularity preservation of the terminal condition, see e.g. Hairer and al. \cite{hairer2015}. \\

The recent paper \cite{Crisan2017} provides the first result in this direction for the PDE \eqref{pde:wasserstein:space} without source term and for non differentiable terminal condition $h$ using Malliavin calculus techniques under the assumption that the time-homogeneous coefficients $b, \, \sigma$ are smooth with respect to the space and measure variables. In particular, the function $h$ has to belong to a certain class of (possibly non-smooth) functions for which Malliavin integration by parts can be applied in order to retrieve the differentiability of the solution in the measure direction. This kind of condition appears to be natural since one cannot expect the solution of the PDE \eqref{pde:wasserstein:space} to preserve regularity in the measure variable in full generality as it is the case for the spatial argument, see Example 5.1 in \cite{Crisan2017} for more details on this loss of regularity. \\

 Under the aforementioned regularity assumptions on the coefficients $b$ and $a$ and if the data $f$ and $h$ admit a linear functional derivative satisfying some mild regularity and growth assumptions, we derive a theory on the existence and uniqueness of classical solutions for the PDE \eqref{pde:wasserstein:space} which is analogous to the one considered in Chapter 1 \cite{friedman:64} for linear parabolic PDEs. \\

 \noindent\textbf{Organization of the paper.} The paper is organized as follows. The basic notions of differentiation on the Wasserstein space with an emphasis on the chain rule and on the regularization property of a map defined on $\pp$ by a smooth flow of probability measures that will play a central role in our analysis are presented in Section \ref{diff:wasserstein:structural:class}. The general set-up together with the assumptions and the main results are described in Section \ref{assumptions:results}. The well-posedness of the martingale problem associated to the SDE \eqref{SDE:MCKEAN} is tackled in Section \ref{martingale:problem:sec}. The existence and the smoothness properties of its transition density are investigated in Section \ref{existence:regularity:transition:density}. Finally, classical solutions to the Cauchy problem related to the PDE \eqref{pde:wasserstein:space} are studied in Section \ref{solving:pde:wasserstein:space}. 


\subsection*{Notations:}
In the following we will denote by $C$ and $K$ some generic positive constants that may depend on the coefficients $b$ and $\sigma$. We reserve the notation $c$ for constants depending on $|\sigma|_\infty$ and $\lambda$ (see assumption \HE\, in Section \ref{assumptions:results}) but not on the time horizon $T$. Moreover, the value of both $C,\, K$ or $c$ may eventually change from line to line.  

We will denote by $\mathcal{P}(\rr^d)$ the space of probability measures on $\rr^d$ and by $\mathcal{P}_q(\rr^d) \subset \mathcal{P}(\rr^d)$, $q\geq1$, the space of probability measures with finite moment of order $q$. 

For a positive variance-covariance matrix $\Sigma$, the function $y\mapsto g(\Sigma , y)$ stands for the $d$-dimensional Gaussian kernel with $\Sigma$ as covariance matrix $g(\Sigma, x) = (2\pi)^{-\frac{d}{2}} (\d \, \Sigma)^{-\frac12} \exp(-\frac12 \langle \Sigma^{-1} x, x \rangle)$. We also define the first and second order Hermite polynomials: $H^{i}_1(\Sigma, x) := -(\Sigma^{-1} x)_{i}$ and $H^{i, j}_2(\Sigma, x) := (\Sigma^{-1} x)_i (\Sigma^{-1} x)_j - (\Sigma^{-1} )_{i, j}$, $1\leq i, j \leq d$ which are related to the previous Gaussian density as follows $\partial_{x_i} g(\Sigma, x) = H^{i}_1(\Sigma, x) g(\Sigma, x)$, $\partial^2_{x_i, x_j} g(\Sigma, x) = H^{i, j}_2(\Sigma, x) g(\Sigma, x)$. Also, when $\Sigma= c I_d$, for some positive constant $c$, the latter notation is simplified to $g(c, x) := (1/(2\pi c))^{d/2} \exp(-|x|^2/(2c))$. 

One of the key inequality that will be used intensively in this work is the following: for any $p, q>0$ and $x\in \rr$, $|x|^p e^{-q x^2} \leq (p/(2qe))^{p/2}$. 
As a direct consequence, we obtain the {\it space-time inequality},
\begin{gather}
\forall p, \, c>0, \quad |x|^p g(c t, x)\leq C t^{p/2} g(c' t, x) \label{space:time:inequality}
\end{gather}
\noindent which in turn gives {the \it standard Gaussian estimates} for the first and second order derivatives of Gaussian density, namely
\begin{align}
\forall c>0, \, |H^{i}_1(c t, x)| g(c t, x) \leq \frac{C}{t^{\frac12}} g(c' t,x) \quad \mbox{and} \quad |H^{i, j}_2(c t, x)| g(c t, x) \leq \frac{C}{t}g(c' t, x) \label{standard}
\end{align}

\noindent for some positive constants $C,\, c'$. Since we will employ it quite frequently, we will often omit to mention it explicitly at some places. We finally define the Mittag-Leffler function $E_{\alpha,\beta}(z) := \sum_{n\geq 0} z^{n}/\Gamma(\alpha n +\beta)$, $z\in \mathbb{R}$, $\alpha, \beta>0$.

\section{Preliminaries: Differentiation on the Wasserstein space and Smoothing properties of McKean-Vlasov equations}\label{diff:wasserstein:structural:class}

\subsection{Differentiation on the Wasserstein space}\label{wasserstein:differentiation}
In this section, we present the reader with a brief overview of the regularity notions used when working with mappings defined on $\mathcal{P}_2(\rr^d)$. We refer the reader to Lions' seminal lectures \cite{lecture:lions:college}, to Cardaliaguet's lectures notes \cite{cardaliaguet}, to the recent work Cardaliaguet et al. \cite{cardaliaguet:delarue:lasry:lions} or to Chapter 5 of Carmona and Delarue's monograph \cite{carmona2018probabilistic} for a more complete and detailed exposition. Unless otherwise specified, we equip the space $\mathcal{P}(\rr^d)$ with the topology induced by the total variation metric $d_{TV}$ defined by
$$
d_{{\rm TV}}(\mu, \nu) = \sup_{A \in \mathcal{B}(\rr^d)}\int_{A} (\mu-\nu)(dx) .
$$

 The space $\pp$ is equipped with the 2-Wasserstein metric
$$
W_2(\mu, \nu) = \inf_{\pi \in \mathcal{P}(\mu, \nu)} \left( \int_{\rr^d \times \rr^d} |x-y|^2 \, \pi(dx, dy) \right)^{\frac12}
$$

\noindent where, for given $\mu, \nu \in \pp$, $\mathcal{P}(\mu, \nu)$ denotes the set of measures on $\rr^d \times \rr^d$ with marginals $\mu$ and $\nu$.

In what follows, we will work with two different notions of differentiation of a continuous map $U$ defined on $\mathcal{P}(\rr^d)$. The first one, called the \emph{linear functional derivative} and denoted by $\delta U/\delta m$, will be intensively employed in our linearization procedure to tackle the martingale problem and to study the smoothing properties of McKean-Vlasov SDEs. The second one is the \emph{Lions' derivative}, $L$-derivative in short, and will be denoted by $\partial_\mu U$. Compared to the flat derivative, the $L$-derivative requires additional smoothness and will be our central object in order to establish the well-posedness of the PDE \eqref{pde:wasserstein:space}.\\

\noindent\textbf{Linear functional derivative.} 
\begin{definition}\label{continuous:L:derivative}
The continuous map $U: \mathcal{P}(\mathbb{R}^d) \rightarrow \mathbb{R}$ is said to have a linear functional derivative if there exists a real-valued bounded measurable function
$$
 \mathcal P(\mathbb{R}^d) \times \mathbb{R}^d  \ni (m, x) \mapsto \frac {\delta U}{\delta m}(m)(x) \in \mathbb{R},
$$
\noindent such that for all $x$ in $\mathbb{R}^d$, the map $\mathcal{P}(\mathbb{R}^d) \ni m \mapsto [\delta U/\delta m](m)(x)$ is continuous and such that for all $m$ and $m'$ in $\mathcal{P}(\mathbb{R}^d)$, it holds
\begin{equation}
\label{limit:equation:linear:functional:derivative}
\lim_{\varepsilon \downarrow 0} \frac{U((1-\varepsilon) m + \varepsilon m') - U(m)}{\varepsilon} = \int_{\mathbb{R}^d} \frac{\delta U}{\delta m}(m)(y) \, d(m'-m)(y).
\end{equation}
The map $y\mapsto [\delta U/\delta m](m)(y)$ being defined up to an additive constant, we will follow the usual normalization convention $\int_{\mathbb{R}^d} [\delta U/\delta m](m)(y) \, dm(y) = 0$. Observe from the above definition that for all $m$ and $m'$ in $\mathcal{P}(\mathbb{R}^d)$
\begin{equation}
\label{mean:value:theorem:flat:derivative}
 U(m')-U(m) = \int_0^1 \int_{\rr^d} \frac{\delta U}{\delta m}((1-\lambda) m + \lambda m')(y) \, d(m'-m)(y)\, d\lambda.
\end{equation}
Note that the boundedness assumption of the map $x\mapsto [\delta U/\delta m](m)(x)$, uniformly in $m$ guarantees the well-posedness of the integral appearing in the right-hand side of \eqref{mean:value:theorem:flat:derivative}.
\end{definition}

\begin{remark}
\label{rem:linderivandLip}
\begin{itemize}
\item[(i)] Such a notion of derivative was introduced in \cite{carmona2018probabilistic}, see also \cite{cardaliaguet:delarue:lasry:lions}, with $\mathcal{P}(\rr^d)$ being replaced by $\pp$, equipped with the 2-Wasserstein metric. In this case, the map $ [\delta U/\delta m](m)(.)$ is assumed to be continuous but allowed to be of quadratic growth, uniformly on bounded set $\mathcal{K} \subset \pp$. \\
\item[(ii)] If a map $U$ admits a flat derivative in the above sense then one may deduce an additional regularity property with respect to the total variation distance. Observe indeed that as $ [\delta U/\delta m](.)(.)$ is bounded, then from \eqref{mean:value:theorem:flat:derivative} it is readily seen that for all $m$ and $m'$ in $\mathcal{P}(\mathbb{R}^d)$
\begin{equation}\label{eq:linFuncDerivtoLipdtv}
 |U(m) - U(m')| \leq \sup_{m'' \in \mathcal{P}(\rr^d)} \|\frac{\delta U}{\delta m}(m'')(.)\|_\infty\, d_{{\rm TV}}(m, m').
\end{equation}
 Therefore, if the map $U$ admits a linear functional derivative in the sense of Definition \ref{continuous:L:derivative} then it is Lipschitz continuous with respect to the total variation metric.
 \end{itemize}
\end{remark}

With the above definition in mind, one may again investigate the smoothness of $m \mapsto [\delta U/\delta m](m)(y)$ for a fixed $y \in \mathbb{R}^d$. We will say that $U$ has two linear functional derivative and denote $[\delta^2 U/\delta m^2](m)(y)$ its second derivative taken at $(m, y)$ if $m \mapsto [\delta U/\delta m](m)(y)$ has a linear functional derivative in the sense of Definition \ref{continuous:L:derivative}. As a consequence, for all $m$ and $m'$ in $\mathcal{P}(\mathbb{R}^d)$ it holds 
$$
\frac{\delta U}{\delta m}(m')(y) -  \frac{\delta U}{\delta m}(m)(y) = \int_0^1 \int_{\rr^d} \frac{\delta^2 U}{\delta m^2}( (1-\lambda) m +\lambda m')(y, y') \, d(m' -m)(y')\ d\lambda
$$
\noindent and if $\mathcal{P}(\mathbb{R}^d)\times (\mathbb{R}^d)^2 \ni (m, y, y') \mapsto [\delta^2 U /\delta m^2](m)(y, y')$ is continuous then $[\delta^2 U /\delta m^2](m)(y , y') = [\delta^2 U /\delta m^2](m)(y' , y)$ for all $(m, y, y') \in \mathcal{P}(\mathbb{R}^d) \times (\mathbb{R}^d)^2$. Again, for more details on the above notion of derivative, we refer to \cite{cardaliaguet:delarue:lasry:lions} and \cite{carmona2018probabilistic}.

 \medskip
\noindent\textbf{The $L$-derivative.}
We now briefly present the second notion of derivatives that we will employ as originally introduced by Lions \cite{lecture:lions:college}. His strategy consists in considering the canonical lift of the real-valued function $U : \pp \ni \mu \mapsto U(\mu)$ into a function $\mathcal{U} : \mathbb{L}_2  \ni Z \mapsto \mathcal{U}(Z) = U([Z]) \in \rr$, $(\Omega, \mathcal{F}, \P)$ standing for an atomless probability space, with $\Omega$ a Polish space, $\mathcal{F}$ its Borel $\sigma$-algebra, $\mathbb{L}_2:=\mathbb{L}_2(\Omega,\mathcal{F},\mathbb{P}, \rr^d)$ standing for the space of $\rr^d$-valued random variables defined on $\Omega$ with finite second moment and $Z$ being a random variable with law $\mu$. Taking advantage of the Hilbert structure of the $\mathbb{L}_2$ space, the function $U$ is then said to be differentiable at $\mu\in \pp$ if its canonical lift $\mathcal{U}$ is Fr\'{e}chet differentiable at some point $Z$ such that $[Z]=\mu$. In that case, its gradient is denoted by $D\mathcal{U}$. Thanks to Riezs' representation theorem, we can identify $D\mathcal{U}$ as an element of $\mathbb{L}^2$. It then turns out that $D\mathcal{U}$ is a random variable which is $\sigma(Z)$-measurable and given by a function $DU(\mu)(.)$ from $\rr^d$ to $\rr^d$, which depends on the law $\mu$ of $Z$ and satisfying $DU(\mu)(.) \in \mathbb{L}^2(\rr^d, \mathcal{B}(\rr^d), \mu; \rr^d)$. Since we will work with mappings $U$ depending on several variables, we will adopt the notation $\partial_\mu U(\mu)(.)$ in order to emphasize that we are taking the derivative of the map $U$ with respect to its measure argument. Thus, inspired by \cite{carmona2018probabilistic}, the $L$-derivative (or $L$-differential) of $U$ at $\mu$ is the map $\partial_\mu U(\mu)(.): \rr^d \ni v \mapsto \partial_\mu U(\mu)(v) \in \rr^d$, satisfying $D \mathcal{U} = \partial_\mu U(\mu)(Z)$. 

It is important to note that this representation holds irrespectively of the choice of the original probability space $(\Omega, \mathcal{F}, \P)$. In what follows, we will only consider functions which are $\mathcal{C}^1$, that is, functions for which the associated canonical lift is $\mathcal{C}^1$ on $\mathbb{L}^2$. We will also restrict our consideration to the class of functions which are $\mathcal{C}^1$ and for which there exists a continuous version of the mapping $\pp \times \rr^d \ni (\mu, v) \mapsto \partial_\mu U(\mu)(v) \in \rr^d$. It then appears that this version is unique. We straightforwardly extend the above discussion to $\rr^d$-valued or $\rr^d\otimes \rr^d$-valued maps $U$ defined on $\pp$, component by component.

\begin{remark}\label{rem:Lderiv:lipwarr}
Let us point out the link between this notion of derivative and the regularity property with respect to the Wasserstein metrics of order one and two. Observe indeed that if a map $U$ is continuously $L$-differentiable and if the Fr\'echet derivative of its lift $D\mathcal U$ is bounded in $\mathbb L_2$ then for all $\mu$ and $\mu'$ in $\pp$ it holds
\begin{align*}
 |U(\mu) - U(\mu')| &= \bigg| \int_0^1 \E[\partial_\mu U([\lambda X + (1-\lambda)X'])(\lambda X + (1-\lambda) X')(X-X')] \, d\lambda \bigg| \\
&\leq  \|D\mathcal U \|_{\mathbb L_2} W_2(\mu,\mu'),
\end{align*}
\noindent thanks to Cauchy-Schwarz's inequality and where above $X$ and $X'$ denote two independent random variables in $\mathbb L_2$ with respective law $\mu$ and $\mu'$. In comparison with Remark \ref{rem:linderivandLip}, more precisely the estimate \eqref{eq:linFuncDerivtoLipdtv}, if one now assumes that the $L$-differential $\partial_\mu U$, viewed as the map $\pp\times \rr^d \ni (\mu,y) \mapsto \partial_\mu U(\mu)(y)$, is bounded in supremum norm then, from the above computations one readily sees that
 \begin{align*}
 |U(\mu) - U(\mu')| &\leq  \sup_{\mu''\in \pp }\|\partial_\mu U(\mu'')(\cdot)\|_{\infty} W_1(\mu,\mu').
\end{align*}

\end{remark}

\noindent\textbf{Linear functional and $L$-derivative, link and examples.}
As underlined in Proposition 5.48 of \cite{carmona2018probabilistic}, the following relation holds between the linear functional and the $L$-derivative. If a map $h: \pp \rightarrow \mathbb{R}$ admits a linear functional derivative $\delta h / \delta m$ (see Remark \ref{rem:linderivandLip} (i)) such that for any $\mu$ in $\pp$, the map $v \mapsto [\delta h / \delta m](\mu)(v)$ is differentiable and its derivative is jointly continuous in $v$ and $\mu$ and at most of linear growth in $v$ uniformly in $\mu$ for any $\mu$ in bounded subset $\mathcal{K}$ of $\pp$ then it holds
\begin{equation}\label{eq:relationLionsFlat}
\partial_\mu h(\mu)(\cdot) = \partial_v[\frac{\delta h}{\delta m}](\mu)(\cdot).
\end{equation}

Below are some examples of functions admitting linear functional. One can thus, under an additional regularity assumption, deduce its $L$-derivative.
\begin{example}\label{exple:linearfunc}
In the following, $h$ denotes a map from $\mathcal{P}(\mathbb{R}^d)$ to $\mathbb{R}$. We can straightforwardly consider their multidimensional version.

\begin{enumerate}
\item First order interaction. We say that $h$ satisfies a first order interaction if it is of following form: for some bounded and mesurable function $\bar{h} :  \mathbb{R}^d \to \mathbb{R}$, it holds
$$
 h(\mu) = \int_{\mathbb{R}^d} \bar h(y) \mu(dy).
$$
\item $N$ order interaction. We say that $h$ satisfies an $N$ order interaction if it is of following form: for some bounded and mesurable function $\bar h : (\mathbb{R}^d)^N \rightarrow \mathbb{R}$, it holds
$$
h(\mu) = \int_{(\mathbb{R}^d)^N} \bar h(y_1, \cdots, y_N) \, \mu(dy_1) \cdots \mu(dy_N).
$$
\item Polynomials on the Wasserstein space. We say that a function $h$ is a polynomial on the Wasserstein space if there exist some real-valued bounded and mesurable functions $\bar h_1, \, \cdots, \bar h_N$ defined on $\mathbb{R}^d$ such that
$$
h(\mu) = \prod_{i=1}^N \int \bar h_i(z) \mu(dz).
$$
\item Scalar interaction. We say that a function $h$ satisfies a scalar interaction if there exist a continuously differentiable real-valued function $\bar{h}$ defined on $\rr^N$ with bounded first order derivative as well as some real-valued bounded and mesurable functions $\bar h_1, \, \cdots, \bar h_N$ defined on $\mathbb{R}^d$ such that 
$$
h(\mu) = \bar{h}\left(\int \bar{h}_1(y) \, \mu(dy), \cdots, \int \bar{h}_N(y) \, \mu(dy)\right).
$$ 
\item Sum, product and more generally any smooth composition of $N$ order interactions, polynomials on Wasserstein space or scalar interaction.
\end{enumerate}
\end{example}

%

\noindent\textbf{Smooth maps defined in the strip $[0,T] \times \rr^d \times \pp$ and associated chain rule formula.}
In order to tackle the PDE \eqref{pde:wasserstein:space} defined in the strip $[0,T] \times \rr^d \times \pp$, we need a chain rule formula for $(U(t, Y_t, [X_t]))_{t\geq0}$, where $(X_t)_{t\geq0}$ and $(Y_t)_{t\geq 0}$ are two It\^o processes defined for sake of simplicity on the same probability space $(\Omega, \mathcal{F}, \mathbb{F}, \mathbb{P})$ assumed to be equipped with a right-continuous and complete filtration $\mathbb{F}=(\mathcal{F}_t)_{t\geq0}$. Their dynamics are given by
\begin{align}
X_t &= X_0 + \int_0^t b_s \, ds + \int_0^t \sigma_s \, dW_s, \, X_0 \in \mathbb{L}_2, \label{ito:process:X} \\
Y_t & = Y_0 + \int_0^t \eta_s \, ds + \int_0^t \gamma_s \, dW_s \label{ito:process:Y}
\end{align}

\noindent where $W=(W_t)_{t\geq0}$ is an $\mathbb{F}$-adapted $d$-dimensional Brownian, $(b_t)_{t\geq 0}$, $(\eta_t)_{t\geq0}$, $(\sigma_t)_{t\geq 0}$ and $(\gamma_t)_{t\geq 0}$ are $\mathbb{F}$-progressively measurable processes, with values in $\rr^d$, $\rr^d$, $\rr^d \otimes \rr^d$ and $\rr^{d\times q}$ respectively, satisfying the following conditions 
\begin{equation}
\label{cond:integrab:ito:process}
\forall T>0, \quad \E\Big[\int_0^T (|b_t|^2 + |\sigma_t|^4) \, dt\Big] < \infty \mbox{ and } \P\left(\int_0^T (|\eta_t| + |\gamma_t|^2) \, dt < + \infty \right) = 1.
\end{equation}

We now introduce two classes of functions we will work with throughout the paper.

\begin{definition}(The space $\mathcal{C}^{p, 2, 2}([0,T] \times \rr^d \times \pp)$, for $p=0, \,1$)\label{def:space:c122}
Let $T>0$ and $p \in \left\{0, 1\right\}$. The continuous function $U : [0,T] \times \rr^d \times \pp$ is in $\mathcal{C}^{p, 2, 2}([0,T] \times \rr^d \times \pp)$ if the following conditions hold:
\begin{itemize}
\item[(i)] For any $\mu \in \mathcal{P}_2(\rr^d)$, the mapping $[0,T] \times \rr^d \ni (t,x) \mapsto U(t, x,\mu)$ is in $\mathcal{C}^{p,2}([0,T] \times \rr^d)$ and the functions $[0,T] \times \rr^d \times \pp \ni (t, x, \mu) \mapsto \partial^{p}_t U(t, x, \mu),\, \partial_x U(t, x, \mu),\, \partial_{x}^2 U(t, x, \mu)$ are continuous.

\item[(ii)] For any $(t, x)\in [0,T] \times \rr^d$, the mapping $\mathcal{P}_2(\rr^d) \ni \mu \mapsto U(t ,x, \mu)$ is continuously $L$-differentiable and for any $\mu \in \mathcal{P}_2(\rr^d)$, we can find a version of the mapping $\rr^d \ni v \mapsto \partial_\mu U(t, x,\mu)(v)$ such that the mapping $[0,T]\times \rr^d \times \mathcal{P}_2(\rr^d) \times \rr^d \ni (t ,x, \mu, v) \mapsto \partial_\mu U(t, x, \mu)(v)$ is locally bounded and is continuous at any $(t, x, \mu, v)$ such that $v \in \supp(\mu)$.

\item[(iii)] For the version of $\partial_\mu U$ mentioned above and for any $(t, x,\mu)$ in $[0,T] \times \rr^d \times \mathcal{P}_2(\rr^d)$, the mapping $\rr^d \ni v \mapsto \partial_\mu U(t, x,\mu)(v)$ is continuously differentiable and its derivative $\partial_v [\partial_\mu U(t, x, \mu)](v) \in \rr^{d\times d}$ is jointly continuous in $(t, x, \mu, v)$ at any point $(t, x, \mu, v)$ such that $v \in \supp(\mu)$.
\end{itemize}
\end{definition}

\begin{remark}\label{space:restriction} We will also consider the space $\mathcal{C}^{1,p}([0,T] \times \pp)$ for $p=1, \, 2$, where we adequately remove the space variable in the Definition \ref{def:space:c122}. More precisely, we will say that $U \in \mathcal{C}^{1, 1}([0,T] \times \pp)$ if $U$ is continuous, $t\mapsto U(t,\mu) \in \mathcal{C}^{1}([0,T])$ for any $\mu \in \pp$, $(t,\mu) \mapsto \partial_t U(t, \mu)$ being continuous and if for any $t\in [0,T]$, $\mu \mapsto U(t,\mu)$ is continuously $L$-differentiable such that we can find a version of $v\mapsto \partial_\mu U(t, \mu)(v)$ satisfying: $(t,\mu, v) \mapsto \partial_\mu U(t, \mu)(v)$ is locally bounded and continuous at any $(t, \mu, v)$ satisfying $v\in \supp(\mu)$. 

We will say that $U\in \mathcal{C}^{1,2}([0,T] \times \pp)$ if $U \in \mathcal{C}^{1,1}([0,T] \times \pp)$ and for the version of $\partial_\mu U$ previously considered, for any $(t, \mu) \in [0,T] \times \pp$, the mapping $\rr^d \ni v \mapsto \partial_\mu U(t, \mu)(v)$ is continuously differentiable and its derivative $\partial_v[\partial_\mu U(t, \mu)](v) \in \rr^{d\times d}$ is jointly continuous in $(t, \mu, v)$ at any point $(t, \mu, v)$ such that $v\in \supp(\mu)$.\\
\end{remark}

With the above definitions, we can now provide the chain rule formula on the Wasserstein space that will play a central role in our analysis.

\begin{prop}[\cite{carmona2018probabilistic}, Proposition 5.102]\label{prop:chain:rule:joint:space:measure}
Let $X$ and $Y$ be two It\^o processes, with respective dynamics \eqref{ito:process:X} and \eqref{ito:process:Y}, satisfying \eqref{cond:integrab:ito:process}. Assume that $U\in \mathcal{C}^{1, 2, 2}([0,T]\times \rr^d \times \pp)$ in the sense of Definition \ref{def:space:c122} such that for any compact set $\mathcal{K}\subset \rr^d \times \pp$, 
\begin{equation}
\label{cond:integrab:ito:process:second:version}
\sup_{(t, x, \mu) \in [0,T] \times \mathcal{K}}\left\{\int_{\rr^d} | \partial_\mu U(t, x, \mu)(v)|^2 \, \mu(dv) +  \int_{\rr^d} |\partial_v [\partial_\mu U(t, x, \mu)](v)|^2 \, \mu(dv) \right\} < \infty.
\end{equation}


 Then, $\P$-a.s., $\forall t\in [0,T]$, one has
\begin{align}
  U(t, Y_t, [X_t]) & = U(0, Y_0, [X_0]) + \int_0^t \partial_x U(s, Y_s, [X_s]) \, . \gamma_s \, dW_s  \nonumber \\
&   + \int_0^t \left\{\partial_s U(s, Y_s, [X_s]) +  \partial_x U(s, Y_s, [X_s]) . \eta_s + \frac12 Tr(\partial^2_{x} U(s, Y_s, [X_s]) \gamma_s \gamma^{T}_s) \right\} ds \label{chain:rule:mes}\\
&   + \int_0^t \left\{ \widetilde{\E}\big[\partial_\mu U(s , Y_s, [X_s])(\widetilde{X}_s) . \widetilde{b}_s\big] + \frac12 \widetilde{\E}\big[Tr(\partial_v [\partial_\mu U(s, Y_s, [X_s])](\widetilde{X}_s)  \widetilde{a}_s)\big] \right\} ds \nonumber
\end{align}

\noindent where the It\^o process $(\widetilde{X}_t, \widetilde{b}_t, \widetilde{\sigma}_t)_{0\leq t \leq T}$ is a copy of the original process $(X_t, b_t, \sigma_t)_{0\leq t \leq T}$ defined on a copy $(\widetilde{\Omega}, \widetilde{\mathcal{F}}, \widetilde{\P})$ of the original probability space $(\Omega, \mathcal{F}, \P)$.
\end{prop}

\subsection{Smoothing properties of McKean-Vlasov semigroup}\label{subsec:structural} One of the central feature of our analysis relies on the smoothing properties of a non-degenerate McKean-Vlasov semi-group. In our current setting, this effect translates into a weakening of the topology with respect to which maps are, a priori, smooth. In particular, the composition of a flat differentiable map with a non-degenerate and smooth flow of probability measures allows to achieve a stronger form of differentiability, in the sense that such composition is now differentiable in the sense of Lions.\\

In order to foster the understanding of the key idea, let us consider a map $h:\mathcal{P}(\mathbb{R}^d) \to \mathbb{R}$ which is assumed to admit a linear functional derivative $\delta h /\delta m$. Recall importantly from \eqref{eq:linFuncDerivtoLipdtv} that this implies that the map $h$ is Lipschitz continuous with respect to the total variation distance. Consider the simplest version of \eqref{SDE:MCKEAN} (i.e. with $d=q=1$, $b \equiv 0$ and $\sigma \equiv1$ therein) that is the process $X_t^{\xi}= \xi +  W_t$, where $[\xi] = \mu \in \mathcal P(\mathbb{R}^d)$ and recall that $W$ is a Brownian motion independent of $\xi$. Observe that, the law $[X^\xi_t]$ only depends on $\xi$ through its law $\mu$.

Let us first show how the noise regularizes the map $\mu \mapsto h([X_t^{\xi}])$ in the sense that it is now Lipschitz with respect to a weaker topology and differentiable in a stronger sense (i.e. in the sense of Lions). Note first that in that setting, $\mu \mapsto h([X_t^{\xi}])$ rewrites $\mu \mapsto h(\mu \star g_t)$, $g_t$ being the Gaussian density with variance $t$ and where $\star$ stands for the usual convolution product that is for all $A$ in $\mathcal{B}(\rr^d)$, $(\mu \star g_t)(A) = \int \int_A  g_t(y-x) \, dy d\mu(x)$. From \eqref{mean:value:theorem:flat:derivative}, for all probability measures $\mu$ and $\mu'$ in $\mathcal P(\rr^d)$, it holds
\begin{align}
h([X_t^{\mu}]) - h([X_t^{\mu'}]) & = h( \mu \star g_t) - h( \mu' \star g_t)\notag\\
&= \int_0^1 \int_{(\rr^d)^2} \frac{\delta h}{\delta m}((\lambda \mu  + (1-\lambda) \mu')\star g_t)(y) \, g_t(y-x) \, d(\mu - \mu')(x) \, dy d\lambda \notag\\
&= \int_0^1 \int_{\rr^d} \frac{\delta h}{\delta m}((\lambda \mu + (1-\lambda)\mu')\star g_t)(y) \, \E[g_t(y-\xi) - g_t(y-\xi')] \, dy d\lambda \label{dvpExsmooth}
\end{align}
where $\xi$ and $\xi'$ have respective law $\mu$ and $\mu'$.

On the one hand, we readily obtain from the previous identity and the mean-value theorem as well as \eqref{standard} that for any $t>0$
\begin{align*}
|h([X_t^{\mu}]) - h([X_t^{\mu'}])| &\leq  C \sup_{m'' \in \mathcal{P}(\rr^d)} \|\frac{\delta h}{\delta m}(m'')(.)\|_\infty t^{-\frac 12}   \inf_{ \pi \in \Pi(m,m')} \int \{|x-y| \wedge 1\} d\pi(x,y) \\
& =: C  \sup_{m'' \in \mathcal{P}(\rr^d)} \|\frac{\delta h}{\delta m}(m'')(.)\|_\infty t^{-\frac 12}  d(\mu, \mu'),
\end{align*}
where $\Pi(\mu,\mu')$ denotes the set of probability measures $\pi \in \mathcal{P}(\mathbb{R}^d\times \mathbb{R}^d)$ with $\mu$ and $\mu'$ as respective marginals. Hence, starting at time $0$ with a map $h$ being Lipschitz in total variation distance, we end up at any time $t>0$ with a map which is Lipschitz w.r.t. the distance $d$ defined above, which is well seen to be less than the total variation distance and the Wasserstein distance $W_1$.

On the other hand, coming back to \eqref{dvpExsmooth} and restricting our considerations to initial conditions with law in $\mathcal P_2(\rr^d)$ (\emph{i.e.} $\xi$ and $\xi'$ are now assumed to belong in $\mathbb L^2$), one can choose $\xi'= \xi + \varepsilon Y$ for some $Y$ in $\mathbb L_2$ and $\varepsilon >0$ and we have, from the dominated convergence theorem, continuity of the integrands in the right hand side and then Fubini's theorem, that for any $t>0$, it holds
\begin{eqnarray*}
&&\lim_{\varepsilon \downarrow 0}\frac{1}{\varepsilon}\Big(h\big(\big[X_t^{[\xi +\varepsilon Y]}\big]\big) - h\big(\big[X_t^{[\xi]}\big]\big)\Big)\\
&=& \lim_{\varepsilon \downarrow 0}\int_0^1 \int_{\rr^d} \frac{\delta h}{\delta m}((\lambda [\xi+\varepsilon Y] + (1-\lambda)[\xi]))\star g_t)(y)\\
&&\quad  \times \E[ \int_0^1 (-H_1\cdot g_t)(y-\lambda' (\xi+\varepsilon Y) - (1-\lambda')(\xi))\cdot Y d \lambda'] dy d\lambda  \\
&=&\E\left[ \left(\int_{\rr^d} \frac{\delta h}{\delta m}(\mu \star g_t)(y) (-H_1\cdot g_t)(y-\xi) dy\right)\cdot Y \right]
\end{eqnarray*}
\noindent where we importantly used the boundedness of the map $(m, x)\mapsto [\delta h/\delta m](m)(x)$ together with the fact that $\lim_{\varepsilon \downarrow 0}d_{{\rm TV}}(\mu\star g_t, (\lambda [\xi+\varepsilon Y] + (1-\lambda)[\xi]))\star g_t)=0$ for any $\lambda \in [0,1]$ and any $t>0$. Thus, the map $\mu \mapsto h([X_t^{\mu}])$ is $L$-differentiable for any $t>0$.\\

Let us eventually conclude this illustration in view of the relation \eqref{eq:relationLionsFlat} between the flat and Lions derivatives. We readily have from \eqref{dvpExsmooth} and the previous computation that for any $t>0$
$$
\partial_\mu [h([X^{\mu}_t])](v) = \partial_v[\frac{\delta }{\delta m}[h([X^{\mu}_t])]](\mu)(v) = \int_{\rr^d} \frac{\delta h}{\delta m}([X^{\mu}_t])(z) (-H_1\cdot g_t)(z-v) \, dz 
$$
\noindent so that, for any $\mu \in \pp$ and any $t>0$, the map $v \mapsto [\delta/\delta m] \big(h([X^{\mu}_t])\big)(\mu)(v)$ is clearly a smooth function with bounded derivatives of any order and $v\mapsto \partial_\mu [h([X_t^\mu])](v)$ is also a smooth and bounded function. This immediately gives that $\mu \mapsto h ([X_t^\mu])$ is Lipschitz continuous with respect to the $2$-Wasserstein metric.\\

From the above simple but quite enlightening illustration, it is then naturally expected that such regularizing effect along smooth flows of probability measures holds in a more general way. Let us recast the above discussion in our framework with the following Proposition which will play a major role in our analysis.

\begin{prop}\label{structural:class} Assume that the continuous map $h: \mathcal{P}(\mathbb{R}^d) \rightarrow \mathbb{R}$ admits a linear functional derivative. Consider a map $ (t, x, \mu) \mapsto p(\mu, t, T, x, z) \in \mathcal{C}^{1,2, 2}([0,T)\times \mathbb{R}^d \times \pp)$, for some prescribed $T>0$, $z\mapsto p(\mu, t, T, x, z)$ being a density function, such that the probability measure given by $(p(\mu, t, T, ., dz) \sharp\mu)$ belongs to $\pp$, locally uniformly with respect to $(t, \mu) \in [0,T) \times \pp$, i.e. uniformly in $(t, \mu) \in \mathcal{K}$, $\mathcal{K}$ being any compact subset of $[0,T) \times \pp$. Assume additionally that the mappings $\mathbb{R}^d \ni v\mapsto  \int_{\rr^d} |\partial^{n}_v[\partial_{\mu} p(\mu, t, T, x, z)](v)| \, dz, \ \int_{\mathbb{R}^d} |\partial^{1+n}_x p(\mu, t, T, v, z)| \, dz$, for $n \in \left\{ 0, 1 \right\}$, are at most of linear growth, uniformly in $(t, \mu, x)$ in compact subsets of $[0,T) \times \pp \times \mathbb{R}^d$ and such that for any compact set $\mathcal{K}' \subset [0,T) \times \pp \times (\rr^d)^2$, and any $n \in \left\{ 0, 1\right\}$
\begin{equation}
\label{integrability:condition}
\int_{\mathbb{R}^d} \sup_{(t, \mu, x, y) \in \mathcal{K}'}  \left\{|\partial^{n}_t p(\mu, t, T, x, z)| + | \partial^{1+n}_x p(\mu, t, T, x, z)| + |\partial^{n}_{v}[\partial_\mu p(\mu, t, T, x, z)](v)|\right\} \, dz < \infty.
\end{equation}

\noindent Consider the map $\Theta: [0,T) \times \pp \rightarrow \pp$ defined by
$$
 \Theta(t, \mu)(dz) = (p(\mu, t, T, ., z) \sharp\mu)(dz) = \int_{\rr^d} p(\mu, t, T, x, z) \mu(dx) \, dz.
 $$
\noindent Then, the following statements hold:
\begin{itemize}
\item the map $ h(\Theta(., .)) \in \mathcal{C}^{1,2}([0,T) \times \pp)$,
\item Its $L$ and time derivatives satisfy for any $n \in \left\{0,1\right\}$ 
\begin{align}
\partial^{n}_v[\partial_{\mu} [h(\Theta(t, \mu))]](v) & = \partial^{n}_v \Big[\partial_{\nu} \Big[\int_{(\mathbb{R}^d)^2} \frac{\delta h}{\delta m} (\Theta(t,\mu))(z) \, p(\nu, t, T, x, z) \, dz \, \nu(dx) \Big]_{|\nu =\mu}\Big](v) \nonumber \\
& =  \int_{\rr^d} \Big[\frac{\delta h}{\delta m}(\Theta(t,\mu))(z) - \frac{\delta h}{\delta m} (\Theta(t,\mu))(v)\Big] \partial^{1+n}_x p(\mu, t, T, v, z) \, dz \label{condition1:structural:class:measure:derivative} \\
& \quad  + \int_{(\rr^d)^2}  \Big[ \frac{\delta h}{\delta m}(\Theta(t,\mu))(z)  - \frac{\delta h}{\delta m}(\Theta(t,\mu))(x) \Big] \partial^{n}_v[ \partial_\mu p(\mu, t, T, x, z)](v) \, dz \,\mu(dx),  \nonumber \\
 \partial_{t} h(\Theta(t, \mu)) & = \partial_{s} \Big[\int_{(\rr^d)^2} \frac{\delta h}{\delta m}(\Theta(t, \mu))(z) \, p(\mu, s, T, x, z) \, dz \, \mu(dx) \Big]_{|s = t} \nonumber \\
 & = \int_{(\mathbb{R}^d)^2} \Big[ \frac{\delta h}{\delta m} (\Theta(t,\mu))(z) -  \frac{\delta h}{\delta m}(\Theta(t,\mu))(x)\Big]\ \partial_t p(\mu, t, T, x, z) \, dz \, \mu(dx).\label{condition1:structural:class:time:derivative}
\end{align}
\end{itemize}
\end{prop}

\smallskip

\begin{remark}\label{remark:def:class}
\noindent $\circ$ Importantly, we note that in the above proposition we do not impose the \emph{intrinsic smoothness} (i.e. smoothness in the sense of Lions) of the map $h$ but only require the existence of a \emph{linear or flat} derivative. In this regard, the composition with the smooth flow $(t, \mu) \mapsto \Theta(t, \mu)$ of probability measures of $\pp$ allows to regularize the map $h$, the regularity being understood for a coarser topology. As already mentioned before, in what follows, the map $\Theta$ will be the one generated by the unique weak solution of the SDE \eqref{SDE:MCKEAN}, i.e. we will be interested in the smoothness of $[0,T) \times \pp \ni (t, \mu) \mapsto h([X^{t, \xi}_T])$. 

\smallskip

$\circ$ For functions $h: \mathcal{P}(\mathbb{R}^d) \rightarrow \mathbb{R}^{d}$ and $h: \mathcal{P}(\mathbb{R}^d) \rightarrow \mathbb{R}^{d\times d}$, we will straightforwardly extend the previous proposition to each component and still denote $[\delta h/\delta m] :  \mathcal{P}(\mathbb{R}^d) \times \mathbb{R}^d \rightarrow \mathbb{R}^d $ and $[\delta h/\delta m] :  \mathcal{P}(\mathbb{R}^d) \times \mathbb{R}^d  \rightarrow \mathbb{R}^{d\times d} $ the corresponding maps.

\smallskip

$\circ$ The second equalities in \eqref{condition1:structural:class:measure:derivative} and \eqref{condition1:structural:class:time:derivative} are related to \emph{cancellation} argument. Such arguments play a key role when investigating the regularity property of a map $h : \mathcal{P}(\mathbb{R}^d) \to \mathbb{R}^d$ composed with a smooth flow of probability measure satisfying the above assumptions when the linear functional derivative of $h$ is further assumed be H\"older continuous in space uniformly with respect to its measure argument. This will be a crucial tool in our analysis of the regularity of the transition density associated to a non-degenerate McKean-Vlasov diffusion process.


\end{remark}

%

\begin{proof} The proof is divided into two steps: we first prove continuity of $[0,T) \times \pp \ni (t, \mu) \mapsto h(\Theta(t, \mu))$ and then its differentiability.\\

\noindent\emph{Step 1: Continuity of the map $[0,T) \times \pp \ni (t, \mu) \mapsto h(\Theta(t, \mu))$.} Let $(t_n, \mu_n)_{n\geq 1}$ be a sequence of $[0,T) \times \pp$ satisfying $\lim_n |t_n-t| = \lim_n W_2(\mu_n, \mu) = 0$. In order to prove that $[0,T) \times \pp \ni (t, \mu) \mapsto h(\Theta(t, \mu))$ is continuous, it is sufficient to prove that $\lim_{n} d_{{\rm TV}}(\Theta(t_n, \mu_n), \Theta(t, \mu)) = 0$. Let $\bar h$ be a bounded and measurable real-valued function defined on $\rr^d$ satisfying $|h|_\infty\leq 1$. We use the following decomposition 
\begin{align*}
\langle \bar h, \Theta(t_n, \mu_n) \rangle -  \langle \bar h, \Theta(t, \mu) \rangle & = \int_{(\rr^d)^2} \bar{h}(z) p(\mu_n, t_n, T, x, z) \, dz \, \mu_n(dx) - \int_{(\rr^d)^2} \bar{h}(z) p(\mu, t, T, x, z) \, dz \, \mu(dx) \\
& = \mathcal{A}_n \bar{h} + \mathcal{B}_n \bar{h}
\end{align*}
\noindent with
\begin{align*}
\mathcal{A}_n \bar{h} & := \int_{(\mathbb{R}^d)^2} \bar{h}(z) p(\mu_n, t_n, T, x, z) \, dz (\mu_n-\mu)(dx), \\
 \mathcal{B}_n\bar{h} & := \int_{(\mathbb{R}^d)^2} \bar{h}(z) (p(\mu_n, t_n, T, x, z) - p(\mu, t, T, x, z)) \, dz \mu(dx). 
\end{align*}
Let us note that from condition \eqref{integrability:condition} and the dominated convergence theorem, one directly gets $\lim_n \sup_{|\bar h|_\infty \leq 1} |\mathcal{B}_n \bar h| \leq  \int_{(\mathbb{R}^d)^2} \lim_{n}  |p(\mu_n, t_n, T, x, z) - p(\mu, t, T, x, z)| \, dz \mu(dx) = 0$, where the supremum on the left-hand side is taken over all bounded and measurable real-valued function $h$ defined on $\mathbb{R}^d$ such that $|h|_\infty \leq 1$. Then, we decompose $\mathcal{A}_n \bar{h}$ as the sum of two terms namely 
\begin{align*}
\mathcal{A}^{1}_n \bar{h} &:= \int_{(\rr^d)^2} \bar{h}(z) p(\mu_n, t_n, T, x, z) \,dz \,\eta_R(x) (\mu_n-\mu)(dx), \\
 \mathcal{A}^2_n \bar{h} & := \int_{(\rr^d)^2} \bar{h}(z) p(\mu_n, t_n, T, x, z) \, dz\, (1-\eta_R)(x) (\mu_n-\mu)(dx)
\end{align*}
\noindent where $\eta_R$, $R>1$, is a non-negative smooth cutoff function such that $0\leq \eta_R \leq 1$, $\eta_R(x) = 1$ for $|x|\leq R$, $\eta_R(x)=0$ for $|x|\geq 2R$ and $|\nabla \eta_R|_\infty \leq C$, $C$ being a positive constant independent of $R$. Observe that the map $f^{n}_R: \rr^d \ni x\mapsto \int_{\rr^d} \bar{h}(z)   p(\mu_n, t_n, T, x, z) \eta_R(x) \, dz $ is continuously differentiable with a first order derivative uniformly bounded by
$$
|\nabla f^{n}_R|_\infty \leq C\bigg(1+ \int_{\rr^d} \sup_{(t, \mu , x) \in \mathcal{K} \times B_{2R}} |\partial_x p(\mu, t, T, x, z)| \, dz\bigg)
$$
\noindent where $\mathcal{K}$ is a compact set of $[0,T) \times \pp$ containing the sequence $(t_n, \mu_n)_{n\geq1}$ and $B_{2R}$ is the closed ball of radius $2R$ around the origin. From the Monge-Kantorovich duality principle, we thus get
$$
\sup_{|\bar{h}|_\infty \leq 1} |\mathcal{A}^{1}_n \bar{h}| \leq C\bigg(1+ \int_{\rr^d} \sup_{(t, \mu , x) \in \mathcal{K} \times B_{2R}} |\partial_x p(\mu, t, T, x, z)| \, dz\bigg) W_1(\mu_n, \mu)
$$
\noindent which clearly yields $\lim_{n} \sup_{|\bar{h}|_\infty \leq 1} |\mathcal{A}^{1}_n \bar{h}| =0$. \\
Now, from the boundedness of $\bar{h}$ and the weak convergence of $(\mu_n)_{n\geq1}$ towards $\mu$, we obtain 
$$
\lim\sup_n \sup_{|\bar h|_\infty \leq 1} |\mathcal{A}^{2}_n \bar{h}| \leq \bigg(\lim\sup_{n}\int_{|x| \geq R} \mu_n(dx) + \int_{|x| \geq R} \mu(dx)\bigg) \leq 2 \int_{|x| \geq R} \mu(dx).
$$
\noindent which in turn, by letting $R\uparrow \infty$, implies $\lim\sup_n \sup_{|\bar h|_\infty \leq 1} |\mathcal{A}^{2}_n \bar{h}|=0$.\\
Combining the previous arguments, we eventually obtain 
$$
\lim_{n} d_{{\rm TV}}( \Theta(t_n, \mu_n), \Theta(t, \mu) ) = \lim_n \sup_{|\bar{h}|_\infty\leq 1} | \langle \bar h, \Theta(t_n, \mu_n)-\Theta(t, \mu) \rangle|  = 0.
$$
This concludes the proof of the first step.

\medskip
\noindent \emph{Step 2: Continuous differentiability of the map $[0,T) \times \pp \ni (t, \mu) \mapsto h(\Theta(t, \mu))$.}
%
We start this step by proving the continuous differentiability of $[0,T) \ni t\mapsto h(\Theta(t, \mu))$ for any fixed $\mu \in \pp$. Let us set $\Theta_{\varepsilon, \lambda}(t, \mu) := (1-\lambda)\Theta(t, \mu) + \lambda \Theta(t+\varepsilon, \mu) $, for a fixed $(t, \mu) \in [0,T) \times \pp$ and $\varepsilon >0$ small enough. Then, 
\begin{align}
&h(\Theta(t+ \varepsilon, \mu))  - h(\Theta(t, \mu)) \nonumber \\
 & = \int_0^1 \int_{(\rr^d)^2} \frac{\delta h}{\delta m} (\Theta_{\varepsilon, \lambda}(t, \mu))(y) \left\{p(\mu, t+\varepsilon, T, x,  y) - p(\mu, t, T, x, y) \right\} \, dy \mu(dx) d\lambda \nonumber \\
& = \int_0^1 \int_{(\rr^d)^2} \left\{ \frac{\delta h}{\delta m} (\Theta_{\varepsilon, \lambda}(t, \mu))(y) - \frac{\delta h}{\delta m} (\Theta(t, \mu))(x) \right\}\big[p(\mu, t+\varepsilon, T, x,  y) - p(\mu, t, T, x, y) \big] \, dy \mu(dx) d\lambda \label{diff:varepsilon:time:h:flow:measure}
\end{align}

\noindent where for the last equality we used the fact that the two maps $y \mapsto p(\mu, t+\varepsilon, T, x,  y)$ and $y\mapsto p(\mu, t, T, x,  y)$ are density functions. Observe that $d_{{\rm TV}}(\Theta_{\varepsilon, \lambda}(t, \mu), \Theta(t, \mu)) \leq \int_{(\rr^d)^2} |p(\mu, t+\varepsilon, T, x, y) - p(\mu, t, T, x, y)| \, dy \mu(dx)$ so that, by \eqref{integrability:condition}, the continuity of $[0,T) \ni t \mapsto p(\mu, t, T, x, y)$ and the dominated convergence theorem, passing to the limit as $\varepsilon \downarrow 0$ clearly yields $\lim_{\varepsilon \downarrow 0} d_{TV}(\Theta_{\varepsilon, \lambda}(t, \mu), \Theta(t, \mu)) = 0$ which in turn, by continuity of $m\mapsto [\delta h/ \delta m](m)(y)$, implies $\lim_{\varepsilon \downarrow 0} [\delta h/ \delta m](\Theta_{\varepsilon, \lambda}(t, \mu))(y)  = [\delta h/\delta m](\Theta(t, \mu))(y)$.\\
Hence, dividing on both sides of \eqref{diff:varepsilon:time:h:flow:measure} by $\varepsilon$ and letting $\varepsilon$ goes to zero yields that $t\mapsto h(\Theta(t, \mu))$ is right-differentiable and also continuously differentiable since the limit is continuous on $[0,T)$. Moreover, the identity \eqref{condition1:structural:class:time:derivative} follows. The continuity of the map $[0,T) \times \pp \ni (t, \mu) \mapsto \partial_t h(\Theta(t, \mu))$ then follows from the relation \eqref{condition1:structural:class:time:derivative} and arguments similar to those employed previously in step 1 of the proof.

\smallskip

We now prove that $\mu \mapsto h(\Theta(t, \mu))$ is continuously $L$-differentiable for any fixed $t\in [0,T)$. Let us introduce for convenience the probability measures $\Theta_{\varepsilon, \lambda}(t, \mu) := (1-\lambda)\Theta(t, \mu) + \lambda \Theta(t, (1-\varepsilon) \mu + \varepsilon \mu') $ and $\mu_{ \varepsilon, \lambda'} := (1-\lambda') \mu + \lambda' [(1-\varepsilon)\mu+\varepsilon \mu']$, for a fixed $(t, \mu, \mu') \in [0,T) \times (\pp)^2$. Then, 
\begin{align*}
& h(\Theta(t, (1-\varepsilon)\mu+\varepsilon \mu'))  - h(\Theta(t, \mu)) \\
 \quad & = \int_0^1 \int_{\rr^d} \frac{\delta h}{\delta m} (\Theta_{\varepsilon, \lambda}(t, \mu))(y) \left\{p((1-\varepsilon)\mu+\varepsilon \mu', t, T,  y) - p(\mu, t, T, y) \right\} \, dy \, d\lambda  \\
& = \int_0^1 \int_{(\rr^d)^2}  \frac{\delta h}{\delta m} (\Theta_{\varepsilon, \lambda}(t, \mu))(y)   \big[p((1-\varepsilon) \mu + \varepsilon \mu', t, T, x,  y)  [(1-\varepsilon) \mu + \varepsilon \mu'](dx) - p(\mu, t, T, x, y) \mu(dx) \big] \, dy \, d\lambda  \\
& =  \varepsilon \int_0^1 \int_{(\rr^d)^2}  \frac{\delta h}{\delta m} (\Theta_{\varepsilon, \lambda}(t, \mu))(y) p((1-\varepsilon) \mu + \varepsilon \mu', t, T, x,  y) \, dy \,  (\mu'-\mu)(dx) \, d\lambda\\
&  +\varepsilon \int_{[0,1]^2} \int_{(\rr^d)^ 3} \left\{ \frac{\delta h }{\delta m} (\Theta_{\varepsilon, \lambda}(t, \mu))(y) - \frac{\delta h}{\delta m}(\Theta(t, \mu))(x)\right\} \frac{\delta}{\delta m} p(\mu_{\varepsilon, \lambda'}, t, T, x, y)(x') \, dz \mu(dx) \, (\mu'-\mu)(dx') \, d\lambda d\lambda'. 
\end{align*}

Observe now that $d_{{\rm TV}}(\Theta_{\varepsilon, \lambda}(t, \mu), \Theta(t, \mu)) \leq \int_{\rr^d} |p((1-\varepsilon)\mu + \varepsilon \mu', t, T, z)-p(\mu, t, T, z)| \, dz$ so that, by the continuity of $\pp \ni \mu \mapsto p(\mu, t, T, z)$, \eqref{integrability:condition} and the dominated convergence theorem, we obtain $\lim_{\varepsilon \downarrow 0} d_{{\rm TV}}(\Theta_{\varepsilon, \lambda}(t, \mu), \Theta(t, \mu))=0$. Note also that $\lim_{\varepsilon \downarrow 0} W_2(\mu_{\varepsilon, \lambda'}, \mu) = 0$. 
 Hence, dividing by $\varepsilon$ both sides  of the previous identity and letting $\varepsilon$ goes to zero yields that $\mu \mapsto h(\Theta(t, \mu))$ admits a continuous linear functional derivative thanks to the dominated convergence theorem as well as the continuity of the maps $\mu \mapsto [\delta h/\delta m](\mu), \, [\delta p/\delta m](\mu, t, T, x, y)(x')$, the boundedness of $ [\delta h/\delta m] $ and \eqref{integrability:condition}. Moreover, it holds 
\begin{align*}
\frac{\delta}{\delta m}[h(\Theta(t, \mu))](v) & = \int_{\mathbb{R}^d}  \frac{\delta h}{\delta m} (\Theta(t, \mu))(z) \, p(\mu, t, T, v,  z) \, dz \\
& \quad + \int_{(\mathbb{R}^d)^2} \left\{ \frac{\delta h}{\delta m} (\Theta(t, \mu))(z) - \frac{\delta h}{\delta m}(\Theta(t, \mu))(x)\right\} \frac{\delta}{\delta m} p(\mu, t, T, x, z)(v) \, dz \mu(dx).
\end{align*}

Each term appearing in the right-hand side, seen as a function of $v$, is continuously differentiable on $\mathbb{R}^d$ so that for any $v_0 \in \rr^d$
\begin{align}
\partial_v \big[\frac{\delta}{\delta m}[h(\Theta(t, \mu))]\big](v) & =  \partial_v \Big(\int_{\mathbb{R}^d}  \frac{\delta h}{\delta m} (\Theta(t, \mu))(z) \, p(\mu, t, T, v,  z) \, dz \Big) \nonumber \\
& \quad + \int_{(\rr^d)^2} \bigg[ \frac{\delta h}{\delta m} (\Theta(t, \mu))(z) - \frac{\delta h}{\delta m}(\Theta(t, \mu))(x)\bigg] \partial_v \frac{\delta}{\delta m} p(\mu, t, T, x, z)(v) \, dz d\mu(x) \nonumber \\
& =  \int_{\rr^d}  \Big[\frac{\delta h}{\delta m} (\Theta(t, \mu))(z) - \frac{\delta h}{\delta m}(\Theta(t, \mu))(v_0)\Big] \, \partial_x p(\mu, t, T, v,  z) \, dz  \nonumber \\
& \quad + \int_{(\rr^d)^2} \bigg[ \frac{\delta h}{\delta m} (\Theta(t, \mu))(z) - \frac{\delta h}{\delta m}(\Theta(t, \mu))(x)\bigg] \partial_v \frac{\delta}{\delta m} p(\mu, t, T, x, z)(y) \, dz d\mu(x) \nonumber \\
& \quad +  \partial_v \Big(\int_{\rr^d}  \frac{\delta h}{\delta m} (\Theta(t, \mu))(v_0) \, p(\mu, t, T, v,  z) \, dz \Big) \nonumber \\
& =  \int_{\rr^d}  \Big[\frac{\delta h}{\delta m} (\Theta(t, \mu))(z) - \frac{\delta h}{\delta m}(\Theta(t, \mu))(v_0)\Big] \, \partial_x p(\mu, t, T, v,  z) \, dz \label{deriv:mu:part1} \\
& \quad + \int_{(\rr^d)^2} \left\{ \frac{\delta h}{\delta m} (\Theta(t, \mu))(z) - \frac{\delta h}{\delta m}(\Theta(t, \mu))(x)\right\} \partial_\mu p(\mu, t, T, x, z)(v) \, dz d\mu(x) \label{deriv:mu:part2}
\end{align}

\noindent where we used the fact that the last term appearing in the last but one equality is 0 since $z\mapsto p(\mu, t, T, v, z)$ is a density function. The joint continuity of the map $(t, \mu, v) \mapsto \partial_v \big[\frac{\delta}{\delta m}[h(\Theta(t, \mu))]\big](v)$ then follows from the above identity and similar arguments as those previously employed in the first step of the proof. Moreover, from the boundedness of $ (v, m)\mapsto [\delta h/\delta m](m)(v)$
and the linear growth of the maps $v \mapsto \int_{\mathbb{R}^d} | \partial_v p(\mu, t, T, v,  z)| \, dz , \, \int_{\mathbb{R}^d} |\partial_\mu p(\mu, t, T, x, z)(v)| \, dz $, uniformly in $\mu \in \mathcal{K}$, we deduce that $\mu \mapsto h(\Theta(t, \mu))$ is continuously $L$-differentiable. The identity \eqref{condition1:structural:class:measure:derivative} for $n=0$ then follows by taking $v_0= v$. 
One may again differentiate \eqref{deriv:mu:part1} and \eqref{deriv:mu:part2} with respect to $v$ for a fixed $(t, \mu) \in [0,T) \times \pp$ and then select $v_0 = v$. Then, one obtains \eqref{condition1:structural:class:measure:derivative} for $n=1$ from the boundedness of $ (v,m)\mapsto [\delta h/\delta m](m)(v)$
and the linear growth $v\mapsto \int_{\mathbb{R}^d} | \partial^2_x p(\mu, t, T, v,  z)| \, dz , \, \int_{\mathbb{R}^d} |\partial_v\partial_\mu p(\mu, t, T, x, z)(v)| \, dz $, uniformly in $\mu \in \mathcal{K}$. The continuity of the map $(t, \mu, v) \mapsto \partial_v [\partial_\mu \big[h(\Theta(t, \mu))\big]](v)$ finally follows from arguments similar to those previously employed in the first step of the proof. The remaining technical details are omitted.

\end{proof}

\section{Overview, assumptions and main results}\label{assumptions:results}

\subsection{On the well-posedness of the martingale problem related to the SDE \eqref{SDE:MCKEAN}.} 

We first present the martingale problem associated to equation \eqref{SDE:MCKEAN}.
\begin{definition}\label{def:martingale:problem}
Let $\mu \in \mathcal{P}(\rr^d)$. We say that the probability measure $\P$ on the canonical space $\mathcal{C}([0,\infty), \rr^d)$ (endowed with the canonical filtration $(\mathcal{F}_t)_{t \geq 0}$) with time marginals $(\P(t))_{t\geq0}$, solves the non-linear martingale problem associated to the SDE \eqref{SDE:MCKEAN} with initial distribution $\mu$ at time $0$ if the canonical process $(y_t)_{t\geq0}$ satisfies the following two conditions:
\begin{itemize}
\item[(i)] $\P(y_0 \in \Gamma) = \mu(\Gamma)$, $\Gamma \in \mathcal{B}(\rr^d)$. 

\item[(ii)] For all $f \in \mathcal{C}^2_b(\rr^d)$, the process
\begin{equation}
\label{martingale:problem:mckean}
f(y_t) - f(y_0) - \int_0^t \left\{ \sum_{i=1}^d b_i(s, y_s, \P(s)) \partial_{x_i} f(y_s) + \frac12 \sum_{i, j=1}^d a_{i, j}(s, y_s,\P(s)) \partial^2_{x_i, x_j}f(y_s) \right\}ds
\end{equation}

\noindent is a square integrable martingale under $\P$. 
\end{itemize}

\end{definition}

\begin{remark}
A similar definition holds by letting the canonical process starts from time $t_0$ with initial distribution $\mu $, in which case we say that the initial condition is $(t_0, \mu)$ and $(i)$ is replaced by the condition: $\P(y(s)\in \Gamma; 0\leq s \leq t_0) = \mu(\Gamma)$.
\end{remark}




Having this definition at hand we now introduce some assumptions on the coefficients:
 \begin{itemize}
%
%

\item[\HR]
\begin{itemize}
\item[(i)] The drift coefficient $b: \rr_+ \times \mathbb{R}^d \times \mathcal{P}(\mathbb{R}^d) \rightarrow \mathbb{R}^d$ is a bounded and measurable function. Moreover, for any $(t, x) \in \mathbb{R}_+ \times \mathbb{R}^d$, the map $m\mapsto b(t, x, m)$ is Lipschitz-continuous for the total variation metric, uniformly with respect to $t, x$, that is, there exists a positive constant $C$ such that for all $(t, x) \in \mathbb{R}_+\times \mathbb{R}^d$, for all $m, m' \in \mathcal{P}(\mathbb{R}^d)$
$$
|b(t, x, m) - b(t, x, m')| \leq C d_{{\rm TV}}(m, m')
$$

\noindent where we remind the reader that $d_{{\rm TV}}$ denotes the total variation metric on $\mathcal{P}(\mathbb{R}^d)$.

\item[(ii)] The diffusion coefficient $a :  \mathbb{R}_+ \times \mathbb{R}^d \times \mathcal{P}(\mathbb{R}^d) \rightarrow \mathbb{R}^d \otimes \mathbb{R}^d$, where $a(t, x, m) = (\sigma \sigma^{*})(t, x, m)$, is a bounded and continuous function. Moreover, for any $(t, m) \in \rr_+\times \mathcal{P}(\rr^d)$, the function $ \mathbb{R}^d \ni x \mapsto a(t, x, m) \in \mathbb{R}^{d} \otimes \mathbb{R}^d$ is uniformly $\eta$-H\"older continuous for some $\eta \in (0,1]$, namely
$$
[a]_{H}:=\sup_{t \geq 0, \, x \neq y, \, m \in \mathcal{P}(\mathbb{R}^d)}\frac{| a(t, x, m) - a(t, y, m) |}{|x-y|^{\eta}} < \infty.
$$

\item[(iii)] For any $(i, j) \in \left\{ 1, \cdots, d\right\}^2$ and any $(t, x) \in \mathbb{R}_+\times \mathbb{R}^d$, the map $\mathcal{P}(\mathbb{R}^d) \ni m \mapsto a_{i, j}(t, x, m)$ has a linear functional derivative.

\item[(iv)] For any $(i, j) \in \left\{ 1, \cdots, d\right\}^2$ and any $(t, m) \in \mathbb{R}_+ \times \mathcal{P}(\mathbb{R}^d)$, the map $(\mathbb{R}^d)^2 \ni (x, y) \mapsto  [\delta a_{i, j}/ \delta m](t, x, m)(y)$ is an $\eta$-H\"older continuous function, for some $\eta \in (0,1]$, uniformly with respect to the variables $t$ and $m$.
\end{itemize}
 
 \medskip
 
\item[\HE] The diffusion coefficient is uniformly elliptic, that is, there exists $\lambda \geq 1$ such that for any $(t, m) \in [0,\infty) \times \mathcal{P}(\mathbb{R}^d)$ and any $(x, z) \in (\mathbb{R}^d)^2 $, $\lambda^{-1} |z|^2 \leq \langle a(t, x, m) z,z \rangle \leq \lambda |z|^2$.
\bigskip

\end{itemize}

\begin{remark} \noindent $\circ$ Assumption \HR(i) may be reformulated as the following slightly stronger assumption: the map $\mathcal{P}(\rr^d) \ni m\mapsto b(t, x, m)$ has a linear functional derivative.\\

\noindent $\circ$ Note that under assumption \HR (iii) and (iv), the map $\mathcal{P}(\mathbb{R}^d)\ni m\mapsto a_{i, j}(t, x, m)$ is Lipschitz-continuous with respect to the distance 
$$
d_\eta(m, m') = \inf_{\pi \in \Pi(m,m')} \int_{(\mathbb{R}^d)^2} \left\{|x-y|^\eta \wedge 1\right\} \, \pi(dx, dy)
$$

\noindent where $\Pi(m, m')$ is the set of all transference plan from $m$ to $m'$. Indeed, for any $m, m' \in \mathcal{P}(\mathbb{R}^d)$ and any transference plan $\pi \in \Pi(m, m')$ 
\begin{eqnarray*}
&&| a_{i, j}(t, x, m) - a_{i, j}(t, x, m') | \\
& =& | \int_0^1 \int_{\mathbb{R}^d} \frac{\delta a_{i, j}}{\delta m} (t, x, (1-\lambda) m + \lambda m')(y') (m-m')(dy') \, d\lambda | \\
& =&  | \int_0^1\int_{\mathbb{R}^d} \frac{\delta a_{i, j}}{\delta m} (t, x, (1-\lambda) m + \lambda m')(x) - \frac{\delta}{\delta m} a_{i, j}(t, x, (1-\lambda) m + \lambda m')(y)  \, d\lambda \, \pi(dx, dy) | \\
& \leq& \sup_{t, x, m}\Big[\frac{\delta a_{i, j}}{\delta m} (t, x, m )(.)\Big]_{H}\int_{(\mathbb{R}^d)^2} \left\{|x-y|^\eta \wedge 1\right\} \, \pi(dx, dy)
\end{eqnarray*}

\noindent where $\sup_{t, x, m}\big[ [\delta a_{i, j}/\delta m](t, x, m )(.)\big]_{H}$ denotes the uniform H\"older modulus of the map $[\delta a_{i, j}/\delta m](t, x, m)(\cdot)$. Finally, the claim follows by taking the infimum in the previous inequality with respect to $\pi \in \Pi(m, m')$.
\end{remark}

Our first main result concerns the well-posedness of the martingale problem associated to the SDE \eqref{SDE:MCKEAN}.

\begin{theorem}\label{thm:martingale:problem} Under \HR\, and \HE, the martingale problem associated with \eqref{SDE:MCKEAN} is well-posed for any initial distribution $\mu \in \mathcal{P}(\mathbb{R}^d)$. In particular, weak uniqueness in law holds for the SDE \eqref{SDE:MCKEAN}. 
\end{theorem}

When investigating strong well-posedness of non-linear SDE an interesting fact is that, combining uniqueness in law for the non-linear SDE together with strong uniqueness result for the associated linear SDE, \emph{i.e.} the same SDE with time-inhomogeneous coefficients, the law argument being now treated as a time-inhomogeneity, immediately yields to strong uniqueness. To be more specific, from the previous well-posedness result we have that any strong solution $Y$ of the SDE \eqref{SDE:MCKEAN} (if it exists) writes
\begin{equation}\label{SDE:falseMCKEAN}
Y_t = \xi + \int_0^t b(s,Y_s, [X^{\xi}_s]) ds + \int_0^t \sigma(s,Y_s,[X^{\xi}_s]) dW_s 
\end{equation}

\noindent implying that, setting $\widehat b : \rr^+ \times \rr^d \ni (t,y) \mapsto b(t, y, [X^{\xi}_t]) \in \rr^d$ and $\widehat \sigma : \rr^+ \times \rr^d \ni (t, y) \mapsto \widehat \sigma (t, y, [X^{\xi}_t]) \in \rr^d \times \rr^d$, it solves 
\begin{equation}\label{SDE:falseMCKEAN2}
Y_t = \xi + \int_0^t \widehat b(s,Y_s) ds + \int_0^t \widehat \sigma(s,Y_s) dW_s. 
\end{equation}

But this linear SDE is well posed in the strong sense under the additional assumption that the diffusion coefficient $\widehat{\sigma}$ is Lipschitz in space (see \cite{veretennikov_strong_1980}). Hence, any strong solutions of \eqref{SDE:falseMCKEAN} are equals $\mathbb{P}$-a.s. so that strong well-posedness follows from the Yamada-Watanabe theorem. This gives the following corollary.
\begin{cor}\label{cor:strong}
Assume that the assumptions of Theorem \ref{thm:martingale:problem} hold and, that for all $(t, m)$ in $\mathbb{R}_+ \times \mathcal{P}(\mathbb{R}^d)$, the map $x \mapsto \sigma(t ,x, m)$ is Lispchitz continuous uniformly with respect to $t$ and $m$. Then, strong uniqueness holds for the SDE \eqref{SDE:MCKEAN}.
\end{cor}

\subsection{On the density of the solution of the SDE \eqref{SDE:MCKEAN} and its regularity properties.} 
Under the assumption of Theorem \ref{thm:martingale:problem}, by weak uniqueness, the law of the process $(X^{s, \xi}_t)_{t\geq s}$ given by the unique solution to the SDE \eqref{SDE:MCKEAN} starting from the initial distribution $\mu = [\xi]$ at time $s$ only depends upon $\xi$ through its law $\mu$. Given $\mu \in \pp$, it thus makes sense to consider $([X^{s, \xi}_t])_{t\geq s}$ as a function of $\mu$ (and also of the time variable $s$) without specifying the choice of the lifted random variable $\xi$ that has $\mu$ as distribution. We then introduce, for any $x\in \rr^d$, the following \emph{decoupled stochastic flow} associated to the SDE \eqref{SDE:MCKEAN}
\begin{equation}
\label{SDE:MCKEAN:decoupling}
X^{s, x, \mu}_t = x + \int_s^t b(r, X^{s, x, \mu}_r, [X^{s, \xi}_r]) \, dr + \int_s^t \sigma(r, X^{s, x , \mu}_r, [X^{s, \xi}_r]) \, dW_r.
\end{equation}

We note that the previous equation is not a McKean-Vlasov SDE since the law appearing in the coefficients is not $[X^{s, x, \mu}_r]$ but rather $[X^{s, \xi}_r]$, that is, the marginal law of the solution to the SDE \eqref{SDE:MCKEAN} (starting at time $s$ from the initial distribution $\mu$) evaluated at time $r$. Under the assumptions of Theorem \ref{thm:martingale:problem}, the time-inhomogeneous martingale problem associated to the SDE \eqref{SDE:MCKEAN:decoupling} is well-posed, see e.g. Stroock and Varadhan \cite{stroock:varadhan}. In particular, weak existence and uniqueness in law holds for the SDE \eqref{SDE:MCKEAN:decoupling}.

Moreover, from Friedman \cite{friedman:64}, see also McKean and Singer \cite{mcke:sing:67}, it follows that the transition density of the SDE \eqref{SDE:MCKEAN:decoupling} exists\footnote{In \cite{friedman:64}, it is proved that if $x\mapsto \bar{b}(r, x) = b(r, x, [X^{s, \xi}_r])$ is bounded and H\"older-continuous then the fundamental solution associated to the infinitesimal generator of \eqref{SDE:MCKEAN:decoupling} exists and is unique by means of the parametrix method. However, existence of the transition density as well as weak existence and weak uniqueness can be derived under the sole assumption that the drift coefficient is bounded and measurable and the diffusion matrix is uniformly elliptic, bounded and H\"older continuous.}. In particular, the random variable $X^{s, x, \mu}_t$ has a density that we denote by $z \mapsto p(\mu, s, t, x, z)$ which admits a representation in infinite series by means of the parametrix method that we now briefly describe. We refer the reader to \cite{friedman:64} or Konakov and Mammen \cite{kona:mamm:00} for a more complete exposition. We first introduce the approximation process $(\widehat{X}^{t_1, x, \mu}_{t_2})_{t_2\geq t_1}$ obtained from the dynamics \eqref{SDE:MCKEAN:decoupling} by removing the drift and freezing the diffusion coefficient in space at a fixed point $y$, namely
\begin{equation}
\label{SDE:MCKEAN:decoupling:parametrix:process}
\widehat{X}^{t_1, x, \mu}_{t_2} = x +  \int_{t_1}^{t_2} \sigma(r, y, [X^{s, \xi}_r]) \, dW_r.
\end{equation}

The process $(\widehat{X}^{t_1, x, \mu}_{t_2})_{t_2\geq t_1}$ is a simple Gaussian process with transition density given explicitly by 
$$
\widehat{p}^{y}(\mu, s, t_1, t_2, x, z) := g\left(\int_{t_1}^{t_2} a(r, y, [X^{s, \xi}_r]) \, dr, z-x\right).
$$

To make the notation simpler, we will write $\widehat{p}(\mu, s, t_1, t_2, x, y) := \widehat{p}^{y}(\mu, s, t_1, t_2, x, y)$ and $\widehat{p}^{y}(\mu, s, t_2, x, z) = \widehat{p}^{y}(\mu, s, s, t_2, x, z)$. Note importantly that the variable $y$ acts twice since it appears as a terminal point where the density is evaluated and also as the point where the diffusion coefficient is frozen. Note also that in what follows we need to separate between the starting time $t_1$ of the approximation process and the starting time $s$ of the original McKean-Vlasov dynamics. We now introduce the two infinitesimal generators associated to the dynamics \eqref{SDE:MCKEAN:decoupling} and \eqref{SDE:MCKEAN:decoupling:parametrix:process}, namely
\begin{align*}
\mathcal{L}_{s, t} f(\mu, t, x)  & = \sum_{i=1}^d b_i(t, x, [X^{s, \xi}_t]) \partial_{x_i} f(\mu, t, x) + \frac12 \sum_{i, j=1}^d a_{i, j}(t, x, [X^{s, \xi}_t]) \partial^2_{x_i, x_j} f(\mu, t, x),   \\
\widehat{\mathcal{L}}_{s, t} f(\mu, t, x)  & =  \frac12 \sum_{i, j=1}^d a_{i, j}(t, y, [X^{s, \xi}_t]) \partial^2_{x_i, x_j} f(\mu, t, x)
\end{align*}

\noindent and define the parametrix kernel $\mH$ for $(\mu, r, x, y) \in \pp \times [s, t) \times (\rr^d)^2$
\begin{align*}
\mH(\mu, s, r, t, x, y) & := (\mathcal{L}_{s, r} - \widehat{\mathcal{L}}_{s, r})  \widehat{p}(\mu, s, r, t, x, y) \\
& = \sum_{i=1}^d b_i(r, x, [X^{s, \xi}_r]) \partial_{x_i} \widehat{p}(\mu, s,  r, t, x, y) \\
& \quad + \frac12 \sum_{i, j=1}^d (a_{i, j}(r, x, [X^{s, \xi}_r]) - a_{i, j}(r, y, [X^{s, \xi}_r]))\partial^2_{x_i, x_j} \widehat{p}(\mu, s, r, t, x, y).
\end{align*}

Now we define the following space-time convolution operator
$$
(f \otimes g)(\mu, s, r, t, x, y) := \int_r^t \int_{\rr^d} f(\mu, s, r, r', x, z) g(\mu, s, r', t, z, y) \, dz \, dr'
$$

\noindent and to simplify the notation we will write $(f \otimes g)(\mu, s, t, x, y) := (f \otimes g)(\mu, s, s, t, x, y)$, $\mH(\mu, s, t, x, z) = \mH(\mu, s, s, t, x, z)$ and proceed similarly for other maps. We also define $f \otimes \mH^{(k)} = (f\otimes \mH^{(k-1)}) \otimes \mH$ for $k\geq1$ with the convention that $f \otimes \mH^{(0)} \equiv f$. With these notations, the following parametrix expansion in infinite series of the transition $p(\mu, s, t, x, z)$ holds. Let $T>0$. For any $0\leq s < t \leq T$ and any $(\mu, x, y) \in \pp \times (\rr^d)^2$
\begin{equation}
p(\mu, s, t, x, y) = \widehat{p}(\mu, s, t, x, y) + p\otimes \mH(\mu, s, t, x, y). \label{first:step:parametrix:series:expansion}
\end{equation}

\noindent so that, by induction
\begin{equation}
p(\mu, s, t, x, y) = \sum_{k\geq0} (\widehat{p} \otimes \mH^{(k)})(\mu, s, t, x, y). \label{parametrix:series:expansion}
\end{equation}

Moreover, the above infinite series converge absolutely and uniformly for $(\mu, x, y) \in \pp \times (\rr^d)^2$ and satisfies the following Gaussian upper-bound: for any $0\leq s < t \leq T$ and any $(\mu, x, y) \in \pp \times (\rr^d)^2$ 
\begin{equation}
\label{bound:density:parametrix}
p(\mu, s, t, x, y) \leq E_{\eta/2, 1}( C (|b|_{\infty}+1)) \, g(c (t-s), y-x)
\end{equation}

\noindent where $C:=C(T, \lambda, \eta)$ and $c:=c(\lambda)$ are two positive constants. We refer to \cite{mcke:sing:67} for a proof based on Kolmogorov's backward and forward equations satisfied by $p$, see also Frikha \cite{hdr:report:noufel} for a proof based on probabilistic arguments. 

Under the additional assumption that $x\mapsto b(t, x, \mu)$ is $\eta$-H\"older continuous, it turns out that $x\mapsto p(\mu, s, t, x, z)$ is two times continuously differentiable. Moreover, the following pointwise Gaussian estimates for its derivatives hold: for any $\beta \in [0,\eta)$, there exist some positive constants $C:=C(T, b, a, \lambda, \eta)$, $C_\beta:=C(T, b, a, \lambda, \eta, \beta)$ and $c:=c(\lambda)$ such that for any $(\mu, x, y, z) \in \pp \times (\mathbb{R}^d)^3$ and any $0\leq s < t \leq T $
\begin{equation}
\label{gradient:estimates:parametrix:series}
|\partial^{n}_x p(\mu, s, t, x, z)| \leq  \frac{C}{(t-s)^{\frac{n}{2}}} \, g(c (t-s), z-x), \quad n=0,1,2
\end{equation}

\noindent and
\begin{equation}
\label{second:derivatives:holder:estimates:parametrix:series}
|\partial^{2}_x p(\mu, s, t, x, z) - \partial^2_x p(\mu, s, t, y, z) | \leq  C_\beta \frac{|x-y|^{\beta}}{(t-s)^{1+ \frac{\beta}{2}}} \,\Big[ g(c (t-s), z-x) + g(c (t-s), z-y)\Big].
\end{equation}
 
 We refer again to \cite{friedman:64} for a proof of the above estimates. Let us point out that the differentiability of the map $[0,t) \times \pp \ni (s, \mu) \mapsto p(\mu, s, t, x, z)$ is the main question that we want to address here. 

A similar representation in infinite series is also valid for the density of the random variable $X^{s, \xi}_t$, denoted by $z\mapsto p(\mu, s, t, z)$, but we will not use it explicitly. Actually, we will make use of the following key relation
\begin{equation}
\label{relation:density:mckean:decoupling:field}
p(\mu, s, t , z) = \int_{\rr^d} p(\mu, s, t ,x ,z) \, \mu(dx).
\end{equation}

The representation in infinite series of $p(\mu, s , t, z)$ is thus obtained by integrating $x\mapsto p(\mu, s, t, x, z)$ against the initial distribution $\mu$, in other words, $z \mapsto p(\mu, s, t, z)$ is the density of the image measure of the map $x\mapsto p(\mu, s, t, x, z)$ by the measure $\mu$.

\medskip
We now introduce the following additional assumption on the coefficients.
\begin{itemize}
\item[\HRp] The coefficients $b$ and $a$ satisfy \HR\, and the following additional assumptions :

\begin{itemize}
\item[(i)]  The drift coefficient $b: \mathbb{R}_+ \times \mathbb{R}^d \times \mathcal{P}(\mathbb{R}^d) \rightarrow \mathbb{R}^d$ is a continuous function. For any $(t, m) \in \rr_+ \times \mathcal{P}(\mathbb{R}^d)$, the function $ \mathbb{R}^d \ni x \mapsto b(t, x, \mu) \in \rr^{d}$ is $\eta$-H\"older continuous for some $\eta \in (0,1]$, uniformly with respect to the variables $t$ and $m$, namely
$$
[b]_{H}:=\sup_{t \geq 0, \, x \neq y, \,  m \in \mathcal{P}(\mathbb{R}^d)}\frac{| b(t, x,m) - b(t, y,m) |}{|x-y|^{\eta}} < \infty.
$$

\item[(ii)] For any $(i,j) \in \left\{1, \cdots, d\right\}^2$ and any $(t, x, y) \in \mathbb{R}_+\times (\mathbb{R}^d)^2$, the map $ \mathcal{P}(\mathbb{R}^d) \ni m \mapsto [\delta a_{i, j}/\delta m](t, x, m)(y)$ has a linear functional derivative, such that the map $(\rr^d)^2 \ni (x,y')\mapsto [\delta^2 a_{i, j} /\delta m^2] (t, x, m)(y, y')$ is $\eta$-H\"older continuous uniformly with respect to the other variables.

\item[(iii)] For any $i \in \left\{1, \cdots, d\right\}$ and any $(t, x) \in \mathbb{R}_+\times \mathbb{R}^d$, the map $\mathcal{P}(\mathbb{R}^d) \ni m \mapsto b_{i}(t, x, m)$ has a linear functional derivative, such that $y\mapsto [\delta b_{i}/\delta m]  (t, x, m)(y)$ is $\eta$-H\"older continuous uniformly with respect to the other variables. Moreover, for any $i \in \left\{1, \cdots, d\right\}$, for any $(t, x, y) \in \mathbb{R}_+\times (\mathbb{R}^d)^2$, the map $ \mathcal{P}(\mathbb{R}^d) \ni m \mapsto [\delta  b_{i}/\delta m](t, x, m)(y)$ has a linear functional derivative, such that $y'\mapsto [\delta^2  b_{i}/\delta m^2](t, x, m)(y, y')$ is $\eta$-H\"older continuous uniformly with respect to the other variables.
\end{itemize}
\end{itemize}

Our next result concerns the regularity properties of the two maps $(s,\mu) \mapsto p(\mu, s, t, z)$ and $(s, \mu, x) \mapsto p(\mu, s, t, x, z)$ and also important pointwise Gaussian estimates on its derivatives. As mentioned above under the assumptions of Theorem \ref{thm:martingale:problem} and \HRp, $\mathbb{R}^d \ni x \mapsto p(\mu, s, t, x, z)$ is two times continuously differentiable. In view of the relation \eqref{relation:density:mckean:decoupling:field}, it thus suffices to investigate the smoothness of the map $(s, \mu, x) \mapsto p(\mu, s, t, x, z)$.

\begin{theorem}\label{derivative:density:sol:mckean:and:decoupling} Assume that \HE\, and \HRp\, hold. Let $T>0$ and $(t, z) \in (0,T] \times \mathbb{R}^d$. Then, the mapping $[0,t) \times \rr^d \times \pp \ni (s, x, \mu) \mapsto p(\mu, s, t, x, z)$ is in $\mathcal{C}^{1, 2, 2}([0,t) \times \mathbb{R}^d \times \pp)$. 
Moreover, for any $(\mu, \mu', s, x, x', v, v') \in (\mathcal{P}_2(\mathbb{R}^d))^2 \times [0,t) \times (\mathbb{R}^d)^4$ and any $(s_1, s_2) \in [0,t)$, 
\begin{align}
|\partial^{n}_v [\partial_\mu p(\mu, s, t, x, z)](v)| & \leq  \frac{C}{(t-s)^{\frac{1+n-\eta}{2}}} g(c (t-s), z-x), \, n \in \left\{0,1\right\}, \label{first:second:lions:derivative:mckean:decoupling} \\
|\partial_s p(\mu, s, t, x, z)| & \leq  \frac{C}{t-s} g(c (t-s), z-x), \label{first:time:derivative:mckean:decoupling}
\end{align}



\noindent for any $\beta \in [0,1]$ if $n=0$ or any $\beta \in [0,\eta)$ if $n=1$,
\begin{align}
 \quad | \partial^{n}_v [\partial_\mu p(\mu, s, t, x, z)](v) & - \partial^{n}_v [\partial_\mu p(\mu, s, t , x', z)] (v) |\nonumber \\
&  \leq C_\beta \frac{|x-x'|^{\beta}}{(t-s)^{\frac{1+n+\beta-\eta}{2}}} \left\{ g(c(t-s), z-x) + g(c(t-s), z-x') \right\},  \label{equicontinuity:second:third:estimate:decoupling:mckean:final} 
\end{align}

\noindent for any $\beta \in [0,\eta)$,
\begin{align}
 |\partial_v [\partial_\mu p(\mu, s, t, x, z)](v) - \partial_v [\partial_\mu p(\mu, s, t, x, z)](v')| & \leq C_\beta \frac{|v-v'|^{\beta}}{(t-s)^{1+ \frac{\beta-\eta}{2}}} g(c(t-s), z-x), \label{equicontinuity:first:estimate:mckean:decoupling}
\end{align}

\noindent for any $\beta \in [0,1]$ if $n \in \left\{0, 1\right\}$ or any $\beta \in [0,\eta)$ if $n=2$,
\begin{align}
 |\partial^{n}_x p(\mu, s, t, x, z) & - \partial^{n}_x p(\mu', s, t, x, z)](v)| \leq C_\beta \frac{W^{\beta}_2(\mu, \mu')}{(t-s)^{ \frac{n+\beta}{2}}} \, g(c(t-s), z-x), \label{regularity:measure:estimate:v1:v2:v3:mckean:decoupling} 
\end{align}

\noindent for any $\beta\in [0,1]$ if $n=0$ or any $\beta\in [0,\eta)$ if $n=1$, 
\begin{align}
 |\partial^{n}_v [\partial_\mu p(\mu, s, t, x, z)](v) - \partial^{n}_v [\partial_\mu p(\mu', s, t, x, z)](v)| & \leq C^{+}_\beta \frac{W^{\beta}_2(\mu, \mu')}{(t-s)^{ \frac{1+n+\beta- \eta}{2}}} g(c(t-s), z-x) \label{regularity:measure:estimate:v1:v2:mckean:decoupling} 
\end{align}

\noindent for any $\beta \in [0,1]$ if $n=0$, any $\beta \in [0,\frac{1+\eta}{2})$ if $n=1$ or any $\beta \in [0, \frac{\eta}{2})$ if $n=2$, 
\begin{align}
 | & \partial^{n}_x p (\mu, s_1, t, x, z) - \partial^{n}_x p(\mu, s_2, t, x, z)  | \nonumber \\
 & \leq C_\beta \left\{ \frac{|s_1-s_2|^{\beta}}{(t-s_1)^{\frac{n}{2} + \beta }} \, g(c(t-s_1), z-x) + \frac{|s_1-s_2|^{\beta}}{(t- s_2)^{ \frac{n}{2} +\beta }} \, g(c(t-s_2), z-x) \right\}, \label{regularity:time:estimate:v1:v2:v3:mckean:decoupling} 
\end{align}

\noindent and for any $\beta \in [0, \frac{1+\eta}{2})$ if $n=0$ or any $\beta \in [0,\frac{\eta}{2})$ if $n=1$,
\begin{align}
 | & \partial^{n}_v [\partial_\mu p(\mu, s_1, t, x, z)](v) - \partial^n_v [\partial_\mu p(\mu, s_2, t, x, z)](v)| \nonumber \\
 & \leq C^{+}_\beta \left\{ \frac{|s_1-s_2|^{\beta}}{(t-s_1)^{ \frac{1+n-\eta}{2}+\beta}} \, g(c(t-s_1), z-x) + \frac{|s_1-s_2|^{\beta}}{(t- s_2)^{\frac{1+n-\eta}{2} + \beta}} \, g(c(t-s_2), z-x) \right\},  \label{regularity:time:estimate:v1:v2:mckean:decoupling} 
\end{align}

 \noindent where $C := C(T, \HR, \, \HE)$, $C_\beta := C(T, \HR, \,\HE, \beta)$, $C^{+}_\beta := C^{+}_\beta(T, \HRp,\, \HE, \beta)$ and $c:=c(\lambda)$ are positive constants.

 From this result and the key relation \eqref{relation:density:mckean:decoupling:field} we deduce the following corollary.
\begin{cor}
Assume that \HE\, and \HRp\, hold. Let $T>0$ and $(t, z) \in (0,T] \times \mathbb{R}^d$. Then, the mapping $[0,t) \times \pp \ni (s, \mu) \mapsto p(\mu, s, t, z)$  in $\mathcal{C}^{1, 2}([0,t) \times \pp)$. Moreover, there exist positive constants $C:=C(T, \HR, \HE)$ and $c:=c(\lambda)$ such that for any $(\mu,  s,  v,) \in \mathcal{P}_2(\mathbb{R}^d) \times [0,t) \times \mathbb{R}^d$, 
\begin{align}
|\partial^{n}_v [\partial_\mu p(\mu, s, t, z)](v)| & \leq  C\Bigg\{\frac{1}{(t-s)^{\frac{1+n}{2}}} g(c (t-s), z-v) \nonumber \\
&+ \frac{1}{(t-s)^{\frac{1+n-\eta}{2}}}\int_{\rr^d} g(c (t-s), z-x) \mu(dx)\Bigg\}, \, n \in \left\{0,1\right\}, \label{first:second:lions:derivative:mckean:decoupling} \\
|\partial_s p(\mu, s, t, z)| & \leq  \frac{C}{t-s} \int_{\rr^d}g(c (t-s), z-x)\mu(dx). \label{first:time:derivative:mckean:decoupling}
\end{align}
\end{cor}

\end{theorem}

\subsection{On the Cauchy problem related to the PDE \eqref{pde:wasserstein:space}.} 

The previous regularity properties on the density of the random variables $X^{s, \xi}_t$ and $X^{s, x, \mu}_t$ allow us in turn to tackle the Cauchy problem related to the PDE \eqref{pde:wasserstein:space} on the Wasserstein space defined in the strip $[0,T] \times \rr^d \times \pp$. We introduce the following assumption on the two real-valued maps $f$ and $h$ appearing in \eqref{pde:wasserstein:space}:
\begin{itemize}
\item[\HST] 
\begin{itemize}
\item[(i)] The two maps $[0,T] \times \mathbb{R}^d \times \pp \ni (t, x, m) \mapsto f(t, x, m)$ and $\mathbb{R}^d \times \pp \ni (x, m) \mapsto h(x, m)$ are continuous and the two maps $\pp \ni m \mapsto f(t, x, m),\, h(x, m)$ have a continuous linear functional derivative for any fixed $(t, x) \in [0,T] \times \mathbb{R}^d$, c.f. Remark \ref{rem:linderivandLip} (i). Moreover, the maps $[0,T] \times (\mathbb{R}^d)^2 \times \pp \ni (t, x, y, m) \mapsto [\delta f/\delta m](t, x, m)(y)$, $(\mathbb{R}^d)^2 \times \pp \ni (x, y, m) \mapsto [\delta h/\delta m](x, m)(y)$ are continuous.


\item[(ii)] The maps $f, \, h, \, [\delta f/\delta m]$ and $[\delta h/\delta m] $ satisfy the following regularity and growth assumptions: there exist $C:=C(T)\geq 0$ and $q\geq1$ such that for any $(t, x, y,  m) \in [0,T] \times (\rr^d)^2 \times \pp$ and any bounded set $D \subset \mathbb{R}^d$,
\begin{align}
\sup_{x \neq x', x, x' \in D } \frac{|f(t, x, m) - f(t, x', m)|}{|x-x'|^{\eta}} & \leq C  (1 + M^{q}_2(m)), \label{local:holder:reg:f} \\
\sup_{y \neq y', y, y' \in D } \frac{|[\delta f/\delta m](t, x, m)(y) - [\delta f/\delta m](t, x, m)(y')|}{|y-y'|^{\eta}} & \leq C  \exp\Big(\alpha \frac{|x|^2}{T}\Big) (1 + M^{q}_2(m)), \label{local:holder:reg:linear:functional:deriv:f} 
\end{align}
\noindent and
\begin{align}
| f(t, x, m) | + | h(x, m) | & \leq C \exp\Big(\alpha \frac{|x|^2}{T}\Big) (1 + M^{q}_2(m)), \label{growth:condition:h:f}\\
| \frac{\delta }{\delta m}f(t, x, m)(y) | + | \frac{\delta }{ \delta m}h(x, m)(y) |  & \leq C  \exp\Big(\alpha \frac{|x|^2}{T}\Big) (1 + |y|^2 + M^{q}_2(m)) \label{growth:condition:tilde:h:f}
\end{align}

\noindent where $M_2(m):=\int_{\mathbb{R}^d} |x|^2 m(dx)$ and $\alpha$ is any non-negative constant satisfying $\alpha < (2c)^{-1}$, the constant $c$ being the maximum among the constants $c$ appearing in the estimates \eqref{gradient:estimates:parametrix:series} and \eqref{first:second:lions:derivative:mckean:decoupling}.
\end{itemize}

\end{itemize}

\begin{theorem}\label{cauchy:problem:wasserstein} Assume that \HE,\, \HRp\, and \HST\, hold. Then, the function $U$ defined by
\begin{align}
U(t, x, \mu) & := \int_{\mathbb{R}^d} h(z, [X^{t, \xi}_T]) \, p(\mu, t, T, x, z) \, dz - \int_t^T \int_{\mathbb{R}^d} f(s, z, [X^{t, \xi}_s]) \,  p(\mu, t, s, x, z) \, dz  \, ds \label{sol:pde:wasserstein} \\
& = \E\left[h(X^{t, x, \mu}_T, [X^{t, \xi}_T]) - \int_t^T f(s, X^{t, x, \mu}_s, [X^{t, \xi}_s]) \, ds\right] \nonumber
\end{align}

\noindent where $\xi \in \mathbb{L}^2$ with $[\xi]=\mu$, is a solution of the Cauchy problem \eqref{pde:wasserstein:space} (in the strip $[0,T] \times \rr^d \times \pp$) and
\begin{equation}
\label{bound:solution:cauchy:problem}
|U(t, x, \mu)| \leq C \exp\Big(\frac{k |x|^2}{T}\Big) (1+ M^{q}_2(\mu)), \quad \mbox{ for } (t, x, \mu) \in [0,T] \times \rr^d \times \pp
\end{equation}

\noindent where $C:=C(T, \HR,\, \HE)$ and $k:=k(\lambda, \alpha)$ are positive constants. 

Moreover, $U$ is unique among all of the classical solutions to the PDE \eqref{pde:wasserstein:space} satisfying \eqref{cond:integrab:ito:process:second:version}, $T$ being replaced by any $T' \in [0,T)$, as well as the exponential growth assumption \eqref{bound:solution:cauchy:problem} and with $h$ and $f$ satisfying \eqref{growth:condition:h:f} and \eqref{growth:condition:tilde:h:f} for some positive constants $k$ and $\alpha$.
\end{theorem}

\begin{remark}
We importantly point out that in order to address the continuity of the map $t\mapsto U(t, x, \mu)$ at time $T$ as well as its differentiability on the interval $[0,T)$, we here impose a stronger regularity assumption on the terminal condition $h$ and source term $f$ compared to \HR\, or \HRp. In particular, we assume that $h$ and $f$ are continuous in the measure direction with respect to 2-Wasserstein distance. This seems natural inasmuch if $(t_n)_{n\geq0}$ is a sequence of $[0,T]$ such that $t_n \rightarrow T$, the convergence of the sequence of probability measures $([X^{t_n, \xi}_T])_{n\geq0}$ only holds weakly in $\mathcal{P}_2(\mathbb{R}^d)$ or equivalently with respect to the $2$-Wasserstein distance and not with respect to the total variation distance.   

\end{remark}

\subsection{Examples}
Here are some examples of McKean-Vlasov SDEs whose coefficients satisfy assumptions \HE\, and \HR \, or \HRp.

\begin{example}(First order interaction) Consider the following non-linear SDE with coefficients $b: \mathbb{R}_+ \times (\mathbb{R}^d)^2  \rightarrow \mathbb{R}^d$ and $\sigma: \mathbb{R}_+ \times (\mathbb{R}^d)^2 \rightarrow \mathbb{R}^d \otimes \mathbb{R}^q$:
\begin{eqnarray}
X^{\xi}_t &=& \xi + \int_0^t \widetilde{\E}[b(s,X^{\xi}_s,\widetilde{X}^{\xi}_s)] ds + \int_0^t \widetilde{\E}[\sigma(s,X^{\xi}_s,\widetilde{X}^{\xi}_s)] dW_s, \quad [\xi] = \mu \in \pp,\label{ex:wp:2}
\end{eqnarray}

\noindent where the process $(\widetilde{X}^{\xi}_t)_{t\geq0}$ is a copy of $(X^{\xi}_t)_{t\geq0}$ defined on a copy $(\widetilde{\Omega}, \widetilde{\mathcal{F}}, \widetilde{\P})$ of the original probability space $(\Omega, \mathcal{F}, \P)$.

\begin{itemize}
\item If $b$ and $\sigma$ are bounded and measurable functions and if moreover $\sigma$ is $\eta$-H\"older continuous uniformly with respect to $t$ then, $b$ and $a=\sigma \sigma^{*}$ satisfy \HR.
\item If in addition $b$ is $\eta$-H\"older continuous uniformly with respect to $t$ then, $b$ and $a=\sigma \sigma^{*}$ satisfy \HRp.
\item If $(\int \sigma(t, x, z) m(dz)) (\int \sigma(t, x, z) m(dz))^{*}$ is uniformly elliptic, uniformly with respect to the variables $t,\, x, \, m$, then \HE\, holds.
\end{itemize}
\end{example}

\begin{example}($N$ order interaction) For some positive integer $N$, we consider the following non-linear SDE with coefficients $b: \rr_+ \times (\rr^d)^{N+1}  \rightarrow \rr^d$ and $\sigma: \rr_+ \times (\rr^d)^{N+1} \rightarrow \rr^{d\times q}$:
\begin{align}
X^{\xi}_t &= \xi + \int_0^t \widetilde{\E}[b(s,X^{\xi}_s,X^{\xi,(1)}_s,\cdots,X^{\xi,(N)}_s)] ds + \int_0^t \widetilde{\E}[\sigma(s,X^{\xi}_s,X^{\xi,(1)}_s,\cdots,X^{\xi,(N)}_s)]  dW_s,\label{ex:wp:1}
\end{align}
where the processes $\left\{(X^{\xi,(i)}_t)_{t\geq0},\ 1 \leq i \leq N\right\}$ are mutually independent copies of the process $(X^{\xi}_t)_{t\geq0}$ defined on a copy $(\widetilde{\Omega}, \widetilde{\mathcal{F}}, \widetilde{\P})$ of the original probability space $(\Omega, \mathcal{F}, \P)$.

\begin{itemize}
\item If $b$ and $\sigma$ are bounded and measurable functions and if moreover $\sigma$ is continuous in $t$ and $\eta$-H\"older continuous uniformly with respect to $t$ then, $b$ and $a=\sigma \sigma^{*}$ satisfy \HR.
\item If in addition $b$ is continuous in $t$ and $\eta$-H\"older continuous uniformly with respect to $t$ then, $b$ and $a=\sigma \sigma^{*}$ satisfy \HRp.
\item Let $m_N$ denotes the $N$-fold product measure of $\mu$. If $(\int \sigma(t, x, z) m_N(dz)) (\int \sigma(t, x, z) m_N(dz))^{*}$ is uniformly elliptic, uniformly with respect to the variables $t,\, x, \, m$, then \HE\, holds.
\end{itemize}
\end{example}

\begin{example}(Scalar interaction(s)) For some $N\geq 0$ and some measurable maps $\psi_1,\varphi_1 \, \cdots, \psi_N,\varphi_N : \rr^d \rightarrow \rr $, we consider the following non-linear SDE with coefficients $b: \rr_+ \times \rr^d \times \rr^N \rightarrow \rr^d$ and $\sigma: \rr_+ \times \rr^d \times \rr^N \rightarrow \rr^{d} \otimes \rr^q$:
\begin{align}
X^{\xi}_t & = \xi + \int_0^t b\Big(s, X^{\xi}_s, \widetilde{\E}\big[ \psi_1(\widetilde{X}^{\xi}_s)\big], \cdots, \widetilde{\E}\big[ \psi_N(\widetilde{X}^{\xi}_s)\big]\Big) ds \nonumber \\
& \quad \quad + \int_0^t \sigma\Big(s, X^{\xi}_s, \widetilde{\E}\big[ \varphi_1(\widetilde{X}^{\xi}_s)\big], \cdots, \widetilde{\E}\big[ \varphi_N(\widetilde{X}^{\xi}_s)\big]\Big) dW_s \label{ex:wp:1}
\end{align}
where the process $(\widetilde{X}^{\xi}_t)_{t\geq0}$ is a copy of $(X^{\xi}_t)_{t\geq0}$ defined on a copy $(\widetilde{\Omega}, \widetilde{\mathcal{F}}, \widetilde{\P})$ of the original probability space $(\Omega, \mathcal{F}, \P)$.
\begin{itemize}
\item If $b$ is a bounded and measurable function which is Lipschitz in space uniformly in time, if $(\psi_i)_{1\leq i\leq N}$ are bounded maps, if $\sigma$ is a bounded and measurable function continuous in $t$ and such that each entry of $\sigma(t,. ,z)$ is uniformly $\eta$-H\"older continuous, $\sigma(t, x,\cdot)$ is continuously differentiable with a bounded derivative uniformly in $t$ and $x$ and if $(\varphi_i)_{1\leq i \leq N}$ are $\eta$-H\"older continuous, then \HR\, holds.
\item If $b,\, \sigma$ are continuous w.r.t. the time variable and if in addition for any $t,\, x$, the maps $b(t,x,\cdot)$ and $a(t,x,\cdot)$ are two times continuously differentiable with bounded derivatives and $(\psi_i)_{1\leq i\leq N}$ are $\eta$-H\"older continuous, then \HRp\, holds.
\item If $a=\sigma\sigma^*$ is uniformly elliptic then \HE\, holds.
\end{itemize}
\end{example}

\begin{example}(Polynomials on the Wasserstein space)  We consider the following scalar non-linear SDE
\begin{eqnarray}
X^{\xi}_t &=& \xi + \int_0^t \prod_{i=1}^N \widetilde{\E}\Big[ \psi_i\big(t, X^{\xi}_s,\widetilde{X}^{\xi}_s\big)\Big] ds + \int_0^t \prod_{i=1}^N \widetilde{\E}\Big[ \varphi_i\big(t, X^{\xi}_s,\widetilde{X}^{\xi}_s\big)\Big] dW_s \label{ex:wp:3}
\end{eqnarray}

\noindent for some measurable maps $\psi_1,\varphi_1 \cdots, \psi_N,\varphi_N : \mathbb{R}_+ \times \mathbb{R}^2 \rightarrow \mathbb{R}$, the process $(\widetilde{X}^{\xi}_t)_{t\geq0}$ being a copy of $(X^{\xi}_t)_{t\geq0}$ defined on a copy $(\widetilde{\Omega}, \widetilde{\mathcal{F}}, \widetilde{\P})$ of the original probability space $(\Omega, \mathcal{F}, \P)$.

\begin{itemize}
\item If the functions $(\psi_i)_i$ are bounded and if the functions $(\varphi_i)_{1\leq i\leq N}$ are continuous in time and $\eta$-H\"older continuous in space uniformly in time, then \HR\, holds.
\item If in addition the functions $(\psi_i)_{1\leq i\leq N}$ are continuous in time and uniformly $\eta$-H\"older continuous in space uniformly in time, then \HRp\, holds.
\item If there exists $\lambda>0$ such that for any $i\in \left\{1, \cdots, N\right\}$ and any $(t, x, z) \in \mathbb{R}_+\times \mathbb{R}^2$, $\lambda < \varphi_i(t, x, z)$ then \HE\, holds.
\end{itemize}
\end{example}

\begin{remark}
A multi-dimensional version of the SDE \eqref{ex:wp:3} in the above example can be described as follows. We consider measurable maps $\varphi_i: \mathbb{R}_+ \times (\mathbb{R}^d)^2 \rightarrow \mathbb{R}^{q_{i-1}} \otimes \rr^{q_{i}}$, $i=1, \cdots, N$, for some positive integers $q_0, \cdots, q_N$ satisfying $q_0=d$ and $q_N=q$ and $\psi_{i, j}: \mathbb{R}_+ \times (\mathbb{R}^d)^2 \rightarrow \rr$, $i=1,\ldots,N$, $j=1, \cdots, d$. Set $b_j(t, x, m) := \prod_{i=1}^{N} \int_{\mathbb{R}^d} \psi_{i, j}(t, x, z) \, m(dz)$ and $a(t, x, m) := (\prod_{i=1}^N \int \varphi_i(t, x, z) \, m(dz))  (\prod_{i=1}^N \int \varphi_i(t, x, z) \, m(dz))^{*}$.
\begin{itemize}
\item If the maps $(\psi_{i, j})_{i, j}$ and $(\varphi_i)_i$ are bounded, if $\varphi_i(t, \cdot, \cdot)$ is uniformly $\eta$-H\"older continuous, and is continuous in time, then \HR\, holds.
\item If in addition each $(\psi_{i, j}(t,\cdot,\cdot))_{i,j}$ is $\eta$-H\"older continuous and continuous in time, then \HRp\, holds.
\item If each $a_i(t, x, m) := (\int \varphi_i(t, x, z) m(dz))  ( \int \varphi_i(t, x, z) m(dz))^{*}$, $i=1, \cdots, N$, is uniformly elliptic then $a(t, x, m)$ is also uniformly elliptic so that \HE\, holds.
\end{itemize}
\end{remark}

\section{Well-posedness of the martingale problem}\label{martingale:problem:sec}
In this section, we tackle the martingale problem associated to the SDE \eqref{SDE:MCKEAN}. We thus assume that  \HE\,and \HR\,are in force throughout this section. As mentioned in the introduction, the proof follows from a fixed point argument, more precisely from the Banach fixed point theorem applied to suitable map and complete metric space. The definition of the underlying complete metric space and map is introduced in subsection \ref{some:definitions:martingale:problem}. In subsection \ref{subsection:distance:between:two:successive:iterate}, we establish some important technical estimates for the difference of the transition densities associated to two successive iterates of the map. The proof of Theorem \ref{thm:martingale:problem} is tackled in subsection \ref{proof:martingale:problem}.

\subsection{Definition of the complete metric space and the map}\label{some:definitions:martingale:problem}
  Let $(s, \mu) \in \mathbb{R}_+ \times \mathcal{P}(\mathbb{R}^d)$ be the initial condition of the martingale problem of Definition \ref{def:martingale:problem}. Recalling that $(\mathcal{P}(\mathbb{R}^d), d_{{\rm TV}})$ is a complete metric space, for any fixed $T>s$, we consider the following set 
\begin{align*}
\mathscr{A}_{s, T, \mu} & = \left\{ P \in \mathcal{C}([s,T], \mathcal{P}(\mathbb{R}^d)): P(s)= \mu \right\}.
\end{align*}
Observe that $\mathscr{A}_{s, T, \mu}$ is a closed subspace of the complete metric $\mathcal{C}([s,T],\mathcal{P}(\rr^d))$ equipped with the uniform metric $d_{s, T}(P, P') = \sup_{t\in [s,T]} d_{{\rm TV}}(P(t),P'(t))$. Hence, $(\mathscr{A}_{s, T, \mu}, d_{s, T})$ is also a complete metric space.

We define the map $\mathscr{T}: \mathscr{A}_{s, T, \mu} \rightarrow \mathscr{A}_{s, T, \mu}$ which to a probability measure $Q \in \mathscr{A}_{s, T, \mu}$ associates the flow of probability measures $\mathscr{T}(Q) \in \mathscr{A}_{s, T, \mu}$ induced by the unique weak solution to the following SDE with dynamics
$$
\bar{X}^{s, \xi}_t = \xi + \int_s^t b(r, \bar X^{s, \xi}_r, Q(r))\, dr + \int_s^t \sigma(r, \bar X^{s, \xi}_r, Q(r)) \, dW_r, \quad t\in[s,T],
$$

\noindent that is, $\mathscr{T}(Q)(t) = [\bar{X}^{s, \xi}_t]$, $t\in [s,T]$. Let us note that under our current assumptions \HE\, and \HR, the martingale problem associated to the above SDE is well-posed so that weak existence and uniqueness holds and the map $\mathscr{T}$ is well-defined. Moreover, any fixed point of $\mathscr{T}$ is a solution to the martingale problem on the considered time interval $[s,T]$.

\subsection{On the distance between two successive iterations of the map}\label{subsection:distance:between:two:successive:iterate}
For fixed $ P_1, \, P_2 \in \mathscr{A}_{s, T, \mu}$, we now consider the two following sequences of SDEs $(\bar{X}^{1, (\ell)})_{\ell \geq 0}$ and $(\bar{X}^{2,(\ell)})_{\ell \geq 0}$ constructed by induction on $\ell$ as the unique weak solution to the following SDEs:
\begin{equation}
\label{approxim:sequence:}
   \bar{X}^{i, (\ell+1)}_t = \xi +  \int_s^t b(r, \bar{X}^{i, (\ell+1)}_r, [\bar{X}^{i, (\ell)}_r]) dr + \int_s^t \sigma(r, \bar{X}^{i, (\ell+1)}_r, [\bar{X}^{i, (\ell)}_r]) dW_r, \, t\in [s,T]
\end{equation}

\noindent where $ [\bar{X}^{i, (0)}_t] = P_i(t)$ for $i=1, 2$. Again, observe that under our current assumptions, for any $\ell \geq0$, the above SDE admits a unique weak solution so that $([\bar{X}^{i, (\ell)}_t])_{t\in [s,T]}$ is uniquely determined. We denote by $P^{(\ell)}_i = \mathscr{T}^{(\ell)}(P_i) = \mathscr{T}(\mathscr{T}^{(\ell-1)}(P_i)) = (P^{(\ell)}_i(t))_{t\in [s,T]} \in \mathscr{A}_{s, T, \mu}$, $\mathscr{T}^{(0)}(P_i)= P_i$, the time marginals induced by $(\bar{X}^{i, (\ell)}_t)_{t\in[s,T]}$. 

The density function of the random vector $\bar{X}^{i, (\ell+1)}_t$ given by the solution to the SDE \eqref{approxim:sequence:} at time $t$ and denoted by $p^{(\ell+1)}_i(\mu, s, t, .)$ satisfies 
$$
p^{(\ell+1)}_i(\mu, s, t, z) = \int p^{(\ell+1)}_i (\mu, s, t, x, z) \,  \mu(dx)
$$

\noindent where $z \mapsto p^{(\ell+1)}_i (\mu, r, t, x, z)$ is the density function of the random variable $\bar{X}^{i, (\ell+1), x, \mu}_t$, given by the unique weak solution  at time $t$ to the SDE obtained as the decoupling field associated to \eqref{approxim:sequence:}. Observe that the notation $p^{(\ell+1)}_i(\mu, s, t, x, z)$ (and also $p^{(\ell+1)}_i(\mu, s, t, z)$) makes sense since, by weak uniqueness, $[\bar{X}^{i, (\ell)}_t]$ only depends on $\xi$ through its law $\mu$. 

We then proceed using the first step of the parametrix expansion \eqref{first:step:parametrix:series:expansion} with $\widehat{p}$, $\mathcal{L}_{s, t}$, $\widehat{\mathcal{L}}_{s, t}$ and $\mH$ defined accordingly, namely, for any $i=1,\, 2$, for any fixed $(t, z) \in (0,T]\times \mathbb{R}^d$, we write
\begin{align}
p^{(\ell+1)}_i(\mu, s , t, x, z) & = \widehat{p}^{(\ell+1)}_i(\mu, s, t, x, z)\nonumber  \\
 & \quad + \int_s^t \int_{\mathbb{R}^d} p^{(\ell+1)}_i(\mu, s , r, x, y) (\mathcal{L}^{(\ell), i}_r - \widehat{\mathcal{L}}^{(\ell), i}_r) \widehat{p}^{(\ell+1)}_i(\mu, r, t, y, z)  \, dy \, dr \label{first:step:mckea:singer}
\end{align}

\noindent where for any $r\in [s, t)$
\begin{align*}
\widehat{p}^{(\ell+1)}_i(\mu, s, t, x, z) & := g\bigg(\int_s^t a(r', z, [\bar{X}^{i, (\ell)}_{r'}]) \, dr', z-x\bigg),\\
 \mathcal{L}^{(\ell), i}_r h(x)  & := \sum_{l=1}^{d} b_l(r, x,  [\bar{X}^{i, (\ell)}_r]) \partial_{x_l} h(x) +  \frac12 \sum_{k, l = 1}^{d} a_{k, l}(r, x, [\bar{X}^{i, (\ell)}_r]) \partial^2_{x_k, x_l}h(x),\\
 \widehat{\mathcal{L}}^{(\ell), i}_r h(x)  & := \frac12 \sum_{k, l = 1}^{d} a_{k, l}(r, z, [\bar{X}^{i, (\ell)}_r]) \partial^2_{x_k, x_l}h(x).
\end{align*}

Clearly, since the law appears as a parameter in the coefficients $b_i$ and $a_{i, j}$, from \HR\, and \HE\,, similarly to \eqref{bound:density:parametrix}, the following (uniform) Gaussian upper-estimate holds : there exist positive constants $C:=C(T, b, a, \eta, \lambda), \, c:=c(\lambda)$ such that for any integer $\ell$, any $(\mu, x, z) \in \pp \times (\mathbb{R}^d)^2$ and any $0\leq s < t \leq T$
\begin{equation}
|p^{(\ell+1)}_i(\mu, s, t, x, z)| \leq C g(c(t-s), z-x). \label{Gaussian:upper:estimate:pi}
\end{equation}
\noindent Also, from the previous expression of $\widehat{p}^{(\ell+1)}_i$, \HE\, and \eqref{standard}, the following estimates are satisfied: for any positive integer $n$, there exist some positive constants $C:=C(T,  b, a, \eta, \lambda), \, c:=c(\lambda)$ such that for any integer $\ell$, any $(\mu, x, z) \in \pp \times (\mathbb{R}^d)^2$ and any $0\leq s < t\leq T$ 
\begin{equation}
\label{estimate:deriv:tilde:p} 
|\partial^{n}_{x} \widehat{p}^{(\ell+1)}_i(\mu, s, t, x, z)| \leq \frac{C}{(t-s)^{\frac{n}{2}}} g(c(t-s), z-x).
\end{equation}

For the sake of clarity, we introduce the following notations. For any $r\in [s, t)$ and any integer $\ell$, we let
\begin{align*}
\delta(\mathcal{L}^{(\ell)}_r - \widehat{\mathcal{L}}^{(\ell)}_r) \widehat{p}^{(\ell+1)}(\mu, r, t, y, z) &:= (\mathcal{L}^{(\ell),1}_r- \widehat{\mathcal{L}}^{(\ell), 1}_r) \widehat{p}^{(\ell+1)}_1(\mu, r, t, y, z) -  (\mathcal{L}^{(\ell),2}_r- \widehat{\mathcal{L}}^{(\ell), 2}_r) \widehat{p}^{(\ell+1)}_2(\mu, r, t, y, z), \\
\delta \widehat{p}^{(\ell+1)}(\mu, r, t, x, z)& := (\widehat{p}^{(\ell+1)}_1 - \widehat{p}^{(\ell+1)}_2)(\mu, r, t, x, z),\\
\delta p^{(\ell+1)}(\mu, r, t, z) & := (p^{(\ell+1)}_1- p^{(\ell+1)}_2)(\mu, r, t, z). 
\end{align*}

 From the first order expansion \eqref{first:step:mckea:singer}, we thus see that $\delta p^{(\ell+1)}(\mu, s, t, x, z):= (p^{(\ell+1)}_1-p^{(\ell+1)}_2)(\mu, s, t, x, z)$ satisfies the following relation
\begin{align}
\delta p^{(\ell+1)}(\mu, s , t, x, z) & = \delta \widehat{p}^{(\ell+1)}(\mu, s, t, x, z) + \int_s^t \int_{\mathbb{R}^d} \delta p^{(\ell+1)}(\mu, s , r, x, y) (\mathcal{L}^{(\ell), 1}_r - \widehat{\mathcal{L}}^{(\ell), 1}_r) \widehat{p}^{(\ell+1)}_1(\mu, r, t, y, z) \, dy \, dr \nonumber \\
& \quad + \int_s^t \int_{\mathbb{R}^d} p^{(\ell+1)}_2(\mu, s, r, x, y) \, \delta(\mathcal{L}^{(\ell)}_r - \widehat{\mathcal{L}}^{(\ell)}_r) \,  \widehat{p}^{(\ell+1)}(\mu, r, t, y, z) \, dy \, dr. \label{diff:dens:iterate:mp1}
\end{align}

Let us note that for the second term of the above identity, one may use the uniform $\eta$-H\"older regularity of $ a(t, ., m)$, \eqref{estimate:deriv:tilde:p} and the space-time inequality \eqref{space:time:inequality}, for any integer $\ell$, it holds
\begin{align*}
 |(\mathcal{L}^{(\ell), 1}_r & - \widehat{\mathcal{L}}^{(\ell), 1}_r) \widehat{p}^{(\ell+1)}_1|(\mu, r, t, y, z)   \leq \frac{C}{(t-r)^{1-\frac{\eta}{2}}} g(c(t-r), z-y)
\end{align*}

\noindent so that
\begin{align}
\Big|  \int_{s}^t \int_{\mathbb{R}^d} & \delta p^{(\ell+1)}(\mu, s , r, x, y) (\mathcal{L}^{(\ell), 1}_r  - \widehat{\mathcal{L}}^{(\ell), 1}_r) \widehat{p}^{(\ell+1)}_1(\mu, r, t, y, z) \, dy dr\Big| \nonumber \\
& \leq C \int_s^t \int_{\mathbb{R}^d} |\delta p^{(\ell+1)}|(\mu, s , r, x, y)  \frac{1}{(t-r)^{1-\frac{\eta}{2}}} g(c(t-r), z-y) \, dy dr. \label{second:term:delta:parametrix:expansion}
\end{align}
 
 The first and third term appearing in the right-hand side of \eqref{diff:dens:iterate:mp1} will be handled thanks to the following technical estimates that will play an important role in order to prove the existence and uniqueness of a fixed point for the map $\mathscr{T}$ if $T>s$ is small enough. Its proof is postponed at the end of this section. 
 
\begin{lem}\label{lemma:technical:result}
For any $T>0$, there exist some positive constants $C:=C(T, \HR, \HE)$ and $c:=c(\lambda)$ such that for any $(\mu, x, z) \in  \pp \times \mathbb{R}^d$, any $0 \leq  s \leq r < t \leq T$, any positive $\ell$ and any $n \in \left\{0, 1, 2\right\}$
\begin{align}
\, |\partial^{n}_x \delta & \widehat{p}^{(\ell+1)}|(\mu, r, t, x, z) \nonumber \\
& \leq \frac{C}{(t-r)^{\frac{2+n}{2}}} \int_r^t \int_{(\mathbb{R}^d)^2}  (|y'-x'|^{\eta} \wedge 1) \,  |\delta p^{(\ell)}|(\mu, s , r', x', y') \, dy' \, \mu(dx') \, dr' \, g(c(t-r), z-x),  \label{bound:delta:frozen:kernel}
\end{align}
\begin{align}
| \delta(\mathcal{L}^{(\ell)}_r & - \widehat{\mathcal{L}}^{(\ell)}_r) \widehat{p}^{(\ell+1)}(\mu, r, t, x, z)| \nonumber \\
& \leq  C \left\{ \frac{1}{(t-r)^{\frac12}}\int_{(\mathbb{R}^d)^2} | \delta p^{(\ell)}|(\mu, s, r, x', y') \, dy' \mu(dx') \right. \nonumber\\
& \quad \left. +\frac{1}{(t-r)^{1-\frac{\eta}{4}}}  \int_{(\mathbb{R}^d)^2} (|y'-x'|^{\frac{\eta}{2}} \wedge 1)  |\delta p^{(\ell)}|(\mu, s, r, x',  y') \, dy'  \, \mu(dx') \right. \label{bound:delta:parametrix:kernel:positive:ell}\\ 
& \quad \left. +  \frac{1}{(t-r)^{2-\frac{\eta}{2}}} \int_r^t \int_{(\mathbb{R}^d)^2}  (|y'-x'|^{\eta} \wedge 1) \,  |\delta p^{(\ell)}|(\mu, s , r', x', y') \, dy' \, \mu(dx') \, dr'  \right\} \, g(c(t-r), z-x),\nonumber
\end{align}

\noindent while for $\ell=0$, 
\begin{align}
|\partial^{n}_x \delta \widetilde{p}^{(1)}|(\mu, r, t, x, z) \leq \frac{C}{(t-r)^{\frac{n}{2}}} d_{s, t}(P_1, P_2) g(c(t-r), z-x), \label{bound:delta:deriv:frozen:kernel:ell=0}
\end{align}
\noindent and
\begin{align}
| \delta(\mathcal{L}^{(0)}_r & - \widehat{\mathcal{L}}^{(0)}_r) \widehat{p}^{(1)}(\mu, r, t, x, z)| \leq  \frac{C}{(t-r)^{1-\frac{\eta}{2}}} d_{s, t}(P_1, P_2) \, g(c(t-r), z-x). \label{bound:delta:parametrix:kernel:ell=0}
\end{align}

\end{lem}

%
%

The previous lemma together with \eqref{second:term:delta:parametrix:expansion}, the Gaussian upper-estimate \eqref{Gaussian:upper:estimate:pi} as well as the semigroup property of Gaussian kernels yield for any positive integer $\ell$
\begin{align*}
& |\delta p^{(\ell+1)}|(\mu, s , t, x, z) \\
& \leq \frac{C}{t-s} \int_s^t \int_{(\mathbb{R}^d)^2}  (|y'-x'|^{\eta} \wedge 1) \,  |\delta p^{(\ell)}|(\mu, s , r, x', y') \, dy' \, \mu(dx') \, dr \, g(c(t-s), z-x) \\
& \quad + C  \int_s^t \int_{\mathbb{R}^d} |\delta p^{(\ell+1)}|(\mu, s , r, x, y) \frac{1}{(t-r)^{1-\frac{\eta}{2}}} g(c(t-r), z-y)  \, dy \, dr
 \\
& \quad + C \int_s^t \frac{1}{(t-r)^{1-\frac{\eta}{4}}}  \int_{(\mathbb{R}^d)^2}(|y'-x'|^{\frac{\eta}{2}} \wedge 1)  |\delta p^{(\ell)}|(\mu, s, r, x',  y') \, dy'  \, \mu(dx') \, dr \, g(c(t-s), z-x) \\
& \quad + C \int_s^t \frac{1}{(t-r)^{\frac12}}  \int_{(\mathbb{R}^d)^2}  |\delta p^{(\ell)}|(\mu, s, r, x',  y') \, dy'  \, \mu(dx') \, dr \, g(c(t-s), z-x) \\
& \quad + C  \int_s^t \frac{1}{(t-r)^{2-\frac{\eta}{2}}} \int_r^t \int_{(\mathbb{R}^d)^2}  (|y'-x'|^{\frac{\eta}{2}} \wedge 1) \,  |\delta p^{(\ell)}|(\mu, s , r', x', y') \, dy' \, \mu(dx') \, dr'  \, dr \, g(c(t-s), z-x)
\end{align*}
 
\noindent which in turn, by Fubini's theorem, implies
\begin{align}
& |\delta p^{(\ell+1)}|(\mu, s , t, x, z) \nonumber \\
& \quad \leq \frac{C}{t-s} \int_s^t \int_{(\mathbb{R}^d)^2}  (|y'-x'|^{\eta} \wedge 1) \,  |\delta p^{(\ell)}|(\mu, s , r, x', y') \, dy'  \mu(dx') \, dr \, g(c(t-s), z-x) \nonumber\\
& \quad +  C \int_s^t \int_{\mathbb{R}^d} |\delta p^{(\ell+1)}|(\mu, s , r, x, y) \frac{1}{(t-r)^{1-\frac{\eta}{2}}}  \, g(c(t-r), z-y)  \, dy \, dr \label{induction:relation:deltapm}
 \\
 & \quad + C \int_s^t \frac{1}{(t-r)^{\frac12}}  \int_{(\mathbb{R}^d)^2}  |\delta p^{(\ell)}|(\mu, s, r, x',  y') \, dy'  \mu(dx') \, dr \, g(c(t-s), z-x) \nonumber \\
& \quad + C \int_s^t \frac{1}{(t-r)^{1-\frac{\eta}{4}}}  \int_{(\mathbb{R}^d)^2}(|y'-x'|^{\frac{\eta}{2}} \wedge 1)  |\delta p^{(\ell)}|(\mu, s, r, x',  y') \, dy'   \mu(dx') \, dr \, g(c(t-s), z-x). \nonumber
\end{align}

For $\ell=0$, we similarly obtain 
\begin{align}
 |\delta p^{(1)}|(\mu, s , t, x, z) &  \leq C d_{s, t}(P_1, P_2) g(c(t-s), z-x)\nonumber\\
& \quad +  C \int_s^t \int_{\mathbb{R}^d} |\delta p^{(1)}|(\mu, s , r, x, y) \frac{1}{(t-r)^{1-\frac{\eta}{2}}}  \, g(c(t-r), z-y)  \, dy \, dr\nonumber
 \\
 & \quad + C \int_s^t \frac{1}{(t-r)^{1-\frac\eta2}} d_{s, r}(P_1, P_2) \, dr \, g(c(t-s), z-x) \nonumber\\
 & \leq C d_{s, t}(P_1, P_2) g(c(t-s), z-x) \nonumber \\
 & \quad + C \int_s^t \int_{\mathbb{R}^d} |\delta p^{(1)}|(\mu, s , r, x, y) \frac{1}{(t-r)^{1-\frac{\eta}{2}}}  \, g(c(t-r), z-y)  \, dy \, dr \label{induction:relation:deltapm:first:step}
\end{align}
\noindent up to a change of the constant $C$.

\subsection{Existence and uniqueness of a fixed point for the map $\mathscr{T}$ in small time.}\label{proof:martingale:problem}

Our aim here is to prove that the map $\mathscr{T}$ admits a unique fixed point on the space $\mathscr{A}_{s, T, \mu}$, if $T>s$ is sufficiently small. The induction relation \eqref{induction:relation:deltapm} established in the previous step will play a central role. Indeed, we will prove that for any $T >0$ there exists some non-decreasing positive function $t\mapsto K(t)$, which is independent of $\mu$, such that for all integer $\ell$
\begin{align}
\sup_{s < t \leq T+s} \int_{(\mathbb{R}^d)^2} (|y'-x'|^{\frac{\eta}{2}} \wedge 1) |\delta p^{(\ell+1)}|(\mu, s, t, x', y') \, dy' \mu(dx') \leq (K(T) T^{\frac{\eta}{4}})^{\ell+1} d_{s, s +T} (P_1, P_2) \label{weight:holder:norm:diff:pm}
\end{align}
\noindent and
\begin{align}
\sup_{s < t \leq T+s} \int_{(\mathbb{R}^d)^2} |\delta p^{(\ell+1)}|(\mu, s, t, x', y') \, dy' \mu(dx') \leq K(T) (K(T) T^{\frac{\eta}{4}})^{\ell} d_{s, s +T} (P_1, P_2). \label{tv:norm:diff:pm}
\end{align}

Before proving \eqref{weight:holder:norm:diff:pm} and \eqref{tv:norm:diff:pm}, let us observe that the latter estimate yields
\begin{align}
d_{s, s+T}(\mathscr{T}^{(\ell+1)}(P_1), \mathscr{T}^{(\ell+1)}(P_2)) &:=  \sup_{t\in [s, s+ T]} d_{{\rm TV}}(\mathscr{T}^{(\ell+1)}(P_1)(t), \mathscr{T}^{(\ell+1)}(P_2)(t))   \nonumber \\
& = \sup_{t\in (s, s+ T]} d_{{\rm TV}}(\mathscr{T}^{(\ell+1)}(P_1)(t), \mathscr{T}^{(\ell+1)}(P_2)(t)) \nonumber \\
& = \sup_{t\in (s, s+ T]} \sup_{|h|_\infty \leq 1} \int_{\mathbb{R}^d} h(z) (p^{ (\ell+1)}_1 - p^{(\ell+1)}_2)(\mu, s, t, z) \, dz \nonumber\\
&  \leq K (K T^{\frac{\eta}{4}})^\ell \, d_{s, s+T}(P_1, P_2). \label{final:ineq:diff:dens}
\end{align}

 Hence, taking $T$ sufficiently small so that $\sum_{\ell \geq 0}  (K(T)T^{\frac{\eta}{4}})^{\ell}  < \infty$, the Banach fixed point theorem guarantees that the map $\mathscr{T}$ admits a unique fixed point $P^{*} \in  \mathscr{A}_{s, s+ T, \mu}$. As a consequence, the martingale problem associated to the SDE \eqref{SDE:MCKEAN} is well-posed on the time interval $[s, s+T]$. The constant $K$ in \eqref{final:ineq:diff:dens} being independent of $\mu$, it is then standard to extend the well-posedness to the whole interval $[0,\infty)$.

We now prove \eqref{weight:holder:norm:diff:pm} and \eqref{tv:norm:diff:pm}. We may assume without loss of generality that $T\leq 1$. We proceed by induction on $\ell$. For the base case $\ell = 0$, we integrate both sides of the inequality \eqref{induction:relation:deltapm:first:step} with respect to the measure $dz \mu(dx)$ on $(\rr^d)^2$, after some simplifications, we obtain 
\begin{align*}
\int_{(\mathbb{R}^d)^2}  |\delta p^{(1)}|(\mu, s , t, x, z)  \, dz \mu(dx) & \leq C d_{s, t}(P_1, P_2) + C \int_s^t \frac{1}{(t-r)^{1-\frac{\eta}{2}}} \int_{(\rr^d)^2}  |\delta p^{(1)}|(\mu, s , r, x, z)  \, dz \mu(dx) \, dr
\end{align*}
\noindent so that
$$
\sup_{s < t \leq T+s}\int_{(\mathbb{R}^d)^2}  |\delta p^{(1)}|(\mu, s , t, x, z)  \, dz \mu(dx) \leq K(T) d_{s, T+s}(P_1,P_2)
$$
\noindent where $[0,T] \ni t\mapsto K(t)$ is some positive non-decreasing function. This concludes the proof of \eqref{tv:norm:diff:pm} for $\ell=0$.

 Similarly, we multiply by $|z-x|^{\frac{\eta}{2}} \wedge 1$ and then integrate with respect to $dz  \mu(dx)$ on $(\mathbb{R}^d)^2$ both sides of the inequality \eqref{induction:relation:deltapm:first:step}, using the space-time inequality \eqref{space:time:inequality} as well as the previous estimate, we obtain for $t \in (s, T+s]$
 \begin{align*}
\int_{(\rr^d)^2}  |\delta p^{(1)}|(\mu, s , t, x, z) (|z-x|^{\frac{\eta}{2}}\wedge 1) \, dz \mu(dx)& \leq C (t-s)^{\frac{\eta}{4}} d_{s, t}(P_1, P_2)\\
& \quad + C  \int_s^t \frac{1}{(t-r)^{1-\frac{\eta}{2}}} \int_{(\rr^d)^2} |\delta p^{(1)}|(\mu, s , r, x, z) \, dz \, \mu(dx) \, dr \\
&   \leq C (t-s)^{\frac{\eta}{4}} d_{s, t}(P_1, P_2)  + C (t-s)^{\frac{\eta}{2}} K(T) d_{s, T+s}(P_1, P_2) \\
& \leq (K(T) T^{\frac{\eta}{4}}) d_{s, T+s}(P_1,P_2)
\end{align*}

\noindent so that, taking the supremum over $t \in (s, T+s]$ on the left-hand side of the previous inequality allows to conclude the proof of \eqref{weight:holder:norm:diff:pm} for $\ell=0$. Assuming now that both \eqref{weight:holder:norm:diff:pm} and \eqref{tv:norm:diff:pm} hold at step $\ell \geq 1$, we proceed similarly, namely we integrate with respect to $dz \mu(dx)$ on $(\mathbb{R}^d)^2$ both sides of the inequality \eqref{induction:relation:deltapm}
\begin{align*}
\int_{(\rr^d)^2}  |\delta p^{(\ell+1)}|(\mu, s , t, x, z)  \, dz \mu(dx) & \leq  \frac{C}{t-s} \int_s^t \int_{(\rr^d)^2}  (|z-x|^{\eta} \wedge 1) \,  |\delta p^{(\ell)}|(\mu, s , r, x, z) \, dz  \mu(dx) \, dr \\
& \quad +  C \int_s^t \frac{1}{(t-r)^{1-\frac{\eta}{2}}}  \int_{(\mathbb{R}^d)^2}  |\delta p^{(\ell+1)}|(\mu, s , r, x, z) \, dz\mu(dx) \, dr \\
& \quad + C \int_s^t \frac{1}{(t-r)^{1-\frac{\eta}{4}}}  \int_{(\rr^d)^2}  |\delta p^{(\ell)}|(\mu, s, r, x,  z) \, dz   \mu(dx) \, dr \\
& \leq C (K(T) T^{\frac{\eta}{4}})^\ell d_{s, s+T}(P_1, P_2) \\
& \quad + C \int_s^t \frac{1}{(t-r)^{1-\frac{\eta}{2}}}  \int_{(\mathbb{R}^d)^2}  |\delta p^{(\ell+1)}|(\mu, s , r, x, z) \, dz\mu(dx) \, dr 
\end{align*} 
\noindent which in turn implies \eqref{tv:norm:diff:pm}. Similarly, we now multiply  both sides of \eqref{induction:relation:deltapm} by $|z-x|^{\eta} \wedge 1$ and then integrate on $(\mathbb{R}^d)^2$ with respect to $dz \mu(dx)$. From the space-time inequality \eqref{space:time:inequality}, we obtain
\begin{align*}
\int_{(\rr^d)^2} & |\delta p^{(\ell+1)}|(\mu, s , t, x, z) (|z-x|^{\frac{\eta}{2}}\wedge 1) \, dz  \mu(dx)\\
&   \leq \frac{C}{(t-s)^{1-\frac{\eta}{4}}} \int_s^t \int_{(\rr^d)^2}  (|z-x|^{\frac{\eta}{2}} \wedge 1) \,  |\delta p^{(\ell)}|(\mu, s , r, x, z) \, dz  \mu(dx) \, dr \\
& \quad + C   \int_s^t \frac{1}{(t-r)^{1-\frac{\eta}{2}}} \int_{(\mathbb{R}^d)^2} |\delta p^{(\ell+1)}|(\mu, s , r, x, z) \, dz  \mu(dx) \, dr \\
& \quad + C T^{\frac{\eta}{4}} \int_s^t \frac{1}{(t-r)^{\frac12}} \int_{(\rr^d)^2}  |\delta p^{(\ell)}|(\mu, s , r, x, z) \, dz  \mu(dx) \, dr\\
& \quad + C T^{\frac{\eta}{4}} \int_s^t \frac{1}{(t-r)^{1-\frac{\eta}{4}}}  \int_{(\rr^d)^2}(|z-x|^{\frac{\eta}{2}} \wedge 1)  |\delta p^{(\ell)}|(\mu, s, r, x,  z) \, dz  \mu(dx) \, dr
\end{align*}

\noindent which in turn, using the induction hypothesis, yields
\begin{align*}
\int_{(\mathbb{R}^d)^2}  & |\delta p^{(\ell+1)}|(\mu, s , t, x, z) (|z-x|^{\frac{\eta}{2}}\wedge 1) \, dz  \mu(dx)  \\
& \leq C T^{\frac{\eta}{4}} (K(T) T^{\frac{\eta}{4}})^{\ell} d_{s, T+s}(P_1,P_2)  + C K(T) T^{\frac{\eta}{2}}  (K(T) T^{\frac{\eta}{4}})^{\ell} d_{s, T+s}(P_1,P_2)  \\
  & \quad + C T^{\frac{\eta}{4}+\frac12} (K(T) T^{\frac{\eta}{4}})^{\ell-1} d_{s, T+s}(P_1,P_2) + CT^{\frac\eta2} (K(T) T^{\frac\eta4})^{\ell}  d_{s, T+s}(P_1,P_2)\\
& \leq  (K(T) T^{\frac{\eta}{4}})^{\ell+1} d_{s, s+T}(P_1, P_2)
\end{align*}
\noindent which allows to conclude the proof of \eqref{weight:holder:norm:diff:pm}. The proof is now complete.\qed

\subsection{Proof of Lemma \ref{lemma:technical:result}}

\noindent \emph{Step 0:  Preliminaries.} Let us first give a general control on the difference of the diffusion coefficients taken along the two iterates. We claim that, for any $k, l$ in $\{1,\ldots,d\}^2$ and $y,z \in \rr^d$, it holds, 
\begin{eqnarray}\label{eq:INTER:PMG:DIFFdesA}
&&\left| [a_{k, l}(r, y, [\bar{X}^{1, (\ell)}_r]) - a_{k, l}(r, y, [\bar{X}^{2, (\ell)}_r]) ] - [a_{k, l}(r, z, [\bar{X}^{1, (\ell)}_r]) - a_{k, l}(r, z, [\bar{X}^{2, (\ell)}_r]) ] \right| \\
&\le & \sup_{(k, l), (t, m)}[\frac{\delta a_{k,l}}{\delta m}(t, ., m)(.)]_{H} \int_{(\mathbb{R}^d)^2} (|y'-x'|^{\eta} \wedge |z-y|^{\eta} \wedge 1) |\delta p^{(\ell)}|(\mu, s, r, x',  y') \, dy'  \, \mu(dx'),\notag
\end{eqnarray}
and 
\begin{eqnarray}\label{eq:INTER:PMG:DIFFdesA0}
&&\left| [a_{k, l}(r, y, [\bar{X}^{1, (0)}_r]) - a_{k, l}(r, y, [\bar{X}^{2, (0)}_r]) ] - [a_{k, l}(r, z, [\bar{X}^{1, (\ell)}_r]) - a_{k, l}(r, z, [\bar{X}^{2, (\ell)}_r]) ] \right| \\
&\le & \sup_{(k, l), (t, m)}[\frac{\delta a_{k,l}}{\delta m}(t, ., m)(.)]_{H} |z-y|^{\eta} d_{s,r}(P_1,P_2).\notag
\end{eqnarray}
Indeed, let $k,l$ in $\{1,\ldots,d\}^2$ and $y,z \in \rr^d$, using the fact that the map $a$ admits a linear functional derivative we have
\begin{eqnarray*}
&&\left\{ [a_{k, l}(r, y, [\bar{X}^{1, (\ell)}_r]) - a_{k, l}(r, y, [\bar{X}^{2, (\ell)}_r]) ] - [a_{k, l}(r, z, [\bar{X}^{1, (\ell)}_r]) - a_{k, l}(r, z, [\bar{X}^{2, (\ell)}_r]) ] \right\} \\
&=&\int_0^1 \int_{\mathbb{R}^d} \left\{ \frac{\delta a_{k, l}}{\delta m}(r, y, \lambda [\bar{X}^{1, (\ell)}_r] + (1-\lambda) [\bar{X}^{2, (\ell)}_r] )(y') -  \frac{\delta a_{k, l}}{\delta m}(r, z, \lambda [\bar{X}^{1, (\ell)}_r] + (1-\lambda)[\bar{X}^{2, (\ell)}_r] )(y') \right\} \\
&& \quad \quad \times (P_1^{(\ell)}-P^{(\ell)}_2)(r)(dy') d\lambda \\
& = & \int_0^1 \int_{\mathbb{R}^d} \textbf{1}_{\left\{y\neq z\right\}} \left\{ \frac{\delta a_{k, l}}{\delta m}(r, y, \lambda [\bar{X}^{1, (\ell)}_r] + (1-\lambda) [\bar{X}^{2, (\ell)}_r] )(y') -  \frac{\delta a_{k, l}}{\delta m}(r, z, \lambda [\bar{X}^{1, (\ell)}_r] + (1-\lambda)[\bar{X}^{2, (\ell)}_r] )(y') \right\} \\
&& \quad \quad \times (P_1^{(\ell)}-P^{(\ell)}_2)(r)(dy') d\lambda
\end{eqnarray*}
If $\ell=0$, using the fact that $\mathbb{R}^d \ni y' \mapsto \big[ [\delta a_{k, l}/\delta m](r, y, m)(y')- [\delta a_{k, l}/\delta m](r, z, m)(y')\big]/|z-y|^\eta$, $y\neq z$, is a measurable and bounded function (uniformly with respect to the other variables), directly yields \eqref{eq:INTER:PMG:DIFFdesA0}. If now $\ell \ge 1$, using the identity $(P_1^{(\ell)}-P^{(\ell)}_2)(r)(dy') = \int_{\mathbb{R}^d}  \delta p^{(\ell)}(\mu, s, r, x', y') \, dy' \mu(dx') $ and the fact that $\int_{\mathbb{R}^d}  \delta p^{(\ell)}(\mu, s, r, x', y') \, dy' = 0$ since each $p^{(\ell)}_i$, $i=1,2$, is a density, we obtain
\begin{eqnarray*}
&&\left\{ [a_{k, l}(r, y, [\bar{X}^{1, (\ell)}_r]) - a_{k, l}(r, y, [\bar{X}^{2, (\ell)}_r]) ] - [a_{k, l}(r, z, [\bar{X}^{1, (\ell)}_r]) - a_{k, l}(r, z, [\bar{X}^{2, (\ell)}_r]) ] \right\} \\
& = & \frac12 \sum_{k, l=1}^d \int_0^1 \int_{(\mathbb{R}^d)^2} \textbf{1}_{\left\{y\neq z\right\}}  \left\{ \frac{\delta a_{k, l}}{\delta m}(r, y, \lambda [\bar{X}^{1, (\ell)}_r] + (1-\lambda) [\bar{X}^{2, (\ell)}_r] )(y') -  \frac{\delta a_{k, l}}{\delta m}(r, z, \lambda [\bar{X}^{1, (\ell)}_r] + (1-\lambda) [\bar{X}^{2, (\ell)}_r] )(y') \right. \\
&&  \left. - \Big[\frac{\delta a_{k, l}}{\delta m}(r, y,  \lambda [\bar{X}^{1, (\ell)}_r] + (1-\lambda) [\bar{X}^{2, (\ell)}_r] )(x') -  \frac{\delta a_{k, l}}{\delta m}(r, z, \lambda [\bar{X}^{1, (\ell)}_r] + (1-\lambda) [\bar{X}^{2, (\ell)}_r] )(x') \Big] \right\} \\
 && \quad \quad \times \delta p^{(\ell)}(\mu, s, r, x',  y') \, dy'  \, \mu(dx') d\lambda,
\end{eqnarray*}

\noindent so that \eqref{eq:INTER:PMG:DIFFdesA} follows from the uniform $\eta$-H\"older regularity of the map $[\delta a/\delta m](t,. ,. )$. Similar computations give
\begin{eqnarray}\label{eq:INTER:PMG:DIFFdesAmesonly}
&&\left| [a_{k, l}(r, y, [\bar{X}^{1, (\ell)}_r]) - a_{k, l}(r, y, [\bar{X}^{2, (\ell)}_r]) ]\right| \\
&\le & \sup_{(k, l), (t, x, m)}[\frac{\delta a_{k,l}}{\delta m}(t, x, m)(.)]_{H} \int_{(\mathbb{R}^d)^2} (|y'-x'|^{\eta} \wedge 1) \delta p^{(\ell)}(\mu, s, r, x',  y') \, dy'  \, \mu(dx')\notag
\end{eqnarray}
and 
\begin{eqnarray}\label{eq:INTER:PMG:DIFFdesA0mesonly}
\left| [a_{k, l}(r, y, [\bar{X}^{1, (0)}_r]) - a_{k, l}(r, y, [\bar{X}^{2, (0)}_r]) ]\right| \le 2 \sup_{(k, l), (t, x, m)} |\frac{\delta a_{k,l}}{\delta m}(t, x, m)(.)|_{\infty} d_{s,r}(P_1,P_2).\\\notag
\end{eqnarray}

\noindent \emph{Step 1: proof of \eqref{bound:delta:frozen:kernel} and \eqref{bound:delta:deriv:frozen:kernel:ell=0}.} From the very definition of $\partial_x^n  \widehat{p}^{(\ell+1)}|(\mu, r, t, x, z)$, we obtain the decomposition
\begin{align}
\partial^{n}_x \delta \widehat{p}^{(\ell+1)}(\mu, r, t, x, z) & = \Bigg[ H_n\bigg(\int_r^t a(r', z, [\bar{X}^{1, (\ell)}_{r'}]) \, dr', z-x\bigg)\label{first_term_lemme-PMG}\\
&\quad \qquad  -  H_n\bigg(\int_s^t a(r', z, [\bar{X}^{2, (\ell)}_{r'}]) \, dr', z-x\bigg) \Bigg] \widehat{p}^{(\ell+1)}_1(\mu, s, t, x, z) \notag\\
& \quad   +  H_n\bigg(\int_s^t a(r', z, [\bar{X}^{2, (\ell)}_{r'}]) \, dr', z-x\bigg) \delta \widehat{p}^{(\ell+1)}(\mu, s, t, x, z), \notag
\end{align}
\noindent where we introduced the notation $H_0(\Sigma, x)=1$, $H_1(\Sigma, x)=(H^{i}_1(\Sigma, x))_{1\leq i \leq d}$ and $H_2(\Sigma, x)=(H^{i, j}_2(\Sigma, x))_{1\leq i, j \leq d}$ for a symmetric and positive matrix $\Sigma$ and $x\in \mathbb{R}^d$. From the very definition of the Hermite polynomials, the mean-value theorem and \eqref{eq:INTER:PMG:DIFFdesAmesonly}, for any positive integer $\ell$, we obtain
\begin{align*}
&\Bigg|  H_n\bigg(\int_r^t a(r', z, [\bar{X}^{1, (\ell)}_{r'}]) \, dr', z-x\bigg)  -  H_n\bigg(\int_r^t a(r', z, [\bar{X}^{2, (\ell)}_{r'}]) \, dr', z-x\bigg) \Bigg| \\
  \leq &C \left\{ \frac{|z-x|}{(t-r)^2} \textbf{1}_{n=1} + \bigg (\frac{|z-x|^2}{(t-r)^{3}} + \frac{1}{(t-r)^2} \bigg)\textbf{1}_{n=2} \right\}  \int_r^t \int_{(\mathbb{R}^d)^2}  (|y'-x'|^{\eta} \wedge 1) \,  |\delta p^{(\ell)}|(\mu, s , r', x', y') \, dy' \, \mu(dx')\, dr' 
\end{align*}
\noindent while for $\ell=0$, \eqref{eq:INTER:PMG:DIFFdesA0mesonly} directly gives
\begin{align*}
\Bigg| & H_n\bigg(\int_r^t a(r', z, [\bar{X}^{1, (0)}_{r'}]) \, dr', z-x\bigg)  -  H_n\bigg(\int_r^t a(r', z, [\bar{X}^{2, (0)}_{r'}]) \, dr', z-x\bigg) \Bigg| \\
& \quad \leq C \left\{ \frac{|z-x|}{(t-r)^2} \textbf{1}_{n=1} + \bigg (\frac{|z-x|^2}{(t-r)^{3}} + \frac{1}{(t-r)^2} \bigg)\textbf{1}_{n=2} \right\}\int_r^t d_{s, r'}(P_1, P_2) \, dr'.
\end{align*}

Hence, using the space-time inequality \eqref{space:time:inequality} and \eqref{estimate:deriv:tilde:p}, for any positive integer $\ell$, we deduce that the first term in the right hand side of \eqref{first_term_lemme-PMG} can be bounded as follows 
\begin{align*}
\Bigg| \Bigg[ & H_n\bigg(\int_r^t a(r', z, [\bar{X}^{1, (\ell)}_{r'}]) \, dr', z-x\bigg) -  H_n\bigg(\int_r^t a(r', z, [\bar{X}^{2, (\ell)}_{r'}]) \, dr', z-x\bigg) \Bigg] \widehat{p}^{(\ell+1)}_1(\mu, r, t, x, z) \Bigg| \\
& \leq  \frac{C}{(t-r)^{\frac{2+n}{2}}} \int_r^t \int_{(\mathbb{R}^d)^2}  (|y'-x'|^{\eta} \wedge 1) \,  |\delta p^{(\ell)}|(\mu, s , r', x', y') \, dy' \, \mu(dx') \, dr' \, g(c(t-r), z-x)
\end{align*}
\noindent while for $\ell =0$,
\begin{align*}
\Bigg| \Bigg[ & H_n\bigg(\int_r^t a(r', z, [\bar{X}^{1, (0)}_{r'}]) \, dr', z-x\bigg) -  H_n\bigg(\int_r^t a(r', z, [\bar{X}^{2, (0)}_{r'}]) \, dr', z-x\bigg) \Bigg] \widehat{p}^{(1)}_1(\mu, r, t, x, z) \Bigg| \\
& \leq  \frac{C}{(t-r)^{\frac{2+n}{2}}}\int_r^t d_{s, r'}(P_1, P_2) \, dr'  \, g(c(t-r), z-x)  \\
& \leq \frac{C}{(t-r)^{\frac{n}{2}}} d_{s, t}(P_1, P_2) \, g(c(t-r), z-x)
\end{align*}

\noindent where we used the fact that $[s, T]\ni r \mapsto d_{s, r}(P_1, P_2)$ is non-decreasing for the last inequality. We now consider the second term in the right hand side of \eqref{first_term_lemme-PMG}. First, for any integer $\ell$, by the mean-value theorem, we get
\begin{align}
\delta \widehat{p}^{(\ell+1)}(\mu, r, t, x, z)  & = \sum_{i ,j=1}^d \int_0^1   \frac12 (H^{i, j}_2 . g)\bigg(\int_r^t (\lambda a(r', z, [\bar{X}^{1, (\ell)}_{r'}]) + (1-\lambda) a(r', z, [\bar{X}^{2, (\ell)}_{r'}]) dr', z-x\bigg) \label{first:comparaison:approxy:process} \\
& \quad\quad  \times \int_r^t (a_{i, j}(r', z, [\bar{X}^{1, (\ell)}_{r'}]) - a_{i, j}(r', z, [\bar{X}^{2, (\ell)}_{r'}])) dr' \, d\lambda \nonumber.
\end{align}

Hence, we deduce from \eqref{eq:INTER:PMG:DIFFdesAmesonly} and \eqref{standard} that
$$
|\delta \widehat{p}^{(\ell+1)}|(\mu, r, t, x, z) \leq \frac{C}{t-r} \int_r^t \int_{(\mathbb{R}^d)^2}  (|y'-x'|^{\eta} \wedge 1) \,  |\delta p^{(\ell)}|(\mu, s , r', x', y') \, dy' \, \mu(dx') \, dr' \, g(c(t-r), z-x)
$$
\noindent for $\ell\geq1$ while for $\ell=0$, using \eqref{eq:INTER:PMG:DIFFdesA0mesonly}, the previous bound writes
$$
|\delta \widehat{p}^{(1)}|(\mu, r, t, x, z) \leq C d_{s, t}(P_1, P_2) g(c(t-r), z-x).
$$

The two previous estimates together with \eqref{standard} thus imply
\begin{align*}
\Big| & H_n\bigg(\int_r^t  a(r', z, [\bar{X}^{2, (\ell)}_{r'}]) \, dr', z-x\bigg)  \delta \widehat{p}^{(\ell+1)}(\mu, r, t, x, z) \Big| \\
&  \leq \frac{C}{(t-r)^{\frac{n+2}{2}}}  \int_r^t \int_{(\mathbb{R}^d)^2}  (|y'-x'|^{\eta} \wedge 1) \,  |\delta p^{(\ell)}|(\mu, s , r', x', y') \, dy' \, \mu(dx') \, dr' \, g(c(t-r), z-x), \, 
\end{align*}
\noindent for $\ell\geq1$ and 
\begin{align*}
\Big| H_n\bigg(\int_r^t a(r', z, [\bar{X}^{2, (0)}_{r'}]) \, dr', z-x\bigg) \delta \widehat{p}^{(1)}(\mu, r, t, x, z) \Big| \leq \frac{C}{(t-r)^{\frac{n}{2}}} d_{s, t}(P_1, P_2) \, g(c(t-r), z-x).
\end{align*}
This completes the proof of \eqref{bound:delta:frozen:kernel} and \eqref{bound:delta:deriv:frozen:kernel:ell=0}.\\

\noindent \emph{Step 2: proof of \eqref{bound:delta:parametrix:kernel:positive:ell} and \eqref{bound:delta:parametrix:kernel:ell=0}.} In order to establish both estimates, we make use of the following decomposition
\begin{eqnarray}\label{decompdiffgenPMG}
\delta(\mathcal{L}^{(\ell)}_r - \widehat{\mathcal{L}}^{(\ell)}_r) \widehat{p}^{(\ell+1)}(\mu, r, t, y, z)&=&[(\mathcal{L}^{(\ell),1}_r- \widehat{\mathcal{L}}^{(\ell), 1}_r) - (\mathcal{L}^{(\ell),2}_r- \widehat{\mathcal{L}}^{(\ell), 2}_r) ] \widehat{p}^{(\ell+1)}_1(\mu, r, t, y, z)\notag\\
&&+(\mathcal{L}^{(\ell),2}_r- \widehat{\mathcal{L}}^{(\ell), 2}_r)  \delta \widehat{p}^{(\ell+1)}(\mu, r, t, y, z)\notag\\
&=:& {\rm I}^{(\ell)} + {\rm II}^{(\ell)}
\end{eqnarray}
with ${\rm I}^{(\ell)} =: {\rm I}^{(\ell)}_1+{\rm I}^{(\ell)}_2$ where
\begin{align*}
{\rm I}^{(\ell)}_1&  =  \sum_{l=1}^d [b_{l}(r, y, [\bar{X}^{1, (\ell)}_r]) - b_{l}(r, y, [\bar{X}^{2, (\ell)}_r])] \partial_{x_l}\widehat{p}^{(\ell+1)}_1(\mu, r, t, y, z),\\
{\rm I}^{(\ell)}_2 &=  \frac12 \sum_{k, l=1}^d \left\{ [a_{k, l}(r, y, [\bar{X}^{1, (\ell)}_r]) - a_{k, l}(r, y, [\bar{X}^{2, (\ell)}_r]) ] - [a_{k, l}(r, z, [\bar{X}^{1, (\ell)}_r]) - a_{k, l}(r, z, [\bar{X}^{2, (\ell)}_r]) ] \right\} \\
& \quad \times \partial^2_{x_k, x_l}\widehat{p}^{(\ell+1)}_1(\mu, r, t, y, z).
\end{align*}
We first handle the term ${\rm I}^{(\ell)}$ in \eqref{decompdiffgenPMG}. From \HR (i) and \eqref{estimate:deriv:tilde:p}, for any positive integer $\ell$, one has
 $$
 |{\rm I}^{(\ell)}_1| \leq \frac{C}{(t-r)^{\frac12}}\int_{(\rr^d)^2} | \delta p^{(\ell)}|(\mu, s, r, x', y') \, dy' \mu(dx') \, g(c(t-r), z-y),
 $$
while for $\ell=0$, one obtains
  $$
 |{\rm I}^{(0)}_1| \leq \frac{C}{(t-r)^{\frac12}} d_{s, r}(P_1,P_2) \, g(c(t-r), z-y) \leq  \frac{C}{(t-r)^{\frac12}} d_{s, t}(P_1,P_2) \, g(c(t-r), z-y)
 $$
where we used the fact that $[s, T] \ni r\mapsto d_{s, r}(P_1, P_2)$ is non-decreasing for the last inequality. Then, from \eqref{eq:INTER:PMG:DIFFdesA}, for any positive integer $\ell$,
\begin{align*}
| {\rm I}^{(\ell)}_2 | \leq \frac{C}{t-r} \int_{(\mathbb{R}^d)^2}(|y'-x'|^{\eta} \wedge |z-y|^{\eta} \wedge 1)  |\delta p^{(\ell)}|(\mu, s, r, x',  y') \, dy'  \, \mu(dx') \, g(c(t-r), z-y).
\end{align*}
Now, we may break the integral appearing in the right-hand side of the above inequality into two parts by dividing the domain of $dy'$ integration into the two disjoint domains: $\left\{y' \in \rr^d :  |y'-x'| \geq |z-y| \right\}$ and $\left\{ y' \in \rr^d: |y'-x'| < |z-y| \right\}$. On the first domain, we have $(|y'-x'|^{\eta} \wedge |z-y|^{\eta} \wedge 1) \le |z-y|^{\eta}\leq |y'-x'|^{\eta/2} |z-y|^{\eta/2}$ while on the second one $(|y'-x'|^{\eta} \wedge |z-y|^{\eta} \wedge 1)\le |y'-x'|^{\eta}\leq  |y'-x'|^{\eta/2} |z-y|^{\eta/2}$. Therefore
\begin{align*}
| {\rm I}^{(\ell)}_2 |&   \leq C \frac{|z-y|^{\frac{\eta}{2}}}{t-r} \int_{(\rr^d)^2}(|y'-x'|^{\frac{\eta}{2}} \wedge 1)  |\delta p^{(\ell)}|(\mu, s, r, x',  y') \, dy'  \, \mu(dx') \, g(c(t-r), z-y)\\
&\leq \frac{C}{(t-r)^{1-\frac{\eta}{4}}}  \int_{(\rr^d)^2}(|y'-x'|^{\frac{\eta}{2}} \wedge 1)  |\delta p^{(\ell)}|(\mu, s, r, x',  y') \, dy'  \, \mu(dx') \, g(c(t-r), z-y)
\end{align*}

\noindent for any positive integer $\ell$, where we used the space-time inequality \eqref{space:time:inequality} for the last inequality. For $\ell=0$, \eqref{eq:INTER:PMG:DIFFdesA0mesonly} yields
$$
|{\rm I}^{(0)}_2 | \leq C \frac{|z-y|^\eta}{t-r} d_{s,r}(P_1, P_2) \, g(c(t-r), z-y) \leq  \frac{C}{(t-r)^{1-\frac{\eta}{2}}} d_{s, t}(P_1, P_2) \, g(c(t-r), z-y)
$$
\noindent where we used the space-time inequality \eqref{space:time:inequality} and the fact that $[s, T] \ni r\mapsto d_{s, r}(P_1, P_2)$ is non-decreasing for the last inequality.\\

We now handle the term {\rm II} in \eqref{decompdiffgenPMG}. To do so, we make use of the uniform $\eta$-H\"older regularity of $ a_{k, l}(t, ., m)$, the estimate \eqref{bound:delta:frozen:kernel} for $n=1, 2$ and the space-time inequality \eqref{space:time:inequality} to obtain
$$
|{\rm II}^{(\ell)}| \leq \frac{C}{(t-r)^{2-\frac{\eta}{2}}} \int_r^t \int_{(\rr^d)^2}  (|y'-x'|^{\eta} \wedge 1) \,  |\delta p^{(\ell)}|(\mu, s , r', x', y') \, dy' \, \mu(dx') \, dr' \, g(c(t-r), z-y)
$$
\noindent for any positive integer $\ell$ while for $\ell=0$, from \eqref{bound:delta:deriv:frozen:kernel:ell=0} and the space-time inequality \eqref{space:time:inequality}, we get
$$
|{\rm II}^{(0)}| \leq \frac{C}{(t-r)^{1-\frac{\eta}{2}}} d_{s, t}(P_1, P_2) \, g(c(t-r), z-y).
$$

Gathering the above estimates on $ {\rm I} = {\rm I}_1+{\rm I}_2$ and  ${\rm II}$ concludes the proof of \eqref{bound:delta:parametrix:kernel:positive:ell} and \eqref{bound:delta:parametrix:kernel:ell=0}.\qed

\section{Existence and regularity properties of the transition density}\label{existence:regularity:transition:density}

This section is dedicated to the proof of Theorem \ref{derivative:density:sol:mckean:and:decoupling}. Hence, throughout this section, we assume that \HE\, and \HRp \, are in force.

\subsection{Strategy of proof}\label{strategy:proof:main:thm}

Our strategy is based on a Picard approximation scheme for the transition density $p(\mu, s, t, x, z)$ and sharp uniform estimates on its derivatives from which we can extract a uniformly convergent subsequence.
 
\medskip
\noindent \emph{Step 1: Construction of an approximation sequence and related estimates}

\medskip

For a given initial condition $(s,\mu) \in \mathbb{R}_+ \times \pp$ and a probability measure $\nu \in \pp$, $\nu \neq \mu$, we let $\P^{(0)} = (\P^{(0)}(t))_{t\geq s}$ be the probability measure on $\mathcal{C}([s,\infty), \rr^d)$, endowed with its canonical filtration, satisfying $\P^{(0)}(t)=\nu$, $t\geq s$, and we consider the following recursive sequence of probability measures $\left\{ \P^{(m)}; m\geq 0 \right\}$, with time marginals $(\P^{(m)}(t))_{t\geq s}$, where, $\P^{(m)}$ being given, $\P^{(m+1)}$ is the unique solution to the following martingale problem
\begin{itemize}
\item[(i)] $\P^{(m+1)}(y(r) \in \Gamma; 0\leq r \leq s)  =  \mu(\Gamma)$, for all $\Gamma \in \mathcal{B}(\rr^d)$.
\item[(ii)] For all $f\in \mathcal{C}^2_b(\rr^d)$, 
$$
f(y_t)  - f(y_s) - \int_s^t \left\{ \sum_{i=1}^d b_i(r, y_r, \P^{(m)}(r)) \partial_i f(y_r) + \sum_{i, j =1}^d \frac12 a_{i, j}(r, y_r, \P^{(m)}(r)) \partial^2_{i, j} f(y_r) \right\} \, dr
$$

\noindent is a continuous square-integrable martingale under $\mathbb{P}^{(m+1)}$.
\end{itemize}

Note that, under the considered assumptions, the well-posedness of the above martingale problem follows from standard results, see e.g. Chapter 7 in \cite{stroock:varadhan}, so that for any given initial condition $(s,\mu) \in \mathbb{R}_+ \times \pp$, there exists a unique weak solution to the SDE
\begin{align}
X_t^{s,\xi, (m+1)} & = \xi + \int_s^t b(r,X_r^{s,\xi, (m+1)},[X_r^{s,\xi, (m)}]) dr + \int_s^t \sigma(r,X_r^{s,\xi, (m+1)},[X_r^{s,\xi, (m)}]) d W_r. \label{iter:mckean} 
\end{align}

 We also associate to the above dynamics the decoupled stochastic flow given by the unique weak solution to SDE 
\begin{align}
X_t^{s, x, \mu, (m+1)} & = x + \int_s^t b(r,X_r^{s, x,\mu, (m+1)},[X_r^{s,\xi, (m)}]) dr + \int_s^t \sigma(r,X_r^{s,x, \mu, (m+1)},[X_r^{s,\xi, (m)}]) d W_r.  \label{iter:mckean:decoupl}
\end{align}

We recall that the notation $X^{s, x, \mu, (m+1)}_t$ makes sense since by weak uniqueness of solution to the SDE \eqref{iter:mckean}, the law $[X^{s, \xi, (m)}_t]$ only depends on the initial condition $\xi$ through its law $\mu$. 

From \cite{friedman:64}, for any positive integer $m$, both random variables $X^{s,\xi,(m)}_t$ and $X^{s, x, \mu,(m)}_t$ admit a density w.r.t the Lebesgue measure on $\mathbb{R}^d$ that we denote by $p_m(\mu, s, t, z)$ and $p_m(\mu, s, t, x, z)$ respectively. Moreover, the following relation is satisfied 
\begin{equation}
 p_m(\mu, s, t, z) = \int p_m(\mu, s, t, x, z) \mu(dx), \label{conv:relation:step:m}
\end{equation}

\noindent where for all $m\geq1$
\begin{align}
p_m(\mu, s, t, x, z) & = \sum_{k \geq 0} (\widehat{p}_{m} \otimes \mH^{(k)}_m)(\mu, s, t, x, z), \label{series:approx:mckean} 
\end{align}

\noindent with
\begin{align}
\widehat{p}^{y}_m(\mu, s, r,  t , x, z) & = g\left(\int_{r}^{t} a(r', y, [X^{s,\xi, (m-1)}_{r'}]) dr', z-x\right), \label{definition:general:phat}\\
\widehat{p}_m(\mu, s, r,  t , x, z) & = \widehat{p}^{z}_m(\mu, s, r,  t , x, z) = g\left(\int_{r}^{t} a(r', z, [X^{s,\xi, (m-1)}_{r'}]) dr', z-x\right), \label{definition:phat:end:point:frozen} \\
\mH_m(\mu, s, r ,t ,x ,z) & = \left\{- \sum_{i=1}^d b_i(r, x, [X^{s, \xi, (m-1)}_r]) H^{i}_1\left(\int_r^{t} a(r', z, [X^{s, \xi, (m-1)}_{r'}]) dr', z-x\right) \right. \label{definition:parametrix:kernel:iterate:m} \\
&  \left. + \frac12 (a_{i, j}(r, x, [X^{s, \xi, (m-1)}_{r}]) - a_{i, j}(r, z, [X^{s, \xi, (m-1)}_r])) \right. \nonumber \\
& \left. \quad \quad \times H^{i, j}_2\left(\int_r^{t} a(r', z, [X^{s, \xi, (m-1)}_{r'}]) dr', z-x\right)  \right\} \widehat{p}_m(\mu, s, r, t, x, z) \nonumber
\end{align}

\noindent and $\mH^{(k+1)}_m(\mu, s, t, x, z) = (\mH_m \otimes \mH^{(k)}_m)(\mu, s, t, x, z)$, $\mH^{(0)}_m = I_d$, with the convention that $[X^{s, \xi, (0)}_t] = \P^{(0)}(t)=\nu$, $t\geq0$. In what follows, we will often make use of the following estimates: there exist constants $c:= c(\lambda)>1$, $K:=K(T, b, a , \eta, \lambda)>0$, such that for any integer $k$, for any $(\mu, x, z) \in \pp \times (\mathbb{R}^d)^2$, for any $0\leq s < t \leq T$ and any $r\in [s, t)$, it holds
\begin{equation}
\label{iter:parametrix:kernel}
| \mH^{(k)}_m(\mu, s, r, t, x, z)| \leq \frac{K^{k}}{(t-r)^{1-k \frac{\eta}{2}}} \prod_{\ell=1}^{k-1} B\left(\ell \frac{\eta}{2}, \frac{\eta}{2}\right) \, g(c(t-r), z-x)
\end{equation}
\noindent and
\begin{equation}
\label{iter:param:classic}
| \widehat{p}_{m} \otimes \mH^{(k)}_m(\mu, s, t, x, z)| \leq K^{k} (t-s)^{k \frac{\eta}{2}} \prod_{\ell=1}^{k} B\left(1+ \frac{(\ell-1)\eta}{2}, \frac{\eta}{2}\right) g(c(t-s), z-x)
\end{equation}

\noindent where $B(k, \ell) = \int_0^1 (1-v)^{-1+k} v^{-1+\ell} dv$ stands for the Beta function. From the asymptotics of the Beta function, the series \eqref{series:approx:mckean} converges absolutely and uniformly for $(\mu, x, z) \in \pp \times (\rr^d)^2$ and satisfies the following Gaussian upper-estimates: there exist some positive constants $K:=K(T, b, a, \eta, \lambda)$ and $c:=c(\lambda)$ such that for any positive integer $m$, for any $0\leq s < t \leq T$ and any $(\mu, x, z) \in \pp \times (\mathbb{R}^d)^2$ 
\begin{equation}
\label{bound:derivative:heat:kernel}
| \partial^{n}_x p_m(\mu, s, t, x, z)| \leq K (t-s)^{-\frac{n}{2}} \, g(c (t-s), z-x), \, n=0,1,2,
\end{equation}

\noindent and for any $\beta \in [0,1]$ if $n\in \left\{ 0, 1\right\}$ or for any $\beta \in [0, \eta)$ if $n=2$, there exist some positive constants $K:=K(T, b, a, \eta, \lambda, \beta)$ and $c:=c(\lambda)$ such that for any positive integer $m$ and any $(x, x')\in (\mathbb{R}^d)^2$
\begin{align}
\label{reg:heat:kernel:deriv}
 |\partial^{n}_x p_{m}(\mu, s, t, x, z) & - \partial^{n}_x p_{m}(\mu, s, t , x', z)| \nonumber \\
 & \leq K \frac{|x-x'|^{\beta}}{(t-s)^{\frac{n+\beta}{2}}} \left\{ g(c (t-s), z-x) + g(c(t-s), z-x') \right\}.
\end{align}

We refer to Friedman \cite{friedman:64} for a proof of the above estimates. Denote now by $\Phi_m(\mu, s, r, t, x_1, x_2)$ the solution to the Volterra integral equation
\begin{equation}
\Phi_m(\mu, s, r, t, x_1, x_2) = \mH_m(\mu, s,  r, t, x_1, x_2) + (\mH_m \otimes \Phi_m)(\mu, s, r, t, x_1, x_2). 
\end{equation}  

From the space-time inequality \eqref{space:time:inequality}, it is easily seen that the singular kernel $\mH_m(\mu, s, r, t, x_1, x_2)$ induces an integrable singularity in time in the above space-time convolution so that the solution exists and is given by the (uniform) convergent series
\begin{equation}
\label{infinite:series:Phi:step:m}
\Phi_m(\mu, s, r, t, x_1, x_2) = \sum_{k\geq1} \mH^{(k)}_m(\mu, s, r, t, x_1, x_2)
\end{equation}

\noindent which by \eqref{iter:parametrix:kernel} and the asymptotics of the Beta function satisfies
\begin{equation}
\label{Gaussian:estimate:Phim}
| \Phi_m(\mu, s, r, t, x_1, x_2) | \leq \frac{K}{(t-r)^{1-\frac{\eta}{2}}} \, g(c(t-r), x_2-x_1)
\end{equation}
\noindent for some positive constants $K:=K(T, b, a, \eta, \lambda), \, c:=c(\lambda)$. Moreover, the infinite series \eqref{series:approx:mckean} satisfies
\begin{equation}
p_m(\mu, s, t, x, z) = \widehat{p}_m(\mu, s,  t , x, z)  + \int_s^t \int_{\rr^d} \widehat{p}_m(\mu, s,  r , x, y) \Phi_m(\mu, s, r, t, y, z) \, dy \, dr. \label{other:representation:parametrix:series}
\end{equation}

From Chapter 1 in \cite{friedman:64}, for any positive integer $m$, the maps $x\mapsto \mH_m(\mu, s, r, t, x, z), \,  \Phi_m(\mu, s, r, t, x, z)$ are H\"older-continuous. More precisely, for any $\beta \in [0,\eta]$ and any $\beta' \in[0, \eta)$, there exist some positive constants $K:=K(T, b, a, \eta, \lambda, \beta)$, $K':=K(T, b, a, \eta, \lambda, \beta')$, $c:=c(\lambda)$, which do not depend on $m$, such that for any $0\leq s \leq r < t \leq T$ and any $(\mu, x, y,  z) \in \pp \times (\mathbb{R}^d)^2$ 
\begin{align}
|\mH_m(\mu, s, r, t, x, z) & - \mH_m(\mu, s, r, t, y, z) |\nonumber \\
& \leq K \frac{|x-y|^{\beta}}{(t-r)^{1+ \frac{\beta-\eta}{2}}} \left\{ g(c(t-r), z -x) + g(c(t-r), z-y)\right\} \label{holder:reg:parametrix:kernel}
\end{align}

\noindent and 
\begin{align}
|\Phi_m(\mu, s, r, t, x, z) & - \Phi_m(\mu, s, r, t, y, z) |\nonumber \\
& \leq K' \frac{|x-y|^{\beta'}}{(t-r)^{1+ \frac{\beta'-\eta}{2}}} \left\{ g(c(t-r), z -x) + g(c(t-r), z-y)\right\}. \label{holder:reg:voltera:kernel}
\end{align}

Before stating the main result of this section, we introduce the following notation. For some positive integer $m$, $n\in \left\{0,1\right\}$, $\beta \in [0, 1+\eta)$ if $n=0$ or $\beta \in [0,\eta)$ if $n=1$, $C \geq0$ and $t\in [0,T]$, we define
$$
\mathscr{C}^{n ,\beta}_m(C, t):=  \sum_{k=1}^{m} C^{k} t^{(k-1) \frac{\eta}{2}} \prod_{i=1}^{k-1} B\left(\frac{\eta}{2}, \frac{1-n+\eta-\beta}{2} +  (i-1) \frac{\eta}{2} \right).
$$

With the above notations and properties, we prove the following key proposition whose proof is postponed to the next subsection.

\begin{prop}\label{proposition:reg:density:recursive:scheme:mckean} Let $T>0$. For any fixed $(t, z) \in (0,T] \times \rr^d$, for any positive integer $m$, the following properties hold:
\begin{itemize}
\item The mapping $[0,t) \times \rr^d \times \pp \ni (s, x, \mu) \mapsto p_m(\mu, s, t, x, z)$ is in $\mathcal{C}^{1, 2, 2}([0,t) \times \rr^d \times \pp)$.

\item There exist positive constants $C := C(T, \HR, \HE), \, C_\beta:=C(T, \HR, \HE, \beta), \, c:=c(\lambda)$, which do not depend on $m$, such that, for any $(\mu, s, x, x', z, v, v') \in \mathcal{P}_2(\rr^d) \times [0,t) \times (\rr^d)^5$
\begin{align}
|\partial^{n}_v [\partial_\mu p_m(\mu, s, t, x, z)](v)| & \leq  \frac{\mathscr{C}^{n, 0}_m(C, t-s)}{(t-s)^{\frac{1+n-\eta}{2}}}  \, g(c (t-s), z-x), \quad n=0, 1,\label{first:second:estimate:induction:decoupling:mckean} \\
| \partial_s p_m(\mu, s, t, x, z) | & \leq  \frac{\mathscr{C}^{1, 0}_m(C, t-s)}{t-s} \, g(c (t-s), z-x), \label{time:derivative:induction:decoupling:mckean}
\end{align}

\begin{align}
  | \partial^{n}_v [\partial_\mu p_{m}(\mu, s, t, x, z)](v) & - \partial^{n}_v [\partial_\mu p_{m}(\mu, s, t , x', z)] (v) |\nonumber \\
&  \leq \mathscr{C}^{n, \beta}_{m}(C_\beta, t-s) \frac{|x-x'|^{\beta}}{(t-s)^{\frac{1+n+\beta-\eta}{2}}} \left\{ g(c(t-s), z-x) + g(c(t-s), z-x') \right\},  \label{equicontinuity:second:third:estimate:decoupling:mckean} 
\end{align}
\noindent where $\beta \in [0,1]$ for $n=0$ and $\beta \in [0, \eta)$ for $n=1$,

\begin{align}
 |\partial_v [\partial_\mu p_m(\mu, s, t, x, z)](v) & - \partial_v [\partial_\mu p_m(\mu, s, t, x, z)](v')| \nonumber \\
& \leq \mathscr{C}^{1, \beta}_{m}(C_\beta, t-s) \frac{|v-v'|^{\beta}}{(t-s)^{1+\frac{\beta-\eta}{2}}} \, g(c(t-s), z-x) \label{equicontinuity:first:estimate:decoupling:mckean}
\end{align}
\noindent where $ \beta\in [0,\eta)$.
\item There exist three positive constants $C^{+}_\beta := C(T, \HRp, \HE, \beta)$, $C_\beta:= C(T, \HR, \HE, \beta)$, $c:=c(\lambda)$, which do not depend on $m$, such that for any $(\mu, \mu', s, x, z, v) \in (\mathcal{P}_2(\rr^d))^2 \times [0,t) \times (\rr^d)^3$


\begin{align}
 |\partial^{n}_x p_{m}(\mu, s, t, x, z) & - \partial^{n}_x p_{m}(\mu', s, t, x, z)]| \leq C_\beta \frac{W^{\beta}_2(\mu, \mu')}{(t-s)^{ \frac{n+\beta}{2}}} \, g(c(t-s), z-x), \label{regularity:measure:estimate:v1:v2:v3:decoupling:mckean} 
\end{align}

\noindent where $\beta \in [0,1]$ for $n \in \left\{0, 1\right\}$ and $\beta \in [0,\eta)$ for $n=2$,
\begin{align}
 |\partial^{n}_v [\partial_\mu p_{m}(\mu, s, t, x, z)](v) & - \partial^{n}_v [\partial_\mu p_{m}(\mu', s, t, x, z)](v)| \nonumber \\
 & \leq \mathscr{C}^{n, \beta}_{m}(C^+_{\beta}, t-s)  \frac{W^{\beta}_2(\mu, \mu')}{(t-s)^{ \frac{1+n+\beta- \eta}{2}}} \, g(c(t-s), z-x), \label{regularity:measure:estimate:v1:v2:decoupling:mckean} 
\end{align}

\noindent where $\beta \in [0,1]$ for $n=0$ and $\beta \in [0,\eta)$ for $n=1$, and for all $(s_1, s_2) \in [0,t)^2$,


\begin{align}
 | & \partial^{n}_x p_{m}(\mu, s_1, t, x, z) - \partial^{n}_x p_{m}(\mu, s_2, t, x, z)  | \nonumber \\
 & \leq C_\beta \left\{ \frac{|s_1-s_2|^{\beta}}{(t-s_1)^{\frac{n}{2} + \beta }} \, g(c(t-s_1), z-x) + \frac{|s_1-s_2|^{\beta}}{(t- s_2)^{ \frac{n}{2} +\beta }} \, g(c(t-s_2), z-x) \right\}, \label{regularity:time:estimate:v1:v2:v3:decoupling:mckean} 
\end{align}

\noindent where $\beta \in [0,1]$ for $n=0$, $\beta \in [0,(1+\eta)/2)$ for $n=1$ and $\beta \in [0, \eta/2)$ for $n=2$ and
\begin{align}
 | & \partial^{n}_v [\partial_\mu p_{m}(\mu, s_1, t, x, z)](v) - \partial^{n}_v [\partial_\mu p_{m}(\mu, s_2, t, x, z)](v)| \nonumber \\
 & \leq \mathscr{C}^{n, 2 \beta}_{m}(C^+_{\beta}, t-s_1 \vee s_2) \left\{ \frac{|s_1-s_2|^{\beta}}{(t-s_1)^{\frac{1+n-\eta}{2} + \beta }} \, g(c(t-s_1), z-x) + \frac{|s_1-s_2|^{\beta}}{(t- s_2)^{ \frac{1+n-\eta}{2} +\beta }} \, g(c(t-s_2), z-x) \right\}, \label{regularity:time:estimate:v1:v2:decoupling:mckean} 
\end{align}

\noindent where $\beta \in [0,(1+\eta)/2)$ if $n=0$ and $\beta \in [0,\eta/2)$ if $n=1$.

%
%

\end{itemize}

\end{prop}

\begin{remark}\label{remark:convergence:series:gaussian:upper:bounds}
Let us indicate that for any positive constant $C$, any $T>0$, any $n\in \left\{0, 1\right\}$ and any $\beta \in [0, 1-n+\eta)$, the series $\mathscr{C}^{n, \beta}_\infty(C,T) = \lim_{m\uparrow \infty} \mathscr{C}^{n,\beta}_{m}(C, T)$ converges. Indeed, from the relation $B(a, b) = \Gamma(a) \Gamma(b)/\Gamma(a+b)$, $a, b>0, $ it follows
$$
\mathscr{C}^{n, \beta}_{m}(C, T) = \sum_{k=1}^{m} \frac{(C\Gamma(\eta/2))^{k}}{\Gamma\left(\frac{1-n+\eta-\beta}{2}+ (k-1)\frac{\eta}{2}\right)} T^{k-1}
$$ 
\noindent so that, by Stirling's formula, the above $m$th partial sums converge and satisfy 
$$
 \mathscr{C}^{n, \beta}_{\infty}(C, T) := \lim_{m \uparrow \infty }\mathscr{C}^{n, \beta}_{m}(C, T) = C\Gamma(\eta/2) E_{\eta/2, \frac{1-n+\eta-\beta}{2}}(C\Gamma(\eta/2))< \infty
$$
\noindent where we recall that $z\mapsto E_{\alpha, \beta}(z)$ stands for the Mittag-Leffler function.
\end{remark}
The proof of Proposition \ref{proposition:reg:density:recursive:scheme:mckean} being rather long and technical, we postpone it to the Appendix \ref{proof:main:prop}. 

\medskip

\noindent \emph{Step 2: Extraction of a convergent subsequence}

\medskip

Our next step now is to extract from the following sequences $\left\{\mathbb{L}^{2} \ni \xi \mapsto \widetilde{p}_m(\xi, s, t, x, z) , m \geq 0 \right\}$ (the lifting of the sequence $\left\{ \pp \ni \mu \mapsto p_m(\mu, s, t, x, z), m\geq0\right\}$), $\left\{ \rr^d \ni v \mapsto \partial_\mu p_m(\mu, s, t, x, z)(v) , m \geq 0 \right\}$, $\left\{ \rr^d \ni v \mapsto \partial_v [\partial_\mu p_m(\mu, s, t, x, z)](v) , m \geq 0 \right\}$ the corresponding subsequences which converge locally uniformly using the Arzel\`a-Ascoli theorem.

Since the coefficients $b_i, \, a_{i, j}$ are bounded and the initial condition $\mu \in \pp$, the sequence of probability measures $(\P^{(m)})_{m\geq0}$ is tight. Relabelling the indices if necessary, we may assert that $(\P^{(m)})_{m \geq 0}$ converges weakly to a probability measure $\P^{\infty}$. Under our current assumptions \HE\, and \HRp, from standard arguments that we omit (namely passing to the limit in the characterisation of the martingale problem solved by $\P^{(m)}$) we deduce that $\P^\infty$ is the probability measure $\P$ induced by the unique weak solution to the McKean-Vlasov SDE \eqref{SDE:MCKEAN}. As a consequence, every convergent subsequence converges to the same limit $\P$ and so does the original sequence $(\P^{(m)})_{m\geq1}$. 

By the dominated convergence, for any fixed $t >0$ and $z \in \rr^d$, using \eqref{iter:param:classic}, one may pass to the limit as $m\uparrow \infty$ in the parametrix infinite series \eqref{series:approx:mckean} and thus deduce that the sequence of functions $\left\{\mathcal{K} \ni (s, x, \mu) \mapsto p_m(\mu, s, t, x, z), \, m\geq1\right\}$, $\mathcal{K}$ being a compact set of $[0,t) \times \rr^d \times \pp$, converges to the map $(s, x, \mu) \mapsto p(\mu, s, t, x, z)$, for any fixed $(s, x, \mu) \in \mathcal{K}$, given by the infinite series \eqref{parametrix:series:expansion}. Moreover, it is clearly uniformly bounded and from \eqref{regularity:measure:estimate:v1:v2:v3:decoupling:mckean}, \eqref{regularity:time:estimate:v1:v2:v3:decoupling:mckean} and \eqref{reg:heat:kernel:deriv}, it is equicontinuous. Relabelling the indices if necessary, from the Arzel\`a-Ascoli theorem, we may assert that it converges uniformly. We thus deduce that the map $[0,t) \times \rr^d \times \pp \ni (s, x, \mu) \mapsto p(\mu, s, t, x, z)$ is continuous.

For any $\mu \in \pp$ and any integer $m$, the mapping $ (s, x, \mu) \mapsto p_m(\mu, s, t ,x ,z)$ is in $\mathcal{C}^{0, 2, 0}([0,t) \times \mathbb{R}^d \times \pp)$. Moreover, from the estimates \eqref{regularity:measure:estimate:v1:v2:v3:decoupling:mckean}, \eqref{regularity:time:estimate:v1:v2:v3:decoupling:mckean}, \eqref{bound:derivative:heat:kernel} and \eqref{reg:heat:kernel:deriv} (for $n=1,2$), the sequence of functions $\{\mathcal{K} \ni (s, x, \mu) \mapsto \partial_x p_m(\mu, s, t, x, z)$, $\partial^2_x p_m(\mu, s, t, x, z) , m\geq 0\}$, are uniformly bounded and equicontinuous. Hence, from the Arzel\`a-Ascoli theorem, we may assert that $(s, x) \mapsto p(\mu, s, t, x, z)\in \mathcal{C}^{0, 2 }([0,t) \times \rr^d)$ and that the mappings $[0,t) \times \rr^d \times \pp \ni (s, x, \mu) \mapsto \partial_x p(\mu, s, t, x, z)$, $\partial^2_x p(\mu, s, t, x, z)$ are continuous.

We now consider the sequence $\left\{ \mathcal{K} \ni \xi \mapsto D\widetilde{p}_{m}(\xi, s, t, x, z) = \partial_\mu p _m([\xi], s, t, x, z)(\xi),\, m\geq 1 \right\}$, $\mathcal{K}$ being any compact set of $\mathbb{L}^{2}$. From \eqref{first:second:estimate:induction:decoupling:mckean} (with $n=0$), it is uniformly bounded. From \eqref{first:second:estimate:induction:decoupling:mckean} (with $n=1$) and \eqref{regularity:measure:estimate:v1:v2:decoupling:mckean} (with $n=0$), it is equicontinuous. Relabelling the indices if necessary, from the Arzel\`a-Ascoli theorem, we may assert that it converges uniformly. We thus deduce that the map $\mathcal{K} \ni \xi \mapsto \widetilde{p}(\xi, s, t, x, z)$ is continuously differentiable. As a consequence, $\pp \ni \mu \mapsto p(\mu, s, t, x, z)$ is continuously $L$-differentiable. 

From the estimates \eqref{equicontinuity:second:third:estimate:decoupling:mckean}, \eqref{first:second:estimate:induction:decoupling:mckean}, \eqref{regularity:measure:estimate:v1:v2:decoupling:mckean} and \eqref{regularity:time:estimate:v1:v2:decoupling:mckean} (with $n=0$) and \eqref{first:second:estimate:induction:decoupling:mckean} (for $n=1$), the sequence $ \{\mathcal{K} \ni (s, x, \mu, v) \mapsto \partial_\mu p_{m}(\mu, s, t, x, z)(v)$, $m\geq 1\}$, $\mathcal{K}$ being a compact set of $[0,t) \times \rr^d \times \pp \times \rr^d$, is uniformly bounded and equicontinuous so that the map $[0,t) \times \rr^d \times \pp \times \rr^d \ni (s, x, \mu, v) \mapsto \partial_\mu p (\mu, s, t, x , z)(v)$ is continuous. 

From the estimates \eqref{first:second:estimate:induction:decoupling:mckean} and \eqref{equicontinuity:first:estimate:decoupling:mckean} (both for $n=1$), the sequence $\{\rr^d \supset B(0,R) \ni v \mapsto \partial_v [\partial_\mu p_{m}(\mu, s, t, x, z)](v)$, $m\geq1\}$, is bounded and equicontinuous so that $\rr^d \ni v \mapsto \partial_\mu p(\mu, s, t, x, z)(v)$ is continuously differentiable. 

The continuity of the map $[0,t) \times \rr^d \times \pp \times \rr^d \ni(s, x, \mu, v) \mapsto \partial_v [\partial_\mu p(\mu, s, t, x, z)](v)$ is finally deduced from the uniform convergence of the sequence of continuous mappings $\{ \mathcal{K}  \ni (s, x, \mu, v) \mapsto \partial_v [\partial_\mu p_{m}(\mu, s, t, x, z)](v)$, $m\geq1\}$, $\mathcal{K}$ being a compact set of $[0,t) \times \rr^d \times \pp \times\rr^d$, along a subsequence, obtained from the estimates \eqref{equicontinuity:second:third:estimate:decoupling:mckean}, \eqref{equicontinuity:first:estimate:decoupling:mckean}, \eqref{regularity:measure:estimate:v1:v2:decoupling:mckean} and \eqref{regularity:time:estimate:v1:v2:decoupling:mckean} (with $n=1$) combined with the Arzel\`a-Ascoli theorem. The estimates \eqref{first:second:lions:derivative:mckean:decoupling} and \eqref{equicontinuity:second:third:estimate:decoupling:mckean:final} to \eqref{regularity:time:estimate:v1:v2:mckean:decoupling} then follow by passing to the limit in the corresponding upper-bounds proved in the first step.

 \medskip

\noindent \emph{Step 3: $\mathcal{C}^{1, 2,2}([0,t) \times \rr^d \times \pp)$ regularity and related time estimates.}

\medskip

Let us now prove that $(s, x, \mu) \mapsto p(\mu, s, t, x, z)$ is in $\mathcal{C}^{1, 2, 2}([0,t) \times \rr^d \times \pp)$. From the Markov property satisfied by the SDE \eqref{SDE:MCKEAN}, stemming from the well-posedness of the related martingale problem, for any $0\leq h \leq s$, the following relation is satisfied 
$$
p(\mu, s-h, t, x, z) = \E[p([X^{s-h, \xi}_s], s, t, X^{s-h, x, \mu}_s, z)].
$$

Combining the estimates \eqref{first:second:lions:derivative:mckean:decoupling} and \eqref{gradient:estimates:parametrix:series} (for $n=1$) with the chain rule formula of Proposition \ref{prop:chain:rule:joint:space:measure} (with respect to the space and measure variables only) we obtain
$$
  \E[p([X^{s-h, \xi}_s], s, t, X^{s-h, x, \mu}_s, z)] = p(\mu, s, t, x, z) + \E\left[\int_{s-h}^{s} \mathcal{L}_r p([X^{s-h, \xi}_r], s, t, X^{s-h, x, \mu}_r, z) \, dr \right]
$$

\noindent \noindent with 
\begin{align*}
\mathcal{L}_r h(x,  \mu) & := \sum_{i=1}^d b_i(r, x, \mu) \partial_{x_i}h(x, \mu) +  \frac12 \sum_{i, j=1}^d a_{i, j}(r, x, \mu) \partial^2_{x_i, x_j}h( x, \mu) \\
& \quad +\int_{\rr^d} \left\{ \sum_{i=1}^d b_i(r, v, \mu) [\partial_{\mu} h(x, \mu)(v)]_i + \frac12  \sum_{i, j=1}^d a_{i, j}(r, v, \mu) \partial_{v_i} [\partial_{\mu} h(x, \mu)(v)]_j \right\} \, \mu(dv).
\end{align*}

 We thus deduce
$$
\frac{1}{h}(p(\mu, s-h, t, x, z)  - p(\mu, s, t, x, z)) = \frac{1}{h} \E\left[ \int_{s-h}^s \mathcal{L}_r p([X^{s-h, \xi}_r], s, t, X^{s-h, x, \mu}_r, z) \, dr \right]
$$


\noindent so that, letting $h\downarrow 0$, from the boundedness and continuity of the coefficients as well as the continuity of the maps $(\mu, x) \mapsto p(\mu, s, t, x, z), \, \partial^{1+n}_x p(\mu, s, t, x, z), \, \partial^{n}_v[\partial_\mu p(\mu, s, t, x, z)]$, for $n=0,1$, we deduce that $[0, t) \ni s\mapsto p(\mu, s, t, x, z)$ is left-differentiable. Still from the continuity of the coefficients and of the map $(s, x, \mu)\mapsto \mathcal{L}_s p(\mu, s, t, x, z)$, we then conclude that it is differentiable in time on the interval $[0,t)$ with a time derivative satisfying  
 \begin{eqnarray*}
\partial_s p(\mu, s, t, x, z) = - \mathcal{L}_s p(\mu, s, t, x, z) \quad \mbox{ on } [0,t) \times \rr^d \times \pp.
\end{eqnarray*} 

The time derivative estimate \eqref{first:time:derivative:mckean:decoupling} now follows from the previous relation as well as the estimates \eqref{gradient:estimates:parametrix:series} and \eqref{first:second:lions:derivative:mckean:decoupling}.

\section{Solving the related PDE on the Wasserstein space}\label{solving:pde:wasserstein:space}
This section is devoted to the proof of Theorem \ref{cauchy:problem:wasserstein}. Thanks to the regularity properties provided by Theorem \ref{derivative:density:sol:mckean:and:decoupling}, we are able to tackle the Cauchy problem \eqref{pde:wasserstein:space} in any strip $[0,T] \times \rr^d \times \pp$. We start with the following proposition.

\begin{prop}\label{reg:map:U}Under the assumptions of Theorem \ref{cauchy:problem:wasserstein}, the mapping $[0,T] \times \rr^d \times \pp \ni (t, x, \mu) \mapsto U(t, x, \mu)$ defined by \eqref{sol:pde:wasserstein} is continuous, belongs to $\mathcal{C}^{1, 2, 2}([0,T) \times \rr^d \times \pp)$, satisfies \eqref{bound:solution:cauchy:problem} and for any $(t, x, v, \mu) \in [0,T) \times (\rr^d)^2 \times \pp$
\begin{equation}
\Big| \partial^{n}_v[\partial_\mu U(t, x, \mu)](v) \Big| \leq C (T-t)^{-\frac{(1+n)}{2}} \exp(\frac{k |x|^2}{T}) (1+ |v|^2 + M^{q}_2(\mu)), \, n=0, 1, \label{bound:mes:deriv:sol:pde}
\end{equation}

\noindent and
\begin{equation}
\Big| \partial_x U(t, x, \mu) \Big| \leq K (T-t)^{-\frac{1}{2}} \exp(\frac{k |x|^2}{T}) (1+ M^{q}_2(\mu)), \, n=0, 1, \label{bound:space:deriv:sol:pde}
\end{equation}

\noindent where $C:=C(T, \HR, \HE)$, $K:=K(T, b, a , \lambda, \eta)$ and $k:=k(\lambda, \alpha)$ are positive constants. Moreover, $U$ is a solution to the Cauchy problem \eqref{pde:wasserstein:space} in the strip $[0,T] \times \rr^d \times \pp$.
\end{prop}

\begin{proof} 
The proof is divided into three steps. We first address the continuity of the map $U$ in $[0,T) \times \mathbb{R}^d \times \pp$. We then prove that $\lim_{t \uparrow T} U(t, x, \mu) = h(x, \mu)$. We eventually prove that the map $U$ belongs to $\mathcal{C}^{1, 2, 2}([0,T) \times \mathbb{R}^d \times \pp)$, solves the PDE \eqref{pde:wasserstein:space} and establish the related estimates.

\medskip

\noindent \emph{Step 1: continuity of the map $[0,T) \times \mathbb{R}^d \times \pp \ni (t, x, \mu) \mapsto U(t, x, \mu)$.} 

\medskip

Let $(\mu_n, t_n, x_n)_{n\geq1}$ be a sequence of $\pp \times [0,T) \times \mathbb{R}^d$ satisfying $\lim_{n} |t_n-t| = \lim_{n} W_2(\mu_n, \mu) = \lim_{n} |x_n-x| = 0$, for some $( \mu, t, x) \in \pp \times [0,T) \times \mathbb{R}^d$. By weak uniqueness, the sequence of probability measures $([X^{t_n, \xi_n}_T])_{n\geq1}$ converges weakly to $[X^{t, \xi}_T]$, where $[\xi_n] = \mu_n$ and $[\xi]=\mu$. Hence, since the $2$-Wasserstein distance metrizes weak convergence in $\pp$, it suffices to prove that $\lim_n \int_{\mathbb{R}^d} |z|^2 p(\mu_n, t_n, T, z) \, dz = \int_{\mathbb{R}^d} |z|^2 p(\mu, t, T, z)\, dz$ in order to derive that $\lim_n W_2([X^{t_n, \xi_n}_T], [X^{t, \xi}_T])=0$. Taking the same notation as in the first step of Proposition \ref{structural:class} with $\bar{h}(z) = |z|^2$, we write
\begin{align*}
 \int_{\mathbb{R}^d} & |z|^2 p(\mu_n, t_n, T, z) \, dz   - \int_{\mathbb{R}^d} |z|^2 p(\mu, t, T, z)\, dz \\
 &  =   \int_{(\mathbb{R}^d)^2} \bar{h}(z) p(\mu_n, t_n, T, x, z) \, dz (\mu_n-\mu)(dx) + \int_{(\mathbb{R}^d)^2} \bar{h}(z) (p(\mu_n, t_n, T, x, z) - p(\mu, t, T, x, z)) \, dz \mu(dx) \\
 & =: \mathcal{A}_n \bar{h} + \mathcal{B}_n \bar{h}
\end{align*}
\noindent where $\mathcal{A}_n \bar{h}  = \mathcal{A}^{1}_n \bar{h} + \mathcal{A}^{2}_n \bar{h}$ with
\begin{align*}
\mathcal{A}^{1}_n \bar{h} &:= \int_{(\mathbb{R}^d)^2} \bar{h}(z) p(\mu_n, t_n, T, x, z) \,dz \,\eta_R(x) (\mu_n-\mu)(dx), \\
 \mathcal{A}^2_n \bar{h} & := \int_{(\mathbb{R}^d)^2} \bar{h}(z) p(\mu_n, t_n, T, x, z) \, dz\, (1-\eta_R)(x) (\mu_n-\mu)(dx)
\end{align*}
\noindent where we recall that $\eta_R$, $R>1$, is a non-negative smooth cutoff function such that $0\leq \eta_R \leq 1$, $\eta_R(x) = 1$ for $|x|\leq R$, $\eta_R(x)=0$ for $|x|\geq 2R$ and $|\nabla \eta_R|_\infty \leq C$, $C$ being a positive constant independent of $R$. We first deal with $\mathcal{A}_n \bar{h}$. From the Gaussian estimates \eqref{gradient:estimates:parametrix:series} with $n=0,1$, the map $f^{n}_R: \mathbb{R}^d \ni x\mapsto \int_{\mathbb{R}^d} \bar{h}(z)   p(\mu_n, t_n, T, x, z) \eta_R(x) \, dz $ is continuously differentiable with a first order derivative uniformly bounded by $|\nabla f^{n}_R|_\infty \leq C(1+R^2)$ so that, from the Monge-Kantorovich duality principle, $ \lim\sup_n |\mathcal{A}^{1}_n \bar{h}| \leq C (1+ R^2) \lim\sup_{n}W_1(\mu_n, \mu)=0$, which in turn clearly yields $\lim_{n} \mathcal{A}^{1}_n \bar h = 0$. 

Again the Gaussian estimates \eqref{gradient:estimates:parametrix:series} with $n=0$ and the weak convergence in $\mathcal{P}_2(\mathbb{R}^d)$ of $(\mu_n)_{n\geq1}$ towards $\mu$ yield 
$$
\lim\sup_n  |\mathcal{A}^{2}_n \bar{h}| \leq \bigg(\lim\sup_{n}\int_{|x| \geq R} (1+|x|^2) \mu_n(dx) + \int_{|x| \geq R} (1+|x|^2)\mu(dx)\bigg) \leq 2 \int_{|x| \geq R} (1+|x|^2)\mu(dx).
$$
\noindent so that, by letting $R\uparrow \infty$, we eventually deduce $\lim_n \mathcal{A}^{2}_n \bar{h} =0$.

 We now deal with $\mathcal{B}_n \bar{h}$. By continuity of the map $[0,T) \times \pp \ni (t, \mu) \mapsto p(\mu, t, T, x, z)$, we deduce that the sequence $(p(\mu_n, t_n, T, x, z))_{n\geq 1}$ converges to $p(\mu, t, T, x, z)$ for any fixed $ x, z$. From the pointwise Gaussian upper-estimate \eqref{bound:density:parametrix} and the dominated convergence theorem, we deduce that $\lim_{n}  \mathcal{B}_n \bar{h} = 0$. From the above arguments, we thus conclude that $\lim_n W_2([X^{t_n, \xi_n}_T], [X^{t, \xi}_T])=0$.

 By continuity of the maps $h$ and $f$, the two sequences $(h(z, [X^{t_n, \xi_n}_T]))_{n\geq1}$ and $(f(s, z, [X^{t_n, \xi_n}_s]))_{n\geq1}$ converge respectively to  $h(z, [X^{t, \xi}_T])$ and $f(s, z, [X^{t, \xi}_s])$. Using again the fact that the sequence $(p(\mu_n, t_n, T, x_n, z))_{n\geq1}$ converges to $p(\mu, t, T, x, z)$, the Gaussian upper-estimate \eqref{bound:density:parametrix} satisfied by $z\mapsto p(\mu_n, t_n, T, x_n, z)$, the growth assumption \eqref{growth:condition:h:f} with $\alpha < (2c)^{-1}$, $c:=c(\lambda)$ being the constant appearing in \eqref{bound:density:parametrix} together with the fact that $M_2([X^{t_n, \xi_n}_s]) \leq C(1+ M_2(\mu))$ for some positive constant $C$ independent of $n$ and  the dominated convergence theorem, we finally deduce that the sequence $(U(t_n, x_n, \mu_n))_{n\geq1}$ converges to $U(t, x, \mu)$. We thus conclude that the map $[0,T) \times \mathbb{R}^d \times \pp \ni (t, x, \mu) \mapsto U(t, x, \mu)$ is continuous. \\

\noindent \emph{Step 2: $\lim_{t \uparrow T} U(t, x, \mu) = h(x, \mu)$.} 

\medskip

Following the same lines of reasonings as those employed in chapter 1, Friedman \cite{friedman:64}, we deduce that $\lim_{t\uparrow T} p(\mu, t, T, x, z) = \delta_z(x)$ in the weak sense so that $[X^{t, \xi}_T]$ converges weakly to $\mu$ as $t \uparrow T$. As in the previous step, in order to derive that the latter convergence holds with respect to the $2$-Wasserstein metric, it remains to prove that $\lim_{t\uparrow T} \int_{(\mathbb{R}^d)^2} |z|^2 p(\mu, t, T, x, z) \, dz \mu(dx) = \int_{\mathbb{R}^d} |x|^2 \mu(dx)$. We first write
$$
\int_{(\mathbb{R}^d)^2} |z|^2 p(\mu, t, T, x, z) \, dz \mu(dx) - \int_{\mathbb{R}^d} |x|^2 \mu(dx) = \int_{\mathbb{R}^d} \int_{\mathbb{R}^d} (|z|^2-|x|^2) p(\mu, t, T, x, z) \,dz \,  \mu(dx)
$$

\noindent then use the inequality $||z|^2-|x|^2| \leq |z-x| (|z|+|x|)$ together with the Gaussian upper-estimate \eqref{bound:density:parametrix} and the space-time inequality \eqref{space:time:inequality} so that
$$
|\int_{(\mathbb{R}^d)^2} |z|^2 p(\mu, t, T, x, z) \, dz \mu(dx) - \int_{\mathbb{R}^d} |x|^2 \mu(dx)| \leq C (T-t)^{\frac12} (1+ \int_{\mathbb{R}^d} |x| \mu(dx)).
$$

Passing to the limit as $t\uparrow T$ in both sides of the previous inequality yields the result. We thus conclude that $\lim_{t\uparrow T} W_2([X^{t, \xi}_T], \mu) = 0$. Now, by the continuity of $h$, $\lim_{t\uparrow T} h(z, [X^{t, \xi}_T]) = h(z, \mu)$ and since the convergence also holds locally uniformly in $z$, we deduce that $\lim_{t\uparrow T} \int_{\mathbb{R}^d} h(z, [X^{t, \xi}_T]) p(\mu, t, T, x, z) \, dz = h(x, \mu)$. Combining the Gaussian upper-estimate \eqref{bound:density:parametrix} and the growth assumption \eqref{growth:condition:h:f} allows to derive that $\lim_{t \uparrow T} \int_{t}^{T} \int_{\mathbb{R}^d} f(s, z, [X^{t, \xi}_s]) \,  p(\mu, t, s, x, z) \, dz  \, ds = 0$.  We thus conclude that $\lim_{t \uparrow T} U(t, x, \mu) = h(x, \mu)$.\\

\noindent \emph{Step 3: the map $U$ belongs to $\mathcal{C}^{1, 2, 2}([0,T) \times \mathbb{R}^d \times \pp)$, solves the PDE \eqref{pde:wasserstein:space} and related estimates.} 

\medskip

 We now prove that $ \mathcal{P}_2(\rr^d) \ni (x,\mu) \mapsto U(t, x, \mu) \in \mathcal{C}^{2,2}(\mathbb{R}^d \times \pp)$, for any $t \in [0,T)$ and that $[0,T) \times \mathbb{R}^d \times \mathcal{P}_2(\mathbb{R}^d) \ni (t, x,\mu) \mapsto \mathcal{L}_t U(t, x,\mu)$ is continuous, where the operator $\mathcal{L}_t$ is defined by \eqref{inf:generator:mckean:vlasov}.

 From Theorem \ref{derivative:density:sol:mckean:and:decoupling} and the relation \eqref{relation:density:mckean:decoupling:field}, the map $ \pp \ni  \mu  \mapsto p(\mu, t, T, z)$ is partially $\mathcal{C}^{2}(\pp)$ (see Chapter 5 of \cite{carmona2018probabilistic} for a definition of partial $\mathcal{C}^{2}(\pp)$ regularity) with derivatives given by
\begin{align*}
\partial^{n}_v[\partial_\mu p(\mu, t, T, z)](v) & = \partial^{1+n}_x p(\mu, t, T, v, z) + \int_{\rr^d} \partial^{n}_v[\partial_\mu p(\mu, t, T, x, z)](v) \, \mu(dx), \quad n \in \left\{0,\, 1\right\}.
\end{align*}

From Proposition \ref{structural:class}, we thus deduce that the two maps $ \pp \ni \mu \mapsto h(z, [X^{t, \xi}_T])$, $ \pp \ni  \mu \mapsto f(s, z, [X^{t, \xi}_s])$ are partially $\mathcal{C}^{2}( \pp)$ for any fixed $T>0, \, s >t\geq 0$ and $z\in \rr^d$. Note carefully that in Proposition \ref{structural:class}, the linear functional derivative is assumed to be bounded for sake of simplicity while here the linear functional derivatives $ [\delta h/\delta m](z, m)(.)$ and $ [\delta f/ \delta m](t, z, m)(.)$ are of quadratic growth, see \eqref{growth:condition:tilde:h:f}. However, using the pointwise Gaussian estimates \eqref{gradient:estimates:parametrix:series} and \eqref{first:second:lions:derivative:mckean:decoupling}, one can extend the analysis performed in step 2 of the proof of Proposition \ref{structural:class} to the current setting. According to \eqref{condition1:structural:class:measure:derivative}, their $L$-derivatives are given by
\begin{align}
\partial^{n}_v[\partial_\mu [h(z, [X^{t, \xi}_T])]](v) & = \int_{\mathbb{R}^d} \frac{\delta h}{\delta m}(z, [X^{t, \xi}_T])(y) \, \partial^{1+n}_x p(\mu, t, T, v, y) \, dy \nonumber \\
& \quad + \int_{(\mathbb{R}^d)^2} \frac{\delta h}{\delta m}(z, [X^{t, \xi}_T])(y) \, \partial^{n}_v[\partial_\mu p(\mu, t, T, x, y)](v) \, dy \, \mu(dx) \label{deriv:measure:h}
\end{align}

\noindent and
\begin{align}
\partial^{n}_v[\partial_\mu [f(s, z, [X^{t, \xi}_s])]](v) & = \int_{\mathbb{R}^d} \Big[\frac{\delta f}{\delta m}(s, z, [X^{t, \xi}_s])(y) - \frac{\delta f}{\delta m}(s, z, [X^{t, \xi}_s])(v)\Big] \, \partial^{1+n}_x p(\mu, t, s, v, y) \, dy\nonumber \\
& \quad +  \int_{(\mathbb{R}^d)^2} \frac{\delta f}{\delta m}(s, z, [X^{t, \xi}_s])(y)  \, \partial^{n}_v[\partial_\mu p(\mu, t, s, x, y)](v) \, dy \, \mu(dx). \label{deriv:measure:source}
\end{align}

We may break the first integral appearing in the right-hand side of \eqref{deriv:measure:source} into two parts ${\rm J}_1$ and ${\rm J}_2$ by dividing the domain of integration into two domains. In the first part ${\rm J}_1$, the $dy$-integration is taken over a bounded domain $D$ containing $v$ such that $|y-v| \geq 1$ if $y \notin D$. Using \eqref{local:holder:reg:linear:functional:deriv:f}, that is, the $\eta$-H\"older regularity of $ [\delta f/\delta m](s, z, \mu)(.)$ on $D$, \eqref{gradient:estimates:parametrix:series}, the space-time inequality \eqref{space:time:inequality} and noting that $M_2([X^{t, \xi}_s]) \leq C(1+M_2(\mu))$, we get
$$
|{\rm J}_1| \leq C \exp\left(\alpha \frac{|z|^2}{T}\right) \, (s-t)^{\frac{-1- n + \eta}{2}}  (1+ M^{q}_2(\mu)).
$$

As for ${\rm J}_2$, from \eqref{growth:condition:tilde:h:f} and the space-time inequality \eqref{space:time:inequality}, we obtain
\begin{align*}
|{\rm J}_2| & \leq C \exp\left(\alpha \frac{|z|^2}{T}\right) \,  \int_{|y-v| \geq 1} (s-t)^{-\frac{1+n}{2}} (1+ |y|^2 + |v|^2+ M^{q}_2([X^{t, \xi}_s])) \, g(c(s-t), y-v) \, dy \,  \\
& \leq C \exp\left(\alpha \frac{|z|^2}{T}\right) \,  (s-t)^{\frac{-1-n +\eta}{2}} (1+ |v|^{2} + M^{q}_2(\mu)). 
\end{align*}

Also, from \eqref{first:second:lions:derivative:mckean:decoupling} and \eqref{growth:condition:tilde:h:f}, we derive
\begin{align*}
\Big|  \int_{(\mathbb{R}^d)^2} \frac{\delta f}{\delta m}(s, z, [X^{t, \xi}_s])(y)  \, \partial^{n}_v[\partial_\mu p(\mu, t, s, x, y)](v) \, dy \, \mu(dx) \Big| \leq C \exp\left(\alpha \frac{|z|^2}{T}\right) \,  (s-t)^{\frac{-1-n  + \eta}{2}} (1+ M^{q}_2(\mu)).
\end{align*}

\noindent Gathering the previous estimates, we obtain
\begin{align}
|\partial^{n}_v[\partial_\mu [f(s, z, [X^{t, \xi}_s])]](v)| \leq C \exp\left(\alpha \frac{|z|^2}{T}\right) \,  (s-t)^{\frac{-1-n+\eta}{2}} (1+ |v|^2 + M^{q}_2(\mu)). \label{bound:mes:deriv:source}
\end{align}

From \eqref{deriv:measure:h}, \eqref{growth:condition:tilde:h:f}, the estimates \eqref{gradient:estimates:parametrix:series}, \eqref{growth:condition:tilde:h:f} and similar computations
\begin{align}
|\partial^{n}_v[\partial_\mu [h(z, [X^{t, \xi}_T])]](v)| \leq C \exp\left(\alpha \frac{|z|^2}{T}\right) \,  (T-t)^{-\frac{(1+n)}{2}} (1+ |v|^2 + M^{q}_2(\mu)). \label{bound:mes:deriv:terminal:condition}
\end{align}

The estimates \eqref{gradient:estimates:parametrix:series}, \eqref{first:second:lions:derivative:mckean:decoupling}, \eqref{bound:mes:deriv:source} and \eqref{bound:mes:deriv:terminal:condition} allow to conclude that if $\alpha < (2c)^{-1}$, the constant $c$ being the maximum among the constants $c$ appearing in the estimates \eqref{gradient:estimates:parametrix:series} and \eqref{first:second:lions:derivative:mckean:decoupling}, then the map $(x, \mu) \mapsto U(t, x, \mu)$ is in $\mathcal{C}^{2, 2}( \rr^d \times \pp)$ with derivatives given by

\begin{align}
\partial^{n}_v[\partial_\mu U(t, x, \mu)](v)  & = \int_{\mathbb{R}^d} h(z, [X^{t,\xi}_T]) \, \partial^{n}_v[\partial_\mu p(\mu, t, T, x, z)](v) \, dz  + \int_{\rr^d} \partial^{n}_v[\partial_\mu [h(z, [X^{t, \xi}_T])]](v) \, p(\mu, t, T, x, z) \, dz  \nonumber \\
& \quad  - \int_t^T  \int_{\mathbb{R}^d} \partial^{n}_v[\partial_\mu [f(s, z, [X^{t, \xi}_s])]](v) \, p(\mu, t, s, x, z) \, dz \, ds \label{mes:derivative:U} \\
& \quad - \int_t^T \int_{\mathbb{R}^d} f(s, z, [X^{t,\xi}_s]) \, \partial^{n}_v [\partial_\mu p(\mu, t, s, x, z)](v) \, dz \, ds \nonumber
\end{align}

\noindent for $n=0, \,1$ and
\begin{align}
\partial^{n}_x U(t, x, \mu)  & = \int_{\mathbb{R}^d} h(z, [X^{t,\xi}_T]) \, \partial^{n}_x p(\mu, t, T, x, z) \, dz   \nonumber \\
& \quad  - \int_t^T  \int_{\mathbb{R}^d} [ f(s, z, [X^{t, \xi}_s]) - f(s, x, [X^{t, \xi}_s]) ] \, \partial^{n}_x p(\mu, t, s, x, z) \, dz \, ds \label{space:derivative:U} 
\end{align}

\noindent for $n= 1, \, 2$. Note that we may break the last integral appearing in the right-hand side of \eqref{space:derivative:U} into two parts by dividing the domain of integration into two domains as we did before. Then, using \eqref{local:holder:reg:f}, that is, the local $\eta$- H\"older continuity of $ f(s, ., m)$, \eqref{growth:condition:h:f} and the estimate \eqref{gradient:estimates:parametrix:series}, we get
\begin{align*}
\Big| \int_{\rr^d} & [ f(s, z, [X^{t, \xi}_s]) - f(s, x, [X^{t, \xi}_s]) ] \, \partial^{1+n}_x p(\mu, t, s, x, z) \, dz \Big| \\
& \leq C (s-t)^{\frac{-1-n + \eta}{2}} \left\{ \int_{\rr^d} e^{\alpha \frac{|z|^2}{T}} g(c(s-t), z-x) \, dz + e^{\alpha \frac{|x|^2}{T}} \right\} (1+M^{q}_2(\mu))\\
& \leq C (s-t)^{ \frac{-1-n + \eta}{2}}  e^{k \frac{|x|^2}{T}} (1+M^{q}_2(\mu))
\end{align*}

\noindent for some positive constant $k:=k(c, \alpha)$, $\alpha \mapsto k(c, \alpha)$ being non-decreasing, where for the last inequality we used the fact that the constant $\alpha$ is sufficiently small, recall that $\alpha < (2c)^{-1}$, $c$ being the constant appearing in \eqref{gradient:estimates:parametrix:series} and the inequality: for any positive constants $\alpha$ and $c'$ satisfying $0< \alpha< c'$, there exists a positive constant $C:=C(c', \alpha)$ (take e.g. $C = c' \alpha/(c'-\alpha)$) such that for any $(z, x)\in (\rr^d)^2$,  
\begin{equation}\label{simple:ineq:quadratic}
\alpha |z|^2 - c' |z-x|^2 \leq C |x|^2.
\end{equation}

 The previous estimate as well as \eqref{bound:mes:deriv:terminal:condition}, \eqref{bound:mes:deriv:source}, \eqref{gradient:estimates:parametrix:series} and \eqref{first:second:lions:derivative:mckean:decoupling} ensure that the integrals appearing in \eqref{mes:derivative:U} and \eqref{space:derivative:U} are well defined if $\alpha < (2c)^{-1}$, where we recall that $c$ is the maximum among the constants $c$ appearing in the estimates \eqref{gradient:estimates:parametrix:series} and \eqref{first:second:lions:derivative:mckean:decoupling}. We thus conclude from \eqref{mes:derivative:U}, \eqref{space:derivative:U} and the continuity of the coefficients $b_i$, $a_{i, j}$ that $[0,T) \times \rr^d \times \mathcal{P}_2(\rr^d) \ni (t, x,\mu) \mapsto \mathcal{L}_t U(t, x,\mu)$ is continuous.

Finally, from \eqref{bound:density:parametrix}, \eqref{growth:condition:h:f} and \eqref{simple:ineq:quadratic}, setting e.g. $k = k(c, \alpha) := (2c)^{-1} \alpha/((2c)^{-1}-\alpha)$, we get
\begin{align*}
|U(t, x, \mu)| & \leq C \left\{\int_{\rr^d} \exp\Big(\alpha \frac{|z|^2}{T}\Big) g(c(T-t), z-x) \, dz + \int_t^T\int_{\rr^d} \exp\Big(\alpha \frac{|z|^2}{T}\Big) g(c(s-t), z-x) \, dz \, ds \right\} \\
& \quad \times (1+ M^{q}_2(\mu)) \\
& \leq C \exp\Big(\frac{k |x|^2}{T}\Big) (1+M^q_2(\mu))
\end{align*}

\noindent and the proof of \eqref{bound:space:deriv:sol:pde} follows similarly by combining \eqref{space:derivative:U} with \eqref{gradient:estimates:parametrix:series} (with $n=1$), \eqref{growth:condition:h:f} and \eqref{simple:ineq:quadratic}. The proof of \eqref{bound:mes:deriv:sol:pde} also follows from similar arguments using \eqref{mes:derivative:U}, \eqref{growth:condition:h:f}, \eqref{bound:mes:deriv:source}, \eqref{bound:mes:deriv:terminal:condition} and \eqref{simple:ineq:quadratic}.

Let us now prove that $U$ is in $\mathcal{C}^{1, 2, 2}([0,T) \times \rr^d \times \pp)$. From the Markov property satisfied by the SDE \eqref{SDE:MCKEAN} (which is inherited from the well-posedness of the associated martingale problem) we obtain the following identity
\begin{align*}
U(t-h,x,\mu) = \E\left[ U(t,X_t^{t-h,x,\mu},[X_t^{t-h,\xi}]) - \int_{t-h}^t f(r, X^{t-h, x, \mu}_r, [X^{t-h, \xi}_r]) dr\right], \quad 0 < h < t.
\end{align*} 

 The chain rule formula of Proposition \ref{prop:chain:rule:joint:space:measure} (with respect to the space and measure variables only) together with the estimate \eqref{bound:mes:deriv:sol:pde} yield
\begin{align*}
U(t,X_t^{t-h, x, \mu},[X_t^{t-h,\xi}]) & = U(t, x, \mu) + \int_{t-h}^{t} \mathcal{L}_r U(t, X_r^{t-h, x,\mu},[X_r^{t-h,\xi}])  dr \\
& + \int_{t-h}^{t} \partial_x U(t, X_r^{t-h, x,\mu}, [X_r^{t-h,\xi}]) . \sigma(r, X_r^{t-h, x,\mu}, [X_r^{t-h,\xi}]) dW_r.
\end{align*}

From \eqref{bound:space:deriv:sol:pde} and \eqref{bound:density:parametrix}, taking $h$ small enough, it follows that the last term appearing in the right-hand side of the previous equality is a square integrable martingale. Hence,
\begin{align*}
\E\left[U(t,X_t^{t-h, x, \mu},[X_t^{t-h,\xi}])\right] = U(t, x, \mu) + \E\left[ \int_{t-h}^t \mathcal{L}_r U(t, X_r^{t-h, x,\mu},[X_r^{t-h,\xi}])  dr\right]
\end{align*}

\noindent and
\begin{eqnarray*}
\frac 1h \left(U(t-h,x,\mu) - U(t,x,\mu) \right) = \frac 1h  \E\left[ \int_{t-h}^t \left\{\mathcal{L}_rU(t, X_r^{t-h , x, \mu},[X_r^{t-h,\xi}]) - f(r, X^{t-h, x, \mu}_r, [X^{t-h, \xi}_r]) \right\} dr \right] 
\end{eqnarray*} 

\noindent so that letting $h\downarrow 0$, from the boundedness and continuity of the coefficients $b_i$, $a_{i, j}$ and $f$, we deduce that $U$ is left differentiable in time at any time $t\in [0,T)$. Still from the continuity of the coefficients and $f$, we then conclude that it is differentiable in time with  
 \begin{eqnarray*}
\partial_t U(t, x, \mu) = -\mathcal{L}_tU(t, x,\mu) + f(t, x,\mu).
\end{eqnarray*}

Hence, the map $U$ solves the PDE \eqref{pde:wasserstein:space}. 
\end{proof}

In order to get the uniqueness result of Theorem \ref{cauchy:problem:wasserstein}, first fix any $0\leq t \leq s <T$ and consider any solution $V$ to the Cauchy problem \eqref{pde:wasserstein:space} satisfying \eqref{cond:integrab:ito:process:second:version} on any interval $[0,T']$, with $T'<T$, as well as \eqref{bound:solution:cauchy:problem}. We apply the chain rule formula of Proposition \ref{prop:chain:rule:joint:space:measure} to $\left\{ V(s, X^{t, x, \mu}_{s}, [X^{t, \xi}_s]), \,  t \leq s  < T  \right\}$ and use the fact that $(\partial_t + \mathcal{L}_t)V(t, x, \mu)  = f(t, x, \mu)$, for any $(t, x, \mu) \in [0,T) \times \rr^d \times \pp$, to get
\begin{align*}
V(s, X^{t, x, \mu}_s, [X^{t, \xi}_s]) &  = V(t, x, \mu) + \int_t^s f(r, X^{t, x, \mu}_r, [X^{t, \xi}_r]) \, dr  \\
& \quad  + \int_{t}^s \sum_{i=1}^d \sum_{j=1}^q \sigma_{i, j}(r, X^{t, x, \mu}_r, [X^{t,\xi}_r]) \, \partial_{x_i}V(r, X^{t, x, \mu}_r, [X^{t, \xi}_r]) dW^{j}_r.
\end{align*}

The local martingale appearing in the right-hand side of the above equality is in fact a true martingale since $V(s, X^{t, x, \mu}_s, [X^{t, \xi}_s])$ and $\int_t^s f(r, X^{t, x, \mu}_r, [X^{t, \xi}_r]) \, dr $ are both square integrable if the constants $\alpha$ and $k$ appearing in the two conditions \eqref{growth:condition:h:f} and \eqref{bound:solution:cauchy:problem} are small enough. Namely, it is sufficient to take $\alpha$ and $k$ strictly less than $(4c)^{-1}$, $c:=c(\lambda)$ being the constant appearing in \eqref{bound:density:parametrix}.

Hence, taking expectation in the previous equality, then passing to the limit as $s\uparrow T$ and finally using the continuity assumption at the boundary, we obtain
$$
V(t, x, \mu) = \E\left[h(X^{t, x, \mu}_T, [X^{t, \xi}_T]) - \int_{t}^T f(r, X^{t, x, \mu}_r, [X^{t, \xi}_r]) \, dr\right]  
$$

\noindent which completes the proof of Theorem \ref{cauchy:problem:wasserstein}.

\appendix
\section{Proof of Proposition \ref{proposition:reg:density:recursive:scheme:mckean}}\label{proof:main:prop} 
This appendix is dedicated to the proof of Proposition \ref{proposition:reg:density:recursive:scheme:mckean}. While relying on somehow classical Gaussian like computations, the proof appears to be quite long and technical. In order to focus on the main steps of the proof, part of intermediate results are collected into several technical lemmas and associated corollaries. Their proof are postponed to the second part of this appendix, see Appendix \ref{appendix}.

This section is thus organized as follows: the first section  \ref{subsubsection:base:case} deals with the base case $(m=1)$, then the first part of the induction step (namely the estimates \eqref{first:second:estimate:induction:decoupling:mckean} to \eqref{equicontinuity:first:estimate:decoupling:mckean}) is treated in Section \ref{subsubsection:proof:first:part:induction:step}. Then, as a consequence of the first part, we prove the estimates \eqref{regularity:measure:estimate:v1:v2:v3:decoupling:mckean} and \eqref{regularity:time:estimate:v1:v2:v3:decoupling:mckean} in Section \ref{proofs:intermediate:estimates:as:consequence}. We eventually address the second part of the induction step, namely the estimates \eqref{regularity:measure:estimate:v1:v2:decoupling:mckean} and \eqref{regularity:time:estimate:v1:v2:decoupling:mckean} in Section \ref{second:part:induction:step}.\\

%
%
%
%

\smallskip

\noindent \textbf{ Some additional notations.} In order to simplify the notations, in the following, we will denote by $K$ a generic constant that can depend on $T$
 and the parameters in \HR\, and \HE\, but does not depend on $m$ or the constant $C$. We will proceed similarly and denote by $K^{+}$ a generic constant that can depend on $T$ and the parameters in \HRp\, and \HE. 
We reserve the notation $c$ for a constant that depends only on $\lambda$ and $d$. In particular, the constants $K, K^{+}$ and $c$ are uniform with respect to $m$ and the constants $C$, $C_\beta$ and $C^{+}_\beta$ that appear in the estimates \eqref{first:second:estimate:induction:decoupling:mckean} to \eqref{regularity:time:estimate:v1:v2:decoupling:mckean} and their values may change from line to line. We will emphasize the dependence of the constants $K$ or $K^{+}$ with respect to a prescribed parameter $\beta$ by writing $K_\beta$ or $K_\beta^{+}$.  

\smallskip

Apart from section \ref{subsubsection:base:case} which concerns the base case, we will work under the following assumption. For a fixed positive time horizon $T>0$ and positive integer $m$, we assume that for any fixed $(t, z) \in (0,T] \times \mathbb{R}^d$, the map $(s, x, \mu) \mapsto p_m(\mu, s, t, x, z)$ defined by \eqref{series:approx:mckean} belongs to $\mathcal{C}^{1,2, 2}([0,t) \times \mathbb{R}^d \times \pp)$ and denote by $[X^{s, \xi, (m)}_t]$ the probability measure on $\mathbb{R}^d$ with density function $z\mapsto (p_m(\mu, t, T, ., z) \sharp\mu)$.\\

\subsection{Base case $m=1$.\\ }\label{subsubsection:base:case}

\noindent \emph{Step 1: $(s, x, \mu) \mapsto p_1(\mu, s, t, x, z) \in \mathcal{C}^{1, 2, 2}([0,t)\times \rr^d \times \pp)$.\\ }
Let us first observe that since $P^{(0)}(t) = \nu$ for any $t\geq s$, it is readily seen from \eqref{definition:general:phat}, \eqref{definition:phat:end:point:frozen} and \eqref{definition:parametrix:kernel:iterate:m} for $m=1$ that the law argument in the coefficients depends neither on the initial measure $\mu$ nor on the initial time $s$ but only on $\nu$. From \cite{friedman:64}, we thus conclude that the map $[0,t) \times \rr^d \ni (s, x) \mapsto p_1(\mu, s, t, x, z) =  \sum_{k \geq 0} (\widehat{p}_{1} \otimes \mH^{(k)}_1)(\mu, s, t, x, z)$ belongs to $\mathcal{C}^{1, 2}([0,t) \times \rr^d)$ with derivatives that do not depend on $\mu$. Obviously, the map $\pp \ni \mu \mapsto p_1(\mu, s, t, x, z)$ is continuously $L$-differentiable and satisfies $\partial_\mu p_1(\mu, s, t, x, z)(v) = \partial_v [\partial_\mu p_1(\mu, s, t, x, z)](v) = 0 $ for any $(s, x ,\mu, v) \in [0,t) \times \rr^d \times \pp \times \rr^d$. We thus conclude that the map $[0,t) \times \rr^d \times \pp \ni (s, x, \mu) \mapsto p_1(\mu, s, t, x, z)$ is in $\mathcal{C}^{1, 2, 2}([0,t)\times \rr^d \times \pp)$.\\

\noindent \emph{Step 2: Gaussian estimates on the derivatives of the map $(s, x, \mu) \mapsto p_1(\mu, s, t, x, z)$.\\}
According to the preceding discussion, the estimates \eqref{first:second:estimate:induction:decoupling:mckean}, \eqref{equicontinuity:second:third:estimate:decoupling:mckean} to \eqref{regularity:measure:estimate:v1:v2:decoupling:mckean} and \eqref{regularity:time:estimate:v1:v2:decoupling:mckean} are straightforward. The estimate \eqref{time:derivative:induction:decoupling:mckean} is a direct consequence of \eqref{bound:derivative:heat:kernel} and the fact that $[0,t) \times \rr^d \ni (s, x) \mapsto p_1(\mu, s, t, x, z)$ is the fundamental solution to the backward Kolmogorov PDE associated to the SDE \eqref{iter:mckean:decoupl} with $m=0$, see e.g. \cite{friedman:64} and \cite{Friedman2011}. The estimate \eqref{regularity:time:estimate:v1:v2:v3:decoupling:mckean} for $m=1$ is a consequence of the estimate (3.33) of Theorem 3.5 in Garroni and Menaldi \cite{1992green} in the case $|s_1 - s_2|\leq  (t- s_1 \vee s_2)/2$ and a consequence of the Gaussian estimate \eqref{bound:derivative:heat:kernel} if $|s_1-s_2| > (t-s_1 \vee s_2)/2$. This concludes the proof of the base case.\\

\subsection{Some preparatory technical results}\label{subsubsection:technical:lemmas}
 To proceed with our induction procedure, we have to prove that the statements obtained in the base case $m=1$ indeed propagate at step $m+1$ provided they hold at step $m$. To do so, we start with the process $(X^{s,\xi,(m+1)}_t, t\in [s, T])$ with dynamics given by \eqref{iter:mckean} and coefficients frozen in their measure argument at the law of the Picard iteration scheme at step $m$. The key observation is that the density function of the random variable $X^{s, \xi, (m+1)}_t$ denoted by $z \mapsto p_{m+1}(\mu, s, t, z)$ satisfies the relation \eqref{conv:relation:step:m} where $z\mapsto p_{m+1}(\mu, s, t, x, z)$ denotes the transition density of the decoupling SDE $(X^{s, x, \mu, (m+1)}_t , t\in [s, T])$. 

As already emphasized, the crucial point is that this transition density enjoys a representation in infinite series given by \eqref{series:approx:mckean} which involves space-time iterated convolutions of the so-called parametrix kernel $\mathcal H_{m+1}$ given by \eqref{definition:parametrix:kernel:iterate:m} against the Gaussian type kernel $\widehat p_{m+1}$ given by \eqref{definition:general:phat}. As suggested by the indices $m$ in the notation, the point is that such quantities depend on the density $p_{m}$ built at the previous step of the Picard iteration scheme, so that, when investigating the $\mathcal{C}^{1, 2, 2}([0,t)\times \rr^d \times \pp)$ smoothness of $p_{m+1}$ and its related estimates, we will be lead to handle those terms. Namely, as a preparatory step of our induction argument, we need to investigate the regularity properties and to establish some adequate estimates for the coefficients $b_i(t, x, [X^{s, \xi, (m)}_t]),\, a_{i, j}(t, x, [X^{s, \xi, (m)}_t])$, the Gaussian type kernel $\widehat p_{m+1}$, the parametrix kernel $\mH_{m+1}$ and its iterated space time convolution $\mH^{(k)}_{m+1},\, k\geq 1$ defined just after \eqref{definition:parametrix:kernel:iterate:m}) in order to prove that the estimates in Proposition \ref{proof:main:prop} indeed propagates from one step to another. 

This is the purpose of this part and the associated results are respectively given by Lemma \ref{lem:diff:and:control:deriv:coeff} and Corollaries \ref{cor:deriv:time:and:mes:phat}, \ref{cor:deriv:time:and:mes:parametrix:kernel} and \ref{cor:mes:time:deriv:iterated:parametrix:kernel}. As the proofs rely on rather classical but sometimes tricky Gaussian types computations (sometimes hinted in the part of the current proof), we decided to respectively postpone their proofs to Sections \ref{section:proof:lem:diff:and:control:deriv:coeff}, \ref{section:proof:cor:deriv:time:and:mes:phat}, \ref{section:proof:cor:deriv:time:and:mes:parametrix:kernel} and \ref{section:proof:cor:mes:time:deriv:iterated:parametrix:kernel}.

%
\begin{lem}\label{lem:diff:and:control:deriv:coeff}
 For any fixed $(t, x) \in (0, T] \times \mathbb{R}^d$ and any $(i, j) \in \left\{1, \cdots, d\right\}^2$, the maps $(s, \mu) \mapsto b_i(t, x, [X^{s, \xi, (m)}_t]), \, a_{i, j}(t, x, [X^{s, \xi, (m)}_t])$ belong to $\mathcal{C}^{1, 2}([0,t)\times \pp)$ and satisfy the following estimates: for any $\beta \in [0,\eta)$ and any $\beta' \in [0,1]$, there exist positive constants $K$, $K_\beta$, $K_{\beta'}$ such that for any $(t, x, z) \in (0,T] \times (\mathbb{R}^d)^2$, for any $(s, \mu , v, v')\in [0,t) \times \pp \times (\mathbb{R}^d)^2$ and any $(i, j) \in \left\{1, \cdots, d\right\}$
\begin{align}
| \partial^n_v [\partial_\mu & [b_i(t, x, [X^{s,\xi, (m)}_t])]](v) |  + |\partial^n_v [\partial_\mu [a_{i, j}(t, x, [X^{s,\xi, (m)}_t])]](v)| \nonumber \\
&  \leq K \left\{ \frac{1}{(t-s)^{\frac{1+n - \eta}{2}}} + \int_{(\mathbb{R}^d)^2} (|y-x'|^{\eta}\wedge 1) | \partial^{n}_v[\partial_\mu p_{m}(\mu, s, t, x', y)](v)| \, \mu(dx') \, dy   \right\}, \label{recursive:bound:deriv:a:or:b}
\end{align}
\begin{align}
| & \partial_v[\partial_\mu [b_{i}(t, x, [X^{s, \xi, (m)}_{t}])]](v)  - \partial_v[\partial_\mu [b_{i}(t, x, [X^{s, \xi, (m)}_{t}])]](v')| \nonumber \\
& \quad \quad + |  \partial_v[\partial_\mu [a_{i, j}(t, x, [X^{s, \xi, (m)}_{t}])]](v)  - \partial_v[\partial_\mu [ a_{i, j}(t, x, [X^{s, \xi, (m)}_{t}])]](v')| \nonumber \\
& \leq K_\beta \left\{ \frac{|v-v'|^\beta}{(t-s)^{1+\frac{\beta-\eta}{2}}} \right.  \label{recursive:bound:deriv:mes:reg:holder:a:or:b} \\
& \quad\quad \left. + \int_{(\mathbb{R}^d)^2} (|y-x'|^\eta \wedge 1) | | \partial_v[\partial_\mu p_{m}(\mu, s, t, x', y)](v) - \partial_v[\partial_\mu p_{m}(\mu, s, t, x', y)](v')| \, \mu(dx') \, dy \right\}, \notag
\end{align}
\begin{align}
|  & \partial^{n}_v [\partial_\mu [a_{i, j}(t, x, [X^{s, \xi, (m)}_{t}]) - a_{i, j}(t, z, [X^{s, \xi, (m)}_t])]](v) | \nonumber \\
 &\quad \quad \leq  K_{\beta'} |z-x|^{\beta' \eta} \left\{ \frac{1}{(t-s)^{\frac{1+n-(1-\beta')\eta}{2}}} \right. \label{recursive:bound:deriv:mes:holder:reg:a} \\
 & \left. \quad \quad \quad +  \int_{(\mathbb{R}^d)^2} (|y-x'|^{(1-\beta')\eta} \wedge 1) |\partial^{n}_v[\partial_\mu p_{m}(\mu, s, t, x', y)](v)| \, \mu(dx')\, dy\right\}, \nonumber
\end{align}
\begin{align}
|   \partial^{n}_v [\partial_\mu & [a_{i, j}(t, x, [X^{s, \xi, (m)}_{t}])  - a_{i, j}(t, z, [X^{s, \xi, (m)}_t])]](v) -  [\partial^{n}_v [\partial_\mu [a_{i, j}(t, x, [X^{s, \xi, (m)}_{t}]) - a_{i, j}(t, z, [X^{s, \xi, (m)}_t])]](v') ] | \nonumber \\
 &\quad  \leq  K_\beta \left\{\frac{(|z-x|^\eta \wedge 1)}{(t-s)^{1+\frac{\beta}{2}}} \wedge \frac{1}{(t-s)^{1+\frac{\beta-\eta}{2}}} \right\} \left\{ {|v-v'|^\beta} \right.  \label{recursive:bound:deriv:mes:double:reg:holder:a} \\
 & \quad \quad \left. + (t-s)^{1+\frac{\beta-\eta}{2}} \int_{(\mathbb{R}^d)^2} (|y'-x'|^\eta \wedge 1) |\partial_v [\partial_{\mu}p_{m}(\mu, s , t, x', y')](v) - \partial_v[\partial_{\mu}p_{m}(\mu, s , t, x', y')](v')| \, \mu(dx') \, dy' \right. \nonumber \\
 & \quad \quad \left.  + (t-s)^{1+\frac{\beta}{2}}\int_{(\mathbb{R}^d)^2} \Big|\partial_v\Big[\partial_{\mu}p_{m}(\mu, s , t, x', y')\Big](v) - \partial_v\Big[\partial_{\mu}p_{m}(\mu, s , t, x', y')\Big](v')\Big| \, \mu(dx') \, dy' \right\}, \nonumber
\end{align}

\begin{align}
|\partial_s [b_i (t, x, [X^{s,\xi, (m)}_t])]|  &+ |\partial_s[a_{i, j} (t, x, [X^{s,\xi, (m)}_t])]|\notag \\
&  \leq K \int_{(\mathbb{R}^d)^2} (|y-x'|^\eta \wedge 1) |\partial_s p_m(\mu, s, t,  x', y)| \, \mu(dx') \, dy, \label{recursive:bound:time:deriv:a:or:b}
\end{align}
\begin{align}
 | \partial_s [a_{i, j} (t, x, [X^{s,\xi, (m)}_t]) & -  a_{i, j} (t, z, [X^{s,\xi, (m)}_t])] |\notag  \\
 & \leq  K |z-x|^{\beta' \eta} \int_{(\mathbb{R}^d)^2} (|y-x'|^{(1-\beta')\eta} \wedge 1) |\partial_s p_m(\mu, s, t,  x', y)| \, \mu(dx') \, dy. \label{recursive:bound:time:deriv::holder:reg:a}
\end{align}
\end{lem}

Let us also importantly point out that under the assumptions of Lemma \ref{lem:diff:and:control:deriv:coeff} and if the estimates \eqref{first:second:estimate:induction:decoupling:mckean} and \eqref{time:derivative:induction:decoupling:mckean} are satisfied at step $m$ then from the space-time inequality \eqref{space:time:inequality}, the following estimates for the derivatives of the maps $(s, \mu)\mapsto b_i(t, x, [X^{s, \xi, (m)}_t])$, $a_{i, j}(t, x, [X^{s, \xi, (m)}_t])$ hold
\begin{align}
| \partial^n_v [\partial_\mu  [b_i(t, x, [X^{s,\xi, (m)}_t])]](v) |  & + |\partial^n_v [\partial_\mu [a_{i, j}(t, x, [X^{s,\xi, (m)}_t])]](v)| \notag\\
& \leq K  (t-s)^{-\frac{1+n - \eta}{2}} (1+ \mathscr{C}^{n,0}_m(C, t-s) (t-s)^{\frac{\eta}{2}}) \label{time:degeneracy:estimate:deriv:mes:coeff}
\end{align}
\noindent and 
\begin{align}\label{time:degeneracy:estimate:deriv:time:coeff}
|\partial_s [b_i(t, x, [X^{s, \xi, (m)}_t])]| + | \partial_s[a_{i, j}(t, x, [X^{s, \xi, (m)}_t])]| \leq K \mathscr{C}^{n, 1}_m(C, t-s)  (t-s)^{-1 + \frac{\eta}{2}}
\end{align}
\noindent up to a modification of the constant $K$.
\begin{cor}\label{cor:deriv:time:and:mes:phat} Assume that the estimates \eqref{first:second:estimate:induction:decoupling:mckean} and \eqref{time:derivative:induction:decoupling:mckean} (at step $m$) are satisfied for some positive constant $C$. Then, for any $(t, y, z) \in (0,T] \times (\mathbb{R}^d)^2$ and any $r \in (0,t) $, the maps $(s, x, \mu) \mapsto \widehat{p}^{y}_{m+1}(\mu, s, r,  t , x, z)$, $\widehat{p}^{y}_{m+1}(\mu, s,  t , x, z) = \widehat{p}^{y}_{m+1}(\mu, s, s,  t , x, z)$ belong to $\mathcal{C}^{1, 0, 2}([0,r)\times \mathbb{R}^d \times \pp)$ and $\mathcal{C}^{1, 0,  2}([0,t)\times \mathbb{R}^d \times \pp)$ respectively with continuous derivatives with respect to the variables $s$, $x$, $\mu$ and $v$.\\
Moreover, the derivatives satisfy the following pointwise Gaussian estimates: there exist positive constants $K$ and $c$ such that for any $( x, z, y, v, \mu)\in (\mathbb{R}^d)^4 \times \pp$ and any $0\leq s \leq r < t \leq T$
\begin{align}
| \partial^{n}_v & [\partial_\mu \widehat{p}^{y}_{m+1}(\mu, s, r, t, x, z)](v)| \notag \\
 & \leq \frac{K}{t-r} \left\{ \int_r^t \frac{1}{(r'-s)^{\frac{1+n - \eta}{2}}} \,dr'  \right. \label{cross:mes:deriv:p:hat:s:r:t}\\
 & \quad \quad \left. +  \int_r^t  \int_{(\mathbb{R}^d)^2} (|y'-x'|^{\eta}\wedge 1) | \partial^{n}_v[\partial_\mu p_{m}(\mu, s, r', x', y')](v)| \, \mu(dx') \, dy' \, dr' \right\}  g(c(t-r), z-x),  \notag
\end{align}
\begin{align}
|\partial_s & \widehat{p}^{y}_{m+1}(\mu, s, r, t, x, z) |  \notag \\
&  \leq \frac{K}{t-r}\int_r^t  \int_{(\rr^d)^2} (|y'-x'|^{\eta}\wedge 1) | \partial_s p_{m}(\mu, s, r', x', y') | \, \mu(dx') \, dy' \, dr'  g(c(t-r), z-x), \, r \neq s, \label{time:deriv:p:hat:s:r:t}
\end{align}
\begin{align}
|\partial_s & \widehat{p}^{y}_{m+1}(\mu, s, t, x, z) |  \notag \\
&  \leq \frac{K}{t-s} \left\{ 1 + \int_s^t  \int_{(\rr^d)^2} (|y'-x'|^{\eta}\wedge 1) | \partial_s p_{m}(\mu, s, r', x', y') | \, \mu(dx') \, dy' \, dr' \right\} \label{time:deriv:p:hat:s:t}\\
& \quad \quad \times g(c(t-s), z-x). \notag
\end{align}

For any $\beta \in [0,1]$ and any $\beta' \in [0,\eta)$, there exist positive constants $K$, $K_{\beta'}$ and $c$ such that for any $(\mu, x, z, y, v, v')\in (0,T] \times \pp \times (\mathbb{R}^d)^5$, for any $0\leq s \leq r < t \leq T$ and any $(x_1, x_2),  (y_1, y_2) \in (\mathbb{R}^d)^2$ 
\begin{align}
| \partial^{n}_v&  [\partial_\mu \widehat{p}^{y}_{m+1}(\mu, s, t, x_1, z)](v) -  \partial^{n}_v [\partial_\mu \widehat{p}^{y}_{m+1}(\mu, s, t, x_2, z)](v)| \notag \\
&   \leq K \frac{|x_1-x_2|^\beta}{(t-s)^{\frac{\beta}{2}}}\left\{ \frac{1}{(t-s)^{\frac{1+n-\eta}{2}}} + \frac{1}{t-s} \int_s^t \int_{(\rr^d)^2} (|y'-x'|^{\eta} \wedge 1)   |\partial^{n}_v[\partial_\mu p_{m}(\mu, s, r', x', y')](v)| \,  \mu(dx') \, dy' \, dr' \right\} \label{cross:mes:deriv:holder:p:hat:s:t}\\
& \quad \quad \times \left\{ g(c(t-s) , z-x_1) + g(c(t-s) , z-x_2)\right\}, \notag
\end{align}
\begin{align}
|\partial_s & \widehat{p}^{y_1}_{m+1}(\mu, s, t, x, z) - \partial_s  \widehat{p}^{y_2}_{m+1}(\mu, s, t, x, z) |  \notag \\
&  \leq K  \frac{|y_1-y_2|^{\beta \eta}}{t-s} \left\{ 1 + \int_s^t  \int_{(\rr^d)^2} (|y'-x'|^{(1-\beta)\eta}\wedge 1) | \partial_s p_{m}(\mu, s, r', x', y') | \, \mu(dx') \, dy' \, dr' \right\} \label{time:deriv:holder:reg:p:hat:s:t}\\
& \quad \quad \times g(c(t-s), z-x), \notag
\end{align}
\begin{align}
 | \partial_v & [\partial_\mu \widehat{p}^y_{m+1}(\mu, s, r, t, x, z)](v)   -   \partial_v [\partial_\mu \widehat{p}^y_{m+1}(\mu, s, r, t, x, z)](v') | \notag \\
& \leq \frac{K_{\beta'}}{t - r}  \left\{ \int_{r}^{t} \Big[ \frac{|v-v'|^{\beta'}}{(r'-s)^{1+ \frac{\beta'-\eta}{2}}}\right. \notag\\
&  \quad  \quad \left.+ \int_{(\mathbb{R}^d)^2} (|y'-x'|^{\eta}\wedge 1) |\partial_v [\partial_\mu p_{m}(\mu, s, r', x', y')](v) - \partial_v [\partial_\mu p_{m}(\mu, s, r', x', y')](v') | \, \mu(dx')\, dy' \,  \Big] dr' \right\}   \label{cross:mes:deriv:reg:holder:terminal point:p:hat:s:r:t} \\
& \quad \times g(c(t-r), z- x). \nonumber
\end{align}
\end{cor}

\begin{cor}\label{cor:deriv:time:and:mes:parametrix:kernel} Assume that the estimates \eqref{first:second:estimate:induction:decoupling:mckean} and \eqref{time:derivative:induction:decoupling:mckean} are satisfied (at step $m$) for some positive constant $C$. For any $(t, x, z) \in (0,T] \times (\mathbb{R}^d)^2$ and any $r\in (0,t)$, the map $[0,r) \ni (s, \mu) \mapsto \mH_{m+1}(\mu, s, r, t, x, z)$ belongs to $\mathcal{C}^{1, 2}([0,r)\times \pp)$, its derivatives $\partial_s  \mH_{m+1}(\mu, s, r, t, x, z)$, $\partial_v^{n}[ \partial_\mu[ \mH_{m+1}(\mu, s, r, t, x, z)]](v)$, $n=0, 1$ being continuous with respect to the variables $s$, $x$, $\mu$ and $v$

Moreover, the derivatives satisfy the following pointwise Gaussian estimates: for any $\beta \in [0,1]$ and any $\beta' \in [0,\eta)$, there exist positive constant $K_{\beta}$, $K_{\beta'}$ and $c$ such that for any $(t, x, z)\in (0,T] \times (\mathbb{R}^d)^2$, for any $s \in [0, t)$, for any $r\in (s, t)$ and any $(\mu, v, v') \in \pp \times (\mathbb{R}^d)^2$
\begin{align}
| \partial^{n}_v & [\partial_\mu \mH_{m+1}(\mu, s, r, t, x, z)](v)| \notag \\
& \leq \frac{K_\beta}{(t-r)^{1-\frac{\beta\eta}{2}}(r-s)^{\frac{1+n-(1-\beta)\eta}{2}}}  \label{cross:mes:deriv:parametrix:kernel:s:r:t:with:beta} \\
  & \quad \times \left(1 +  (r-s)^{\frac{1+n-(1-\beta)\eta}{2}} \int_{(\mathbb{R}^d)^2} (|y'-x'|^{(1-\beta)\eta}\wedge 1) |\partial^{n}_v[\partial_\mu p_{m}(\mu, s, r, x', y')](v)| \, \mu(dx')\, dy'  \right. \notag \\
& \quad \quad \left. + \frac{(r-s)^{\frac{1+n-(1-\beta)\eta}{2}}}{(t-r)^{1+\frac{(\beta-1)\eta}{2}}} \int_r^t \int_{(\mathbb{R}^d)^2} (|y'-x'|^{\eta}\wedge 1) |\partial^{n}_v[\partial_\mu p_{m}(\mu, s, r', x', y')](v)| \, \mu(dx') \, dy' \, dr'   \right) \notag \\
& \quad \quad \times g(c(t-r), z-x), \notag
\end{align}
\begin{align}
|  & \partial_s \mH_{m+1}(\mu, s, r, t, x, z)| \notag \\
 & \leq  \frac{K_\beta}{(t-r)^{1-\frac{\beta \eta}{2}}}\left\{\int_{(\mathbb{R}^d)^2} (|y'-x'|^{(1-\beta)\eta} \wedge 1) \,   |\partial_s p_{m}(\mu, s , r, x', y')| \, \mu(dx') \, dy'  \right.  \label{cross:time:deriv:parametrix:kernel:s:r:t} \\
& \quad \left. + \frac{1}{(t-r)^{1+\frac{(\beta-1)\eta}{2}}}\int_r^t  \int_{(\rr^d)^2} (|y'-x'|^\eta \wedge 1) \,   |\partial_s p_{m}(\mu, s , r', x', y')| \, \mu(dx') \, dy'\, dr'  \right\} g(c(t-r), z-x). \notag
\end{align}
\noindent and 
\begin{align}
& \Big|  \partial_v\Big[\partial_\mu \mH_{m+1}(\mu, s, r ,t ,x ,z) \Big](v) -  \partial_v\Big[\partial_\mu \mH_{m+1}(\mu, s, r ,t ,x ,z) \Big](v') \Big| \notag \\
& \leq K_{\beta'}  \left\{ \frac{1}{(t-r)^{1-\frac{\eta}{2}} (r-s)^{1+\frac{\beta'}{2}}} \wedge  \frac{1}{(t-r)(r-s)^{1+\frac{\beta'-\eta}{2}}} \right\}  \label{cross:mes:deriv:parametrix:kernel:s:r:t:reg:holder} \\
& \quad \times \Big[|v-v'|^{\beta'} + (r-s)^{1+\frac{\beta'}{2}} \int_{(\mathbb{R}^d)^2} |\partial_v[\partial_{\mu}p_{m}(\mu, s , r, x', y')](v) - \partial_v[\partial_{\mu}p_{m}(\mu, s , r, x', y')](v')| \, \mu(dx') \, dy'  \notag\\
& \quad \quad +(r-s)^{1+\frac{\beta'-\eta}{2}} \int_{(\mathbb{R}^d)^2} (|y'-x'|^\eta \wedge 1) |\partial_v[\partial_{\mu}p_{m}(\mu, s , r, x', y')](v) - \partial_v[\partial_{\mu}p_{m}(\mu, s , r, x', y')](v')| \, \mu(dx')\, dy' \notag \\
& \quad \quad \quad +  \frac{(r-s)^{1 + \frac{\beta'-\eta}{2}}}{t-r}   \int_r^t \int_{(\mathbb{R}^d)^2} (|y'-x'|^\eta \wedge 1) | | \partial_v[\partial_\mu p_{m}(\mu, s, r', x', y')](v) - \partial_v[\partial_\mu p_{m}(\mu, s, r', x', y')](v')| \, \mu(dx') \, dy' \, dr' \Big] \notag \\
& \quad \quad \quad \quad \times g(c(t-r), z-x). \notag
\end{align}

\end{cor}

Again, let us importantly point out that if the estimates \eqref{first:second:estimate:induction:decoupling:mckean} and \eqref{time:derivative:induction:decoupling:mckean} (at step $m$) are satisfied then from \eqref{cross:mes:deriv:parametrix:kernel:s:r:t:with:beta}, taking the minimum between the upper-bounds obtained in the two cases $\beta=0$ and $\beta=1$,
\begin{align}
|\partial^n_v [\partial_\mu & \mH_{m+1}(\mu, s, r ,t ,x ,z)](v)| \notag \\ 
& \leq K \left( \frac{1}{(t-r)^{1-\frac{\eta}{2}}(r-s)^{\frac{1+n}{2}}} \wedge  \frac{1}{(t-r)(r-s)^{\frac{1+n-\eta}{2}}}  \right) \notag \\
  & \quad \times \left(1 + (r-s)^{\frac{1+n-\eta}{2}} \int_{(\mathbb{R}^d)^2} (|y'-x'|^{\eta}\wedge 1) |\partial^{n}_v[\partial_\mu p_{m}(\mu, s, r, x', y')](v)| \, dy' \mu(dx') \right.  \label{cross:mes:deriv:parametrix:kernel:s:r:t} \\
& \quad \quad \left. + (r-s)^{\frac{1+n}{2}}\int_{(\rr^d)^2}  |\partial^{n}_v[\partial_\mu p_{m}(\mu, s, r, x', y')](v)| \, dy' \mu(dx') \right. \notag \\
& \quad \quad \left. + \frac{(r-s)^{\frac{1+n-\eta}{2}}}{t-r} \int_r^t \int_{(\mathbb{R}^d)^2} (|y'-x'|^{\eta}\wedge 1) |\partial^{n}_v[\partial_\mu p_{m}(\mu, s, r', x', y')](v)| \, dy' \mu(dx') \, dr'   \right) \notag \\
& \quad \times g(c(t-r), z-x)\notag
\end{align}
\noindent or taking $\beta=1/2$ and using \eqref{first:second:estimate:induction:decoupling:mckean} as well as the space-time inequality \eqref{space:time:inequality}
\begin{align}
|\partial^n_v [\partial_\mu & \mH_{m+1}(\mu, s, r ,t ,x ,z)](v)|   \leq  \frac{K_m}{(t-r)^{1-\frac{\eta}{4}}(r-s)^{\frac{1+n}{2}-\frac{\eta}{4}}} g(c(t-r),z-x)\label{time:degeneracy:estimate:deriv:mes:parametrix:kernel:bis}
\end{align}

\noindent for some positive constant $K_m:= K(T, \HR, \HE, m)$. In a completely analogous manner, taking $\beta=1/2$ in \eqref{cross:time:deriv:parametrix:kernel:s:r:t} and using \eqref{time:derivative:induction:decoupling:mckean} as well as the space-time inequality \eqref{space:time:inequality}, we obtain
\begin{align}\label{time:degeneracy:estimate:deriv:time:parametrix:kernel}
|\partial_s  \mH_{m+1}(\mu, s, r ,t ,x ,z)|  \leq \frac{K_m}{(t-r)^{1 - \frac{\eta}{4}}(r-s)^{1-\frac{\eta}{4}}} g(c(t-r),z-x).
\end{align}

%

The previous controls will be useful in the sequel. In particular, we can now state the following result.
\begin{cor}\label{cor:mes:time:deriv:iterated:parametrix:kernel} Assume that the estimates \eqref{first:second:estimate:induction:decoupling:mckean} and \eqref{time:derivative:induction:decoupling:mckean} are satisfied (at step $m$) for some positive constant $C$. For any positive integer $k$, for any $(t, x, z) \in (0,T] \times (\mathbb{R}^d)^2$ and any $r\in (0,t)$, the map $[0,r) \ni (s, \mu) \mapsto \mH^{(k)}_{m+1}(\mu, s, r, t, x, z)$ belongs to $\mathcal{C}^{1, 2}([0,r)\times \pp)$ with continuous derivatives with respect to the variables $s$, $x$, $\mu$ and $v$.

Moreover, the following estimates hold: there exists $c>0$ such that for any positive integer $k$, for any $(t, x, z)\in (0,T] \times (\mathbb{R}^d)^2$, for any $s \in [0, t)$, for any $r\in (s, t)$ and any $(\mu, v) \in \pp \times \mathbb{R}^d$
\begin{align}
|\partial^n_v [\partial_\mu & \mH^{(k)}_{m+1}(\mu, s, r ,t ,x ,z)](v)| \notag \\
& \leq \frac{k K^{k-1} K_m}{(r-s)^{\frac{1+n}{2}-\frac{\eta}{4}}(t-r)^{1-(k-1)\frac{\eta}{2}-\frac{\eta}{4}}} \prod_{\ell =1}^{k-1} B\left(\frac{\eta}{4}, \frac{\eta}{4} + (\ell-1)\frac{\eta}{2}\right) \, g(c(t-r), z-x) \label{cross:mes:deriv:iterated:parametrix:kernel:s:r:t}
\end{align}
\noindent and
\begin{align}
|\partial_s & \mH^{(k)}_{m+1}(\mu, s, r ,t ,x ,z)| \notag \\
& \leq \frac{k K^{k-1} K_m}{(r-s)^{1-\frac{\eta}{4}}(t-r)^{1-(k-1)\frac{\eta}{2}-\frac{\eta}{4}}} \prod_{\ell =1}^{k-1} B\left(\frac{\eta}{4}, \frac{\eta}{4} + (\ell-1)\frac{\eta}{2}\right) \, g(c(t-r), z-x) \label{cross:time:deriv:iterated:parametrix:kernel:s:r:t}
\end{align}
\noindent where $K$ and $K_m$ are the constants appearing respectively in the right-hand side of the inequalities \eqref{iter:parametrix:kernel} and \eqref{time:degeneracy:estimate:deriv:mes:parametrix:kernel:bis}-\eqref{time:degeneracy:estimate:deriv:time:parametrix:kernel}.
\end{cor}

\subsection{First part of the induction step}\label{subsubsection:proof:first:part:induction:step}

In this section, our aim is to prove the first two points of the induction step of Proposition \ref{proposition:reg:density:recursive:scheme:mckean}. To be more specific, we here prove that if the map $(s, x, \mu) \mapsto p_{m}(\mu, s, t, x, z)$ belongs to $\mathcal{C}^{1, 2,  2}([0,t)\times \mathbb{R}^d \times \pp)$ and if the pointwise Gaussian estimates \eqref{first:second:estimate:induction:decoupling:mckean} and \eqref{time:derivative:induction:decoupling:mckean} are satisfied for some positive constant $C$ (the constant $C$ being the one appearing in the definition of the mth partial sums $\mathscr{C}^{n,\beta}_m(C, t)$) then $(s, x, \mu) \mapsto p_{m+1}(\mu, s, t, x, z) \in \mathcal{C}^{1, 2,  2}([0,t)\times \mathbb{R}^d \times \pp)$. Additionally, we prove that if the estimates \eqref{first:second:estimate:induction:decoupling:mckean} to \eqref{equicontinuity:first:estimate:decoupling:mckean} are satisfied at step $m$ for some adequate specification of the constants $C$ and $C_\beta$, then they remain valid at step $m+1$.

\begin{prop}\label{prop:induction:step:first:part} Assume that the estimates \eqref{first:second:estimate:induction:decoupling:mckean}, \eqref{time:derivative:induction:decoupling:mckean} and \eqref{equicontinuity:first:estimate:decoupling:mckean} (at step $m$) are satisfied for some positive constants $C$ and $C_\beta$. For any $(t, x, z) \in (0,T] \times (\mathbb{R}^d)^2$, the map $[0,t) \times \mathbb{R}^d \times \pp \ni (s, x, \mu) \mapsto p_{m+1}(\mu, s, t, x, z)$ belongs to $\mathcal{C}^{1, 2,  2}([0,t)\times \mathbb{R}^d \times \pp)$ and satisfies 
\begin{align}
\partial^{n}_v[\partial_\mu p_{m+1}(\mu, s , t, x, z)](v) & = \sum_{k\geq0} (\partial^{n}_v [\partial_\mu \widehat{p}_{m+1}] + p_{m+1} \otimes \partial^{n}_v [\partial_\mu \mH_{m+1}]) \otimes \mH^{(k)}_{m+1}(\mu, s, t, x, z)(v), \, n \in \left\{0, 1\right\}\label{representation:formula:cross:lions:deriv:dens}
\end{align}
\noindent and
\begin{align}
\partial_s p_{m+1}(\mu, s , t, x, z) & = \partial_s \widehat{p}_{m+1}(\mu, s, t, x, z) - \Phi_{m+1}(\mu, s, t, x, z) + \partial_s \widehat{p}_{m+1}\otimes \Phi_{m+1}(\mu, s, t, x, z) \nonumber \\
& \quad + p_{m+1} \otimes \partial_s \mH_{m+1}(\mu, s ,t, x, z) + [(p_{m+1}\otimes \partial_s \mH_{m+1})\otimes \Phi_{m+1}](\mu, s, t, x, z).\label{representation:formula:time:deriv:dens}
\end{align}
Moreover, the four following Gaussian estimates hold: for any $\beta \in [0,1]$ if $n=0$ or any $\beta \in [0,\eta)$ if $n=1$, there exist positive constants $K_{\beta}$, $K$, $c$ such that for any $(t, x, x', z)\in (0,T] \times (\mathbb{R}^d)^3$, for any $(s, \mu, v) \in [0,t) \times \pp \times \mathbb{R}^d$ and any value of the positive constants $C$ and $C_{\beta}$ appearing in the definition of the mth partial sums $\mathscr{C}^{n, \beta}_m(C, t-s)$ and $\mathscr{C}^{n, \beta}_m(C_\beta, t-s)$ of \eqref{first:second:estimate:induction:decoupling:mckean}, \eqref{time:derivative:induction:decoupling:mckean}, \eqref{equicontinuity:second:third:estimate:decoupling:mckean} and \eqref{equicontinuity:first:estimate:decoupling:mckean}
\begin{align}
|\partial^{n}_v & [\partial_\mu p_{m+1}(\mu, s, t, x, z)](v) | \nonumber \\
& \leq   \frac{K}{(t-s)^{\frac{1+n-\eta}{2}}} \left\{ 1 + \sum_{k=1}^{m} C^{k} (t-s)^{k \frac{\eta}{2}}  \prod_{i=1}^{k} B\left(\frac{\eta}{2}, \frac{1-n+\eta}{2} +  (i-1)\frac{\eta}{2} \right) \right\} \label{estimate:deriv:mes:dens:stepmp1:cor} \\
& \quad \times  g(c(t-s), z-x), \, n \in \left\{0, 1\right\},\nonumber
\end{align}
\begin{align}
|\partial^{n}_v & [\partial_\mu p_{m+1}(\mu, s, t, x, z)](v) - \partial^{n}_v [\partial_\mu p_{m+1}(\mu, s, t, x', z)](v) | \nonumber \\
& \leq   K_\beta \frac{|x-x'|^{\beta}}{(t-s)^{\frac{1+n + \beta-\eta}{2}}} \left\{ 1 + \sum_{k=1}^{m} C^{k} (t-s)^{k \frac{\eta}{2}}  \prod_{i=1}^{k} B\left(\frac{\eta}{2}, \frac{1-n+\eta-\beta}{2} +  (i-1)\frac{\eta}{2} \right)  \right\} \label{estimate:deriv:mes:holder:reg:dens:stepmp1:cor} \\
& \quad \times  ( g(c(t-s), z-x) + g(c(t-s), z-x') ), \, n \in \left\{0, 1\right\}, \nonumber
\end{align}
\begin{align}
|\partial_v & [\partial_\mu p_{m+1}(\mu, s, t, x, z)](v) - \partial_v [\partial_\mu p_{m+1}(\mu, s, t, x, z)](v') | \nonumber \\
& \leq   K_\beta \frac{|v-v'|^{\beta}}{(t-s)^{1+\frac{ \beta-\eta}{2}}} \left\{ 1 + \sum_{k=1}^{m} C_\beta^{k} (t-s)^{k \frac{\eta}{2}}  \prod_{i=1}^{k} B\left(\frac{\eta}{2}, \frac{1-n+\eta-\beta}{2} +  (i-1)\frac{\eta}{2} \right)  \right\} \label{estimate:deriv:mes:holder:reg:terminal:point:mes:dens:stepmp1:cor} \\
& \quad \times  g(c(t-s), z-x), \nonumber
\end{align}
\noindent and
\begin{align}
|  \partial_s p_{m+1}(\mu, s, t, x, z)|  & \leq   \frac{K}{t-s} \left\{ 1 + \sum_{k=1}^{m} C^{k} (t-s)^{k \frac{\eta}{2}}  \prod_{i=1}^{k} B\left(\frac{\eta}{2}, \frac{\eta}{2} +  (i-1)\frac{\eta}{2} \right) \right\} \label{estimate:deriv:time:dens:stepmp1:cor} \\
& \quad \times  g(c(t-s), z-x). \nonumber
\end{align}

\end{prop}

\bigskip 

\noindent \emph{Conclusion of the first part of the induction step:} In view of Proposition \ref{prop:induction:step:first:part}, it suffices to set the constants $C$ and $C_{\beta}$ in the mth partial sums $\mathscr{C}^{n, \beta}_m(C, t-s)$, $\mathscr{C}^{n, \beta}_m(C_\beta, t-s)$ which are used in the statement of the Gaussian estimates \eqref{first:second:estimate:induction:decoupling:mckean} to \eqref{equicontinuity:first:estimate:decoupling:mckean} to be equal to $K$ or $K_{\beta}$ where $K$ and $K_\beta$ are the constants appearing in the right-hand side of the Gaussian estimates \eqref{estimate:deriv:mes:dens:stepmp1:cor} to \eqref{estimate:deriv:time:dens:stepmp1:cor}. Indeed, doing so and by the very definition of $\mathscr{C}^{n, \beta}_{m+1}(K , t-s)$ and $\mathscr{C}^{n, \beta}_{m+1}(K_\beta , t-s)$, we deduce that the map $[0,t) \times \mathbb{R}^d \times \pp \ni (s, x, \mu) \mapsto p_{m+1}(\mu, s, t, x, z)$ belongs to $\mathcal{C}^{1, 2,  2}([0,t)\times \mathbb{R}^d \times \pp)$ and the estimates \eqref{estimate:deriv:mes:dens:stepmp1:cor} to \eqref{estimate:deriv:time:dens:stepmp1:cor} directly yield the desired estimates \eqref{first:second:estimate:induction:decoupling:mckean} to \eqref{equicontinuity:first:estimate:decoupling:mckean} at step $m+1$. We thus conclude that for any positive integer $m$, $[0,t) \times \mathbb{R}^d \times \pp \ni (s, x, \mu) \mapsto p_{m}(\mu, s, t, x, z) \in \mathcal{C}^{1, 2,  2}([0,t)\times \mathbb{R}^d \times \pp)$ and that the pointwise Gaussian estimates \eqref{first:second:estimate:induction:decoupling:mckean} to \eqref{equicontinuity:first:estimate:decoupling:mckean} are satisfied. This completes the proof of the first of part of Proposition \ref{proposition:reg:density:recursive:scheme:mckean}.

As already mentioned in Remark \ref{remark:convergence:series:gaussian:upper:bounds},  in order to tackle the second part of the induction step, from now on and without explicitly mentioning it in the sequel, we will use the pointwise Gaussian estimates \eqref{first:second:estimate:induction:decoupling:mckean} to \eqref{equicontinuity:first:estimate:decoupling:mckean} with constants $\mathscr{C}^{n, 0}_\infty(K, T)$ and $ \mathscr{C}^{n, \beta}_\infty(K_\beta, T)$ instead of $\mathscr{C}^{n, \beta}_m(K, T)$ and $\mathscr{C}^{n, \beta}_m(K_\beta, T)$. We now move to the proof of Proposition \ref{prop:induction:step:first:part}.\\

\begin{proof}[Proof of Proposition \ref{prop:induction:step:first:part}]\quad\\

\noindent \emph{Step 1: the map $[0,t)\times \mathbb{R}^d \times \pp \ni (s, x,  \mu) \mapsto p_{m+1}(\mu, s, t, x, z) \in \mathcal{C}^{1, 2, 2}([0,t)\times \mathbb{R}^d \times \pp)$.}\\

We first prove that $[0,t)\times \mathbb{R}^d \times \pp \ni  (s, x,  \mu) \mapsto p_{m+1}(\mu, s, t, x, z),  \, \mH_{m+1}(\mu, s, t, x, z), \, \Phi_{m+1}(\mu, s, r, t, x, z)$ are continuous. We consider a sequence $(s_n, x_n, \mu_n)_{n\geq1}$ taking values in $[0,t)\times \rr^d \times \pp$ and satisfying $\lim_{n}|s_n -s| = \lim_{n} |x_n -x| = \lim_{n} W_2(\mu_n, \mu) = 0$ for some $(s, x, \mu) \in [0,t) \times \rr^d \times \pp$. Following the same lines of reasonings as those employed in the first step of the proof of Proposition \ref{structural:class}, we deduce that $\lim_{n}d_{{\rm TV}}([X^{s_n, \xi_n, (m)}_t], [X^{s, \xi, (m)}_t])=0$, where we denote by $(\xi_n)_{n\geq1}$ any sequence of random variables such that $[\xi_n] = \mu_n$ and by $\xi$ any random variable with law $\mu$. The continuity of the coefficients then implies that $\lim_{n} a_{i, j}(t, x_n, [X^{s_n, \xi_n, (m)}_t]) = a_{i, j}(t, x,  [X^{s, \xi, (m)}_t])$ and $\lim_{n} b_{i}(t, x_n, [X^{s_n, \xi_n, (m)}_t]) = b_{i}(t, x,  [X^{s, \xi, (m)}_t])$ which in turn from the representation in infinite series \eqref{series:approx:mckean}, \eqref{definition:parametrix:kernel:iterate:m}, \eqref{infinite:series:Phi:step:m} (at step $m+1$), the estimates \eqref{iter:parametrix:kernel}, \eqref{iter:param:classic}, \eqref{Gaussian:estimate:Phim} and the dominated convergence theorem yield the continuity of the maps $[0,t) \times \rr^d \times \pp \ni (s, x, \mu) \mapsto p_{m+1}(\mu, s, t, x, z), \, \mH_{m+1}(\mu, s, t, x, z), \, \Phi_{m+1}(\mu, s, r, t, x, z)$. 

We now prove that $\mathbb{R}^d \ni x\mapsto p_{m+1}(\mu, s, t, x, z)$ is two times differentiable and that the functions $[0,t) \times \mathbb{R}^d \times \pp \ni (s, x, \mu) \mapsto \partial_x p_{m+1}(\mu, s, t, x, z), \, \partial^2_x p_{m+1}(\mu, s, t, x, z)$ are continuous. For fixed $0\leq s \leq r < t \leq T$, we differentiate $n$-times ($n=1,2$) with respect to the variable $x$ the relation 
\begin{align*}
 \int_{\rr^d} \widehat{p}_{m+1}(\mu, s, r, x, y) &\Phi_{m+1}(\mu, s, r, t, y, z) \, dy\\
 &  =  \int_{\rr^d} \widehat{p}_{m+1}(\mu, s, r, x, y) [\Phi_{m+1}(\mu, s, r, t, y, z) -\Phi_{m+1}(\mu, s, r, t, x', z)]\, dy \\
 & \quad + \Phi_{m+1}(\mu, s, r, t, x', z) \int_{\rr^d} [\widehat{p}^{y}_{m+1}(\mu, s, r, x, y) - \widehat{p}^{x'}_{m+1}(\mu, s, r, x, y)] \ dy \\
 & \quad +  \Phi_{m+1}(\mu, s, r, t, x', z)
 \end{align*}
 \noindent and then chose $x'=x$ so that by the dominated convergence theorem, we get 
\begin{align*}
 \int_{\mathbb{R}^d} \partial^{n}_x\widehat{p}_{m+1}(\mu, s, r, x, y) & \Phi_{m+1}(\mu, s, r, t, y, z) \, dy \\
 &  =  \int_{\mathbb{R}^d} \partial^{n}_x\widehat{p}_{m+1}(\mu, s, r, x, y) [\Phi_{m+1}(\mu, s, r, t, y, z) - \Phi_{m+1}(\mu, s, r, t, x, z)]\, dy \\
 & \quad + \Phi_{m+1}(\mu, s, r, t, x, z) \int_{\mathbb{R}^d} [\partial^n_x\widehat{p}^y_{m+1}(\mu, s, r, x, y) - \partial^{n}_x\widehat{p}^{x}_{m+1}(\mu, s, r, x, y)] \ dy. 
 \end{align*}
 
The mean-value theorem and the uniform $\eta$-H\"older regularity of $a(t, ., m)$ imply that for any $y', y'' \in \mathbb{R}^d$
\begin{align}
|\partial^n_x\widehat{p}^{y'}_{m+1}(\mu, s, r, x, y) - \partial^{n}_x\widehat{p}^{y''}_{m+1}(\mu, s, r, x, y)| &  \leq K (|y'-y''|^\eta \wedge 1) (r-s)^{-\frac{n}{2}} g(c(r-s), y-x) \label{diff:freezing:point:deriv:space:phat} 
\end{align}
\noindent which in turn by taking $y'=y$, $y''=x$ and using the space-time inequality \eqref{space:time:inequality} yield
\begin{align*}
 \int_{\mathbb{R}^d} |\partial^n_x\widehat{p}^y_{m+1}(\mu, s, r, x, y) - \partial^{n}_x\widehat{p}^{x}_{m+1}(\mu, s, r, x, y)| \ dy & \leq K (r-s)^{\frac{-n+\eta}{2}}.
\end{align*}

Moreover, from \eqref{Gaussian:estimate:Phim}, for $r\in [s, (t+s)/2]$, one gets
$$
| \Phi_{m+1}(\mu, s, r, t, x, z) | \leq \frac{K}{t-s} \, g(c(t-s), z-x) 
$$
\noindent so that 
$$
|\Phi_{m+1}(\mu, s, r, t, x, z) \int_{\mathbb{R}^d} [\partial^n_x\widehat{p}^y_{m+1}(\mu, s, r, x, y) - \partial^{n}_x\widehat{p}^{x}_{m+1}(\mu, s, r, x, y)] \ dy| \leq \frac{K}{(t-s) (r-s)^{\frac{n-\eta}{2}}} \, g(c(t-s), z-x). 
$$
\noindent Combining \eqref{holder:reg:voltera:kernel} with the space-time inequality \eqref{space:time:inequality}, we get
\begin{align*}
\Big|\int_{\mathbb{R}^d} \partial^{n}_x \widehat{p}_{m+1}(\mu, s, r, x, y) & [\Phi_{m+1}(\mu, s, r, t, y, z)  - \Phi_{m+1}(\mu, s, r, t, x, z)]\, dy \Big| \\
&  \leq K (t-r)^{-1+\frac{\eta-\beta}{2}}(r-s)^{\frac{-n+\beta}{2}} g(c(t-s), z-x) 
\end{align*}

\noindent for any $\beta \in [0,\eta)$. Hence, from the dominated convergence theorem, it follows that the map $x\mapsto \int_s^{(t+s)/2} \int_{\mathbb{R}^d} \widehat{p}_{m+1}(\mu, s, r, x, y) \Phi_{m+1}(\mu, s, r, t, y, z) \, dy \, dr$ is $n$-times continuously differentiable, $n=1, 2$, and satisfies 
\begin{align*}
\partial_x^{n}  & \int_s^{\frac{t+s}{2}} \int_{\mathbb{R}^d} \widehat{p}_{m+1}(\mu, s, r, x, y) \Phi_{m+1}(\mu, s, r, t, y, z) \, dy \, dr \\
& = \int_{s}^{\frac{t+s}{2}} \int_{\mathbb{R}^d} \partial^{n}_x\widehat{p}_{m+1}(\mu, s, r, x, y) [\Phi_{m+1}(\mu, s, r, t, y, z) - \Phi_{m+1}(\mu, s, r, t, x, z)]\, dy \, dr \\
 & \quad + \int_{s}^{\frac{t+s}{2}} \Phi_{m+1}(\mu, s, r, t, x, z) \int_{\mathbb{R}^d} [\partial^n_x\widehat{p}^y_{m+1}(\mu, s, r, x, y) - \partial^{n}_x\widehat{p}^{x}_{m+1}(\mu, s, r, x, y)] \ dy \,dr.
\end{align*}
 
 Moreover, from \eqref{Gaussian:estimate:Phim}, the standard Gaussian estimate $|\partial_x^{n} \widehat{p}_{m+1}(\mu, s, r, x, y)| \leq K (r-s)^{-\frac{n}{2}} \, g(c(r-s), y-x)$ and the dominated convergence theorem, $x\mapsto  \int_{(t+s)/2}^{t} \int_{\mathbb{R}^d} \widehat{p}_{m+1}(\mu, s, r, x, y) \Phi_{m+1}(\mu, s, r, t, y, z) \, dy \, dr$ is $n$-times continuously differentiable with 
 \begin{align*}
 \partial_x^n \int_{\frac{t+s}{2}}^{t}&  \int_{\mathbb{R}^d} \widehat{p}_{m+1}(\mu, s, r, x, y) \Phi_{m+1}(\mu, s, r, t, y, z) \, dy \, dr \\
 &  = \int_{ \frac{t+s}{2}}^{t} \int_{\mathbb{R}^d}  \partial_x^{n} \widehat{p}_{m+1}(\mu, s, r, x, y) \Phi_{m+1}(\mu, s, r, t, y, z) \, dy \, dr.
 \end{align*}
 
Indeed, note carefully that the term $ \partial_x^{n} \widehat{p}_{m+1}(\mu, s, r, x, y)$ does not induce any time singularity for $r\in [(t+s)/2, t]$. One may thus differentiate $n$-times the relation \eqref{other:representation:parametrix:series} with respect to the variable $x$ and deduce that $x\mapsto  p_{m+1}(\mu, s, t, x, z)$ is two times continuously differentiable with
\begin{align}
\partial^{n}_x p_{m+1}(\mu, s, t, x, z) & = \partial^{n}_x \widehat{p}_{m+1}(\mu, s, t, x, z) \nonumber \\
&   \quad + \int_s^{\frac{t+s}{2}} \int_{\rr^d} \partial^{n}_x \widehat{p}_{m+1}(\mu, s, r, x, y) \Big[\Phi_{m+1}(\mu, s, r, t, y, z) - \Phi_{m+1}(\mu, s, r, t, x, z)\Big] \, dy \, dr \label{deriv:space:parametrix:series} \\
&  \quad  + \int_s^{\frac{t+s}{2}} \Phi_{m+1}(\mu, s, r, t, x, z) \int_{\rr^d} \Big[ \partial^{n}_x\widehat{p}^{y}_{m+1}(\mu, s, r, x, y) - \partial^{n}_x \widehat{p}^{x}_{m+1}(\mu, s, r, x, y)\Big] \, dy \, dr \nonumber \\
& \quad   + \int_{\frac{t+s}{2}}^{t}  \partial_x^{n} \widehat{p}_{m+1}(\mu, s, r, x, y) \Phi_{m+1}(\mu, s, r, t, y, z) \, dy \, dr\nonumber
\end{align}
\noindent for $n \in \left\{1, 2\right\}$. Now, the previous relation together with the continuity of $[0,t) \times \mathbb{R}^d\times \pp \ni (s, x, \mu) \mapsto \Phi_{m+1}(\mu, s, r, t, x, z), \, \partial^{n}_x\widehat{p}^{y'}_{m+1}(\mu, s, r, x, y) $, for any $y'\in \mathbb{R}^d$, and again the dominated convergence theorem eventually yield the continuity of $[0,t)\times \mathbb{R}^d \times \pp \ni (s, x, \mu) \mapsto \partial^{n}_x p_{m}(\mu, s, t, x, z)$, $n \in \left\{1, 2\right\}$. 

We now investigate the smoothness of $\pp \ni \mu \mapsto p_{m+1}(\mu, s, t, x, z)$. Combining \eqref{cross:mes:deriv:p:hat:s:r:t} with \eqref{first:second:estimate:induction:decoupling:mckean} and then using \eqref{iter:parametrix:kernel}, we deduce that for some positive constant $K_m$ one has
\begin{align*}
|\partial^{n}_v[\partial_\mu \widehat{p}_{m+1}] \otimes \mH^{(k)}_{m+1}(\mu, s, t, x, z)(v)| & \leq  C^{k}K_m \prod_{\ell=1}^{k-1} B\left(\ell \frac{\eta}{2}, \frac{\eta}{2}\right)  \int_s^t \frac{1}{(r-s)^{\frac{1+n-\eta}{2}}} \frac{1}{(t-r)^{1-k \frac{\eta}{2}}} \, dr\, g(c(t-s), z-x)\\
& \leq \frac{K^{k}K_m}{ (t-s)^{\frac{1+n}{2} - (k+1) \frac{\eta}{2}}} \prod_{\ell=1}^{k} B\left(\ell \frac{\eta}{2}, \frac{\eta}{2}\right) \, g(c(t-s), z-x)
\end{align*}
\noindent so that, using the fact that $\prod_{\ell=1}^{k} B\left(\ell \eta/2, \eta/2\right) = \prod_{\ell=1}^{k} \Gamma(\eta/2) \Gamma(\ell \eta/2) / \Gamma((\ell+1) \eta/2) = \Gamma(\eta/2)^{k+1} / \Gamma((k+1)\eta/2)$ and Stirling's formula, we deduce that the series $\sum_{k\geq 0}\partial^{n}_v[\partial_\mu \widehat{p}_{m+1}] \otimes \mH^{(k)}_{m+1}(\mu, s, t, x, z)(v)$ is absolutely convergent and locally uniformly in $(s, \mu, x, v) \in [0,t) \times \pp \times (\mathbb{R}^d)^2$ so that it is continuous with respect to the variables $s$, $x$, $\mu$ and $v$. 
Similarly from \eqref{definition:phat:end:point:frozen}, \HE\, and \eqref{cross:mes:deriv:iterated:parametrix:kernel:s:r:t}
\begin{align*}
| \widehat{p}_{m+1}\otimes & \partial^{n}_v[\partial_\mu \mH^{(k)}_{m+1}(\mu, s, t, x, z)](v)|  \\
& \leq k K^{k-1} K_m \prod_{\ell =1}^{k-1} B\left(\frac{\eta}{4}, \frac{\eta}{4} + (\ell-1) \frac{\eta}{2}\right) \int_s^t \frac{1}{(r-s)^{\frac{1+n}{2}-\frac{\eta}{4}}(t-r)^{1-(k-1)\frac{\eta}{2}-\frac{\eta}{4}}} \, dr \, g(c(t-s), z-x)\\
& \leq \frac{k K^{k-1} K_m}{(t-s)^{\frac{1+n}{2}-k\frac{\eta}{2}}} \prod_{\ell =1}^{k} B\left(\frac{\eta}{4}, \frac{\eta}{4} + (\ell-1) \frac{\eta}{2}\right) \,g(c(t-s), z-x)
\end{align*}
\noindent so that, again from the asymptotics of the Beta function, the series $\sum_{k\geq0}  \widehat{p}_{m+1}\otimes \partial^{n}_v[\partial_\mu \mH^{(k)}_{m+1}(\mu, s, t, x, z)](v)$ converges absolutely and locally uniformly in $(s, \mu, x, v) \in [0,t) \times \pp \times (\mathbb{R}^d)^2$ so that it is continuous with respect to the variables $s$, $x$, $\mu$ and $v$. Hence, from the representation in infinite series \eqref{series:approx:mckean} and the dominated convergence theorem, we conclude that $\mu \mapsto p_{m+1}(\mu, s, t, x, z)$ is continuously $L$-differentiable, $[0,t) \times \mathbb{R}^d \times \pp \times \mathbb{R}^d  \ni (s, x, \mu, v) \mapsto \partial_\mu p_{m+1}(\mu, s, t, x, z)(v)$ being continuous and such that for any $(s, x, \mu) \in [0,t) \times \mathbb{R}^d \times \pp$, $\mathbb{R}^d \ni v\mapsto  \partial_\mu p_{m+1}(\mu, s, t, x, z)(v)$ is continuously differentiable, $[0,t) \times \mathbb{R}^d \times \pp \times \mathbb{R}^d  \ni (s, x, \mu, v) \mapsto \partial_v \partial_\mu p_{m+1}(\mu, s, t, x, z)(v)$ being continuous with
$$
\partial^{n}_v[\partial_\mu p_{m+1}(\mu, s, t, x, z)](v) = \sum_{k\geq 0} \partial^{n}_v[\partial_\mu \widehat{p}_{m+1}] \otimes \mH^{(k)}_{m+1}(\mu, s, t, x, z)(v) +  \widehat{p}_{m+1}\otimes \partial^{n}_v[\partial_\mu \mH^{(k)}_{m+1}(\mu, s, t, x, z)](v)
$$
\noindent and
\begin{equation}\label{gaussian:estimate:cros:mes:derivative:p}
|\partial^{n}_v[\partial_\mu p_{m+1}(\mu, s, t, x, z)](v)| \leq \frac{K_m}{(t-s)^{\frac{1+n}{2}-\frac{\eta}{2}}} \,g(c(t-s), z-x).
\end{equation}

We now discuss the continuous differentiability of the map $[0,t) \ni s \mapsto p_{m+1}(\mu, s, t, x, z)$ as well as the continuity of the function $[0,t)\times \mathbb{R}^d \times \pp \ni (s, x,  \mu) \mapsto  \partial_s p_{m+1}(\mu, s, t, x, z)$. First, let us note that from Corollary \ref{cor:mes:time:deriv:iterated:parametrix:kernel}, in particular from the estimate \eqref{cross:time:deriv:iterated:parametrix:kernel:s:r:t} and the dominated convergence theorem, we derive that the map $s\mapsto \Phi_{m+1}(\mu, s, r, t, x, z) = \sum_{k\geq1} \mH^{(k)}_{m+1}(\mu, s , r,  t, x, z) \in \mathcal{C}^{1}([0,t))$ with a derivative $(s, x, \mu) \mapsto \partial_s  \Phi_{m+1}(\mu, s, r, t, x, z) = \sum_{k\geq1} \partial_s\mH^{(k)}_{m+1}(\mu, s , r,  t, x, z) $ being continuous on $[0,t)\times \mathbb{R}^d \times \pp$ which satisfies
\begin{align}\label{time:deriv:phi:mp1}
|\partial_s  \Phi_{m+1}(\mu, s, r, t, x, z)| \leq \sum_{k\geq1} |\partial_s\mH^{(k)}_{m+1}(\mu, s , r,  t, x, z)| \leq \frac{K_m}{(r-s)^{1-\frac{\eta}{4}}(t-r)^{1-\frac{\eta}{4}}} \, g(c(t-r), z-x).
\end{align}

The previous estimate as well as \eqref{iter:parametrix:kernel} and \eqref{time:degeneracy:estimate:deriv:time:parametrix:kernel} combined with the dominated convergence theorem allow to differentiate with respect to the variable $s$ the relation $\Phi_{m+1}(\mu, s, r, t, x, z) = \mH_{m+1}(\mu, s, r, t, x, z) + (\mH_{m+1}\otimes \Phi_{m+1})(\mu, s, r, t, x, z)$. We thus get
\begin{align*}
\partial_s \Phi_{m+1}(\mu, s, r, t, x, z) & = \partial_s \mH_{m+1}(\mu, s, r, t, x, z) + (\partial_s \mH_{m+1} \otimes \Phi_{m+1})(\mu, s, r, t, x, z)  \\
& \quad + (\mH_{m+1} \otimes \partial_s \Phi_{m+1})(\mu, s, r, t, x, z). 
\end{align*}

The kernel $ \partial_s \mH_{m+1} + (\partial_s \mH_{m+1} \otimes \Phi_{m+1})$ yields an integrable time singularity so that, iterating the previous relation and then using \eqref{infinite:series:Phi:step:m}, we deduce the key identity
\begin{align}
\partial_s \Phi_{m+1}(\mu, s,  r, t, x, z) & = \sum_{k \geq 0} \Big(\mH^{(k)}_{m+1} \otimes [ \partial_s \mH_{m+1} + \partial_s \mH_{m+1} \otimes \Phi_{m+1}]\Big)(\mu, s, r, t, x, z) \nonumber\\
& = (\partial_s \mH_{m+1} + \partial_s \mH_{m+1} \otimes \Phi_{m+1})(\mu, s, r, t, x, z) \label{deriv:Phi:mp1:var:mu:key:relation}\\
& \quad + (\Phi_{m+1} \otimes [\partial_s \mH_{m+1} + \partial_s \mH_{m+1} \otimes \Phi_{m+1}] )(\mu, s ,r ,t, x, z). \nonumber
\end{align}

We now study the differentiability of $[0,t) \ni s\mapsto (\widehat{p}_{m+1} \otimes  \Phi_{m+1})(\mu, s, t, x, z) $. We first consider the map $[0,r) \ni s\mapsto \int_{\mathbb{R}^d} \widehat{p}_{m+1}(\mu, s, r, x ,y) \Phi_{m+1}(\mu, s, r, t, y, z) \, dy$. From the above arguments and the dominated convergence theorem, it is continuously differentiable and satisfies
\begin{align*}
\partial_s \int_{\mathbb{R}^d} \widehat{p}_{m+1}(\mu, s, r, x ,y) & \Phi_{m+1}(\mu, s, r, t, y, z) \, dy \\
& = \int_{\mathbb{R}^d}  \partial_s \widehat{p}_{m+1}(\mu, s, r, x ,y) \Phi_{m+1}(\mu, s, r, t, y, z) \, dy\\
& +  \int_{\mathbb{R}^d}   \widehat{p}_{m+1}(\mu, s, r, x ,y) \partial_s\Phi_{m+1}(\mu, s, r, t, y, z) \, dy \\
& =  \int_{\mathbb{R}^d}  \partial_s \widehat{p}_{m+1}(\mu, s, r, x ,y) [\Phi_{m+1}(\mu, s, r, t, y, z) - \Phi_{m+1}(\mu, s, r, t, x, z)]  \, dy\\
& + \Phi_{m+1}(\mu, s, r, t, x, z) \int_{\mathbb{R}^d} [\partial_s \widehat{p}^{y}_{m+1}(\mu, s, r, x ,y) - \partial_s \widehat{p}^{x}_{m+1}(\mu, s, r, x ,y)] \, dy\\
& +  \int_{\mathbb{R}^d}   \widehat{p}_{m+1}(\mu, s, r, x ,y) \partial_s\Phi_{m+1}(\mu, s, r, t, y, z) \, dy \\
& =: ({\rm I} + {\rm II} + {\rm III})(\mu, s, r, t, x, z)
\end{align*}
\noindent where for the last equality we used the identity $\int_{\mathbb{R}^d} \partial_s \widehat{p}^{x}_{m+1}(\mu, s, r, x ,y) \, dy = \partial_s \int_{\mathbb{R}^d} \widehat{p}^{x}_{m+1}(\mu, s, r, x ,y) \, dy = 0$.

Now, for the first term ${\rm I}$, we use Corollary \ref{cor:deriv:time:and:mes:phat} and the continuity of $(s, x, \mu) \mapsto  \Phi_{m+1}(\mu, s, r, t, x, z) $ to deduce that $[0,r) \times \mathbb{R}^d \times \pp \ni (s, x, \mu) \mapsto {\rm I}(\mu, s, r, t, x, z)$ is continuous. Moreover, one may combine \eqref{holder:reg:voltera:kernel} with the space-time inequality \eqref{space:time:inequality} in order to get rid off the time singularity induced by $\partial_s \widehat{p}_{m+1}(\mu, s, r, x, y)$ in the pointwise Gaussian estimate \eqref{time:deriv:p:hat:s:t} (which has to be combined with \eqref{time:derivative:induction:decoupling:mckean}). We obtain 
$$
|{\rm I}(\mu, s, r, t, x, z)| \leq \frac{K_m}{(r-s)^{1-\frac{\beta}{2}}(t-r)^{1+\frac{\beta-\eta}{2}}} g(c(t-s), z-x)
$$
\noindent for any $\beta \in [0,\eta)$.
For the second term ${\rm II}$, from Corollary \ref{cor:deriv:time:and:mes:phat}, the map $[0,r) \times \mathbb{R}^d \times \pp \ni (s, x, \mu) \mapsto  \partial_s \widehat{p}^{y}_{m+1}(\mu, s, r, x, y)  - \partial_s \widehat{p}^{x}_{m+1}(\mu, s, r, x, y) $ is continuous and \eqref{time:deriv:holder:reg:p:hat:s:t} together with \eqref{time:derivative:induction:decoupling:mckean} and the space-time inequality \eqref{space:time:inequality} we obtain $| \partial_s \widehat{p}^{y}_{m+1}(\mu, s, r, x, y)  - \partial_s \widehat{p}^{x}_{m+1}(\mu, s, r, x, y)|\leq K_m (r-s)^{-1+ \beta \eta/2} g(c(r-s),y-x)$ for any $\beta \in [0,1)$. We thus deduce that the map $[0,r) \times \mathbb{R}^d \times \pp \ni (s, x, \mu) \mapsto {\rm II}(\mu, s, r,  t, x, z)$ is continuous and satisfies 
$$
|{\rm II}(\mu, s, r, t, x, z)| \leq \frac{K_m}{(r-s)^{1-\beta \frac{\eta}{2}}(t-r)^{1- \frac{\eta}{2}}} g(c(t-r), z-x) \leq  \frac{K_m}{(r-s)^{1- \beta\frac{\eta}{2}}(t-s)^{1- \frac{\eta}{2}}} g(c(t-s), z-x)
$$

\noindent for any $r \in [s, (t+s)/2]$. 

Also, note that if $r \in [(t+s)/2, t]$, then the time degeneracy induced by the pontwise Gaussian estimate \eqref{time:deriv:p:hat:s:t} on $\partial_s \widehat{p}_{m+1}(\mu, s, r, x ,y)$ is not singular  
\begin{align*}
\left| \int_{\mathbb{R}^d}  \partial_s \widehat{p}_{m+1}(\mu, s, r, x ,y) \Phi_{m+1}(\mu, s, r, t, y, z) \, dy \right| & = |({\rm I} + {\rm II})(\mu, s, r, t, x, z)| \\
& \leq \frac{K_m}{(t-s)(t-r)^{1-\frac{\eta}{2}}} \, g(c(t-s), z-x).
\end{align*}
Similarly, from the continuity of the maps $(s, x, \mu) \mapsto  \widehat{p}_{m+1}(\mu, s, r, x ,y)$ and $(s, \mu) \mapsto \partial_s\Phi_{m+1}(\mu, s, r, t, y, z) $ as well as the estimate \eqref{time:deriv:phi:mp1}, we deduce that the map $[0,r) \times \mathbb{R}^d \times \pp \ni (s, x, \mu) \mapsto {\rm III}(\mu, s, r,  t, x, z)$ is continuous and satisfies
$$
|{\rm III}(\mu, s, r, t, x, z)| \leq \frac{K_m}{(r-s)^{1-\frac{\eta}{4}}(t-r)^{1-\frac{\eta}{4}}} g(c(t-s), z-x).
$$
The dominated convergence theorem combined with the above estimates allows to conclude that the map $[0,t) \ni s \mapsto (\widehat{p}_{m+1}\otimes \Phi_{m+1})(\mu, s, t, x, z)$ is continuously differentiable and satisfies
\begin{align*}
\partial_s (\widehat{p}_{m+1} & \otimes  \Phi_{m+1})(\mu, s, t, x, z) \\
& = -\Phi_{m+1}(\mu, s, t, x, z) +  (\partial_s \widehat{p}_{m+1}\otimes \Phi_{m+1})(\mu, s, t, x, z) +  (\widehat{p}_{m+1} \otimes \partial_s  \Phi_{m+1})(\mu, s, t, x, z)\\
& =  -\Phi_{m+1}(\mu, s, t, x, z) \\
& \quad + \int_s^{\frac{t+s}{2}} \int_{\mathbb{R}^d}  \partial_s \widehat{p}_{m+1}(\mu, s, r, x ,y) [\Phi_{m+1}(\mu, s, r, t, y, z) - \Phi_{m+1}(\mu, s, r, t, x, z)]  \, dy \, dr \\
& \quad +  \int_s^{\frac{t+s}{2}} \Phi_{m+1}(\mu, s, r, t, x, z) \int_{\mathbb{R}^d} (\partial_s \widehat{p}^{y}_{m+1}(\mu, s, r, x ,y) - \partial_s \widehat{p}^{x}_{m+1}(\mu, s, r, x ,y)) \, dy\, dr \\
& \quad + \int_{\frac{t+s}{2}}^{t}  \int_{\mathbb{R}^d}  \partial_s \widehat{p}_{m+1}(\mu, s, r, x ,y) \Phi_{m+1}(\mu, s, r, t, y, z) \, dy\\
& \quad +\int_{s}^{t}  \int_{\mathbb{R}^d}   \widehat{p}_{m+1}(\mu, s, r, x ,y) \partial_s\Phi_{m+1}(\mu, s, r, t, y, z) \, dy \, dr
\end{align*}
\noindent so that $[0,t) \times \mathbb{R}^d \times \pp \ni (s, x, \mu) \mapsto \partial_s (\widehat{p}_{m+1}\otimes \Phi_{m+1})(\mu, s, t, x, z)$ is continuous.
From the key relation \eqref{other:representation:parametrix:series}, we eventually conclude that the map $ s\mapsto p_{m+1}(\mu, s, t, x, z) \in \mathcal{C}^{1}([0,t))$, $\partial_s p_{m+1}(., ., t, ., z)$ being continuous on $[0,t) \times \mathbb{R}^d \times \pp$ for any fixed $(t, z) \in (0,T] \times \pp$, and satisfies the relation
\begin{align}
\partial_s p_{m+1}(\mu, s, t, x, z) & = \partial_s \widehat{p}_{m+1}(\mu, s, t, x, z) - \Phi_{m+1}(\mu, s , t, x, z) +(\partial_s \widehat{p}_{m+1}\otimes \Phi_{m+1})(\mu, s, t, x, z) \nonumber \\
& \quad +  (\widehat{p}_{m+1} \otimes \partial_s  \Phi_{m+1})(\mu, s, t, x, z). \label{deriv:time:pmp1:inside:proof}
\end{align}

\smallskip

\noindent \emph{Step 2: proof of the representation formulae \eqref{representation:formula:cross:lions:deriv:dens} and \eqref{representation:formula:time:deriv:dens}.}\\

From \eqref{gaussian:estimate:cros:mes:derivative:p}, \eqref{time:degeneracy:estimate:deriv:mes:parametrix:kernel:bis} and the dominated convergence theorem, one may differentiate with respect to the variables $\mu$ and then with respect to $v$, the relation 
$$
p_{m+1}(\mu, s, t, x, z) = \widehat{p}_{m+1}(\mu, s, t, x, z) + (p_{m+1}\otimes \mH_{m+1})(\mu, s , t, x, z)
$$
\noindent so that
\begin{align*}
\partial^{n}_v[\partial_\mu p_{m+1}(\mu, s , t, x, z)](v) & = \partial^{n}_v[\partial_\mu \widehat{p}_{m+1}(\mu, s, t, x, z)](v) + (p_{m+1} \otimes \partial^{n}_v[\partial_\mu \mH_{m+1}])(\mu, s ,t, x, z)(v) \\
& \quad + (\partial^{n}_v[\partial_\mu p_{m+1}] \otimes \mH_{m+1})(\mu, s, t, x, z)(v).
\end{align*}

The estimates \eqref{cross:mes:deriv:p:hat:s:r:t} combined with \eqref{first:second:estimate:induction:decoupling:mckean} and \eqref{time:degeneracy:estimate:deriv:mes:parametrix:kernel:bis} show that both space-time kernels $(t, z) \mapsto \partial^{n}_v[\partial_\mu \widehat{p}_{m+1}(\mu, s, t, x, z)](v)$ and $(t, z) \mapsto (p_{m+1} \otimes \partial^{n}_v[\partial_\mu \mH_{m+1}])(\mu, s ,t, x, z)(v)$ yield an integrable time singularity so that one may iterate the previous relation. This eventually yields the representation in infinite series \eqref{representation:formula:cross:lions:deriv:dens} and, by arguments similar to those used in the first step, the series is absolutely convergent.

The representation formula \eqref{representation:formula:time:deriv:dens} follows from \eqref{deriv:time:pmp1:inside:proof} by plugging first the identity \eqref{deriv:Phi:mp1:var:mu:key:relation} and then by using the relation \eqref{other:representation:parametrix:series}. \\

We will now establish the pointwise Gaussian estimates \eqref{estimate:deriv:mes:dens:stepmp1:cor}, \eqref{estimate:deriv:mes:holder:reg:dens:stepmp1:cor}, \eqref{estimate:deriv:mes:holder:reg:terminal:point:mes:dens:stepmp1:cor} and  \eqref{estimate:deriv:time:dens:stepmp1:cor}. We again emphasize that we denote by $K$ a positive constant that depends on $T$, $b$, $a$, $\delta b/\delta m$, $\delta a/\delta m$, $\lambda$, $\eta$ and possibly on $\beta$ (in which case we denote it by $K_\beta$) which may change from line to line but is \emph{independent of $m$, $C$ and $C_\beta$}, $C$ and $C_\beta$ being the constants appearing in the definition of the mth partial sums $\mathscr{C}^{n, \beta}_m:=\mathscr{C}^{n, \beta}_m(C, t-s)$ and $\mathscr{C}^{n, \beta}_m:=\mathscr{C}^{n, \beta}_m(C_\beta, t-s)$ of the estimates \eqref{first:second:estimate:induction:decoupling:mckean} to \eqref{equicontinuity:first:estimate:decoupling:mckean}.\\

\noindent \emph{Step 3: proof of the Gaussian estimate \eqref{estimate:deriv:mes:dens:stepmp1:cor}.}\\

The Gaussian estimates \eqref{cross:mes:deriv:p:hat:s:r:t} (with $r=s$) of Corollary \ref{cor:deriv:time:and:mes:phat} combined with \eqref{first:second:estimate:induction:decoupling:mckean} and the space-time inequality \eqref{space:time:inequality} yields
\begin{align}
| \partial^n_v &[ \partial_\mu \widehat{p}_{m+1}(\mu, s, t, x, z)](v)|  \nonumber\\
& \leq K \left\{ \frac{1}{(t-s)^{\frac{1+n - \eta}{2}}} + \frac{1}{t-s} \int_s^t \frac{\mathscr{C}^{n,0}_{m}(C, r-s)}{(r-s)^{\frac{1+n}{2}-\eta}} \, dr \right\} \, g(c(t-s) , z-x) \nonumber \\
& \leq K \left\{ \frac{1}{(t-s)^{\frac{1+n - \eta}{2}}} +  \int_s^t \frac{\mathscr{C}^{n,0}_{m}(C, r-s)}{(t-r)^{1-\frac{\eta}{2}} (r-s)^{\frac{1+n-\eta}{2}}} \, dr \right\} \, g(c(t-s) , z-x) \nonumber \\
&  \leq K  \left\{ \frac{1}{(t-s)^{\frac{1+n - \eta}{2}}} + \frac{1}{(t-s)^{\frac{1 + n  - \eta}{2}}} \sum_{k=1}^{m} C^{k} (t-s)^{k \frac{\eta}{2}}  \prod_{i=1}^{k} B\left(\frac{\eta}{2}, \frac{1-n+\eta}{2} +  (i-1)\frac{\eta}{2} \right) \right\} \label{deriv:p:hat:mu:bound}\\
& \quad \times g(c(t-s) , z-x). \nonumber
\end{align}
Hence, by induction on $r$, we deduce 
\begin{align*}
|(\partial^{n}_v[\partial_\mu \widehat{p}_{m+1}] \otimes \mH^{(r)}_{m+1})(\mu, s, t, x, z)(v)| & \leq  K^{r} \left\{ 1 + \sum_{k=1}^{m} C^{k} (t-s)^{k \frac{\eta}{2}} \prod_{i=1}^{k} B\left(\frac{\eta}{2}, \frac{1-n+\eta}{2} +  (i-1)\frac{\eta}{2} \right)  \right\} \\
&  \quad \times (t-s)^{-\frac{(1 +n -\eta)}{2} + r \frac{\eta}{2}} \prod_{i=1}^{r} B\left(\frac{\eta}{2}, \frac{1-n+\eta}{2} + (i-1) \frac{\eta}{2}\right)\\
& \quad \times g(c(t-s), z-x),
\end{align*}

\noindent which in turn implies
\begin{align}
\sum_{r\geq 0}& |(\partial^{n}_v[\partial_\mu \widehat{p}_{m+1}] \otimes \mH^{(r)}_{m+1})(\mu, s, t, x, z)(v)| \nonumber \\
& \leq  \frac{K}{(t-s)^{\frac{1+n-\eta}{2}}} \left\{ 1 + \sum_{k=1}^{m} C^{k}  (t-s)^{k \frac{\eta}{2}} \prod_{i=1}^{k} B\left(\frac{\eta}{2}, \frac{1-n+\eta}{2} +  (i-1)\frac{\eta}{2} \right) \right\}  g(c(t-s), z-x). \label{deriv:hat:pm:mu:1st:term:induction}
\end{align}
In order to deal with the second term appearing in the series \eqref{representation:formula:cross:lions:deriv:dens}, we first combine the Gaussian estimate \eqref{cross:mes:deriv:parametrix:kernel:s:r:t} of Corollary \ref{cor:deriv:time:and:mes:parametrix:kernel} with \eqref{first:second:estimate:induction:decoupling:mckean} and the space-time inequality \eqref{space:time:inequality}
\begin{align}
| \partial^{n}_v [\partial_\mu \mH_{m+1}(\mu, s, r ,t , x, z)](v) | & \leq K \left( \frac{1}{(t-r)^{1-\frac{\eta}{2}}(r-s)^{\frac{1+n}{2}}} \wedge  \frac{1}{(t-r)(r-s)^{\frac{1+n-\eta}{2}}}  \right)\nonumber \\
&  \quad \times \Big[ 1+ \mathscr{C}^{n,0}_{m}(C, r-s) (r-s)^{\frac{\eta}{2}} +\frac{1}{t-r} \int_r^t \mathscr{C}^{n, 0}_{m}(C, r'-s) (r'-s)^{\frac{\eta}{2}} \, dr' \Big] \label{deriv:mes:parametrix:kernel:iterate:induction} \\
& \quad \times g(c(t-r),z-x). \nonumber
\end{align}

Now, our aim is to establish an upper-bound for the quantity $p_{m+1} \otimes \partial^{n}_v[\partial_\mu \mH_{m+1}](\mu, s ,t , x, z)(v)$. The key idea is to remark that the estimate \eqref{deriv:mes:parametrix:kernel:iterate:induction} allows to balance the singularity in time induced by $\partial^{n}_v[\partial_\mu \mH_{m+1}](\mu, s, r, t, x, z)(v)$. Indeed, assuming first that $r \in [s, (t+s)/2]$, one has $t-r\geq (t-s)/2$ which directly implies
\begin{align*}
\int_{\mathbb{R}^d} &  |p_{m+1}(\mu, s, r, x, y)|  | \partial^{n}_v [\partial_\mu \mH_{m+1}(\mu, s, r ,t , y, z)](v)| dy \\
& \leq \frac{K}{(t-s)(r-s)^{\frac{1+n-\eta}{2}}}\Big[ 1+ \mathscr{C}^{n,0}_{m}(C, r-s) (r-s)^{\frac{\eta}{2}} + \frac{1}{t-r} \int_r^t \mathscr{C}^{n, 0}_{m}(C, r'-s) (r'-s)^{\frac{\eta}{2}} \, dr' \Big]  \\
& \quad \times g(c(t-s),z-x),
\end{align*}

\noindent so that using the fact that $[s, t] \ni r\mapsto \mathscr{C}^{n, 0}_{m}(C, r-s) (r-s)^{\eta/2}$ is non-decreasing and then integrating by parts
\begin{align*}
\int_s^{\frac{t+s}{2}} \int_{\mathbb{R}^d} & |p_{m+1}(\mu, s, r, x, y)|  | \partial^{n}_v[\partial_\mu \mH_{m+1}(\mu, s, r ,t , y, z)](v)| dy \, dr \\
& \leq K \left\{ \frac{1}{(t-s)^{\frac{1+n-\eta}{2}}}  + \frac{1}{(t-s)^2} \int_s^{t} \frac{1}{(r-s)^{\frac{1+n-\eta}{2}}} \int_r^t  \mathscr{C}^{n, 0}_{m}(C, r'-s) (r'-s)^{\frac{\eta}{2}} \, dr' \,  dr \right\} \, g(c(t-s) , z-x) \\
& \leq K \left\{ \frac{1}{(t-s)^{\frac{1+n-\eta}{2}}} + \frac{1}{(t-s)^2} \int_s^t (r-s)^{1-\frac{(1+n-\eta)}{2}} \mathscr{C}^{n, 0}_{m}(C, r-s) (r-s)^{\frac{\eta}{2}} \, dr  \right\} \, g(c(t-s) , z-x)\\
& \leq K \left\{ \frac{1}{(t-s)^{\frac{1+n-\eta}{2}}}  +  \int_s^t \frac{\mathscr{C}^{n, 0}_{m}(C, r-s)}{(t-r)^{1-\frac{\eta}{2}}(r-s)^{\frac{1+n-\eta}{2}}} \, dr \right\} \,  g(c(t-s) , z-x) \\
& \leq  \frac{K}{(t-s)^{\frac{1+n-\eta}{2}}}  \left\{ 1 + \sum_{k=1}^{m} C^{k} (t-s)^{k \frac{\eta}{2}} \prod_{i=1}^{k} B\left(\frac{\eta}{2}, \frac{1-n+\eta}{2} +  (i-1)\frac{\eta}{2} \right)  \right\} \,  g(c(t-s) , z-x).
\end{align*}

 Then, assuming that $r \in [(t+s)/2, t]$ so that $r-s \geq (t-s)/2$ we obtain 
\begin{align*}
\int_{\mathbb{R}^d} & |p_{m+1}(\mu, s, r, x, y)|  |\partial^{n}_v [\partial_\mu \mH_{m+1}(\mu, s, r ,t , y, z)](v)| dy \\
&  \leq \frac{K}{(t-s)^{\frac{1+n}{2}} (t-r)^{1-\frac{\eta}{2}}} \left[1+\mathscr{C}^{n, 0}_{m}(C, r-s) (r-s)^{\frac{\eta}{2}} + \frac{1}{t-r}\int_r^t\mathscr{C}^{n, 0}_{m}(C, r'-s) (r'-s)^{\frac{\eta}{2}} \, dr' \right] \\
& \quad \times g(c(t-s),z-x),
\end{align*}  

\noindent which in turn, by Fubini's theorem, directly yields
\begin{align*}
\int_{\frac{t+s}{2}}^{t} \int_{\mathbb{R}^d} & |p_{m+1}(\mu, s, r, x, y)|  | \partial^{n}_v[\partial_\mu \mH_{m+1}(\mu, s, r ,t , y, z)](v)| \, dy \, dr \\
&  \leq  \frac{K}{(t-s)^{\frac{1+n}{2}}} \int_{\frac{t+s}{2}}^{t} \frac{1}{(t-r)^{1-\frac{\eta}{2}}}  \left[1+\mathscr{C}^{n, 0}_{m}(C, r-s) (r-s)^{\frac{\eta}{2}} + \frac{1}{t-r}\int_r^t\mathscr{C}^{n, 0}_{m}(C, r'-s) (r'-s)^{\frac{\eta}{2}} \, dr' \right]\, dr \\
& \quad \times g(c(t-s),z-x) \\
& \leq K \left\{ \frac{1}{(t-s)^{\frac{1+n-\eta}{2}}}  +  \int_s^t \frac{\mathscr{C}^{n, 0}_{m}(C, r-s)}{(t-r)^{1-\frac{\eta}{2}} (r-s)^{\frac{1+n-\eta}{2}}} \, dr \right\} \, g(c(t-s) , z-x) \\
& \leq  \frac{K}{(t-s)^{\frac{1+n-\eta}{2}}}  \left\{ 1 + \sum_{k=1}^{m} C^{k} (t-s)^{k \frac{\eta}{2}} \prod_{i=1}^{k} B\left(\frac{\eta}{2}, \frac{1-n+\eta}{2} +  (i-1)\frac{\eta}{2} \right)  \right\} \,  g(c(t-s) , z-x).
\end{align*}

Gathering the two previous cases, we thus conclude
\begin{align}
| p_{m+1} & \otimes \partial^{n}_v[\partial_\mu \mH_{m+1}(\mu, s ,t , x, z)](v) |  \nonumber \\
&  \leq \frac{K}{(t-s)^{\frac{1+n-\eta}{2}}} \left( 1 + \sum_{k=1}^{m} C^{k} (t-s)^{k \frac{\eta}{2}} \prod_{i=1}^{k}   B\left( \frac{\eta}{2}, \frac{1-n+\eta}{2} + (i-1) \frac{\eta}{2}\right) \right) g(c(t-s),z-x) \label{estimate:convol:pmp1:cross:deriv:mes:Hmp1}
\end{align}

\noindent so that  
\begin{align}
\sum_{r\geq 0}| (p_{m+1} & \otimes \partial^{n}_v[\partial_\mu \mH_{m+1}]) \otimes \mH^{(r)}_{m+1}(\mu, s, t, x, z) (v) | \nonumber \\
& \leq \frac{K}{(t-s)^{\frac{1+n-\eta}{2}}} \left\{ 1 + \sum_{k=1}^{m} C^{k} (t-s)^{k \frac{\eta}{2}} \prod_{i=1}^{k} B\left(\frac{\eta}{2}, \frac{1-n+\eta}{2} +  (i-1)\frac{\eta}{2} \right) \right\}\,  g(c(t-s), z-x). \label{deriv:hat:pm:mu:2nd:term:induction}
\end{align}

The estimates \eqref{deriv:hat:pm:mu:1st:term:induction} and \eqref{deriv:hat:pm:mu:2nd:term:induction} together with the representation formula \eqref{representation:formula:cross:lions:deriv:dens} imply that there exist two positive constants $K:= K(T,  \HR, \HE)$ and $c:=c(\lambda)$ such that 
\begin{align*}
|\partial^{n}_v & [\partial_\mu p_{m+1}(\mu, s, t, x, z)](v) | \\
& \leq   \frac{K}{(t-s)^{\frac{1+n-\eta}{2}}} \left\{ 1 + \sum_{k=1}^{m} C^{k}(t-s)^{k \frac{\eta}{2}} \prod_{i=1}^{k} B\left(\frac{\eta}{2}, \frac{1-n+\eta}{2} +  (i-1)\frac{\eta}{2} \right) \right\} \,  g(c(t-s), z-x).
\end{align*}
This completes the proof of \eqref{estimate:deriv:mes:dens:stepmp1:cor}.\\

\noindent \emph{Step 4: proof of the Gaussian estimate \eqref{estimate:deriv:mes:holder:reg:dens:stepmp1:cor}.}\\

We will be brief on the proof of the Gaussian estimate \eqref{estimate:deriv:mes:holder:reg:dens:stepmp1:cor} inasmuch it essentially follows from similar arguments. Indeed, the estimate \eqref{cross:mes:deriv:holder:p:hat:s:t} of Corollary \ref{cor:deriv:time:and:mes:phat} combined with \eqref{first:second:estimate:induction:decoupling:mckean} and the space-time inequality \eqref{space:time:inequality} yield for any $\beta \in [0,1]$ if $n=0$ or any $\beta \in [0,\eta)$ if $n=1$
\begin{align*}
| \partial^n_v &[ \partial_\mu \widehat{p}_{m+1}(\mu, s, t, x, z)](v) - \partial^n_v [ \partial_\mu \widehat{p}_{m+1}(\mu, s, t, x', z)](v)|  \nonumber\\
& \leq K_\beta |x-x'|^\beta \left\{ \frac{1}{(t-s)^{\frac{1+n + \beta - \eta}{2}}} + \frac{1}{(t-s)^{1+\frac{\beta}{2}}} \int_s^t \frac{\mathscr{C}^{n,0}_{m}(C, r-s)}{(r-s)^{\frac{1+n}{2}-\eta }} \, dr \right\} \, (g(c(t-s) , z-x) + g(c(t-s) , z-x') ) \nonumber \\
& \leq K_\beta |x-x'|^\beta  \left\{ \frac{1}{(t-s)^{\frac{1+n + \beta- \eta}{2}}} +  \int_s^t \frac{\mathscr{C}^{n,0}_{m}(C, r-s)}{(t-r)^{1-\frac{\eta}{2}} (r-s)^{\frac{1+n+\beta-\eta}{2}}} \, dr \right\} \, (g(c(t-s) , z-x) + g(c(t-s) , z-x')) \nonumber \\
&  \leq K_\beta  \frac{ |x-x'|^\beta}{(t-s)^{\frac{1+n + \beta - \eta}{2}}}\left\{ 1 + \sum_{k=1}^{m} C^{k} (t-s)^{k \frac{\eta}{2}}  \prod_{i=1}^{k-1} B\left(\frac{\eta}{2}, \frac{1-n+\eta}{2} +  (i-1)\frac{\eta}{2} \right) B\left(\frac{\eta}{2}, \frac{1-n+ \eta-\beta}{2} + k \frac{\eta}{2}\right)\right\} \nonumber \\
& \quad \times (g(c(t-s) , z-x) + g(c(t-s) , z-x'))\nonumber \\
& \leq  K_\beta  \frac{ |x-x'|^\beta}{(t-s)^{\frac{1+n + \beta - \eta}{2}}}\left\{ 1 + \sum_{k=1}^{m} C^{k} (t-s)^{k \frac{\eta}{2}}  \prod_{i=1}^{k} B\left(\frac{\eta}{2}, \frac{1-n+\eta-\beta}{2} +  (i-1)\frac{\eta}{2} \right) \right\}\\
& \quad \times (g(c(t-s) , z-x)  + g(c(t-s) , z-x'))\nonumber
\end{align*}
\noindent where we used the fact that $ B\left(\frac{\eta}{2}, \frac{1-n+\eta}{2} +  (i-1)\frac{\eta}{2} \right) \leq B(\frac{\eta}{2}, \frac{1-n+\eta-\beta}{2} +  (i-1)\frac{\eta}{2})$, $i=1, \cdots k-1$, for the last inequality.
Hence, by induction on $r$, we obtain 
\begin{align*}
|(\partial^{n}_v[\partial_\mu \widehat{p}_{m+1}] &\otimes \mH^{(r)}_{m+1})(\mu, s, t, x, z)(v) -(\partial^{n}_v[\partial_\mu \widehat{p}_{m+1}] \otimes \mH^{(r)}_{m+1})(\mu, s, t, x', z)(v)| \\
 & \leq  K_\beta^{r} \left\{ 1 + \sum_{k=1}^{m} C^{k} (t-s)^{k \frac{\eta}{2}} \prod_{i=1}^{k} B\left(\frac{\eta}{2}, \frac{1-n+\eta - \beta}{2} +  (i-1)\frac{\eta}{2} \right)  \right\} \\
&  \quad \times |x-x'|^\beta (t-s)^{-\frac{(1 +n + \beta -\eta)}{2} + r \frac{\eta}{2}} \prod_{i=1}^{r} B\left(\frac{\eta}{2}, \frac{1-n+\eta-\beta}{2} + (i-1) \frac{\eta}{2}\right)\\
& \quad \times (g(c(t-s), z-x) + g(c(t-s), z-x'))
\end{align*}
\noindent which in turn implies
\begin{align}
\sum_{r\geq 0}& |(\partial^{n}_v[\partial_\mu \widehat{p}_{m+1}] \otimes \mH^{(r)}_{m+1})(\mu, s, t, x, z)(v) - (\partial^{n}_v[\partial_\mu \widehat{p}_{m+1}] \otimes \mH^{(r)}_{m+1})(\mu, s, t, x', z)(v)| \nonumber \\
& \leq  K_\beta \frac{|x-x'|^\beta}{(t-s)^{\frac{1+n + \beta-\eta}{2}}} \left\{ 1 + \sum_{k=1}^{m} C^{k}  (t-s)^{k \frac{\eta}{2}} \prod_{i=1}^{k} B\left(\frac{\eta}{2}, \frac{1-n+\eta-\beta}{2} +  (i-1)\frac{\eta}{2} \right) \right\}  \label{deriv:hat:pm:mu:reg:holder:1st:term:induction} \\
& \quad \times ( g(c(t-s), z-x) + g(c(t-s), z-x')).\notag
\end{align}

In order to establish an upper-bound for the quantity $p_{m+1} \otimes \partial^{n}_v[\partial_\mu \mH_{m+1}](\mu, s ,t , x, z)(v) - p_{m+1} \otimes \partial^{n}_v[\partial_\mu \mH_{m+1}](\mu, s ,t , x', z)(v)$, we proceed as before by balancing the time singularity thanks to the estimate \eqref{deriv:mes:parametrix:kernel:iterate:induction}. Indeed, assuming first that $r \in [s, (t+s)/2]$, from  \eqref{reg:heat:kernel:deriv} we get
\begin{align*}
\int_{\mathbb{R}^d} &  |p_{m+1}(\mu, s, r, x, y) - p_{m+1}(\mu, s, r, x', y)|  | \partial^{n}_v [\partial_\mu \mH_{m+1}(\mu, s, r ,t , y, z)](v)| \, dy \\
& \leq K_\beta \frac{|x-x'|^\beta}{(t-s)(r-s)^{\frac{1+n + \beta-\eta}{2}}}\Big[ 1+ \mathscr{C}^{n,0}_{m}(C, r-s) (r-s)^{\frac{\eta}{2}} + \frac{1}{t-r} \int_r^t \mathscr{C}^{n, 0}_{m}(C, r'-s) (r'-s)^{\frac{\eta}{2}} \, dr' \Big]  \\
& \quad \times ( g(c(t-s),z-x) + g(c(t-s), z-x') ),
\end{align*}

\noindent so that using again the inequality $t-r\geq (t-s)/2$, the fact that $[s, t] \ni r\mapsto \mathscr{C}^{n, 0}_{m}(C, r-s) (r-s)^{\frac{\eta}{2}}$ is non-decreasing and then integrating by parts
\begin{align*}
\int_s^{\frac{t+s}{2}} \int_{\mathbb{R}^d} & |p_{m+1}(\mu, s, r, x, y)-p_{m+1}(\mu, s, r, x', y)|  | \partial^{n}_v[\partial_\mu \mH_{m+1}(\mu, s, r ,t , y, z)](v)| dy \, dr \\
& \leq K_\beta |x-x'|^\beta \left\{ \frac{1}{(t-s)^{\frac{1+n + \beta -\eta}{2}}}  + \frac{1}{(t-s)^2} \int_s^{t} \frac{1}{(r-s)^{\frac{1+n+\beta-\eta}{2}}} \int_r^t  \mathscr{C}^{n, 0}_{m}(C, r'-s) (r'-s)^{\frac{\eta}{2}} \, dr' \,  dr \right\} \\
& \quad \times ( g(c(t-s) , z-x) + g(c(t-s), z-x') ). \\
& \leq K_\beta |x-x'|^\beta \left\{ \frac{1}{(t-s)^{\frac{1+n+ \beta -\eta}{2}}}  + \frac{1}{(t-s)^2} \int_s^t (r-s)^{1-\frac{(1+n + \beta -\eta)}{2}} \mathscr{C}^{n, 0}_{m}(C, r-s) (r-s)^{\frac{\eta}{2}} \, dr  \right\} \\
 & \quad  \times ( g(c(t-s) , z-x) + g(c(t-s), z-x') )\\
& \leq K_\beta |x-x'|^\beta \left\{ \frac{1}{(t-s)^{\frac{1+n+ \beta -\eta}{2}}}  +  \int_s^t \frac{\mathscr{C}^{n, 0}_{m}(C, r-s)}{(t-r)^{1-\frac{\eta}{2}}(r-s)^{\frac{1+n + \beta -\eta}{2}}} \, dr \right\} \\
& \quad ( g(c(t-s) , z-x) + g(c(t-s), z-x') ) \\
& \leq K_\beta \frac{|x-x'|^\beta}{(t-s)^{\frac{1+n+\beta-\eta}{2}}}  \left\{ 1 + \sum_{k=1}^{m} C^{k} (t-s)^{k \frac{\eta}{2}} \prod_{i=1}^{k} B\left(\frac{\eta}{2}, \frac{1-n+\eta-\beta}{2} +  (i-1)\frac{\eta}{2} \right)  \right\} \\
& \quad \times  ( g(c(t-s) , z-x) + g(c(t-s), z-x') ).
\end{align*}

 Similarly, assuming now that $r \in [(t+s)/2, t]$ so that $r-s \geq (t-s)/2$ we obtain 
\begin{align*}
\int_{\mathbb{R}^d} & |p_{m+1}(\mu, s, r, x, y) - p_{m+1}(\mu, s, r, x', y) |  |\partial^{n}_v [\partial_\mu \mH_{m+1}(\mu, s, r ,t , y, z)](v)| dy \\
&  \leq  K_\beta \frac{|x-x'|^\beta}{(t-s)^{\frac{1+n+\beta}{2}} (t-r)^{1-\frac{\eta}{2}}} \left[1+\mathscr{C}^{n, 0}_{m}(C, r-s) (r-s)^{\frac{\eta}{2}} + \frac{1}{t-r}\int_r^t\mathscr{C}^{n, 0}_{m}(C, r'-s) (r'-s)^{\frac{\eta}{2}} \, dr' \right] \\
& \quad \times (g(c(t-s),z-x) + g(c(t-s), z-x')),
\end{align*}  

\noindent which in turn, by Fubini's theorem, directly yields
\begin{align*}
\int_{\frac{t+s}{2}}^{t} \int_{\mathbb{R}^d} & |p_{m+1}(\mu, s, r, x, y) - p_{m+1}(\mu, s, r, x', y) |  | \partial^{n}_v[\partial_\mu \mH_{m+1}(\mu, s, r ,t , y, z)](v)| \, dy \, dr \\
&  \leq K_\beta  \frac{|x-x'|^\beta}{(t-s)^{\frac{1+n+\beta}{2}}} \int_{\frac{t+s}{2}}^{t} \frac{1}{(t-r)^{1-\frac{\eta}{2}}}  \left[1+\mathscr{C}^{n, 0}_{m}(C, r-s) (r-s)^{\frac{\eta}{2}} + \frac{1}{t-r}\int_r^t\mathscr{C}^{n, 0}_{m}(C, r'-s) (r'-s)^{\frac{\eta}{2}} \, dr' \right]\, dr \\
& \quad \times (g(c(t-s),z-x) + g(c(t-s), z-x')\\
& \leq K_\beta |x-x'|^\beta \left\{ \frac{1}{(t-s)^{\frac{1+n + \beta -\eta}{2}}}  +  \int_s^t \frac{\mathscr{C}^{n, 0}_{m}(C, r-s)}{(t-r)^{1-\frac{\eta}{2}} (r-s)^{\frac{1+n+ \beta -\eta}{2}}} \, dr \right\} \\
& \quad \times  ( g(c(t-s) , z-x) + g(c(t-s), z-x') ) \\
& \leq  K_\beta \frac{|x-x'|^\beta}{(t-s)^{\frac{1+n+\beta-\eta}{2}}}  \left\{ 1 + \sum_{k=1}^{m} C^{k} (t-s)^{k \frac{\eta}{2}} \prod_{i=1}^{k} B\left(\frac{\eta}{2}, \frac{1-n+\eta-\beta}{2} +  (i-1)\frac{\eta}{2} \right)  \right\} \\
& \quad \times  ( g(c(t-s) , z-x) + g(c(t-s), z-x')).
\end{align*}

Gathering the two previous cases, we thus conclude
\begin{align}
| p_{m+1} & \otimes \partial^{n}_v[\partial_\mu \mH_{m+1}(\mu, s ,t , x, z)](v) - p_{m+1}  \otimes \partial^{n}_v[\partial_\mu \mH_{m+1}(\mu, s ,t , x', z)](v)|  \nonumber \\
&  \leq K_\beta \frac{|x-x'|^\beta}{(t-s)^{\frac{1+n+\beta-\eta}{2}}} \left( 1 + \sum_{k=1}^{m} C^{k} (t-s)^{k \frac{\eta}{2}} \prod_{i=1}^{k}   B\left( \frac{\eta}{2}, \frac{1-n+\eta-\beta}{2} + (i-1) \frac{\eta}{2}\right) \right) \label{estimate:convol:pmp1:cross:deriv:mes:Hmp1}\\
& \quad \times  ( g(c(t-s),z-x) + g(c(t-s), z-x') ) \notag
\end{align}

\noindent so that  
\begin{align}
\sum_{r\geq 0}| (p_{m+1} & \otimes \partial^{n}_v[\partial_\mu \mH_{m+1}]) \otimes \mH^{(r)}_{m+1}(\mu, s, t, x, z) (v) -(p_{m+1}  \otimes \partial^{n}_v[\partial_\mu \mH_{m+1}]) \otimes \mH^{(r)}_{m+1}(\mu, s, t, x', z) (v) | \nonumber \\
& \leq K_\beta \frac{|x-x'|^\beta}{(t-s)^{\frac{1+n+\beta-\eta}{2}}} \left\{ 1 + \sum_{k=1}^{m} C^{k} (t-s)^{k \frac{\eta}{2}} \prod_{i=1}^{k} B\left(\frac{\eta}{2}, \frac{1-n+\eta-\beta}{2} +  (i-1)\frac{\eta}{2} \right) \right\} \label{deriv:hat:pm:mu:reg:holder:2nd:term:induction} \\
& \quad \times ( g(c(t-s), z-x) + g(c(t-s), z-x') ). \notag
\end{align}

Gathering the estimates \eqref{deriv:hat:pm:mu:reg:holder:1st:term:induction} and \eqref{deriv:hat:pm:mu:reg:holder:2nd:term:induction} eventually implies that there exist two positive constants $K_\beta:= K(T, \HR, \HE, \beta)$ and $c:=c(\lambda)$ such that 
\begin{align*}
|\partial^{n}_v & [\partial_\mu p_{m+1}(\mu, s, t, x, z)](v) - \partial^{n}_v  [\partial_\mu p_{m+1}(\mu, s, t, x', z)](v)  | \\
& \leq  K_\beta \frac{|x-x'|^\beta}{(t-s)^{\frac{1+n + \beta -\eta}{2}}} \left\{ 1 + \sum_{k=1}^{m} C^{k}(t-s)^{k \frac{\eta}{2}} \prod_{i=1}^{k} B\left(\frac{\eta}{2}, \frac{1-n+\eta-\beta}{2} +  (i-1)\frac{\eta}{2} \right) \right\} \\
& \quad \times  ( g(c(t-s), z-x) + g(c(t-s), z-x') )
\end{align*}
\noindent for any $\beta \in [0,1]$ if $n=0$ and for any $\beta \in [0,\eta)$ if $n=1$. This concludes the proof of \eqref{estimate:deriv:mes:holder:reg:dens:stepmp1:cor}.\\

\noindent \emph{Step 5: proof of the Gaussian estimate \eqref{estimate:deriv:mes:holder:reg:terminal:point:mes:dens:stepmp1:cor}.}\\

We now prove the Gaussian estimate \eqref{estimate:deriv:mes:holder:reg:terminal:point:mes:dens:stepmp1:cor}. Again, we will brief on some technical arguments as it follows from similar arguments as those employed in previous proofs. The estimate \eqref{cross:mes:deriv:reg:holder:terminal point:p:hat:s:r:t} of Corollary \ref{cor:deriv:time:and:mes:phat} (taken with $r=s$) combined with \eqref{equicontinuity:first:estimate:decoupling:mckean} and the space-time inequality \eqref{space:time:inequality} yields for any $\beta \in [0,\eta)$
\begin{align*}
| \partial_v &[ \partial_\mu \widehat{p}_{m+1}(\mu, s, t, x, z)](v) - \partial_v [ \partial_\mu \widehat{p}_{m+1}(\mu, s, t, x, z)](v')|  \nonumber\\
& \leq K_\beta |v-v'|^\beta \left\{ \frac{1}{(t-s)^{1+\frac{ \beta - \eta}{2}}} + \frac{1}{t-s} \int_s^t \frac{\mathscr{C}^{1,\beta}_{m}(C_\beta, r-s)}{(r-s)^{1+\frac{\beta}{2}-\eta }} \, dr \right\} \, g(c(t-s) , z-x) \nonumber \\
& \leq K_\beta |v-v'|^\beta  \left\{ \frac{1}{(t-s)^{1+\frac{ \beta- \eta}{2}}} + \int_s^t \frac{\mathscr{C}^{1,\beta}_{m}(C_\beta, r-s)}{(t-r)^{1-\frac{\eta}{2}} (r-s)^{1+\frac{\beta-\eta}{2}}} \, dr \right\} \, g(c(t-s) , z-x)  \nonumber \\
&  \leq K_\beta  \frac{ |v-v'|^\beta}{(t-s)^{1+\frac{\beta - \eta}{2}}}\left\{ 1 + \sum_{k=1}^{m} C_\beta^{k} (t-s)^{k \frac{\eta}{2}}  \prod_{i=1}^{k} B\left(\frac{\eta}{2}, \frac{\eta-\beta}{2} +  (i-1)\frac{\eta}{2} \right) \right\}  \, g(c(t-s) , z-x). 
\end{align*}

Hence, by induction on $r$, there exists a positive constant $K_\beta:=K(T, \HR,\HE, \beta)$ (which may change from line to line but is independent of $m$ and $C_\beta$) such that
\begin{align*}
|(\partial_v \partial_\mu& \widehat{p}_{m+1} \otimes \mH^{(r)}_{m+1})(\mu, s, t, x, z)(v) - (\partial_v \partial_\mu \widehat{p}_{m+1} \otimes \mH^{(r)}_{m+1})(\mu, s, t, x, z)(v')| \\
 & \leq  K_\beta^{r} |v-v'|^{\beta} \left\{ 1 + \sum_{k=1}^{m} C_\beta^{k}  (t-s)^{k \frac{\eta}{2}} \prod_{i=1}^{k} B\left(\frac{\eta}{2},  \frac{\eta-\beta}{2} +  (i-1)\frac{\eta}{2} \right) \right\} (t-s)^{-1+\frac{(\eta-\beta)}{2} + r \frac{\eta}{2}} \\
& \quad \times \prod_{i=1}^{r} B\left(\frac{\eta}{2}, \frac{\eta-\beta}{2} + (i-1) \frac{\eta}{2}\right) \, g(c(t-s), z-x),
\end{align*}

\noindent which in turn implies
\begin{align}
\sum_{r\geq 0} |(\partial_v \partial_\mu& \widehat{p}_{m+1} \otimes \mH^{(r)}_{m+1})(\mu, s, t, x, z)(v) - (\partial_v \partial_\mu \widehat{p}_{m+1} \otimes \mH^{(r)}_{m+1})(\mu, s, t, x, z)(v')| \nonumber \\
& \leq  K_\beta \frac{|v-v'|^{\beta}}{(t-s)^{1 + \frac{\beta-\eta}{2}}} \left\{ 1 + \sum_{k=1}^{m} C_\beta^{k} (t-s)^{k \frac{\eta}{2}}  \prod_{i=1}^{k} B\left(\frac{\eta}{2},  \frac{\eta-\beta}{2} +  (i-1)\frac{\eta}{2} \right) \right\} \,  g(c(t-s), z-x) \nonumber \\
& \leq K_\beta \frac{|v-v'|^{\beta}}{(t-s)^{1 + \frac{\beta-\eta}{2}}} \left\{ 1 + \sum_{k=1}^{m} C_\beta^{k} (t-s)^{k \frac{\eta}{2}}  \prod_{i=1}^{k} B\left(\frac{\eta}{2},  \frac{\eta - \beta}{2} +  (i-1)\frac{\eta}{2} \right) \right\} \, g(c(t-s), z-x). \label{diff:deriv:hat:pm:mu:v:reg:holder:1st:term}
\end{align}

We now investigate the H\"older regularity of the map $v\mapsto p_{m+1} \otimes \partial_v [\partial_\mu \mH_{m+1}(\mu, s ,t , x, z)](v)$. The estimate \eqref{cross:mes:deriv:parametrix:kernel:s:r:t:reg:holder} allows to balance the time singularity induced by the $L$-derivative of the underlying parametrix kernel so that we are naturally lead to separate the time integral of the space-time convolution into the two disjoint intervals $[s, (t+s)/2]$ and $[(t+s)/2, t]$ as we did before. Skipping some technical details, from \eqref{cross:mes:deriv:parametrix:kernel:s:r:t:reg:holder} and \eqref{equicontinuity:first:estimate:decoupling:mckean}, we obtain
\begin{align*}
| p_{m+1} & \otimes  \partial_v [\partial_\mu \mH_{m+1}(\mu, s ,t , x, z)](v) - p_{m+1}  \otimes  \partial_v [\partial_\mu \mH_{m+1}(\mu, s ,t , x, z)](v') | \\
&  \leq K_\beta |v-v'|^{\beta}  \left( \frac{1}{(t-s)^{1 + \frac{\beta-\eta}{2}}} +  \int_s^t  \frac{\mathscr{C}^{1,\beta}_{m}(C_\beta, r-s)}{(t-r)^{1-\frac{\eta}{2}} (r-s)^{1+\frac{\beta-\eta}{2}}} \, dr \right)  g(c(t-s),z-x) \\
& \leq K_\beta \frac{ |v-v'|^{\beta}}{(t-s)^{1+ \frac{\beta-\eta}{2}}} \left( 1 + \sum_{k=1}^{m} C_\beta^{k} (t-s)^{k \frac{\eta}{2}}\prod_{i=1}^{k} B\left( \frac{\eta}{2}, \frac{\eta - \beta}{2} + (i-1) \frac{\eta}{2}\right)  \right)  g(c(t-s),z-x),
\end{align*}

\noindent so that  
\begin{align}
\sum_{r\geq 0}| (p_{m+1} & \otimes  \partial_v \partial_\mu \mH_{m+1})\otimes \mH^{(r)}_{m+1}(\mu, s, t, x, z)(v) - (p_{m+1}  \otimes  \partial_v \partial_\mu \mH_{m+1})\otimes \mH^{(r)}_{m+1}(\mu, s, t, x, z)(v')  | \nonumber \\
& \leq K_\beta \frac{ |v-v'|^{\beta}}{(t-s)^{1 + \frac{\beta-\eta}{2}}} \left( 1 + \sum_{k=1}^{m} C_\beta^{k} (t-s)^{k \frac{\eta}{2}} \prod_{i=1}^{k} B\left( \frac{\eta}{2}, \frac{\eta - \beta}{2} + (i-1) \frac{\eta}{2}\right)  \right)  g(c(t-s), z-x). \label{diff:deriv:hat:pm:mu:v:reg:holder:2nd:term}
\end{align}

Now, the representation formula \eqref{representation:formula:cross:lions:deriv:dens} together with the estimates \eqref{diff:deriv:hat:pm:mu:v:reg:holder:1st:term} and \eqref{diff:deriv:hat:pm:mu:v:reg:holder:2nd:term} show that for any $\beta \in [0,\eta)$ there exist two positive constants $K_\beta:= K(T, \HR, \HE, \beta)$, $c:=c(\lambda)$ such that for any $(\mu, x, z, v) \in \pp \times (\mathbb{R}^d)^3$, for any $0\leq s < t \leq T$ and any positive constant $C_\beta$ 
\begin{align*}
| & \partial_v [\partial_\mu p_{m+1}(\mu, s, t, x, z)](v) - \partial_v [\partial_\mu p_{m+1}(\mu, s, t, x, z)](v') |\\
&  \leq  K_\beta \frac{|v-v'|^{\beta}}{(t-s)^{1 + \frac{\beta-\eta}{2}}} \left( 1 + \sum_{k=1}^{m} C_\beta^{k} (t-s)^{k \frac{\eta}{2}} \prod_{i=1}^{k} B\left( \frac{\eta}{2}, \frac{\eta - \beta}{2} + (i-1) \frac{\eta}{2}\right)  \right) g(c(t-s), z-x).
\end{align*}

The proof of \eqref{estimate:deriv:mes:holder:reg:terminal:point:mes:dens:stepmp1:cor} is now complete.\\

\noindent \emph{Step 6: proof of the pointwise Gaussian estimate \eqref{estimate:deriv:time:dens:stepmp1:cor}.}\\

The proof of \eqref{estimate:deriv:time:dens:stepmp1:cor} again follows from similar arguments so we will be brief on some technical details. The road is clear inasmuch one has to establish a pointwise Gaussian estimate for each term appearing in the identity \eqref{representation:formula:time:deriv:dens}. The estimates \eqref{time:deriv:p:hat:s:t}, \eqref{Gaussian:estimate:Phim} and \eqref{time:derivative:induction:decoupling:mckean} together with the space-time inequality \eqref{space:time:inequality} clearly yield
\begin{align*}
|\partial_s & \widehat{p}_{m+1}(\mu, s, t, x, z) | + | \Phi_{m+1}(\mu, s , t, x, z)|  \\
&  \leq K \left\{ \frac{1}{t-s} + \frac{1}{t-s} \int_s^t  \int_{(\rr^d)^2} (|y'-x'|^{\eta}\wedge 1) | \partial_s p_{m}(\mu, s, r, x', y') | \, \mu(dx') \, dy' \, dr \right\}  g(c(t-s), z-x)\\
& \leq K \left\{ \frac{1}{t-s} + \frac{1}{t-s} \int_s^t \frac{\mathscr{C}^{1,0}_{m}(C, r-s)}{(r-s)^{1 - \frac{\eta}{2}}} \, dr \right\}  g(c(t-s), z-x)\\
& \leq K \left\{ \frac{1}{t-s} + \frac{1}{(t-s)^{\frac{\eta}{2}}} \int_s^t \frac{\mathscr{C}^{1,0}_{m}(C, r-s)}{(t-r)^{1-\frac{\eta}{2}}(r-s)^{1 - \frac{\eta}{2}}} \, dr \right\}  g(c(t-s), z-x)\\
& \leq \frac{K}{t-s}   \left\{ 1+ \sum_{k=1}^{m} C^{k}  (t-s)^{k \frac{\eta}{2}} \prod_{i=1}^{k} B\left(\frac{\eta}{2}, \frac{\eta}{2} +  (i-1)\frac{\eta}{2} \right)  \right\}  g(c(t-s), z-x).
\end{align*}
We now deal with $ \partial_s \widehat{p}_{m+1}\otimes \Phi_{m+1}(\mu, s, t, x, z)$ and use the decomposition
\begin{align*}
 \partial_s \widehat{p}_{m+1}& \otimes \Phi_{m+1}(\mu, s, t, x, z) \\
 & =  \int_s^{\frac{t+s}{2}} \int_{\mathbb{R}^d}  \partial_s \widehat{p}_{m+1}(\mu, s, r, x ,y) [\Phi_{m+1}(\mu, s, r, t, y, z) - \Phi_{m+1}(\mu, s, r, t, x, z)]  \, dy \, dr \\
& \quad +  \int_s^{\frac{t+s}{2}} \Phi_{m+1}(\mu, s, r, t, x, z) \int_{\mathbb{R}^d} [\partial_s \widehat{p}^{y}_{m+1}(\mu, s, r, x ,y) - \partial_s \widehat{p}^{x}_{m+1}(\mu, s, r, x ,y)] \, dy\, dr \\
& \quad + \int_{\frac{t+s}{2}}^{t}  \int_{\mathbb{R}^d}  \partial_s \widehat{p}_{m+1}(\mu, s, r, x ,y) \Phi_{m+1}(\mu, s, r, t, y, z) \, dy\\
& =: {\rm I} + {\rm II} + {\rm III}.
\end{align*}
In order to deal with ${\rm I}$, we use the estimate \eqref{time:deriv:p:hat:s:t} and remove the time singularity by combining the estimate \eqref{holder:reg:voltera:kernel} with the space-time inequality \eqref{space:time:inequality}. We obtain
\begin{align*}
|{\rm I}| & \leq K \int_s^{\frac{t+s}{2}} \frac{1}{(r-s)^{1-\frac{\beta}{2}}(t-r)^{1+\frac{\beta-\eta}{2}}} \left\{ 1+ \int_r^t  (|y'-x'|^{\eta}\wedge 1) | \partial_s p_{m}(\mu, s, r', x', y') | \, \mu(dx') \, dy' \, dr' \right\}\, dr\\
& \quad \times g(c(t-s), z-x)\\
& \leq \frac{K}{(t-s)^{1+\frac{\beta-\eta}{2}}} \int_s^{\frac{t+s}{2}} \frac{1}{(r-s)^{1-\frac{\beta}{2}}} \left\{ 1 + \int_r^t \frac{\mathscr{C}^{1,0}_{m}(C, r'-s)}{(r'-s)^{1 - \frac{\eta}{2}}} \, dr'\right\} \, dr \,  g(c(t-s), z-x) \\
& \leq \frac{K}{(t-s)^{1-\frac{\eta}{2}}} \left\{ 1 + \int_s^t  \frac{\mathscr{C}^{1,0}_{m}(C, r-s)}{(r-s)^{1 - \frac{\eta}{2}}} \, dr \right\} \, g(c(t-s), z-x)\\
& \leq \frac{K}{t-s}   \left\{ 1+ \sum_{k=1}^{m} C^{k}  (t-s)^{k \frac{\eta}{2}} \prod_{i=1}^{k} B\left(\frac{\eta}{2}, \frac{\eta}{2} +  (i-1)\frac{\eta}{2} \right)  \right\}  g(c(t-s), z-x).
\end{align*}

For the second term, we use the fact that $(t-s)/2\leq t-r \leq t-s$ for $r\in[s, (t+s)/2]$, the estimate \eqref{time:deriv:holder:reg:p:hat:s:t} with $\beta=1/2$ together with the space-time inequality \eqref{space:time:inequality} and finally the fact that $t \mapsto \mathscr{C}^{1,0}_{m}(C, t)$ is non-decreasing. We obtain
\begin{align*}
|{\rm II}| & \leq \frac{K}{(t-s)^{1-\frac{\eta}{2}}} \int_s^{\frac{t+s}{2}} \frac{1}{(r-s)^{1-\frac{\eta}{4}}} \left(1 + \int_s^r \frac{\mathscr{C}^{1,0}_{m}(C, r'-s)}{(r'-s)^{1-\frac{\eta}{4}}} \, dr'\right) \, dr \, g(c(t-s), z-x) \\
& \leq \frac{K}{(t-s)^{1-\frac{\eta}{2}}} \left( (t-s)^{\frac{\eta}{4}} + \int_s^{t} \frac{\mathscr{C}^{1,0}_{m}(C, r-s)}{(r-s)^{1-\frac{\eta}{2}}} \, dr\right) \, g(c(t-s), z-x)\\
& \leq \frac{K}{t-s}   \left\{ 1+ \sum_{k=1}^{m} C^{k}  (t-s)^{k \frac{\eta}{2}} \prod_{i=1}^{k} B\left(\frac{\eta}{2}, \frac{\eta}{2} +  (i-1)\frac{\eta}{2} \right)  \right\}  g(c(t-s), z-x).
\end{align*}

Finally, using again \eqref{time:deriv:p:hat:s:t}, \eqref{Gaussian:estimate:Phim}, the space-time inequality \eqref{space:time:inequality} as well as the fact that $(t-s)/2\leq r-s \leq t-s$ for $r\in[(t+s)/2, t]$, we obtain
\begin{align*}
|{\rm III}| & \leq \frac{K}{t-s} \int_{\frac{t+s}{2}}^{t} \frac{1}{(t-r)^{1-\frac{\eta}{2}}} \left(1+ \int_s^r  \frac{\mathscr{C}^{1,0}_{m}(C, r'-s)}{(r'-s)^{1-\frac{\eta}{2}}} \, dr'\right)\, dr  \, g(c(t-s), z-x)\\
& \leq  \frac{K}{(t-s)^{1-\frac{\eta}{2}}} \left(1+ \int_s^t  \frac{\mathscr{C}^{1,0}_{m}(C, r-s)}{(r-s)^{1-\frac{\eta}{2}}} \, dr\right)  \, g(c(t-s), z-x)\\
& \leq \frac{K}{t-s}   \left\{ 1+ \sum_{k=1}^{m} C^{k}  (t-s)^{k \frac{\eta}{2}} \prod_{i=1}^{k} B\left(\frac{\eta}{2}, \frac{\eta}{2} +  (i-1)\frac{\eta}{2} \right)  \right\}  g(c(t-s), z-x).
\end{align*}
Gathering the estimates on ${\rm I}$, ${\rm II}$ and ${\rm III}$, we thus conclude
$$
|\partial_s \widehat{p}_{m+1} \otimes \Phi_{m+1}(\mu, s, t, x, z)| \leq  \frac{K}{t-s}   \left\{ 1+ \sum_{k=1}^{m} C^{k}  (t-s)^{k \frac{\eta}{2}} \prod_{i=1}^{k} B\left(\frac{\eta}{2}, \frac{\eta}{2} +  (i-1)\frac{\eta}{2} \right)  \right\}  g(c(t-s), z-x).
$$

In order to handle the quantity $p_{m+1} \otimes \partial_s \mH_{m+1}(\mu, s ,t, x, z)$, we write the time integral on the interval $[s, t]$ of the space-time convolution as the sum of the time integrals on $[s, (t+s)/2]$ and on $[(t+s)/2, t]$. For the first time integral, we use the estimate \eqref{cross:time:deriv:parametrix:kernel:s:r:t} with $\beta=0$ while for the second we use the same estimate but with $\beta=1$. Skipping some standard computations, we obtain
\begin{align*}
| p_{m+1} & \otimes \partial_s \mH_{m+1}(\mu, s ,t, x, z) |\\
&   \leq K \left(  \int_s^t \left\{\frac{\mathscr{C}^{1,0}_{m}(C, r-s)}{(t-s)(t-r)^{1-\frac{\eta}{2}}} + \frac{\mathscr{C}^{1,0}_{m}(C, r-s)}{(t-s)(r-s)^{1-\frac{\eta}{2}}} \right. \right. \\
& \left. \left. \quad \quad+\frac{1}{(t-r)^{2-\frac{\eta}{2}}}\int_r^t \frac{\mathscr{C}^{1,0}_{m}(C, r'-s)}{(r'-s)^{1-\frac{\eta}{2}}} \, dr'  \right\} \, dr \right) \, g(c(t-s), z-x)\\
& \leq K  \left( \frac{1}{(t-s)^{\frac{\eta}{2}}}\int_s^t \frac{\mathscr{C}^{1,0}_{m}(C, r-s)}{(t-r)^{1-\frac{\eta}{2}}(r-s)^{1-\frac{\eta}{2}}} \, dr  + \frac{1}{(t-s)^{1-\frac{\eta}{2}}}\int_s^t \frac{\mathscr{C}^{1,0}_{m}(C, r-s)}{(r-s)^{1-\frac{\eta}{2}}} \, dr   \right) \, g(c(t-s), z-x)\\
& \leq \frac{K}{t-s}   \left\{ \sum_{k=1}^{m} C^{k}  (t-s)^{k \frac{\eta}{2}} \prod_{i=1}^{k} B\left(\frac{\eta}{2}, \frac{\eta}{2} +  (i-1)\frac{\eta}{2} \right)  \right\}  g(c(t-s), z-x).
\end{align*}
The previous estimate in turn yields
\begin{align*}
 |[(p_{m+1} & \otimes \partial_s \mH_{m+1})  \otimes \Phi_{m+1}](\mu, s, t, x, z)| \\
 & \leq K \int_s^t  \frac{1}{(r-s)}   \left\{ \sum_{k=1}^{m} C^{k}  (r-s)^{k \frac{\eta}{2}} \prod_{i=1}^{k} B\left(\frac{\eta}{2}, \frac{\eta}{2} +  (i-1)\frac{\eta}{2} \right)  \right\} \frac{1}{(t-r)^{1-\frac{\eta}{2}}} \, dr \, g(c(t-s), z-x) \\
 & \leq \frac{K}{(t-s)^{1-\frac{\eta}{2}}} \sum_{k=1}^{m} C^{k}  (t-s)^{k \frac{\eta}{2}} \prod_{i=1}^{k} B\left(\frac{\eta}{2}, \frac{\eta}{2} +  (i-1)\frac{\eta}{2} \right) \, g(c(t-s), z-x).
\end{align*}
Gathering the above estimates allows to conclude the proof of the pointwise Gaussian estimate \eqref{estimate:deriv:time:dens:stepmp1:cor}.
\end{proof}

\subsection{Proofs of the estimates \eqref{regularity:measure:estimate:v1:v2:v3:decoupling:mckean} and \eqref{regularity:time:estimate:v1:v2:v3:decoupling:mckean}.}\label{proofs:intermediate:estimates:as:consequence}

We here establish the estimates \eqref{regularity:measure:estimate:v1:v2:v3:decoupling:mckean} and \eqref{regularity:time:estimate:v1:v2:v3:decoupling:mckean} which both appear as a consequence of the Gaussian estimates established in the previous section. In this regard, they cannot be considered as part of the second part of the induction step. 

In what follows, to make the notations simpler, for two maps $h$ and $f$ defined respectively on $\pp$ and $[0 ,t)$, for some prescribed $t>0$, for any probability measures $\mu, \mu' \in \pp$ and any $(s_1, s_2) \in [0,t)^2$, we will write $\Delta_{\mu, \mu'} h(\mu) = h(\mu) - h(\mu')$ and $\Delta_{s_1, s_2} f(s) = f(s_1\vee s_2) - f(s_1\wedge s_2)$. In particular, $\Delta_{\mu, \mu'} p_m(\mu, s, t, x, z) = p_m(\mu, s, t, x, z) - p_m(\mu', s, t, x, z)$ and $\Delta_{s_1, s_2} p_m(\mu, s, t, x, z) = p_m(\mu, s_1\vee s_2, t, x, z) - p_m(\mu, s_1\wedge s_2, t, x, z)$.

We start with the following auxiliary result that will be useful in order to establish the estimates \eqref{regularity:measure:estimate:v1:v2:v3:decoupling:mckean}. Its proof is postponed to Section \ref{section:proof:lem:technical:estimate:reg:wasserstein:metric:coefficients:densities} of the appendix.

\begin{lem}\label{lem:technical:estimate:reg:wasserstein:metric:coefficients:densities}  There exist some positive constants $K$, $c$ such that for any $\beta \in [\eta,1]$, any $\beta' \in [0,1]$, any $\alpha \in [0, \eta]$, any $\beta'' \in [\alpha, 1]$, any integer $m$, any $t\in (0,T]$, any $(s, x, y, y', z) \in [0,t) \times (\mathbb{R}^d)^4$, any $r\in (s, t)$, any $\mu, \mu' \in \mathcal{P}_2(\mathbb{R}^d)$ (denoting by $\xi$ and $\xi'$ any random variables with respective law $\mu$ and $\mu'$) and any $(i, j) \in \left\{1, \cdots, d\right\}^2$
\begin{align}
|a_{i, j}(t, x, [X^{s, \xi, (m)}_t])  - & a_{i, j}(t, x, [X^{s, \xi', (m)}_t])| + |b_{i}(t, x, [X^{s, \xi, (m)}_t])  - b_{i}(t, x, [X^{s, \xi', (m)}_t])| \notag \\
& \leq K \frac{W^{\beta}_2(\mu, \mu')}{(t-s)^{\frac{\beta-\eta}{2}}}, \label{diff:mes:drift:diff:coefficients}
\end{align}
\begin{align}
 | \Delta_{\mu, \mu'} \partial^{n}_x \widehat{p}_{m+1}(\mu, s, t, x, z) | \leq K \frac{W^{\beta}_2(\mu, \mu')}{(t-s)^{\frac{n+\beta-\eta}{2}}} \, g(c(t-s), z-x), \quad n\in \left\{0,1, 2\right\}, \label{gaussian:bound:diff:deriv:hat:pm:same:time}
\end{align}
\begin{align}
 | \Delta_{\mu, \mu'} \partial^{n}_x \widehat{p}_{m+1}(\mu, s, r, t, x, z) | \leq K \frac{W^{\beta}_2(\mu, \mu')}{(t-r)^{\frac{n}{2}}(r-s)^{\frac{\beta-\eta}{2}}} \, g(c(t-r), z-x), \quad n\in \left\{0,1, 2, 3\right\}, \label{gaussian:bound:diff:deriv:hat:pm:different:time}
\end{align}
\begin{align}
 | \Delta_{\mu, \mu'} p_{m+1}(\mu, s, t, x, z) | & \leq K \frac{W^{\beta}_2(\mu, \mu')}{(t-s)^{\frac{\beta-\eta}{2}}} \, g(c(t-s), z-x), \label{gaussian:bound:diff:pmp1:mu:mup}
\end{align}
\begin{align}
|a_{i, j}& (t, x, [X^{s, \xi, (m)}_t])  -  a_{i, j}(t, x, [X^{s, \xi', (m)}_t]) - (a_{i, j}(t, y, [X^{s, \xi, (m)}_t])  - a_{i, j}(t, y, [X^{s, \xi', (m)}_t]))| \notag \\
 & + |b_{i}(t, x, [X^{s, \xi, (m)}_t])  - b_{i}(t, x, [X^{s, \xi', (m)}_t]) - (b_{i}(t, y, [X^{s, \xi, (m)}_t])  - b_{i}(t, y, [X^{s, \xi', (m)}_t]))| \notag \\
& \quad \quad \leq  K ( |x-y|^{\eta-\alpha} \wedge 1)\frac{W^{\beta''}_2(\mu, \mu')}{(t-s)^{\frac{\beta''-\alpha}{2}}},  \label{diff:mes:with:holder:reg:space:drift:diff:coefficients}
\end{align}
\begin{align}
 | &\Delta_{\mu, \mu'} \mH_{m+1}(\mu, s, r ,t , x ,z) |  \notag \\
 & \leq K \left\{\frac{1}{(t-r)^{1-\frac{\eta}{2}}(r-s)^{\frac{\beta}{2}}} \wedge\frac{1}{(t-r) (r-s)^{\frac{\beta-\eta}{2}}} \right\} W^{\beta}_2(\mu, \mu') \, g(c(t-r), z - x), \label{gaussian:bound:diff:Hm:mu}
\end{align}
\begin{align}
 | &\Delta_{\mu, \mu'} \Phi_{m+1}(\mu, s, r ,t , x ,z) |  \leq K_{\beta'} \frac{W^{\beta'}_2(\mu, \mu')}{(t-r)^{1-\frac{\eta}{2}}(r-s)^{\frac{\beta'}{2}}}   \, g(c(t-r), z - x), \label{gaussian:bound:diff:Phimp1:mu}
\end{align}
\begin{align}
 | &\Delta_{\mu, \mu'} [\mH_{m+1}(\mu, s, r ,t , x ,z ) - \mH_{m+1}(\mu, s, r ,t , y ,z )] |  \notag \\
 & \leq K |x-y|^{\eta-\alpha} \frac{ W^{\beta}_2(\mu, \mu')}{(t-r)^{1+\frac{\eta-\alpha}{2}} (r-s)^{\frac{\beta-\eta}{2}}}   \left\{ g(c(t-r), z - x)  + g(c(t-r), z - y) \right\}, \label{gaussian:bound:diff:Hm:mu:holder:reg}
\end{align}

\begin{align}
 | \Delta_{\mu, \mu'} [\partial^2_{x_i, x_j} [\widehat{p}^{y}_{m+1}& -  \widehat{p}^{y'}_{m+1}](\mu, s ,t , x ,z )]  |  \leq K (|y-y'|^{\eta} \wedge 1) \frac{ W^{\beta'}_2(\mu, \mu')}{(t-s)^{1+\frac{\beta'}{2}} }   g(c(t-s), z - x).\label{gaussian:bound:diff:second:deriv:p:hat:mu:holder:reg}
\end{align}

\end{lem}

We are now ready to prove that the estimates \eqref{regularity:measure:estimate:v1:v2:v3:decoupling:mckean} hold for any positive integer $m$.

\begin{prop}\label{prop:representation:and:gaussian:estimate:diff:mes:deriv:heat:kernel}
For any $(\mu, \mu', x, z) \in (\mathcal{P}_2(\rr^d))^2  \times (\mathbb{R}^d)^2$ and any $0\leq s < t \leq T$, the following representation in infinite series holds
\begin{align}
\Delta_{\mu, \mu'}  & p_{m+1} (\mu, s, t, x, z) \notag\\
& = \sum_{k\geq0} \left\{ \Delta_{\mu, \mu'} \widehat{p}_{m+1} + p^{\mu'}_{m+1}\otimes \Delta_{\mu, \mu'} \mH_{m+1} \right\} \otimes \mH^{(k)}_{m+1}(\mu, s, t, x, z)\label{representation:diff:deriv:pmp1}
\end{align}

\noindent where $p^{\mu'}_{m+1}$ stands for the density function $z\mapsto p_{m+1}(\mu', s, t, x, z)$.

Moreover, the following pointwise Gaussian estimates hold: for any $\beta \in [\eta, 1]$ if $n\in \left\{0, 1 \right\}$ or any $\beta \in [0,\eta)$ if $n=2$, there exist some positive constants $K_\beta$ and $c$ such that for any $(\mu, \mu', x, z) \in (\mathcal{P}_2(\mathbb{R}^d))^2 \times (\mathbb{R}^d)^2$ and any $0\leq s < t \leq T$ 
\begin{equation}
\label{gaussian:estimate:diff:mes:deriv:heat:kernel}
 | \Delta_{\mu, \mu'} \partial^{n}_x p_{m+1}(\mu, s, t, x, z) |   \leq K_\beta \frac{W^{\beta}_2(\mu, \mu')}{(t-s)^{\frac{n+\beta-\eta}{2}}} \, g(c(t-s), z-x), \quad n \in \left\{0,1\right\},
 \end{equation}
 \noindent and
 \begin{equation}
\label{gaussian:estimate:diff:mes:deriv:heat:kernel:part2:statement:proposition}
 | \Delta_{\mu, \mu'} \partial^{2}_x p_{m+1}(\mu, s, t, x, z) |   \leq K_\beta \frac{W^{\beta}_2(\mu, \mu')}{(t-s)^{1+\frac{\beta}{2}}} \, g(c(t-s), z-x).
 \end{equation}
\end{prop}

\bigskip

We importantly note that the two estimates \eqref{gaussian:estimate:diff:mes:deriv:heat:kernel} directly give the estimates \eqref{regularity:measure:estimate:v1:v2:v3:decoupling:mckean} for $n=0$ and $n=1$. Indeed, if $W^2_2(\mu, \mu') \leq t-s$, then it suffices to remark that \eqref{gaussian:estimate:diff:mes:deriv:heat:kernel} is still valid for any $\beta \in [0,1]$ and if $W^2_2(\mu, \mu') \geq t-s$, we directly use the Gaussian estimates \eqref{bound:derivative:heat:kernel} to derive \eqref{regularity:measure:estimate:v1:v2:v3:decoupling:mckean}. 

\bigskip

\begin{proof}
It directly follows from \eqref{gaussian:bound:diff:deriv:hat:pm:same:time}, \eqref{gaussian:bound:diff:Hm:mu} and an induction on $k$ that for any $\beta \in [\eta,1]$
\begin{align*}
|\Delta_{\mu, \mu'} \widehat{p}_{m+1} \otimes \mH^{(k)}_{m+1}(\mu, s, t, x, z)| & \leq K^k W^{\beta}_2(\mu, \mu') (t-s)^{\frac{(\eta-\beta)}{2} +  k\frac{\eta}{2}} \prod_{i=1}^{k} B\left(\frac{\eta}{2}, 1 + \frac{\eta-\beta}{2} + (i-1)\frac{\eta}{2}\right) \\
& \quad \times g(c(t-s), z-x)
\end{align*}
\noindent and
\begin{align*}
|(p^{\mu'}_{m+1} \otimes \Delta_{\mu, \mu'} \mH_{m+1}) \otimes \mH^{(k)}_{m+1}(\mu, s, t, x, z)| & \leq K^k W^{\beta}_2(\mu, \mu') (t-s)^{\frac{(\eta-\beta)}{2} +  k\frac{\eta}{2}} \prod_{i=1}^{k} B\left(\frac{\eta}{2}, 1 + \frac{\eta-\beta}{2} + (i-1)\frac{\eta}{2}\right) \\
& \quad \times g(c(t-s), z-x).
\end{align*}

Now, the representation in infinite series \eqref{series:approx:mckean} satisfied by $p_{m+1}(\mu, s, t, x, z)$ which equivalently writes $p_{m+1} = \widehat{p}_{m+1} + p_{m+1}\otimes \mH_{m+1}$ implies
\begin{align*}
\Delta_{\mu, \mu'} p_{m+1}  &  = \Delta_{\mu, \mu'} \widehat{p}_{m+1} + p^{\mu'}_{m+1} \otimes \Delta_{\mu, \mu'} \mH_{m+1} + \Delta_{\mu, \mu'} p_{m+1} \otimes \mH_{m+1},
\end{align*}

\noindent which in turn, from the two previous estimates, yields
\begin{align}
\Delta_{\mu, \mu'} p_{m+1}   = \sum_{k\geq0} \left\{\Delta_{\mu, \mu'} \widehat{p}_{m+1} + p^{\mu'}_{m+1}\otimes \Delta_{\mu, \mu'} \mH_{m+1} \right\} \otimes \mH^{(k)}_{m}. \label{representation:diff:mes:pm}
\end{align}

\noindent Moreover, the previous series converges absolutely and locally uniformly with respect to the variables $\mu, \mu', s, x, z$. Observe also that from \eqref{infinite:series:Phi:step:m}, the above series writes
\begin{align*}
\Delta_{\mu, \mu'} p_{m+1} (\mu, s, t, x, z) & = \Delta_{\mu, \mu'} \widehat{p}_{m+1}(\mu, s, t, x, z)  + (p_{m+1}\otimes (\Delta_{\mu, \mu'} \mH_{m+1}))(\mu', s, t, x, z)   \\
& \quad + \left\{\Delta_{\mu, \mu'} \widehat{p}_{m+1} + p^{\mu'}_{m+1}\otimes \Delta_{\mu, \mu'} \mH_{m+1} \right\} \otimes \Phi_{m+1}(\mu, s, t, x, z).
\end{align*}

We now aim at differentiating $n$ times ($n=1, 2$) with respect to the space variable $x$ the above representation formula, so that, formally speaking, one has 
\begin{align}
\partial^{n}_x \Delta_{\mu, \mu'} p_{m+1} (\mu, s, t, x, z) & =  \Delta_{\mu, \mu'} \partial^n_x \widehat{p}_{m+1}(\mu, s, t, x, z)  + (\partial_x^n p_{m+1}\otimes (\Delta_{\mu, \mu'} \mH_{m+1}))(\mu', s, t, x, z)  \nonumber \\
& \quad + \left\{\Delta_{\mu, \mu'} \partial_x^n \widehat{p}_{m+1} + \partial_x^n p^{\mu'}_{m+1}\otimes \Delta_{\mu, \mu'} \mH_{m+1} \right\} \otimes \Phi_{m+1}(\mu, s, t, x, z). \label{representation:diff:mes:pm:p1:bis}
\end{align}

Let us rigorously justify the interchange of differentiation and summation. We proceed as follows. From the estimates \eqref{gaussian:bound:diff:deriv:hat:pm:same:time} and the dominated convergence theorem, we first deduce that the map $x\mapsto \Delta_{\mu, \mu'} \widehat{p}_{m+1} \otimes  \Phi_{m+1}(\mu, s, t, x, z)$ is continuously differentiable and satisfies for any $\beta \in [\eta, 1]$ 
\begin{align}
| \Delta_{\mu, \mu'} \partial_x\widehat{p}_{m+1}& \otimes  \Phi_{m+1}(\mu, s, t, x, z)| \leq K \frac{W^{\beta}_2(\mu, \mu')}{(t-s)^{1+\frac{\beta}{2}-\eta}} \, g(c(t-s), z-x). \label{iter:first:deriv:diff:mes:part:1}
\end{align}
 
We then split the domain of the time integral of the space-time convolution $(\partial_x p_{m+1}\otimes \Delta_{\mu, \mu'} \mH_{m+1})(\mu, s, t, x, z)$ into the two disjoint time intervals $[s, (t+s)/2]$ and $[(t+s)/2, t]$. From \eqref{bound:derivative:heat:kernel} and \eqref{gaussian:bound:diff:Hm:mu}, for any $\beta \in [\eta,1]$, one has
\begin{align*}
\int_s^{\frac{t+s}{2}} \int_{\mathbb{R}^d} & |\partial_x p_{m+1}(\mu, s, r, x, y)|  |\Delta_{\mu, \mu'} \mH_{m+1}(\mu, s , r , t, y, z)| \, dy \, dr \\
& \leq K W^{\beta}_2(\mu, \mu') \int_s^{\frac{t+s}{2}} \frac{1}{(r-s)^{\frac12}} \frac{1}{(t-r)(r-s)^{\frac{\beta-\eta}{2}}} \, g(c(t-s), z-x) \\
& \leq K \frac{W^{\beta}_2(\mu, \mu')}{(t-s)^{\frac{1+\beta-\eta}{2}}} \,  g(c(t-s), z-x)
\end{align*}
\noindent and similarly
\begin{align*}
\int_{\frac{t+s}{2}}^{t} \int_{\mathbb{R}^d} & |\partial_x p_{m+1}(\mu, s, r, x, y)|  |\Delta_{\mu, \mu'} \mH_{m+1}(\mu, s , r , t, y, z)| \, dy \, dr \\
& \leq K W^{\beta}_2(\mu, \mu') \int_{\frac{t+s}{2}}^{t} \frac{1}{(r-s)^{\frac12}} \frac{1}{(t-r)^{1-\frac{\eta}{2}}(r-s)^{\frac{\beta}{2}}} \, g(c(t-s), z-x) \\
& \leq K \frac{W^{\beta}_2(\mu, \mu')}{(t-s)^{\frac{1+\beta-\eta}{2}}} \,  g(c(t-s), z-x).
\end{align*}

We thus deduce from the dominated convergence theorem that the map $x\mapsto (p_{m+1} \otimes  \Delta_{\mu, \mu'} \mH_{m+1})(\mu, s, t, x, z)$ is continuously differentiable and satisfies 
$$
|(\partial_x p_{m+1} \otimes  \Delta_{\mu, \mu'} \mH_{m+1})(\mu, s, t, x, z)| \leq K \frac{W^{\beta}_2(\mu, \mu')}{(t-s)^{\frac{1+\beta-\eta}{2}}} \,  g(c(t-s), z-x).
$$
 In a similar manner, we conclude that the map $x\mapsto (p_{m+1} \otimes \Delta_{\mu, \mu'} \mH_{m+1}) \otimes  \Phi_{m+1}(\mu, s, t, x, z)$ is continuously differentiable and satisfies 
\begin{align}
|(\partial_x p_{m+1} \otimes  &\Delta_{\mu, \mu'} \mH_{m+1}) \otimes  \Phi_{m+1}(\mu, s, t, x, z)| \leq K \frac{W_2^{\beta}(\mu, \mu')}{(t-s)^{\frac{1+\beta}{2}-\eta}} \,  g(c(t-s), z-x). \label{iter:first:deriv:diff:mes:part:2}
\end{align}

We thus conclude that the map $x\mapsto \Delta_{\mu, \mu'} p_{m+1} (\mu, s, t, x, z) $ is continuously differentiable, satisfies the announced representation \eqref{representation:diff:mes:pm:p1:bis} as well as the estimate \eqref{gaussian:estimate:diff:mes:deriv:heat:kernel} for $n=1$.

We now prove that $x\mapsto \Delta_{\mu, \mu'} p_{m+1}(\mu, s, t, x, z)$ is twice continuously differentiable, satisfies the representation \eqref{representation:diff:mes:pm:p1:bis} and that the estimate \eqref{gaussian:estimate:diff:mes:deriv:heat:kernel:part2:statement:proposition} is valid.
In order to prove that the map $x\mapsto (\Delta_{\mu, \mu'} \widehat{p}_{m+1} \otimes \Phi_{m+1})(\mu, s, t, x, z)$ is twice continuously differentiable, we proceed as follows. For any $r \in [s, (t+s)/2]$ and any $x_0 \in \mathbb{R}^d$, it holds
\begin{align*}
\partial_x^2 \int_{\mathbb{R}^d} & \Delta_{\mu, \mu'}  \widehat{p}_{m+1}(\mu, s, r, x, y) \Phi_{m+1}(\mu, s, r, t, y, z) \, dy \\
& = \int_{\mathbb{R}^d} \Delta_{\mu, \mu'}  \partial^2_x\widehat{p}_{m+1}(\mu, s, r, x, y) [\Phi_{m+1}(\mu, s, r, t, y, z)  - \Phi_{m+1}(\mu, s, r, t, x_0, z) ] \, dy \\
& \quad + \Phi_{m+1}(\mu, s, r, t, x_0, z) \int_{\mathbb{R}^d} \Delta_{\mu, \mu'}  [\partial^2_x (\widehat{p}^{y}_{m+1}-\widehat{p}^{x_0}_{m+1})(\mu, s, r, x, y)] \, dy
\end{align*}
\noindent where we used the fact that $\int_{\mathbb{R}^d} \widehat{p}^{x_0}_{m+1}(\mu, s, r, x, y) \, dy = 1$, for any $x_0 \in \mathbb{R}^d$. We now select $x_0=x$ in the above identity. From the estimates \eqref{holder:reg:voltera:kernel}, \eqref{gaussian:bound:diff:deriv:hat:pm:same:time} and the space-time inequality \eqref{space:time:inequality}, we obtain
\begin{align*}
\Big|\int_{\mathbb{R}^d} \Delta_{\mu, \mu'}  \partial^2_x\widehat{p}_{m+1}(\mu, s, r, x, y) & [\Phi_{m+1}(\mu, s, r, t, y, z)  - \Phi_{m+1}(\mu, s, r, t, x, z) ] \, dy \Big|  \\
& \leq K_\alpha  \frac{W^{\beta}_2(\mu, \mu')}{(r-s)^{1+\frac{\beta-\alpha-\eta}{2}} (t-r)^{1+\frac{\alpha-\eta}{2}}} \, g(c(t-s), z-x)
\end{align*}
\noindent for any $\beta \in [\eta, 1]$ and any $\alpha \in [0,\eta)$. Note that if $W^2_2(\mu, \mu') \leq r-s$ the above estimate remains valid for any $\beta\in [0,1]$. Otherwise, if $W^2_2(\mu, \mu') \geq r-s$, from \eqref{holder:reg:voltera:kernel}, \eqref{standard} and the space-time inequality \eqref{space:time:inequality}, we directly get
 \begin{align*}
\Big| \int_{\mathbb{R}^d} \Delta_{\mu, \mu'}  \partial^2_x\widehat{p}_{m+1}(\mu, s, r, x, y) & [\Phi_{m+1}(\mu, s, r, t, y, z)  - \Phi_{m+1}(\mu, s, r, t, x_0, z) ] \, dy \Big|  \\
& \leq K_\alpha \frac{1}{(r-s)^{1-\frac{\alpha}{2}}(t-r)^{1+\frac{\alpha-\eta}{2}}} \, g(c(t-s), z-x) \\
& \leq K_\alpha  \frac{W^{\beta}_2(\mu, \mu')}{(r-s)^{1+\frac{\beta-\alpha}{2}} (t-r)^{1+\frac{\alpha-\eta}{2}}} \, g(c(t-s), z-x)
\end{align*}
\noindent for any $\beta \in [0,1]$ and any $\alpha \in [0,\eta)$. From \eqref{gaussian:bound:diff:second:deriv:p:hat:mu:holder:reg}, the space time inequality \eqref{space:time:inequality} and recalling that $r\in [s, \frac{t+s}{2}]$, we obtain
$$
\Big| \Phi_{m+1}(\mu, s, r, t, x, z) \int_{\mathbb{R}^d} \Delta_{\mu, \mu'}  [\partial^2_x (\widehat{p}^{y}_{m+1}-\widehat{p}^{x_0}_{m+1})(\mu, s, r, x, y)] \, dy \Big| \leq K \frac{W^{\beta}_2(\mu, \mu')}{(t-r)^{1-\frac{\eta}{2}}(r-s)^{1+\frac{\beta-\eta}{2}}} \, g(c(t-s), z-x).
$$

Hence, taking $\beta \in [0,\eta)$ and eventually $\alpha$ such that $\beta < \alpha < \eta$ in the above estimates, from the dominated convergence theorem, we derive that the map $x\mapsto \int_s^{\frac{t+s}{2}}\int_{\mathbb{R}^d} \Delta_{\mu, \mu'}  \widehat{p}_{m+1}(\mu, s, r, x, y) \Phi_{m+1}(\mu, s, r, t, y, z) \, dy \, dr$ is twice continuously differentiable and satisfies
\begin{align*}
\Big| \partial_x^2 \int_s^{\frac{t+s}{2}}\int_{\mathbb{R}^d}&  \Delta_{\mu, \mu'}  \widehat{p}_{m+1}(\mu, s, r, x, y) \Phi_{m+1}(\mu, s, r, t, y, z) \, dy \, dr \Big| \\
& \leq K_\beta W^{\beta}_2(\mu, \mu') \int_s^{\frac{t+s}{2}} \left\{ \frac{1}{(r-s)^{1+\frac{\beta-\alpha}{2}} (t-r)^{1+\frac{\alpha-\eta}{2}}} + \frac{1}{(t-r)^{1-\frac{\eta}{2}}(r-s)^{1+\frac{\beta-\eta}{2}}}\right\} \, dr \, g(c(t-s), z-x) \\
& \leq K_\beta \frac{W^{\beta}_2(\mu, \mu')}{(t-s)^{1+\frac{\beta-\eta}{2}}} \, g(c(t-s), z-x)
\end{align*}
\noindent for any $\beta \in [0,\eta)$.

Now, if $r\in [\frac{t+s}{2}, t]$ since $r-s \geq (t-s)/2$, the kernel $ \Delta_{\mu, \mu'}  \partial^2_x\widehat{p}_{m+1}(\mu, s, r, x, y)$ does not generate any time singularity. Hence, if $W^2_2(\mu, \mu') \geq r-s$, from \eqref{Gaussian:estimate:Phim} and \eqref{standard}, one directly gets
\begin{align*}
 \Big| \int_{\mathbb{R}^d}&  \Delta_{\mu, \mu'}   \partial_x^2\widehat{p}_{m+1}(\mu, s, r, x, y) \Phi_{m+1}(\mu, s, r, t, y, z) \, dy \Big|  \\
 & \leq  \int_{\mathbb{R}^d} ( |\partial_x^2\widehat{p}_{m+1}(\mu, s, r, x, y)| + |\partial_x^2\widehat{p}_{m+1}(\mu', s, r, x, y)|) |\Phi_{m+1}(\mu, s, r, t, y, z)| \, dy\\
 & \leq K \frac{1}{(r-s)(t-r)^{1-\frac{\eta}{2}}} \, g(c(t-s), z-x) \\
 & \leq  K \frac{W^{\beta}_2(\mu, \mu')}{(t-s)^{1+ \frac{\beta}{2}}(t-r)^{1-\frac{\eta}{2}}} \, g(c(t-s), z-x)
\end{align*}
\noindent for any $\beta \in [0,1]$. Otherwise, if $W^2_2(\mu, \mu') < r-s$, from \eqref{gaussian:bound:diff:deriv:hat:pm:same:time} and \eqref{Gaussian:estimate:Phim}, we obtain
\begin{align*}
 \Big| \int_{\mathbb{R}^d}&  \Delta_{\mu, \mu'}   \partial_x^2\widehat{p}_{m+1}(\mu, s, r, x, y) \Phi_{m+1}(\mu, s, r, t, y, z) \, dy \Big|  \\
 & \leq  K  \frac{W^{\beta}_2(\mu, \mu')}{(r-s)^{1+\frac{\beta-\eta}{2}}(t-r)^{1-\frac{\eta}{2}}} \, g(c(t-s), z-x)
\end{align*}
\noindent for any $\beta \in [\eta,1]$. Observe again that since $W^2_2(\mu, \mu') < r-s$, the above estimate is actually valid for any $\beta \in [0,1]$.
From the above discussion and the dominated convergence, we deduce that the map $x\mapsto \int_{\frac{t+s}{2}}^{t} \int_{\mathbb{R}^d}  \Delta_{\mu, \mu'}  \widehat{p}_{m+1}(\mu, s, r, x, y) \Phi_{m+1}(\mu, s, r, t, y, z) \, dy \, dr$ is twice continuously differentiable and satisfies
\begin{align*}
\Big| \partial_x^2&  \int_{\frac{t+s}{2}}^{t} \int_{\mathbb{R}^d}  \Delta_{\mu, \mu'}  \widehat{p}_{m+1}(\mu, s, r, x, y) \Phi_{m+1}(\mu, s, r, t, y, z) \, dy \, dr \Big| \\
& = \Big|   \int_{\frac{t+s}{2}}^{t} \int_{\mathbb{R}^d}  \Delta_{\mu, \mu'} \partial_x^2 \widehat{p}_{m+1}(\mu, s, r, x, y) \Phi_{m+1}(\mu, s, r, t, y, z) \, dy \, dr \Big| \\
&  \leq  K \frac{W^{\beta}_2(\mu, \mu')}{(t-s)^{1+\frac{\beta-\eta}{2}}} \, g(c(t-s), z-x)
\end{align*}
\noindent for any $\beta \in [0,1]$. We finally conclude that $x\mapsto (\Delta_{\mu, \mu'}  \widehat{p}_{m+1}\otimes\Phi_{m+1})(\mu, s, t, x, z)$ is twice continuously differentiable and satisfies 
$$
|  \partial^2_x(\Delta_{\mu, \mu'}  \widehat{p}_{m+1}\otimes\Phi_{m+1})(\mu, s, t, x, z)| \leq  K_\beta \frac{W^{\beta}_2(\mu, \mu')}{(t-s)^{1+\frac{\beta-\eta}{2}}} \, g(c(t-s), z-x)
$$
\noindent for any $\beta \in [0,\eta)$.

In order to handle the last term appearing in the right-hand of \eqref{representation:diff:mes:pm:p1:bis} for $n=2$, we first notice that
\begin{align*}
\partial^2_x \int_{\mathbb{R}^d} & p_m(\mu', s, r, x, y) \Delta_{\mu, \mu'} \mH_m(\mu, s, r , t, y ,z) \, dy  \\
& =  \int_{\mathbb{R}^d} \partial^2_x p_m(\mu', s, r, x, y) \Delta_{\mu, \mu'} [\mH_m(\mu, s, r , t, y ,z) -  \mH_m(\mu, s, r , t, x_0 ,z)] \, dy  \\
& \quad + \Delta_{\mu, \mu'} \mH_m(\mu, s, r , t, x_0 ,z) \partial^2_x \int_{\mathbb{R}^d}  p_m(\mu', s, r, x, y)  \, dy \\
& =  \int_{\mathbb{R}^d} \partial^2_x p_m(\mu', s, r, x, y) \Delta_{\mu, \mu'} [\mH_m(\mu, s, r , t, y ,z) -  \mH_m(\mu, s, r , t, x_0 ,z)] \, dy  
\end{align*}
\noindent for any $x_0 \in \mathbb{R}^d$. We now set $x_0=x$ and split our computations according to the two cases $W^2_2(\mu, \mu') \leq r-s$ and $W^2_2(\mu, \mu') \geq r-s$. In the first case, we use \eqref{gaussian:bound:diff:Hm:mu:holder:reg} which is here valid for any $\beta \in [0,1]$ while in the second case, we use \eqref{holder:reg:parametrix:kernel} with $\beta= \beta' \in [0,\eta]$. Gathering the two cases and using \eqref{bound:derivative:heat:kernel} as well as the space-time inequality \eqref{space:time:inequality}, for any $\beta \in [0,\eta)$, any $\beta' \in [0,\eta]$ and any $r\in [s, \frac{t+s}{2}]$, we obtain
\begin{align*}
\Big| \int_{\mathbb{R}^d} &  \partial^2_x p_m(\mu, s, r, x, y) \Delta_{\mu, \mu'} [\mH_m(\mu, s, r , t, y ,z) -  \mH_m(\mu, s, r , t, x ,z)] \, dy \Big| \\
& \leq  K_{\beta'} W^{\beta}_2(\mu, \mu') \Big( \frac{1}{(r-s)^{1-\frac{(\eta-\alpha)}{2}}} \frac{1}{(t-r)^{1+\frac{\eta-\alpha}{2}}(r-s)^{\frac{\beta-\eta}{2}}} + \frac{1}{(r-s)^{1-\frac{\beta'}{2}}}\frac{1}{(t-r)^{1+\frac{\beta'-\eta}{2}}(r-s)^{\frac{\beta}{2}}}\Big)  \\
& \quad \times   g(c(t-s), z-x) \\
& \leq  K_{\beta'} W^{\beta}_2(\mu, \mu') \Big( \frac{1}{(r-s)^{1+\frac{\alpha+\beta}{2}-\eta}} \frac{1}{(t-s)^{1+\frac{\eta-\alpha}{2}}} + \frac{1}{(r-s)^{1+\frac{\beta-\beta'}{2}}}\frac{1}{(t-s)^{1+\frac{\beta'-\eta}{2}}}\Big)  \, g(c(t-s), z-x)
\end{align*}

\noindent and, selecting $\alpha \in [0,\eta)$ and $\beta' \in (\beta, \eta)$ so that $\beta < \eta < 2 \eta-\alpha $ and $1+(\beta-\beta')/2< 1$, by the dominated convergence theorem, we deduce that $x\mapsto \int_s^{\frac{t+s}{2}}\int_{\mathbb{R}^d} p_m(\mu, s, r, x, y) \Delta_{\mu, \mu'} \mH_m(\mu, s, r , t, y ,z) \, dy \, dr $ is twice continuously differentiable and satisfies
\begin{align*}
\Big| & \int_s^{\frac{t+s}{2}}\int_{\mathbb{R}^d}  p_m(\mu, s, r, x, y) \Delta_{\mu, \mu'} \mH_m(\mu, s, r , t, y ,z) \, dy \, dr \Big| \\
 &  \quad \leq  K_\beta W^{\beta}_2(\mu, \mu') \int_s^{\frac{t+s}{2}} \Big( \frac{1}{(r-s)^{1-\frac{(\eta-\alpha)}{2}}} \frac{1}{(t-r)^{1+\frac{\eta-\alpha}{2}}(r-s)^{\frac{\beta-\eta}{2}}} + \frac{1}{(r-s)^{1-\frac{\beta'}{2}}}\frac{1}{(t-r)^{1+\frac{\beta'-\eta}{2}}(r-s)^{\frac{\beta}{2}}}\Big)  \, dr \\
 &\quad \quad \times   g(c(t-s), z-x) \\
& \quad\leq  K_\beta \frac{W^{\beta}_2(\mu, \mu')}{(t-s)^{1+ \frac{\beta-\eta}{2}}} \, g(c(t-s), z-x).
\end{align*}

From \eqref{bound:derivative:heat:kernel}, \eqref{iter:parametrix:kernel} (with $k=1$) if $W_2^2(\mu, \mu') \geq r-s$ or \eqref{gaussian:bound:diff:Hm:mu} which is valid for any $\beta \in [0,1]$ if $W_2^2(\mu, \mu') < r-s$ and the dominated convergence theorem, we similarly deduce that the map $x\mapsto \int_{\frac{t+s}{2}}^{t}\int_{\mathbb{R}^d} p_m(\mu, s, r, x, y) \Delta_{\mu, \mu'} \mH_m(\mu, s, r , t, y ,z) \, dy \, dr $ is twice continuously differentiable with
\begin{align*}
\Big| \int_{\frac{t+s}{2}}^{t}& \int_{\mathbb{R}^d} \partial^2_x p_m(\mu, s, r, x, y) \Delta_{\mu, \mu'} \mH_m(\mu, s, r , t, y ,z) \, dy \, dr \Big| \\
&  \leq  K W^{\beta}_2(\mu, \mu') \int_{\frac{t+s}{2}}^{t} \frac{1}{r-s} \frac{1}{(t-r)^{1-\frac{\eta}{2}}(r-s)^{\frac{\beta}{2}}}\, dr \, g(c(t-s), z-x)\\
& \leq  K \frac{W^{\beta}_2(\mu, \mu')}{(t-s)^{1+ \frac{\beta-\eta}{2}}} \, g(c(t-s), z-x).
\end{align*}

We thus conclude that $x\mapsto  (p_{m+1} \otimes \Delta_{\mu, \mu'} \mH_{m+1})(\mu, s, t, x, z)$ is twice continuously differentiable and satisfies
$$
 |(\partial^2_x p_{m+1} \otimes \Delta_{\mu, \mu'} \mH_{m+1})(\mu, s, t, x, z)| \leq   K_\beta  \frac{W^{\beta}_2(\mu, \mu')}{(t-s)^{1+ \frac{\beta-\eta}{2}}} g(c(t-s), z-x)
$$

\noindent which in turn, using \eqref{Gaussian:estimate:Phim}, yields 
\begin{align*}
|(\partial^2_x p_{m+1} \otimes \Delta_{\mu, \mu'} \mH_{m+1}) \otimes \Phi_{m+1}(\mu, s, t, x, z)| & \leq  K_\beta \frac{W^{\beta}_2(\mu, \mu')}{(t-s)^{1+ \frac{\beta}{2}-\eta}} g(c(t-s), z-x)
\end{align*}
\noindent for any $\beta\in [0, \eta)$. The proof of \eqref{gaussian:estimate:diff:mes:deriv:heat:kernel:part2:statement:proposition} is now complete.

\end{proof}

Similarly, in order to tackle the estimates \eqref{regularity:time:estimate:v1:v2:v3:decoupling:mckean}, we first need the following auxiliary result whose proof is postponed to Section \ref{section:proof:lem:technical:estimate:reg:time:metric:coefficients:densities} of the appendix.

\begin{lem}\label{lem:technical:estimate:reg:time:metric:coefficients:densities} For any $\beta \in [0,1]$ and any $\alpha \in [0,\eta]$, there exist positive constants $K$, $K_{\alpha}$ and $c$ such that for any positive integer $m$, any $t\in (0,T]$, any $(s_1, s_2, x, y, y', z) \in [0,t)^2 \times (\mathbb{R}^d)^4$, any $r\in (s_1 \vee s_2, t)$ and any $(i, j) \in \left\{1, \cdots, d\right\}^2$
\begin{align}
& |\Delta_{s_1,s_2} p_{m}(\mu, s, t, x, z) | \nonumber \\
&  \leq K \left\{ \frac{|s_1-s_2|^{\beta}}{(t-s_1)^{\beta}} g(c(t-s_1), z-x) +  \frac{|s_1-s_2|^{\beta}}{(t-s_2)^{\beta}} g(c(t-s_2), z-x)  \right\}, \label{diff:time:heat:kernel}
\end{align}
\begin{align}
|a_{i, j}(t, x, [X^{s_1 \vee s_2, \xi, (m)}_t])  - & a_{i, j}(t, x, [X^{s_1 \wedge s_2, \xi, (m)}_t])| + |b_{i}(t, x, [X^{s_1 \vee s_2, \xi, (m)}_t])  - b_{i}(t, x, [X^{s_1 \wedge s_2, \xi', (m)}_t])| \notag \\
& \leq K \left\{ \frac{|s_1-s_2|^\beta}{(t-s_1)^{\beta-\frac{\eta}{2}}} +\frac{|s_1-s_2|^\beta}{(t-s_2)^{\beta-\frac{\eta}{2}}}  \right\}, \label{diff:time:drift:diff:coefficients}
\end{align}
\begin{align}
 | \Delta_{s_1, s_2}&  \partial^{n}_x \widehat{p}_{m+1}(\mu, s, t, x, z) |  \notag \\
 & \leq K \left\{ \frac{|s_1-s_2|^\beta}{(t-s_1)^{\frac{n}{2}+\beta}} \, g(c(t-s_1), z-x) + \frac{|s_1-s_2|^\beta}{(t-s_2)^{\frac{n}{2}+\beta}} \, g(c(t-s_2), z-x) \right\}, \quad n\in \left\{0,1, 2\right\}, \label{gaussian:bound:diff:time:hat:pm:same:time}
\end{align}
\begin{align}
|a_{i, j}& (t, x, [X^{s_1 \vee s_2, \xi, (m)}_t])  -  a_{i, j}(t, x, [X^{s_1 \wedge s_2, \xi, (m)}_t]) - (a_{i, j}(t, y, [X^{s_1 \vee s_2, \xi, (m)}_t])  - a_{i, j}(t, y, [X^{s_1 \wedge s_2, \xi, (m)}_t]))| \notag \\
\quad + |b_{i}& (t, x, [X^{s_1 \vee s_2, \xi, (m)}_t])  -  b_{i}(t, x, [X^{s_1 \wedge s_2, \xi, (m)}_t]) - (b_{i}(t, y, [X^{s_1 \vee s_2, \xi, (m)}_t])  - b_{i}(t, y, [X^{s_1 \wedge s_2, \xi, (m)}_t]))| \notag \\
& \quad \quad \leq   K_{\alpha} (|y-x|^\alpha \wedge 1) |s_1-s_2|^\beta \left\{\frac{1}{(t-s_1)^{\beta+\frac{\alpha-\eta}{2}}} + \frac{1}{(t-s_2)^{\beta+\frac{\alpha-\eta}{2}}}\right\},  \label{diff:time:with:holder:reg:space:drift:diff:coefficients}
\end{align} 
\begin{align}
 | \Delta_{s_1, s_2} & [\partial^n_{x}  [\widehat{p}^{y}_{m+1} - \widehat{p}^{y'}_{m+1}](\mu, s ,t , x ,z )]  |  \notag  \\
 & \leq  K (|y-y'|^\eta \wedge 1)  \left\{ \frac{|s_1-s_2|^\beta}{(t-s_1)^{\frac{n}{2}+\beta}} \, g(c(t-s_1), z-x) + \frac{|s_1-s_2|^\beta}{(t-s_2)^{\frac{n}{2}+\beta}} \, g(c(t-s_2), z-x) \right\}, \, \beta \in [0,1), \, n\in \left\{1, 2\right\},\label{gaussian:bound:diff:time:second:deriv:p:hat:mu:holder:reg}
\end{align}

\begin{align}
 | \Delta_{s_1, s_2} \partial^{n}_x \widehat{p}_{m+1}(\mu, s, r, t, x, z) | \leq K \frac{|s_1-s_2|^\beta}{(t-r)^{\frac{n}{2}}(r-s_1\vee s_2)^{\beta}} \, g(c(t-r), z-x), \quad n\in \left\{0,1, 2, 3\right\}, \label{gaussian:bound:diff:time:hat:pm:different:time}
\end{align}
\begin{align}
 |  \Delta_{s_1, s_2} \mH_{m+1}(\mu, s, r ,t , x ,z) |  \leq K  \frac{|s_1-s_2|^{\beta}}{(t-r)^{1-\frac{\eta}{2}}(r-s_1\vee s_2)^{\beta}}  \, g(c(t-r), z-x), \label{gaussian:bound:diff:time:Hm}
\end{align}
\begin{align}
 | &  \Delta_{s_1, s_2} [\mH_{m+1}(\mu, s, r ,t , x ,z) - \mH_{m+1}(\mu, s, r ,t , y ,z)] |  \notag \\
 & \quad \leq K (|y-x|^\alpha \wedge 1) \frac{|s_1-s_2|^{\beta}}{(t-r)^{1+\frac{\alpha- \eta}{2}}(r-s_1\vee s_2)^{\beta}}  \, \left\{ g(c(t-r), z-x) + g(c(t-r), z-y)\right\}, \label{gaussian:bound:diff:time:plus:holder:regularity:Hm}
\end{align}
\begin{align}
 | \Delta_{s_1, s_2} \Phi_{m+1}(\mu, s, r ,t , x ,z) |  \leq K  \frac{|s_1-s_2|^{\beta}}{(t-r)^{1-\frac{\eta}{2}}(r-s_1\vee s_2)^{\beta}}  \, g(c(t-r), z-x) \label{gaussian:bound:diff:time:Phim}
\end{align}
\noindent and
\begin{align}
 | & \Delta_{s_1, s_2} [\Phi_{m+1}(\mu, s, r ,t , x ,z) - \Phi_{m+1}(\mu, s, r ,t , y ,z)] | \notag \\
 & \quad  \leq K_{\alpha} (|y-x|^\alpha \wedge 1) \frac{|s_1-s_2|^{\beta}}{(t-r)^{1+\frac{\alpha- \eta}{2}}(r-s_1\vee s_2)^{\beta}}  \, \left\{ g(c(t-r), z-x) + g(c(t-r), z-y)\right\}, \, \alpha \in [0,\eta). \label{gaussian:bound:diff:time:plus:holder:regularity:Phim}
\end{align}
\end{lem}

Having the above technical result at hand, we are now ready to establish the estimates \eqref{regularity:time:estimate:v1:v2:v3:decoupling:mckean}.

\begin{prop}\label{prop:representation:and:gaussian:estimate:diff:time:deriv:space:heat:kernel}
Let $n\in \left\{0, 1, 2\right\}$. The following Gaussian estimates are satisfied: for any $\beta \in [0,1]$ if $n=0$ or any $\beta \in [0,  (1+\eta)/2)$ if $n=1$ or any $\beta \in [0, \eta/2)$ if $n=2$, there exist some positive constants $K_\beta$ and $c$ such that for any positive integer $m$, any $(\mu, t, x, z) \in \pp \times (0,T] \times (\mathbb{R}^d)^2$ and any $(s_1, s_2) \in [0,t)^2$
\begin{align}
 | & \partial^{n}_x p_{m}(\mu, s_1, t, x, z) - \partial^{n}_x p_{m}(\mu, s_2, t, x, z)  | \nonumber \\
 & \leq K_\beta \left\{ \frac{|s_1-s_2|^{\beta}}{(t-s_1)^{\frac{n}{2} + \beta }} \, g(c(t-s_1), z-x) + \frac{|s_1-s_2|^{\beta}}{(t- s_2)^{ \frac{n}{2} +\beta }} \, g(c(t-s_2), z-x) \right\}. \label{regularity:time:estimate:v1:v2:v3:decoupling:mckean:prop:statement} 
\end{align}
\end{prop}

\begin{proof}
The estimate \eqref{regularity:time:estimate:v1:v2:v3:decoupling:mckean:prop:statement} for $n=0$ corresponds exactly to \eqref{diff:time:heat:kernel}. We now deal with the two remaining cases $n=1$ and $n=2$. Let $\beta \in [0, (1+\eta/2)$ if $n=1$ or $\beta \in [0, \eta/2)$ if $n=2$. We first remark that if $|s_1-s_2| \geq t-s_1\vee s_2$, both estimates follow from \eqref{bound:derivative:heat:kernel}. From now on we assume that $|s_1-s_2| \leq t-s_1 \vee s_2$. We recall that the relation \eqref{other:representation:parametrix:series} is $n$-times continuously differentiable with respect to the variable $x$, for $n \in \left\{1, 2\right\}$ and
 $$
 \partial^{n}_x p_{m+1}(\mu, s, t, x, z) = \partial^{n}_x \widehat{p}_{m+1}(\mu, s, t, x, z) + \int_s^t \int_{\mathbb{R}^d} \partial^{n}_x \widehat{p}_{m+1}(\mu, s, t, x, y) \Phi_{m+1}(\mu, s, r, t, y, z) \, dy \, dr, 
 $$
 \noindent which also satisfies the relation \eqref{deriv:space:parametrix:series}. The previous relation directly gives
 \begin{align*}
 \Delta_{s_1, s_2} &  \partial^n_x p_{m+1}(\mu, s, t, x, z) \\
 & = \Delta_{s_1, s_2} \partial^n_x \widehat{p}_{m+1}(\mu, s, t, x, z)  \\
 & \quad + \int_{s_1\vee s_2}^t \int_{\mathbb{R}^d} \Delta_{s_1, s_2} \partial^n_x \widehat{p}_{m+1}(\mu, s, r, x, y) \Phi_{m+1}(\mu, s_1\vee s_2, r, t, y, z)  \, dy\, dr \\
 & \quad + \int_{s_1\vee s_2}^t \int_{\mathbb{R}^d} \partial^n_x \widehat{p}_{m+1}(\mu, s_1\wedge s_2, r, x, y) \,  \Delta_{s_1, s_2} \Phi_{m+1}(\mu, s, r, t, y, z) \, dy \, dr \\
 & \quad - \int_{s_1\wedge s_2}^{s_1\vee s_2} \int_{\mathbb{R}^d}  \partial^n_x \widehat{p}_{m+1}(\mu, s_1\wedge s_2, r, x, y) \,  \Phi_{m+1}(\mu, s_1\wedge s_2, r, t, y, z)  \, dy\, dr \\
 & =: {\rm I} + {\rm II} + {\rm III} + {\rm IV}.
 \end{align*}

We now investigate each term of the above decomposition. From \eqref{gaussian:bound:diff:time:hat:pm:same:time}, one directly gets
$$
|{\rm I}| \leq K \left\{ \frac{|s_1-s_2|^{\beta}}{(t-s_1)^{\frac{n}{2} + \beta }} \, g(c(t-s_1), z-x) + \frac{|s_1-s_2|^{\beta}}{(t- s_2)^{ \frac{n}{2} +\beta }} \, g(c(t-s_2), z-x) \right\}.
$$

In order to deal with ${\rm II}$, we separate the time integral into the two disjoint intervals $[s_1\vee s_2, (t+s_1\vee s_2)/2]$ and $((t+s_1\vee s_2)/2, t]$. If $r \in [s_1\vee s_2,  (t+s_1\vee s_2)/2]$, we balance the time singularity generated by the $n$th-derivatives of $\widehat{p}_{m+1}$ by writing 
\begin{align*}
\int_{\mathbb{R}^d} \Delta_{s_1, s_2} & \partial^n_x \widehat{p}_{m+1}(\mu, s, r, x, y) \Phi_{m+1}(\mu, s_1\vee s_2, r, t, y, z)  \, dy \\
& = \int_{\mathbb{R}^d} \Delta_{s_1, s_2} \partial^n_x \widehat{p}_{m+1}(\mu, s, r, x, y) [ \Phi_{m+1}(\mu, s_1\vee s_2, r, t, y, z) - \Phi_{m+1}(\mu, s_1\vee s_2, r, t, x, z) ] \, dy \\
& \quad + \Phi_{m+1}(\mu, s_1\vee s_2, r, t, x, z) \int_{\mathbb{R}^d} \Delta_{s_1, s_2} [\partial^n_x \widehat{p}^y_{m+1}(\mu, s, r, x, y) - \partial^n_x \widehat{p}^x_{m+1}(\mu, s, r, x, y)]  \, dy. 
\end{align*}

Now, from \eqref{gaussian:bound:diff:time:hat:pm:same:time}, \eqref{holder:reg:voltera:kernel} and the space-time inequality \eqref{space:time:inequality}, we obtain
\begin{align}
\Big| \int_{\mathbb{R}^d} \Delta_{s_1, s_2} & \partial^n_x \widehat{p}_{m+1}(\mu, s, r, x, y) [ \Phi_{m+1}(\mu, s_1\vee s_2, r, t, y, z) - \Phi_{m+1}(\mu, s_1\vee s_2, r, t, x, z) ] \, dy \Big| \notag \\
& \leq K_\alpha \frac{|s_1-s_2|^\beta}{(t-r)^{1+\frac{\alpha-\eta}{2}}} \, \left\{ \frac{1}{(r-s_1)^{\frac{n}{2}+\beta- \frac{\alpha}{2}}} g(c(t-s_1), z-x) + \frac{1}{(r- s_2)^{\frac{n}{2}+\beta- \frac{\alpha}{2}}} g(c(t-s_2), z-x) \right\} \label{first:bound:term:II:diff:time}
\end{align}
\noindent for any $\alpha \in [0,\eta)$. Moreover, from \eqref{gaussian:bound:diff:time:second:deriv:p:hat:mu:holder:reg}, \eqref{Gaussian:estimate:Phim} and again the space-time inequality \eqref{space:time:inequality}, if $r \in [s_1\vee s_2, (t+s_1\vee s_2)/2]$ we get
\begin{align*}
|  & \Phi_{m+1}(\mu, s_1\vee s_2, r, t, x, z)  \int_{\mathbb{R}^d} \Delta_{s_1, s_2} [\partial^n_x \widehat{p}^y_{m+1}(\mu, s, r, x, y) - \partial^n_x \widehat{p}^x_{m+1}(\mu, s, r, x, y)]  \, dy  | \\
& \quad  \leq K \frac{|s_1-s_2|^\beta}{(t-r)^{1-\frac{\eta}{2}}} \, \left\{ \frac{1}{(r-s_1)^{\frac{n}{2}+\beta- \frac{\eta}{2}}}  + \frac{1}{(r- s_2)^{\frac{n}{2}+\beta- \frac{\eta}{2}}} \right\}g(c(t-s_1\vee s_2), z-x).
\end{align*}

We now select $\alpha\in [0,\eta)$ in \eqref{first:bound:term:II:diff:time} such that $\frac{n}{2}+\beta - \frac{\alpha}{2} < 1$. Hence, the two previous estimates give
\begin{align*}
& \int_{s_1\vee s_2}^{\frac{t+s_1\vee s_2}{2}} \int_{\mathbb{R}^d} |\Delta_{s_1, s_2}  \partial^n_x \widehat{p}_{m+1}(\mu, s, r, x, y)|  |\Phi_{m+1}(\mu, s_1\vee s_2, r, t, y, z)|  \, dy \, dr\\
& \leq  K_\beta  \int_{s_1\vee s_2}^{\frac{t+s_1\vee s_2}{2}} \left\{ \frac{|s_1-s_2|^\beta}{(t-r)^{1+\frac{\alpha-\eta}{2}}(r-s_1)^{\frac{n}{2}+\beta- \frac{\alpha}{2}}} g(c(t-s_1), z-x) + \frac{|s_1-s_2|^\beta}{(t-r)^{1+\frac{\alpha-\eta}{2}}(r- s_2)^{\frac{n}{2}+\beta- \frac{\alpha}{2}}} g(c(t- s_2), z-x) \right\} \, dr\\
& \quad + K_\beta  \int_{s_1\vee s_2}^{\frac{t+s_1\vee s_2}{2}} \frac{|s_1-s_2|^\beta}{(t-r)^{1-\frac{\eta}{2}}} \, \left\{ \frac{1}{(r-s_1)^{\frac{n}{2}+\beta - \frac{\eta}{2}}}  + \frac{1}{(r- s_2)^{\frac{n}{2}+\beta- \frac{\eta}{2}}} \right\}\, dr \, g(c(t-s_1\vee s_2), z-x)  \\
& \leq K_\beta \left\{ \frac{|s_1-s_2|^\beta}{(t-s_1)^{\frac{n}{2}+\beta-\frac{\eta}{2}} } g(c(t-s_1), z-x) + \frac{|s_1-s_2|^\beta}{(t-s_2)^{\frac{n}{2}+\beta - \frac{\eta}{2}}} g(c(t- s_2), z-x) \right\}. 
\end{align*}

Otherwise, if $r \in ((t+s_1\vee s_2)/2, t]$, the kernel $ \Delta_{s_1, s_2} \partial^n_x \widehat{p}_{m+1}(\mu, s, r, x, y)$ does not generate any time singularity. Indeed, from \eqref{gaussian:bound:diff:time:hat:pm:same:time} and  \eqref{Gaussian:estimate:Phim}, we directly derive
\begin{align*}
& \int_{\frac{t+s_1\vee s_2}{2} }^{t} \int_{\mathbb{R}^d} | \Delta_{s_1, s_2}  \partial^n_x \widehat{p}_{m+1}(\mu, s, r, x, y)|  |\Phi_{m+1}(\mu, s_1\vee s_2, r, t, y, z)| \, dy \, dr\\
& \leq K \int_{\frac{t+s_1\vee s_2}{2}}^{t} \left\{ \frac{|s_1-s_2|^\beta }{(r-s_1\vee s_2)^{\frac{n}{2}+\beta}(t-r)^{1-\frac{\eta}{2}}} \, g(c(t-s_1\vee s_2), z-x) \right. \\
& \quad \left. +\frac{|s_1-s_2|^\beta }{(r-s_1\wedge s_2)^{\frac{n}{2}+\beta}(t-r)^{1-\frac{\eta}{2}}} \, g(c(t-s_1\wedge s_2), z-x)  \right\} \, dr \\
& \leq K \left\{ \frac{|s_1-s_2|^\beta}{(t-s_1\vee s_2)^{\frac{n}{2}+\beta - \frac{\eta}{2}}} \,  g(c(t-s_1\vee s_2), z-x) +  \frac{|s_1-s_2|^\beta}{(t-s_1\wedge s_2)^{\frac{n}{2}+\beta - \frac{\eta}{2}}} \,  g(c(t-s_1\wedge s_2), z-x) \right\}.
\end{align*}

Gathering the two previous estimates yields
$$
| {\rm II} | \leq  K_\beta \left\{ \frac{|s_1-s_2|^\beta}{(t-s_1\vee s_2)^{\frac{n}{2}+\beta - \frac{\eta}{2}}} \,  g(c(t-s_1\vee s_2), z-x) +  \frac{|s_1-s_2|^\beta}{(t-s_1\wedge s_2)^{\frac{n}{2}+\beta - \frac{\eta}{2}}} \,  g(c(t-s_1\wedge s_2), z-x) \right\}.
$$

We handle the third term ${\rm III}$ in a similar manner. Namely, if $r \in [s_1\vee s_2, (t+s_1\vee s_2)/2]$, we first write
\begin{align*}
\int_{\mathbb{R}^d} & \partial^n_x \widehat{p}_{m+1}(\mu, s_1\wedge s_2, r, x, y) \Delta_{s_1, s_2} \Phi_{m+1}(\mu, s, r, t, y, z)  \, dy \\
& = \int_{\mathbb{R}^d} \partial^n_x \widehat{p}_{m+1}(\mu, s_1\wedge s_2, r, x, y) \Delta_{s_1, s_2} [ \Phi_{m+1}(\mu, s, r, t, y, z) - \Phi_{m+1}(\mu, s, r, t, x, z) ] \, dy \\
& \quad + \Delta_{s_1, s_2}  \Phi_{m+1}(\mu, s, r, t, x, z) \int_{\mathbb{R}^d} [\partial^n_x \widehat{p}^y_{m+1}(\mu, s_1\wedge s_2, r, x, y) - \partial^n_x \widehat{p}^x_{m+1}(\mu, s_1 \wedge s_2, r, x, y)]  \, dy. 
\end{align*}

Now, from \eqref{gaussian:bound:diff:time:plus:holder:regularity:Phim} and the space-time inequality \eqref{space:time:inequality}, if $|s_1-s_2| \leq r-s_1\vee s_2$, we get
\begin{align*}
 \int_{\mathbb{R}^d}&  |\partial^n_x \widehat{p}_{m+1}(\mu, s_1\wedge s_2, r, x, y)| |\Delta_{s_1, s_2} [ \Phi_{m+1}(\mu, s, r, t, y, z) - \Phi_{m+1}(\mu, s, r, t, x, z) ]| \, dy \\
 & \leq K_{\alpha} \frac{|s_1-s_2|^\beta}{(t-r)^{1+\frac{\alpha-\eta}{2}}(r-s_1\vee s_2)^{\frac{n}{2}+\beta}} (r-s_1\wedge s_2)^{\frac{\alpha}{2}} \, \left\{ g(c(t-s_1\wedge s_2), z-x) + g(c(t-s_1\vee s_2), z-x) \right\} \\
 & \leq K_{\alpha} \frac{|s_1-s_2|^\beta}{(t-r)^{1+\frac{\alpha-\eta}{2}}(r-s_1\vee s_2)^{\frac{n}{2}+\beta-\frac{\alpha}{2}}} \, \left\{ g(c(t-s_1\wedge s_2), z-x) + g(c(t-s_1\vee s_2), z-x) \right\}
\end{align*}
\noindent where we used the inequality $r-s_1\wedge s_2 \leq 2 (r-s_1\vee s_2)$ for the last inequality. Otherwise, if $|s_1-s_2| \geq r-s_1\vee s_2$, we use \eqref{holder:reg:voltera:kernel} and the space-time inequality \eqref{space:time:inequality} so that
\begin{align*}
 \int_{\mathbb{R}^d}&  |\partial^n_x \widehat{p}_{m+1}(\mu, s_1\wedge s_2, r, x, y)| |\Delta_{s_1, s_2} [ \Phi_{m+1}(\mu, s, r, t, y, z) - \Phi_{m+1}(\mu, s, r, t, x, z) ]| \, dy \\
 & \leq K_{\alpha}  \frac{1}{(t-r)^{1+\frac{\alpha-\eta}{2}}(r-s_1\vee s_2)^{\frac{n}{2}-\frac{\alpha}{2}}} \, \left\{ g(c(t-s_1\wedge s_2), z-x) + g(c(t-s_1\vee s_2), z-x) \right\} \\
 & \leq K_{\alpha}  \frac{|s_1-s_2|^\beta}{(t-r)^{1+\frac{\alpha-\eta}{2}}(r-s_1\vee s_2)^{\frac{n}{2}+\beta-\frac{\alpha}{2}}} \, \left\{ g(c(t-s_1\wedge s_2), z-x) + g(c(t-s_1\vee s_2), z-x) \right\}.
\end{align*}

From \eqref{gaussian:bound:diff:time:Phim} and the inequality $|(\partial_x^{n} \widehat{p}^{y}_{m+1} - \partial_x^{n} \widehat{p}^{y'}_{m+1})(\mu, s, t, x, z)| \leq K (|y-y'|^\eta \wedge 1) (t-s)^{-\frac{n}{2}} \, g(c(t-s), z-x)$ (which follows from the mean value theorem and the uniform $\eta$-H\"older regularity of $a(t, ., m)$) we obtain
\begin{align*}
\Big| \Delta_{s_1, s_2}  & \Phi_{m+1}(\mu, s, r, t, x, z) \int_{\mathbb{R}^d} [\partial^n_x \widehat{p}^y_{m+1}(\mu, s_1\wedge s_2, r, x, y) - \partial^n_x \widehat{p}^x_{m+1}(\mu, s_1 \wedge s_2, r, x, y)]  \, dy \Big| \\
& \leq  K  \frac{|s_1-s_2|^\beta}{(t-r)^{1-\frac{\eta}{2}}(r-s_1\vee s_2)^{\frac{n}{2}+\beta-\frac{\eta}{2}}} \, g(c(t-s_1\vee s_2), z-x).
\end{align*}

Gathering the above estimates, we eventually get
\begin{align*}
\Big| \int_{\mathbb{R}^d} & \partial^n_x \widehat{p}_{m+1}(\mu, s_1\wedge s_2, r, x, y) \Delta_{s_1, s_2} \Phi_{m+1}(\mu, s, r, t, y, z)  \, dy \Big| \\
& \leq K_{\alpha}  \frac{|s_1-s_2|^\beta}{(t-r)^{1+\frac{\alpha-\eta}{2}}(r-s_1\vee s_2)^{\frac{n}{2}+\beta-\frac{\alpha}{2}}} \, \left\{ g(c(t-s_1\wedge s_2), z-x) + g(c(t-s_1\vee s_2), z-x) \right\} \\
& \quad + K  \frac{|s_1-s_2|^\beta}{(t-r)^{1-\frac{\eta}{2}}(r-s_1\vee s_2)^{\frac{n}{2}+\beta-\frac{\eta}{2}}} \, \left\{ g(c(t-s_1\wedge s_2), z-x) + g(c(t-s_1\vee s_2), z-x) \right\}
\end{align*}
\noindent if $r \in [s_1\vee s_2,  (t+s_1\vee s_2)/2]$. We now select $\alpha \in [0,\eta)$ such that $\frac{n}{2} + \beta - \frac{\alpha}{2} < 1$. This implies
\begin{align*}
\Big|  \int_{s_1\vee s_2}^{ \frac{t+s_1\vee s_2}{2}} \int_{\mathbb{R}^d} & \partial^n_x \widehat{p}_{m+1}(\mu, s_1\wedge s_2, r, x, y) \,  \Delta_{s_1, s_2} \Phi_{m+1}(\mu, s, r, t, y, z) \, dy \, dr \Big| \\
 & \leq K_{\beta}  \frac{|s_1-s_2|^\beta}{(t-s_1\vee s_2)^{\frac{n}{2}+\beta-\frac{\eta}{2}}}\, \left\{ g(c(t-s_1\wedge s_2), z-x) + g(c(t-s_1\vee s_2), z-x) \right\} \\
 & \leq K_{ \beta}   \left\{ \frac{|s_1-s_2|^\beta}{(t-s_1\wedge s_2)^{\frac{n}{2}+\beta-\frac{\eta}{2}}} g(c(t-s_1\wedge s_2), z-x) + \frac{|s_1-s_2|^\beta}{(t-s_1\vee s_2)^{\frac{n}{2}+\beta-\frac{\eta}{2}}} g(c(t-s_1\vee s_2), z-x) \right\}
\end{align*}
\noindent where for the last inequality we used the fact that $t-s_1\wedge s_2 \leq 2 (t-s_1\vee s_2)$. Otherwise, if $r  \in ( (t+s_1\vee s_2)/2, t]$, from \eqref{gaussian:bound:diff:time:Phim}, we obtain
\begin{align*}
\Big| \int_{\mathbb{R}^d} & \partial^n_x \widehat{p}_{m+1}(\mu, s_1\wedge s_2, r, x, y) \Delta_{s_1, s_2} \Phi_{m+1}(\mu, s, r, t, y, z)  \, dy \Big| \\
& \leq K \frac{|s_1-s_2|^\beta}{(t-r)^{1-\frac{\eta}{2}}(r-s_1\vee s_2)^{\frac{n}{2}+\beta}} \, g(c(t-s_1\wedge s_2), z-x) \\
& \leq  K  \frac{|s_1-s_2|^\beta}{(t-r)^{1-\frac{\eta}{2}}(t-s_1\vee s_2)^{\frac{n}{2}+\beta}} \, g(c(t-s_1\wedge s_2), z-x) \\
& \leq  K \frac{|s_1-s_2|^\beta}{(t-r)^{1-\frac{\eta}{2}}(t-s_1\wedge s_2)^{\frac{n}{2}+\beta}} \, g(c(t-s_1\wedge s_2), z-x) 
\end{align*}
\noindent where we used the inequality $t-s_1\wedge s_2 \leq 2 (t-s_1\vee s_2)$ recalling that $|s_1-s_2| \leq t-s_1\vee s_2$. This yields
\begin{align*}
\Big| \int_{\frac{t+s_1\vee s_2}{2}}^{t} & \int_{\mathbb{R}^d} \partial^n_x \widehat{p}_{m+1}(\mu, s_1\wedge s_2, r, x, y) \Delta_{s_1, s_2} \Phi_{m+1}(\mu, s, r, t, y, z)  \, dy\, dr \Big| \\
& \leq K  \frac{|s_1-s_2|^\beta}{(t-s_1\wedge s_2)^{\frac{n}{2}+\beta-\frac{\eta}{2}}} \, g(c(t-s_1\wedge s_2), z-x).
\end{align*}

Gathering the above estimates, we thus obtain
$$
 | {\rm III} |  \leq K_{ \beta}   \left\{ \frac{|s_1-s_2|^\beta}{(t-s_1\wedge s_2)^{\frac{n}{2}+\beta-\frac{\eta}{2}}} g(c(t-s_1\wedge s_2), z-x) + \frac{|s_1-s_2|^\beta}{(t-s_1\vee s_2)^{\frac{n}{2}+\beta-\frac{\eta}{2}}} g(c(t-s_1\vee s_2), z-x) \right\}.
$$

In order to handle the last term ${\rm IV}$, we use a decomposition similar to the one used previously, namely
\begin{align*}
 & \int_{s_1\wedge s_2}^{s_1\vee s_2 } \int_{\mathbb{R}^d}  \partial^n_x \widehat{p}_{m+1}(\mu, s_1\wedge s_2, r, x, y) \,  \Phi_{m+1}(\mu, s_1\wedge s_2, r, t, y, z)  \, dy\, dr \\
 \quad & = \int_{s_1\wedge s_2}^{  s_1 \vee s_2} \int_{\mathbb{R}^d}  \partial^n_x \widehat{p}_{m+1}(\mu, s_1\wedge s_2, r, x, y) \, [ \Phi_{m+1}(\mu, s_1\wedge s_2, r, t, y, z) - \Phi_{m+1}(\mu, s_1\wedge s_2, r, t, x, z) ]  \, dy\, dr \\
 & \quad +  \int_{s_1\wedge s_2}^{  s_1 \vee s_2} \Phi_{m+1}(\mu, s_1\wedge s_2, r, t, x, z)   \int_{\mathbb{R}^d}  [ \partial^n_x \widehat{p}_{m+1}^{y}(\mu, s_1\wedge s_2, r, x, y) - \partial^n_x \widehat{p}^{x}_{m+1}(\mu, s_1\wedge s_2, r, x, y) ] \, dy \, dr.
\end{align*}

For the first term, from \eqref{holder:reg:voltera:kernel} with $\beta' \in [0, \eta)$ if $n=1$ or $\beta' \in (0,\eta)$ if $n=2$, the space-time inequality \eqref{space:time:inequality} and using the fact that $t-s_1\wedge s_2 \leq 2 (t-s_1\vee s_2)$, we obtain
\begin{align*}
\Big| \int_{s_1\wedge s_2}^{ s_1\vee s_2} & \int_{\mathbb{R}^d}  \partial^n_x \widehat{p}_{m+1}(\mu, s_1\wedge s_2, r, x, y) \, [ \Phi_{m+1}(\mu, s_1\wedge s_2, r, t, y, z) - \Phi_{m+1}(\mu, s_1\wedge s_2, r, t, x, z) ]  \, dy\, dr  \Big| \\
& \leq K_{\beta'} \int_{s_1\wedge s_2}^{ s_1\vee s_2 } \frac{1}{(r-s_1\wedge s_2)^{\frac{n-\beta'}{2}}(t-r)^{1+\frac{\beta'-\eta}{2}}} \, dr \, g(c(t-s_1\wedge s_2), z-x) \\
& \leq K_{\beta'} \frac{|s_1-s_2|^{1+\frac{\beta' - n}{2}}}{(t-s_1\vee s_2)^{1+\frac{\beta'-\eta}{2}}} g(c(t-s_1\wedge s_2), z-x) \\
& \leq K_{\beta'} \frac{|s_1-s_2|^{1+\frac{\beta' - n}{2}}}{(t-s_1\wedge s_2)^{1+\frac{\beta'-\eta}{2}}} g(c(t-s_1\wedge s_2), z-x) \\
& \leq K_{\beta} \frac{|s_1-s_2|^{\beta}}{(t-s_1\wedge s_2)^{\frac{n}{2}+\beta-\frac{\eta}{2}}} g(c(t-s_1\wedge s_2), z-x)
\end{align*}
\noindent for any $\beta \in [0,(1+\eta)/2)$ if $n=1$ or any $\beta \in [0,\eta/2)$ if $n=2$. For the second term, similar arguments as those previously used yield
\begin{align*} 
 \Big| \int_{s_1\wedge s_2}^{s_1\vee s_2 } & \Phi_{m+1}(\mu, s_1\wedge s_2, r, t, x, z)   \int_{\mathbb{R}^d}  [ \partial^n_x \widehat{p}_{m+1}^{y}(\mu, s_1\wedge s_2, r, x, y) - \partial^n_x \widehat{p}^{x}_{m+1}(\mu, s_1\wedge s_2, r, x, y) ] \, dy \, dr \Big| \\
 & \leq K \int_{s_1\wedge s_2}^{s_1\vee s_2}  \frac{1}{(r-s_1 \wedge s_2)^{\frac{n-\eta}{2}}} \, dr \, \frac{1}{(t-s_1\vee s_2)^{1-\frac{\eta}{2}}} \, g(c(t-s_1\wedge s_2), z-x) \\
 & \leq K \frac{|s_1-s_2|^{1-\frac{(n-\eta)}{2}}}{(t-s_1\wedge s_2)^{1-\frac{\eta}{2}}} \, g(c(t-s_1\wedge s_2), z-x) \\
 & \leq K \frac{|s_1-s_2|^\beta}{(t-s_1\wedge s_2)^{\frac{n}{2}+\beta - \eta}}  \, g(c(t-s_1\wedge s_2), z-x).
\end{align*}
\noindent for any $\beta \in [0, (1+\eta)/2]$ if $n=1$ or any $\beta \in [0,\eta/2]$ if $n=2$. Gathering the two previous estimates, we obtain
\begin{align*}
|{\rm IV}| & = \Big| \int_{s_1\wedge s_2}^{ s_1\vee s_2 }  \int_{\mathbb{R}^d}  \partial^n_x \widehat{p}_{m+1}(\mu, s_1\wedge s_2, r, x, y) \,  \Phi_{m+1}(\mu, s_1\wedge s_2, r, t, y, z)  \, dy\, dr \Big| \\
& \leq K_\beta\frac{|s_1-s_2|^\beta}{(t-s_1 \wedge  s_2)^{\frac{n}{2}+\beta - \eta}}  \, g(c(t-s_1\wedge s_2), z-x)
\end{align*}
\noindent for any $\beta \in [0, (1+\eta)/2]$ if $n=1$ or any $\beta \in [0,\eta/2]$ if $n=2$. Putting the estimates on ${\rm I}$, ${\rm II}$, ${\rm III}$ and ${\rm IV}$ together eventually concludes the proof of \eqref{regularity:time:estimate:v1:v2:v3:decoupling:mckean:prop:statement}.

\end{proof}

\subsection{Second part of the induction step} \label{second:part:induction:step}

We here prove the second part of the induction step namely the estimates \eqref{regularity:measure:estimate:v1:v2:decoupling:mckean} and \eqref{regularity:time:estimate:v1:v2:decoupling:mckean} at step $m+1$.  

We start with two auxiliary lemmas whose proofs are postponed to Sections \ref{section:proof:lem:estimate:reg:mes:L:derivatives:coefficients:densities} and \ref{section:proof:lemme:technical:estimate:diff:time} of the appendix. 
 
 \begin{lem}\label{lem:estimate:reg:mes:L:derivatives:coefficients:densities}Let $n\in \left\{0, 1\right\}$. For any $\beta \in [0,1]$ if $n=0$ or any $\beta \in [0,\eta)$ if $n=1$, there exists a constant $K^{+}_\beta$ such that for any $t\in (0,T]$, for any $(\mu, \mu', s, x, z, v) \in \pp^2 \times [0,t) \times (\mathbb{R}^d)^3$, for any $r\in [s, t)$ and any $(i, j) \in \left\{1, \cdots, d\right\}^2$
\begin{align}
& | \partial^{n}_v [\partial_\mu [a_{i, j}(t, x, [X^{s,\xi, (m)}_t])]](v) -   \partial^{n}_v [\partial_\mu [a_{i, j}(t, x, [X^{s, \xi', (m)}_t])]](v)| \notag \\
& \quad \quad + | \partial^{n}_v [\partial_\mu [b_{i}(t, x, [X^{s,\xi, (m)}_t])]](v) -   \partial^{n}_v [\partial_\mu [b_{i}(t, x, [X^{s, \xi', (m)}_t])]](v)| \notag \\
& \quad \quad \quad \leq K^{+}_\beta  \left(\frac{W^{\beta}_2(\mu, \mu')}{(t-s)^{\frac{1+n+\beta-\eta}{2}}} + \int_{(\mathbb{R}^d)^2} (|y'-x'|^\eta \wedge 1) \, | \Delta_{\mu, \mu'} \partial^{n}_v[\partial_\mu p_{m}(\mu, s, t, x', y')](v)| \, dy' \, \mu(dx') \right), \label{dec:cross:deriv:diff:mu:a}
\end{align}

\begin{align}
  |& \Delta_{\mu, \mu'}  \partial^{n}_v [\partial_\mu \widehat{p}_{m+1}(\mu, s, r, t, x, z)](v) | \notag \\
  & \leq  K^{+}_\beta  \left( W^{\beta}_2(\mu, \mu')\left( \frac{1}{(t-s)^{ \frac{1+n+\beta-\eta}{2}}} \I_\seq{r=s} + \frac{1}{(r-s)^{ \frac{1+n+\beta-\eta}{2}}} \I_\seq{r>s} \right) \right. \label{diff:L:deriv:pm:mu:mup} \\
  & \quad \left. + \frac{1}{t-r}   \int_r^t  \int_{(\mathbb{R}^d)^2} (|y'-x'|^{\eta} \wedge1) |\Delta_{\mu, \mu'} \partial^{n}_v [\partial_\mu p_{m}(\mu, s, r', x', y')] (v)| \, dy' \, \mu'(dx') \, dr' \right)  \, g(c(t-r), z-x),  \nonumber
\end{align}

\begin{align}
| \Delta_{\mu, \mu'} & \partial^{n}_v[\partial_\mu [a_{i, j} (t, x, [X^{s, \xi, (m)}_t])  -  a_{i, j}(t, z, [X^{s, \xi, (m)}_t])]](v) | \notag \\
& \leq  K^{+}_\beta W_2^{\beta}(\mu, \mu') \left\{ \frac{(|z-x|^\eta\wedge 1)}{(t-s)^{\frac{1+n+\beta}{2}}} \wedge \frac{1}{(t-s)^{\frac{1+n+\beta-\eta}{2}}} \right\} \label{diff:mes:L:deriv:diff:diff:coeff:holder:reg} \\
& \quad + K^{+}_\beta \left\{ ( |z-x|^\eta \wedge 1) \int_{(\mathbb{R}^d)^2}  |\Delta_{\mu, \mu'} \partial^{n}_v [\partial_\mu p_{m}(\mu, s, t, x', y')](v)|  \, dy'  \mu'(dx') \right. \notag \\
& \quad \quad \left. \wedge \, \int_{(\mathbb{R}^d)^2} (|y'-x'|^\eta \wedge 1)  |\Delta_{\mu, \mu'} \partial^{n}_v [\partial_\mu p_{m}(\mu, s, t, x', y')](v)|  \, dy'  \mu'(dx') \right\}, \notag
\end{align}

\noindent and
 \begin{align}
 & | \Delta_{\mu, \mu'} \partial^{n}_v [\partial_\mu \mH_{m+1}(\mu, s, r, t, x, z)] (v)| \nonumber \\
 & \leq K^{+}_\beta\left( W_2^{\beta}(\mu, \mu') \left\{ \frac{1}{(t-r)(r-s)^{\frac{1+n+\beta-\eta}{2}}} \wedge \frac{1}{(t-r)^{1- \frac{\eta}{2}}(r-s)^{\frac{1+n+\beta}{2}}} \right\} \right. \nonumber \\
 & \left. \quad + \left\{ \frac{1}{(t-r)^{1-\frac{\eta}{2}}} \int_{(\mathbb{R}^d)^2}  |\Delta_{\mu, \mu'} \partial^{n}_v [\partial_\mu p_{m}(\mu, s, r, x', y')](v)|  \, dy'  \mu'(dx') \right. \right. \label{diff:mes:L:deriv:parametrix:kernel:pmp1} \\
& \quad \quad \left. \left. \wedge \, \frac{1}{t-r} \int_{(\mathbb{R}^d)^2} (|y'-x'|^\eta \wedge 1)  |\Delta_{\mu, \mu'} \partial^{n}_v [\partial_\mu p_{m}(\mu, s, r, x', y')](v)|  \, dy'  \mu'(dx') \right\} \right. \nonumber  \\
 & \left. \quad +  \frac{1}{(t-r)^{2-\frac{\eta}{2}}}   \int_r^t \int_{(\mathbb{R}^d)^2} (|y'-x'|^{\eta} \wedge1) |\Delta_{\mu, \mu'} \partial^{n}_v [\partial_\mu p_{m}(\mu, s, r', x', y')] (v)| \, dy' \, \mu'(dx') \, dr' \right) \nonumber \\
 & \quad \quad   g(c(t-r), z-x).\nonumber
 \end{align}
\end{lem}

 \begin{lem}\label{lemme:technical:estimate:diff:time}Let $n\in \left\{0, 1\right\}$. For any $\beta \in [0,(1+\eta)/2)$ if $n=0$ or any $\beta \in [0,\eta/2)$ if $n=1$, there exist positive constants $K^{+}_\beta$, $c$ such that for any $t\in (0,T]$, for any $(\mu, s_1, s_2, x, z, v) \in \pp \times [0,t)^2 \times (\mathbb{R}^d)^3$, for any $r\in [s_1\vee s_2, t)$ and any $(i, j) \in \left\{1, \cdots, d\right\}^2$
\begin{align}
& | \partial^{n}_v [\partial_\mu [a_{i, j}(t, x, [X^{s_1 ,\xi, (m)}_t])]](v) -   \partial^{n}_v [\partial_\mu [a_{i, j}(t, x, [X^{s_2, \xi, (m)}_t])]](v)| \notag \\
& \quad \quad + | \partial^{n}_v [\partial_\mu [b_{i}(t, x, [X^{s_1, \xi , (m)}_t])]](v) -   \partial^{n}_v [\partial_\mu [b_{i}(t, x, [X^{s_2, \xi, (m)}_t])]](v)| \notag \\
& \quad \quad \quad \leq K^{+}_\beta  \left(\frac{|s_1-s_2|^\beta}{(t-s_1\vee s_2)^{\frac{1+n-\eta}{2}+\beta}} + \int_{(\mathbb{R}^d)^2} (|y'-x'|^\eta \wedge 1) \, | \Delta_{s_1, s_2} \partial^{n}_v[\partial_\mu p_{m}(\mu, s, t, x', y')](v)| \, dy' \, \mu(dx') \right), \label{diff:time:cross:deriv:diff:and:deriv:coeff}
\end{align}

\begin{align}
  |& \Delta_{s_1, s_2}  \partial^{n}_v [\partial_\mu \widehat{p}_{m+1}(\mu, s, r, t, x, z)](v) | \notag \\
  & \leq K^{+}_\beta \left\{ \frac{|s_1-s_2|^{\beta}}{(t-s_1\vee s_2)^{\frac{1+n-\eta}{2}+\beta}} \right. \label{diff:time:L:deriv:p:hat} \\
 &    \Big. \left.   + \frac{1}{t-r} \int_{r}^{t} \int_{(\mathbb{R}^d)^2} (|y'-x'|^{\eta}\wedge 1)  |\Delta_{s_1, s_2}\partial^{n}_v[\partial_\mu p_{m}(\mu, s, r', x', y')](v) |   \, dy' \mu(dx') dr' \right\} g(c(t-r), z-x), \notag 
 \end{align}
 
 \begin{align}
  |& \Delta_{s_1, s_2}  \partial^{n}_v [\partial_\mu \widehat{p}_{m+1}(\mu, s, t, x, z)](v) | \notag \\
  & \leq K^{+}_\beta \left\{ \frac{|s_1-s_2|^{\beta}}{(t-s_1)^{\frac{1+n-\eta}{2}+\beta}} g(c(t-s_1), z-x) +  \frac{|s_1-s_2|^{\beta}}{(t-s_2)^{\frac{1+n-\eta}{2}+\beta}} g(c(t-s_2), z-x) \right. \label{diff:time:L:deriv:p:hat:same:time} \\
 &    \Big. \left.   + \frac{1}{t-s_1 \vee s_2} \int_{s_1 \vee s_2}^{t} \int_{(\mathbb{R}^d)^2} (|y'-x'|^{\eta}\wedge 1)  |\Delta_{s_1, s_2}\partial^{n}_v[\partial_\mu p_{m}(\mu, s, r', x', y')](v) |   \, dy' \mu(dx') dr'  g(c(t-s_1 \vee s_2), z-x)\right\} , \notag 
 \end{align}

\begin{align}
| \Delta_{s_1, s_2} & \partial^{n}_v[\partial_\mu [a_{i, j} (t, x, [X^{s, \xi, (m)}_t])  -  a_{i, j}(t, z, [X^{s, \xi, (m)}_t])]](v) | \notag \\
& \leq  K^{+}_\beta |s_1-s_2|^\beta \left\{ \frac{(|z-x|^\eta\wedge 1)}{(t-s_1 \vee s_2)^{\frac{1+n}{2}+\beta}} \wedge \frac{1}{(t-s_1 \vee s_2)^{\frac{1+n-\eta}{2} + \beta }} \right\} \label{diff:time:L:deriv:diff:diff:coeff:holder:reg} \\
& \quad + K^{+}_\beta \left\{ |z-x|^{\eta}\wedge 1 \int_{(\mathbb{R}^d)^2} |\Delta_{s_1, s_2}\partial^{n}_v[\partial_\mu p_{m}(\mu, s, t, x', y')](v) |   \, dy' \mu(dx')  \right. \notag \\
& \quad \quad \left. \wedge  \int_{(\mathbb{R}^d)^2} (|y'-x'|^\eta \wedge 1)  |\Delta_{s_1, s_2}\partial^{n}_v[\partial_\mu p_{m}(\mu, s, t, x', y')](v) |   \, dy' \mu(dx') \right\} \notag
\end{align}

\noindent and
 \begin{align}
 & | \Delta_{s_1, s_2} \partial^{n}_v [\partial_\mu \mH_{m+1}(\mu, s, r, t, x, z)] (v)| \nonumber \\
 & \leq K^{+}_\beta \left( |s_1-s_2|^\beta \left\{ \frac{1}{(t-r)(r-s_1\vee s_2)^{\frac{1+n-\eta}{2}+\beta}} \wedge \frac{1}{(t-r)^{1- \frac{\eta}{2}}(r-s_1 \vee s_2)^{\frac{1+n}{2} + \beta}} \right\} \right. \nonumber \\
 & \left. \quad + \left\{ \frac{1}{(t-r)^{1-\frac{\eta}{2}}} \int_{(\mathbb{R}^d)^2}  |\Delta_{s_1, s_2} \partial^{n}_v [\partial_\mu p_{m}(\mu, s, r, x', y')](v)|  \, dy'  \mu'(dx') \right. \right. \label{diff:time:L:deriv:parametrix:kernel:pmp1} \\
& \quad \quad \left. \left. \wedge \, \frac{1}{t-r} \int_{(\mathbb{R}^d)^2} (|y'-x'|^\eta \wedge 1)  |\Delta_{s_1, s_2} \partial^{n}_v [\partial_\mu p_{m}(\mu, s, r, x', y')](v)|  \, dy'  \mu'(dx') \right\} \right. \nonumber  \\
 & \left. \quad +  \frac{1}{(t-r)^{2-\frac{\eta}{2}}}   \int_r^t \int_{(\rr^d)^2} (|y'-x'|^{\eta} \wedge1) |\Delta_{s_1, s_2} \partial^{n}_v [\partial_\mu p_{m}(\mu, s, r', x', y')] (v)| \, dy' \, \mu'(dx') \, dr' \right) \nonumber \\
 & \quad \quad   g(c(t-r), z-x).\nonumber
 \end{align}
\end{lem}

With the above technical results at hand, we are now ready to prove the induction step for the estimates \eqref{regularity:measure:estimate:v1:v2:decoupling:mckean} and \eqref{regularity:time:estimate:v1:v2:decoupling:mckean}. We proceed as we did before in the proof of the first part of the induction step.

 \begin{prop}\label{prop:sensi:space:time:lions:deriv} Let $n\in \left\{0,1\right\}$. For any $\beta \in  [0,1]$ if $n=0$ or any $\beta  \in [0,\eta) $ if $n=1$, there exist positive constants $K^{+}_\beta$ and $c$ such that for any $(t, x, z, v)\in (0,T] \times (\mathbb{R}^d)^3$, any $(s, \mu, \mu') \in [0,t) \times \pp^2$ and any value of the constant $C_\beta^{+}$ appearing in the estimates \eqref{regularity:measure:estimate:v1:v2:decoupling:mckean}
\begin{align}
|\partial^{n}_v & [\partial_\mu p_{m+1}(\mu, s, t, x, z)](v) - \partial^{n}_v [\partial_\mu p_{m+1}(\mu', s, t, x, z)](v) | \nonumber \\
& \leq   K^{+}_\beta \frac{W_2^{\beta}(\mu, \mu')}{(t-s)^{\frac{1+n + \beta-\eta}{2}}} \left\{ 1 + \sum_{k=1}^{m} (C_\beta^{+})^{k} (t-s)^{k \frac{\eta}{2}}  \prod_{i=1}^{k} B\left(\frac{\eta}{2}, \frac{1-n+\eta-\beta}{2} +  (i-1)\frac{\eta}{2} \right)  \right\} \label{estimate:deriv:mes:holder:reg:mes:mes:dens:stepmp1:cor} \\
& \quad \times  g(c(t-s), z-x). \nonumber
\end{align}
\noindent Similarly, for any $ \beta \in  [0, (1+\eta)/2)$ if $n=0$ or any $\beta \in [0,\eta/2)$ if $n=1$, there exist positive constants $K^{+}_\beta$ and $c$ such that for any $(t, x, z, v)\in (0,T] \times (\mathbb{R}^d)^3$, any $(s_1, s_2) \in [0,t)^2$ and any value of the constant $C_\beta^{+}$ appearing in the estimates \eqref{regularity:time:estimate:v1:v2:decoupling:mckean} 
\begin{align}
|\partial^{n}_v & [\partial_\mu p_{m+1}(\mu, s_1, t, x, z)](v) - \partial^{n}_v [\partial_\mu p_{m+1}(\mu, s_2, t, x, z)](v) | \nonumber \\
& \leq   K^{+}_{\beta}  \left\{ 1 + \sum_{k=1}^{m} (C_\beta^{+})^{k} (t-s_1 \vee s_2)^{k \frac{\eta}{2}}  \prod_{i=1}^{k} B\left(\frac{\eta}{2}, \frac{1-n+\eta}{2}-\beta +  (i-1)\frac{\eta}{2} \right)  \right\} \label{estimate:deriv:mes:holder:reg:time:dens:stepmp1:cor} \\
& \quad \times  \left\{ \frac{|s_1-s_2|^{\beta}}{(t-s_1)^{\frac{1+n -\eta}{2}+\beta}} g(c(t-s_1), z-x) + \frac{|s_1-s_2|^{\beta}}{(t-s_2)^{\frac{1+n -\eta}{2}+\beta}} g(c(t-s_2), z-x) \right\}. \nonumber
\end{align}

\end{prop}

\bigskip

\noindent \emph{Conclusion of the second part of the induction step:}\\
Similarly to the conclusion of the first part of the induction step, we set the constant $C_\beta^{+}$ in the mth partial sums $\mathscr{C}^{n, \beta}_m(C^{+}_{\beta}, t-s)$ and $\mathscr{C}^{n, 2\beta}_m(C^{+}_{\beta}, t-s)$ which are used in the statement of the Gaussian estimates \eqref{regularity:measure:estimate:v1:v2:decoupling:mckean} and \eqref{regularity:time:estimate:v1:v2:decoupling:mckean} to be equal to the maximum between the two constants $K^{+}_\beta$ (that we still denote by $K^{+}_\beta$) appearing in the right-hand side of the Gaussian estimates \eqref{estimate:deriv:mes:holder:reg:mes:mes:dens:stepmp1:cor} and \eqref{estimate:deriv:mes:holder:reg:time:dens:stepmp1:cor}. In doing so, by the very definition of $\mathscr{C}^{n, \beta}_{m+1}(K^+_\beta, t-s)$ and $\mathscr{C}^{n, 2\beta}_{m+1}(K^{+}_\beta, t-s)$, we conclude that the estimates \eqref{estimate:deriv:mes:holder:reg:mes:mes:dens:stepmp1:cor} and \eqref{estimate:deriv:mes:holder:reg:time:dens:stepmp1:cor} directly yield the desired estimates \eqref{regularity:measure:estimate:v1:v2:decoupling:mckean} and \eqref{regularity:time:estimate:v1:v2:decoupling:mckean} at step $m+1$. We thus conclude that the Gaussian estimates \eqref{regularity:measure:estimate:v1:v2:decoupling:mckean} and \eqref{regularity:time:estimate:v1:v2:decoupling:mckean} hold for any positive integer $m$. This completes the proof of the second of part of Proposition \ref{proposition:reg:density:recursive:scheme:mckean}.

\bigskip

\begin{proof}

\noindent \emph{Step 1: proof of \eqref{estimate:deriv:mes:holder:reg:mes:mes:dens:stepmp1:cor}}\\
From the identities \eqref{representation:formula:cross:lions:deriv:dens} and \eqref{infinite:series:Phi:step:m}, we directly obtain
\begin{align*}
\partial^{n}_v [\partial_\mu p_{m+1}(\mu, s, t, x, z)] (v) & = (\partial^{n}_v [\partial_\mu \widehat{p}_{m+1}] + p_{m+1} \otimes \partial^{n}_v [\partial_\mu \mH_{m+1}]) \\
& \quad + (\partial^{n}_v [\partial_\mu \widehat{p}_{m+1}] + p_{m+1} \otimes \partial^{n}_v [\partial_\mu \mH_{m+1}]) \otimes \Phi_{m+1}(\mu, s, t, x, z)(v)
\end{align*}

\noindent so that
\begin{align}
\Delta_{\mu, \mu'} \partial^{n}_v & [\partial_\mu p_{m+1}(\mu, s, t, x, z)] (v) \nonumber \\
& = \Delta_{\mu, \mu'} \partial^{n}_v [\partial_\mu \widehat{p}_{m+1}(\mu, s, t, x, z)](v) + p_{m+1} \otimes (\Delta_{\mu, \mu'} \partial^{n}_v \partial_\mu \mH_{m+1})(\mu', s, t, x, z)(v) \nonumber \\
& \quad + (\Delta_{\mu, \mu'} p_{m+1}) \otimes \partial^{n}_v \partial_\mu \mH_{m+1}(\mu, s, t, x, z)(v) \label{relation:diff:mes:L:deriv:heat:kernel:pmp1} \\
& \quad + \Big(\Delta_{\mu, \mu'}\Big[\partial^{n}_v [\partial_\mu \widehat{p}_{m+1}] + p_{m+1} \otimes \partial^{n}_v [\partial_\mu \mH_{m+1}]\Big]\Big) \otimes \Phi_{m+1}(\mu, s, t, x, z)(v) \nonumber\\
& \quad + (\partial^{n}_v [\partial_\mu \widehat{p}_{m+1}] + p_{m+1} \otimes \partial^{n}_v [\partial_\mu \mH_{m+1}]) \otimes(\Delta_{\mu, \mu'} \Phi_{m+1})(\mu', s, t, x, z)(v). \nonumber
\end{align}

We now establish an appropriate upper-bound for each term of the above decomposition. 

From \eqref{diff:L:deriv:pm:mu:mup}, \eqref{regularity:measure:estimate:v1:v2:decoupling:mckean} at step $m$ and the space-time inequality \eqref{space:time:inequality}, we obtain
\begin{align}
|\Delta_{\mu, \mu'} \partial^{n}_v & [\partial_\mu \widehat{p}_{m+1}(\mu, s, t, x, z)] (v)| \notag  \\
&\leq K^{+}_\beta \left\{ \frac{1}{(t-s)^{\frac{1+n+\beta-\eta}{2}}} + \frac{1}{t-s} \int_s^t  \frac{\mathscr{C}^{n, \beta}_{m}(C_\beta^{+}, r-s)}{(r-s)^{\frac{1+n+\beta}{2}-\eta}} dr \right\} \, W^{\beta}_2(\mu, \mu') \,  g(c(t-s), z-x) \notag \\
&\leq K^{+}_\beta \left\{ \frac{1}{(t-s)^{\frac{1+n+\beta-\eta}{2}}} +\int_s^t  \frac{\mathscr{C}^{n, \beta}_{m}(C_\beta^{+}, r-s)}{(t-r)^{1-\frac{\eta}{2}}(r-s)^{\frac{1+n+\beta-\eta}{2}}} dr \right\} \, W^{\beta}_2(\mu, \mu') \,  g(c(t-s), z-x) \notag \\
& \leq \frac{K^{+}_\beta}{(t-s)^{\frac{1+n+\beta-\eta}{2}}} \left\{ 1+ \sum_{k=1}^{m} (C_\beta^{+})^{k} (t-s)^{k \frac{\eta}{2}} \prod_{i=1}^{k} B\left(\frac{\eta}{2}, \frac{1-n +\eta-\beta}{2} +  (i-1)\frac{\eta}{2} \right)  \right\} \label{induction:step:diff:mes:L:deriv:phat}\\
& \quad \times W^{\beta}_2(\mu, \mu') \, g(c(t-s), z-x) \notag
\end{align}
\noindent for any $\beta \in [0,1]$ if $n=0$ and any $\beta \in [0,\eta)$ of $n=1$.

In order to derive an estimate for the second term, we separate the time integral of the space-time convolution into the two disjoint intervals $[s, (t+s)/2]$ and $[(t+s)/2, t]$ in order to balance the time singularity induced by \eqref{diff:mes:L:deriv:parametrix:kernel:pmp1}. Indeed, from \eqref{diff:mes:L:deriv:parametrix:kernel:pmp1}, \eqref{regularity:measure:estimate:v1:v2:decoupling:mckean} at step $m$ and the space-time inequality \eqref{space:time:inequality}, one gets
\begin{align*}
& \int_{s}^{\frac{t+s}{2}}  \int_{\mathbb{R}^d} |p_{m+1}(\mu', s, r, x, y)|  |\Delta_{\mu, \mu'} \partial^{n}_v \partial_\mu \mH_{m+1}(\mu, s, r,  t, y, z)| \, dy \, dr\\
& \leq K^{+}_\beta  \int_s^{\frac{t+s}{2}} \left(\frac{1}{(t-r)(r-s)^{\frac{1+n+\beta - \eta}{2}}} + \frac{\mathscr{C}^{n,\beta}_m(C_\beta^{+},r-s)}{(t-r)(r-s)^{\frac{1+n+\beta}{2}-\eta}} \right. \\
& \quad \left. + \frac{1}{(t-r)^{2-\frac{\eta}{2}}} \int_r^t \frac{\mathscr{C}^{n, \beta}_m(C^{+}_\beta,  r'-s)}{(r'-s)^{\frac{1+n+ \beta}{2} -\eta}}\, dr'\right) \, dr \, W_2^\beta(\mu, \mu') \, g(c(t-s), z-x) \\
& \leq K^{+}_\beta \left( \frac{1}{(t-s)^{\frac{1+n+\beta-\eta}{2}}} + \int_s^{t} \frac{\mathscr{C}^{n,\beta}_m(C_\beta^{+}, r-s)}{(t-r)^{1-\frac{\eta}{2}}(r-s)^{\frac{1+n+\beta-\eta}{2}}} \,dr   \right) \, W_2^\beta(\mu, \mu')\,  g(c(t-s), z-x) \\
& \leq \frac{K^{+}_\beta}{(t-s)^{\frac{1+n+\beta-\eta}{2}}} \left\{ 1+ \sum_{k=1}^{m} (C_\beta^{+})^{k} (t-s)^{k \frac{\eta}{2}} \prod_{i=1}^{k} B\left(\frac{\eta}{2}, \frac{1-n +\eta-\beta}{2} +  (i-1)\frac{\eta}{2} \right)  \right\}\\
& \quad \times W^{\beta}_2(\mu, \mu') \, g(c(t-s), z-x)
\end{align*}
\noindent for any $\beta \in [0,1]$ if $n=0$ and any $\beta \in [0,\eta)$ of $n=1$.

From similar arguments, we obtain
\begin{align*}
& \int_{\frac{t+s}{2}}^{t}  \int_{\mathbb{R}^d} |p_{m+1}(\mu', s, r, x, y)|  |\Delta_{\mu, \mu'} \partial^{n}_v \partial_\mu \mH_{m+1})(\mu, s, r,  t, y, z)| \, dy \, dr\\
& \leq K^{+}_\beta  \int_{\frac{t+s}{2}}^{t} \left(\frac{1}{(t-r)^{1-\frac{\eta}{2}}(r-s)^{\frac{1+n+\beta }{2}}} + \frac{\mathscr{C}^{n,\beta}_m(C_\beta^{+}, r-s)}{(t-r)^{1-\frac{\eta}{2}}(r-s)^{\frac{1+n+\beta-\eta}{2}}} \right. \\
& \quad \left. + \frac{1}{(t-r)^{2-\frac{\eta}{2}}} \int_r^t \frac{\mathscr{C}^{n, \beta}_m(C_\beta^{+}, r'-s)}{(r'-s)^{\frac{1+n+ \beta}{2} -\eta}}\, dr'\right) \, dr \, W_2^\beta(\mu, \mu') \, g(c(t-s), z-x) \\
& \leq K^{+}_\beta \left( \frac{1}{(t-s)^{\frac{1+n+\beta-\eta}{2}}} + \int_s^{t} \frac{\mathscr{C}^{n,\beta}_m(C_\beta^{+}, r-s)}{(t-r)^{1-\frac{\eta}{2}}(r-s)^{\frac{1+n+\beta-\eta}{2}}} \,dr   \right) \, W_2^\beta(\mu, \mu')\,  g(c(t-s), z-x) \\
& \leq \frac{K^{+}_\beta}{(t-s)^{\frac{1+n+\beta-\eta}{2}}} \left\{ 1+ \sum_{k=1}^{m} (C_\beta^{+})^{k} (t-s)^{k \frac{\eta}{2}} \prod_{i=1}^{k} B\left(\frac{\eta}{2}, \frac{1-n +\eta-\beta}{2} +  (i-1)\frac{\eta}{2} \right)  \right\}\\
& \quad \times W^{\beta}_2(\mu, \mu') \, g(c(t-s), z-x)
\end{align*}
\noindent for any $\beta \in [0,1]$ if $n=0$ and any $\beta \in [0,\eta)$ of $n=1$.

Gathering the two previous estimates, we thus conclude
\begin{align}
|p_{m+1} & \otimes (\Delta_{\mu, \mu'} \partial^{n}_v \partial_\mu \mH_{m+1})(\mu', s, t, x, z)(v)| \notag  \\
& \leq \frac{K^{+}_\beta}{(t-s)^{\frac{1+n+\beta-\eta}{2}}} \left\{ 1+ \sum_{k=1}^{m} (C_\beta^{+})^{k} (t-s)^{k \frac{\eta}{2}} \prod_{i=1}^{k} B\left(\frac{\eta}{2}, \frac{1-n +\eta-\beta}{2} +  (i-1)\frac{\eta}{2} \right)  \right\} \label{convol:p:diff:mes:L:deriv:parametrix:kernel} \\
& \quad \times W^{\beta}_2(\mu, \mu') \, g(c(t-s), z-x) \notag 
\end{align}
\noindent for any $\beta \in [0,1]$ if $n=0$ and any $\beta \in [0,\eta)$ if $n=1$.

For the third term, we first remark that \eqref{cross:mes:deriv:parametrix:kernel:s:r:t} together with \eqref{first:second:estimate:induction:decoupling:mckean} (recall again that $\mathscr{C}^{n,0}_\infty = \lim_{m \uparrow \infty} \mathscr{C}^{n, 0}_m < \infty$) yields
\begin{equation}
\label{estimate:uniform:L:deriv:parametrix:kernel}
|\partial^n_v [\partial_\mu  \mH_{m+1}(\mu, s, r ,t ,x ,z)](v)|  \leq K \left( \frac{1}{(t-r)^{1-\frac{\eta}{2}}(r-s)^{\frac{1+n}{2}}} \wedge  \frac{1}{(t-r)(r-s)^{\frac{1+n-\eta}{2}}}  \right) \, g(c(t-r), z-x)
\end{equation}

\noindent so that using \eqref{regularity:measure:estimate:v1:v2:v3:decoupling:mckean} with $n=0$ and separating the time integral of the space-time convolution into the two disjoint intervals $[s, (t+s)/2]$ and $[(t+s)/2, t]$ as before to equilibrate the time singularity yield
 \begin{equation}
 \label{diff:mes:p:convol:L:deriv:parametrix:kernel}
|(\Delta_{\mu, \mu'} p_{m+1}) \otimes \partial^{n}_v \partial_\mu \mH_{m+1}(\mu, s, t, x, z)(v)| \leq K_\beta \frac{W^{\beta}_2(\mu, \mu')}{(t-s)^{\frac{1+n+\beta-\eta}{2}}} \, g(c(t-s), z-x)
 \end{equation}
 \noindent for any $\beta \in [0,1]$ if $n=0$ and any $\beta \in [0,\eta)$ if $n=1$.
 
From \eqref{induction:step:diff:mes:L:deriv:phat}, \eqref{convol:p:diff:mes:L:deriv:parametrix:kernel} and \eqref{diff:mes:p:convol:L:deriv:parametrix:kernel}, we also deduce
 \begin{align*}
\Big| \Delta_{\mu, \mu'}\Big[\partial^{n}_v&  [\partial_\mu \widehat{p}_{m+1}] + p_{m+1} \otimes \partial^{n}_v [\partial_\mu \mH_{m+1}]\Big](\mu, s, t, x, z) \Big|\\
& \leq \frac{K^{+}_\beta}{(t-s)^{\frac{1+n+\beta-\eta}{2}}} \left\{ 1+ \sum_{k=1}^{m} (C_\beta^{+})^{k} (t-s)^{k \frac{\eta}{2}} \prod_{i=1}^{k} B\left(\frac{\eta}{2}, \frac{1-n +\eta-\beta}{2} +  (i-1)\frac{\eta}{2} \right)  \right\}\\
& \quad \times W^{\beta}_2(\mu, \mu') \, g(c(t-s), z-x)
 \end{align*}
 
 \noindent which combined with \eqref{Gaussian:estimate:Phim} implies
  
\begin{align*}
\Big| \Big(\Delta_{\mu, \mu'}\Big[\partial^{n}_v&  [\partial_\mu \widehat{p}_{m+1}] + p_{m+1} \otimes \partial^{n}_v [\partial_\mu \mH_{m+1}]\Big]\Big) \otimes \Phi_{m+1}(\mu, s, t, x, z) \Big|\\
&\leq \frac{K^{+}_\beta}{(t-s)^{\frac{1+n+\beta-\eta}{2}}} \left\{ 1+ \sum_{k=1}^{m} (C_\beta^{+})^{k} (t-s)^{k \frac{\eta}{2}} \prod_{i=1}^{k} B\left(\frac{\eta}{2}, \frac{1-n +\eta-\beta}{2} +  (i-1)\frac{\eta}{2} \right)  \right\}\\
& \quad \times W^{\beta}_2(\mu, \mu') \, g(c(t-s), z-x).
\end{align*}
 
 For the last term, we first combine \eqref{cross:mes:deriv:p:hat:s:r:t} (with $r=s$) with \eqref{first:second:estimate:induction:decoupling:mckean} (recalling again that $\mathscr{C}^{n, 0}_m \leq K:= \mathscr{C}^{n,0}_\infty < \infty$) and use \eqref{gaussian:bound:diff:Phimp1:mu} so that
 $$
 |\partial^{n}_v \partial_\mu \widehat p_{m+1} \otimes (\Delta_{\mu, \mu'} \Phi_{m+1})(\mu', s, t, x , z)(v)| \leq K_\beta \frac{W_2^{\beta}(\mu, \mu')}{(t-s)^{\frac{1+n+\beta-\eta}{2}}} \, g(c(t-s), z-x)
 $$
 \noindent for any $\beta \in [0,1]$ if $n=0$ and any $\beta \in [0,\eta)$ if $n=1$. We then use \eqref{bound:derivative:heat:kernel} (with $n=0$), \eqref{estimate:uniform:L:deriv:parametrix:kernel} and split the time integral on $[s, t]$ into the two disjoint intervals $[s, (t+s)/2)$ and $[(t+s)/2, t]$ in order to balance the time singularity in the space-time convolution $p_{m+1} \otimes \partial^{n}_v [\partial_\mu \mH_{m+1}(\mu, s, t, x, z)$ as we did previously. After some standard computations that we omit, we obtain
 $$
 | p_{m+1} \otimes \partial^{n}_v [\partial_\mu \mH_{m+1}(\mu, s, t, x, z) | \leq \frac{K}{(t-s)^{\frac{1+n-\eta}{2}}} \, g(c(t-s), z-x)
 $$
 \noindent which in turn, using \eqref{gaussian:bound:diff:Phimp1:mu}, yields
 $$
 | (p_{m+1} \otimes \partial^{n}_v [\partial_\mu \mH_{m+1}]) \otimes (\Delta_{\mu, \mu'} \Phi_{m+1})(\mu', s, t, x, z)(v)| \leq K_\beta \frac{W_2^{\beta}(\mu, \mu')}{(t-s)^{\frac{1+n+\beta}{2}-\eta}} \, g(c(t-s), z-x)
 $$
  \noindent for any $\beta \in [0,1]$ if $n=0$ and any $\beta \in [0,\eta)$ if $n=1$.
 
We collect the previous estimates and eventually conclude
\begin{align*}
 |  \Delta_{\mu, \mu'} \partial^{n}_v &  [\partial_\mu p_{m+1}(\mu, s, t, x, z)] (v)| \\
 & \leq \frac{K^{+}_\beta}{(t-s)^{\frac{1+n+\beta-\eta}{2}}} \left\{ 1+ \sum_{k=1}^{m} (C_\beta^{+})^{k} (t-s)^{k \frac{\eta}{2}} \prod_{i=1}^{k} B\left(\frac{\eta}{2}, \frac{1-n +\eta-\beta}{2} +  (i-1)\frac{\eta}{2} \right)  \right\}\\
& \quad \times W^{\beta}_2(\mu, \mu') \, g(c(t-s), z-x)
\end{align*} 
  \noindent for any $\beta \in [0,1]$ if $n=0$ and any $\beta \in [0,\eta)$ if $n=1$. \\

\noindent \emph{Step 2: proof of \eqref{estimate:deriv:mes:holder:reg:time:dens:stepmp1:cor}.} \\

If $|s_1-s_2| \geq t - s_1 \vee s_2$, the estimate \eqref{estimate:deriv:mes:holder:reg:time:dens:stepmp1:cor} directly stems from \eqref{first:second:estimate:induction:decoupling:mckean} and the fact that $\mathscr{C}^{n, 0}_m \leq \mathscr{C}^{n, 0}_\infty < \infty$. For the rest of the proof, we thus assume that $|s_1 - s_2| \leq t-s_1\vee s_2$. Similarly to the proof of the previous step, we start from the representation in infinite series \eqref{representation:formula:cross:lions:deriv:dens} which in view of the identity \eqref{infinite:series:Phi:step:m} directly gives the following decomposition
\begin{align}
\Delta_{s_1, s_2} \partial^{n}_v[\partial_\mu p_{m+1}(\mu, s, t, x, z)](v) & = \Delta_{s_1, s_2} \partial^{n}_v[\partial_\mu \widehat{p}_{m+1}(\mu, s, t , x, z)](v) \nonumber \\
& +  \Delta_{s_1, s_2} (p_{m+1} \otimes \partial^{n}_v[\partial_\mu \mH_{m+1}])(\mu, s, t, x, z)(v) \nonumber\\
& + \Delta_{s_1, s_2} (\partial^{n}_v[\partial_\mu \widehat{p}_{m+1}]\otimes \Phi_{m+1})(\mu, s, t, x, z)(v) \label{decomp:delta:time:mu:deriv}\\
& + \Delta_{s_1, s_2} ((p_{m+1}\otimes \partial^{n}_v[\partial_\mu \mH_{m+1}])\otimes \Phi_{m+1})(\mu, s, t, x, z)(v).\nonumber
\end{align}

We now investigate each term of the above decomposition

From \eqref{diff:time:L:deriv:p:hat:same:time}, \eqref{regularity:time:estimate:v1:v2:decoupling:mckean} at step $m$ and the space-time inequality \eqref{space:time:inequality}, we get
\begin{align}
| \Delta_{s_1, s_2} & \partial^{n}_v[\partial_\mu \widehat{p}_{m+1}(\mu, s, t , x, z)](v) |  \nonumber \\
& \leq K^{+}_\beta \left\{ \frac{|s_1-s_2|^{\beta}}{(t-s_1)^{\frac{1+n-\eta}{2}+\beta}} g(c(t-s_1), z-x) +  \frac{|s_1-s_2|^{\beta}}{(t-s_2)^{\frac{1+n-\eta}{2}+\beta}} g(c(t-s_2), z-x) \right. \nonumber \\
 &    \Big. \left.   + \frac{|s_1-s_2|^{\beta}}{t-s_1 \vee s_2} \int_{s_1 \vee s_2}^{t} \frac{\mathscr{C}^{n, 2\beta}_{m}(C_\beta^{+}, r'-s_1 \vee s_2)}{(r'-s_1 \vee s_2)^{\frac{1+n}{2}-\eta+\beta}} dr'  g(c(t-s_1 \vee s_2), z-x)\right\} \nonumber \\
 &  \leq K^{+}_\beta \left\{ \frac{|s_1-s_2|^{\beta}}{(t-s_1)^{\frac{1+n-\eta}{2}+\beta}} g(c(t-s_1), z-x) +  \frac{|s_1-s_2|^{\beta}}{(t-s_2)^{\frac{1+n-\eta}{2}+\beta}} g(c(t-s_2), z-x) \right. \nonumber \\
 &    \Big. \left.   + |s_1-s_2|^{\beta} \int_{s_1 \vee s_2}^{t} \frac{\mathscr{C}^{n, 2\beta}_{m}(C_\beta^{+}, r'-s_1 \vee s_2)}{(t-r')^{1-\frac{\eta}{2}}(r'-s_1 \vee s_2)^{\frac{1+n-\eta}{2}+\beta}} dr'  g(c(t-s_1 \vee s_2), z-x)\right\} \nonumber\\
 & \leq  K^{+}_\beta  \left\{ 1 + \sum_{k=1}^{m} (C_\beta^{+})^{k} (t-s_1 \vee s_2)^{k \frac{\eta}{2}}  \prod_{i=1}^{k} B\left(\frac{\eta}{2}, \frac{1-n+\eta}{2}-\beta +  (i-1)\frac{\eta}{2} \right)  \right\} \label{diff:time:phat:mp1:recursive:bound}\\
& \quad \times  \left\{ \frac{|s_1-s_2|^{\beta}}{(t-s_1)^{\frac{1+n -\eta}{2}+\beta}} g(c(t-s_1), z-x) + \frac{|s_1-s_2|^{\beta}}{(t-s_2)^{\frac{1+n -\eta}{2}+\beta}} g(c(t-s_2), z-x) \right\}. \nonumber
\end{align}

We now turn our attention to the quantity $ \Delta_{s_1, s_2} (p_{m+1} \otimes \partial^{n}_v[\partial_\mu \mH_{m+1}])(\mu, s, t, x, z)(v) $ and use of the following decomposition 
$$
 \Delta_{s_1, s_2} (p_{m+1} \otimes \partial^{n}_v[\partial_\mu \mH_{m+1}])(\mu, s, t, x, z)(v) = {\rm I} + {\rm II} + {\rm III},
$$

\noindent with
\begin{align*}
{\rm I} & := \int_{s_1 \vee s_2}^t \int_{\mathbb{R}^d} \Delta_{s_1, s_2} p_{m+1}(\mu, s, r, x, y)  \partial^{n}_v[\partial_\mu \mH_{m+1}(\mu, s_1\vee s_2, r, t, y, z)](v) \, dy \, dr, 
\end{align*}

\begin{align*}
{\rm II} & := \int_{s_1 \vee s_2}^t \int_{\mathbb{R}^d}  p_{m+1}(\mu, s_1\wedge s_2, r, x, y) \Delta_{s_1, s_2}  \partial^{n}_v[\partial_\mu \mH_{m+1}(\mu, s, r, t, y, z)(v) \, dy \, dr,
\end{align*}
\noindent and
\begin{align*}
{\rm III} & := - \int_{s_1 \wedge s_2}^{s_1 \vee s_2} \int_{\mathbb{R}^d}   p_{m+1}(\mu, s_1 \wedge s_2, r, x, y)  \partial^{n}_v[\partial_\mu \mH_{m+1}(\mu, s_1\wedge s_2, r, t, y, z)](v) \, dy \, dr.
\end{align*}

We first use the estimates \eqref{regularity:time:estimate:v1:v2:v3:decoupling:mckean:prop:statement} and \eqref{estimate:uniform:L:deriv:parametrix:kernel}, then split the time integral into the two intervals $[s_1\vee s_2, (t+s_1\vee s_2)/2)$ and $ [(t+s_1\vee s_2)/2, t]$ to balance the time singularity and eventually use the fact that $t-s_1 \wedge s_2 \leq 2 (t-s_1 \vee s_2)$. After some standard computations that we omit, we obtain
$$
 |{\rm I}| \leq K \left\{\frac{|s_1-s_2|^{\beta}}{(t-s_1)^{\frac{1+n-\eta}{2}+\beta}} g(c(t-s_1), z-x) + \frac{|s_1-s_2|^{\beta}}{(t-s_2)^{\frac{1+n-\eta}{2}+\beta}} g(c(t-s_2) , z-x) \right\}.
$$

To deal with ${\rm II}$, we use \eqref{diff:time:L:deriv:parametrix:kernel:pmp1}. To be more specific, we again split the time integral into the two disjoint intervals $[s_1\vee s_2, (t+s_1\vee s_2)/2)$ and $ [(t+s_1\vee s_2)/2, t]$ as previously done. For the time integral on $[s_1\vee s_2, (t+s_1\vee s_2)/2)$, we bound the first term appearing on the right-hand side of \eqref{diff:time:L:deriv:parametrix:kernel:pmp1} which writes as a minimum by $K^{+}_\beta |s_1-s_2|^{\beta} (t-r)^{-1} (r-s_1\vee s_2)^{- \frac{(1+n-\eta)}{2}-\beta} g(c(t-r), z-x)$ while for the second term which also writes as a minimum, we bound it by $K^{+}_\beta (t-r)^{-1} \int_{(\mathbb{R}^d)^2} (|y'-x'|^\eta \wedge 1)  |\Delta_{s_1, s_2} \partial^{n}_v [\partial_\mu p_{m}(\mu, s, r, x', y')](v)|  \, dy'  \mu'(dx') \, g(c(t-r), z-x)$. For the time integral on $[(t+s_1\vee s_2)/2), t]$, we bound the first term appearing on the right-hand side of \eqref{diff:time:L:deriv:parametrix:kernel:pmp1} by $K^{+}_\beta |s_1-s_2|^{\beta} (t-r)^{-1+\frac{\eta}{2}} (r-s_1\vee s_2)^{- \frac{(1+n)}{2}-\beta} g(c(t-r), z-x)$ while for the second term, we bound it by $K^{+}_\beta (t-r)^{-1+\frac{\eta}{2}} \int_{(\mathbb{R}^d)^2} |\Delta_{s_1, s_2} \partial^{n}_v [\partial_\mu p_{m}(\mu, s, r, x', y')](v)|  \, dy'  \mu'(dx') \, g(c(t-r), z-x)$. For the third term, in both cases, we use Fubini's theorem. After some standard computations that we omit, we obtain
\begin{align*}
|{\rm II}| & \leq K^{+}_\beta \left\{\frac{1}{(t-s_1\vee s_2)^{\frac{1+n-\eta}{2}+\beta}} +  \int_{s_1\vee s_2}^{t} \frac{\mathscr{C}^{n, 2 \beta}_{m}(C^{+}_\beta , r-s_1\vee s_2)}{(t-r)^{1-\frac{\eta}{2}}(r-s_1\vee s_2)^{\frac{1+n}{2}+\beta-\frac{\eta}{2}}} \, dr \, \right\} \\
& \quad \times |s_1-s_2|^{\beta} g(c(t-s_1\wedge s_2), z-x)
\end{align*}

\noindent for any $\beta \in [0, (1+\eta)/2)$ if $n=0$ or any $\beta \in [0,\eta/2)$ if $n=1$.

We eventually deal with ${\rm III}$ by using \eqref{estimate:uniform:L:deriv:parametrix:kernel}. We get 
\begin{align*}
|{\rm III}| & \leq K \int_{s_1 \wedge s_2}^{s_1 \vee s_2} \frac{1}{(t-r)(r-s_1 \wedge s_2)^{\frac{1+n-\eta}{2}}} \, dr \, g(c(t-s_1 \wedge s_2), z-x)\\
& \leq K \frac{|s_1-s_2|^{\frac{1-n+\eta}{2}}}{t-s_1\vee s_2} \, g(c(t- s_1\wedge s_2), z-x) \\
& \leq K \frac{|s_1-s_2|^{\beta}}{(t-s_1\wedge s_2)^{\frac{1+n-\eta}{2}+\beta}} \, g(c(t-s_1\wedge s_2) , z-x)
\end{align*}

\noindent for any $\beta \in [0, (1+\eta)/2]$ if $n=0$ and any $\beta \in [0, \eta/2]$ if $n=1$ where we used the fact that $t-s_1 \wedge s_2 \leq 2 (t-s_1 \vee s_2)$ for the last inequality. Gathering the three previous estimates and using \eqref{regularity:time:estimate:v1:v2:decoupling:mckean} at step $m$ finally yield
\begin{align}
| \Delta_{s_1, s_2} (p_{m+1} & \otimes \partial^{n}_v[\partial_\mu \mH_{m+1}])(\mu, s, t, x, z)(v)|\nonumber \\
&  \leq K^{+}_\beta \left\{\frac{|s_1-s_2|^{\beta}}{(t-s_1)^{\frac{1+n-\eta}{2}+\beta}} g(c(t-s_1), z-x) + \frac{|s_1-s_2|^{\beta}}{(t-s_2)^{\frac{1+n-\eta}{2}+\beta}} g(c(t-s_2) , z-x) \right\} \nonumber \\
& + K^{+}_\beta |s_1-s_2|^{\beta} \int_{s_1\vee s_2}^{t} \frac{\mathscr{C}^{n, 2\beta}_{m}(C^{+}_\beta, r-s_1 \vee s_2)}{(t-r)^{1-\frac{\eta}{2}}(r-s_1\vee s_2)^{\frac{1+n}{2}+\beta-\frac{\eta}{2}}} \, dr \, g(c(t-s_1\wedge s_2), z-x) \nonumber \\
& \leq  K^{+}_\beta  \left\{ 1 + \sum_{k=1}^{m} (C^{+}_\beta)^{k} (t-s_1 \vee s_2)^{k \frac{\eta}{2}}  \prod_{i=1}^{k} B\left(\frac{\eta}{2}, \frac{1-n+\eta}{2}-\beta +  (i-1)\frac{\eta}{2} \right)  \right\} \label{bound:delta:time:convol:pmp1:and:L:deriv:H} \\
& \quad \times  \left\{ \frac{|s_1-s_2|^{\beta}}{(t-s_1)^{\frac{1+n -\eta}{2}+\beta}} g(c(t-s_1), z-x) + \frac{|s_1-s_2|^{\beta}}{(t-s_2)^{\frac{1+n -\eta}{2}+\beta}} g(c(t-s_2), z-x) \right\}. \nonumber
\end{align}

We now consider the quantity $\Delta_{s_1, s_2} [\partial^{n}_v [\partial_\mu \widehat{p}_{m+1}] \otimes \Phi_{m+1}](\mu, s, t, x, z)$ and use the following decomposition
$$
\Delta_{s_1, s_2} [\partial^{n}_v [\partial_\mu \widehat{p}_{m+1}] \otimes \Phi_{m+1}](\mu, s, t, x, z)(v) = {\rm I} + {\rm II} + {\rm III},
$$

\noindent with
\begin{align*}
{\rm I} & := \int_{s_1 \vee s_2}^t \int_{\mathbb{R}^d} \Delta_{s_1, s_2} \partial^{n}_v [\partial_\mu \widehat{p}_{m+1}(\mu, s, r, x, y)](v) \Phi_{m+1}(\mu, s_1\vee s_2, r, t, y, z) \, dy \, dr, 
\end{align*}

\begin{align*}
{\rm II} & := \int_{s_1 \vee s_2}^t \int_{\mathbb{R}^d}   \partial^{n}_v[\partial_\mu \widehat{p}_{m+1}(\mu, s_1 \wedge s_2, r, x, y)](v) \Delta_{s_1, s_2}\Phi_{m+1}(\mu, s, r, t, y, z) \, dy \, dr
\end{align*}
\noindent and
\begin{align*}
{\rm III} & := - \int_{s_1 \wedge s_2}^{s_1 \vee s_2} \int_{\mathbb{R}^d}    \partial^{n}_v[\partial_\mu  \widehat{p}_{m+1}(\mu, s_1 \wedge s_2, r, x, y)](v) \Phi_{m+1}(\mu, s_1\wedge s_2, r, t, y, z) \, dy \, dr.
\end{align*}

From \eqref{diff:time:L:deriv:p:hat:same:time}, \eqref{regularity:time:estimate:v1:v2:decoupling:mckean} at step $m$, \eqref{Gaussian:estimate:Phim}, the space-time inequality \eqref{space:time:inequality} and using the fact that $r\mapsto \mathscr{C}^{n, 2\beta}_{m}(C^{+}_\beta, r-s_1\vee s_2)$ is non-decreasing, we derive
\begin{align*}
|{\rm I}| & \leq   K^{+}_\beta \left\{\frac{|s_1-s_2|^{\beta}}{(t-s_1)^{\frac{1+n}{2}+\beta-\eta}} g(c(t-s_1), z-x) + \frac{|s_1-s_2|^{\beta}}{(t-s_2)^{\frac{1+n}{2}+\beta-\eta}} g(c(t-s_2) , z-x) \right\}\\
& \quad + K^{+}_\beta |s_1-s_2|^{\beta} \int_{s_1\vee s_2}^{t} \frac{\mathscr{C}^{n, 2\beta}_{m}( C^{+}_\beta, r-s_1\vee s_2)}{(t-r)^{1-\frac{\eta}{2}}(r-s_1\vee s_2)^{\frac{1+n}{2}+\beta-\eta}} \, dr \, g(c(t-s_1\vee s_2), z-x).
\end{align*}

From \eqref{cross:mes:deriv:p:hat:s:r:t} (with $r=s$) combined with \eqref{first:second:estimate:induction:decoupling:mckean} and the space-time inequality \eqref{space:time:inequality} and using \eqref{gaussian:bound:diff:time:Phim}, one gets
\begin{align*}
 |{\rm II}| & \leq K_\beta |s_1 -s_2|^{\beta} \int_{s_1 \vee s_2}^{t} \frac{1}{(t-r)^{1-\frac{\eta}{2}}(r-s_1 \vee s_2)^{\frac{1+n-\eta}{2}+\beta}} \, dr \, g(c(t-s_1\wedge s_2), z-x) \\
 & \leq K_\beta \frac{|s_1-s_2|^{\beta}}{(t-s_1\wedge s_2)^{\frac{1+n}{2}+\beta-\eta}}  g(c(t-s_1\wedge s_2), z-x) 
\end{align*}

\noindent where we used the fact that $t-s_1 \wedge s_2 \leq 2 (t-s_1 \vee s_2)$ for the last inequality.
\noindent Finally, using again \eqref{cross:mes:deriv:p:hat:s:r:t} (with $r=s$), \eqref{first:second:estimate:induction:decoupling:mckean}, \eqref{Gaussian:estimate:Phim} as well as the fact that $t-s_1 \wedge s_2 \leq 2 (t-s_1 \vee s_2)$, one obtains
\begin{align*}
|{\rm III}| & \leq K \int_{s_1 \wedge s_2}^{s_1 \vee s_2} \frac{1}{(t-r)^{1-\frac{\eta}{2}}(r-s_1\wedge s_2)^{\frac{1+n-\eta}{2}}} \, dr \, g(c(t-s_1\wedge s_2), z-x)\\
& \leq K \frac{|s_1-s_2|^{\frac{1-n+\eta}{2}}}{(t-s_1\vee s_2)^{1-\frac{\eta}{2}}}  \, g(c(t-s_1\wedge s_2), z-x) \\
& \leq K \frac{|s_1-s_2|^{\beta}}{(t-s_1\vee s_2)^{\frac{1+n}{2}+\beta-\eta}}  \, g(c(t-s_1\wedge s_2), z-x) \\
& \leq K \frac{|s_1-s_2|^{\beta}}{(t-s_1\wedge s_2)^{\frac{1+n}{2}+\beta-\eta}} \,  g(c(t-s_1\wedge s_2), z-x) 
\end{align*}

\noindent for any $\beta \in [0, (1+\eta)/2]$ if $n=0$ and any $\beta \in [0, \eta/2]$ if $n=1$.

Gathering the previous estimates eventually yields
\begin{align}
|\Delta_{s_1, s_2} [\partial^{n}_v \partial_\mu[\widehat{p}_{m+1}] & \otimes \Phi_{m+1}](\mu, s, t, x, z)(v) | \nonumber  \\
& \leq  K^{+}_\beta \left\{\frac{|s_1-s_2|^{\beta}}{(t-s_1)^{\frac{1+n}{2}+\beta-\eta}} g(c(t-s_1), z-x) + \frac{|s_1-s_2|^{\beta}}{(t-s_2)^{\frac{1+n}{2}+\beta-\eta}} g(c(t-s_2) , z-x) \right\} \nonumber\\ 
& \quad + K^{+}_\beta |s_1-s_2|^{\beta} \int_{s_1\vee s_2}^{t} \frac{\mathscr{C}^{n, 2\beta}_{m}( C^{+}_\beta, r-s_1\vee s_2)}{(t-r)^{1-\frac{\eta}{2}}(r-s_1\vee s_2)^{\frac{1+n}{2}+\beta-\eta}} \, dr \, g(c(t-s_1\vee s_2), z-x) \nonumber \\
& \leq  K^{+}_\beta  \left\{ 1 + \sum_{k=1}^{m} (C^{+}_\beta)^{k} (t-s_1 \vee s_2)^{k \frac{\eta}{2}}  \prod_{i=1}^{k} B\left(\frac{\eta}{2}, \frac{1-n+\eta}{2}-\beta +  (i-1)\frac{\eta}{2} \right)  \right\} \label{bound:diff:time:convol:L:deriv:phat:convol:Phi} \\
& \quad \times  \left\{ \frac{|s_1-s_2|^{\beta}}{(t-s_1)^{\frac{1+n -\eta}{2}+\beta}} g(c(t-s_1), z-x) + \frac{|s_1-s_2|^{\beta}}{(t-s_2)^{\frac{1+n -\eta}{2}+\beta}} g(c(t-s_2), z-x) \right\}. \nonumber
\end{align}

For the last term, namely $\Delta_{s_1, s_2} ((p_{m+1}\otimes \partial^{n}_v[\partial_\mu \mH_{m+1}])\otimes \Phi_{m+1})(\mu, s, t, x, z)(v)$, as previously done, we decompose it as the sum of the three following terms 
\begin{align*}
{\rm I} & := \int_{s_1 \vee s_2}^{t} \int_{\mathbb{R}^d}  \Delta_{s_1, s_2} (p_{m+1}\otimes \partial^{n}_v[\partial_\mu \mH_{m+1}])(\mu, s, r, x, y)\, \Phi_{m+1}(\mu, s_1\vee s_2 , r, t, y, z)\, dy \, dr,\\
{\rm II} & := \int_{s_1 \vee s_2}^{t}   \int_{\mathbb{R}^d} (p_{m+1}\otimes \partial^{n}_v[\partial_\mu \mH_{m+1}])(\mu, s_1 \wedge s_2, r, x, y) \, \Delta_{s_1, s_2} \Phi_{m+1}(\mu, s , r, t, y, z)\, dy \, dr,\\
{\rm III} & := -\int_{s_1 \wedge s_2}^{s_1 \vee s_2}  \int_{\mathbb{R}^d} (p_{m+1}\otimes \partial^{n}_v[\partial_\mu \mH_{m+1}])(\mu, s_1 \wedge s_2, r, x, y) \, \Phi_{m+1}(\mu, s_1 \wedge s_2 , r, t, y, z)\, dy \, dr.
\end{align*}

We deal with ${\rm I}$ by using \eqref{bound:delta:time:convol:pmp1:and:L:deriv:H} and \eqref{Gaussian:estimate:Phim}. In order to deal with ${\rm II}$, we first remark that the estimates \eqref{bound:derivative:heat:kernel} and \eqref{estimate:uniform:L:deriv:parametrix:kernel} yield 
\begin{equation}
\label{estimate:convol:pmp1:L:deriv:parametrix:kernel}
|(p_{m+1}\otimes \partial^{n}_v[\partial_\mu \mH_{m+1}])(\mu, s, t, x, y)| \leq \frac{K}{(t-s)^{\frac{1+n-\eta}{2}}} \, g(c(t-s), z-x)
\end{equation}
\noindent and then use \eqref{gaussian:bound:diff:time:Phim}. We finally deal with ${\rm III}$ by using \eqref{estimate:convol:pmp1:L:deriv:parametrix:kernel} and \eqref{Gaussian:estimate:Phim} as well as the fact that $t-s_1 \wedge s_2 \leq 2 (t-s_1 \vee s_2)$. Skipping some technical but standard computations, we obtain
\begin{align}
|\Delta_{s_1, s_2} ((p_{m+1}&\otimes \partial^{n}_v[\partial_\mu \mH_{m+1}])\otimes \Phi_{m+1})(\mu, s, t, x, z)(v)|\nonumber \\
&  \leq  K^{+}_\beta  \left\{ 1 + \sum_{k=1}^{m} (C^{+}_\beta)^{k} (t-s_1 \vee s_2)^{k \frac{\eta}{2}}  \prod_{i=1}^{k} B\left(\frac{\eta}{2}, \frac{1-n+\eta}{2}-\beta +  (i-1)\frac{\eta}{2} \right)  \right\} \label{bound:diff:time:convol:p:and:mu:deriv:H:convol:Phi} \\
& \quad \times  \left\{ \frac{|s_1-s_2|^{\beta}}{(t-s_1)^{\frac{1+n -\eta}{2}+\beta}} g(c(t-s_1), z-x) + \frac{|s_1-s_2|^{\beta}}{(t-s_2)^{\frac{1+n -\eta}{2}+\beta}} g(c(t-s_2), z-x) \right\}. \nonumber
\end{align}

Gathering the estimates \eqref{diff:time:phat:mp1:recursive:bound}, \eqref{bound:delta:time:convol:pmp1:and:L:deriv:H}, \eqref{bound:diff:time:convol:L:deriv:phat:convol:Phi} and \eqref{bound:diff:time:convol:p:and:mu:deriv:H:convol:Phi} eventually completes the proof of \eqref{estimate:deriv:mes:holder:reg:time:dens:stepmp1:cor}.

\end{proof}

\section{Proofs of the technical results}\label{appendix}
\subsection{Proof of Lemma \ref{lem:diff:and:control:deriv:coeff}.\\}\label{section:proof:lem:diff:and:control:deriv:coeff}

\noindent \emph{Step 1: smoothness of the maps $(s, \mu) \mapsto b_i(t, x, [X^{s, \xi, (m)}_t]), \, a_{i, j}(t, x, [X^{s, \xi, (m)}_t])$.}\\

We apply Proposition \ref{structural:class} to the density function $(s, x, \mu) \mapsto p_m(\mu, s, t, x, z) \in \mathcal{C}^{1, 2, 2}([0,t)\times \rr^d \times \pp) $ and to both maps $h(.) = b_i(t, x, .)$ and $h(.) = a_{i, j}(t, x, .)$ respectively. Note that from the estimates \eqref{first:second:estimate:induction:decoupling:mckean}, \eqref{time:derivative:induction:decoupling:mckean} and \eqref{bound:derivative:heat:kernel}, the map $[0,t) \times \mathbb{R}^d \times \pp \ni (s, x, \mu) \mapsto p_m(\mu, s, t, x, z)$ satisfies the conditions of Proposition \ref{structural:class}, in particular the condition \eqref{integrability:condition} is satisfied for any fixed positive integer $m$. We thus deduce that $(s, \mu) \mapsto  b_i (t, x, [X^{s,\xi, (m)}_t]), \, a_{i, j}(t, x, [X^{s,\xi, (m)}_t]) \in \mathcal{C}^{1,2}([0,t) \times \pp)$ with derivatives satisfying 
\begin{align}
& \partial^{n}_v [\partial_\mu [b_i(t, x, [X^{s,\xi, (m)}_t])]](v) \notag \\
& =  \int_{\mathbb{R}^d} \bigg(\frac{\delta b_i}{\delta m}(t, x, [X^{s,\xi, (m)}_t])(y) - \frac{\delta b_i}{\delta m}(t, x, [X^{s,\xi, (m)}_t])(v)\bigg)  \, \partial^{1+n}_x p_{m}(\mu, s, t, v, y) \, dy \notag \\
& \quad \quad + \int_{(\mathbb{R}^d)^2} \bigg(\frac{\delta b_i}{\delta m}(t, x, [X^{s,\xi, (m)}_t])(y) - \frac{\delta b_i}{\delta m}(t, x, [X^{s,\xi, (m)}_t])(x')\bigg) \, \partial^{n}_v[\partial_\mu p_{m}(\mu, s, t, x', y)](v) \, \mu(dx')\, dy, \label{dec:cross:deriv:b}
\end{align}
\begin{align}
\partial_s & [b_i(t, x, [X^{s,\xi, (m)}_t])] \notag \\
&  = \int_{(\mathbb{R}^d)^2} \bigg( \frac{\delta b_i }{\delta m}(t, x, [X^{s,\xi, (m)}_t])(y) - \frac{\delta b_i}{\delta m}(t, x, [X^{s,\xi, (m)}_t])(x') \bigg) \, \partial_s p_{m}(\mu, s, t, x', y) \, \mu(dx')\, dy, \label{time:deriv:bi}
\end{align}
\begin{align}
& \partial^{n}_v [\partial_\mu [a_{i, j}(t, x, [X^{s,\xi, (m)}_t])]](v) \notag \\
& =  \int_{\mathbb{R}^d} \bigg(\frac{\delta a_{i, j}}{\delta m}(t, x, [X^{s,\xi, (m)}_t])(y) - \frac{\delta a_{i, j}}{\delta m}(t, x, [X^{s,\xi, (m)}_t])(v)\bigg)  \, \partial^{1+n}_x p_{m}(\mu, s, t, v, y) \, dy \nonumber \\
&  \quad \quad + \int_{(\mathbb{R}^d)^2} \bigg(\frac{\delta a_{i, j}}{\delta m}(t, x, [X^{s,\xi, (m)}_t])(y) - \frac{\delta a_{i, j}}{\delta m}(t, x,  [X^{s,\xi, (m)}_t])(x') \bigg) \, \partial^{n}_v[\partial_\mu p_{m}(\mu, s, t, x', y)](v) \, \mu(dx')\, dy, \label{dec:cross:deriv:a}
\end{align}
\begin{align}
\partial_s & [a_{i, j}(t, x, [X^{s,\xi, (m)}_t])] \notag \\
&= \int_{(\mathbb{R}^d)^2} \bigg(\frac{\delta a_{i, j}}{\delta m}(t, x, [X^{s,\xi, (m)}_t])(y) - \frac{\delta a_{i, j}}{\delta m} (t, x, [X^{s,\xi, (m)}_t])(x') \bigg) \, \partial_s p_{m}(\mu, s, t, x', y) \, \mu(dx') \,  dy, \label{time:deriv:aij}
\end{align}
\noindent and
\begin{align}
\partial_s & [a_{i, j}(t, x, [X^{s,\xi, (m)}_t]) -   a_{i, j}(t, z, [X^{s,\xi, (m)}_t])] \notag \\
& = \int_{(\rr^d)^2} \left\{ \bigg(\frac{\delta a_{i, j}}{\delta m}(t, x, [X^{s,\xi, (m)}_t])(y) - \frac{\delta a_{i, j}}{\delta m} (t, x, [X^{s,\xi, (m)}_t])(x') \bigg) \right.  \label{time:deriv::holder:a} \\
& \quad \left. - \bigg(\frac{\delta a_{i, j}}{\delta m}(t, z, [X^{s,\xi, (m)}_t])(y) - \frac{\delta a_{i, j}}{\delta m} (t, z, [X^{s,\xi, (m)}_t])(x') \bigg) \right\} \, \partial_s p_{m}(\mu, s, t, x', y)  \, \mu(dx')\, dy. \notag
\end{align}

\noindent Now, from the preceding identities and employing similar arguments as those used in step 1 of the proof of Proposition \ref{structural:class}, we eventually deduce that the maps $[0,t) \times \pp \ni (s, \mu) \mapsto \partial_s [b_i (t, x, [X^{s,\xi, (m)}_t])]$, \, $\partial_s[a_{i, j}(t, x, [X^{s, \xi, (m)}_t])]$ and $ [0,t) \times \pp \times \mathbb{R}^d \ni (s, \mu, v)\mapsto \partial^{n}_v [\partial_\mu [b_i(t, x, [X^{s,\xi, (m)}_t])]](v)$, \, $ \partial^{n}_v [\partial_\mu [a_{i, j}(t, x, [X^{s,\xi, (m)}_t])]](v)$, $n \in \left\{0,1\right\}$, are continuous. \\

\noindent \emph{Step 2: proofs of \eqref{recursive:bound:deriv:a:or:b}, \eqref{recursive:bound:deriv:mes:reg:holder:a:or:b}, \eqref{recursive:bound:deriv:mes:holder:reg:a}, \eqref{recursive:bound:deriv:mes:double:reg:holder:a}, \eqref{recursive:bound:time:deriv:a:or:b} and \eqref{recursive:bound:time:deriv::holder:reg:a}.}\\

Combining the $\eta$-H\"older regularity of $[\delta b_{i}/\delta m](t, x, m)(.), [\delta a_{i, j} / \delta m](t, x, m)(.)$, uniformly with respect to the variables $t, x, m$, with the identities \eqref{dec:cross:deriv:b}, \eqref{dec:cross:deriv:a}, the estimates \eqref{bound:derivative:heat:kernel} as well as the space-time inequality \eqref{space:time:inequality}, we obtain
\begin{align*}
| \partial^n_v & [\partial_\mu [b_i(t, x, [X^{s,\xi, (m)}_t])]](v) |  + |\partial^n_v [\partial_\mu [a_{i, j}(t, x, [X^{s,\xi, (m)}_t])]](v)|  \\
&  \leq K \left\{ \frac{1}{(t-s)^{\frac{1+n - \eta}{2}}} + \int_{(\mathbb{R}^d)^2} (|y-x'|^{\eta}\wedge 1) | \partial^{n}_v[\partial_\mu p_{m}(\mu, s, t, x', y)](v)| \, \mu(dx')\, dy  \right\} 
\end{align*}
\noindent for $n \in \left\{0, 1\right\}$. This concludes the proof of \eqref{recursive:bound:deriv:a:or:b}.

We now prove the estimate \eqref{recursive:bound:deriv:mes:reg:holder:a:or:b} for the difference $\partial_v [\partial_\mu [a_{i, j}(t, x, [X^{s,\xi, (m)}_t])]](v) - \partial_v [\partial_\mu [a_{i, j}(t, x, [X^{s,\xi, (m)}_t])]](v')$ as completely analogous arguments apply for $\partial_v [\partial_\mu [b_{i}(t, x, [X^{s,\xi, (m)}_t])]](v) - \partial_v[\partial_\mu [b_{i}(t, x, [X^{s,\xi, (m)}_t])]](v')$. From \eqref{dec:cross:deriv:a}, it holds
\begin{align}
& \partial_v [\partial_\mu [a_{i, j}(t, x, [X^{s,\xi, (m)}_t])]](v) - \partial_v[\partial_\mu [a_{i, j}(t, x, [X^{s,\xi, (m)}_t])]](v') \notag \\
& =  \int_{\mathbb{R}^d} \frac{\delta a_{i, j}}{\delta m}(t, x, [X^{s,\xi, (m)}_t])(y)  \, [\partial^{2}_x p_{m}(\mu, s, t, v, y)-\partial^{2}_x p_{m}(\mu, s, t, v', y)] \, dy \notag \\
&  \quad \quad + \int_{(\mathbb{R}^d)^2} \bigg(\frac{\delta a_{i, j}}{\delta m}(t, x, [X^{s,\xi, (m)}_t])(y) - \frac{\delta a_{i, j}}{\delta m}(t, x,  [X^{s,\xi, (m)}_t])(x') \bigg) \label{decomp:diff:deriv:cross:mes:reg:holder:aij} \\ 
& \quad \times [\partial_v[\partial_\mu p_{m}(\mu, s, t, x', y)](v) - \partial_v[\partial_\mu p_{m}(\mu, s, t, x', y)](v')]  \, \mu(dx')\, dy. \notag
\end{align}
We estimate the first integral appearing on the right-hand side of the above equality by splitting the computations into the two following cases: $|v - v'|^2 \leq t-s$ and $|v - v'|^2 \geq t-s$. In the first case, we remark that
\begin{align*}
 \int_{\mathbb{R}^d}& \frac{\delta a_{i, j}}{\delta m}(t, x, [X^{s,\xi, (m)}_t])(y)  \, \left[\partial^{2}_x p_{m}(\mu, s, t, v, y)-\partial^{2}_x p_{m}(\mu, s, t, v', y)\right] \, dy\\
 & =  \int_{\mathbb{R}^d} \left[\frac{\delta a_{i, j}}{\delta m}(t, x, [X^{s,\xi, (m)}_t])(y) - \frac{\delta a_{i, j}}{\delta m}(t, x, [X^{s,\xi, (m)}_t])(v)\right]    \, \left[\partial^{2}_x p_{m}(\mu, s, t, v, y)-\partial^{2}_x p_{m}(\mu, s, t, v', y)\right] \, dy\end{align*}

\noindent so that, from \eqref{reg:heat:kernel:deriv} with $n=2$, $\beta \in [0,\eta)$ and using the uniform $\eta$-H\"older regularity of the map $v\mapsto [\delta a_{i, j}/\delta m](t, x, m)(v)$ as well as the space-time inequality \eqref{space:time:inequality} and the fact that $|y-v|^\eta \leq |y-v'|^\eta + |v-v'|^\eta \leq |y-v'|^\eta + (t-s)^{\eta/2}$, we get 
\begin{align}
\Big|\int_{\mathbb{R}^d} & \frac{\delta a_{i, j}}{\delta m}(t, x, [X^{s,\xi, (m)}_t])(y)  \, [\partial^{2}_x p_{m}(\mu, s, t, v, y)-\partial^{2}_x p_{m}(\mu, s, t, v', y)] \, dy\Big| \nonumber \\
& \leq K_\beta |v-v'|^{\beta} \int_{\mathbb{R}^d} \frac{|y-v|^{\eta} }{(t-s)^{1+\frac{\beta}{2}}} \left\{ g(c(t-s), y-v) + g(c(t-s), y-v')  \right\} \, dy \nonumber \\
& \leq K_\beta \frac{|v-v'|^{\beta}}{(t-s)^{1+ \frac{\beta-\eta}{2}}}. \label{ineq:sec:deriv:tilde:reg:holder:function}
\end{align}

\noindent Otherwise, if $|v-v'|^2 \geq t-s$, we rather write
\begin{align*}
 \int & \frac{\delta a_{i, j}}{\delta m}(t, x, [X^{s,\xi, (m)}_t])(y)  [\partial^2_x p_{m}(\mu, s, t, v, y) - \partial^2_x p_{m}(\mu, s, t , v', y)] \, dy \\
 & =  \int_{\mathbb{R}^d} \left[\frac{\delta a_{i, j}}{\delta m}(t, x, [X^{s,\xi, (m)}_t])(y) - \frac{\delta a}{\delta m}(t, x, [X^{s, \xi, (m)}_t])(v)\right] \partial^2_x p_{m}(\mu, s, t, v, y) \, dy\\
 & \quad - \int_{\mathbb{R}^d} \left[\frac{\delta a_{i, j}}{\delta m}(t, x, [X^{s,\xi, (m)}_t])(y) - \frac{\delta a}{\delta m}(t, x, [X^{s, \xi, (m)}_t])(v')\right] \partial^2_x p_{m}(\mu, s, t , v', y)] \, dy 
\end{align*}

\noindent and combine the uniform $\eta$-H\"older regularity of $ [\delta a_{i, j}/\delta m](t, x, m)(.)$ with \eqref{bound:derivative:heat:kernel} for $n=2$ as well as the space-time inequality \eqref{space:time:inequality}. This yields \eqref{ineq:sec:deriv:tilde:reg:holder:function}. We handle the second integral appearing on the right-hand side of \eqref{decomp:diff:deriv:cross:mes:reg:holder:aij} using again the boundedness and uniform $\eta$-H\"older regularity of $[\delta a_{i, j}/\delta m](t, x, m)(.)$. Gathering the previous estimates concludes the proof of \eqref{recursive:bound:deriv:mes:reg:holder:a:or:b}.

We now prove \eqref{recursive:bound:deriv:mes:holder:reg:a}. From \eqref{dec:cross:deriv:a} we directly obtain
\begin{align}
 \partial^{n}_v [\partial_\mu & [a_{i, j}(t, x, [X^{s, \xi, (m)}_{t}])  - a_{i, j}(t, z, [X^{s, \xi, (m)}_t])]](v)\nonumber \\
 &=  \int_{(\mathbb{R}^d)^2} \Big\{ \big(\frac{\delta a_{i, j}}{\delta m}  (t, x, [X^{s, \xi, (m)}_{t}])(y) - \frac{\delta a_{i, j}}{\delta m} (t, x, [X^{s, \xi, (m)}_t])(x')\big)\nonumber \\
  &\quad  - \big(\frac{\delta a_{i, j}}{\delta m}  (t, z, [X^{s, \xi, (m)}_{t}])(y)  - \frac{\delta a_{i, j}}{\delta m} (t, z, [X^{s, \xi, (m)}_t])(x')\big)\Big\} 
  \partial^{n}_v\Big[\partial_{\mu}p_{m}(\mu, s , t, x', y)\Big](v) \, \mu(dx') \, dy  \label{eq:decompJ}\\
 & + \int_{\mathbb{R}^d}  \Big\{\big(\frac{\delta a_{i, j}}{\delta m}  (t, x, [X^{s, \xi, (m)}_{t}])(y) - \frac{\delta a_{i, j}}{\delta m} (t, x, [X^{s, \xi, (m)}_t])(v)\big)\notag\\
 & \quad - \big(\frac{\delta a_{i, j}}{\delta m}  (t, z, [X^{s, \xi, (m)}_{t}])(y) - \frac{\delta a_{i, j}}{\delta m} (t, z, [X^{s, \xi, (m)}_t])(v)\big) \Big\}  \partial^{1+n}_x p_{m}(\mu, s, t , v, y) \, dy  \notag\\
 & =: {\rm J}^{n,1}_{i, j}(v) + {\rm J}^{n, 2}_{i, j}(v). \label{decomposition:Jij}
\end{align}

Now, observe that the uniform $\eta$-H\"older regularity of the map $ [\delta a_{i, j}/ \delta m](t, ., m)(.)$ gives
\begin{align}
\left|{\rm J}^{n,1}_{i, j}(v)\right|  &\leq  K \int_{(\mathbb{R}^d)^2} (|y-x'|^{\eta}\wedge |z-x|^\eta \wedge 1) |\partial^{n}_v[\partial_\mu p_{m}(\mu, s, t, x', y)](v)| \, \mu(dx')\, dy  \notag \\
& \leq K  |z-x|^{\beta \eta} \int_{(\mathbb{R}^d)^2} (|y-x'|^{(1-\beta)\eta} \wedge 1) |\partial^{n}_v[\partial_\mu p_{m}(\mu, s, t, x', y)](v)| \,  \mu(dx')\, dy \label{eq:estiJ:1}
\end{align}
\noindent and similarly combining also \eqref{bound:derivative:heat:kernel} with the space-time inequality \eqref{space:time:inequality}
\begin{align}
\left|{\rm J}^{n,2}_{i, j}(v)\right|  &\leq  K  |z-x|^{\beta \eta}  \int_{(\mathbb{R}^d)^2} (|y-v|^{(1-\beta)\eta}\wedge 1) | \partial^{1+n}_x p_{m}(\mu, s, t , v, y)| \, dy \leq K_\beta \frac{ |z-x|^{\beta \eta} }{(t-s)^{\frac{1+n-(1-\beta)\eta}{2}}} \label{eq:estiJ:2}
\end{align}
\noindent for any $\beta \in [0,1]$. Gathering \eqref{eq:estiJ:1} and \eqref{eq:estiJ:2} completes the proof of \eqref{recursive:bound:deriv:mes:holder:reg:a}. 

We now establish the estimate \eqref{recursive:bound:deriv:mes:double:reg:holder:a}. According to the notations previously introduced, we have to deal with the two terms ${\rm J}^{1,1}_{i, j}(v)-{\rm J}^{1,1}_{i, j}(v')$ and ${\rm J}^{1,2}_{i, j}(v)-{\rm J}^{1,2}_{i, j}(v')$. We first write
\begin{align*}
 {\rm J}^{1, 1}_{i, j}(v)-{\rm J}^{1, 1}_{i, j}(v')& = \int_{(\mathbb{R}^d)^2} \Big\{ \big(\frac{\delta a_{i, j}}{\delta m}  (t, x, [X^{s, \xi, (m)}_{t}])(y') - \frac{\delta a_{i, j}}{\delta m} (t, x, [X^{s, \xi, (m)}_t])(x')\big)\nonumber \\
  &\quad  - \big(\frac{\delta a_{i, j}}{\delta m}  (t, z, [X^{s, \xi, (m)}_{t}])(y') - \frac{\delta a_{i, j}}{\delta m} (t, z, [X^{s, \xi, (m)}_t])(x')\big)\Big\} \\
  & \quad \quad \times [\partial_v\Big[\partial_{\mu}p_{m}(\mu, s , t, x', y')\Big](v) - \partial_v\Big[\partial_{\mu}p_{m}(\mu, s , t, x', y')\Big](v')] \,  \mu(dx')\, dy'.
\end{align*}
 
 On the one hand, the uniform $\eta$-H\"older regularity of the map $[\delta a_{i, j}/\delta m](t, x, m)(.)$ yields
 $$
 | {\rm J}^{1, 1}_{i, j}(v)-{\rm J}^{1, 1}_{i, j}(v') | \leq K  \int_{(\mathbb{R}^d)^2} (|y'-x'|^\eta \wedge 1) \Big|\partial_v\Big[\partial_{\mu}p_{m}(\mu, s , t, x', y')\Big](v) - \partial_v\Big[\partial_{\mu}p_{m}(\mu, s , t, x', y')\Big](v')\Big| \, \mu(dx')\, dy'  \, .
 $$
 \noindent On the other hand, the uniform $\eta$-H\"older regularity of the map $ [\delta a_{i, j}/\delta m](t, ., m)(v)$ implies
 $$
 | {\rm J}^{1, 1}_{i, j}(v)-{\rm J}^{1, 1}_{i, j}(v') | \leq K (|z-x|^\eta\wedge 1)  \int_{(\mathbb{R}^d)^2} \Big|\partial_v\Big[\partial_{\mu}p_{m}(\mu, s , t, x', y')\Big](v) - \partial_v\Big[\partial_{\mu}p_{m}(\mu, s , t, x', y')\Big](v')\Big| \, \mu(dx')\, dy' .
 $$
 
In order to deal with the second term, we write
 \begin{align*}
 {\rm J}^{1,2}_{i, j}(v) - {\rm J}^{1, 2}_{i, j}(v') &= \int_{\mathbb{R}^d}  \Big\{\big(\frac{\delta a_{i, j}}{\delta m}  (t, x, [X^{s, \xi, (m)}_{t}])(y') - \frac{\delta a_{i, j}}{\delta m} (t, x, [X^{s, \xi, (m)}_t])(v)\big)\notag\\
 & \quad - \big(\frac{\delta a_{i, j}}{\delta m}  (t, z, [X^{s, \xi, (m)}_{t}])(y') - \frac{\delta a_{i, j}}{\delta m} (t, z, [X^{s, \xi, (m)}_t])(v)\big) \Big\} \\
 & \quad \quad \times (\partial^{2}_x p_{m}(\mu, s, t , v, y')-\partial^{2}_x p_{m}(\mu, s, t , v', y'))  \, dy'. 
 \end{align*} 

First, from \eqref{reg:heat:kernel:deriv} (with $n=2$) and the uniform $\eta$-H\"older regularity of $[\delta a_{i, j}/\delta m](t, ., m)(v)$, we obtain
$$
| {\rm J}^{1,2}_{i, j}(v) - {\rm J}^{1, 2}_{i, j}(v') | \leq K_\beta \frac{|z-x|^\eta\wedge 1}{(t-s)^{1+\frac{\beta}{2}}} |v-v'|^\beta
$$
\noindent for any $ \beta \in [0,\eta)$. Then, in the diagonal regime $|v-v'|\leq (r-s)^{1/2}$, from \eqref{reg:heat:kernel:deriv} (with $n=2$) and the uniform $\eta$-H\"older regularity of $[\delta a_{i, j}/\delta m](t, x, m)(.)$, for any $\beta \in [0,\eta)$, one has
  $$
 |{\rm J}^{1,2}_{i, j}(v) - {\rm J}^{1, 2}_{i, j}(v')| \leq K_\beta \frac{|v-v'|^\beta}{(t-s)^{1+\frac{\beta-\eta}{2}}}
 $$
 \noindent while in the off-diagonal regime $|v-v'|> (r-s)^{1/2}$, writing
  \begin{align*}
 {\rm J}^{1,2}_{i, j}(v) - {\rm J}^{1, 2}_{i, j}(v') &= \int_{\mathbb{R}^d} \big(\frac{\delta a_{i, j}}{\delta m}  (t, x, [X^{s, \xi, (m)}_{t}])(y') - \frac{\delta a_{i, j}}{\delta m} (t, x, [X^{s, \xi, (m)}_t])(v)\big) \partial^{2}_x p_{m}(\mu, s, t , v, y') \, dy'\notag\\
 & \quad - \int_{\mathbb{R}^d} \big(\frac{\delta a_{i, j}}{\delta m}  (t, z, [X^{s, \xi, (m)}_{t}])(y') - \frac{\delta a_{i, j}}{\delta m} (t, z, [X^{s, \xi, (m)}_t])(v')\big)\partial^{2}_x p_{m}(\mu, s, t , v', y')  \, dy'
 \end{align*} 
 \noindent and using \eqref{bound:derivative:heat:kernel} with $n=2$ as well as the space-time inequality \eqref{space:time:inequality} yield the same estimate as in the diagonal regime. Hence, for any $(v, v')\in (\mathbb{R}^d)^2$ and any $\beta \in [0,\eta)$, it holds
 $$
 |{\rm J}^{1,2}_{i, j}(v) - {\rm J}^{1, 2}_{i, j}(v')| \leq K_\beta \frac{|v-v'|^\beta}{(t-s)^{1+\frac{\beta-\eta}{2}}}.
 $$
 Gathering the above estimates, we thus conclude
 \begin{align*}
 |& {\rm J}^{1}_{i, j}(v)  - {\rm J}^{1}_{i, j}(v')| \\
 & \leq K_\beta (|z-x|^\eta \wedge 1) \left\{ \frac{|v-v'|^\beta}{(t-s)^{1+\frac{\beta}{2}}}  + \int_{(\mathbb{R}^d)^2} \Big|\partial_v\Big[\partial_{\mu}p_{m}(\mu, s , t, x', y')\Big](v) - \partial_v\Big[\partial_{\mu}p_{m}(\mu, s , t, x', y')\Big](v')\Big| \, \mu(dx') \, dy'\right\}
 \end{align*}
 \noindent and
  \begin{align*}
 |& {\rm J}^{1}_{i, j}(v)  - {\rm J}^{1}_{i, j}(v')| \\
 & \leq K_\beta \left\{ \frac{|v-v'|^\beta}{(t-s)^{1+\frac{\beta-\eta}{2}}}  + \int_{(\mathbb{R}^d)^2} (|y'-x'|^\eta \wedge 1) \Big|\partial_v\Big[\partial_{\mu}p_{m}(\mu, s , t, x', y')\Big](v) - \partial_v\Big[\partial_{\mu}p_{m}(\mu, s , t, x', y')\Big](v')\Big| \, \mu(dx') \, dy'\right\}.
 \end{align*}

Again, combining the two previous estimates yields \eqref{recursive:bound:deriv:mes:double:reg:holder:a}.

We now prove \eqref{recursive:bound:time:deriv:a:or:b}. Similarly, from the identities \eqref{time:deriv:bi}, \eqref{time:deriv:aij} and the uniform $\eta$-H\"older regularity of $[\delta b_{i}/\delta m](t, x, m)(.)$, $[\delta a_{i, j}/\delta m](t, x, m)(.)$ one gets
\begin{align*}
|\partial_s [b_i (t, x, [X^{s,\xi, (m)}_t])]| + |\partial_s[a_{i, j}(t, x, [X^{s,\xi, (m)}_t])]| & \leq K \int_{(\mathbb{R}^d)^2} (|y-x'|^\eta \wedge 1) |\partial_s p_m(\mu, s, t, x', y)| \,  \mu(dx')\, dy 
\end{align*}
\noindent for some positive constant $K:=K(T, \HR, \HE)$.

In order to derive \eqref{recursive:bound:time:deriv::holder:reg:a}, we use either the uniform $\eta$-H\"older regularity of the map $ [\delta a_{i, j}/\delta m](t, ., m)(y)$ or the uniform $\eta$-H\"older regularity of the map $ [\delta a_{i, j}/\delta m](t, x, m)(.)$ so that 
\begin{align*}
\Big| \frac{\delta a_{i, j}}{\delta m}(t, x, [X^{s,\xi, (m)}_t])(y) & - \frac{\delta a_{i, j}}{\delta m} (t, x, [X^{s,\xi, (m)}_t])(x') \\
&  - \bigg(\frac{\delta a_{i, j}}{\delta m}(t, z, [X^{s,\xi, (m)}_t])(y) - \frac{\delta a_{i, j}}{\delta m} (t, z, [X^{s,\xi, (m)}_t])(x') \bigg)\Big| \\
& \leq K (|y-x'|^\eta \wedge |x-z|^\eta \wedge 1)\\
& \leq K |x-z|^{\beta' \eta} ( |y-x'|^{(1-\beta')\eta} \wedge1)
\end{align*}
\noindent for any $\beta' \in [0,1]$.
Combining the above upper-estimate with \eqref{time:deriv::holder:a} directly yields \eqref{recursive:bound:time:deriv::holder:reg:a}.

\subsection{Proof of Corollary \ref{cor:deriv:time:and:mes:phat}.\\}\label{section:proof:cor:deriv:time:and:mes:phat}

\emph{Step 1: smoothness of the maps $(s, x, \mu) \mapsto \widehat{p}^{y}_{m+1}(\mu, s, r,  t , x, z)$, $\widehat{p}^{y}_{m+1}(\mu, s,  t , x, z)$.}

\smallskip

Combining Lemma \ref{lem:diff:and:control:deriv:coeff} with the estimates \eqref{time:degeneracy:estimate:deriv:mes:coeff}, \eqref{time:degeneracy:estimate:deriv:time:coeff} and the dominated convergence theorem, we deduce that the maps $(s, \mu) \mapsto \int_r^t a(r', y, [X^{s, \xi, (m)}_{r'}]) \, dr'$, \,  $\int_s^t a(r', y, [X^{s, \xi, (m)}_{r'}]) \, dr'$ belong to $\mathcal{C}^{1, 2}([0,r) \times \pp)$ and $\mathcal{C}^{1, 2}([0,t) \times \pp)$ respectively. Hence, we conclude that the maps $(s, x,  \mu) \mapsto \widehat{p}^{y}_{m+1}(\mu, s, r,  t , x, z)$, $\widehat{p}^{y}_{m+1}(\mu, s,  t , x, z) = \widehat{p}^{y}_{m+1}(\mu, s, s,  t , x, z)$ defined by \eqref{definition:general:phat} belong to $\mathcal{C}^{1, 0, 2}([0,r)\times \pp)$ and $\mathcal{C}^{1, 0, 2}([0,t)\times \pp)$ respectively with continuous derivatives with respect to the variables $s$, $x$, $\mu$ and $v$. We now prove the announced pointwise Gaussian estimates.\\

\noindent \emph{Step 2: proofs of the related pointwise Gaussian estimates on the derivatives.}

\smallskip

From \eqref{definition:general:phat}, \eqref{time:degeneracy:estimate:deriv:mes:coeff} and Jacobi's formula, for any $r\in [s, t)$, we derive
\begin{align}
& \partial^n_v[\partial_\mu \widehat{p}^{y}_{m+1}(\mu, s, r, t, x, z)](v)  \notag\\
& = -\frac12\left\{ \tr\left(\left(\int_r^t a(r', y, [X^{s, \xi, (m)}_{r'}])\, dr'\right)^{-1} \int_r^t \partial^n_v[\partial_\mu [a(r', y, [X^{s, \xi, (m)}_{r'}])]](v) \, dr' \right) \right. \notag\\
& \quad \left. - (z-x)^{t} \left(\int_r^t a(r', y, [X^{s, \xi, (m)}_{r'}])\, dr'\right)^{-1} \int_r^t \partial^n_v[\partial_\mu [a(r', y, [X^{s, \xi, (m)}_{r'}])]](v) \, dr'  \right. \label{representation:formula:deriv:mes:p:hat}\\ 
& \quad \left. \quad \times \left(\int_r^t a(r', y, [X^{s, \xi, (m)}_{r'}])\, dr'\right)^{-1} (z-x) \right\} \widehat{p}^{y}_{m+1}(\mu, s, r, t, x, z) \notag
\end{align}
\noindent where
\begin{align*}
& \tr\left(\left(\int_r^t a(r', y, [X^{s, \xi, (m)}_{r'}])\, dr'\right)^{-1}  \int_r^t  \partial^n_v[\partial_\mu [a(r', y, [X^{s, \xi, (m)}_{r'}])]](v) \, dr' \right)  \\
& = \sum_{k, \ell =1}^{d} \Big[\left(\int_r^t a(r', y, [X^{s, \xi, (m)}_{r'}])\, dr'\right)^{-1}\Big]_{k, \ell}  \int_r^t \partial^n_v[\partial_\mu [a_{\ell, k}(r', y, [X^{s, \xi, (m)}_{r'}])]](v) \, dr'
\end{align*}
\noindent and
\begin{align*}
& (z-x)^{t} \left(\int_r^t a(r', y, [X^{s, \xi, (m)}_{r'}])\, dr'\right)^{-1} \int_r^t  \partial^n_v[\partial_\mu [a(r', y, [X^{s, \xi, (m)}_{r'}])]](v) \, dr'  \left(\int_r^t a(r', y, [X^{s, \xi, (m)}_{r'}])\, dr'\right)^{-1} (z-x) \\
& = \sum_{k, \ell=1}^d \Big[(z-x)^{t} \left(\int_r^t a(r', y, [X^{s, \xi, (m)}_{r'}])\, dr'\right)^{-1}\Big]_{k} \int_r^t  \partial^n_v[\partial_\mu [a_{k, \ell}(r', y, [X^{s, \xi, (m)}_{r'}])]](v) \, dr' \\
& \quad \quad \times\Big[\left(\int_r^t a(r', y, [X^{s, \xi, (m)}_{r'}])\, dr'\right)^{-1} (z-x)\Big]_{\ell}.
\end{align*}
The preceding identity combined with \HE, the space-time inequality \eqref{space:time:inequality} and then \eqref{recursive:bound:deriv:a:or:b} yields
\begin{align*}
& | \partial^n_v[\partial_\mu \widehat{p}^{y}_{m+1}(\mu, s, r, t, x, z)](v)|\\
&  \leq \frac{K}{t-r} \int_r^t  \max_{k, \ell}|\partial^n_v[\partial_\mu [a_{k, \ell}(r', y, [X^{s, \xi, (m)}_{r'}])]](v)| \, dr' \, g(c(t-r), z-x)\\
& \leq \frac{K}{t-r} \left\{ \int_r^t \frac{1}{(r'-s)^{\frac{1+n - \eta}{2}}} \, dr' +  \int_r^t  \int_{(\mathbb{R}^d)^2} (|y'-x'|^{\eta}\wedge 1) | \partial^{n}_v[\partial_\mu p_{m}(\mu, s, r', x', y')](v)| \, \mu(dx') \, dy' \, dr' \right\}\\
 & \quad \quad \times g(c(t-r), z-x).
\end{align*}
This concludes the proof of \eqref{cross:mes:deriv:p:hat:s:r:t}. 
In order to get \eqref{time:deriv:p:hat:s:r:t}, we employ similar lines of reasonings, namely, from \eqref{definition:general:phat}, \eqref{time:degeneracy:estimate:deriv:time:coeff} and Jacobi's formula,
\begin{align}
& \partial_s \widehat{p}^{y}_{m+1}(\mu, s, r, t, x, z)  \notag\\
& = -\frac12\left\{ \tr\left(\left(\int_r^t a(r', y, [X^{s, \xi, (m)}_{r'}])\, dr'\right)^{-1} \int_r^t \partial_s [a(r', y, [X^{s, \xi, (m)}_{r'}])] \, dr' \right) \right. \notag\\
& \quad \left. - (z-x)^{t} \left(\int_r^t a(r', y, [X^{s, \xi, (m)}_{r'}])\, dr'\right)^{-1} \int_r^t \partial_s [a(r', y, [X^{s, \xi, (m)}_{r'}])] \, dr'  \right. \label{representation:formula:deriv:time:p:hat:s:r:t}\\ 
& \quad \left. \quad \times \left(\int_r^t a(r', y, [X^{s, \xi, (m)}_{r'}])\, dr'\right)^{-1} (z-x) \right\} \widehat{p}^{y}_{m+1}(\mu, s, r, t, x, z) \notag
\end{align}
\noindent with $ \partial_s [a(r', y, [X^{s, \xi, (m)}_{r'}]]) = ( \partial_s [a_{k, \ell}(r', y, [X^{s, \xi, (m)}_{r'}])])_{1\leq k, \ell \leq d}$, which in turn by \HE, the space-time inequality \eqref{space:time:inequality} and then \eqref{recursive:bound:time:deriv:a:or:b} yield the Gaussian estimate \eqref{time:deriv:p:hat:s:r:t}.
We now remark that the derivative of $s\mapsto \widehat{p}^{y}_{m+1}(\mu, s, t, x, z)= \widehat{p}^{y}_{m+1}(\mu, s, s, t, x, z)$ satisfies the relation
\begin{align}
& \partial_s \widehat{p}^{y}_{m+1}(\mu, s, t, x, z)  \notag\\
& = - \frac12 \sum_{k, \ell=1}^d a_{k,\ell }(s, y, \mu) H^{k, \ell}_{2}\left(\int_s^t a(r', y, [X^{s, \xi ,(m)}_{r'}]) \, dr', z-x\right) \,  \widehat{p}^{y}_{m+1}(\mu, s, t, x, z)\notag \\
& \quad -\frac12\left\{ \tr\left(\left(\int_s^t a(r', y, [X^{s, \xi, (m)}_{r'}])\, dr'\right)^{-1} \int_s^t \partial_s [a(r', y, [X^{s, \xi, (m)}_{r'}])] \, dr' \right) \right. \notag\\
& \quad \quad \left. - (z-x)^{t} \left(\int_s^t a(r', y, [X^{s, \xi, (m)}_{r'}])\, dr'\right)^{-1} \int_s^t \partial_s [a(r', y, [X^{s, \xi, (m)}_{r'}])] \, dr'  \right. \label{representation:formula:deriv:time:p:hat:s:t}\\ 
& \quad \quad \quad \left. \quad \times \left(\int_s^t a(r', y, [X^{s, \xi, (m)}_{r'}])\, dr'\right)^{-1} (z-x) \right\} \widehat{p}^{y}_{m+1}(\mu, s, t, x, z) \notag
\end{align}
\noindent which in turn, again by \HE, the space-time inequality \eqref{space:time:inequality} and then \eqref{recursive:bound:time:deriv:a:or:b} yield the Gaussian estimate \eqref{time:deriv:p:hat:s:t}. We now prove \eqref{cross:mes:deriv:holder:p:hat:s:t}. If $|x_1-x_2| \geq (t-s)^{1/2}$, then the announced estimate easily follows from \eqref{cross:mes:deriv:p:hat:s:r:t} (with $r=s$). Assuming now that $|x_1-x_2| < (t-s)^{1/2}$, from the mean-value theorem and the space-time inequality \eqref{space:time:inequality}
\begin{align*}
| \widehat{p}^{y}_{m+1}(\mu, s, t, x_1, z) & - \widehat{p}^{y}_{m+1}(\mu, s, t, x_2, z) | \\
& \leq K \frac{|x_1-x_2|}{(t-s)^{\frac12}}  \left\{ g(c(t-s), z-x_1) + g(c(t-s), z-x_2) \right\}\\
& \leq K \frac{|x_1-x_2|^{\beta}}{(t-s)^{\frac{\beta}{2}}}  \left\{ g(c(t-s), z-x_1) + g(c(t-s), z-x_2) \right\}
\end{align*}
\noindent which combined with the identity \eqref{representation:formula:deriv:mes:p:hat} for $r=s$, \eqref{recursive:bound:deriv:a:or:b} yields \eqref{cross:mes:deriv:holder:p:hat:s:t}.
 
 In order to prove \eqref{time:deriv:holder:reg:p:hat:s:t}, one may assume without loss of generality that $|y_1-y_2| \leq 1$. Starting from the identity \eqref{representation:formula:deriv:time:p:hat:s:t}, one uses \eqref{recursive:bound:time:deriv:a:or:b}, \eqref{recursive:bound:time:deriv::holder:reg:a}, the uniform $\eta$-H\"older regularity of $ a(t, ., m)$, the inequality $| p^{y_1}_{m+1}(\mu, s, t, x, z) - p^{y_2}_{m+1}(\mu, s, t, x, z)| \leq K |y_1- y_2|^\eta g(c(t-s), z-x)$, which stems from the mean-value theorem and again the $\eta$-H\"older regularity of $ a(t, ., m)$, and finally the space-time inequality \eqref{space:time:inequality}. We omit the remaining technical details.
 
We finally prove \eqref{cross:mes:deriv:reg:holder:terminal point:p:hat:s:r:t}. We start again from the identity \eqref{representation:formula:deriv:mes:p:hat} and use \HE, the space-time inequality \eqref{space:time:inequality} and then \eqref{recursive:bound:deriv:mes:reg:holder:a:or:b}. We obtain
\begin{align*}
&  | \partial_v [\partial_\mu \widehat{p}^y_{m+1}(\mu, s, r, t, x, z)](v) - \partial_v [\partial_\mu \widehat{p}^y_{m+1}(\mu, s, r, t, x, z)](v') | \nonumber\\
& \leq  \frac{K}{t-r}  \int_r^t \max_{i, j} | \partial_v[\partial_\mu [a_{i, j}(r', y, [X^{s, \xi, (m)}_{r'}])]](v) - \partial_v[\partial_\mu [a_{i, j}(r', y, [X^{s, \xi, (m)}_{r'}])]](v')| \, dr'   g(c(t-r), z-x)\nonumber \\
& \leq \frac{K_\beta}{t - r}  \left\{ \int_{r}^{t} \Big[ \frac{|v-v'|^{\beta}}{(r'-s)^{1+ \frac{(\beta-\eta)}{2}}}\right. \notag\\
&  \quad  \quad \left.+ \int_{(\mathbb{R}^d)^2} (|y'-x'|^{\eta}\wedge 1) |\partial_v [\partial_\mu p_{m}(\mu, s, r', x', y')](v) - \partial_v [\partial_\mu p_{m}(\mu, s, r', x', y')](v') | \, dy' \, \mu(dx') \Big] dr' \right\}  \\
& \quad \times g(c(t-r), z- x).
\end{align*}

\subsection{Proof of Corollary \ref{cor:deriv:time:and:mes:parametrix:kernel}.\\}\label{section:proof:cor:deriv:time:and:mes:parametrix:kernel}

\emph{Step 1: smoothness of the map $(s, \mu) \mapsto \mH_{m+1}(\mu, s, r,  t , x, z)$.}\\

We apply again Lemma \ref{lem:diff:and:control:deriv:coeff} and make use of the estimates \eqref{time:degeneracy:estimate:deriv:mes:coeff}, \eqref{time:degeneracy:estimate:deriv:time:coeff} together with the dominated convergence theorem to deduce that each term appearing in the expression of $\mH_{m+1}(\mu, s, r, t, x, z)$ given by \eqref{definition:parametrix:kernel:iterate:m} is $\mathcal{C}^{1, 2}([0,r)\times \pp)$ with continuous derivatives with respect to the variables $s$, $x$, $\mu$ and $v$. We thus conclude that the map $(s, \mu) \mapsto \mH_{m+1}(\mu, s, r,  t , x, z) \in \mathcal{C}^{1, 2}([0,r) \times \pp)$ and admits continuous derivatives with respect to the variables $s$, $x$, $\mu$ and $v$.\\

\noindent \emph{Step 2: proof of the pointwise Gaussian estimate \eqref{cross:mes:deriv:parametrix:kernel:s:r:t:with:beta}.}\\

We use the following decomposition
\begin{equation}
\partial^n_v [\partial_\mu \mH_{m+1}(\mu, s, r ,t ,x ,z)](v)  = :{\rm I}^{n}(v)+ {\rm II}^{n}(v)+{\rm III }^{n}(v), \, \quad n \in \left\{0,1\right\}, \label{deriv:mu:H:mp1}
\end{equation}
\noindent with
\begin{align*}
{\rm I }^{n} (v) &: =  \left\{ - \sum_{i=1}^d H^{i}_1\left(\int_r^{t} a(r', z, [X^{s, \xi, (m)}_{r'}]) dr', z-x\right)  \partial^{n}_v [\partial_\mu[b_i(r, x, [X^{s, \xi, (m)}_r])]](v) \right\} \, \widehat{p}_{m+1}(\mu, s, r, t, x, z)  \\
& + \Bigg\{ - \sum_{i=1}^d b_i(r, x, [X^{s, \xi, (m)}_r]) \partial^{n}_v \left[\partial_\mu  H^{i}_1\left(\int_r^{t} a(r',z, [X^{s, \xi, (m)}_{r'}]) dr', z-x\right)\right](v)   \Bigg\} \, \widehat{p}_{m+1}(\mu, s, r, t, x, z) \\
& =: {\rm I}^{n}_1(v)+{\rm I}^{n}_2(v),  \\
{\rm II}^{n}(v) & : =  \left\{ \frac12 \sum_{i, j=1}^d  \partial^{n}_v [\partial_\mu [a_{i, j}(r, x, [X^{s, \xi, (m)}_{r}]) - a_{i, j}(r, z, [X^{s, \xi, (m)}_r])]](v) \right.  \\
&\quad \left. \times H^{i, j}_2\left(\int_r^{t} a(r', z , [X^{s, \xi, (m)}_{r'}]) dr', z-x\right)  \right\}   \widehat{p}_{m+1}(\mu, s, r, t, x, z) \\
&+  \left\{ \frac12 \sum_{i, j=1}^d  \left(a_{i, j}(r, x, [X^{s, \xi, (m)}_{r}]) - a_{i, j}(r, z, [X^{s, \xi, (m)}_r])\right) \right.  \\ 
& \quad \left. \times \partial^{n}_v \left[\partial_\mu H^{i, j}_2\left(\int_r^{t} a(r', z, [X^{s, \xi, (m)}_{r'}]) dr', z-x\right)\right](v) \right\} \,  \widehat{p}_{m+1}(\mu, s, r, t, x, z) \\
& =: {\rm II}^{n}_1(v) + {\rm II}^{n}_2(v),  \\ \\
{\rm III}^n(v) &:= \left\{- \sum_{i=1}^d b_i(r, x, [X^{s, \xi, (m)}_r]) H^{i}_1\left(\int_r^{t} a(r', z, [X^{s, \xi, (m)}_{r'}]) dr', z-x\right)\right.\\
& +\left.\frac12 \sum_{i, j=1}^d (a_{i, j}(r, x, [X^{s, \xi, (m)}_{r}]) - a_{i, j}(r, z, [X^{s, \xi, (m)}_r])) H^{i, j}_2\left(\int_r^{t} a(r', z, [X^{s, \xi, (m)}_{r'}]) dr', z-x\right) \right\} \\
& \quad \times \partial^{n}_v [\partial_\mu \widehat{p}_{m+1}(\mu, s, r, t, x, z)](v).
\end{align*}
\noindent We now provide an estimate for each term separately. First, from \eqref{recursive:bound:deriv:a:or:b} and the space-time inequality \eqref{space:time:inequality}
\begin{align*}
\left| {\rm I}^{n}_1(v) \right| & = \left| - \sum_{i=1}^d  H^{i}_1\left(\int_r^{t} a(v, z, [X^{s, \xi, (m)}_v]) dv, z-x\right) \partial^{n}_v  [\partial_\mu [ b_i(r, x, [X^{s, \xi, (m)}_r])\right] ] (v)| \widehat{p}_{m+1}(\mu, r, t, x, z) \\
&\leq  \frac{K}{(t-r)^{\frac{1}{2}}}\left\{ \frac{1}{(r-s)^{\frac{1+n-\eta}{2}}} +  \int_{(\rr^d)^2} (|y'-x'|^{\eta}\wedge 1) | \partial^{n}_v[\partial_\mu p_{m}(\mu, s, r, x', y')](v)| \, \mu(dx') \, dy'   \right\} \\
& \quad \times g(c(t-r), z-x). 
\end{align*}
Then, again from \eqref{recursive:bound:deriv:a:or:b} and the space-time inequality \eqref{space:time:inequality} 
\begin{align*}
\left| {\rm I}^{n}_2(v) \right| & = \left| - \sum_{i=1}^d  b_i(r, x, [X^{s, \xi, (m)}_r]) \partial^{n}_v \left[\partial_\mu  H^{i}_1\left(\int_r^{t} a(r', z, [X^{s, \xi, (m)}_{r'}]) dr', z-x\right)\right](v) \right| \widehat{p}_{m+1}(\mu, r, t, x, z)\\
& \leq  K \frac{|z-x|}{(t-r)^2}\int_r^t \left( \frac{1}{(r'-s)^{\frac{1+n-\eta}{2}}} +  \int_{(\rr^d)^2} (|y'-x'|^{\eta}\wedge 1) | \partial^{n}_v[\partial_\mu p_{m}(\mu, s, r', x', y')](v)| \, \mu(dx') \, dy' \right) \, dr' \\
& \quad \times g(c(t-r), z-x)\\
& \leq  \frac{K}{(t-r)^{\frac12}} \left(\frac{1}{(r-s)^{\frac{1+n-\eta}{2}}} +\frac{1}{t-r} \int_r^t  \int_{(\rr^d)^2} (|y'-x'|^{\eta}\wedge 1) | \partial^{n}_v[\partial_\mu p_{m}(\mu, s, r', x', y')](v)| \, \mu(dx') \, dy' \, dr' \right)\\
& \quad \times g(c(t-r), z-x).
\end{align*}
We thus conclude
\begin{align*}
| {\rm I}^{n}(v) | & \leq \frac{K}{(t-r)^{\frac{1}{2}}}\left\{ \frac{1}{(r-s)^{\frac{1+n-\eta}{2}}} +  \int_{(\rr^d)^2} (|y'-x'|^{\eta}\wedge 1) | \partial^{n}_v[\partial_\mu p_{m}(\mu, s, r, x', y')](v)| \, \mu(dx') \, dy' \, \right.\\
& \left.  +\frac{1}{t-r} \int_r^t  \int_{(\rr^d)^2} (|y'-x'|^{\eta}\wedge 1) | \partial^{n}_v[\partial_\mu p_{m}(\mu, s, r', x', y')](v)| \, \mu(dx') \, dy' \, dr'   \right\} g(c(t-r), z-x).
\end{align*}
We now provide an estimate for ${\rm II}^{n}$. 
It follows from \eqref{recursive:bound:deriv:mes:holder:reg:a}, \eqref{standard} and the space-time inequality \eqref{space:time:inequality} that
\begin{align*}
\left| {\rm II}^{n}_1(v)\right| &\leq  K_\beta  \frac{|z-x|^{\beta \eta}}{(t-r)(r-s)^{\frac{1+n-(1-\beta)\eta}{2}}} \\
& \quad \times \left(1 + (r-s)^{\frac{1+n-(1-\beta)\eta}{2}} \int_{(\mathbb{R}^d)^2} (|y'-x'|^{(1-\beta)\eta}\wedge 1) |\partial^{n}_v[\partial_\mu p_{m}(\mu, s, r, x', y')](v)| \, \mu(dx')\, dy'  \right) \notag \\
& \quad \quad \times g(c(t-r), z-x)\\
&\leq  \frac{K_\beta}{(t-r)^{1-\frac{\beta\eta}{2}}(r-s)^{\frac{1+n-(1-\beta)\eta}{2}}} \\
& \quad \times \left(1 + (r-s)^{\frac{1+n-(1-\beta)\eta}{2}} \int_{(\mathbb{R}^d)^2} (|y'-x'|^{(1-\beta)\eta}\wedge 1) |\partial^{n}_v[\partial_\mu p_{m}(\mu, s, r, x', y')](v)| \, \mu(dx')\, dy' \right) \notag \\
& \quad \quad \times g(c(t-r), z-x).
\end{align*}
We now turn to ${\rm II}^{n}_2(v)$. From the very definition of $H^{i, j}_2$, \eqref{recursive:bound:deriv:a:or:b} and noticing that for any smooth map $\mathcal{P}_2(\rr^d) \ni  \nu \mapsto \Sigma(\nu)$ taking values in the set of positive definite matrix it holds 
\begin{align}
\partial^{n}_v & [\partial_\mu  (\Sigma^{-1}(\mu))_{i, j}](v) \notag \\
& = - (\Sigma^{-1}(\mu) \partial^{n}_v [\partial_\mu \Sigma(\mu)](v) \Sigma^{-1}(\mu))_{i, j} = - \sum_{k_1, k_2} (\Sigma^{-1}(\mu))_{i, k_1} \partial^{n}_v[\partial_\mu(\Sigma(\mu))_{{k_1,k_2}}](v) (\Sigma^{-1}(\mu))_{k_2, j}, \quad n=0,1 \label{deriv:mes:cross:smooth:matrix}
\end{align} 
\noindent we get
\begin{align} 
\Big|\partial^{n}_v \partial_\mu & \left[ H^{i, j}_2\left(\int_r^{t} a(r', z, [X^{s, \xi, (m)}_{r'}]) dr', z-x\right)\right](v) \Big| \notag \\
&  \leq   K \left( \frac{|z-x|^2}{(t-r)^3} + \frac{1}{(t-r)^2} \right)  \int_r^{t}  \max_{k, \ell}  \Big| \partial^{n}_v [\partial_\mu [a_{k, \ell}(r', z, [X^{s, \xi, (m)}_{r'}])]](v) \Big|dr' \notag \\
&\leq  K \left(\frac{|z-x|^2}{(t-r)^{2} } + \frac{1}{t-r} \right) \left(\frac{1}{(r-s)^{\frac{1+n-\eta}{2}}} \right. \notag \\
& \quad \left. + \frac{1}{t-r} \int_r^t \int_{(\mathbb{R}^d)^2} (|y'-x'|^{\eta}\wedge 1) |\partial^{n}_v[\partial_\mu p_{m}(\mu, s, r', x', y')](v)| \, \mu(dx') \, dy'  \, dr'  \right), \label{deriv:mes:second:order:polyn:hermite}
\end{align}
\noindent so that, from the uniform $\eta$-H\"older regularity of $a(t, ., m)$ and the space-time inequality \eqref{space:time:inequality}
\begin{align*}
\left| {\rm II}^{n}_2(v)\right| &\leq  K \left(\frac{ |z-x|^{2+ \eta}}{(t-r)^2} + \frac{|z-x|^{\eta}}{(t-r)} \right) \\
& \quad \times  \left(\frac{1}{(r-s)^{\frac{1+n-\eta}{2}}} + \frac{1}{t-r} \int_r^t \int_{(\mathbb{R}^d)^2} (|y'-x'|^{\eta}\wedge 1) |\partial^{n}_v[\partial_\mu p_{m}(\mu, s, r', x', y')](v)| \, \mu(dx') \, dy'  \, dr'  \right) \\
& \quad  \quad \times g(c(t-r),z-x)\\
&\leq \frac{K}{(t-r)^{1-\frac{\eta}{2}}(r-s)^{\frac{1+n-\eta}{2}}}  \\
& \quad \times  \left(1+ \frac{(r-s)^{\frac{1+n-\eta}{2}}}{t-r} \int_r^t \int_{(\mathbb{R}^d)^2} (|y'-x'|^{\eta}\wedge 1) |\partial^{n}_v[\partial_\mu p_{m}(\mu, s, r', x', y')](v)| \, \mu(dx') \, dy'  \, dr'  \right) \\
& \quad \quad \times g(c(t-r),z-x).
\end{align*}
Hence, gathering the estimates on ${\rm II}^{n}_1$ and ${\rm II}^{n}_2$, we obtain
\begin{align*}
  \left| {\rm II}^n(v) \right| &\leq \frac{K_\beta}{(t-r)^{1-\frac{\beta\eta}{2}}(r-s)^{\frac{1+n-(1-\beta)\eta}{2}}}\\
  & \quad \times \left(1 +  (r-s)^{\frac{1+n-(1-\beta)\eta}{2}} \int_{(\mathbb{R}^d)^2} (|y'-x'|^{(1-\beta)\eta}\wedge 1) |\partial^{n}_v[\partial_\mu p_{m}(\mu, s, r, x', y')](v)| \, \mu(dx') \, dy'  \right. \notag \\
& \quad \quad \left. + \frac{(r-s)^{\frac{1+n-(1-\beta)\eta}{2}}}{(t-r)^{1+\frac{(\beta-1)\eta}{2}}} \int_r^t \int_{(\mathbb{R}^d)^2} (|y'-x'|^{\eta}\wedge 1) |\partial^{n}_v[\partial_\mu p_{m}(\mu, s, r', x', y')](v)| \, \mu(dx')\, dy' \, dr'   \right) \\
& \quad \times g(c(t-r), z-x)
\end{align*}
\noindent for any $\beta \in [0,1]$.
Finally, using the estimate  \eqref{cross:mes:deriv:p:hat:s:r:t} of Corollary \ref{cor:deriv:time:and:mes:phat}, the uniform $\eta$-H\"older regularity of $ a_{i, j}(t, ., m)$ and the space-time inequality \eqref{space:time:inequality}, we get
\begin{align*}
|{\rm III}^n(v)| & \leq  \frac{K}{(t-r)^{1-\frac{\eta}{2}}(r-s)^{\frac{1+n-\eta}{2}}} \\
& \quad \times  \left(1+ \frac{(r-s)^{\frac{1+n-\eta}{2}}}{t-r} \int_r^t \int_{(\mathbb{R}^d)^2} (|y'-x'|^{\eta}\wedge 1) |\partial^{n}_v[\partial_\mu p_{m}(\mu, s, r', x', y')](v)| \, \mu(dx') \, dy'  \, dr' \right)  \\
& \quad \quad \times g(c(t-r),z-x).
\end{align*}
We conclude the proof of \eqref{cross:mes:deriv:parametrix:kernel:s:r:t:with:beta} by gathering the estimates on ${\rm I}^n(v)$, ${\rm II}^n(v)$ and ${\rm III}^n(v)$.\\

\noindent \emph{Step 3: proof of the pointwise Gaussian estimate \eqref{cross:time:deriv:parametrix:kernel:s:r:t}.}\\

The strategy is similar to the one developed previously. We here use the following decomposition
$$
\partial_s \mH_{m+1}(\mu, s, r, t, x, z) =: {\rm I} + {\rm II} + {\rm III} 
$$
\noindent where
\begin{align*}
{\rm I} & = \Big[- \sum_{i=1}^d \partial_s [b_i \left(r, x, [X^{s, \xi, (m)}_r]\right)] H^{i}_1\left(\int_r^t a(r', z, [X^{s, \xi ,(m)}_{r'}]) \, dr', z-x\right) \nonumber\\
& - \sum_{i=1}^d b_i\left(r, x, [X^{s, \xi, (m)}_r]\right) \partial_s H^{i}_1\left(\int_r^t a(r', z, [X^{s, \xi ,(m)}_{r'}]) \, dr', z-x\right) \Big] \widehat{p}_{m+1}(\mu, s, r, t, x, z)\nonumber \\
& =: {\rm I}_{1} + {\rm I}_{2}, \\
{\rm II} & = \Big[ \frac12 \sum_{i, j=1}^d (\partial_s [a_{i, j}(r, x, [X^{s, \xi, (m)}_r]) -  a_{i, j}(r, z, [X^{s, \xi , (m)}_r])])  H^{i, j}_2(\int_r^t a(r', z, [X^{s, \xi, (m)}_{r'}]) \, dr', z-x) \nonumber \\
& + \frac12 \sum_{i, j=1}^d \left(a_{i, j}(r, x, [X^{s, \xi, (m)}_r]) - a_{i, j}(r, z, [X^{s, \xi , (m)}_r])\right) \\
& \quad \times \partial_s H^{i, j}_2\left(\int_r^t a(r', z, [X^{s, \xi, (m)}_{r'}]) \, dr', z-x\right) \Big] \widehat{p}_{m+1}(\mu, s, r, t, x, z)\\
& =: {\rm II}_1 + {\rm II}_2, \\
{\rm III} &=  \Big[ - \sum_{i=1}^d b_i(r, x, [X^{s, \xi, (m)}_r]) H^{i}_1(\int_r^t a(r', z, [X^{s, \xi ,(m)}_{r'}]) \, dr', z-x) \nonumber \\
& \quad + \frac{1}{2} \sum_{i, j=1}^d \left(a_{i, j}(r, x, [X^{s, \xi, (m)}_r]) - a_{i, j}(r, z, [X^{s, \xi, (m)}_r])\right) \nonumber \\
& \quad \quad \times H^{i, j}_2\left(\int_r^t a(r', z, [X^{s, \xi, (m)}_{r'}) \, dr', z-x\right) \Big] \partial_s \widehat{p}_{m+1}(\mu, s, r, t, x, z).
\end{align*}

We again provide an estimate for each term separately. We will be brief on some technical details inasmuch as the proof of the following estimates follows similar lines of reasonings as those employed in the previous step. 

From \eqref{recursive:bound:time:deriv:a:or:b} and the space-time inequality \eqref{space:time:inequality}, we obtain
$$
|{\rm I}_1| \leq \frac{K}{(t-r)^{\frac12}} \int_{(\mathbb{R}^d)^2} (|y'-x'|^\eta \wedge 1) |\partial_s p_m(\mu, s, r,  x', y')| \, \mu(dx') \, dy' \, g(c(t-r), z-x).
$$
 
 From the very definition of $H^{i}_1$ and using the fact that for any differentiable map $[0,r) \ni  s \mapsto \Sigma(s)$ taking values in the set of positive definite matrix one has 
$$
\partial_s (\Sigma^{-1}(s))_{i, j} = - (\Sigma^{-1}(s) \partial_s \Sigma(s) \Sigma^{-1}(s))_{i, j} = - \sum_{k_1, k_2} (\Sigma^{-1}(s))_{i, k_1} \partial_s \Sigma_{k_1,k_2}(s) (\Sigma^{-1}(s))_{k_2, j}$$
\noindent we get
\begin{align}
\label{bound:time:derivative:H1}
\Big| \partial_s H^{i}_1\left(\int_r^t a(r', z, [X^{s, \xi ,(m)}_{r'}]) \, dr', z-x\right) \Big| \leq K\frac{|z-x|}{(t-r)^2} \int_r^t \max_{k, l} |\partial_s [a_{k, l}(r', z, [X^{s, \xi ,(m)}_{r'}])]| \, dr'
\end{align}

\noindent and
\begin{align}
\label{bound:time:derivative:H2}
\Big| \partial_s&  H^{i, j}_2\left(\int_r^t a(r', z, [X^{s, \xi ,(m)}_{r'}]) \, dr', z-x\right) \Big| \\
& \leq   K \left( \frac{|z-x|^2}{(t-r)^3} + \frac{1}{(t-r)^2} \right)  \int_r^{t}  \max_{k, \ell}  \Big| \partial_s [a_{k, \ell}(r', z, [X^{s, \xi, (m)}_{r'}])] \Big| dr'. \notag 
\end{align}

It now follows from \eqref{bound:time:derivative:H1} together with \eqref{recursive:bound:time:deriv:a:or:b} and the space-time inequality \eqref{space:time:inequality} that
\begin{align*}
|{\rm I}_2| & \leq K \frac{|z-x|}{(t-r)^{2}} \int_r^t \int_{(\mathbb{R}^d)^2} (|y'-x'|^\eta \wedge 1) |\partial_s p_m(\mu, s, r',  x', y')| \, \mu(dx') \, dy' \, dr' \, g(c(t-r), z-x)\\
& \leq  \frac{K}{(t-r)^{\frac32}} \int_r^t \int_{(\mathbb{R}^d)^2} (|y'-x'|^\eta \wedge 1) |\partial_s p_m(\mu, s, r,  x', y')| \, \mu(dx') \, dy'\, dr' \, g(c(t-r), z-x)
\end{align*}
\noindent so that
\begin{align*}
|{\rm I}|&  \leq \frac{K}{(t-r)^{\frac12}} \left\{  \int_{(\mathbb{R}^d)^2} (|y'-x'|^\eta \wedge 1) |\partial_s p_m(\mu, s, r,  x', y')| \, \mu(dx') \, dy' \right. \\ 
& \quad \quad \left.+ \frac{1}{t-r} \int_r^t \int_{(\mathbb{R}^d)^2} (|y'-x'|^\eta \wedge 1) |\partial_s p_m(\mu, s, r,  x', y')| \, \mu(dx') \, dy'\, dr'  \right\}  \, g(c(t-r), z-x).
\end{align*}

It follows from \eqref{recursive:bound:time:deriv::holder:reg:a} and the space-time inequality \eqref{space:time:inequality} that 
$$
|{\rm II}_1| \leq \frac{K_\beta}{(t-r)^{1-\frac{\beta \eta}{2}}} \int_{(\rr^d)^2} (|y'-x'|^{(1-\beta)\eta} \wedge 1) \,   |\partial_s p_{m}(\mu, s , r, x', y')| \, \mu(dx') \, dy'  \, g(c(t-r), z-x).
$$
\noindent for any $\beta \in [0,1]$. From \eqref{bound:time:derivative:H2}, \eqref{recursive:bound:time:deriv:a:or:b}, the uniform $\eta$-H\"older regularity of $a_{i, j}(t, ., m)$, the boundedness of $b_i$ and the space-time inequality \eqref{space:time:inequality}, we get
\begin{align*}
|{\rm II}_2| & \leq K  \left\{ \frac{|z-x|^{2+\eta}}{(t-r)^3} + \frac{|z-x|^\eta}{(t-r)^2}\right\} \int_r^t \max_{k,l} | \partial_s [a_{k, l}(r', z, [X^{s, \xi , (m)}_{r'}])]| \, dr' \, g(c(t-r), z-x) \\
& \leq \frac{K}{(t-r)^{2-\frac{\eta}{2}}} \int_r^t  \int_{(\mathbb{R}^d)^2} (|y'-x'|^\eta \wedge 1) \,   |\partial_s p_{m}(\mu, s , r', x', y')| \, \mu(dx') \, dy'\, dr' \, g(c(t-r), z-x).
\end{align*}
Gathering the two previous estimates, we obtain
\begin{align*}
|{\rm II}| & \leq \frac{K_\beta}{(t-r)^{1-\frac{\beta \eta}{2}}}\left\{\int_{(\mathbb{R}^d)^2} (|y'-x'|^{(1-\beta)\eta} \wedge 1) \,   |\partial_s p_{m}(\mu, s , r, x', y')| \, \mu(dx') \, dy'  \right. \\
& \quad \left. + \frac{1}{(t-r)^{1+\frac{(\beta-1)\eta}{2}}}\int_r^t  \int_{(\mathbb{R}^d)^2} (|y'-x'|^\eta \wedge 1) \,   |\partial_s p_{m}(\mu, s , r', x', y')| \, \mu(dx') \, dy'\, dr'  \right\} g(c(t-r), z-x).
\end{align*}
Finally, the Gaussian estimate \eqref{time:deriv:p:hat:s:r:t}, the uniform $\eta$-H\"older regularity of $a_{i, j}(t, ., m)$ and the space-time inequality \eqref{space:time:inequality} clearly imply
\begin{align*}
|{\rm III}| & \leq \frac{K}{(t-r)^{2-\frac{\eta}{2}}}  \int_r^t  \int_{(\rr^d)^2} (|y'-x'|^{\eta}\wedge 1) | \partial_s p_{m}(\mu, s, r', x', y') | \, \mu(dx') \, dy' \, dr'  g(c(t-r), z-x).
\end{align*}
Gathering the previous estimates on ${\rm I}$, ${\rm II}$ and ${\rm III}$ concludes the proof of \eqref{cross:time:deriv:parametrix:kernel:s:r:t}. \\

\noindent \emph{Step 4: proof of the pointwise Gaussian estimate \eqref{cross:mes:deriv:parametrix:kernel:s:r:t:reg:holder}.}\\
 
We start from the decomposition formula \eqref{deriv:mu:H:mp1} with $n=1$ and establish appropriate Gaussian estimate for the difference of each term evaluated at $v$ and $v'$ respectively. From \eqref{recursive:bound:deriv:mes:reg:holder:a:or:b}, \HE\, and the space-time inequality \eqref{space:time:inequality}, we get
 \begin{align*}
  | &{\rm I}^1_1(v)   -   {\rm I}^1_1(v')|  \\
 & \leq \frac{K}{(t-r)^{\frac12}}  \sum_{i=1}^d  \Big|\partial_v  [\partial_\mu [ b_i(r, x, [X^{s, \xi, (m)}_r])]](v) -  \partial_v [\partial_\mu [ b_i(r, x, [X^{s, \xi, (m)}_r])]](v')\Big|  g(c(t-r), z-x)\\
 & \leq \frac{K_\beta}{(t-r)^{\frac12}} \left\{ \frac{|v-v'|^\beta}{(r-s)^{1+\frac{\beta-\eta}{2}}} \right. \\
 & \quad \left.+ \int_{(\mathbb{R}^d)^2} (|y'-x'|^\eta \wedge 1) | | \partial_v[\partial_\mu p_{m}(\mu, s, r, x', y)](v) - \partial_v[\partial_\mu p_{m}(\mu, s, r, x', y)](v')| \, \mu(dx') \, dy' \right\}  g(c(t-r), z-x).
\end{align*}

Next, from the very definition of the first order Hermite polynomial $H^{i}_1$, the identity \eqref{deriv:mes:cross:smooth:matrix}, the estimate \eqref{recursive:bound:deriv:mes:reg:holder:a:or:b}, \HE\, and the space-time inequality \eqref{space:time:inequality}, we deduce 
\begin{align*}
 |& {\rm I}^1_2(v) -   {\rm I}^1_2(v')|  \\
  & \leq \frac{K}{(t-r)^{\frac32}} \int_r^t \max_{i, j} \Big| \partial_v [\partial_\mu [a_{i, j}(r', z,[X^{s, \xi, (m)}_{r'}])]](v) - \partial_v [\partial_\mu [a_{i, j}(r', z,[X^{s, \xi, (m)}_{r'}])]](v') \Big| \, dr' \,  g(c(t-r), z-x)\\
  & \leq  K_\beta \left\{ \frac{|v-v'|^\beta}{(t-r)^{\frac12}(r-s)^{1+\frac{\beta-\eta}{2}}} \right.\\
  & \left. \quad + \frac{1}{(t-r)^{\frac32}} \int_r^t \int_{(\mathbb{R}^d)^2} (|y'-x'|^\eta \wedge 1) | | \partial_v[\partial_\mu p_{m}(\mu, s, r', x', y')](v) - \partial_v[\partial_\mu p_{m}(\mu, s, r', x', y')](v')| \, \mu(dx') \, dy' \, dr' \right\}\\
  & \quad \quad \times g(c(t-r), z-x).
\end{align*}

Gathering the two previous estimates, we thus conclude
\begin{align*}
 |& {\rm I}^1(v) -   {\rm I}^1(v')| \\
 & \leq K_\beta \left\{ \frac{|v-v'|^\beta}{(t-r)^{\frac12}(r-s)^{1+\frac{\beta-\eta}{2}}} \right. \\
 & \quad \left. + \frac{1}{(t-r)^{\frac12}} \int_{(\mathbb{R}^d)^2} (|y'-x'|^\eta \wedge 1) | | \partial_v[\partial_\mu p_{m}(\mu, s, r, x', y)](v) - \partial_v[\partial_\mu p_{m}(\mu, s, r, x', y)](v')| \, \mu(dx') \, dy'  \right.\\
 & \left. \quad+ \frac{1}{(t-r)^{\frac32}} \int_r^t \int_{(\mathbb{R}^d)^2} (|y'-x'|^\eta \wedge 1) | | \partial_v[\partial_\mu p_{m}(\mu, s, r', x', y)](v) - \partial_v[\partial_\mu p_{m}(\mu, s, r', x', y)](v')| \, \mu(dx') \, dy' \, dr' \right\} \\
   & \quad \quad \times g(c(t-r), z-x).
 \end{align*}

  In order to handle the difference ${\rm II}^{1}_1(v) - {\rm II}^{1}_1(v')$, we use \eqref{recursive:bound:deriv:mes:double:reg:holder:a}, \eqref{standard} and the space-time inequality \eqref{space:time:inequality}. We obtain
\begin{align*}
 |& {\rm II}^1_1(v) - {\rm II}^1_1(v')| \\
 & \leq K_\beta \left\{\frac{1}{(t-r)^{1-\frac{\eta}{2}}(r-s)^{1+\frac{\beta}{2}}} \wedge \frac{1}{(t-r) (r-s)^{1+\frac{\beta-\eta}{2}}} \right\} \\
 & \quad \times \left\{ |v-v'|^\beta + (r-s)^{1+\frac{\beta}{2}} \int_{(\mathbb{R}^d)^2} \Big|\partial_v\Big[\partial_{\mu}p_{m}(\mu, s , r, x', y')\Big](v) - \partial_v\Big[\partial_{\mu}p_{m}(\mu, s , r, x', y')\Big](v')\Big| \, \mu(dx') \, dy' \right. \\
 & \quad \quad \left. +(r-s)^{1+\frac{\beta-\eta}{2}} \int_{(\mathbb{R}^d)^2} (|y'-x'|^\eta \wedge 1) \Big|\partial_v\Big[\partial_{\mu}p_{m}(\mu, s , r, x', y')\Big](v) - \partial_v\Big[\partial_{\mu}p_{m}(\mu, s , r, x', y')\Big](v')\Big| \, \mu(dx')\, dy'\right\}\\
 & \quad \quad \quad  \times g(c(t-r), z-x).
\end{align*}

We deal with ${\rm II}^{1}_2(v) - {\rm II}^{1}_2(v')$ by using the very definition of the Hermite polynomial $H_2^{i, j}$, the identity \eqref{deriv:mes:cross:smooth:matrix}, the estimate \eqref{recursive:bound:deriv:mes:reg:holder:a:or:b}, \HE\, and the space-time inequality \eqref{space:time:inequality}. We obtain
\begin{align*}
 |& {\rm II}^1_2(v) - {\rm II}^1_2(v')| \\
 & \leq K\left\{ \frac{|z-x|^{2+\eta}}{(t-r)^3} + \frac{|z-x|^\eta}{(t-r)^2}\right\}  \int_r^t \max_{i, j} \Big| \partial_v [\partial_\mu [a_{i, j}(r', z,[X^{s, \xi, (m)}_{r'}])]](v) - \partial_v [\partial_\mu [a_{i, j}(r', z,[X^{s, \xi, (m)}_{r'}])]](v') \Big| \, dr' \\
  & \quad \times g(c(t-r), z-x) \\
& \leq  K_\beta \left\{ \frac{|v-v'|^\beta}{(t-r)^{1-\frac{\eta}{2}} (r-s)^{1+\frac{\beta-\eta}{2}}}  \right. \\
& \quad \left. + \frac{1}{(t-r)^{2 - \frac{\eta}{2}}} \int_r^t \int_{(\mathbb{R}^d)^2} (|y'-x'|^\eta \wedge 1) | | \partial_v[\partial_\mu p_{m}(\mu, s, r', x', y')](v) - \partial_v[\partial_\mu p_{m}(\mu, s, r', x', y')](v')| \, \mu(dx') \, dy' \, dr'   \right\}\\
& \quad \quad \times g(c(t-r), z-x).
\end{align*}

Putting the above terms together, we conclude
\begin{align*}
 |& {\rm II}^1(v) - {\rm II}^1(v')| \\
 & \leq K_\beta \left\{\frac{1}{(t-r)^{1-\frac{\eta}{2}}(r-s)^{1+\frac{\beta}{2}}} \wedge \frac{1}{(t-r) (r-s)^{1+\frac{\beta-\eta}{2}}} \right\} \\
 & \quad \times \left\{ |v-v'|^\beta + (r-s)^{1+\frac{\beta}{2}} \int_{(\mathbb{R}^d)^2} \Big|\partial_v\Big[\partial_{\mu}p_{m}(\mu, s , r, x', y')\Big](v) - \partial_v\Big[\partial_{\mu}p_{m}(\mu, s , r, x', y')\Big](v')\Big| \, \mu(dx') \, dy' \right. \\
 & \quad \quad \left. +(r-s)^{1+\frac{\beta-\eta}{2}} \int_{(\mathbb{R}^d)^2} (|y'-x'|^\eta \wedge 1) \Big|\partial_v\Big[\partial_{\mu}p_{m}(\mu, s , r, x', y')\Big](v) - \partial_v\Big[\partial_{\mu}p_{m}(\mu, s , r, x', y')\Big](v')\Big| \, \mu(dx')\, dy' \right. \\
 & \left. \quad \quad +  \frac{(r-s)^{1+\frac{\beta-\eta}{2}}}{t-r} \int_r^t \int_{(\mathbb{R}^d)^2} (|y'-x'|^\eta \wedge 1) | | \partial_v[\partial_\mu p_{m}(\mu, s, r', x', y')](v) - \partial_v[\partial_\mu p_{m}(\mu, s, r', x', y')](v')| \, \mu(dx') \, dy' \, dr' \right\}\\
 & \quad \quad \quad  \times g(c(t-r), z-x).
\end{align*}

For the last term, from the uniform $\eta$-H\"older regularity of $a_{i, j}(t, ., m)$, the space-time inequality \eqref{space:time:inequality} and \eqref{cross:mes:deriv:reg:holder:terminal point:p:hat:s:r:t}, it holds
\begin{align*}
 | & {\rm III}^1(v)  - {\rm III}^1(v') |  \\
 & \leq K \frac{1}{(t-r)^{1-\frac{\eta}{2}}} |\partial_v [\partial_\mu \widehat{p}_{m+1}(\mu, s, r, t, x, z)](v) - \partial_v [\partial_\mu \widehat{p}_{m+1}(\mu, s, r, t, x, z)](v') | \\
 & \leq \frac{K_\beta}{(t-r)^{1-\frac{\eta}{2}}(r-s)^{1 + \frac{\beta-\eta}{2}}} \Big\{|v-v'|^{\beta}   \\
 & + \frac{(r-s)^{1 + \frac{\beta-\eta}{2}}}{t-r}   \int_r^t \int_{(\mathbb{R}^d)^2} (|y'-x'|^\eta \wedge 1) | | \partial_v[\partial_\mu p_{m}(\mu, s, r', x', y')](v) - \partial_v[\partial_\mu p_{m}(\mu, s, r', x', y')](v')| \, \mu(dx') \, dy' \, dr'  \Big\} \\
 & \quad \times g(c(t-r), z- x).
\end{align*}

\noindent Gathering the previous estimates, we eventually deduce
\begin{align*}
\Big| & \partial_v\Big[\partial_\mu \mH_{m+1}(\mu, s, r ,t ,x ,z) \Big](v) -  \partial_v\Big[\partial_\mu \mH_{m+1}(\mu, s, r ,t ,x ,z) \Big](v') \Big| \\
& \leq K_\beta  \left\{ \frac{1}{(t-r)^{1-\frac{\eta}{2}} (r-s)^{1+\frac{\beta}{2}}} \wedge  \frac{1}{(t-r)(r-s)^{1+\frac{\beta-\eta}{2}}} \right\} \\
& \quad \times \Big[|v-v'|^\beta + (r-s)^{1+\frac{\beta}{2}} \int_{(\mathbb{R}^d)^2} \Big|\partial_v\Big[\partial_{\mu}p_{m}(\mu, s , r, x', y')\Big](v) - \partial_v\Big[\partial_{\mu}p_{m}(\mu, s , r, x', y')\Big](v')\Big| \, \mu(dx') \, dy' \\
& \quad \quad +(r-s)^{1+\frac{\beta-\eta}{2}} \int_{(\mathbb{R}^d)^2} (|y'-x'|^\eta \wedge 1) \Big|\partial_v\Big[\partial_{\mu}p_{m}(\mu, s , r, x', y')\Big](v) - \partial_v\Big[\partial_{\mu}p_{m}(\mu, s , r, x', y')\Big](v')\Big| \, \mu(dx') \, dy' \\
& \quad \quad \quad +  \frac{(r-s)^{1 + \frac{\beta-\eta}{2}}}{t-r}   \int_r^t \int_{(\mathbb{R}^d)^2} (|y-x'|^\eta \wedge 1) | | \partial_v[\partial_\mu p_{m}(\mu, s, r', x', y)](v) - \partial_v[\partial_\mu p_{m}(\mu, s, r', x', y)](v')| \, \mu(dx') \, dy \, dr' \Big] \\
& \quad \quad \quad \quad \times g(c(t-r), z-x).
\end{align*}

The proof of \eqref{cross:mes:deriv:parametrix:kernel:s:r:t:reg:holder} is now complete.

\subsection{Proof of Corollary \ref{cor:mes:time:deriv:iterated:parametrix:kernel}.\\}\label{section:proof:cor:mes:time:deriv:iterated:parametrix:kernel}

We proceed by induction on $k$. The case $k=1$ directly follows from Corollary \ref{cor:deriv:time:and:mes:parametrix:kernel} and the estimates \eqref{time:degeneracy:estimate:deriv:mes:parametrix:kernel:bis} and \eqref{time:degeneracy:estimate:deriv:time:parametrix:kernel}. We now assume that the map $[0,r) \times \pp \ni (s, \mu) \mapsto \mH^{(k)}_{m+1}(\mu, s, r, t, x, z)$ belongs to $\mathcal{C}^{1, 2}([0,r)\times \pp)$, with continuous derivatives with respect to $s$, $x$, $\mu$ and $v$ and that the estimates \eqref{cross:mes:deriv:iterated:parametrix:kernel:s:r:t} and \eqref{cross:time:deriv:iterated:parametrix:kernel:s:r:t} are valid at step $k$. Recall that by the very definition of the space-time convolution, it holds
\begin{equation}
\label{iterated:parametrix:kernel:next:step}
\mH^{(k+1)}_{m+1}(\mu, s, r, t, x, z) = \int_r^t \int_{ \mathbb{R}^d }  \mH_{m+1}(\mu, s, r , r' , x , y)  \mH^{(k)}_{m+1}(\mu, s, r' ,t , y ,z) \, dy \, dr'
\end{equation}
\noindent so that it directly follows from the estimates \eqref{time:degeneracy:estimate:deriv:mes:parametrix:kernel:bis}, \eqref{time:degeneracy:estimate:deriv:time:parametrix:kernel}, \eqref{cross:mes:deriv:iterated:parametrix:kernel:s:r:t}, \eqref{cross:time:deriv:iterated:parametrix:kernel:s:r:t} and the dominated convergence theorem that the map $[0,r) \times \pp \ni (s, \mu) \mapsto \mH^{(k+1)}_{m+1}(\mu, s, r, t, x, z)$ belongs to $\mathcal{C}^{1, 2}([0,r)\times \pp)$ and admits continuous derivatives with respect to the variables $s$, $x$, $\mu$ and $v$.  

Now, from \eqref{iter:parametrix:kernel} and \eqref{time:degeneracy:estimate:deriv:mes:parametrix:kernel:bis} 
\begin{align*}
\int_r^t \int_{ \mathbb{R}^d }&   |\partial^n_v [\partial_\mu \mH_{m+1}(\mu, s, r , r' , x , y)](v)| |\mH^{(k)}_{m+1}(\mu, s, r' ,t , y ,z)| \, dy \, dr'\\
& \leq \frac{K^{k} K_m}{(r-s)^{\frac{1+n}{2}-\frac{\eta}{4}}}  \prod_{\ell =1}^{k-1} B\left(\ell \frac{\eta}{2}, \frac{\eta}{2}\right)  \int_r^t \frac{1}{(r'-r)^{1-\frac{\eta}{4}}} \frac{1}{(t-r')^{1-k\frac{\eta}{2}}} \, dr' \, g(c(t-r), z-x) \\
& = \frac{K^{k} K_m}{(r-s)^{\frac{1+n}{2}-\frac{\eta}{4}} (t-r)^{1- k\frac{\eta}{2}-\frac{\eta}{4}}}  \prod_{\ell =1}^{k-1} B\left(\ell \frac{\eta}{2}, \frac{\eta}{2}\right) B\left(\frac{\eta}{4}, k \frac{\eta}{2}\right) \, g(c(t-r), z-x)\\
& \leq \frac{K^{k} K_m}{(r-s)^{\frac{1+n}{2}-\frac{\eta}{4}}(t-r)^{1-k\frac{\eta}{2}-\frac{\eta}{4}}} \prod_{\ell =1}^{k} B\left(\frac{\eta}{4}, \frac{\eta}{4} + (\ell-1)\frac{\eta}{2}\right) \, g(c(t-r), z-x)
\end{align*}
\noindent and similarly from \eqref{iter:parametrix:kernel} with $k=1$ and the induction hypothesis \eqref{cross:mes:deriv:iterated:parametrix:kernel:s:r:t}
\begin{align*}
\int_r^t \int_{ \mathbb{R}^d }&   |\mH_{m+1}(\mu, s, r , r' , x , y)| |\partial_v[ \partial_\mu  \mH^{(k)}_{m+1}(\mu, s, r' ,t , y ,z)](v)| \, dy\, dr' \\
& \leq k K^{k} K_m  \prod_{\ell =1}^{k-1} B\left(\frac{\eta}{4}, \frac{\eta}{4} + (\ell-1)\frac{\eta}{2}\right)  \int_r^t \frac{1}{(r'-r)^{1-\frac{\eta}{2}}} \frac{1}{ (r'-s)^{\frac{1+n}{2}-\frac{\eta}{4}}(t-r')^{1- (k-1) \frac{\eta}{2}- \frac{\eta}{4}}} \, dr' \, g(c(t-r), z-x) \\
& \leq \frac{k K^{k} K_m}{(r-s)^{\frac{1+n}{2}- \frac{\eta}{4}} (t-r)^{1- k\frac{\eta}{2}-\frac{\eta}{4}}} \prod_{\ell =1}^{k-1} B\left(\frac{\eta}{4}, \frac{\eta}{4} + (\ell-1)\frac{\eta}{2}\right) B\left(\frac{\eta}{2}, \frac{\eta}{4} + (k-1)\frac{\eta}{2}\right) \, g(c(t-r), z-x)\\
& \leq \frac{k K^{k} K_m}{(r-s)^{\frac{1+n}{2}- \frac{\eta}{4}} (t-r)^{1- k\frac{\eta}{2}-\frac{\eta}{4}}} \prod_{\ell =1}^{k} B\left(\frac{\eta}{4}, \frac{\eta}{4} + (\ell-1)\frac{\eta}{2}\right) \, g(c(t-r), z-x).
\end{align*}
Summing the two previous upper-bounds allows to conclude that \eqref{cross:mes:deriv:iterated:parametrix:kernel:s:r:t} is satisfied at step $k+1$. The proof of \eqref{cross:time:deriv:iterated:parametrix:kernel:s:r:t} follows from similar arguments and is thus omitted.

\subsection{Proof of Lemma \ref{lem:technical:estimate:reg:wasserstein:metric:coefficients:densities}.\\} \label{section:proof:lem:technical:estimate:reg:wasserstein:metric:coefficients:densities}

\smallskip

\noindent \emph{Step 1: proof of \eqref{diff:mes:drift:diff:coefficients}.}\\

We use the decomposition
\begin{align}
h(x) & :=a_{i, j} (t, x, [X^{s, \xi, (m)}_t])  - a_{i, j}(t, x, [X^{s, \xi', (m)}_t]) \nonumber \\
& =\int_0^1 \int_{\mathbb{R}^d} \frac{\delta a_{i, j}}{\delta m}(t, x, \Theta^{(m)}_{\lambda, t})(y') \, [p_{m}(\mu, s, t, y') - p_{m}(\mu', s, t, y')]\, dy' \, d\lambda \nonumber \\
& = \int_0^1  \int_{(\rr^d)^2} \frac{\delta a_{i, j}}{\delta m}(t, x, \Theta^{(m)}_{\lambda, t})(y') p_{m}(\mu, s, t, x', y') \, dy'  \, (\mu-\mu')(dx') \,  d\lambda  \nonumber \\
&  \quad +  \int_0^1 \int_{(\rr^d)^2} \Big(\frac{\delta a_{i, j}}{\delta m}(t, x, \Theta^{(m)}_{\lambda, t})(y') - \frac{\delta a_{i, j}}{\delta m}(t, x, \Theta^{(m)}_{\lambda, t})(x') \Big) \Delta_{\mu, \mu'} p_{m}(\mu, s, t , x', y')  \, dy'  \mu'(dx') \, d\lambda \nonumber \\
& =: {\rm I}(x) + {\rm II}(x) \label{decomposition:diff:aij}
\end{align}
\noindent where we introduced the notation $\Theta^{(m)}_{\lambda, t} := (1-\lambda) [X^{s, \xi, (m)}_t] + \lambda [X^{s, \xi', (m)}_t]$. We now claim that for any $\beta \in [\eta,1]$, there exists a positive constant $K_\beta$ such that for any $x', \, x'' \in \mathbb{R}^d$
\begin{equation}
\label{holder:regularity:estimate:functional:derivative}
\Big| \int_{\mathbb{R}^d} \frac{\delta a_{i, j}}{\delta m}(t, x, \Theta^{(m)}_{\lambda, t})(y') [p_{m}(\mu, s, t, x', y') - p_{m}(\mu, s, t, x'', y')]  \, dy' \Big| \leq K_\beta\frac{|x'-x''|^\beta}{(t-s)^{\frac{\beta-\eta}{2}}}.
\end{equation}

The proof of the above estimate is quite standard. If $|x'-x''|> (t-s)^{1/2}$, we first write
\begin{align*}
\int_{\rr^d} & \frac{\delta a_{i, j}}{\delta m}(t, x, \Theta^{(m)}_{\lambda, t})(y') [p_{m}(\mu, s, t, x', y') - p_{m}(\mu, s, t, x'', y')]  \, dy' \\
& = \frac{\delta a_{i, j}}{\delta m}(t, x, \Theta^{(m)}_{\lambda, t})(x') - \frac{\delta a_{i, j}}{\delta m}(t, x, \Theta^{(m)}_{\lambda, t})(x'') \\
& \quad + \int_{\rr^d} \Big[\frac{\delta a_{i, j}}{\delta m}(t, x, \Theta^{(m)}_{\lambda, t})(y') -  \frac{\delta a_{i, j}}{\delta m}(t, x, \Theta^{(m)}_{\lambda, t})(x') \Big] p_{m}(\mu, s, t, x', y')  dy' \\
&  \quad - \int_{\rr^d} \Big[\frac{\delta a_{i, j}}{\delta m}(t, x, \Theta^{(m)}_{\lambda, t})(y') -  \frac{\delta a_{i, j}}{\delta m}(t, x, \Theta^{(m)}_{\lambda, t})(x'') \Big] p_{m}(\mu, s, t, x'', y')  dy' 
\end{align*}
\noindent and then combine the uniform $\eta$-H\"older regularity of $[\delta a_{i, j}/\delta m](t, x, m)(.)$ with the Gaussian upper-estimate \eqref{bound:derivative:heat:kernel} with $n=0$ and the space-time inequality \eqref{space:time:inequality} so that
\begin{align*}
\Big| \int_{\mathbb{R}^d} & \frac{\delta a_{i, j}}{\delta m}(t, x, \Theta^{(m)}_{\lambda, t})(y') [p_{m}(\mu, s, t, x', y') - p_{m}(\mu, s, t, x'', y')]  \, dy' \Big| \\
& \leq K (|x'-x''|^\eta + (t-s)^{\eta/2}) \\
& \leq K \frac{|x'-x''|^\beta}{(t-s)^{\frac{\beta-\eta}{2}}}
\end{align*}
\noindent for any $\beta \in [\eta, 1]$. Otherwise, if $|x'-x''| \leq (t-s)^{1/2}$, using the mean-value theorem, \eqref{bound:derivative:heat:kernel} with $n=1$, the uniform $\eta$-H\"older regularity of $ [\delta a_{i, j}/\delta m](t, x, m)(.)$, the space-time inequality \eqref{space:time:inequality} and finally noting that in the current diagonal regime for any $\lambda' \in [0,1]$ and any $0 < c' < c$
\begin{equation}
\label{diagonal:regime:heat:kernel}
 \exp\left(-c \frac{|\lambda' x' + (1-\lambda') x'' - y'|^2}{t-s}\right) \leq K \exp\left(-c' \frac{|y'-x'|^2}{t-s}\right)
\end{equation}
\noindent we obtain
\begin{align*}
&\Big| \int_{\rr^d}  \frac{\delta a_{i, j}}{\delta m}(t, x, \Theta^{(m)}_{\lambda, t})(y') [p_{m}(\mu, s, t, x', y') - p_{m}(\mu, s, t, x'', y')]  \, dy' \Big| \\
& = \Big| \int_0^1  \int_{\rr^d} \frac{\delta a_{i, j}}{\delta m}(t, x, \Theta^{(m)}_{\lambda, t})(y') \, \partial_x p_{m}(\mu, s, t, \lambda' x' + (1-\lambda') x'', y') (x'-x'') \, dy' d\lambda' \Big| \\ 
& = \Big| \int_0^1  \int_{\rr^d} \left[\frac{\delta a_{i, j}}{\delta m}(t, x, \Theta^{(m)}_{\lambda, t})(y') - \frac{\delta a_{i, j}}{\delta m}(t, x, \Theta^{(m)}_{\lambda, t})(\lambda x' + (1-\lambda) x'')\right]\\
 & \quad \quad  \partial_x p_{m}(\mu, s, t, \lambda' x' + (1-\lambda') x'', y') (x'-x'') \, dy' d\lambda' \Big|\\ 
& \leq K \frac{|x'-x''|}{(t-s)^{\frac{1-\eta}{2}}} \\
&  \leq K \frac{|x'-x''|^\beta}{(t-s)^{\frac{\beta-\eta}{2}}}.
\end{align*}
This concludes the proof of \eqref{holder:regularity:estimate:functional:derivative}. Now, from \eqref{holder:regularity:estimate:functional:derivative}, we directly deduce
$$
| {\rm I} | \leq K_\beta \frac{W_2^{\beta}(\mu, \mu')}{(t-s)^{\frac{\beta-\eta}{2}}}.
$$

In order to deal with ${\rm II}$, we first remark that if $W_2(\mu, \mu') \leq (t-s)^{1/2}$, then for any two random variables $\xi$ and $\xi'$ in $\mathbb{L}^{2}$ with respective law $\mu$ and $\mu'$, one can find $\lambda \in [0,1]$ such that
\begin{align*}
| \Delta_{\mu, \mu'} p_{m+1}(\mu, s, t , x, z) | & = \left| \E\Big[\partial_\mu p_{m+1}([\lambda \xi + (1-\lambda) \xi'], s, t, x, z)(\lambda \xi + (1-\lambda) \xi') (\xi-\xi') \Big] \right| \\
& \leq K\frac{\mathbb{E}[|\xi -\xi'|]}{(t-s)^{\frac{1-\eta}{2}}} \, g(c(t-s), z-x)
\end{align*}

\noindent where we used \eqref{first:second:estimate:induction:decoupling:mckean} (see also Remark \ref{remark:convergence:series:gaussian:upper:bounds}). Using the Cauchy-Schwarz inequality and taking the infimum over the random variables $\xi$ and $\xi'$ with prescribed marginal distributions $\mu$ and $\mu'$, we deduce
\begin{equation}
\label{diagonal:estimate:delta:mes:heat:kernel}
|\Delta_{\mu, \mu'} p_{m+1}(\mu, s, t , x, z)  | \leq K \frac{W_2(\mu, \mu')}{(t-s)^{\frac{1-\eta}{2}}} g(c(t-s), z-x) \leq K \frac{W^{\beta}_2(\mu, \mu')}{(t-s)^{\frac{\beta-\eta}{2}}}g(c(t-s), z-x).
\end{equation}

\noindent for any $\beta \in [0,1]$, recalling that $W_2(\mu, \mu') \leq (t-s)^{1/2}$. Otherwise, if $W_2(\mu, \mu') > (t-s)^{1/2}$, the Gaussian upper-estimate \eqref{bound:derivative:heat:kernel} directly yields
$$
 |\Delta_{\mu, \mu'} p_{m+1}(\mu, s, t , x, z)  | \leq K \frac{W^{\beta}_2(\mu, \mu')}{(t-s)^{\frac{\beta}{2}}}g(c(t-s), z-x).
$$

The preceding estimate combined with the uniform $\eta$-H\"older regularity of $ [\delta a_{i, j}/\delta m](t, x, m)(.)$ and the space-time inequality \eqref{space:time:inequality} yield
$$
| {\rm II} | \leq K \frac{W^{\beta}_2(\mu, \mu')}{(t-s)^{\frac{\beta-\eta}{2}}}.
$$
Gathering the estimates on ${\rm I}$ and ${\rm II}$ allows to conclude that
$$
|a_{i, j}(t, x, [X^{s, \xi, (m)}_t])  -  a_{i, j}(t, x, [X^{s, \xi', (m)}_t])|\leq  K \frac{W^{\beta}_2(\mu, \mu')}{(t-s)^{\frac{\beta-\eta}{2}}}
$$
\noindent for any $\beta \in [\eta, 1]$. The corresponding estimate on $|b_{i}(t, x, [X^{s, \xi, (m)}_t])  -  b_{i}(t, x, [X^{s, \xi', (m)}_t])|$ is obtained in a completely analogous manner. We thus omit its proof.\\

 \noindent \emph{Step 2: proofs of \eqref{gaussian:bound:diff:deriv:hat:pm:same:time} and \eqref{gaussian:bound:diff:deriv:hat:pm:different:time}.}\\
 
The mean-value theorem and \eqref{diff:mes:drift:diff:coefficients} yield
\begin{align*}
| & \Delta_{\mu, \mu'} \widehat{p}_{m+1}(\mu, s, t, x, z)  | \nonumber \\
& \leq \frac{K}{t-s} \int_s^t \max_{i, j} | a_{i, j}(r, z, [X^{s, \xi, (m)}_r]) - a_{i, j}(r, z, [X^{s, \xi', (m)}_r])| \, dr \, g(c(t-s), z-x) \nonumber \\
& \leq K \frac{W^{\beta}_2(\mu, \mu')}{(t-s)^{\frac{\beta-\eta}{2}}} \, g(c(t-s), z-x) \nonumber
\end{align*}
\noindent and
\begin{align*}
| & \Delta_{\mu, \mu'} \widehat{p}_{m+1}(\mu, s, r, t, x, z)  | \nonumber \\
& \leq \frac{K}{t-r} \int_r^t \max_{i, j} | a_{i, j}(r', z, [X^{s, \xi, (m)}_{r'}]) - a_{i, j}(r', z, [X^{s, \xi', (m)}_{r'}])| \, dr' \, g(c(t-r), z-x) \nonumber \\
& \leq K\frac{W^{\beta}_2(\mu, \mu')}{(r-s)^{\frac{\beta-\eta}{2}}} \, g(c(t-r), z-x). \nonumber
\end{align*}

 Similarly, differentiating $n$-times the maps $x \mapsto \widehat{p}_{m+1}(\mu, s, t, x, z), \, \widehat{p}_{m+1}(\mu, s, r, t, x, z)$, from the mean-value theorem, \eqref{diff:mes:drift:diff:coefficients} and the space-time inequality \eqref{space:time:inequality}, for all $ \alpha \in [\eta,1]$ and $n\in \left\{1, 2, 3\right\}$, one gets
\begin{align*}
 | \Delta_{\mu, \mu'} \partial^{n}_x \widehat{p}_{m}(\mu, s, t, x, z) | \leq K \frac{W^{\beta}_2(\mu, \mu')}{(t-s)^{\frac{n+\beta-\eta}{2}}} \, g(c(t-s), z-x) 
\end{align*}
\noindent and
\begin{align*}
 | \Delta_{\mu, \mu'} \partial^{n}_x \widehat{p}_{m}(\mu, s, r, t, x, z) | \leq K \frac{W^{\beta}_2(\mu, \mu')}{(t-r)^{\frac{n}{2}}(r-s)^{\frac{\beta-\eta}{2}}} \, g(c(t-r), z-x). 
\end{align*}

\noindent \emph{Step 3: proof of \eqref{gaussian:bound:diff:pmp1:mu:mup}.}\\

We here use the decomposition \eqref{other:representation:parametrix:series} which writes
$$
p_{m+1}(\mu, s, t ,x ,z) = \widehat{p}_{m+1}(\mu, s, t, x , z) + \widehat{p}_{m+1} \otimes \Phi_{m+1}(\mu, s, t , x, z)
$$
\noindent where, by \eqref{Gaussian:estimate:Phim}, the last term satisfies
$$
|\widehat{p}_{m+1} \otimes \Phi_{m+1}(\mu, s, t , x, z)| \leq K (t-s)^{\eta/2} g(c(t-s), z-x).
$$
We thus deduce from \eqref{gaussian:bound:diff:deriv:hat:pm:same:time} that if $W_2(\mu, \mu') > (t-s)^{\frac12}$, one has
\begin{align*}
|\Delta_{\mu, \mu'} p_{m+1}(\mu, s, t , x, z)  | & \leq  |\Delta_{\mu, \mu'} \widehat{p}_{m+1}(\mu, s, t , x, z) | + K (t-s)^{\frac{\eta}{2}} g(c(t-s), z-x) \\
& \leq K \frac{W^{\beta}_2(\mu, \mu')}{(t-s)^{\frac{\beta-\eta}{2}}} \, g(c(t-s), z-x).
\end{align*}

The previous estimates together with \eqref{diagonal:estimate:delta:mes:heat:kernel} concludes the proof of \eqref{gaussian:bound:diff:pmp1:mu:mup}.\\

\noindent \emph{Step 4: proof of \eqref{diff:mes:with:holder:reg:space:drift:diff:coefficients}.}\\

As in step 1, we only deal with the difference $a_{i, j} (t, x, [X^{s, \xi, (m)}_t])  - a_{i, j}(t, x, [X^{s, \xi', (m)}_t]) - [ a_{i, j} (t, y, [X^{s, \xi, (m)}_t])  - a_{i, j}(t, y, [X^{s, \xi', (m)}_t]) ] $. Having in mind the notation of step 1, namely from \eqref{decomposition:diff:aij}, it holds
\begin{align*}
h(x) - h(y) & = a_{i, j} (t, x, [X^{s, \xi, (m)}_t])  - a_{i, j}(t, x, [X^{s, \xi', (m)}_t]) - [ a_{i, j} (t, y, [X^{s, \xi, (m)}_t])  - a_{i, j}(t, y, [X^{s, \xi', (m)}_t]) ] \\
& = \int_0^1  \int_{(\rr^d)^2} \left[\frac{\delta a_{i, j}}{\delta m}(t, x, \Theta^{(m)}_{\lambda, t})(y') -  \frac{\delta a_{i, j}}{\delta m}(t, y, \Theta^{(m)}_{\lambda, t})(y')\right] p_{m}(\mu, s, t, x', y') \, dy'  \, (\mu-\mu')(dx') \,  d\lambda   \\
&   +  \int_0^1 \int_{(\rr^d)^2} \Big(\frac{\delta a_{i, j}}{\delta m}(t, x, \Theta^{(m)}_{\lambda, t})(y') - \frac{\delta a_{i, j}}{\delta m}(t, y, \Theta^{(m)}_{\lambda, t})(y') \Big) \Delta_{\mu, \mu'} p_{m}(\mu, s, t , x', y')  \, dy'  \mu'(dx') \, d\lambda  \\
& =: {\rm I} + {\rm II}.
\end{align*}

We now follow similar lines of reasonings as those developed in step 1. We first prove that there exists a positive constant $K:= K(T,\HR, \HE)$ such that for any $\alpha \in [0,\eta]$, any $\beta \in [\alpha,1]$ and any $x', \, x'' \in \mathbb{R}^d$
\begin{align}
\Big| \int_{\mathbb{R}^d} &  \left[\frac{\delta a_{i, j}}{\delta m}(t, x, \Theta^{(m)}_{\lambda, t})(y') -  \frac{\delta a_{i, j}}{\delta m}(t, y, \Theta^{(m)}_{\lambda, t})(y')\right] [p_{m}(\mu, s, t, x', y') - p_{m}(\mu, s, t, x'', y')]  \, dy' \Big| \notag \\
& \leq K (|x-y|^{\eta-\alpha} \wedge 1) \frac{|x'-x''|^\beta}{(t-s)^{\frac{\beta-\alpha}{2}}}.\label{double:holder:regularity:estimate:functional:derivative}
\end{align}

We again split the computations into the two disjoint cases $|x'-x''| > (t-s)^{1/2}$ and $|x'-x''| \leq (t-s)^{1/2}$. If $|x'-x''|> (t-s)^{1/2}$, we again write
\begin{align*}
\int_{\rr^d} & \left[\frac{\delta a_{i, j}}{\delta m}(t, x, \Theta^{(m)}_{\lambda, t})(y') -  \frac{\delta a_{i, j}}{\delta m}(t, y, \Theta^{(m)}_{\lambda, t})(y')\right]  [p_{m}(\mu, s, t, x', y') - p_{m}(\mu, s, t, x'', y')]  \, dy' \\
& = \left[\frac{\delta a_{i, j}}{\delta m}(t, x, \Theta^{(m)}_{\lambda, t})(x') -  \frac{\delta a_{i, j}}{\delta m}(t, y, \Theta^{(m)}_{\lambda, r})(x')\right] - \left[\frac{\delta a_{i, j}}{\delta m}(t, x, \Theta^{(m)}_{\lambda, t})(x'') -  \frac{\delta a_{i, j}}{\delta m}(t, y, \Theta^{(m)}_{\lambda, t})(x'')\right] \\
& \quad + \int_{\rr^d} \left[\left[\frac{\delta a_{i, j}}{\delta m}(t, x, \Theta^{(m)}_{\lambda, t})(y') -  \frac{\delta a_{i, j}}{\delta m}(t, y, \Theta^{(m)}_{\lambda, t})(y')\right]   \right. \\
&\quad \quad \left.  -  \left[\frac{\delta a_{i, j}}{\delta m}(t, x, \Theta^{(m)}_{\lambda, t})(x') -  \frac{\delta a_{i, j}}{\delta m}(t, y, \Theta^{(m)}_{\lambda, t})(x')\right] \right] p_{m}(\mu, s, t, x', y')  dy' \\
&  \quad - \int_{\rr^d} \left[\left[\frac{\delta a_{i, j}}{\delta m}(t, x, \Theta^{(m)}_{\lambda, t})(y') -  \frac{\delta a_{i, j}}{\delta m}(t, y, \Theta^{(m)}_{\lambda, t})(y')\right]  \right. \\
& \quad \quad \left. -  \left[\frac{\delta a_{i, j}}{\delta m}(t, x, \Theta^{(m)}_{\lambda, t})(x'') -  \frac{\delta a_{i, j}}{\delta m}(t, y, \Theta^{(m)}_{\lambda, t})(x'')\right]  \right] p_{m}(\mu, s, t, x'', y')  dy' 
\end{align*}
\noindent and then combine the uniform $\eta$-H\"older regularity of $[\delta a_{i, j}/\delta m](t, ., m)(.)$ with the Gaussian upper-estimate \eqref{bound:derivative:heat:kernel} with $n=0$ and the space-time inequality \eqref{space:time:inequality} so that
\begin{align*}
\Big| \int_{\rr^d} & \left[\frac{\delta a_{i, j}}{\delta m}(t, x, \Theta^{(m)}_{\lambda, t})(y') -  \frac{\delta a_{i, j}}{\delta m}(t, y, \Theta^{(m)}_{\lambda, t})(y')\right]  [p_{m}(\mu, s, t, x', y') - p_{m}(\mu, s, t, x'', y')]  \, dy' \Big| \\
& \leq K (|x-y|^\eta \wedge |x'-x''|^\eta \wedge 1 + |x-y|^\eta \wedge (t-s)^{\eta/2} \wedge 1)\\
& \leq K (|x-y|^{\eta-\alpha} \wedge 1) (|x'-x''|^{\alpha} + (t-s)^{\alpha/2}) \\
& \leq K (|x-y|^{\eta-\alpha} \wedge 1) \frac{|x'-x''|^\beta}{(t-s)^{\frac{\beta-\alpha}{2}}}
\end{align*}
\noindent for any $\alpha \in [0,\eta]$ and any $\beta \in [\alpha, 1]$. Otherwise, if $|x'-x''| \leq (t-s)^{1/2}$, using the mean-value theorem, \eqref{bound:derivative:heat:kernel} with $n=1$, the uniform $\eta$-H\"older regularity of $ [\delta a_{i, j}/\delta m](t, ., m)(.)$ and the space-time inequality \eqref{space:time:inequality}, we obtain
\begin{align*}
&\Big| \int_{\rr^d} \left[\frac{\delta a_{i, j}}{\delta m}(t, x, \Theta^{(m)}_{\lambda, t})(y') -  \frac{\delta a_{i, j}}{\delta m}(t, y, \Theta^{(m)}_{\lambda, r})(y')\right] [p_{m}(\mu, s, t, x', y') - p_{m}(\mu, s, t, x'', y')]  \, dy' \Big| \\
& = \Big| \int_0^1  \int_{\mathbb{R}^d}\left[\frac{\delta a_{i, j}}{\delta m}(t, x, \Theta^{(m)}_{\lambda, t})(y') -  \frac{\delta a_{i, j}}{\delta m}(t, y, \Theta^{(m)}_{\lambda, t})(y')\right]\, \partial_x p_{m}(\mu, s, t, \lambda' x' + (1-\lambda') x'', y') (x'-x'') \, dy' d\lambda' \Big| \\ 
& = \Big| \int_0^1  \int_{\rr^d} \left[\left[\frac{\delta a_{i, j}}{\delta m}(t, x, \Theta^{(m)}_{\lambda, t})(y') -  \frac{\delta a_{i, j}}{\delta m}(t, y, \Theta^{(m)}_{\lambda, t})(y')\right]  \right. \\
& \left. \quad  -\left[\frac{\delta a_{i, j}}{\delta m}(t, x, \Theta^{(m)}_{\lambda, t})(\lambda x' + (1-\lambda) x'') -  \frac{\delta a_{i, j}}{\delta m}(t, y, \Theta^{(m)}_{\lambda, t})(\lambda x' + (1-\lambda) x'')\right] \right]\\
 & \quad  \quad \partial_x p_{m}(\mu, s, t, \lambda' x' + (1-\lambda') x'', y') (x'-x'') \, dy' d\lambda' \Big|\\ 
& \leq K \frac{|x'-x''|}{(t-s)^{\frac12}} ( |x-y|^\eta \wedge (t-s)^{\eta/2} \wedge 1) \\
& \leq K ( |x-y|^{\eta-\alpha} \wedge 1) \frac{|x'-x''|^\beta}{(t-s)^{\frac{\beta-\alpha}{2}}}
\end{align*}
\noindent for any $\alpha \in [0,\eta]$ and any $\beta \in [0, 1]$. This concludes the proof of \eqref{double:holder:regularity:estimate:functional:derivative}. 

From \eqref{double:holder:regularity:estimate:functional:derivative}, we thus derive
$$
| {\rm I} | \leq K ( |x-y|^{\eta-\alpha} \wedge 1) \frac{W_2^{\beta}(\mu, \mu')}{(t-s)^{\frac{\beta-\alpha}{2}}}
$$

\noindent for any $\alpha \in [0,\eta]$ and any $\beta \in [\alpha, 1]$. 
%

In order to deal with ${\rm II}$, we consider the two cases $W_2(\mu, \mu') < (t-s)^{1/2}$ and $W_2(\mu, \mu') \geq (t-s)^{1/2}$. In the first case, from \eqref{diagonal:estimate:delta:mes:heat:kernel} and the uniform $\eta$-H\"older regularity of $ [\delta a_{i, j}/\delta m](t, ., m)(y)$, we get
 \begin{align*}
&  \Big| \int_{\rr^d}  \Big(\frac{\delta a_{i, j}}{\delta m}(t, x, \Theta^{(m)}_{\lambda, t})(y') - \frac{\delta a_{i, j}}{\delta m}(t, y, \Theta^{(m)}_{\lambda, t})(y') \Big) \Delta_{\mu, \mu'} p_{m}(\mu, s, t , x', y') \, dy' \Big| \\
 & \leq K (|x-y|^\eta \wedge 1)  \frac{W^{\beta}_2(\mu, \mu')}{(t-s)^{\frac{\beta-\eta}{2}}} \\
 & \leq K ( |x-y|^{\eta-\alpha} \wedge 1) \frac{W^{\beta}_2(\mu, \mu')}{(t-s)^{\frac{\beta-\alpha}{2}}}
 \end{align*}
 \noindent so that
 $$
 |{\rm II}| \leq K |x-y|^{\eta-\alpha}\frac{W^{\beta}_2(\mu, \mu')}{(t-s)^{\frac{\beta-\alpha}{2}}}
 $$
 \noindent for any $\alpha \in [0,\eta]$ and any $\beta \in [0, 1]$.
Otherwise, if $W_2(\mu, \mu') \geq (t-s)^{1/2}$, we first write
\begin{align*}
&   \int_{\rr^d}  \Big(\frac{\delta a_{i, j}}{\delta m}(t, x, \Theta^{(m)}_{\lambda, t})(y') - \frac{\delta a_{i, j}}{\delta m}(t, y, \Theta^{(m)}_{\lambda, t})(y') \Big) \Delta_{\mu, \mu'} p_{m}(\mu, s, t , x', y') \, dy'  \\
 & =  \int_{\mathbb{R}^d}  \Big( \frac{\delta a_{i, j}}{\delta m}(t, x, \Theta^{(m)}_{\lambda, t})(y') - \frac{\delta a_{i, j}}{\delta m}(t, y, \Theta^{(m)}_{\lambda, t})(y') - \Big[\frac{\delta a_{i, j}}{\delta m}(t, x, \Theta^{(m)}_{\lambda, t})(x') - \frac{\delta a_{i, j}}{\delta m}(t, y, \Theta^{(m)}_{\lambda, t})(x')\Big]\Big) \\
 & \quad \quad \times \Delta_{\mu, \mu'} p_{m}(\mu, s, t , x', y') \, dy' 
\end{align*}
\noindent and then use the uniform $\eta$-H\"older regularity of $ [\delta a_{i, j}/\delta m](t, ., m)(.)$, the Gaussian upper-estimate \eqref{bound:derivative:heat:kernel} and finally the space-time inequality \eqref{space:time:inequality}. We thus derive
$$
|{\rm II}| \leq K (|x-y|^\eta \wedge (t-s)^{\eta/2} \wedge 1) \leq K (|x-y|^\eta \wedge (t-s)^{\eta/2}\wedge 1) \frac{W^{\beta}_2(\mu, \mu')}{(t-s)^{\frac{\beta}{2}} } \leq K ( |x-y|^{\eta-\alpha} \wedge 1) \frac{W^{\beta}_2(\mu, \mu')}{(t-s)^{\frac{\beta-\alpha}{2}}}. 
$$

Gathering the estimates on ${\rm I}$ and ${\rm II}$, we thus conclude
$$
|h(x) - h(y)| \leq K ( |x-y|^{\eta-\alpha} \wedge 1) \frac{W^{\beta}_2(\mu, \mu')}{(t-s)^{\frac{\beta-\alpha}{2}}} 
$$
\noindent  for any $\alpha \in [0,\eta]$ and any $\beta \in [\alpha, 1]$.\\

\noindent \emph{Step 5: proof of \eqref{gaussian:bound:diff:Hm:mu} and \eqref{gaussian:bound:diff:Phimp1:mu}.}\\

We use the following decomposition
\begin{equation}
\label{decomposition:delta:mes:Hmp1}
\Delta_{\mu, \mu'} \mH_{m+1}(\mu, s, r ,t , x ,z)  = {\rm I}(x) + {\rm II}(x) + {\rm III}(x) + {\rm IV}(x),
\end{equation}

\noindent with
\begin{align*}
{\rm I }(x) & :=  \sum_{i=1}^d [b_i(r, x, [X^{s, \xi, (m)}_r]) - b_i(r, x, [X^{s, \xi', (m)}_r])] \partial_{x_i}\widehat{p}_{m+1}(\mu, s, r, t, x, z), \\
{\rm II }(x) & :=  \sum_{i=1}^d b_i(r, x, [X^{s, \xi', (m)}_r]) \Delta_{\mu, \mu'} \partial_{x_i}\widehat{p}_{m+1}(\mu, s, r, t, x, z), \\
{\rm III }(x)  & := \frac12 \sum_{i, j=1}^d \Big[(a_{i, j}(r, x, [X^{s, \xi, (m)}_r]) - a_{i, j}(r, z , [X^{s, \xi, (m)}_r])) \\
& \quad \quad - (a_{i, j}(r, x, [X^{s, \xi', (m)}_r]) - a_{i, j}(r, z, [X^{s, \xi', (m)}_r])) \Big]   \partial^2_{x_i, x_j}\widehat{p}_{m+1}(\mu, s, r, t, x , z), \\
{\rm IV}(x) & : = \frac12 \sum_{i, j=1}^d   (a_{i, j}(r, x, [X^{s, \xi', (m)}_r]) - a_{i, j}(r, z, [X^{s, \xi', (m)}_r])) \Delta_{\mu, \mu'} \partial^2_{x_i, x_j} \widehat{p}_{m+1}(\mu, s, r, t, x, z).
\end{align*}

\noindent and prove the following estimates: there exists positive constants $K:=K(T, \HR, \HE), \, c:=c(\lambda) >0$ such that for all $\beta \in [\eta,1]$
\begin{align*}
 |{\rm I} (x)| & \leq K \frac{W^{\beta}_2(\mu, \mu')}{(t-r)^{\frac12}(r-s)^{\frac{\beta-\eta}{2}}} \, g(c(t-r) , z- x), \\
  |{\rm II} (x)| & \leq K \frac{W^{\alpha}_2(\mu, \mu')}{(t-r)^{\frac12}(r-s)^{\frac{\beta-\eta}{2}}} \, g(c(t-r), z-x),   \\
|{\rm III} (x)| & \leq K \left\{\frac{1}{(t-r)^{1-\frac{\eta}{2}}(r-s)^{\frac{\beta}{2}}} \wedge\frac{1}{(t-r) (r-s)^{\frac{\beta-\eta}{2}}} \right\} \,W^{\beta}_2(\mu, \mu') \,  g(c(t-r), z - x), \\
|{\rm IV} (x)| & \leq  K  \frac{W^{\beta}_2(\mu, \mu')}{(t-r)^{1-\frac{\eta}{2}}(r-s)^{\frac{\beta-\eta}{2}}} \, g(c(t-r), z - x).
\end{align*}

 The estimates on ${\rm I}$, ${\rm II}$ and ${\rm IV}$ follow from \eqref{diff:mes:drift:diff:coefficients}, \eqref{gaussian:bound:diff:deriv:hat:pm:different:time}, the uniform $\eta$-H\"older regularity of $a_{i, j}(t, ., m)$, the boundedness of $b_i$ and the space-time inequality \eqref{space:time:inequality}. We now deal with ${\rm III}$ which is the tricky part of the proof. From \eqref{diff:mes:with:holder:reg:space:drift:diff:coefficients} (with $\alpha=0$ and $\alpha=\eta$) and \eqref{standard}, we obtain 
\begin{align*}
 |{\rm III} | &  \leq K  \left\{ \frac{|z-x|^\eta}{(t-r) (r-s)^{\frac{\beta}{2}}} \wedge \frac{1}{(t-r)(r-s)^{\frac{\beta-\eta}{2}}}\right\} \, W_2^{\beta}(\mu, \mu')\, g(c(t-r), z - x) \\
 & \leq K  \left\{ \frac{1}{(t-r)^{1-\frac{\eta}{2}} (r-s)^{\frac{\beta}{2}}} \wedge \frac{1}{(t-r)(r-s)^{\frac{\beta-\eta}{2}}}\right\} \, W_2^{\beta}(\mu, \mu')\, g(c(t-r), z - x) 
\end{align*}
\noindent where we used the space-time inequality \eqref{space:time:inequality} for the last inequality. This concludes the proof of \eqref{gaussian:bound:diff:Hm:mu}. \\

In order to prove \eqref{gaussian:bound:diff:Phimp1:mu}, we first remark that using \eqref{gaussian:bound:diff:Hm:mu} if $W_2(\mu, \mu') \leq (r-s)^{1/2}$ (which is thus valid for any $\beta \in [0,1]$) and the standard estimate \eqref{iter:parametrix:kernel} with $k=1$ otherwise yields
 \begin{equation}
 \label{diff:mes:parametrix:kernel:mp1}
 |\Delta_{\mu, \mu'} \mH_{m+1}(\mu, s, r, t, x , z)| \leq K \frac{W^{\beta}_2(\mu, \mu')}{(t-r)^{1-\frac{\eta}{2}}(r-s)^{\frac{\beta}{2}}} \, g(c(t-r), z-x)
 \end{equation}
 \noindent for any $\beta \in [0,1]$. From the identity $\Phi_{m+1}(\mu, s, r, t, x, z) = \mH_{m+1}(\mu, s, r, t, x, z) + \mH_{m+1}\otimes \Phi_{m+1}(\mu, s, r, t, x, z)$, we then obtain the following relation 
\begin{align}
\Delta_{\mu, \mu'} \Phi_{m+1}(\mu, s, r, t, x, z) & = \Delta_{\mu, \mu'} \mH_{m+1}(\mu, s, r, t, x, z) \nonumber  \\
& \quad + \int_r^t \int_{\mathbb{R}^d} \Delta_{\mu, \mu'} \mH_{m+1}(\mu, s, r, r', x, y)  \Phi_{m+1}(\mu, s , r', t, y, z) \, dr' \, dy \label{delta:measure:Phi:mp1}\\
& \quad + \int_r^t \int_{\mathbb{R}^d}  \mH_{m+1}(\mu', s, r, r', x, y)  \Delta_{\mu, \mu'} \Phi_{m+1}(\mu, s, r', t, y, z) \, dr' \, dy \nonumber
\end{align}
\noindent which in turn, using the estimates \eqref{diff:mes:parametrix:kernel:mp1} and \eqref{Gaussian:estimate:Phim}, yields
\begin{align*}
| \Delta_{\mu, \mu'} \Phi_{m+1}(\mu, s, r, t, x, z) | & \leq K \frac{W^{\beta}_2(\mu, \mu')}{(t-r)^{1-\frac{\eta}{2}}(r-s)^{\frac{\beta}{2}}} \, g(c(t-r), z-x) \\
& + K \int_{r}^t \int_{\mathbb{R}^d} \frac{1}{(t-r)^{1-\frac{\eta}{2}}} g(c(t-r'), y-x) \, |\Delta_{\mu, \mu'} \Phi_m(\mu, s, r', t, y, z)| \, dy\, dr'. 
\end{align*}

The space-time kernel $[r, t] \times \mathbb{R}^d \ni(r', x) \mapsto (t-r')^{-1+\eta/2} (r'-s)^{-\beta} g(c(t-r'), z-x)$ being non-singular, iterating the previous inequality implies
$$
| \Delta_{\mu , \mu'} \Phi_{m+1}(\mu, s, r, t, x, z) | \leq  K \frac{W^{\beta}_2(\mu, \mu')}{(t-r)^{1-\frac{\eta}{2}}(r-s)^\beta} \, g(c(t-r), z-x)
$$

\noindent for any $\beta \in [0,1]$. The proof of \eqref{gaussian:bound:diff:Phimp1:mu} is now complete. \\

\noindent \emph{Step 6: proof of \eqref{gaussian:bound:diff:Hm:mu:holder:reg}.}

In the off-diagonal regime $|x-y|^2 \geq t-r$, the estimate \eqref{gaussian:bound:diff:Hm:mu:holder:reg} directly follows \eqref{gaussian:bound:diff:Hm:mu}. From now on, we assume that the diagonal regime $|x-y|^2 < t-r$ is in force. It suffices to investigate the H\"older regularity of each term appearing in the decomposition \eqref{decomposition:delta:mes:Hmp1}. 

We write
$$
{\rm I}(x) - {\rm I}(y) = {\rm I}_1(x) - {\rm I}_1(y) + {\rm I}_2(x) - {\rm I}_2(y)
$$
\noindent with
\begin{align*}
& {\rm I}_1(x)  - {\rm I}_1(y) \\
&\quad :=   \sum_{i=1}^d [b_i(r, x, [X^{s, \xi, (m)}_r]) - b_i(r, x, [X^{s, \xi', (m)}_r]) - (b_i(r, y, [X^{s, \xi, (m)}_r]) - b_i(r, y, [X^{s, \xi', (m)}_r]) )] \partial_{x_i}\widehat{p}_{m+1}(\mu, s, r, t, x, z)
\end{align*}
\noindent and
\begin{align*}
& {\rm I}_2(x)  - {\rm I}_2(y) \\
&\quad :=   \sum_{i=1}^d (b_i(r, x, [X^{s, \xi, (m)}_r]) - b_i(r, x, [X^{s, \xi', (m)}_r])) [\partial_{x_i}\widehat{p}_{m+1}(\mu, s, r, t, x, z) - \partial_{x_i}\widehat{p}_{m+1}(\mu, s, r, t, y, z)].
\end{align*}

From \eqref{diff:mes:with:holder:reg:space:drift:diff:coefficients} and \eqref{standard}, it directly follows
\begin{align*}
| {\rm I}_1(x)  - {\rm I}_1(y) | & \leq K W^{\beta}_2(\mu, \mu') \frac{|x-y|^{\eta-\alpha}}{(t-r)^{\frac12} (r-s)^{\frac{\beta - \alpha}{2}}} \,  g(c(t-r), z-x) 
\end{align*}
\noindent for any $\alpha \in [0,\eta]$ and any $\beta \in [\alpha, 1]$. From \eqref{diff:mes:drift:diff:coefficients} and standard regularity estimates of Gaussian kernels
 \begin{align*}
| {\rm I}_2(x)  - {\rm I}_2(y) | & \leq K \frac{W^{\beta}_2(\mu, \mu')}{(r-s)^{\frac{\beta-\eta}{2}}}  \frac{|x-y|^{\eta-\alpha}}{(t-r)^{\frac{1+\eta-\alpha}{2}}}   \, \left\{ g(c(t-r), z-x) + g(c(t-r), z-y) \right\} 
\end{align*}
\noindent for any $\alpha \in [0,\eta]$ and any $\beta \in [\eta,1]$.
Similarly, we write
$$
{\rm II }(x) -{\rm II}(y)  = {\rm II}_1(x) - {\rm II}_1(y) + {\rm II}_2(x) - {\rm II}_2(y)  
$$
\noindent with
$$
{\rm II}_1(x) - {\rm II}_1(y) := \sum_{i=1}^d [ b_i(r, x, [X^{s, \xi', (m)}_r]) -  b_i(r, y, [X^{s, \xi', (m)}_r]) ] \Delta_{\mu, \mu'} \partial_{x_i}\widehat{p}_{m+1}(\mu, s, r, t, x, z), 
$$
\noindent and
$$
{\rm II}_2(x) - {\rm II}_2(y) := \sum_{i=1}^d  b_i(r, x, [X^{s, \xi', (m)}_r])  \Delta_{\mu, \mu'} [\partial_{x_i}\widehat{p}_{m+1}(\mu, s, r, t, x, z) - \partial_{x_i}\widehat{p}_{m+1}(\mu, s, r, t, y, z)].
$$

The uniform $\eta$-H\"older regularity of $b_i(t, ., m)$ and \eqref{gaussian:bound:diff:deriv:hat:pm:different:time} yield
\begin{align*}
| {\rm II}_1(x) - {\rm II}_1(y) | & \leq K |x-y|^\eta  \frac{W^{\beta}_2(\mu, \mu')}{(t-r)^{\frac12}(r-s)^{\frac{\beta-\eta}{2}}} \, g(c(t-r), z-x)\\
& \leq K |x-y|^{\eta-\alpha}  \frac{W^{\beta}_2(\mu, \mu')}{(t-r)^{\frac{1-\alpha}{2}} (r-s)^{\frac{\beta-\eta}{2}}} g(c(t-r), z-x)
\end{align*}

\noindent recalling that $|x-y|^2 < t-r$ for the last inequality. From the mean-value theorem, \eqref{gaussian:bound:diff:deriv:hat:pm:different:time} and \eqref{diagonal:regime:heat:kernel}, we obtain
\begin{align*} 
|\Delta_{\mu, \mu'} & [\partial_{x_i}\widehat{p}_{m+1}(\mu, s, r, t, x, z) -\partial_{x_i}\widehat{p}_{m+1}(\mu, s, r, t, y, z)]| \\
& \leq \int_0^1 |\Delta_{\mu, \mu'} \partial^2_{x, x_i} \widehat{p}_{m+1}(\mu, s, r, t, \lambda x +(1-\lambda) y, z)| |x-y| \, d\lambda \\
& \leq K |x-y|  \frac{W^{\beta}_2(\mu, \mu') }{(t-r) (r-s)^{\frac{\beta-\eta}{2}}}  \,  g(c(t-r), z-x) \\
& \leq K |x-y|^{\eta-\alpha} \frac{W^{\beta}_2(\mu, \mu')}{(t-r)^{\frac{1+\eta-\alpha}{2}} (r-s)^{\frac{\beta-\eta}{2}}} \, g(c(t-r), z-x) 
\end{align*}
\noindent which, by the boundedness of the drift coefficient, yields
\begin{align*}
| {\rm II}_2(x) - {\rm II}_2(y) | &  \leq K |x-y|^{\eta-\alpha} \frac{W^{\beta}_2(\mu, \mu')}{(t-r)^{\frac{1+\eta-\alpha}{2}} (r-s)^{\frac{\beta-\eta}{2}}} \, g(c(t-r), z-x) 
\end{align*}
\noindent for any $\alpha \in [0,\eta]$ and any $\beta \in [\eta, 1]$.

We then use the decomposition
$$
{\rm III}(x) - {\rm III}(y) = {\rm III}_1(x)  - {\rm III}_1(y)  + {\rm III}_2(x)  - {\rm III}_2(y)  
$$
\noindent with
\begin{align*}
{\rm III}_1(x)  - {\rm III}_1(y) &:= \frac12 \sum_{i, j=1}^d \Big[(a_{i, j}(r, x, [X^{s, \xi, (m)}_r]) - a_{i, j}(r, x, [X^{s, \xi', (m)}_r])) \\
& \quad \quad - (a_{i, j}(r, y, [X^{s, \xi, (m)}_r]) - a_{i, j}(r, y, [X^{s, \xi', (m)}_r])) \Big] \partial^2_{x_i, x_j}\widehat{p}_{m+1}(\mu, s, r, t, x , z) 
\end{align*}
\noindent and
\begin{align*}
{\rm III}_2(x) & - {\rm III}_2(y) := \frac12 \sum_{i, j=1}^d \Big[(a_{i, j}(r, x, [X^{s, \xi, (m)}_r]) - a_{i, j}(r, z , [X^{s, \xi, (m)}_r])) \\
& \quad \quad - (a_{i, j}(r, x, [X^{s, \xi', (m)}_r]) - a_{i, j}(r, z, [X^{s, \xi', (m)}_r])) \Big] [\partial^2_{x_i, x_j}\widehat{p}_{m+1}(\mu, s, r, t, x , z) - \partial^2_{x_i, x_j}\widehat{p}_{m+1}(\mu, s, r, t, y , z)].
\end{align*}

On the one hand, from \eqref{diff:mes:with:holder:reg:space:drift:diff:coefficients}, one has
\begin{align*}
|{\rm III}_1(x)  - {\rm III}_1(y)|&  \leq K |x-y|^{\eta-\alpha} \frac{W^{\beta}_2(\mu, \mu') }{(t-r) (r-s)^{\frac{\beta-\alpha}{2}}}  \, g(c(t-r), z-x)
\end{align*}
\noindent for any $\alpha \in [0,\eta]$ and any $\beta \in [\alpha,1]$. On the other hand, from \eqref{diff:mes:with:holder:reg:space:drift:diff:coefficients} (with $\alpha=\eta$) and standard regularity estimates of Gaussian kernels, one gets
\begin{align*}
|{\rm III}_2(x)  - {\rm III}_2(y)|&  \leq  K |x-y|^{\eta-\alpha} \frac{W^{\beta}_2(\mu, \mu') }{(t-r)^{1+\frac{\eta-\alpha}{2}}(r-s)^{\frac{\beta-\eta}{2}}}  \left\{ g(c(t-r), z-x) + g(c(t-r), z-y) \right\}
\end{align*}
\noindent for any $ \alpha \in [0,\eta]$ and any $\beta \in [\eta,1]$.

Finally, we write
$$
{\rm IV}(x) -{\rm IV}(y) = {\rm IV}_1(x)-{\rm IV}_1(y) + {\rm IV}_2(x) - {\rm IV}_2(y)  
$$
\noindent with
\begin{align*}
 {\rm IV}_1(x) - {\rm IV}_1(y) & = \frac12 \sum_{i, j=1}^d  [a_{i, j}(r, x, [X^{s, \xi', (m)}_r]) - a_{i, j}(r, y, [X^{s, \xi', (m)}_r]) ]  \Delta_{\mu, \mu'} \partial^2_{x_i, x_j} \widehat{p}_{m+1}(\mu, s, r, t, x, z)
\end{align*}
\noindent and
\begin{align*}
  {\rm IV}_2(x) - {\rm IV}_2(y)  & = \frac12 \sum_{i, j=1}^d  [a_{i, j}(r, x, [X^{s, \xi', (m)}_r]) - a_{i, j}(r, z, [X^{s, \xi', (m)}_r])] \\
 & \quad \times \Delta_{\mu, \mu'}[ \partial^2_{x_i, x_j} \widehat{p}_{m+1}(\mu, s, r, t, x, z) - \partial^2_{x_i, x_j} \widehat{p}_{m+1}(\mu, s, r, t, y, z)].
\end{align*}

From the uniform $\eta$-H\"older regularity of $ a_{i, j}(t, ., m)$ and \eqref{gaussian:bound:diff:deriv:hat:pm:different:time}, we directly obtain
\begin{align*}
| {\rm IV}_1(x) - {\rm IV}_1(y)| & \leq K (|x-y|^\eta \wedge 1) \frac{W^{\beta}_2(\mu, \mu')}{(t-r)(r-s)^{\frac{\beta-\eta}{2}}} \, g(c(t-r), z-x) \\
& \leq K (|x-y|^{\eta-\alpha} \wedge 1) \frac{ W^{\beta}_2(\mu, \mu')}{(t-r)^{1-\frac{\alpha}{2}}(r-s)^{\frac{\beta-\eta}{2}}} \,   g(c(t-r), z-x) 
\end{align*}
\noindent  for any $\beta \in [\eta,1]$ and any $ \alpha \in [0,\eta]$.

From the mean-value theorem, \eqref{gaussian:bound:diff:deriv:hat:pm:different:time} and \eqref{diagonal:regime:heat:kernel}, recalling that the diagonal regime $|x-y|^2 < (t-r)$ is in force, we obtain
\begin{align*} 
|\Delta_{\mu, \mu'} & \partial^2_{x_i, x_j}\widehat{p}_{m+1}(\mu, s, r, t, x, z) - \Delta_{\mu, \mu'} \partial^2_{x_i, x_j}\widehat{p}_{m+1}(\mu, s, r, t, y, z)| \\
& \leq \int_0^1 |\Delta_{\mu, \mu'} \partial^3_{x, x_i, x_j} \widehat{p}_{m+1}(\mu, s, r, t, \lambda x +(1-\lambda) y, z)| |x-y| \, d\lambda \\
& \leq K |x-y| \frac{W^{\beta}_2(\mu, \mu')}{(t-r)^{\frac32} (r-s)^{\frac{\beta-\eta}{2}}}  \, g(c(t-r), z-x)  \\
& \leq K |x-y|^{\eta-\alpha}  \frac{ W^{\beta}_2(\mu, \mu')}{(t-r)^{1+\frac{\eta-\alpha}{2}} (r-s)^{\frac{\beta-\eta}{2}}} \, g(c(t-r), z-x) 
\end{align*}
\noindent which in turn, using the uniform $\eta$-H\"older regularity of $ a_{i, j}(t, ., m)$ and the space-time inequality \eqref{space:time:inequality} yields
\begin{align*}
| {\rm IV}_2(x) - {\rm IV}_2(y) | 
& \leq K |x-y|^{\eta-\alpha}  \frac{ W^{\beta}_2(\mu, \mu')}{(t-r)^{1-\frac{\alpha}{2}} (r-s)^{\frac{\beta-\eta}{2}}} \, g(c(t-r), z-x) 
\end{align*}

\noindent for any $\beta \in [\eta, 1]$ and any $\alpha \in [0,\eta]$. We conclude the proof of \eqref{gaussian:bound:diff:Hm:mu:holder:reg} by putting together the previous estimates. \\

\noindent \emph{Step 7: proof of \eqref{gaussian:bound:diff:second:deriv:p:hat:mu:holder:reg}.} \\

We first remark that if $W^2_2(\mu, \mu') >  t-s$, then from \eqref{diff:freezing:point:deriv:space:phat}, we obtain
\begin{align*}
 | \Delta_{\mu, \mu'} & [\partial^2_{x_i, x_j} [\widehat{p}^{y}_{m+1}  -  \widehat{p}^{y'}_{m+1}](\mu, s ,t , x ,z )]  |    \\
& \leq |\partial^2_{x_i, x_j} \widehat{p}^{y}_{m+1}(\mu, s ,t , x ,z ) - \partial^2_{x_i, x_j} \widehat{p}^{y'}_{m+1}(\mu, s ,t , x ,z )|  + |\partial^2_{x_i, x_j} \widehat{p}^{y}_{m+1}(\mu', s ,t , x ,z ) - \partial^2_{x_i, x_j} \widehat{p}^{y'}_{m+1}(\mu', s ,t , x ,z )|  \\
 & \leq K (|y-y'|^{\eta} \wedge 1) \frac{1}{t-s }  \,  g(c(t-s), z - x) \\
 & \leq K (|y-y'|^{\eta} \wedge 1) \frac{ W^{\beta}_2(\mu, \mu')}{(t-s)^{1+\frac{\beta}{2}} }  \,  g(c(t-s), z - x)
\end{align*}
\noindent for any $\beta \in [0,1]$.
From now on we assume that $W^{2}_2(\mu, \mu') \leq t-s$. Starting from the relation 
$$
\partial^{2}_{x_i , x_j} \widehat{p}^{y}_{m+1}(\mu, s, t, x, z) = H^{i, j}_2(\int_s^t a(r, y, [X^{s, \xi , (m)}_r]) \, dr , z-x) \widehat{p}^{y}_{m+1}(\mu, s, t, x, z),
$$
\noindent we derive the following decomposition
$$
\Delta_{\mu, \mu'} [\partial^{2}_{x_i , x_j} ( \widehat{p}^{y}_{m+1} - \widehat{p}^{y'}_{m+1}) (\mu, s, t, x, z) ]  = {\rm I} + {\rm II} + {\rm III} + {\rm IV}
$$
\noindent with
\begin{align*}
{\rm I}  &  := \Delta_{\mu, \mu'} \Big[ H^{i, j}_2\Big(\int_s^t a(r, y, [X^{s, \xi , (m)}_r]) \, dr , z-x\Big) - H^{i, j}_2\Big(\int_s^t a(r, y', [X^{s, \xi , (m)}_r]) \, dr , z-x\Big) \Big] \widehat{p}^{y}_{m+1}(\mu, s, t, x, z), \\
{\rm II} &  :=  \Big[H^{i, j}_2\Big(\int_s^t a(r, y, [X^{s, \xi' , (m)}_r]) \, dr , z-x\Big) - H^{i, j}_2\Big(\int_s^t a(r, y', [X^{s, \xi' , (m)}_r]) \, dr , z-x\Big)\Big] \Delta_{\mu, \mu'} \widehat{p}^{y}_{m+1}(\mu, s, t, x, z), \\
{\rm III} & := \Delta_{\mu, \mu'} H^{i, j}_2\Big(\int_s^t a(r, y', [X^{s, \xi , (m)}_r] ) \, dr , z-x\Big) \Big[ \widehat{p}^{y}_{m+1}(\mu, s, t, x, z) - \widehat{p}^{y'}_{m+1}(\mu, s, t, x, z) \Big], \\
{\rm IV} & := H^{i, j}_2\Big(\int_s^t a(r, y', [X^{s, \xi' , (m)}_r]) \, dr , z-x\Big) \Delta_{\mu, \mu'}\Big[ \widehat{p}^{y}_{m+1}(\mu, s, t, x, z) - \widehat{p}^{y'}_{m+1}(\mu, s, t, x, z) \Big].
\end{align*}

We now establish an appropriate estimate for each term of the above decomposition. Before doing so, we introduce some notations, namely, the invertible matrix $A^{s, t}_\lambda(\mu) := \int_s^t [\lambda a(r, y, [X^{s, \xi, (m)}_r]) + (1-\lambda) a(r, y', [X^{s, \xi, (m)}_r])] \, dr$,  and the associated Gaussian kernel $\widehat{p}^{\lambda, y, y'}_{m+1}(\mu, s, t, x, z) = g(A^{s, t}_\lambda(\mu), z-x)$, $\lambda \in [0,1]$. Now, the mean-value theorem yields
\begin{align}
& \left(\int_s^t a(r, y, [X^{s, \xi , (m)}_r])  \, dr  \right)^{-1}  - \left(\int_s^t a(r, y', [X^{s, \xi , (m)}_r]) \, dr  \right)^{-1} \nonumber \\
& \quad = \int_0^1 (A^{s, t}_{\lambda}(\mu))^{-1} \int_s^t [a(r, y, [X^{s, \xi, (m)}_r]) -  a(r, y', [X^{s, \xi, (m)}_r])] \, dr (A^{s, t}_{\lambda}(\mu))^{-1} \, d\lambda. \label{diff:inverse:diffusion:matrix:freezing:point}
\end{align}

On the one hand, using \eqref{diff:mes:drift:diff:coefficients}, we obtain
\begin{align}
|\Delta_{\mu, \mu'}( A^{s, t}_\lambda(\mu) ^{-1})| & \leq \frac{K}{(t-s)^2} \int_s^t (|\Delta_{\mu, \mu'} a(r, y, [X^{s, \xi, (m)}_r])| + |\Delta_{\mu, \mu'} a(r, y', [X^{s, \xi, (m)}_r])| ) \, dr \nonumber \\
& \leq K \frac{W^{\beta}_2(\mu, \mu')}{(t-s)^2} \int_s^t \frac{1}{(r-s)^{\frac{\beta-\eta}{2}}} \, dr \nonumber \\
& \leq K  \frac{W^{\beta}_2(\mu, \mu')}{(t-s)^{1+ \frac{\beta-\eta}{2}}} \label{diff:measure:inverse:A:s:t:lambda}
\end{align}
\noindent for any $\beta\in [\eta, 1]$. On the other hand, from \eqref{diff:mes:with:holder:reg:space:drift:diff:coefficients} with $\alpha=0$, we obtain 
\begin{align}
\Big|\int_s^t \Delta_{\mu, \mu'}[a(r, y, [X^{s, \xi, (m)}_{r}]) - a(r, y', [X^{s, \xi, (m)}_{r}])]  \, dr \Big| \leq K (|y-y'|^\eta \wedge 1) W^{\beta}_2(\mu ,\mu') (t-s)^{1 - \frac{\beta}{2}} \label{diff:mes:integral:diff:mes:diffusion:coeff}
\end{align}
\noindent for any $\beta \in [0,1]$. The two previous estimates with \eqref{diff:inverse:diffusion:matrix:freezing:point} thus imply
\begin{align}
\label{diff:measure:diff:inverse:diffusion:matrix:freezing:point}
\Delta_{\mu, \mu'}\Big[ & \left(\int_s^t a(r, y, [X^{s, \xi , (m)}_r])  \, dr  \right)^{-1}   - \left(\int_s^t a(r, y', [X^{s, \xi , (m)}_r]) \, dr  \right)^{-1} \Big] \\
& \quad \leq K (|y-y'|^\eta \wedge 1) \frac{W^{\beta}_2(\mu, \mu')}{(t-s)^{1+\frac{\beta}{2}}}. \nonumber
\end{align}

From the very definition of $H^{i, j}_2$ and the estimate \eqref{diff:measure:diff:inverse:diffusion:matrix:freezing:point}, one directly obtains
\begin{align*}
\Big|\Delta_{\mu, \mu'} \Big[ H^{i, j}_2\Big(\int_s^t a(r, y, [X^{s, \xi , (m)}_r]) \, dr , z-x\Big) & - H^{i, j}_2\Big(\int_s^t a(r, y', [X^{s, \xi , (m)}_r]) \, dr , z-x\Big) \Big]\Big| \\
& \leq K (|y-y'|^\eta \wedge 1) \frac{W^{\beta}_2(\mu, \mu')}{(t-s)^{1+\frac{\beta}{2}}} \left\{ \frac{|z-x|^2}{t-s} + 1\right\}
\end{align*}
\noindent which in turn using the space-time inequality \eqref{space:time:inequality} yields
$$
|{\rm I}| \leq  K (|y-y'|^\eta \wedge 1) \frac{W^{\beta}_2(\mu, \mu')}{(t-s)^{1+\frac{\beta}{2}}} \, g(c(t-s), z-x)
$$
\noindent for any $\beta \in [\eta, 1]$.
From the uniform $\eta$-H\"older regularity of $a(t, ., m)$, \eqref{gaussian:bound:diff:deriv:hat:pm:same:time} and the space-time inequality \eqref{space:time:inequality}, for any $\beta \in [\eta, 1]$, one has
$$
|{\rm II}| \leq K (|y-y'|^\eta \wedge 1) \frac{W^{\beta}_2(\mu, \mu')}{(t-s)^{1+\frac{\beta-\eta}{2}}} \, g(c(t-s), z-x).
$$
The mean value theorem together with \eqref{diff:mes:drift:diff:coefficients}, \eqref{diff:freezing:point:deriv:space:phat} and the space-time inequality \eqref{space:time:inequality} yield
$$
|{\rm III}| \leq K (|y-y'|^\eta \wedge 1) \, \frac{W^{\beta}_2(\mu, \mu')}{(t-s)^{1+\frac{\beta-\eta}{2}}} \, g(c(t-s), z-x).
$$

Finally, from \eqref{definition:general:phat} and Jacobi's formula, one has
\begin{align}
& (\widehat{p}^{y}_{m+1}-\widehat{p}^{y'}_{m+1})(\mu, s, t, x, z) \notag\\
& = \int_0^1 \partial_\lambda \widehat{p}^{\lambda, y, y'}_{m+1}(\mu, s, t, x, z) \, d\lambda \notag \\
& = \int_0^1 -\frac12\left\{ \tr\left(A^{s, t}_\lambda(\mu) ^{-1} \int_s^t [a(r, y, [X^{s, \xi, (m)}_{r}]) - a(r, y', [X^{s, \xi, (m)}_{r}])]  \, dr \right) \right. \notag\\
& \quad \left. - (z-x)^{t} A^{s, t}_\lambda(\mu) ^{-1} \int_s^t [a(r, y, [X^{s, \xi, (m)}_{r}]) - a(r, y', [X^{s, \xi, (m)}_{r}])]  \, dr A^{s, t}_\lambda(\mu) ^{-1} (z-x) \right\} \, \widehat{p}^{\lambda, y, y'}_{m+1}(\mu, s, t, x, z) \, d\lambda. \label{mean:value:theorem:diff:freezing:point:gaussian:kernel}
\end{align}

Using \eqref{diff:measure:inverse:A:s:t:lambda} and \eqref{diff:mes:integral:diff:mes:diffusion:coeff}, the inequality $|\Delta_{\mu, \mu'} \widehat{p}^{\lambda, y, y'}_{m+1}(\mu, s, t, x, z)| \leq K W^{\beta}_2(\mu, \mu') (t-s)^{\frac{\eta-\beta}{2}} g(c(t-s), z-x)$ which stems from the mean-value theorem and \eqref{diff:mes:drift:diff:coefficients}, and the space time inequality \eqref{space:time:inequality} yield
$$
\Big| \Delta_{\mu, \mu'}\Big[ [\widehat{p}^{y}_{m+1}- \widehat{p}^{y'}_{m+1}](\mu, s, t, x, z) \Big] \Big| \leq K (|y-y'|^\eta \wedge 1) \frac{W^{\beta}_2(\mu ,\mu')}{(t-s)^{\frac{\beta}{2}}} \, g(c(t-s), z-x). 
$$
\noindent so that, again by the space time inequality \eqref{space:time:inequality}
$$
|{\rm IV}| \leq K (|y-y'|^\eta \wedge 1) \frac{W^{\beta}_2(\mu ,\mu')}{(t-s)^{1+ \frac{\beta}{2}}} \, g(c(t-s), z-x). 
$$
Gathering the estimates on ${\rm I}$, ${\rm II}$, ${\rm III}$ and ${\rm IV}$, we conclude
\begin{align*}
 | \Delta_{\mu, \mu'} & \partial^2_{x_i, x_j} [\widehat{p}^{y}_{m+1}-\widehat{p}^{y'}_{m+1}](\mu, s ,t , x ,z )  |  \leq K (|y-y'|^{\eta} \wedge 1) \frac{ W^{\beta}_2(\mu, \mu')}{(t-s)^{1+\frac{\beta}{2}} }  \,  g(c(t-s), z - x)
\end{align*}
\noindent for any $\beta \in [\eta, 1]$. Finally, since $W^2_2(\mu, \mu') \leq t-s$, the above estimate still holds for any $\beta \in [0,1]$. The proof of \eqref{gaussian:bound:diff:second:deriv:p:hat:mu:holder:reg} is now complete.

\subsection{Proof of Lemma \ref{lem:technical:estimate:reg:time:metric:coefficients:densities}.\\}\label{section:proof:lem:technical:estimate:reg:time:metric:coefficients:densities}

\noindent \emph{Step 1: proof of \eqref{diff:time:heat:kernel}.}\\

We split the computations into the two cases $|s_1-s_2| \geq t-s_1\vee s_2$ and $|s_1-s_2| \leq t-s_1\vee s_2$. In the first case, the announced estimate directly follows from \eqref{bound:derivative:heat:kernel} with $n=0$. Indeed, observe that for any $\beta \in [0,1]$
\begin{align*}
|\Delta_{s_1,s_2}&  p_{m}(\mu, s, t, x, z) | \nonumber \\
&  \leq K \left\{  g(c(t-s_1 \vee s_2), z-x) +   g(c(t-s_1\wedge s_2), z-x)  \right\} \\
& \leq K \left\{ \frac{|s_1-s_2|^{\beta}}{(t-s_1\vee s_2)^{\beta}}  g(c(t-s_1 \vee s_2), z-x) + \frac{(t-s_1\vee s_2)^{\beta} + |s_1-s_2|^\beta}{(t-s_1\wedge s_2)^{\beta}}  g(c(t-s_1 \wedge s_2), z-x) \right\} \\
 & \leq K \left\{ \frac{|s_1-s_2|^{\beta}}{(t-s_1)^{\beta}} g(c(t-s_1), z-x) +  \frac{|s_1-s_2|^{\beta}}{(t-s_2)^{\beta}} g(c(t-s_2), z-x)  \right\}.
\end{align*}

In the second case, \eqref{diff:time:heat:kernel} follows from the mean-value theorem, \eqref{time:derivative:induction:decoupling:mckean} (recall that $\mathscr{C}^{n, 0}_{m} \leq  \mathscr{C}^{n, 0}_{\infty} := \lim_{m \uparrow \infty} \mathscr{C}^{n, 0}_{m} < \infty$) and the inequality $(t-s_1\vee s_2)^{-1} \leq 2 (t-s_1\wedge s_2)^{-1}$. We omit the remaining technical details. \\

\noindent \emph{Step 2: proof of \eqref{diff:time:drift:diff:coefficients}.}\\

Using the mean-value theorem and the fact that $\int_{\mathbb{R}^d} \Delta_{s_1,s_2} p_m(\mu, s, t, x, y) \, dy = 0$, one gets 
\begin{align*}
& a_{i, j}(t, x, [X^{s_1\vee s_2, \xi, (m)}_t])   - a_{i, j}(t, x, [X^{s_1\wedge s_2, \xi, (m)}_t]) \\
&\quad = \int_0^1 \int_{\mathbb{R}^d} \frac{\delta a_{i, j}}{\delta m} (t, x, \Theta^{(m)}_{\lambda, t})(y) \, \Delta_{s_1,s_2} p_m(\mu, s, t, y) \, dy \, d\lambda\\
& \quad = \int_0^1 \int_{(\mathbb{R}^d)^2} \Big(\frac{\delta a_{i, j}}{\delta m} (t, x, \Theta^{(m)}_{\lambda, t})(y) - \frac{\delta a_{i, j}}{\delta m} (t, x, \Theta^{(m)}_{\lambda, t})(x') \Big) \Delta_{s_1,s_2} p_m(\mu, s, t, x', y) \, dy \, \mu(dx') d\lambda
\end{align*}

\noindent where we introduced the notation $\Theta^{(m)}_{\lambda, t} := \lambda [X^{s_1\vee s_2, \xi, (m)}_t] + (1-\lambda) [X^{s_1\wedge s_2, \xi, (m)}_t]$. The uniform $\eta$-H\"older regularity of $ [\delta  a_{i, j}/\delta m](t, x, m)(.)$ combined with \eqref{diff:time:heat:kernel} and the space-time inequality \eqref{space:time:inequality} thus yield
\begin{equation}
\label{diff:time:diff:coefficient:temporary}
|a_{i, j}(t, x, [X^{s_1\vee s_2, \xi, (m)}_t])  - a_{i, j}(t, x, [X^{s_1\wedge s_2, \xi, (m)}_t])| \leq K |s_1-s_2|^\beta \left( \frac{1}{(t-s_{1})^{\beta-\frac{\eta}{2}}} + \frac{1}{(t-s_{2})^{\beta-\frac{\eta}{2}}} \right)
\end{equation}
\noindent which concludes the proof.\\

\noindent \emph{Step 3: proof of \eqref{gaussian:bound:diff:time:hat:pm:same:time}.}\\

We first remark that from \eqref{time:deriv:p:hat:s:t}, \eqref{time:derivative:induction:decoupling:mckean} and the space-time inequality \eqref{space:time:inequality}, one gets
\begin{equation}
\label{gaussian:estimate:deriv:time:p:hat}
|\partial_s \widehat{p}_{m+1}(\mu, s, t, x, z)| \leq \frac{K}{t-s} \, g(c(t-s), z-x)
\end{equation}
\noindent so that \eqref{gaussian:bound:diff:time:hat:pm:same:time} for $n=0$ follows from similar lines of reasonings as those employed to prove \eqref{diff:time:heat:kernel}. We thus omit the remaining technical details. Now, let $n\in \left\{1, 2\right\}$. If $|s_1-s_2| > t-s_1\vee s_2$, \eqref{gaussian:bound:diff:time:hat:pm:same:time} follows from the standard Gaussian estimate $|\partial_x^n  \widehat{p}_{m+1}(\mu, s, t, x, z)| \leq (t-s)^{-n/2} \, g(c(t-s), z-x)$. From now on we assume that $|s_1-s_2| \leq t-s_1\vee s_2$. We use the decomposition
\begin{align*}
\partial^{n}_x & \widehat{p}_{m+1}(\mu, s_1\vee s_2, t, x, z)  - \partial_x^n \widehat{p}_{m+1}(\mu, s_1 \wedge s_2, t, x, z) \\
& = \Big[ H_n(\int_{s_1\vee s_2}^t a(r', z, [X^{s_1\vee s_2, \xi , (m)}_{r'}]) \, dr', z-x) - H_n(\int_{s_1\wedge s_2}^t a(r', z, [X^{s_1\wedge s_2, \xi , (m)}_{r'}]) \, dr', z-x) \Big] \\
& \quad \times \widehat{p}_{m+1}(\mu, s_1\vee s_2, t, x, z) \\
& \quad + H_n(\int_{s_1\wedge s_2}^t a(r', z, [X^{s_1\wedge s_2, \xi , (m)}_{r'}]) \, dr', z-x) \Big[ \widehat{p}_{m+1}(\mu, s_1\vee s_2, t, x, z) - \widehat{p}_{m+1}(\mu, s_1\wedge s_2, t, x, z)\Big]\\
& =: {\rm A}_n + {\rm B}_n
\end{align*}
\noindent where we introduced the notation $H_n$ for the vector $(H^{i}_1)_{1\leq i \leq d}$ if $n=1$ or the matrix $(H^{i, j}_2)_{1\leq i, j \leq d}$ if $n=2$. 

In order to deal with ${\rm A}_n$, we first observe that $\partial_s [\int_s^{t} a(r', z, [X^{s, \xi , (m)}_{r'}]) \, dr'] = -a(s, z, \mu) + \int_s^t \partial_s a(r', z, [X_{r'}^{s, \xi , (m)}]) \, dr'$ so that combining \eqref{recursive:bound:time:deriv:a:or:b} with \eqref{time:derivative:induction:decoupling:mckean}, we derive that $|\partial_s [\int_s^{t} a(r', z, [X^{s, \xi , (m)}_{r'}]) \, dr']| \leq K$ which in turn by the mean-value theorem yields
\begin{align}
\Big|\Big(\int_{s_1\vee s_2}^{t} a(r', z, [X^{s_1\vee s_2, \xi , (m)}_{r'}]) \, dr'\Big)^{-1} & - \Big(\int_{s_1\wedge s_2}^{t} a(r', z, [X^{s_1\wedge s_2, \xi , (m)}_{r'}]) \, dr'\Big)^{-1}\Big| \notag \\
& \leq K \frac{|s_1-s_2| }{(t-s_1\vee s_2)^2} \notag \\
& \leq K \frac{|s_1-s_2|^{\beta}}{(t-s_1\vee s_2)^{1+\beta}} \label{diff:time:inv:diff:coeff}
\end{align}
\noindent for any $\beta \in [0,1]$, recalling that $|s_1-s_2| \leq t-s_1\vee s_2$. The previous bound together with the space-time inequality \eqref{space:time:inequality} then implies
\begin{align*}
\Big| H_n(\int_{s_1\vee s_2}^t&  a(r', z, [X^{s_1\vee s_2, \xi , (m)}_{r'}]) \, dr', z-x)  - H_n(\int_{s_1\wedge s_2}^t a(r', z, [X^{s_1\wedge s_2, \xi , (m)}_{r'}]) \, dr', z-x) \Big| \widehat{p}_{m+1}(\mu, s_1\vee s_2, t, x, z)\\
& \leq K  \frac{|s_1-s_2|^{\beta}}{(t-s_1\vee s_2)^{\frac{n}{2}+\beta}} g(c(t-s_1 \vee s_2), z-x)
\end{align*}
\noindent so that
$$
|{\rm A}_n| \leq K  \frac{|s_1-s_2|^{\beta}}{(t-s_1\vee s_2)^{\frac{n}{2}+\beta}} \, g(c(t-s_1\vee s_2) , z-x).
$$
In order to deal with ${\rm B}_n$, we first remark that from the very definition of $H_n$ one gets
\begin{align}
\Big| & H_n(\int_{r}^t a(r', z, [X^{s_1\wedge s_2, \xi , (m)}_{r'}]) \, dr', z-x) \Big| \nonumber \\
& \quad  \leq K \left\{ \frac{|z-x|}{t-r} \textbf{1}_{\left\{ n = 1 \right\}}+ \left(\frac{|z-x|^2}{(t-r)^2}+  \frac{1}{t-r} \right) \textbf{1}_{\left\{ n = 2 \right\}} + \left(\frac{|z-x|^3}{(t-r)^3}+  \frac{|z-x|}{(t-r)^2} \right) \textbf{1}_{\left\{ n = 3 \right\}}  \right\} \label{bound:hermite:polynomial:Hn}
\end{align}
\noindent for any $r\in [0,t)$ and any $n\in \left\{1, 2, 3\right\}$. 

Combining \eqref{bound:hermite:polynomial:Hn}, the estimate \eqref{gaussian:bound:diff:time:hat:pm:same:time} for $n=0$ and the space-time inequality \eqref{space:time:inequality}, we eventually deduce
\begin{align*}
| {\rm B}_n | & \leq K \left\{ \frac{|z-x|}{t-s_1\wedge s_2} + \left(\frac{|z-x|^2}{(t-s_1\wedge s_2)^2}+  \frac{1}{t-s_1\wedge s_2} \right) \textbf{1}_{\left\{ n = 2 \right\}}\right\}  \\
& \quad \times \left\{ \frac{ |s_1-s_2|^\beta}{(t-s_1)^{\beta}} g(c(t-s_1) , z-x) + \frac{ |s_1-s_2|^\beta}{(t-s_2)^{\beta}} g(c(t-s_2), z-x) \right\} \\
& \leq K \left\{ \frac{|s_1-s_2|^{\beta}}{(t-s_1)^{\frac{n}{2}+\beta}} g(c(t-s_1) , z-x)  + \frac{|s_1-s_2|^{\beta}}{(t-s_2)^{\frac n2 + \beta}}  g(c(t-s_2) , z-x)  \right\}.
\end{align*}

Gathering the bounds on ${\rm A}_n$ and ${\rm B}_n$ concludes the proof of \eqref{gaussian:bound:diff:time:hat:pm:same:time}. \\

\noindent \emph{Step 4: proof of \eqref{diff:time:with:holder:reg:space:drift:diff:coefficients}.}\\

We prove the announced estimate only for the part related to the difference of the diffusion coefficient as the other part is handled in a completely analogous manner. We write
\begin{align*}
& a_{i, j} (t, x, [X^{s_1\vee s_2, \xi, (m)}_t])   - a_{i, j}(t, x, [X^{s_1\wedge s_2, \xi, (m)}_t]) - [ a_{i, j} (t, y, [X^{s_1\vee s_2, \xi, (m)}_t])  - a_{i, j}(t, y, [X^{s_1 \wedge s_2, \xi, (m)}_t]) ] \\
& = \int_0^1  \int_{(\rr^d)^2} \left[\frac{\delta a_{i, j}}{\delta m}(t, x, \Theta^{(m)}_{\lambda, t})(y') -  \frac{\delta a_{i, j}}{\delta m}(t, y, \Theta^{(m)}_{\lambda, t})(y')\right] \\
& \quad \times [ p_{m}(\mu, s_1\vee s_2, t, x', y') -  p_{m}(\mu, s_1\wedge s_2, t, x', y') ]\, dy'  \, \mu(dx') \,  d\lambda  \\
& = \int_0^1  \int_{(\rr^d)^2} \left\{ \left[\frac{\delta a_{i, j}}{\delta m}(t, x, \Theta^{(m)}_{\lambda, t})(y') -  \frac{\delta a_{i, j}}{\delta m}(t, y, \Theta^{(m)}_{\lambda, t})(y')\right] - \left[\frac{\delta a_{i, j}}{\delta m}(t, x, \Theta^{(m)}_{\lambda, t})(x') -  \frac{\delta a_{i, j}}{\delta m}(t, y, \Theta^{(m)}_{\lambda, t})(x')\right] \right\} \\
& \quad \times [ p_{m}(\mu, s_1\vee s_2, t, x', y') -  p_{m}(\mu, s_1\wedge s_2, t, x', y') ]\, dy'  \, \mu(dx') \,  d\lambda
\end{align*}
\noindent where we introduced the notation $ \Theta^{(m)}_{\lambda, t} := \lambda  [X^{s_1\vee s_2, \xi, (m)}_t] + (1-\lambda) [X^{s_1\wedge s_2, \xi, (m)}_t]$ and used the fact that $\int_{\mathbb{R}^d}  [ p_{m}(\mu, s_1\vee s_2, t, x', y') -  p_{m}(\mu, s_1\wedge s_2, t, x', y') ]\, dy'  = 0$. From the uniform $\eta$-H\"older regularity of the map $[\delta a_{i, j}/\delta m](t, ., m)(.)$, we get
\begin{align*}
\Big| \left[\frac{\delta a_{i, j}}{\delta m}(t, x, \Theta^{(m)}_{\lambda, t})(y') -  \frac{\delta a_{i, j}}{\delta m}(t, y, \Theta^{(m)}_{\lambda, t})(y')\right] & - \left[\frac{\delta a_{i, j}}{\delta m}(t, x, \Theta^{(m)}_{\lambda, t})(x') -  \frac{\delta a_{i, j}}{\delta m}(t, y, \Theta^{(m)}_{\lambda, t})(x')\right] \ \Big|  \\
& \leq K (|y-x|^\eta \wedge |y'-x'|^\eta \wedge 1)\\
& \leq K |y-x|^\alpha (|y'-x'|^{\eta-\alpha}\wedge 1)
\end{align*}
\noindent for any $\alpha \in [0,\eta]$, which combined with \eqref{diff:time:heat:kernel} and the space-time inequality \eqref{space:time:inequality} yields
\begin{align*}
|a_{i, j} (t, x, [X^{s_1\vee s_2, \xi, (m)}_t])  & - a_{i, j}(t, x, [X^{s_1\wedge s_2, \xi, (m)}_t])  - [ a_{i, j} (t, y, [X^{s_1\vee s_2, \xi, (m)}_t])  - a_{i, j}(t, y, [X^{s_1 \wedge s_2, \xi, (m)}_t]) ]  |\\
& \leq K_\alpha |y-x|^\alpha |s_1-s_2|^\beta \left\{\frac{1}{(t-s_1)^{\beta+\frac{\alpha-\eta}{2}}} + \frac{1}{(t-s_2)^{\beta+\frac{\alpha-\eta}{2}}}\right\}.
\end{align*}
This concludes the proof.\\

\noindent \emph{Step 5: proof of \eqref{gaussian:bound:diff:time:second:deriv:p:hat:mu:holder:reg}.}\\

Observe that if $|s_1-s_2| >  t-s_1\vee s_2$, then from \eqref{diff:freezing:point:deriv:space:phat}, we obtain
\begin{align*}
 | & \Delta_{s_1, s_2}  \partial^n_{x} [\widehat{p}^{y}_{m+1}-\widehat{p}^{y'}_{m+1}](\mu, s ,t , x ,z )  |    \\
& \leq |\partial^n_{x} [\widehat{p}^{y}_{m+1}-\widehat{p}^{y'}_{m+1}](\mu, s_1 ,t , x ,z ) |  +|\partial^n_{x} [\widehat{p}^{y}_{m+1}-\widehat{p}^{y'}_{m+1}](\mu, s_2 ,t , x ,z ) |  \\
 & \leq K (|y-y'|^{\eta} \wedge 1) \left\{ \frac{1}{(t-s_1\vee s_2)^{\frac{n}{2}}}  \,  g(c(t-s_1\vee s_2), z - x) +  \frac{1}{(t-s_1 \wedge s_2)^{\frac{n}{2}}}  \,  g(c(t-s_1 \wedge s_2), z - x) \right\} \\
 & \leq K (|y-y'|^{\eta} \wedge 1) \left\{ \frac{|s_1-s_2|^\beta}{(t-s_1\vee s_2)^{\frac{n}{2}+\beta}}  \,  g(c(t-s_1\vee s_2), z - x) +  \frac{|t-s_1\vee s_2|^\beta + |s_1-s_2|^\beta}{(t-s_1\wedge s_2)^{\frac{n}{2}+\beta}}  \,  g(c(t-s_1\wedge s_2), z - x) \right\} \\
 & \leq K (|y-y'|^{\eta} \wedge 1) \left\{ \frac{|s_1-s_2|^\beta}{(t-s_1)^{\frac{n}{2}+\beta}}  \,  g(c(t-s_1), z - x) +  \frac{|s_1-s_2|^\beta}{(t-s_2)^{\frac{n}{2}+\beta}}  \,  g(c(t-s_2), z - x) \right\} 
\end{align*}
\noindent for any $\beta \in [0,1]$.
From now on, we assume that $|s_1-s_2| \leq  t-s_1\vee s_2$. Starting from the relation 
$$
\partial^{n}_x \widehat{p}^{y}_{m+1}(\mu, s, t, x, z) = H_n(\int_s^t a(r, y, [X^{s, \xi , (m)}_r]) \, dr , z-x) \widehat{p}^{y}_{m+1}(\mu, s, t, x, z),
$$
\noindent we derive the following decomposition
$$
\Delta_{s_1, s_2} [\partial^{n}_{x} ( \widehat{p}^{y}_{m+1} - \widehat{p}^{y'}_{m+1}) (\mu, s, t, x, z) ]  = {\rm I} + {\rm II} + {\rm III} + {\rm IV}
$$
\noindent with
\begin{align*}
{\rm I}  &  := \Delta_{s_1, s_2} \Big[ H_n\Big(\int_s^t a(r, y, [X^{s, \xi , (m)}_r]) \, dr , z-x\Big) - H_n\Big(\int_s^t a(r, y', [X^{s, \xi , (m)}_r]) \, dr , z-x\Big) \Big] \widehat{p}^{y}_{m+1}(\mu, s_1 \vee s_2, t, x, z), \\
{\rm II} &  :=  \Big[H_n\Big(\int_{s_1\wedge s_2}^t a(r, y, [X^{s_1\wedge s_2, \xi , (m)}_r]) \, dr , z-x\Big) - H_n\Big(\int_{s_1\wedge s_2}^t a(r, y', [X^{s_1\wedge s_2, \xi , (m)}_r]) \, dr , z-x\Big)\Big] \\
& \quad \times \Delta_{s_1, s_2} \widehat{p}^{y}_{m+1}(\mu, s, t, x, z), \\
{\rm III} & := \Delta_{s_1, s_2} H_n\Big(\int_s^t a(r, y', [X^{s, \xi , (m)}_r] ) \, dr , z-x\Big) \Big[ \widehat{p}^{y}_{m+1}(\mu, s_1\vee s_2, t, x, z) - \widehat{p}^{y'}_{m+1}(\mu, s_1\vee s_2, t, x, z) \Big], \\
{\rm IV} & := H_n\Big(\int_{s_1\wedge s_2}^t a(r, y', [X^{s_1\wedge s_2, \xi , (m)}_r]) \, dr , z-x\Big) \Delta_{s_1, s_2}\Big[ \widehat{p}^{y}_{m+1}(\mu, s, t, x, z) - \widehat{p}^{y'}_{m+1}(\mu, s, t, x, z) \Big].
\end{align*}

We now establish an appropriate estimate for each term of the above decomposition. 

As in the proof of \eqref{gaussian:bound:diff:second:deriv:p:hat:mu:holder:reg}, we first introduce the invertible matrix $A^{s, t}_\lambda = A^{s, t}_\lambda(\mu) := \int_s^t [\lambda a(r, y, [X^{s, \xi, (m)}_r]) + (1-\lambda) a(r, y', [X^{s, \xi, (m)}_r])] \, dr$ and the associated Gaussian kernel $\widehat{p}^{\lambda, y, y'}_{m+1}(\mu, s, t, x, z) = g(A^{s, t}_\lambda(\mu), z-x)$, $\lambda \in [0,1]$. We then note that from the mean-value theorem, one has
\begin{align*}
|\Delta_{s_1, s_2} (A^{s, t}_\lambda)^{-1}| & \leq  \frac{K}{(t-s_1\vee s_2)^2} \Big[\Big| \Delta_{s_1, s_2}\int_{s}^{t}  a(r, y, [X^{s, \xi, (m)}_r]) \, dr \Big| +  \Big| \Delta_{s_1, s_2} \int_s^t a(r, y', [X^{s, \xi, (m)}_r])  \, dr \Big| \Big].
\end{align*}

Moreover, from \eqref{diff:time:drift:diff:coefficients}, recalling that $|s_1-s_2| \leq t-s_1\vee s_2$, for any $x\in \mathbb{R}^d$ and any $\beta \in [0,1]$, one has
\begin{align*}
\Big| \Delta_{s_1, s_2}\int_{s}^{t}  & a(r, x, [X^{s, \xi, (m)}_r]) \, dr \Big| \\
& \leq  \int_{s_1\vee s_2}^{t} | a(r, x, [X^{s_1\vee s_2, \xi, (m)}_r]) - a(r, x, [X^{s_1\wedge s_2, \xi, (m)}_r]) | \, dr  + \int_{s_1 \wedge s_2}^{s_1\vee s_2} | a(r, x, [X^{s_1\wedge s_2, \xi, (m)}_r]) | \, dr\\
& \leq K |s_1-s_2|^\beta \int_{s_1 \vee s_2}^{t} \Big[ \frac{1}{(r-s_1\vee s_2)^{\beta-\frac{\eta}{2}}} +  \frac{1}{(r-s_1\wedge s_2)^{\beta-\frac{\eta}{2}}} \Big] \, dr  + K |s_1-s_2| \\
& \leq K |s_1-s_2|^\beta [ (t-s_1\vee s_2)^{1-\beta +\frac{\eta}{2}} + (t-s_1\wedge s_2)^{1-\beta + \frac{\eta}{2}}] + K |s_1-s_2|^\beta (t-s_1\vee s_2)^{1-\beta} \\
& \leq K |s_1-s_2|^\beta (t-s_1\vee s_2)^{1-\beta} 
\end{align*}
\noindent where for the last inequality we used the fact that $t-s_1\wedge s_2 \leq 2 (t-s_1\vee s_2)$. We thus deduce that
\begin{equation}
\label{delta:time:invertible:matrix:mean:value:thm}
|\Delta_{s_1, s_2} (A^{s, t}_\lambda)^{-1}| \leq K \frac{|s_1-s_2|^\beta}{(t-s_1\vee s_2)^{1+\beta}}
\end{equation}

\noindent for any $\beta \in [0,1]$. Also, from \eqref{diff:time:with:holder:reg:space:drift:diff:coefficients} with $\alpha=\eta$ and similar computations, we obtain
\begin{align}
\Big| \Delta_{s_1, s_2} &  \left[\int_s^t [a(r, y, [X^{s, \xi, (m)}_r]) - a(r, y', [X^{s, \xi, (m)}_r])] \, dr\right] \Big| \notag \\
& \leq K \Big[ \Big| \int_{s_1\vee s_2}^{t} \Delta_{s_1, s_2} [a(r, y, [X^{s, \xi, (m)}_r]) - a(r, y', [X^{s, \xi, (m)}_r])] \, dr\Big| + |s_1-s_2| (|y-y'|^\eta \wedge 1)\Big] \notag \\
& \leq K |s_1-s_2|^\beta (|y-y'|^\eta \wedge 1) \Big[ \int_{s_1\vee s_2}^{t} \Big[\frac{1}{(r-s_1)^{\beta}} + \frac{1}{(r-s_2)^{\beta}}\Big] \, dr + (t-s_1\vee s_2)^{1-\beta} \Big] \notag \\
& \leq K |s_1-s_2|^\beta (|y-y'|^\eta \wedge 1) (t-s_1\vee s_2)^{1-\beta} \label{delta:time:diff:freezing:point:diff:coeff}
\end{align}
\noindent for any $\beta \in [0,1)$, where we again used the fact that $t-s_1\wedge s_2 \leq 2 (t-s_1\vee s_2)$.
Now, from the mean-value theorem, we obtain
\begin{align*}
\left(\int_s^t a(r, y, [X^{s, \xi, (m)}_r]) \, dr\right)^{-1} & - \left(\int_s^t a(r, y', [X^{s, \xi, (m)}_r]) \, dr\right)^{-1} \\
& = -\int_0^1  (A^{s, t}_\lambda)^{-1} \left[\int_s^t a(r, y, [X^{s, \xi, (m)}_r]) - a(r, y', [X^{s, \xi, (m)}_r]) \, dr\right] (A^{s, t}_\lambda)^{-1} \, d\lambda
\end{align*}

\noindent so that, using \eqref{delta:time:invertible:matrix:mean:value:thm}, \eqref{delta:time:diff:freezing:point:diff:coeff} together with some standard computations that we omit, one gets
\begin{align*}
\Big| \Delta_{s_1, s_2} &  \Big[\left(\int_s^t a(r, y, [X^{s, \xi, (m)}_r]) \, dr\right)^{-1}   - \left(\int_s^t a(r, y', [X^{s, \xi, (m)}_r]) \, dr\right)^{-1}\Big] \Big|\\
& \leq K (|y-y'|^\eta \wedge 1) \frac{|s_1-s_2|^\beta}{(t-s_1\vee s_2)^{1+\beta}}
\end{align*}
\noindent for any $\beta \in [0,1)$.
The previous bound and the space-time inequality \eqref{space:time:inequality} allow to conclude
\begin{align*}
\Big|\Delta_{s_1, s_2} \Big[ H_n\Big(\int_s^t a(r, y, [X^{s, \xi , (m)}_r]) \, dr , z-x\Big) & - H_n\Big(\int_s^t a(r, y', [X^{s, \xi , (m)}_r]) \, dr , z-x\Big) \Big]\Big|  \widehat{p}^{y}_{m+1}(\mu, s_1 \vee s_2, t, x, z) \\
& \leq K (|y-y'|^\eta \wedge 1) \frac{|s_1-s_2|^\beta}{(t-s_1\vee s_2)^{\frac{n}{2}+\beta}} \, g(c(t-s_1\vee s_2), z-x)
\end{align*}
\noindent so that
$$
|{\rm I}| \leq  K (|y-y'|^\eta \wedge 1) \frac{|s_1-s_2|^\beta}{(t-s_1\vee s_2)^{\frac{n}{2}+\beta}} \, g(c(t-s_1\vee s_2), z-x)
$$
\noindent for any $\beta \in [0, 1)$.
From the uniform $\eta$-H\"older regularity of $ a(t, ., m)$, one gets
\begin{align*}
 \Big|H_n\Big(\int_{s_1\wedge s_2}^t & a(r, y, [X^{s_1\wedge s_2, \xi , (m)}_r]) \, dr , z-x\Big)  - H_n\Big(\int_{s_1\wedge s_2}^t a(r, y', [X^{s_1\wedge s_2, \xi , (m)}_r]) \, dr , z-x\Big)\Big|  \\
& \leq K (|y-y'|^\eta \wedge 1) \left\{ \frac{|z-x|}{t-s_1\wedge s_2} \textbf{1}_{n=1}+ \Big[\frac{|z-x|^2}{(t-s_1\wedge s_2)^2} + \frac{1}{t-s_1\wedge s_2} \Big] \textbf{1}_{n=2} \right\}
\end{align*}
 \noindent which, combined with \eqref{gaussian:bound:diff:time:hat:pm:same:time} and the space-time inequality \eqref{space:time:inequality}, yields
$$
|{\rm II}| \leq K (|y-y'|^\eta \wedge 1)  \left\{ \frac{|s_1-s_2|^\beta}{(t-s_1)^{\frac{n}{2}+\beta}} \, g(c(t-s_1), z-x) + \frac{|s_1-s_2|^\beta}{(t-s_2)^{\frac{n}{2}+\beta}} \, g(c(t-s_2), z-x)  \right\}.
$$
The estimate \eqref{diff:time:inv:diff:coeff} and standard computations imply
$$
\Big|\Delta_{s_1, s_2} H_n\Big(\int_s^t a(r, y', [X^{s, \xi , (m)}_r] ) \, dr , z-x\Big) \Big| \leq K  \frac{|s_1-s_2|^\beta}{(t-s_1\vee s_2)^{1+\beta}} \left\{ |z-x| \textbf{1}_{n=1} + \Big[ \frac{|z-x|^2}{t-s_1\vee s_2} + 1\Big] \textbf{1}_{n=2} \right\}.
$$

\noindent which combined with \eqref{diff:freezing:point:deriv:space:phat} and the space-time inequality \eqref{space:time:inequality} yields
$$
|{\rm III}| \leq K (|y-y'|^\eta \wedge 1)\frac{|s_1-s_2|^\beta}{(t-s_1\vee s_2)^{\frac{n}{2}+\beta}} \, g(c(t-s_1\vee s_2), z-x).
$$

Finally, in spirit of the proof of \eqref{gaussian:bound:diff:second:deriv:p:hat:mu:holder:reg}, we deal with the last term ${\rm IV}$ by using the identity \eqref{mean:value:theorem:diff:freezing:point:gaussian:kernel} and then \eqref{delta:time:invertible:matrix:mean:value:thm}, \eqref{delta:time:diff:freezing:point:diff:coeff} as well as the inequality 
$$
|\Delta_{s_1, s_2} \widehat{p}^{\lambda, y, y'}_{m+1}(\mu, s, t, x, z)| \leq K  \left\{ \frac{|s_1-s_2|^\beta}{(t-s_1)^{\beta}} \, g(c(t-s_1), z-x) + \frac{|s_1-s_2|^\beta}{(t-s_2)^{\beta}} \, g(c(t-s_2), z-x) \right\}
$$
\noindent which stems from similar arguments as those employed to establish \eqref{gaussian:bound:diff:time:hat:pm:same:time}, and the space time inequality \eqref{space:time:inequality}. Omitting the remaining technical details, we obtain
\begin{align*}
\Big| \Delta_{s_1, s_2}\Big[ \widehat{p}^{y}_{m+1}(\mu, s, t, x, z) &  - \widehat{p}^{y'}_{m+1}(\mu, s, t, x, z) \Big] \Big| \\
& \leq K (|y-y'|^\eta \wedge 1)  \left\{ \frac{|s_1-s_2|^\beta}{(t-s_1)^{\beta}} \, g(c(t-s_1), z-x) + \frac{|s_1-s_2|^\beta}{(t-s_2)^{\beta}} \, g(c(t-s_2), z-x) \right\}. 
\end{align*}
Hence, again by the space time inequality \eqref{space:time:inequality}, we derive
$$
|{\rm IV}| \leq K (|y-y'|^\eta \wedge 1)  \left\{ \frac{|s_1-s_2|^\beta}{(t-s_1)^{\frac{n}{2}+\beta}} \, g(c(t-s_1), z-x) + \frac{|s_1-s_2|^\beta}{(t-s_2)^{\frac{n}{2}+\beta}} \, g(c(t-s_2), z-x) \right\}. 
$$
Putting together the estimates on ${\rm I}$, ${\rm II}$, ${\rm III}$ and ${\rm IV}$, we thus conclude
\begin{align*}
 | \Delta_{s_1, s_2} & [\partial^n_{x} [\widehat{p}^{y}_{m+1}- \widehat{p}^{y'}_{m+1}](\mu, s ,t , x ,z )]  | \\
 & \leq K (|y-y'|^\eta \wedge 1)  \left\{ \frac{|s_1-s_2|^\beta}{(t-s_1)^{\frac{n}{2}+\beta}} \, g(c(t-s_1), z-x) + \frac{|s_1-s_2|^\beta}{(t-s_2)^{\frac{n}{2}+\beta}} \, g(c(t-s_2), z-x) \right\}
\end{align*}
\noindent for any $\beta \in [0, 1)$. The proof of \eqref{gaussian:bound:diff:time:second:deriv:p:hat:mu:holder:reg} is now complete.\\

\noindent \emph{Step 6: proof of \eqref{gaussian:bound:diff:time:hat:pm:different:time}.}\\

We proceed as in the previous step. Let us first remark that if $|s_1-s_2| > r-s_1\vee s_2$ then one directly gets the estimate \eqref{gaussian:bound:diff:time:hat:pm:different:time} from the standard Gaussian bound $|\partial_x^n \widehat{p}_{m+1}(\mu, s, r, t, x, z)| \leq K (t-r)^{-\frac n2} g(c(t-r),  z-x)$.
 Assuming now that $|s_1-s_2| \leq r-s_1\vee s_2$, from \eqref{time:deriv:p:hat:s:r:t} and \eqref{time:derivative:induction:decoupling:mckean} (recall that $\mathscr{C}^{1, 0}_{m}< \mathscr{C}^{1,0}_{\infty} < \infty$) and the space-time inequality \eqref{space:time:inequality}, one gets
\begin{equation}
\label{gaussian:estimate:deriv:time:p:hat:s:r:t}
|\partial_s \widehat{p}_{m+1}(\mu, s, r, t, x, z)| \leq \frac{K}{(r-s)^{1-\frac{\eta}{2}}} \, g(c(t-r), z-x)
\end{equation}
\noindent which combined with the mean-value theorem yields the estimate \eqref{gaussian:bound:diff:time:hat:pm:different:time} for $n=0$. Now, let $n\in \left\{1, 2, 3 \right\}$ and assume again that $|s_1-s_2| \leq r-s_1\vee s_2$. We proceed again as in the previous step. Namely, we write
\begin{align*}
\partial^{n}_x & \widehat{p}_{m+1}(\mu, s_1\vee s_2, r,  t, x, z)  - \partial_x^n \widehat{p}_{m+1}(\mu, s_1 \wedge s_2, r, t, x, z) \\
& = \Big[ H_n(\int_{r}^t a(r', z, [X^{s_1\vee s_2, \xi , (m)}_{r'}]) \, dr', z-x) - H_n(\int_{r}^t a(r', z, [X^{s_1\wedge s_2, \xi , (m)}_{r'}]) \, dr', z-x) \Big] \\
& \quad \times \widehat{p}_{m+1}(\mu, s_1\vee s_2, r, t, x, z) \\
& \quad + H_n(\int_{r}^t a(r', z, [X^{s_1\wedge s_2, \xi , (m)}_{r'}]) \, dr', z-x) \Big[ \widehat{p}_{m+1}(\mu, s_1\vee s_2, r,  t, x, z) - \widehat{p}_{m+1}(\mu, s_1\wedge s_2, r,  t, x, z)\Big]\\
& =: {\rm A}_n + {\rm B}_n
\end{align*}
 
 In order to deal with ${\rm A}_n$, we again combine \eqref{recursive:bound:time:deriv:a:or:b} with \eqref{time:derivative:induction:decoupling:mckean} to deduce that $|\int_r^{t} \partial_s  a(r', z, [X^{s, \xi , (m)}_{r'}]) \, dr'| \leq K (r-s)^{-1+\eta/2} (t-r)$ which in turn by the mean-value theorem implies
\begin{align*}
\Big|\Big(\int_{r}^{t} a(r', z, [X^{s_1\vee s_2, \xi , (m)}_{r'}]) \, dr'\Big)^{-1} & - \Big(\int_{r}^{t} a(r', z, [X^{s_1\wedge s_2, \xi , (m)}_{r'}]) \, dr'\Big)^{-1}\Big| \\
& \leq K \frac{|s_1-s_2| } {(t-r)(r-s_1\vee s_2)^{1-\frac{\eta}{2}}} \\
& \leq K \frac{|s_1-s_2|^{\beta}}{(t-r)(r-s_1\vee s_2)^{\beta-\frac{\eta}{2}}}
\end{align*}
\noindent for any $\beta \in [0,1]$, recalling that $|s_1-s_2| \leq r-s_1\vee s_2$. The previous bound together with the space-time inequality \eqref{space:time:inequality} implies
\begin{align}
\Big| H_n(\int_{r}^t&  a(r', z, [X^{s_1\vee s_2, \xi , (m)}_{r'}]) \, dr', z-x)  - H_n(\int_{r}^t a(r', z, [X^{s_1\wedge s_2, \xi , (m)}_{r'}]) \, dr', z-x) \Big| \widehat{p}_{m+1}(\mu, s_1\vee s_2, r, t, x, z) \notag\\
& \leq K  \frac{|s_1-s_2|^{\beta}}{(t-r)^{\frac{n}{2}}(r-s_1\vee s_2)^{\beta-\frac{\eta}{2}} } g(c(t-r), z-x) \label{bound:diff:time:hermite:polynomial:Hn}
\end{align}
\noindent so that
$$
|{\rm A}_n| \leq K  \frac{|s_1-s_2|^{\beta}}{(t-r)^{\frac{n}{2}}(r-s_1\vee s_2)^{\beta-\frac{\eta}{2}} } g(c(t-r), z-x).
$$

 In order to deal with ${\rm B}_n$, we use \eqref{bound:hermite:polynomial:Hn}, the estimate \eqref{gaussian:bound:diff:time:hat:pm:different:time} for $n=0$ and finally the space-time inequality \eqref{space:time:inequality}. This yields
 \begin{align*}
| {\rm B}_n | & \leq K \left\{ \frac{|z-x|}{t- r} \textbf{1}_{\left\{ n = 1 \right\}} + \left(\frac{|z-x|^2}{(t- r)^2}+  \frac{1}{t- r} \right) \textbf{1}_{\left\{ n = 2 \right\}} +  \left(\frac{|z-x|^3}{(t- r)^3}+  \frac{|z-x|}{(t- r)^2} \right) \textbf{1}_{\left\{ n = 3 \right\}}\right\}  \\
& \quad \times \frac{ |s_1-s_2|^\beta}{(r-s_1 \vee s_2)^{\beta-\frac{\eta}{2}}} g(c(t-r) , z-x)  \\
& \leq K \frac{|s_1-s_2|^\beta}{(t-r)^{\frac{n}{2}}(r-s_1 \vee s_2)^{\beta-\frac{\eta}{2}}} g(c(t-r) , z-x).
\end{align*}

Gathering the bounds on ${\rm A}_n$ and ${\rm B}_n$ allows to conclude the proof of \eqref{gaussian:bound:diff:time:hat:pm:different:time}. \\

\noindent \emph{Step 7: proof of \eqref{gaussian:bound:diff:time:Hm}.}\\

As in the previous steps, we first observe that the announced estimate directly follows from \eqref{iter:parametrix:kernel} with $k=1$ if $|s_1-s_2| > r-s_1\vee s_2$. From now on we assume that $|s_1-s_2| \leq r-s_1\vee s_2$. In order to prove the estimate \eqref{gaussian:bound:diff:time:Hm}, we use a similar decomposition as the one employed for $\Delta_{\mu,\mu'} \mH_{m+1}(\mu, s, r, t, x, y)$, namely
\begin{equation}
\label{decomposition:delta:time:parametrix:kernel}
 \Delta_{s_1, s_2} \mH_{m+1}(\mu, s, r, t, x, z) = {\rm I} + {\rm II} + {\rm III} + {\rm IV} 
\end{equation}

\noindent with
\begin{align*}
{\rm I} & := -\sum_{i=1}^d \Delta_{s_1, s_2} b_i(r, x, [X^{s, \xi ,(m)}_r])  \, \partial_{x_i} \widehat{p}_{m+1}(\mu, s_1 \vee s_2, r, t, x, z), \\
{\rm II} & := -\sum_{i=1}^d  b_i(r, x, [X^{s_1 \wedge s_2, \xi ,(m)}_r]) \,  \Delta_{s_1, s_2} \partial_{x_i} \widehat{p}_{m+1}(\mu, s_1 \vee s_2, r, t, x, z), \\
{\rm III} & := \frac12 \sum_{i, j=1}^d \Delta_{s_1, s_2}   [a_{i, j}(r, x, [X^{s, \xi , (m)}_r]) - a_{i, j}(r, z, [X^{s, \xi, (m)}_r])] \,  \partial^2_{x_i, x_j} \widehat{p}_{m+1}(\mu, s_1 \vee s_2, r, t, x, z), \\
{\rm IV} & := \frac12 \sum_{i, j=1}^d    [a_{i, j}(r, x, [X^{s_1 \wedge s_2, \xi , (m)}_r]) - a_{i, j}(r, z, [X^{s_1\wedge s_2, \xi, (m)}_r])] \, \Delta_{s_1, s_2} \partial^2_{x_i, x_j} \widehat{p}_{m+1}(\mu, s_1 \vee s_2, r, t, x, z). 
\end{align*}
 
We now prove the following estimates: for any $\beta \in [0,1]$ and any $\alpha \in [0,\eta]$, there exist some positive constants $K:=K(T, \HR, \HE), \, K_\alpha:=K(T,  \HR, \HE, \alpha),\, c:=c(\lambda)$ such that it holds
\begin{align*}
|{\rm I} | & \leq K \frac{|s_1-s_2|^{\beta}}{(t-r)^{\frac12}} \left\{ \frac{1}{(r-s_1)^{\beta-\frac{\eta}{2}}} + \frac{1}{(r-s_2)^{\beta-\frac{\eta}{2}}}  \right\}\, g(c(t-r), z-x), \\
|{\rm II} | & \leq K \frac{|s_1-s_2|^{\beta}}{(t-r)^{\frac12}(r-s_1\vee s_2)^{\beta-\frac{\eta}{2}}} \, g(c(t-r), z-x), \\
|{\rm III} | & \leq K_\alpha  \frac{|s_1-s_2|^{\beta}}{(t-r)^{1-\frac{\alpha}{2}}}\left\{\frac{1}{(r-s_1)^{\beta+\frac{\alpha-\eta}{2}}} + \frac{1}{(r-s_2)^{\beta+\frac{\alpha-\eta}{2}}}\right\} \, g(c(t-r), z-x), \\
|{\rm IV} | & \leq K \frac{|s_1-s_2|^{\beta}}{(t-r)^{1-\frac{\eta}{2}}(r-s_1\vee s_2)^{\beta-\frac{\eta}{2}}} \, g(c(t-r), z-x).
\end{align*}
 
The estimate on ${\rm I}$ follows from \eqref{diff:time:drift:diff:coefficients} and the space-time inequality \eqref{space:time:inequality}. The estimate on ${\rm II}$ follows from \eqref{gaussian:bound:diff:time:hat:pm:different:time} and the boundedness of the drift coefficient. The estimate on ${\rm III}$ is a consequence of \eqref{diff:time:with:holder:reg:space:drift:diff:coefficients} and the space-time inequality \eqref{space:time:inequality}. The estimate on ${\rm IV}$ follows by combining the uniform $\eta$-H\"older regularity of $ a(t, ., m)$, the estimate \eqref{gaussian:bound:diff:time:hat:pm:different:time} and the space-time inequality \eqref{space:time:inequality}. It now suffices to put together the above estimates and to set $\alpha=\eta$. \\

\noindent \emph{Step 8: proof of \eqref{gaussian:bound:diff:time:plus:holder:regularity:Hm}.}\\

We first note that in the off-diagonal regime, that is, if $|y-x|^2 > t-r$, the result directly follows from \eqref{gaussian:bound:diff:time:Hm}. From now on, we assume that the diagonal regime $|y-x|^2 \leq t-r$ holds. We investigate the H\"older regularity of each term appearing in the decomposition \eqref{decomposition:delta:time:parametrix:kernel}.

We write ${\rm I}(x) - {\rm I}(y) = {\rm I}_1(x)-{\rm I}_1(y) + {\rm I}_2(x) - {\rm I}_2(y)$ with
$$
{\rm I}_1(x) - {\rm I}_1(y) = -\sum_{i=1}^d [ \Delta_{s_1, s_2} [b_i(r, x, [X^{s, \xi ,(m)}_r]) - b_i(r, y, [X^{s, \xi ,(m)}_r]) ] ]  \, \partial_{x_i} \widehat{p}_{m+1}(\mu, s_1 \vee s_2, r, t, x, z)
$$
\noindent and 
$$
{\rm I}_2(x) - {\rm I}_2(y) = -\sum_{i=1}^d  \Delta_{s_1, s_2}[ b_i(r, y, [X^{s, \xi ,(m)}_r])]  \, [  \partial_{x_i} \widehat{p}_{m+1}(\mu, s_1 \vee s_2, r, t, x, z) - \partial_{x_i} \widehat{p}_{m+1}(\mu, s_1 \vee s_2, r, t, y, z) ].
$$

From \eqref{diff:time:with:holder:reg:space:drift:diff:coefficients} with $\alpha=\eta$ and the space-time inequality \eqref{space:time:inequality}, we directly obtain
\begin{align*}
|{\rm I}_1(x) - {\rm I}_1(y) | & \leq K (|y-x|^\eta \wedge 1) \frac{|s_1-s_2|^\beta}{(t-r)^\frac12(r-s_1 \vee s_2)^{\beta}} \, g(c(t-r), z-x) \\
& \leq K (|y-x|^\alpha \wedge 1)   \frac{|s_1-s_2|^\beta}{(t-r)^\frac{1+\alpha-\eta}{2}(r-s_1 \vee s_2)^{\beta}} \, g(c(t-r), z-x)
\end{align*}
\noindent for any $\alpha \in [0,\eta]$. From \eqref{diff:time:drift:diff:coefficients} and standard computations on Gaussian kernels, we obtain
\begin{align*}
|{\rm I}_2(x) - {\rm I}_2(y) | & \leq K  (|y-x| \wedge 1)  \frac{|s_1-s_2|^\beta}{(t-r)(r-s_1 \vee s_2)^{\beta}}  \, \left\{ g(c(t-r), z-x)  + g(c(t-r), z-y) \right\}\\
&  \leq K (|y-x|^\alpha \wedge 1) \frac{|s_1-s_2|^\beta}{(t-r)^\frac{1+\alpha}{2}(r-s_1 \vee s_2)^{\beta}} \, \left\{ g(c(t-r), z-x) + g(c(t-r), z-y) \right\}.
\end{align*}

We again write ${\rm II}(x) - {\rm II}(y) = {\rm II}_1(x) - {\rm II}_1(y) + {\rm II}_2(x) - {\rm II}_2(y)$ with
$$
{\rm II}_1(x) - {\rm II}_1(y) := -\sum_{i=1}^d  [ b_i(r, x, [X^{s_1 \wedge s_2, \xi ,(m)}_r]) - b_i(r, y, [X^{s_1 \wedge s_2, \xi ,(m)}_r]) ] \,  \Delta_{s_1, s_2} \partial_{x_i} \widehat{p}_{m+1}(\mu, s, r, t, x, z)
$$
\noindent and 
$$
{\rm II}_2(x) - {\rm II}_2(y) :=  -\sum_{i=1}^d  b_i(r, y, [X^{s_1 \wedge s_2, \xi ,(m)}_r])  \,  \Delta_{s_1, s_2} [\partial_{x_i} \widehat{p}_{m+1}(\mu, s, r, t, x, z) - \partial_{x_i} \widehat{p}_{m+1}(\mu, s, r, t, y, z)].
$$

On the one hand, from \eqref{gaussian:bound:diff:time:hat:pm:different:time} and the uniform boundedness and $\eta$-H\"older regularity of $ b_i(t, ., m)$, we get
\begin{align*}
| {\rm II}_1(x) - {\rm II}_1(y) |&  \leq K (|y-x|^\eta \wedge 1) \frac{|s_1-s_2|^\beta}{(t-r)^{\frac12}(r-s_1\vee s_2)^{\beta}} g(c(t-r), z-x) \\
& \leq K (|y-x|^\alpha \wedge 1) \frac{|s_1-s_2|^\beta}{(t-r)^{\frac{1+\alpha-\eta}{2}}(r-s_1\vee s_2)^{\beta}} g(c(t-r), z-x). 
\end{align*}

On the other hand, combining the mean-value theorem with \eqref{gaussian:bound:diff:time:hat:pm:different:time} and then using the inequality \eqref{diagonal:regime:heat:kernel} in the current diagonal regime $|y-x|^2 \leq t-r$, we obtain
\begin{align*}
| \Delta_{s_1, s_2} & [\partial_{x_i} \widehat{p}_{m+1}(\mu, s, r, t, x, z) - \partial_{x_i} \widehat{p}_{m+1}(\mu, s, r, t, y, z)]| \\
 &= | \int_0^1 \Delta_{s_1, s_2} \partial^2_{x, x_i} \widehat{p}_{m+1}(\mu, s_1 \vee s_2, r, t, \lambda x +(1-\lambda) y, z) . (x-y) \, d\lambda| \\
 & \leq K (|y-x|\wedge 1) \frac{ |s_1-s_2|^\beta}{(t-r) (r-s_1\vee s_2)^\beta} \, g(c(t-r), z-x) \\
 & \leq K (|y-x|^\alpha \wedge 1)  \frac{ |s_1-s_2|^\beta}{(t-r)^{\frac{1+\alpha}{2}} (r-s_1\vee s_2)^\beta} \, g(c(t-r), z-x).
\end{align*}
 The previous estimate and the boundedness of the drift coefficient then yield
\begin{align*}
| {\rm II}_2(x) - {\rm II}_2(y) | & \leq K (|y-x|^\alpha \wedge 1)\frac{|s_1-s_2|^\beta}{(t-r)^{\frac{1+\alpha}{2}}(r-s_1\vee s_2)^{\beta}} g(c(t-r), z-x). 
\end{align*}

For the third term, we use the decomposition ${\rm III}(x) - {\rm III}(y) = {\rm III}_1(x) - {\rm III}_1(y) + {\rm III}_2(x) - {\rm III}_2(y)$ with
$$
{\rm III}_1(x) - {\rm III}_1(y)  := \frac12 \sum_{i, j=1}^d \Delta_{s_1, s_2}   [a_{i, j}(r, x, [X^{s, \xi , (m)}_r]) - a_{i, j}(r, y, [X^{s, \xi, (m)}_r])] \,  \partial^2_{x_i, x_j} \widehat{p}_{m+1}(\mu, s_1 \vee s_2, r, t, x, z) 
$$
\noindent and 
\begin{align*}
{\rm III}_2(x) - {\rm III}_2(y)  & := \frac12 \sum_{i, j=1}^d \Delta_{s_1, s_2}   [a_{i, j}(r, y, [X^{s, \xi , (m)}_r]) - a_{i, j}(r, z, [X^{s, \xi, (m)}_r])] \\
& \quad  \times[ \partial^2_{x_i, x_j} \widehat{p}_{m+1}(\mu, s_1 \vee s_2, r, t, x, z) - \partial^2_{x_i, x_j} \widehat{p}_{m+1}(\mu, s_1 \vee s_2, r, t, y, z)].
\end{align*}

In order to deal with ${\rm III}_1(x) - {\rm III}_1(y)$, we use \eqref{diff:time:with:holder:reg:space:drift:diff:coefficients} with $\alpha = \eta$ and then the space-time inequality \eqref{space:time:inequality}
\begin{align*}
|{\rm III}_1(x) - {\rm III}_1(y) | & \leq K (|y-x|^\eta \wedge 1) \frac{|s_1-s_2|^\beta}{(t-r)(r-s_1\vee s_2)^\beta} \, g(c(t-r), z-x)\\
& \leq K (|y-x|^\alpha \wedge 1) \frac{|s_1-s_2|^\beta}{(t-r)^{1+\frac{\alpha-\eta}{2}}(r-s_1\vee s_2)^\beta} \, g(c(t-r), z-x).
\end{align*}

We deal with the difference ${\rm III}_2(x) - {\rm III}_2(y)$ by using the mean-value theorem together with the inequality \eqref{diagonal:regime:heat:kernel} (recall that $|y-x|^2 \leq t-r$), \eqref{diff:time:with:holder:reg:space:drift:diff:coefficients} with $\alpha = \eta$, the inequality $|y-z|^\eta \leq |y-x|^\eta + |z-x|^\eta \leq (t-r)^{\eta/2} + |z-x|^\eta$ and finally the space-time inequality \eqref{space:time:inequality}. Skipping some technical details, we obtain
\begin{align*}
|{\rm III}_2(x) - {\rm III}_2(y)| & \leq K |y-z|^\eta (|y-x| \wedge 1) \frac{|s_1-s_2|^\beta}{(t-r)^{\frac32}(r-s_1\vee s_2)^\beta} \, g(c(t-r), z-x) \\
& \leq K  (|y-x|^{\alpha} \wedge 1) \frac{|s_1-s_2|^\beta}{(t-r)^{1+\frac{\alpha-\eta}{2}}(r-s_1\vee s_2)^\beta} \, g(c(t-r), z-x).
\end{align*}

Finally, for the last term, we again write ${\rm IV}(x) - {\rm IV}(y) = {\rm IV}_1(x) - {\rm IV}_1(y) + {\rm IV}_2(x) - {\rm IV}_2(y)$ with
$$
{\rm IV}_1(x) - {\rm IV}_1(y)  := \frac12 \sum_{i, j=1}^d   [ a_{i, j}(r, x, [X^{s_1 \wedge s_2, \xi , (m)}_r]) - a_{i, j}(r, y, [X^{s_1\wedge s_2, \xi, (m)}_r]) ] \, \Delta_{s_1, s_2} \partial^2_{x_i, x_j} \widehat{p}_{m+1}(\mu, s, r, t, x, z)
$$
\noindent and 
\begin{align*}
{\rm IV}_2(x) - {\rm IV}_2(y) & :=  \frac12 \sum_{i, j=1}^d   [ a_{i, j}(r, y, [X^{s_1 \wedge s_2, \xi , (m)}_r]) - a_{i, j}(r, z, [X^{s_1\wedge s_2, \xi, (m)}_r]) ]  \\
& \quad \times \Delta_{s_1, s_2} [ \partial^2_{x_i, x_j} \widehat{p}_{m+1}(\mu, s, r, t, x, z) - \partial^2_{x_i, x_j} \widehat{p}_{m+1}(\mu, s, r, t, y, z)].
\end{align*}

From the uniform $\eta$-H\"older regularity of $a(t, ., m)$ and \eqref{gaussian:bound:diff:time:hat:pm:different:time}, we obtain
\begin{align*}
|{\rm IV}_1(x) - {\rm IV}_1(y)| & \leq K (|y-x|^\eta \wedge 1) \frac{|s_1-s_2|^\beta}{(t-r) (r-s_1\vee s_2)^{\beta}} \, g(c(t-r), z-x)\\
& \leq K  (|y-x|^{\alpha} \wedge 1)  \frac{|s_1-s_2|^\beta}{(t-r)^{1+\frac{\alpha-\eta}{2}}(r-s_1\vee s_2)^\beta} \, g(c(t-r), z-x).
\end{align*}

We handle the difference ${\rm IV}_2(x) - {\rm IV}_2(y)$ by employing similar arguments as those used to deal with ${\rm II}_2(x) - {\rm II}_2(y)$. Namely, we combine the mean-value theorem with \eqref{gaussian:bound:diff:time:hat:pm:different:time} and then use the inequality \eqref{diagonal:regime:heat:kernel} in our current diagonal regime $|y-x|^2 < t-r$. Hence,
\begin{align*}
| \Delta_{s_1, s_2} & [\partial^2_{x_i, x_j} \widehat{p}_{m+1}(\mu, s, r, t, x, z) - \partial^2_{x_i, x_j} \widehat{p}_{m+1}(\mu, s, r, t, y, z)]| \\
 &= | \int_0^1 \Delta_{s_1, s_2} \partial^3_{x, x_i, x_j} \widehat{p}_{m+1}(\mu, s, r, t, \lambda x +(1-\lambda) y, z) . (x-y) \, d\lambda| \\
 & \leq K (|y-x| \wedge 1) \frac{ |s_1-s_2|^\beta}{(t-r)^{\frac32} (r-s_1\vee s_2)^\beta} \, g(c(t-r), z-x) \\
 & \leq K (|y-x|^\alpha \wedge 1)  \frac{ |s_1-s_2|^\beta}{(t-r)^{1+\frac{\alpha}{2}} (r-s_1\vee s_2)^\beta} \, g(c(t-r), z-x).
\end{align*}

Now, the previous estimate, the uniform $\eta$-H\"older regularity of $ a(t, ., m)$, the inequality $|z-y|^\eta \leq |z-x|^\eta + |y-x|^\eta \leq |z-x|^\eta + (t-r)^{\frac{\eta}{2}}$ and finally the space-time inequality \eqref{space:time:inequality} imply
$$
|{\rm IV}_2(x) - {\rm IV}_2(y)| \leq K (|y-x|^\alpha \wedge 1)  \frac{ |s_1-s_2|^\beta}{(t-r)^{1+\frac{\alpha-\eta}{2}} (r-s_1\vee s_2)^\beta} \, g(c(t-r), z-x).
$$

We conclude the proof of \eqref{gaussian:bound:diff:time:plus:holder:regularity:Hm} by putting together the previous estimates.\\

\noindent \emph{Step 9: proof of \eqref{gaussian:bound:diff:time:Phim}.}\\

From the identity $\Phi_{m+1}(\mu, s, r, t, x, z) = \mH_{m+1}(\mu, s, r, t, x, z) + \mH_{m+1}\otimes \Phi_{m+1}(\mu, s, r, t, x, z)$, we obtain the following relation 
\begin{align}
\Delta_{s_1,s_2} \Phi_{m+1}(\mu, s, r, t, x, z) & = \Delta_{s_1, s_2} \mH_{m+1}(\mu, s, r, t, x, z) \nonumber  \\
& \quad + \int_r^t \int_{\mathbb{R}^d} \Delta_{s_1, s_2} \mH_{m+1}(\mu, s, r, r', x, y)  \Phi_{m+1}(\mu, s_1 \vee s_2, r', t, y, z) \, dr' \, dy \label{delta:s:Phi:mp1}\\
& \quad + \int_r^t \int_{\mathbb{R}^d}  \mH_{m+1}(\mu, s_1 \wedge s_2, r, r', x, y)  \Delta_{s_1, s_2} \Phi_{m+1}(\mu, s, r', t, y, z) \, dr' \, dy \nonumber
\end{align}
\noindent which in turn, using the estimates \eqref{gaussian:bound:diff:time:Hm} and \eqref{Gaussian:estimate:Phim}, yields
\begin{align*}
| \Delta_{s_1,s_2} \Phi_{m+1}(\mu, s, r, t, x, z) | & \leq K \frac{|s_1-s_2|^\beta}{(t-r)^{1-\frac{\eta}{2}}(r-s_1\vee s_2)^\beta} \, g(c(t-r), z-x) \\
& + K \int_{r}^t \int_{\mathbb{R}^d} \frac{1}{(t-r)^{1-\frac{\eta}{2}}} g(c(r'-r), y-x) \, |\Delta_{s_1, s_2} \Phi_m(\mu, s, r', t, y, z)| \, dy\, dr'. 
\end{align*}

The space-time kernel $[r, t] \times \mathbb{R}^d \ni(r', x) \mapsto (t-r')^{-1+\eta/2} (r'-s_1\vee s_2)^{-\beta} g(c(t-r'), z-x)$ being non-singular, the previous estimate implies
$$
| \Delta_{s_1,s_2} \Phi_{m+1}(\mu, s, r, t, x, z) | \leq K \frac{|s_1-s_2|^\beta}{(t-r)^{1-\frac{\eta}{2}}(r-s_1\vee s_2)^\beta} \, g(c(t-r), z-x).
$$

The proof of \eqref{gaussian:bound:diff:time:Phim} is now complete. \\
 
\noindent \emph{Step 10: proof of \eqref{gaussian:bound:diff:time:plus:holder:regularity:Phim}.}\\

From \eqref{delta:s:Phi:mp1}, \eqref{gaussian:bound:diff:time:plus:holder:regularity:Hm}, \eqref{Gaussian:estimate:Phim}, \eqref{holder:reg:parametrix:kernel} and \eqref{gaussian:bound:diff:time:Phim}, for any $\alpha \in [0,\eta)$, we obtain
\begin{align*}
| & \Delta_{s_1, s_2} [\Phi_{m+1}(\mu, s, r ,t , x ,z) -  \Phi_{m+1}(\mu, s, r ,t , y ,z)] |  \\
&  \leq K_\alpha (|y-x|^\alpha \wedge 1) \frac{|s_1-s_2|^{\beta}}{(t-r)^{1+\frac{\alpha- \eta}{2}}(r-s_1\vee s_2)^{\beta}}  \, \left\{ g(c(t-r), z-x) + g(c(t-r), z-y)\right\}\\
& \quad + K_\alpha (|y-x|^\alpha \wedge 1) \frac{|s_1-s_2|^\beta}{(r-s_1\vee s_2)^{\beta}} \int_{r}^t  \frac{1}{(t-r')^{1-\frac{\eta}{2}}(r'-r)^{1+\frac{\alpha-\eta}{2}}} dr' \left\{ g(c(t-r), z-x) + g(c(t-r), z-y)\right\}\\
&  \leq K_\alpha(|y-x|^\alpha \wedge 1) \frac{|s_1-s_2|^{\beta}}{(t-r)^{1+\frac{\alpha- \eta}{2}}(r-s_1\vee s_2)^{\beta}}  \, \left\{ g(c(t-r), z-x) + g(c(t-r), z-y)\right\}
\end{align*}
\noindent which concludes the proof.
 
\subsection{Proof of Lemma \ref{lem:estimate:reg:mes:L:derivatives:coefficients:densities}.}\label{section:proof:lem:estimate:reg:mes:L:derivatives:coefficients:densities}\quad\\

\noindent \emph{Step 1: proof of \eqref{dec:cross:deriv:diff:mu:a}.}

Below, we prove \eqref{dec:cross:deriv:diff:mu:a} only for the term corresponding to the $L$-derivative of the diffusion coefficient since the second term can be handled in a completely analogous way. We start from the identity \eqref{dec:cross:deriv:a} and deduce the decomposition
\begin{align*}
 &\partial^n_v[\partial_\mu  [a_{i, j }(t, x, [X^{s, \xi , (m)}_t])]](v)  -  \partial^n_v[\partial_\mu [a_{i, j }(t, x, [X^{s, \xi' , (m)}_t])]](v) \\
 & =  \int_{\mathbb{R}^d} \Big[ \frac{\delta a_{i, j}}{\delta m}(t, x, [X^{s,\xi, (m)}_t])(y') - \frac{\delta a_{i, j}}{\delta m}(t, x, [X^{s,\xi', (m)}_t])(y') \Big]  \, \partial^{1+n}_x p_{m}(\mu, s, t, v, y') \, dy' \\
 &  +  \int_{\mathbb{R}^d} \Big[ \frac{\delta a_{i, j}}{\delta m}(t, x, [X^{s,\xi', (m)}_t])(y') - \frac{\delta a_{i, j}}{\delta m}(t, x, [X^{s,\xi', (m)}_t])(v) \Big]  \, \Big[ \partial^{1+n}_x p_{m}(\mu, s, t, v, y') - \partial^{1+n}_x p_{m}(\mu', s, t, v, y') \Big] \, dy' \\
 &  +  \int_{(\mathbb{R}^d)^2} \Big[\frac{\delta a_{i, j}}{\delta m}(t, x, [X^{s,\xi, (m)}_t])(y') - \frac{\delta a_{i, j}}{\delta m}(t, x,  [X^{s, \xi', (m)}_t])(y') \Big] \, \partial^{n}_v[\partial_\mu p_{m}(\mu, s, t, x', y')](v) \, dy' \, \mu(dx') \\
 &  +  \int_{(\mathbb{R}^d)^2} \frac{\delta a_{i, j}}{\delta m}(t, x, [X^{s,\xi', (m)}_t])(y')  \, \partial^{n}_v[\partial_\mu p_{m}(\mu, s, t, x', y')](v) \, dy' \, (\mu -\mu')(dx') \\
 &  +  \int_{(\mathbb{R}^d)^2} \Big[\frac{\delta a_{i, j}}{\delta m}(t, x, [X^{s,\xi', (m)}_t])(y') - \frac{\delta a_{i, j}}{\delta m}(t, x,  [X^{s, \xi', (m)}_t])(x') \Big] \\
 & \quad \quad \times \Big[\partial^{n}_v[ \partial_\mu p_{m}(\mu, s, t, x', y')](v) - \partial^{n}_v[ \partial_\mu p_{m}(\mu', s, t, x', y')](v) \Big] \, dy' \, \mu'(dx')\\
 & =: {\rm A } + {\rm B} + {\rm C} + {\rm  D} + {\rm E}. 
\end{align*}

In order to deal with A, we distinguish the two cases $W^2_2(\mu, \mu')\geq t-s$ and $W^{2}_2(\mu, \mu') \leq t-s$. In the first case, it is enough to remark that
\begin{align*}
{\rm A} & = \int_{\mathbb{R}^d} \Big[ \frac{\delta a_{i, j}}{\delta m}(t, x, [X^{s,\xi, (m)}_t])(y') - \frac{\delta a_{i, j}}{\delta m}(t, x, [X^{s,\xi, (m)}_t])(v) \Big]   \, \partial^{1+n}_x p_{m}(\mu, s, t, v, y') \, dy' \\
& \quad -  \int_{\mathbb{R}^d} \Big[ \frac{\delta a_{i, j}}{\delta m}(t, x, [X^{s,\xi', (m)}_t])(y') - \frac{\delta a_{i, j}}{\delta m}(t, x, [X^{s,\xi', (m)}_t])(v) \Big]   \, \partial^{1+n}_x p_{m}(\mu, s, t, v, y') \, dy' 
\end{align*}
\noindent and then, for both term, to combine the uniform $\eta$-H\"older regularity of the map $ [\delta a_{i, j} / \delta m](t, x, m)(.)$ with the Gaussian estimate \eqref{bound:derivative:heat:kernel} and the space-time inequality \eqref{space:time:inequality}. This yields
$$
|{\rm A}| \leq K \frac{1}{(t-s)^{\frac{1+n-\eta}{2}}} \leq K \frac{W^{\beta}_2(\mu, \mu')}{(t-s)^{\frac{1+n+\beta-\eta}{2}}} 
$$
\noindent for any $\beta \in [0,1]$. In the second case, following similar lines of reasonings as those used to prove \eqref{diff:mes:drift:diff:coefficients}, we deduce that for any $\beta \in [\eta, 1]$, there exists $K^{+}:=K(T, \HRp, \HE)$ such that
\begin{align}
\Big| \frac{\delta a_{i, j}}{\delta m}(t, x, [X^{s,\xi, (m)}_t])(y') - \frac{\delta a_{i, j}}{\delta m}(t, x, [X^{s,\xi', (m)}_t])(y') \Big| \leq K^{+} \frac{W_2^{\beta}(\mu, \mu')}{(t-s)^{\frac{\beta-\eta}{2}}} \label{delta:mes:linear:function:deriv:diff:coeff}
\end{align}
\noindent for any $\beta \in [\eta, 1]$ and since $W_2(\mu, \mu') \leq (t-s)^{1/2}$ the above estimate remains valid for any $\beta \in [0,1]$. Using \eqref{bound:derivative:heat:kernel}, we again conclude
$$
|{\rm A}| \leq K^{+} \frac{W^{\beta}_2(\mu, \mu')}{(t-s)^{\frac{1+n+\beta-\eta}{2}}} 
$$
\noindent for any $\beta \in [0,1]$.
We deal with B by again splitting the computation into the two cases $W^2_2(\mu, \mu')\geq t-s$ and $W^{2}_2(\mu, \mu') \leq t-s$. In the first case, from the Gaussian estimate \eqref{bound:derivative:heat:kernel}, the uniform $\eta$-H\"older regularity of the map $ [\delta a_{i, j} / \delta m](t, x, m)(.)$ and the space-time inequality \eqref{space:time:inequality}, we get
$$
|{\rm B}| \leq K \frac{1}{(t-s)^{\frac{1+n-\eta}{2}}} \leq K \frac{W^{\beta}_2(\mu, \mu')}{(t-s)^{\frac{1+n+\beta-\eta}{2}}}
$$
\noindent for any $\beta \in [0,1]$. In the second case, using \eqref{gaussian:estimate:diff:mes:deriv:heat:kernel} or \eqref{gaussian:estimate:diff:mes:deriv:heat:kernel:part2:statement:proposition} and again the uniform $\eta$-H\"older regularity of the map $ [\delta a_{i, j} / \delta m](t, x, m)(.)$ as well as the space-time inequality \eqref{space:time:inequality}, we obtain
$$
|{\rm B}| \leq K_\beta \frac{W^{\beta}_2(\mu, \mu')}{(t-s)^{\frac{1+n+\beta-\eta}{2}}}
$$
\noindent for any $\beta \in [0,1]$ if $n=0$ or any $\beta \in [0,\eta)$ if $n=1$. 

In order to deal with C, we employ \eqref{first:second:estimate:induction:decoupling:mckean} (recall that $\mathscr{C}^{n, 0}_{\infty}=\lim_{m\uparrow \infty} \mathscr{C}^{n,0}_{m} < \infty$), \eqref{delta:mes:linear:function:deriv:diff:coeff} in the case $W_2^{2}(\mu, \mu') \leq t-s$ or the uniform boundedness of the map $ [\delta a_{i, j} / \delta m](t, x, m)(.)$ in the case $W_2^2(\mu, \mu') \geq t-s$. This yields
$$
|{\rm C}| \leq K \frac{W^{\beta}_2(\mu, \mu')}{(t-s)^{\frac{1+n+\beta-\eta}{2}}}
$$
\noindent for any $\beta \in [0,1]$ if $n=0$ or any $\beta \in [0,\eta)$ if $n=1$. 

In order to deal with D, it suffices to use \eqref{equicontinuity:second:third:estimate:decoupling:mckean} (recall again that $\mathscr{C}^{n, \beta}_{\infty}=\lim_{m\uparrow \infty} \mathscr{C}^{n,\beta}_{m} < \infty$) which combined with the uniform boundedness of the map $ [\delta a_{i, j} / \delta m](t, x, m)(.)$ directly yields
\begin{align*}
 \Big| \int_{\mathbb{R}^d} \frac{\delta a_{i, j}}{\delta m}(t, x, [X^{s,\xi', (m)}_t])(y')  &\, [ \partial^{n}_v[\partial_\mu p_{m}(\mu, s, t, x', y')](v) - \partial^{n}_v[\partial_\mu p_{m}(\mu, s, t, x'', y')](v)] \, dy' \Big| \\
 & \leq K_\beta \frac{|x-x'|^\beta}{(t-s)^{\frac{1+n+\beta-\eta}{2}}}
\end{align*}
\noindent for any $\beta \in [0,1]$ if $n=0$ and any $\beta \in [0,\eta)$ if $n=1$. We thus deduce
$$
|{\rm D}| \leq K_\beta \frac{W_2^{\beta}(\mu, \mu')}{(t-s)^{\frac{1+n+\beta-\eta}{2}}}.
$$

Finally, again from the boundedness and uniform $\eta$-H\"older regularity of $ [\delta a_{i, j} / \delta m](t, x, m)(.)$, we obtain
$$
|{\rm E}| \leq K  \int_{(\mathbb{R}^d)^2} (|y'-x'|^\eta \wedge 1) |\partial^{n}_v[ \partial_\mu p_{m}(\mu, s, r, x', y')](v) - \partial^{n}_v[ \partial_\mu p_{m}(\mu', s, r, x', y')](v)| \, dy' \, \mu(dx').
$$
Gathering the previous estimates allows to conclude the proof. \\

\noindent \emph{Step 2: proof of \eqref{diff:L:deriv:pm:mu:mup}. }\\

From the identity \eqref{representation:formula:deriv:mes:p:hat}, we obtain the following decomposition:
$$
\Delta_{\mu, \mu'} \partial^{n}_v [\partial_\mu \widehat{p}_{m+1}(\mu, s, r, t, x, z)] (v) =   {\rm I} + {\rm II} + {\rm III} + {\rm IV}, 
$$

\noindent with
\begin{align*}
{\rm I} & := -\frac12 \tr\left(\Big[ \left(\int_r^t a(r', y, [X^{s, \xi, (m)}_{r'}])\, dr'\right)^{-1} - \left(\int_r^t a(r', y, [X^{s, \xi', (m)}_{r'}])\, dr'\right)^{-1} \Big]  \right.\\
& \quad \quad\left.\int_r^t \partial^n_v[\partial_\mu [a(r', y, [X^{s, \xi, (m)}_{r'}])]](v) \, dr' \right) \times  \widehat{p}^{y}_{m+1}(\mu, s, r, t, x, z), \\
{\rm II} & := -\frac12   \tr\left(\left(\int_r^t a(r', y, [X^{s, \xi', (m)}_{r'}])\, dr'\right)^{-1}  \int_r^t [ \partial^n_v[\partial_\mu [a(r', y, [X^{s, \xi, (m)}_{r'}])]](v) -  \partial^n_v[\partial_\mu [a(r', y, [X^{s, \xi', (m)}_{r'}])]](v) ]\, dr'  \right) \\
& \quad \times \widehat{p}^{y}_{m+1}(\mu, s, r, t, x, z) \\
{\rm III} & :=  \frac12  (z-x)^{t} \Big[ \left(\int_r^t a(r', y, [X^{s, \xi, (m)}_{r'}])\, dr'\right)^{-1} - \left(\int_r^t a(r', y, [X^{s, \xi', (m)}_{r'}])\, dr'\right)^{-1} \Big] \\
& \quad \quad \int_r^t \partial^n_v[\partial_\mu a(r', y, [X^{s, \xi, (m)}_{r'}])](v) \, dr' \,  \left(\int_r^t a(r', y, [X^{s, \xi, (m)}_{r'}])\, dr'\right)^{-1} (z-x) \times \widehat{p}^{y}_{m+1}(\mu, s, r, t, x, z) \\
& \quad +  \frac12  (z-x)^{t} \left(\int_r^t a(r', y, [X^{s, \xi', (m)}_{r'}])\, dr'\right)^{-1} \int_r^t [ \partial^n_v[\partial_\mu [a(r', y, [X^{s, \xi, (m)}_{r'}])]](v) -  \partial^n_v[\partial_\mu [a(r', y, [X^{s, \xi', (m)}_{r'}])]](v) ]\, dr'  \\
&  \quad \times \left(\int_r^t a(r', y, [X^{s, \xi, (m)}_{r'}])\, dr'\right)^{-1} (z-x) \times \widehat{p}^{y}_{m+1}(\mu, s, r, t, x, z)  \\
& \quad + \frac12  (z-x)^{t} \left(\int_r^t a(r', y, [X^{s, \xi', (m)}_{r'}])\, dr'\right)^{-1} \int_r^t \partial^n_v[\partial_\mu [a(r', y, [X^{s, \xi', (m)}_{r'}])]](v) \, dr'  \\
&  \quad \times \Big[ \left(\int_r^t a(r', y, [X^{s, \xi, (m)}_{r'}])\, dr'\right)^{-1} - \left(\int_r^t a(r', y, [X^{s, \xi', (m)}_{r'}])\, dr'\right)^{-1} \Big] (z-x) \times \widehat{p}^{y}_{m+1}(\mu, s, r, t, x, z), \\
{\rm IV} & :=  -\frac12\left\{ \tr\left(\left(\int_r^t a(r', y, [X^{s, \xi', (m)}_{r'}])\, dr'\right)^{-1} \int_r^t \partial^n_v[\partial_\mu [a(r', y, [X^{s, \xi', (m)}_{r'}])]](v) \, dr' \right) \right. \notag\\
& \quad \left. - (z-x)^{t} \left(\int_r^t a(r', y, [X^{s, \xi', (m)}_{r'}])\, dr'\right)^{-1} \int_r^t \partial^n_v[\partial_\mu [a(r', y, [X^{s, \xi', (m)}_{r'}])]](v) \, dr'  \right. \\ 
& \quad \left. \quad \times \left(\int_r^t a(r', y, [X^{s, \xi', (m)}_{r'}])\, dr'\right)^{-1} (z-x) \right\} [ \widehat{p}^{y}_{m+1}(\mu, s, r, t, x, z) - \widehat{p}^{y}_{m+1}(\mu', s, r, t, x, z) ].
\end{align*}

From the mean-value theorem, the estimate \eqref{recursive:bound:deriv:a:or:b} together with \eqref{first:second:estimate:induction:decoupling:mckean} (recall that $\mathscr{C}^{n, 0}_{\infty}=\lim_{m\uparrow \infty} \mathscr{C}^{n,0}_{m} < \infty$), we obtain
$$
|{\rm I}| \leq \frac{K}{(t-r)^2} \int_r^t \max_{i, j} |a_{i, j}(r', y, [X^{s, \xi, (m)}_{r'}]) - a_{i, j}(r', y, [X^{s, \xi', (m)}_{r'}])| \, dr' \int_r^t (r'-s)^{-\frac{1+n-\eta}{2}} \, dr' \,  g(c(t-r), z-x).
$$
 Now, for the first time integral appearing in the right-hand side of the above inequality, we use \eqref{diff:mes:drift:diff:coefficients} if $W_2^2(\mu, \mu') \leq r'-s$ (which is thus valid for any $\beta \in [0,1]$) or the uniform boundedness of the diffusion coefficient if $W_2^2(\mu, \mu') \geq r'-s$. This yields
 \begin{align*}
 |{\rm I}| & \leq \frac{K}{(t-r)^2} \int_r^t \frac{W^{\beta}_2(\mu, \mu')}{(r'-s)^{\frac{\beta}{2}}} \, dr' \,  \int_r^t (r'-s)^{- \frac{1+n-\eta}{2}} \, dr' \,  g(c(t-r), z-x)\\
 & \leq K W^{\beta}_2(\mu, \mu') \left(\frac{1}{(t-s)^{\frac{1+n+\beta-\eta}{2}}} \I_\seq{r=s} + \frac{1}{(r-s)^{\frac{1+n+\beta-\eta}{2}}} \I_\seq{r>s} \right) \, g(c(t-r), z-x)
 \end{align*}
 \noindent for any $\beta \in [0,1]$.
 
 We deal with ${\rm II}$ by using \eqref{dec:cross:deriv:diff:mu:a} so that
 \begin{align*}
 |{\rm II}| &  \leq K_\beta^{+} \left(W^{\beta}_2(\mu, \mu')  \left(\frac{1}{(t-s)^{\frac{1+n+\beta-\eta}{2}}} \I_\seq{r=s} + \frac{1}{(r-s)^{\frac{1+n+\beta-\eta}{2}}} \I_\seq{r>s}\right) \right. \\
 & \left. \quad + \frac{1}{t-r} \int_r^t   \int_{(\rr^d)^2} (|y'-x'|^\eta \wedge 1) \, | \Delta_{\mu, \mu'} \partial^{n}_v[\partial_\mu p_{m}(\mu, s, r', x', y')](v)| \, dy' \, \mu(dx') \, dr' \right) \, g(c(t-r), z-x)
 \end{align*}
 \noindent for any $\beta \in [0,1]$ if $n=0$ and any $\beta \in [0,\eta)$ if $n=1$.
  
 We deal with ${\rm III}$ by using similar estimates as those employed to deal with ${\rm I}$ and ${\rm II}$. Omitting some technical details, we obtain
 \begin{align*}
 |{\rm III}| &  \leq K_\beta^{+}  \left( W^{\beta}_2(\mu, \mu')\left(\frac{1}{(t-s)^{\frac{1+n+\beta-\eta}{2}}} \I_\seq{r=s} + \frac{1}{(r-s)^{\frac{1+n+\beta-\eta}{2}}} \I_\seq{r>s}\right) \right. \\
 & \left. \quad + \frac{1}{t-r} \int_r^t   \int_{(\rr^d)^2} (|y'-x'|^\eta \wedge 1) \, | \Delta_{\mu, \mu'} \partial^{n}_v[\partial_\mu p_{m}(\mu, s, r', x', y')](v)| \, dy' \, \mu(dx') \, dr' \right) \, g(c(t-r), z-x).
 \end{align*} 
 
 In order to deal with ${\rm IV}$, we use again the estimate \eqref{recursive:bound:deriv:a:or:b} together with \eqref{first:second:estimate:induction:decoupling:mckean}, either the estimates \eqref{gaussian:bound:diff:deriv:hat:pm:different:time},\eqref{gaussian:bound:diff:deriv:hat:pm:same:time}, with $n=0$, if $W_2(\mu, \mu') \leq r-s$ (which are thus valid for any $\beta \in [0,1]$) or the Gaussian upper-estimate for the kernels $z \mapsto \widehat{p}^{y}_{m+1}(\mu, s, r, t, x, z)$ and $z\mapsto \widehat{p}^{y}_{m+1}(\mu', s, r, t, x, z)$ if $W_2(\mu, \mu') \geq r-s$ and finally the space-time inequality \eqref{space:time:inequality}. We thus obtain
 \begin{align*}
 |{\rm IV}| & \leq K W^{\beta}_2(\mu, \mu') \left(\frac{1}{(t-s)^{\frac{1+n+\beta-\eta}{2}}} \I_\seq{r=s} + \frac{1}{(r-s)^{\frac{1+n+\beta-\eta}{2}}} \I_\seq{r>s} \right) \, g(c(t-r), z-x)
 \end{align*}
 \noindent for any $\beta \in [0,1]$.
 
 This last inequality concludes the proof of the estimate \eqref{diff:L:deriv:pm:mu:mup}. \\

 \noindent \emph{Step 3: proof of \eqref{diff:mes:L:deriv:diff:diff:coeff:holder:reg}.} \\

 The relation \eqref{eq:decompJ} directly gives the following decomposition
$$
 \Delta_{\mu, \mu'} \partial^{n}_v[ \partial_\mu [a_{i, j}(t, x, [X^{s, \xi , (m)}_t]) - a_{i, j}(t, z, [X^{s, \xi , (m)}_t]) ]](v)   = {\rm I}_{i, j} + {\rm II}_{i, j} + {\rm III}_{i, j} + {\rm IV}_{i, j} + {\rm V}_{i, j},
$$

\noindent with 
\begin{align*}
{\rm I}_{i, j}   & := \int_{\mathbb{R}^d} \Big[ \frac{\delta a_{i, j}}{\delta m}(t, x, [X^{s, \xi , (m)}_t])(y') - \frac{\delta a_{i, j}}{\delta m}(t, z, [X^{s, \xi , (m)}_t])(y') \\
& \quad - ( \frac{\delta a_{i, j}}{\delta m}(t, x, [X^{s, \xi' , (m)}_t])(y') - \frac{\delta a_{i, j}}{\delta m}(t, z, [X^{s, \xi' , (m)}_t])(y') ) \Big]  \partial^{1+n}_x p_{m}(\mu, s, t, v, y') \,dy',  \\
{\rm II}_{i, j}  & :=  \int_{\mathbb{R}^d} \Big[  \frac{\delta a_{i, j}}{\delta m}(t, x, [X^{s, \xi' , (m)}_t])(y') - \frac{\delta a_{i, j}}{\delta m}(t, z,  [X^{s, \xi' , (m)}_t])(y') \\
& \quad - (\frac{\delta a_{i, j}}{\delta m}(t, x, [X^{s, \xi' , (m)}_t])(v) - \frac{\delta a_{i, j}}{\delta m}(t, z, [X^{s, \xi' , (m)}_t])(v) ) \Big] \, \Delta_{\mu, \mu'} \partial^{1+n}_x p_{m}(\mu, s, t, v, y')  \, dy', \\
{\rm III}_{i, j} & := \int_{(\mathbb{R}^d)^2} \Big[\frac{\delta a_{i, j}}{\delta m}(t, x, [X^{s, \xi , (m)}_t])(y') - \frac{\delta a_{i, j}}{\delta m}(t, z, [X^{s, \xi , (m)}_t])(y') \Big] \partial^{n}_v [\partial_\mu p_{m}(\mu, s, t, x', y')](v) \, dy' \,  (\mu-\mu')(dx'), \\
{\rm IV}_{i, j} & := \int_{(\mathbb{R}^d)^2}  \Big[ \frac{\delta a_{i, j}}{\delta m}(t, x, [X^{s, \xi , (m)}_r])(y') - \frac{\delta a_{i, j}}{\delta m}(t, z,  [X^{s, \xi , (m)}_t])(y') \\
& \quad - ( \frac{\delta a_{i, j}}{\delta m}(t, x, [X^{s, \xi' , (m)}_t])(y') - \frac{\delta a_{i, j}}{\delta m}(t, z,  [X^{s, \xi' , (m)}_t])(y') ) \Big]   \partial^{n}_v [\partial_\mu p_{m}(\mu, s, t, x', y')](v) \, dy' \,  \mu'(dx'),  \\
{\rm V}_{i, j} &  := \int_{(\mathbb{R}^d)^2} \Big[  \frac{\delta a_{i, j}}{\delta m}(t, x, [X^{s, \xi' , (m)}_t])(y') - \frac{\delta a_{i, j}}{\delta m}(t, z, [X^{s, \xi' , (m)}_t])(y') \\
& -(\frac{\delta a_{i, j}}{\delta m}(t, x, [X^{s, \xi' , (m)}_t])(x') - \frac{\delta a_{i, j}}{\delta m}(t, z, [X^{s, \xi' , (m)}_t])(x') ) \Big]  \Delta_{\mu, \mu'} \partial^{n}_v [\partial_\mu p_{m}(\mu, s, t, x', y')](v) \, dy'  \mu'(dx').
\end{align*} 
 
 As previously done, we quantify the contribution of each term in the above decomposition. We establish a bound similar to \eqref{diff:mes:with:holder:reg:space:drift:diff:coefficients} but with the map $[\delta a_{i, j}/\delta m]$ instead of $a_{i, j}$. To be more specific, let $\Theta^{(m)}_{\lambda, t}:= (1-\lambda)[X^{s, \xi, (m)}_t] + \lambda [X^{s, \xi', (m)}_t]$. We write
 \begin{align*}
h(x)& := \frac{\delta a_{i, j}}{\delta m} (t, x, [X^{s, \xi, (m)}_t])(y')  - \frac{\delta a_{i, j}}{\delta m}(t, x , [X^{s, \xi', (m)}_t])(y') \\
& = \int_0^1\int_{\mathbb{R}^d} \frac{\delta^2 a_{i, j}}{\delta m^2}(t, x, \Theta^{(m)}_{t, \lambda})(y', z') (p_m(\mu, s, t, z') - p_{m}(\mu', s, t, z')) \, dz' \, d\lambda.
\end{align*}

From similar arguments as those used to derive \eqref{diff:mes:with:holder:reg:space:drift:diff:coefficients} we get
\begin{equation}
\label{holder:property:h:delta:mes}
|h(x) - h(z)| \leq K^{+} (|z-x|^{\eta-\alpha} \wedge 1) \frac{ W^{\beta}_2(\mu, \mu')}{(t-s)^{\frac{\beta-\alpha}{2}}},
\end{equation}
 \noindent for any $\alpha \in [0,\eta]$ and any $\beta \in [\alpha, 1]$. Hence, taking $\alpha =0$ in the preceding inequality and using \eqref{bound:derivative:heat:kernel}, we derive
%
%
$$
|{\rm I}_{i, j}| \leq K^{+} (|z-x|^\eta \wedge 1) \frac{W^{\beta}_2(\mu, \mu')}{(t-s)^{\frac{1+n+\beta}{2}}}.
$$
Now, if $W_2(\mu, \mu') \leq (t-s)^{1/2}$, we take $\alpha = \eta $ in \eqref{holder:property:h:delta:mes}. This yields
$$
|{\rm I}_{i, j}| \leq K \frac{W^{\beta}_2(\mu, \mu')}{(t-s)^{\frac{1+n+\beta-\eta}{2}}}.
$$
\noindent for any $\beta \in [\eta, 1]$. Since $W_2(\mu, \mu') \leq (t-s)^{1/2}$, the above estimate is still valid for any $\beta \in [0,1]$. Otherwise, if $W_2(\mu, \mu') \geq (t-s)^{1/2}$, we instead write 
\begin{align*}
{\rm I}_{i, j}   & := \int_{\mathbb{R}^d} \Big[ \frac{\delta a_{i, j}}{\delta m}(t, x, [X^{s, \xi , (m)}_t])(y') - \frac{\delta a_{i, j}}{\delta m}(t, x, [X^{s, \xi , (m)}_t])(v) \Big]  \partial^{1+n}_x p_{m}(\mu, s, t, v, y') \,dy' \\
&  \quad - \int_{\mathbb{R}^d} \Big[ \frac{\delta a_{i, j}}{\delta m}(t, z, [X^{s, \xi , (m)}_t])(y') -  \frac{\delta a_{i, j}}{\delta m}(t, z, [X^{s, \xi , (m)}_t])(v) \Big]  \partial^{1+n}_x p_{m}(\mu, s, t, v, y') \,dy' \\
& \quad - \int_{\mathbb{R}^d} \Big[ \frac{\delta a_{i, j}}{\delta m}(t, x, [X^{s, \xi' , (m)}_t])(y')  - \frac{\delta a_{i, j}}{\delta m}(t, x, [X^{s, \xi' , (m)}_t])(v) \Big] \partial^{1+n}_x p_{m}(\mu, s, t, v, y') \,dy' \\
& \quad + \int_{\mathbb{R}^d} \Big[ \frac{\delta a_{i, j}}{\delta m}(t, z, [X^{s, \xi' , (m)}_t])(y') -  \frac{\delta a_{i, j}}{\delta m}(t, z, [X^{s, \xi' , (m)}_t])(v) \Big]  \partial^{1+n}_x p_{m}(\mu, s, t, v, y') \,dy'
\end{align*}
\noindent and use the uniform $\eta$-H\"older regularity of the map $ [\delta a_{i, j} / \delta m](t, z, m)(.)$ together with the space-time inequality \eqref{space:time:inequality}. We thus derive
$$
| {\rm I}_{i, j}| \leq K \frac{W_2^{\beta}(\mu, \mu')}{(t-s)^{\frac{1+n+\beta-\eta}{2}}}
$$

\noindent for any $\beta \in [0,1]$. Gathering the above estimates, we conclude
$$
| {\rm I}_{i, j}| \leq K W^{\beta}_2(\mu, \mu') \left\{ \frac{ (|z-x|^\eta \wedge 1)}{(t-s)^{\frac{1+n+\beta}{2}}} \wedge \frac{1}{(t-s)^{\frac{1+n+\beta-\eta}{2}}} \right\}
$$
\noindent for any $\beta \in [0,1]$.

From the estimates \eqref{gaussian:estimate:diff:mes:deriv:heat:kernel} if $W^2_2(\mu, \mu') \leq t-s$, \eqref{bound:derivative:heat:kernel} if $W_2^{2}(\mu, \mu') \geq t-s$, \eqref{gaussian:estimate:diff:mes:deriv:heat:kernel:part2:statement:proposition}, the uniform $\eta$-H\"older regularity of $ [\delta a_{i, j} / \delta m](t, ., m)(.)$ and finally the space-time inequality \eqref{space:time:inequality}, we obtain
$$
| {\rm II}_{i, j} | \leq K_\beta W_2^{\beta}(\mu, \mu') \left\{ \frac{ (|z-x|^\eta \wedge 1)}{(t-s)^{\frac{1+n+\beta}{2}}} \wedge \frac{1}{(t-s)^{\frac{1+n+\beta-\eta}{2}}} \right\}
$$
\noindent for any $\beta \in [0,1]$ if $n=0$ or any $\beta \in [0,\eta)$ if $n=1$. 

From \eqref{equicontinuity:second:third:estimate:decoupling:mckean} and the uniform $\eta$-H\"older regularity of $ [\delta a_{i, j} / \delta m](t, ., m)(v)$, we deduce
\begin{align*}
\Big| \int_{\mathbb{R}^d} \Big[\frac{\delta a_{i, j}}{\delta m}(t, x, [X^{s, \xi , (m)}_t])(y') & - \frac{\delta a_{i, j}}{\delta m}(t, z, [X^{s, \xi , (m)}_t])(y') \Big]  [ \partial^{n}_v [\partial_\mu p_{m}(\mu, s, t, x', y')](v) - \partial^{n}_v [\partial_\mu p_{m}(\mu, s, t, x'', y')](v) ] \, dy' \Big| \\
& \leq K_\beta (|z-x|^\eta \wedge 1) \frac{|x'-x''|^\beta}{(t-s)^{\frac{1+n+\beta-\eta}{2}}}
\end{align*}
\noindent for any $\beta \in [0,1]$ if $n=0$ or any $\beta \in [0,\eta)$ if $n=1$ so that
$$
| {\rm III}_{i, j} | \leq K_\beta (|z-x|^\eta \wedge 1) \frac{W_2^{\beta}(\mu, \mu')}{(t-s)^{\frac{1+n+\beta-\eta}{2}}}.
$$
%

Then, we use \eqref{holder:property:h:delta:mes} with $\alpha=0$ and \eqref{first:second:estimate:induction:decoupling:mckean} (recall that $\mathscr{C}^{n, 0}_{\infty}=\lim_{m\uparrow \infty} \mathscr{C}^{n,0}_{m} < \infty$) so that similarly to ${\rm I}_{i, j}$ we get
$$
| {\rm IV}_{i, j} | \leq K ( |z-x|^\eta \wedge 1) \frac{W_2^{\beta}(\mu, \mu')}{(t-s)^{\frac{1+n+\beta-\eta}{2}}} 
$$
\noindent for any $\beta \in [0,1]$. 

In order to deal with ${\rm V}_{i, j}$, we either use the uniform $\eta$-H\"older regularity of $ [\delta a_{i, j}/\delta m](t, ., m)](v)$ or the uniform $\eta$-H\"older regularity of $ [\delta a_{i, j}/\delta m](t, x, m)](.)$. We obtain
\begin{align*}
  | {\rm V}_{i, j} | &  \leq K \left\{ (|z-x|^\eta\wedge 1) \int_{(\mathbb{R}^d)^2}  |\Delta_{\mu, \mu'} \partial^{n}_v [\partial_\mu p_{m}(\mu, s, t, x', y')](v)|  \, dy'  \mu'(dx') \right. \\
  & \quad \left. \wedge \, \int_{(\mathbb{R}^d)^2} (|y'-x'|^\eta \wedge 1)  |\Delta_{\mu, \mu'} \partial^{n}_v [\partial_\mu p_{m}(\mu, s, t, x', y')](v)|  \, dy'  \mu'(dx') \right\}.
\end{align*}

Gathering the above estimates completes the proof of \eqref{diff:mes:L:deriv:diff:diff:coeff:holder:reg}. \\

\noindent \emph{Step 4: proof of \eqref{diff:mes:L:deriv:parametrix:kernel:pmp1}.} \\

 From \eqref{deriv:mu:H:mp1}, we easily obtain the following decomposition
$$
\Delta_{\mu, \mu'} \partial^{n}_v [\partial_\mu \mH_{m+1}(\mu, s, r, t, x,  z)] (v)  =   {\rm A} + {\rm B} + {\rm C} +{\rm D }+ {\rm E}
$$

\noindent with
\begin{align*}
{\rm A }& := {\rm A_1 }+ {\rm A_2}, \\
{\rm A_1} & := -\sum_{i=1}^d \Delta_{\mu, \mu'}  \partial^{n}_v [\partial_\mu [b_i(r, x, [X^{s, \xi , (m)}_r])]](v) \, \partial_{x_i} \widehat{p}_{m+1}(\mu, s, r, t , x , z),  \\
{\rm A_2} & :=  -\sum_{i=1}^d \partial^{n}_v [\partial_\mu [b_i(r, x, [X^{s, \xi', (m)}_r])]](v) \, \Delta_{\mu, \mu'}  \partial_{x_i} \widehat{p}_{m+1}(\mu, s,  r, t , x , z),  \\
{\rm B } & := {\rm B_1} + {\rm B_2 }, \\
{\rm B_1} & := \frac12 \sum_{i, j=1}^d \Delta_{\mu, \mu'}  \partial^{n}_v [\partial_\mu [a_{i, j}(r, x, [X^{s, \xi , (m)}_r]) - a_{i, j}(r, z, [X^{s, \xi , (m)}_r]) ]](v) \, \partial^2_{x_i, x_j} \widehat{p}_{m+1}(\mu, s, r, t , x , z),  \\
{\rm B_2} & :=  \frac12 \sum_{i, j=1}^d \partial^{n}_v [\partial_\mu [a_{i, j} (r, x, [X^{s, \xi' , (m)}_r]) - a_{i, j}(r, z, [X^{s, \xi', (m)}_r]) ]](v) \, \Delta_{\mu, \mu'}  \partial^2_{x_i, x_j}\widehat{p}_{m+1}(\mu, s, r, t , x , z),  \\
{\rm C }& := {\rm C_1 }+ {\rm C_2} + {\rm C_3}, \\
{\rm C_1} & := -\sum_{i=1}^d \Delta_{\mu, \mu'}  b_i(r, x [X^{s, \xi , (m)}_r]) \, \partial^{n}_v \left[\partial_\mu H^{i}_1\left(\int_r^t a(r', z, [X^{s, \xi ,(m)}_{r'}]) dr', z-x\right)\right](v) \, \widehat{p}_{m+1}(\mu, s, r, t , x , z),  \\
{\rm C_2 }& :=  -\sum_{i=1}^d b_i(r, x, [X^{s, \xi' , (m)}_r]) \, \Delta_{\mu, \mu'} \partial^{n}_v \left[\partial_\mu H^{i}_1\left(\int_r^t a(r', z, [X^{s, \xi ,(m)}_{r'}]) dr', z-x\right)\right](v) \, \widehat{p}_{m+1}(\mu, s,  r, t , x , z),  \\
{\rm C_3} & := -\sum_{i=1}^d b_i(r, x, [X^{s, \xi', (m)}_r]) \,  \partial^{n}_v \left[\partial_\mu H^{i}_1\left(\int_r^t a(r', z, [X^{s, \xi' ,(m)}_{r'}]) dr', z-x\right)\right](v) \, \Delta_{\mu, \mu'} \widehat{p}_{m+1}(\mu, s, r, t , x , z),
\end{align*}
\begin{align*}
{\rm D} & := {\rm D_1} + {\rm D_2} + {\rm D_3}, \\
{\rm D_1} & := \frac12 \sum_{i, j=1}^d \Delta_{\mu, \mu'} \Big(a_{i, j}(r, x, [X^{s, \xi , (m)}_r]) - a_{i, j}(r, z, [X^{s, \xi , (m)}_r])\Big) \\
& \quad \quad \quad \times \partial^{n}_v \left[\partial_\mu H^{i, j}_2\left(\int_r^t a(r', z, [X^{s, \xi ,(m)}_{r'}]) dr', z-x \right)\right](v) \, \widehat{p}_{m+1}(\mu, s, r, t , x , z),  \\
{\rm D_2} & :=  \frac12 \sum_{i, j=1}^d ( a_{i, j} (r, x, [X^{s, \xi' , (m)}_r]) - a_{i, j}(r, z, [X^{s, \xi', (m)}_r]) ) \\
& \quad\quad \quad \times \Delta_{\mu, \mu'} \partial^{n}_v \left[\partial_\mu H^{i, j}_2\left(\int_r^t a(r', z, [X^{s, \xi ,(m)}_{v'}]) dr', z-x \right)\right](v) \, \widehat{p}_{m+1}(\mu, s, r, t , x , z),  \\
{\rm D_3} & :=  \frac12 \sum_{i, j=1}^d ( a_{i, j}(r, x, [X^{s, \xi' , (m)}_r]) - a_{i, j}(r, z, [X^{s, \xi' , (m)}_r]) ) \\
& \quad \quad \quad \times  \partial^{n}_v \left[\partial_\mu H^{i, j}_2\left(\int_r^t a(r', z, [X^{s, \xi' ,(m)}_{r'}]) dr', z-x\right)\right](v) \,  \Delta_{\mu, \mu'} \widehat{p}_{m+1}(\mu, s, r, t , x , z),
\end{align*}

\noindent and
\begin{align*}
{\rm E }& := {\rm E_1} + {\rm E_2} + {\rm E_3 }\\
{\rm E_1} & := -\sum_{i=1}^d\Delta_{\mu, \mu'} \Big[ b_i(r, x, [X^{s, \xi , (m)}_r])H^{i}_1\left(\int_r^t a(r', z, [X^{s, \xi ,(m)}_{r'}]) dr', z-x \right)\Big] \, \partial^{n}_v [\partial_\mu \widehat{p}_{m+1}(\mu, s,  r, t , x , z)](v),  \\
{\rm E_2} & := \frac12 \sum_{i, j =1}^d \Delta_{\mu, \mu'} \Big[\Big(a_{i, j}(r, x, [X^{s, \xi , (m)}_r]) - a_{i, j}(r, z, [X^{s, \xi , (m)}_r])\Big) \, H^{i, j}_2\left(\int_r^t a(r', z, [X^{s, \xi ,(m)}_{r'}]) dr', z-x \right)\Big] \\
& \quad \quad \times \partial^{n}_v [\partial_\mu \widehat{p}_{m+1}(\mu, s, r, t , x , z)](v), \\
{\rm E_3} & := \left\{-\sum_{i=1}^d  b_i(r, x, [X^{s, \xi , (m)}_r])H^{i}_1\left(\int_r^t a(r', z, [X^{s, \xi ,(m)}_{v'}]) dr', z-x\right) \right. \\
& \quad \left. + \frac12 \sum_{i, j=1}^d \Big(a_{i, j}(r, x, [X^{s, \xi , (m)}_r]) - a_{i, j}(r, z, [X^{s, \xi , (m)}_r])\Big) \, H^{i, j}_2\left(\int_r^t a(r', z, [X^{s, \xi ,(m)}_{r'}]) dr', z-x\right) \right\}\\
& \quad \quad \times \Delta_{\mu, \mu'} \partial^{n}_v [\partial_\mu \widehat{p}_{m+1}(\mu, s, r, t , x , z)](v).    
\end{align*}

\noindent $\bullet $\textbf{ Estimate on ${\rm A}$:}

From \eqref{dec:cross:deriv:diff:mu:a} and \eqref{standard}, we directly obtain 
\begin{align*}
|{\rm A_1}|& \leq \frac{K_\beta^{+}}{(t-r)^{\frac12}} \left(\frac{W^{\beta}_2(\mu, \mu')}{(r-s)^{\frac{1+n+\beta-\eta}{2}}} + \int_{(\mathbb{R}^d)^2} (|y'-x'|^\eta \wedge 1) \, | \Delta_{\mu, \mu'} \partial^{n}_v[\partial_\mu p_{m}(\mu, s, r, x', y')](v)| \, dy' \, \mu(dx') \right) \\
& \quad \times g(c(t-r), z-x).
\end{align*}

For ${\rm A_2}$, we use the estimate \eqref{recursive:bound:deriv:a:or:b} together with \eqref{first:second:estimate:induction:decoupling:mckean} (recall that $\mathscr{C}^{n, 0}_{\infty}=\lim_{m\uparrow \infty} \mathscr{C}^{n,0}_{m} < \infty$) as well as \eqref{gaussian:bound:diff:deriv:hat:pm:different:time} if $W^2_2(\mu, \mu')\leq r-s$ or \eqref{standard} otherwise. We thus deduce
\begin{align*}
|{\rm A_2}| \leq K \frac{W^{\beta}_2(\mu, \mu')}{(t-r)^{\frac12}(r-s)^{\frac{1+n+\beta-\eta}{2}}} \, g(c(t-r), z-x)
\end{align*}
\noindent for any $\beta \in [0,1]$.

Gathering the two previous estimates, we deduce
\begin{align*}
|{\rm A}| & \leq K_\beta^{+} \left\{ \frac{W^{\beta}_2(\mu, \mu')}{(t-r)^{\frac12}(r-s)^{\frac{1+n+\beta-\eta}{2}}}  + \int_{(\mathbb{R}^d)^2} (|y'-x'|^\eta \wedge 1) \, | \Delta_{\mu, \mu'} \partial^{n}_v[\partial_\mu p_{m}(\mu, s, r, x', y')](v)| \, dy' \, \mu(dx') \right\} \\
& \quad \times g(c(t-r), z-x)
\end{align*}
\noindent for any $\beta \in [0,1]$ if $n=0$ and any $\beta \in [0,\eta)$ if $n=1$.

\noindent $\bullet $\textbf{ Estimate on ${\rm B}$:}

In order to deal with ${\rm B_1}$, we use \eqref{diff:mes:L:deriv:diff:diff:coeff:holder:reg} and the space-time inequality \eqref{space:time:inequality}. This yields
\begin{align*}
|{\rm B_1}|&  \leq  K_\beta^{+} \frac{W_2^{\beta}(\mu, \mu')}{t-r} \left\{ \frac{(|z-x|^\eta\wedge 1)}{(r-s)^{\frac{1+n+\beta}{2}}} \wedge \frac{1}{(r-s)^{\frac{1+n+\beta-\eta}{2}}} \right\} \, g(c(t-r), z-x)  \\
& \quad + \frac{K_\beta^{+}}{t-r} \left\{ ( |z-x|^\eta \wedge 1) \int_{(\mathbb{R}^d)^2}  |\Delta_{\mu, \mu'} \partial^{n}_v [\partial_\mu p_{m}(\mu, s, r, x', y')](v)|  \, dy'  \mu'(dx') \right. \notag \\
& \quad \quad \left. \wedge \, \int_{(\mathbb{R}^d)^2} (|y'-x'|^\eta \wedge 1)  |\Delta_{\mu, \mu'} \partial^{n}_v [\partial_\mu p_{m}(\mu, s, r, x', y')](v)|  \, dy'  \mu'(dx') \right\} \, g(c(t-r), z-x) \\
& \leq K_\beta^{+} W_2^{\beta}(\mu, \mu') \left\{ \frac{1}{(t-r)^{1-\frac{\eta}{2}}(r-s)^{\frac{1+n+\beta}{2}}} \wedge \frac{1}{(t-r)(r-s)^{\frac{1+n+\beta-\eta}{2}}} \right\} \, g(c(t-r), z-x)\\
& \quad +  K_\beta^{+} \left\{ \frac{1}{(t-r)^{1-\frac{\eta}{2}}} \int_{(\mathbb{R}^d)^2}  |\Delta_{\mu, \mu'} \partial^{n}_v [\partial_\mu p_{m}(\mu, s, r, x', y')](v)|  \, dy'  \mu'(dx') \right. \notag \\
& \quad \quad \left. \wedge \, \frac{1}{t-r} \int_{(\mathbb{R}^d)^2} (|y'-x'|^\eta \wedge 1)  |\Delta_{\mu, \mu'} \partial^{n}_v [\partial_\mu p_{m}(\mu, s, r, x', y')](v)|  \, dy'  \mu'(dx') \right\} \, g(c(t-r), z-x)
\end{align*}
\noindent for any $\beta \in [0,1]$ if $n=0$ and any $\beta \in [0,\eta)$ if $n=1$. 

 For ${\rm B_2}$, we use \eqref{recursive:bound:deriv:mes:holder:reg:a} with $\beta'= 1$ and $\beta'=0$ combined with \eqref{first:second:estimate:induction:decoupling:mckean} so that
 \begin{equation}
 \label{bound:diff:deriv:cross:mes:aij}
 | \partial^{n}_v [\partial_\mu[ a_{i, j} (r, x, [X^{s, \xi' , (m)}_r]) - a_{i, j}(r, z, [X^{s, \xi', (m)}_r]) ] ] (v)| \leq K \left\{ \frac{|z-x|^{\eta}}{(r-s)^{\frac{1+n}{2}}} \wedge \frac{1}{(r-s)^{\frac{1+n-\eta}{2}}} \right\}.
 \end{equation}
 
 \noindent We also use \eqref{gaussian:bound:diff:deriv:hat:pm:different:time} if $W_2(\mu, \mu') \leq (r-s)^{1/2}$ (which is thus satisfied for any $\beta \in [0,1]$) and \eqref{standard} otherwise. We thus obtain
 \begin{align*}
 |{\rm B_2}|  & \leq K \left\{ \frac{|z-x|^\eta}{(t-r) (r-s)^{\frac{1+n+\beta-\eta}{2 } }} \wedge  \frac{1}{(t-r)(r-s)^{\frac{1+n+\beta}{2}-\eta}} \right\} W^{\beta}_2(\mu, \mu') \, g(c(t-r), z-x) \\
   & \leq K \left\{ \frac{1}{(t-r)^{1-\frac{\eta}{2}}(r-s)^{\frac{1+n+\beta-\eta}{2 } }} \wedge  \frac{1}{(t-r)(r-s)^{\frac{1+n+\beta}{2}-\eta}} \right\} W^{\beta}_2(\mu, \mu') \, g(c(t-r), z-x)
 \end{align*}
 \noindent for any $\beta \in [0,1]$, where we used the space-time inequality \eqref{space:time:inequality} for the last inequality.
 
 
 Gathering the previous estimates on ${\rm B_1}$ and ${\rm B_2}$, we finally deduce
 \begin{align*}
 |{\rm B}| & \leq K_\beta^{+} W_2^{\beta}(\mu, \mu') \left\{ \frac{1}{(t-r)^{1-\frac{\eta}{2}}(r-s)^{\frac{1+n+\beta}{2}}} \wedge \frac{1}{(t-r)(r-s)^{\frac{1+n+\beta-\eta}{2}}} \right\} \, g(c(t-r), z-x)\\
& \quad +  K_\beta^{+} \left\{ \frac{1}{(t-r)^{1-\frac{\eta}{2}}} \int_{(\mathbb{R}^d)^2}  |\Delta_{\mu, \mu'} \partial^{n}_v [\partial_\mu p_{m}(\mu, s, r, x', y')](v)|  \, dy'  \mu'(dx') \right.  \\
& \quad \quad \left. \wedge \, \frac{1}{t-r} \int_{(\mathbb{R}^d)^2} (|y'-x'|^\eta \wedge 1)  |\Delta_{\mu, \mu'} \partial^{n}_v [\partial_\mu p_{m}(\mu, s, r, x', y')](v)|  \, dy'  \mu'(dx') \right\} \, g(c(t-r), z-x)
\end{align*}
\noindent for any $\beta \in [0,1]$ if $n=0$ and any $\beta \in [0,\eta)$ if $n=1$. \\

 \noindent $\bullet $\textbf{ Estimate on ${\rm C}$:}
  
 For ${\rm C_1}$, we use \eqref{diff:mes:drift:diff:coefficients} if $W_2(\mu, \mu') \leq (r-s)^{1/2}$ and the boundedness of the drift coefficient otherwise so that for any $\beta \in [0,1]$ it holds
 $$
 | \Delta_{\mu, \mu'} b_i(r, x, [X^{s, \xi , (m)}_r])| \leq K \frac{W^{\beta}_2(\mu, \mu')}{(r-s)^{\frac{\beta}{2}}}
$$
 
 \noindent and, from \eqref{recursive:bound:deriv:a:or:b} with \eqref{first:second:estimate:induction:decoupling:mckean} and the space-time inequality \eqref{space:time:inequality}
 \begin{align}
 \Big| \partial^{n}_v \Big[\partial_\mu &H^{i}_1\left(\int_r^t a(r', z, [X^{s, \xi ,(m)}_{v'}]) dr', z-y\right)\Big](v)   \widehat{p}_{m+1}(\mu, s, r, t , x , z) \Big| \nonumber\\
 & \leq K \frac{|z-x|}{(t-r)^{2}} \int_r^t \max_{i, j} | \partial^{n}_v [\partial_{\mu} [a_{i, j}(r', z, [X^{s, \xi ,(m)}_{r'}])]](v)| \, dr'  \, g(c(t-r), z-x) \label{bound:deriv:cross:mes:H1}\\
 & \leq \frac{K}{(t-r)^{\frac12} (r-s)^{\frac{1+n-\eta}{2}}}  \, g(c(t-r), z-x).\nonumber
 \end{align}
 
 \noindent Combining the two previous estimates thus yields
 $$
 |{\rm C_1}| \leq K \frac{W^{\beta}_2(\mu, \mu')}{(t-r)^{\frac12} (r-s)^{\frac{1+n+\beta-\eta}{2}}}  \, g(c(t-r), z-x). 
 $$
 
 In order to deal with ${\rm C_2}$, we use the relation \eqref{deriv:mes:cross:smooth:matrix}, the estimates \eqref{dec:cross:deriv:diff:mu:a}, \eqref{recursive:bound:deriv:a:or:b} (together with \eqref{first:second:estimate:induction:decoupling:mckean}) as well as \eqref{diff:mes:drift:diff:coefficients} if $W_2(\mu, \mu') \leq (r-s)^{1/2}$ or the boundedness of the diffusion coefficient otherwise to deduce
 \begin{align*}
 |\Delta_{\mu, \mu'} & \partial^{n}_v \left[\partial_\mu H^{i}_1\left(\int_r^t a(r', z, [X^{s, \xi ,(m)}_{r'}]) dr', z-x\right)\right](v)| \\
 &  \leq K_\beta^{+} \frac{|z-x|}{(t-r)(r-s)^{\frac{1+n+\beta-\eta}{2}}} W_2^{\beta}(\mu, \mu') \\
 & \quad+  K_\beta^{+} \frac{|z-x|}{(t-r)^2} \int_r^t \left(\frac{W^{\beta}_2(\mu, \mu')}{(r'-s)^{\frac{1+n+\beta-\eta}{2}}} + \int_{(\mathbb{R}^d)^2} (|y'-x'|^\eta \wedge 1) \, | \Delta_{\mu, \mu'} \partial^{n}_v[\partial_\mu p_{m}(\mu, s, r', x', y')](v)| \, dy' \, \mu(dx') \right) \, dr'\\
 & \leq  K_\beta^{+} \frac{|z-x|}{(t-r)(r-s)^{\frac{1+n+\beta-\eta}{2}}} W_2^{\beta}(\mu, \mu') \\
 & \quad +  K_\beta^{+} \frac{|z-x|}{(t-r)^2} \int_r^t  \int_{(\mathbb{R}^d)^2} (|y'-x'|^\eta \wedge 1) \, | \Delta_{\mu, \mu'} \partial^{n}_v[\partial_\mu p_{m}(\mu, s, r', x', y')](v)| \, dy' \, \mu(dx') \, dr'
 \end{align*}

\noindent which in turn, by the space-time inequality \eqref{space:time:inequality}, yields 
 \begin{align*}
 |{\rm C_2}| & \leq K_\beta^{+} \left(\frac{W_2^{\beta}(\mu, \mu')}{(t-r)^{\frac12}(r-s)^{\frac{1+n+\beta-\eta}{2}}}+ \frac{1}{(t-r)^{\frac32}} \int_r^t  \int_{(\mathbb{R}^d)^2} (|y'-x'|^\eta \wedge 1) \, | \Delta_{\mu, \mu'} \partial^{n}_v[\partial_\mu p_{m}(\mu, s, r', x', y')](v)| \, dy' \, \mu(dx') \, dr'\right)\\
 & \quad \times g(c(t-r), z-x).
 \end{align*}
 
 For ${\rm C_3}$, similarly to ${\rm B}_2$ and ${\rm C}_1$, we obtain 
 \begin{align*}
 |{\rm C_3}| & \leq K \frac{W^{\beta}_2(\mu, \mu')}{(t-r)^{\frac12} (r-s)^{\frac{1+n+\beta-\eta}{2}}}  \, g(c(t-r), z-x).
  \end{align*}

 Gathering the previous estimates on ${\rm C_1}$, ${\rm C_2}$ and ${\rm C_3}$, we get
 \begin{align*}
 |{\rm C}|  & \leq K_\beta^{+} \left(\frac{W_2^{\beta}(\mu, \mu')}{(t-r)^{\frac12}(r-s)^{\frac{1+n+\beta-\eta}{2}}}+ \frac{1}{(t-r)^{\frac32}} \int_r^t  \int_{(\mathbb{R}^d)^2} (|y'-x'|^\eta \wedge 1) \, | \Delta_{\mu, \mu'} \partial^{n}_v[\partial_\mu p_{m}(\mu, s, r', x', y')](v)| \, dy' \, \mu(dx') \, dr'\right)\\
 & \quad \times g(c(t-r), z-x).
 \end{align*}
 
 \noindent $\bullet $\textbf{ Estimate on ${\rm D}$:}
 
 In order to deal with ${\rm D_1}$, we first use \eqref{deriv:mes:second:order:polyn:hermite} together with \eqref{first:second:estimate:induction:decoupling:mckean} so that
 \begin{align}
\Big| \partial^{n}_v & \left[\partial_\mu H^{i, j}_2\left(\int_r^t a(v', z, [X^{s, \xi ,(m)}_{v'}]) dv', z-x\right)\right](v) \Big|\nonumber \\
 & \leq K \left( \frac{|z-x|^2}{(t-r)^2} + \frac{1}{t-r} \right) \, \frac{1}{(r-s)^{\frac{1+n-\eta}{2}}}. \label{bound:L:deriv:second:order:hermite:polynomial}
\end{align}

We then use \eqref{diff:mes:with:holder:reg:space:drift:diff:coefficients} with $\alpha = 0$ so that for any $\beta \in [0,1]$ one has
\begin{equation}
\label{delta:mes:diff:diff:coeff:reg:holder}
 \Big|  \Delta_{\mu, \mu'} \Big(a_{i, j}(r, x, [X^{s, \xi , (m)}_r]) - a_{i, j}(r, z, [X^{s, \xi , (m)}_r])\Big) \Big| \leq K  \frac{(|z-x|^\eta \wedge 1)}{(r-s)^{\frac{\beta}{2}}} W^{\beta}_2(\mu, \mu').
\end{equation}

Combining the two previous estimates and using the space-time inequality \eqref{space:time:inequality}, we eventually conclude
$$
|{\rm D}_1| \leq K  \frac{W^{\beta}_2(\mu, \mu')}{(t-r)^{1-\frac{\eta}{2}}(r-s)^{\frac{1+n+\beta-\eta}{2}}} \, g(c(t-r), z-x)
$$

\noindent for any $\beta \in [0,1]$.

%

 For ${\rm D_2}$, we handle $\Delta_{\mu, \mu'} \partial^{n}_v [\partial_\mu H^{i, j}_2(\int_r^t a(r', z, [X^{s, \xi ,(m)}_{r'}]) dr', z-x)](v)$ in a similar way as we did for $\Delta_{\mu, \mu'} \partial^{n}_v [\partial_\mu H^{i}_1(\int_r^t a(r', z, [X^{s, \xi ,(m)}_{v'}]) dr', z-x)](v)$, that is, we first use the identity \eqref{deriv:mes:cross:smooth:matrix} and then the mean-value theorem, the estimates \eqref{dec:cross:deriv:diff:mu:a}, \eqref{recursive:bound:deriv:a:or:b} (together with \eqref{first:second:estimate:induction:decoupling:mckean}) as well as \eqref{diff:mes:drift:diff:coefficients} if $W_2(\mu, \mu') \leq (r-s)^{1/2}$ and the boundedness of the diffusion coefficient otherwise. We obtain
  \begin{align*}
 |\Delta_{\mu, \mu'} & \partial^{n}_v \left[\partial_\mu H^{i, j}_2\left(\int_r^t a(r', z, [X^{s, \xi ,(m)}_{r'}]) dr', z-x\right)\right](v)| \\
 &  \leq K_\beta^{+} \left\{\frac{|z-x|^2}{(t-r)^2} + \frac{1}{t-r} \right\} \frac{1}{(r-s)^{\frac{1+n+\beta-\eta}{2}}} W_2^{\beta}(\mu, \mu') +  K_\beta^{+} \left\{\frac{|z-x|^2}{(t-r)^3} +\frac{1}{(t-r)^2} \right\} \\
 & \quad  \times \int_r^t \left(\frac{W^{\beta}_2(\mu, \mu')}{(r'-s)^{\frac{1+n+\beta-\eta}{2}}} + \int_{(\mathbb{R}^d)^2} (|y'-x'|^\eta \wedge 1) \, | \Delta_{\mu, \mu'} \partial^{n}_v[\partial_\mu p_{m}(\mu, s, r', x', y')](v)| \, dy' \, \mu(dx') \right) \, dr'\\
 \end{align*}
 \noindent which in turn, by the uniform $\eta$-H\"older regularity of $a_{i, j}(t, ., m)$ and the space-time inequality \eqref{space:time:inequality}, yields
 \begin{align*}
|{\rm D_2 }| & \leq K_\beta^{+} \left(\frac{W^{\beta}_2(\mu, \mu')}{(t-r)^{1-\frac{\eta}{2}}(r-s)^{\frac{1+n+\beta-\eta}{2}}} \right. \\
& \left. \quad + \frac{1}{(t-r)^{2-\frac{\eta}{2}}} \int_r^t  \int_{(\mathbb{R}^d)^2} (|y'-x'|^\eta \wedge 1) \, | \Delta_{\mu, \mu'} \partial^{n}_v[\partial_\mu p_{m}(\mu, s, r', x', y')](v)| \, dy' \, \mu(dx') \, dr'\right) \\
& \quad \times  g(c(t-r), z-x)
\end{align*}
\noindent for any $\beta \in [0,1]$ if $n=0$ or any $\beta \in [0,\eta)$ if $n=1$.
 
 To deal with ${\rm D_3}$, we use \eqref{gaussian:bound:diff:deriv:hat:pm:different:time} if $W_2(\mu, \mu') \leq (r-s)^{1/2}$ or \eqref{standard} otherwise, \eqref{bound:L:deriv:second:order:hermite:polynomial}, the uniform $\eta$-H\"older regularity of $a_{i, j}(t, ., m)$ and eventually the space-time inequality \eqref{space:time:inequality}. This yields
 \begin{align*}
 |{\rm D_3}|  & \leq K \frac{W^{\beta}_2(\mu, \mu')}{(t-r)^{1-\frac{\eta}{2}}(r-s)^{\frac{1+n+\beta-\eta}{2}}} \, g(c(t-r), z-x)
 \end{align*}
 \noindent for any $\beta \in [0,1]$.
 
 Gathering the previous estimates, we conclude
 \begin{align*}
 |{\rm D}| & \leq  K_\beta^{+} \left(\frac{W^{\beta}_2(\mu, \mu')}{(t-r)^{1-\frac{\eta}{2}}(r-s)^{\frac{1+n+\beta-\eta}{2}}} \right. \\
& \left. \quad + \frac{1}{(t-r)^{2-\frac{\eta}{2}}} \int_r^t  \int_{(\mathbb{R}^d)^2} (|y'-x'|^\eta \wedge 1) \, | \Delta_{\mu, \mu'} \partial^{n}_v[\partial_\mu p_{m}(\mu, s, r', x', y')](v)| \, dy' \, \mu(dx') \, dr'\right) \\
& \quad \times  g(c(t-r), z-x)
\end{align*}
\noindent for any $\beta \in [0,1]$ if $n=0$ or any $\beta \in [0,\eta)$ if $n=1$.

 \noindent $\bullet $\textbf{ Estimate on ${\rm E}$:}
  
For ${\rm E_1}$, we proceed as for the previous terms. To be more specific, we use \eqref{diff:mes:drift:diff:coefficients} if $W_2(\mu, \mu') \leq (r-s)^{1/2}$ or the boundedness of the drift and diffusion coefficients otherwise, the mean-value theorem, \eqref{cross:mes:deriv:p:hat:s:r:t} combined with \eqref{first:second:estimate:induction:decoupling:mckean} (recall that $\mathscr{C}^{n,0}_{\infty} = \lim_{m \uparrow} \mathscr{C}^{n,0}_{\infty} < \infty$). We obtain
$$
|{\rm E_1}| \leq K \frac{W^{\beta}_2(\mu, \mu')}{(t-r)^{\frac{1}{2}}(r-s)^{\frac{1+n+\beta-\eta}{2}}}\, g(c(t-r), z-x).
$$

For ${\rm E_2}$, from \eqref{delta:mes:diff:diff:coeff:reg:holder}, the mean value theorem combined with \eqref{diff:mes:drift:diff:coefficients} if $W_2(\mu, \mu') \leq (r-s)^{1/2}$ and the boundedness of the diffusion coefficient otherwise, \eqref{cross:mes:deriv:p:hat:s:r:t} combined with \eqref{first:second:estimate:induction:decoupling:mckean} as for the previous estimate and the space-time inequality \eqref{space:time:inequality}, we get
$$
|{\rm E_2}| \leq K \frac{W^{\beta}_2(\mu, \mu')}{(t-r)^{1-\frac{\eta}{2}}(r-s)^{\frac{1+n+\beta-\eta}{2}}} \,  g(c(t-r), z - x).
$$ 
  
 For the last term ${\rm E_3}$, from \eqref{diff:L:deriv:pm:mu:mup}, the uniform $\eta$-H\"older regularity of $a_{i, j}(t, ., m)$ and the space-time inequality \eqref{space:time:inequality}, we get
\begin{align*}
|{\rm E_3}| & \leq K_\beta^{+}  \left\{\frac{W^{\beta}_2(\mu, \mu')}{(t-r)^{1-\frac{\eta}{2}} (r-s)^{\frac{1+n+\beta-\eta}{2}}} \right. \\
& \quad \left. + \frac{1}{(t-r)^{2-\frac{\eta}{2}}}   \int_r^t \int_{(\mathbb{R}^d)^2} (|y'-x'|^{\eta} \wedge1) |\Delta_{\mu, \mu'} \partial^{n}_v [\partial_\mu p_{m}(\mu, s, r', x', y')] (v)| \, dy' \, \mu'(dx') \, dv' \right\} \\
& \quad \times \, g(c(t-r), z-x).
\end{align*}

 Gathering the previous estimates, we finally deduce
 \begin{align*}
|{\rm E}| & \leq K_\beta^{+}  \left\{\frac{W^{\beta}_2(\mu, \mu')}{(t-r)^{1-\frac{\eta}{2}} (r-s)^{\frac{1+n+\beta-\eta}{2}}} \right. \\
& \quad \left. + \frac{1}{(t-r)^{2-\frac{\eta}{2}}}   \int_r^t \int_{(\mathbb{R}^d)^2} (|y'-x'|^{\eta} \wedge1) |\Delta_{\mu, \mu'} \partial^{n}_v [\partial_\mu p_{m}(\mu, s, r', x', y')] (v)| \, dy' \, \mu'(dx') \, dr' \right\} \\
& \quad \times \, g(c(t-r), z-x).
\end{align*}

 We now collect all the previous estimates on ${\rm A}$, ${\rm B}$, ${\rm C}$, ${\rm D}$ and ${\rm E}$. We finally obtain the following bound
 \begin{align*}
 & | \Delta_{\mu, \mu'} \partial^{n}_v [\partial_\mu \mH_{m+1}(\mu, s, r, t, x, z)] (v)| \\
 & \leq K_\beta^{+} \left( W_2^{\beta}(\mu, \mu') \left\{ \frac{1}{(t-r)(r-s)^{\frac{1+n+\beta-\eta}{2}}} \wedge \frac{1}{(t-r)^{1- \frac{\eta}{2}}(r-s)^{\frac{1+n+\beta}{2}}} \right\} \right.\\
 & \left. \quad + \left\{ \frac{1}{(t-r)^{1-\frac{\eta}{2}}} \int_{(\mathbb{R}^d)^2}  |\Delta_{\mu, \mu'} \partial^{n}_v [\partial_\mu p_{m}(\mu, s, r, x', y')](v)|  \, dy'  \mu'(dx') \right. \right.\notag \\
& \quad \quad \left. \left. \wedge \, \frac{1}{t-r} \int_{(\mathbb{R}^d)^2} (|y'-x'|^\eta \wedge 1)  |\Delta_{\mu, \mu'} \partial^{n}_v [\partial_\mu p_{m}(\mu, s, r, x', y')](v)|  \, dy'  \mu'(dx') \right\} \right. \\
 & \left. \quad +  \frac{1}{(t-r)^{2-\frac{\eta}{2}}}   \int_r^t \int_{(\rr^d)^2} (|y'-x'|^{\eta} \wedge1) |\Delta_{\mu, \mu'} \partial^{n}_v [\partial_\mu p_{m}(\mu, s, r', x', y')] (v)| \, dy' \, \mu'(dx') \, dr' \right) \\
 & \quad \quad   g(c(t-r), z-x)
 \end{align*}
 \noindent for any $\beta \in [0,1]$ if $n=0$ and any $\beta \in [0,\eta)$ if $n=1$. The proof is now complete.

\subsection{Proof of Lemma \ref{lemme:technical:estimate:diff:time}.\\}\label{section:proof:lemme:technical:estimate:diff:time}

\noindent \emph{Step 1: proof of \eqref{diff:time:cross:deriv:diff:and:deriv:coeff}.}\\

We only prove \eqref{diff:time:cross:deriv:diff:and:deriv:coeff} for the first term namely for the difference of the $L$-derivative of the diffusion coefficient since the second term can be handled in a completely analogous way. We start from the identity \eqref{dec:cross:deriv:a} and deduce the decomposition
\begin{align*}
 &\partial^n_v[\partial_\mu  [a_{i, j }(t, x, [X^{s_1 \vee s_2, \xi , (m)}_t])]](v)  -  \partial^n_v[\partial_\mu [a_{i, j }(t, x, [X^{s_1 \wedge s_2, \xi , (m)}_t])]](v) \\
 & =  \int_{\mathbb{R}^d} \Big[ \frac{\delta a_{i, j}}{\delta m}(t, x, [X^{s_1\vee s_2,\xi, (m)}_t])(y') - \frac{\delta a_{i, j}}{\delta m}(t, x, [X^{s_1\wedge s_2, \xi, (m)}_t])(y') \Big]  \, \partial^{1+n}_x p_{m}(\mu, s_1\vee s_2, t, v, y') \, dy' \\
 &  +  \int_{\mathbb{R}^d} \Big[ \frac{\delta a_{i, j}}{\delta m}(t, x, [X^{s_1\wedge s_2,\xi, (m)}_t])(y') - \frac{\delta a_{i, j}}{\delta m}(t, x, [X^{s_1\wedge s_2, \xi, (m)}_t])(v) \Big]  \\
 & \quad \times [ \partial^{1+n}_x p_{m}(\mu, s_1\vee s_2, t, v, y') - \partial^{1+n}_x p_{m}(\mu, s_1\wedge s_2, t, v, y')] \, dy' \\
 &  +  \int_{(\rr^d)^2} \Big[\frac{\delta a_{i, j}}{\delta m}(t, x, [X^{s_1\vee s_2, \xi, (m)}_t])(y') - \frac{\delta a_{i, j}}{\delta m}(t, x,  [X^{s_1\wedge s_2, \xi, (m)}_t])(y') \Big] \, \partial^{n}_v[\partial_\mu p_{m}(\mu, s_1\vee s_2, t, x', y')](v) \, dy' \, \mu(dx') \\
 &  +  \int_{(\rr^d)^2} \Big[\frac{\delta a_{i, j}}{\delta m}(t, x, [X^{s_1 \wedge s_1,\xi, (m)}_t])(y') - \frac{\delta a_{i, j}}{\delta m}(t, x,  [X^{s_1 \wedge s_2, \xi, (m)}_t])(x') \Big] \\
 & \quad \quad \times [\partial^{n}_v[ \partial_\mu p_{m}(\mu, s_1 \vee s_2, t, x', y')](v) - \partial^{n}_v[ \partial_\mu p_{m}(\mu, s_1 \wedge s_2, t, x', y')](v) ] \, dy' \, \mu(dx')\\
 & =: {\rm A } + {\rm B} + {\rm C} + {\rm  D}.
\end{align*}

We deal with A by distinguishing the two cases $|s_1-s_2| \geq t-s_1 \vee s_2$ and $|s_1- s_2| \leq t-s_1\vee s_2$. In the first case, we simply note that
\begin{align*}
{\rm A} & = \int_{\mathbb{R}^d} \Big[ \frac{\delta a_{i, j}}{\delta m}(t, x, [X^{s_1 \vee s_2 ,\xi, (m)}_t])(y') - \frac{\delta a_{i, j}}{\delta m}(t, x, [X^{s_1 \vee s_2 , \xi, (m)}_t])(v) \Big]   \, \partial^{1+n}_x p_{m}(\mu, s_1 \vee s_2, t, v, y') \, dy' \\
& \quad -  \int_{\mathbb{R}^d} \Big[ \frac{\delta a_{i, j}}{\delta m}(t, x, [X^{s_1 \wedge s_2 , \xi, (m)}_t])(y') - \frac{\delta a_{i, j}}{\delta m}(t, x, [X^{s_1 \wedge s_2 , \xi, (m)}_t])(v) \Big]   \, \partial^{1+n}_x p_{m}(\mu, s_1 \vee s_2, t, v, y') \, dy' 
\end{align*}
\noindent and then, for both term, we combine the uniform $\eta$-H\"older regularity of the map $ [\delta a_{i, j} / \delta m](t, x, m)(.)$ with the Gaussian estimate \eqref{bound:derivative:heat:kernel} and the space-time inequality \eqref{space:time:inequality}. This yields
$$
|{\rm A}| \leq K \frac{1}{(t-s_1 \vee s_2)^{\frac{1+n-\eta}{2}}} \leq K \frac{|s_1 - s_2|^\beta}{(t-s_1 \vee s_2)^{\frac{1+n-\eta}{2}+\beta}} 
$$
\noindent for any $\beta \in [0,1]$. In the second case, similarly to \eqref{diff:time:drift:diff:coefficients}, we get
\begin{align}
\Big| \frac{\delta a_{i, j}}{\delta m}(t, x, [X^{s_1 \vee s_2,\xi, (m)}_t])(y') & - \frac{\delta a_{i, j}}{\delta m}(t, x, [X^{s_1 \wedge s_2, \xi, (m)}_t])(y') \Big| \nonumber\\
&  \leq K^{+} |s_1-s_2|^\beta \left\{\frac{1}{(t-s_1)^{\beta-\frac{\eta}{2}}} + \frac{1}{(t-s_2)^{\beta-\frac{\eta}{2}}}\right\}\label{delta:time:linear:function:deriv:diff:coeff}
\end{align}
\noindent for any $\beta \in [0, 1]$. Now, if $\beta \in [\eta/2,1]$, using \eqref{bound:derivative:heat:kernel}, we conclude
$$
|{\rm A}| \leq K^{+} \frac{|s_1-s_2|^\beta}{(t-s_1 \vee s_2)^{\frac{1+n-\eta}{2}+\beta}}. 
$$
However, since $|s_1-s_2| \leq t-s_1 \vee s_2$, the above estimate remains valid for any $\beta \in [0, 1]$.

We deal with B by  using \eqref{regularity:time:estimate:v1:v2:v3:decoupling:mckean:prop:statement} and the uniform $\eta$-H\"older regularity of the map $ [\delta a_{i, j} / \delta m](t, x, m)(.)$ as well as the space-time inequality \eqref{space:time:inequality}, we obtain
$$
|{\rm B}| \leq K_\beta \left\{\frac{|s_1-s_2|^{\beta}}{(t-s_1)^{\frac{1+n-\eta}{2}+\beta}} + \frac{|s_1-s_2|^{\beta}}{(t-s_2)^{\frac{1+n-\eta}{2}+\beta}} \right\} \leq K \frac{|s_1-s_2|^{\beta}}{(t-s_1 \vee s_2)^{\frac{1+n-\eta}{2}+\beta}}
$$
\noindent for any $\beta \in [0,(1+\eta)/2)$ if $n=0$ or any $\beta \in [0,\eta/2)$ if $n=1$. 

We deal with C by employing \eqref{first:second:estimate:induction:decoupling:mckean} (recall that $\mathscr{C}^{n, 0}_{\infty}=\lim_{m\uparrow \infty} \mathscr{C}^{n,0}_{m} < \infty$), \eqref{delta:time:linear:function:deriv:diff:coeff} in the case $|s_1-s_2| \leq t-s_1\vee s_2$ or the uniform boundedness of the map $ [\delta a_{i, j} / \delta m](t, x, m)(.)$ in the case $|s_1-s_2| \geq t-s_1 \vee s_2$. This yields
$$
|{\rm C}| \leq K \frac{|s_1-s_2|^{\beta}}{(t-s_1 \vee s_2)^{\frac{1+n+\beta-\eta}{2}}}
$$
\noindent for any $\beta \in [0,1]$ if $n=0$ or any $\beta \in [0,\eta)$ if $n=1$. 


Finally, we deal with the last term by using the uniform boundedness and $\eta$-H\"older regularity of $ [\delta a_{i, j} / \delta m](t, x, m)(.)$. We obtain
$$
|{\rm D}| \leq K  \int_{(\mathbb{R}^d)^2} (|y'-x'|^\eta \wedge 1) |\partial^{n}_v[ \partial_\mu p_{m}(\mu, s_1 \vee s_2, r, x', y')](v) - \partial^{n}_v[ \partial_\mu p_{m}(\mu, s_1 \wedge s_2, r, x', y')](v)| \, dy' \, \mu(dx').
$$
Gathering the previous estimates allows to conclude the proof of \eqref{diff:time:cross:deriv:diff:and:deriv:coeff}. \\

\noindent \emph{Step 2: proof of \eqref{diff:time:L:deriv:p:hat}.}\\

We proceed in the same way as for the proof of \eqref{diff:L:deriv:pm:mu:mup}. Namely, from the identity \eqref{representation:formula:deriv:mes:p:hat}, we obtain the following decomposition:
$$
\Delta_{s_1, s_2} \partial^{n}_v [\partial_\mu \widehat{p}_{m+1}(\mu, s, r, t, x, z)] (v) =   {\rm I} + {\rm II} + {\rm III} + {\rm IV}, 
$$

\noindent with
\begin{align*}
{\rm I} & := -\frac12 \tr\left(\Big[ \left(\int_r^t a(r', y, [X^{s_1 \vee s_2, \xi, (m)}_{r'}])\, dr'\right)^{-1} - \left(\int_r^t a(r', y, [X^{s_1 \wedge s_2, \xi, (m)}_{r'}])\, dr'\right)^{-1} \Big] \right. \\
& \quad \left. \times \int_r^t \partial^n_v[\partial_\mu [a(r', y, [X^{s, \xi, (m)}_{r'}])]](v) \, dr' \right)  \widehat{p}^{y}_{m+1}(\mu, s_1 \vee s_2, r, t, x, z), \\
{\rm II} & := -\frac12   \tr\left(\left(\int_r^t a(r', y, [X^{s_1 \wedge s_2, \xi, (m)}_{r'}])\, dr'\right)^{-1} \right. \\
& \quad \left. \times  \int_r^t [ \partial^n_v[\partial_\mu [a(r', y, [X^{s_1 \vee s_2, \xi, (m)}_{r'}])]](v) -  \partial^n_v[\partial_\mu [a(r', y, [X^{s_1 \wedge s_2, \xi, (m)}_{r'}])]](v) ]\, dr'  \right) \widehat{p}^{y}_{m+1}(\mu, s_1 \vee s_2, r, t, x, z), \\
{\rm III} & :=  \frac12  (z-x)^{t} \Big[ \left(\int_r^t a(r', y, [X^{s_1 \vee s_2, \xi, (m)}_{r'}])\, dr'\right)^{-1} - \left(\int_r^t a(r', y, [X^{s_1 \wedge s_2, \xi, (m)}_{r'}])\, dr'\right)^{-1} \Big]\\
& \quad  \times \int_r^t \partial^n_v[\partial_\mu [a(r', y, [X^{s_1 \vee s_2, \xi, (m)}_{r'}])]](v) \, dr'  \left(\int_r^t a(r', y, [X^{s_1 \vee s_2, \xi, (m)}_{r'}])\, dr'\right)^{-1} (z-x) \times \widehat{p}^{y}_{m+1}(\mu, s_1 \vee s_2, r, t, x, z) \\
& \quad +  \frac12  (z-x)^{t} \left(\int_r^t a(r', y, [X^{s_1 \wedge s_2, \xi, (m)}_{r'}])\, dr'\right)^{-1} \\
& \quad  \times \int_r^t [ \partial^n_v[\partial_\mu [a(r', y, [X^{s_1 \vee s_2, \xi, (m)}_{r'}])]](v) -  \partial^n_v[\partial_\mu [a(r', y, [X^{s_1 \wedge s_2, \xi, (m)}_{r'}])]](v) ]\, dr'  \\
&  \quad \times \left(\int_r^t a(r', y, [X^{s_1 \vee s_2, \xi, (m)}_{r'}])\, dr'\right)^{-1} (z-x) \times \widehat{p}^{y}_{m+1}(\mu, s_1 \vee s_2, r, t, x, z)  \\
& \quad + \frac12  (z-x)^{t} \left(\int_r^t a(r', y, [X^{s_1 \wedge s_2, \xi, (m)}_{r'}])\, dr'\right)^{-1} \int_r^t \partial^n_v[\partial_\mu [a(r', y, [X^{s_1 \wedge s_2, \xi, (m)}_{r'}])]](v) \, dr'  \\
&  \quad \times \Big[ \left(\int_r^t a(r', y, [X^{s_1 \vee s_2, \xi, (m)}_{r'}])\, dr'\right)^{-1} - \left(\int_r^t a(r', y, [X^{s_1 \wedge s_2, \xi, (m)}_{r'}])\, dr'\right)^{-1} \Big] (z-x) \times \widehat{p}^{y}_{m+1}(\mu, s_1 \vee s_2, r, t, x, z), \\
{\rm IV} & :=  -\frac12\left\{ \tr\left(\left(\int_r^t a(r', y, [X^{s_1 \wedge s_2, \xi, (m)}_{r'}])\, dr'\right)^{-1} \int_r^t \partial^n_v[\partial_\mu [a(r', y, [X^{s_1 \wedge s_2, \xi, (m)}_{r'}])]](v) \, dr' \right) \right. \notag\\
& \quad \left. - (z-x)^{t} \left(\int_r^t a(r', y, [X^{s_1 \wedge s_2, \xi, (m)}_{r'}])\, dr'\right)^{-1} \int_r^t \partial^n_v[\partial_\mu[ a(r', y, [X^{s_1 \wedge s_2, \xi, (m)}_{r'}])]](v) \, dr'  \right. \\ 
& \quad \left. \quad \times \left(\int_r^t a(r', y, [X^{s_1 \wedge s_2, \xi, (m)}_{r'}])\, dr'\right)^{-1} (z-x) \right\} [ \widehat{p}^{y}_{m+1}(\mu, s_1 \vee s_2, r, t, x, z) - \widehat{p}^{y}_{m+1}(\mu, s_1 \wedge s_2, r, t, x, z) ].
\end{align*}

From the mean-value theorem, the estimate \eqref{recursive:bound:deriv:a:or:b} together with \eqref{first:second:estimate:induction:decoupling:mckean} (recall that $\mathscr{C}^{n, 0}_{\infty}=\lim_{m\uparrow \infty} \mathscr{C}^{n,0}_{m} < \infty$), we obtain
$$
|{\rm I}| \leq \frac{K}{(t-r)^2} \int_r^t \max_{i, j} |a_{i, j}(r', y, [X^{s_1 \vee s_2, \xi, (m)}_{r'}]) - a_{i, j}(r', y, [X^{s_1 \wedge s_2, \xi, (m)}_{r'}])| \, dr' \int_r^t (r'-s_1 \vee s_2)^{-\frac{1+n-\eta}{2}} \, dr' \,  g(c(t-r), z-x).
$$
 Now, for the first time integral appearing in the right-hand side of the above inequality, we use \eqref{diff:time:drift:diff:coefficients} if $|s_1-s_2| \leq r'-s_1 \vee s_2$. This yields
 $$
 \max_{i, j} |a_{i, j}(r', y, [X^{s_1 \vee s_2, \xi, (m)}_{r'}]) - a_{i, j}(r', y, [X^{s_1 \wedge s_2, \xi, (m)}_{r'}])| \leq \frac{|s_1-s_2|^\beta}{(r'-s_1 \vee s_2)^{\beta- \frac{\eta}{2}}}
 $$
 \noindent for any $\beta  \in [\eta/2,1]$ and since $|s_1-s_2| \leq r'-s_1 \vee s_2$, the above estimate is actually valid for any $\beta \in [0,1]$. Otherwise, if $|s_1-s_2| \geq r'-s_1 \vee s_2$, we rather use the uniform boundedness of the diffusion coefficient. We thus obtain
 \begin{align*}
 |{\rm I}| & \leq \frac{K}{(t-r)^2} \int_r^t \frac{|s_1-s_2|^\beta}{(r'-s_1 \vee s_2)^{\beta}} \, dr' \,  \int_r^t (r'-s_1 \vee s_2)^{- \frac{1+n-\eta}{2}} \, dr' \,  g(c(t-r), z-x)\\
 & \leq K \frac{|s_1-s_2|^\beta}{(r-s_1 \vee s_2)^{\frac{1+n-\eta}{2}+ \beta}} \, g(c(t-r), z-x)
 \end{align*}
 \noindent for any $\beta \in [0,1]$.
 
 We deal with ${\rm II}$ by using \eqref{diff:time:cross:deriv:diff:and:deriv:coeff} so that
 \begin{align*}
 |{\rm II}| &  \leq K_\beta^{+} \left( \frac{|s_1-s_2|^\beta}{(r-s_1\vee s_2)^{\frac{1+n-\eta}{2}+\beta}} \right. \\
 & \left. \quad + \frac{1}{t-r} \int_r^t   \int_{(\mathbb{R}^d)^2} (|y'-x'|^\eta \wedge 1) \, | \Delta_{s_1, s_2} \partial^{n}_v[\partial_\mu p_{m}(\mu, s, r', x', y')](v)| \, dy' \, \mu(dx') \, dr' \right) \, g(c(t-r), z-x)
 \end{align*}
 \noindent for any $\beta \in [0,(1+\eta)/2)$ if $n=0$ or any $\beta \in [0,\eta/2)$ if $n=1$.
  
 We deal with ${\rm III}$ by using similar estimates as those employed to deal with ${\rm I}$ and ${\rm II}$. Omitting some technical details, we obtain
 \begin{align*}
 |{\rm III}| &  \leq K_\beta^{+}  \left( \frac{|s_1-s_2|^\beta}{(r-s_1 \vee s_2)^{\frac{1+n-\eta}{2}+\beta}} \right. \\
 & \left. \quad + \frac{1}{t-r} \int_r^t   \int_{(\mathbb{R}^d)^2} (|y'-x'|^\eta \wedge 1) \, | \Delta_{s_1, s_2} \partial^{n}_v[\partial_\mu p_{m}(\mu, s, r', x', y')](v)| \, dy' \, \mu(dx') \, dr' \right) \, g(c(t-r), z-x).
 \end{align*} 
 
 In order to deal with ${\rm IV}$, we use \eqref{gaussian:bound:diff:time:hat:pm:different:time} together with the estimate \eqref{recursive:bound:deriv:a:or:b} combined with \eqref{first:second:estimate:induction:decoupling:mckean} and finally the space-time inequality \eqref{space:time:inequality}. We thus obtain
 \begin{align*}
 |{\rm IV}| & \leq K \frac{|s_1-s_2|^\beta}{(r-s_1 \vee s_2)^{\frac{1+n-\eta}{2}+ \beta}} \, g(c(t-r), z-x)
 \end{align*}
 \noindent for any $\beta \in [0,1]$. Putting together the estimates on ${\rm I}$, ${\rm II}$, ${\rm III}$ and ${\rm IV}$ concludes the proof of \eqref{diff:time:L:deriv:p:hat}. \\

\noindent \emph{Step 3: proof of \eqref{diff:time:L:deriv:p:hat:same:time}.} \\

We first remark that if $|s_1-s_2| \geq t-s_1\vee s_2$ then the announced estimate directly follows from \eqref{cross:mes:deriv:p:hat:s:r:t} (with $r=s$) combined with \eqref{first:second:estimate:induction:decoupling:mckean} recalling that $\mathscr{C}^{n, 0}_{\infty}=\lim_{m\uparrow \infty} \mathscr{C}^{n,0}_{m} < \infty$. From now on and for the rest of the proof, we assume that $|s_1 - s_2| \leq  t-s_1\vee s_2 $. The proof being quite similar to the previous one, we will be short on some arguments and will omit some technical details. From the identity \eqref{representation:formula:deriv:mes:p:hat} with $r=s$, we deduce the following decomposition
$$
\Delta_{s_1, s_2} \partial^{n}_v [\partial_\mu \widehat{p}_{m+1}(\mu, s, t, x, z)] (v) =   {\rm I} + {\rm II} + {\rm III} + {\rm IV}, 
$$

\noindent with
\begin{align*}
{\rm I} & := -\frac12 \tr\left(\Big[ \left(\int_{s_1 \vee s_2}^t a(r', y, [X^{s_1 \vee s_2, \xi, (m)}_{r'}])\, dr'\right)^{-1} - \left(\int_{s_1 \wedge s_2}^t a(r', y, [X^{s_1 \wedge s_2, \xi, (m)}_{r'}])\, dr'\right)^{-1} \Big] \right. \\
& \quad \left. \times \int_{s_1 \vee s_2}^t \partial^n_v[\partial_\mu [a(r', y, [X^{s_1 \vee s_2, \xi, (m)}_{r'}])]](v) \, dr' \right)  \widehat{p}^{y}_{m+1}(\mu, s_1 \vee s_2, t, x, z), \\
{\rm II} & := -\frac12   \tr\left(\left(\int_{s_1 \wedge s_2}^t a(r', y, [X^{s_1 \wedge s_2, \xi, (m)}_{r'}])\, dr'\right)^{-1} \right. \\
& \quad \left. \times  \Big[ \int_{s_1 \vee s_2}^t  \partial^n_v[\partial_\mu [a(r', y, [X^{s_1 \vee s_2, \xi, (m)}_{r'}])]](v) \, dr' - \int_{s_1 \wedge s_2}^{t} \partial^n_v[\partial_\mu [a(r', y, [X^{s_1 \wedge s_2, \xi, (m)}_{r'}])]](v) \, dr' \Big] \right) \widehat{p}^{y}_{m+1}(\mu, s_1 \vee s_2, t, x, z), \\
{\rm III} & :=  \frac12  (z-x)^{t} \Big[ \left(\int_{s_1 \vee s_2}^t a(r', y, [X^{s_1 \vee s_2, \xi, (m)}_{r'}])\, dr'\right)^{-1} - \left(\int_{s_1 \wedge s_2}^t a(r', y, [X^{s_1 \wedge s_2, \xi, (m)}_{r'}])\, dr'\right)^{-1} \Big] \\
 & \quad \int_{s_1 \vee s_2}^t \partial^n_v[\partial_\mu [a(r', y, [X^{s_1 \vee s_2, \xi, (m)}_{r'}])]](v) \, dr'  \left(\int_{s_1 \vee s_2}^t a(r', y, [X^{s_1 \vee s_2, \xi, (m)}_{r'}])\, dr'\right)^{-1} (z-x) \times \widehat{p}^{y}_{m+1}(\mu, s_1 \vee s_2, t, x, z) \\
& \quad +  \frac12  (z-x)^{t} \left(\int_{s_1 \vee s_2}^t a(r', y, [X^{s_1 \wedge s_2, \xi, (m)}_{r'}])\, dr'\right)^{-1} \\
& \quad  \times \Big[ \int_{s_1 \vee s_2}^t \partial^n_v[\partial_\mu [a(r', y, [X^{s_1 \vee s_2, \xi, (m)}_{r'}])]](v) \, dr' -  \int_{s_1 \wedge s_2}^{t} \partial^n_v[\partial_\mu [a(r', y, [X^{s_1 \wedge s_2, \xi, (m)}_{r'}])]](v) \, dr' \Big]  \\
&  \quad \times \left(\int_{s_1 \vee s_2}^t a(r', y, [X^{s_1 \vee s_2, \xi, (m)}_{r'}])\, dr'\right)^{-1} (z-x) \times \widehat{p}^{y}_{m+1}(\mu, s_1 \vee s_2, t, x, z)  \\
& \quad + \frac12  (z-x)^{t} \left(\int_{s_1 \wedge s_2}^t a(r', y, [X^{s_1 \wedge s_2, \xi, (m)}_{r'}])\, dr'\right)^{-1} \int_{s_1 \wedge s_2}^t \partial^n_v[\partial_\mu [a(r', y, [X^{s_1 \wedge s_2, \xi, (m)}_{r'}])]](v) \, dr'  \\
&  \quad \times \Big[ \left(\int_{s_1 \vee s_2}^t a(r', y, [X^{s_1 \vee s_2, \xi, (m)}_{r'}])\, dr'\right)^{-1} - \left(\int_{s_1 \wedge s_2}^t a(r', y, [X^{s_1 \wedge s_2, \xi, (m)}_{r'}])\, dr'\right)^{-1} \Big] (z-x) \\
& \quad \times \widehat{p}^{y}_{m+1}(\mu, s_1 \vee s_2, t, x, z), \\
{\rm IV} & :=  -\frac12\left\{ \tr\left(\left(\int_{s_1 \wedge s_2}^t a(r', y, [X^{s_1 \wedge s_2, \xi, (m)}_{r'}])\, dr'\right)^{-1} \int_{s_1 \wedge s_2}^t \partial^n_v[\partial_\mu [a(r', y, [X^{s_1 \wedge s_2, \xi, (m)}_{r'}])]](v) \, dr' \right) \right. \notag\\
& \quad \left. - (z-x)^{t} \left(\int_{s_1 \wedge s_2}^t a(r', y, [X^{s_1 \wedge s_2, \xi, (m)}_{r'}])\, dr'\right)^{-1} \int_{s_1 \wedge s_2}^t \partial^n_v[\partial_\mu [a(r', y, [X^{s_1 \wedge s_2, \xi, (m)}_{r'}])]](v) \, dr'  \right. \\ 
& \quad \left. \quad \times \left(\int_{s_1 \wedge s_2}^t a(r', y, [X^{s_1 \wedge s_2, \xi, (m)}_{r'}])\, dr'\right)^{-1} (z-x) \right\} [ \widehat{p}^{y}_{m+1}(\mu, s_1 \vee s_2, t, x, z) - \widehat{p}^{y}_{m+1}(\mu, s_1 \wedge s_2, t, x, z) ].
\end{align*}

From \eqref{recursive:bound:deriv:a:or:b} together with \eqref{first:second:estimate:induction:decoupling:mckean}, the mean value theorem together with \eqref{diff:time:drift:diff:coefficients} if $|s_1-s_2| \leq r'-s_1 \vee s_2$ or the uniform boundedness of the diffusion coefficient otherwise, we obtain
\begin{align*}
|{\rm I}| & \leq \frac{K}{(t-s_1 \vee s_2)^2} \Big[ |s_1-s_2| + \int_{s_1 \vee s_2}^t \max_{i, j} |a_{i, j}(r', y, [X^{s_1 \vee s_2, \xi, (m)}_{r'}]) - a_{i, j}(r', y, [X^{s_1 \wedge s_2, \xi, (m)}_{r'}])| \, dr' \Big] \\
& \quad \times \int_{s_1 \vee s_2}^t (r'-s_1 \vee s_2)^{-\frac{1+n-\eta}{2}} \, dr' \,  g(c(t-s_1 \vee s_2), z-x) \\
& \leq  \frac{K}{(t-s_1 \vee s_2)^{1+ \frac{1+n-\eta}{2}}}\Big[ |s_1-s_2| + \int_{s_1 \vee s_2}^{t} \frac{|s_1-s_2|^\beta}{(r'-s_1\vee s_2)^{\beta}} \, dr'\Big] \,  g(c(t-s_1 \vee s_2), z-x) \\
& \leq K\frac{|s_1-s_2|^\beta}{(t-s_1 \vee s_2)^{\frac{1+n-\eta}{2}+\beta}}\,  g(c(t-s_1 \vee s_2), z-x) 
\end{align*}
\noindent for any $\beta \in [0,1)$.
 
 We deal with ${\rm II}$ by using \eqref{diff:time:cross:deriv:diff:and:deriv:coeff} and the estimate \eqref{recursive:bound:deriv:a:or:b} combined with \eqref{first:second:estimate:induction:decoupling:mckean}. We obtain
 \begin{align*}
 |{\rm II}| &  \leq \frac{K_\beta^{+}}{t-s_1 \vee s_2} \left( \int_{s_1 \wedge s_2}^{s_1 \vee s_2} \frac{1}{(r'-s_1 \wedge s_2)^{\frac{1+n-\eta}{2}}} \, dr' +  \int_{s_1 \vee s_2}^{t} \frac{|s_1-s_2|^\beta}{(r'-s_1\vee s_2)^{\frac{1+n-\eta}{2}+\beta}}\, dr'\right. \\
 & \left. \quad + \int_{s_1 \vee s_2}^t   \int_{(\mathbb{R}^d)^2} (|y'-x'|^\eta \wedge 1) \, | \Delta_{s_1, s_2} \partial^{n}_v[\partial_\mu p_{m}(\mu, s, r', x', y')](v)| \, dy' \, \mu(dx') \, dr' \right) \, g(c(t-s_1 \vee s_2), z-x) \\
 & \leq  K_\beta^{+}\left(  \frac{|s_1-s_2|^\beta}{(t-s_1 \vee s_2)^{\frac{1+n-\eta}{2}+\beta}}  \right. \\
 & \left. \quad + \frac{1}{t-s_1\vee s_2}\int_{s_1 \vee s_2}^t   \int_{(\rr^d)^2} (|y'-x'|^\eta \wedge 1) \, | \Delta_{s_1, s_2} \partial^{n}_v[\partial_\mu p_{m}(\mu, s, r', x', y')](v)| \, dy' \, \mu(dx') \, dr' \right) \, g(c(t-s_1 \vee s_2), z-x)
 \end{align*}
 \noindent for any $\beta \in [0,(1+\eta)/2)$ if $n=0$ or any $\beta \in [0,\eta/2)$ if $n=1$.
  
 We deal with ${\rm III}$ by using similar estimates as those employed to deal with ${\rm I}$ and ${\rm II}$. Omitting some technical details, we obtain
 \begin{align*}
 |{\rm III}|  
 & \leq  K_\beta^{+} \left( \frac{|s_1-s_2|^\beta}{(t-s_1 \vee s_2)^{\frac{1+n-\eta}{2}+\beta}} \right. \\
 & \left. \quad +  \frac{1}{t-s_1 \vee s_2}   \int_{s_1 \vee s_2}^t   \int_{(\rr^d)^2} (|y'-x'|^\eta \wedge 1) \, | \Delta_{s_1, s_2} \partial^{n}_v[\partial_\mu p_{m}(\mu, s, r', x', y')](v)| \, dy' \, \mu(dx') \, dr' \right)\\
 & \quad \times g(c(t-s_1 \vee s_2), z-x)
 \end{align*} 
 \noindent for any $\beta \in [0, (1+\eta)/2)$ if $n=0$ or any $\beta \in [0, \eta/2)$ if $n=1$.
 In order to deal with ${\rm IV}$, we use \eqref{gaussian:bound:diff:time:hat:pm:same:time} together with the estimate \eqref{recursive:bound:deriv:a:or:b} combined with \eqref{first:second:estimate:induction:decoupling:mckean} and finally the space-time inequality \eqref{space:time:inequality}. We thus obtain
 \begin{align*}
 |{\rm IV}| & \leq K \left\{\frac{1}{t-s_1 \wedge s_2} + \frac{|z-x|^2}{(t-s_1 \wedge s_2)^2} \right\} \, \int_{s_1 \wedge s_2}^t \frac{1}{(r'-s_1 \wedge s_2)^{\frac{1+n-\eta}{2}}}\, dr' \\
 & \quad \times \left\{ \frac{|s_1-s_2|^\beta}{(t-s_1 \wedge s_2)^\beta} \, g(c(t-s_1 \wedge s_2), z-x) + \frac{|s_1-s_2|^\beta}{(t-s_1 \vee s_2)^{\beta}} \, g(c(t-s_1 \vee s_2), z-x) \right\} \\
 & \leq K \left\{ \frac{|s_1-s_2|^{\beta}}{(t-s_1)^{\frac{1+n-\eta}{2}+\beta}} g(c(t-s_1), z-x) +  \frac{|s_1-s_2|^{\beta}}{(t-s_2)^{\frac{1+n-\eta}{2}+\beta}} g(c(t-s_2), z-x) \right\}
 \end{align*}
 \noindent for any $\beta \in [0,1]$, where we used the fact that $|s_1- s_2| \leq t-s_1\vee s_2$ for the last inequality. Gathering the previous estimates allows to conclude the proof of \eqref{diff:time:L:deriv:p:hat:same:time}. \\

\noindent \emph{Step 4: proof of \eqref{diff:time:L:deriv:diff:diff:coeff:holder:reg}.}\\

 We again follow similar lines of reasonings as those employed to prove the estimate \eqref{diff:mes:L:deriv:diff:diff:coeff:holder:reg}. The relation \eqref{eq:decompJ} gives the following decomposition
$$
 \Delta_{s_1, s_2} \partial^{n}_v[ \partial_\mu [a_{i, j}(t, x, [X^{s, \xi , (m)}_t]) - a_{i, j}(t, z, [X^{s, \xi , (m)}_t]) ]](v)   = {\rm I}_{i, j} + {\rm II}_{i, j} + {\rm III}_{i, j} + {\rm IV}_{i, j} + {\rm V}_{i, j},
$$

\noindent with 
\begin{align*}
{\rm I}_{i, j}   & := \int_{\mathbb{R}^d} \Big[ \frac{\delta a_{i, j}}{\delta m}(t, x, [X^{s_1 \vee s_2, \xi , (m)}_t])(z') - \frac{\delta a_{i, j}}{\delta m}(t, z, [X^{s_1 \vee s_2, \xi , (m)}_t])(z') \\
& \quad - ( \frac{\delta a_{i, j}}{\delta m}(t, x, [X^{s_1 \wedge s_2, \xi , (m)}_t])(z') - \frac{\delta a_{i, j}}{\delta m}(t, z, [X^{s_1 \wedge s_2, \xi , (m)}_t])(z') ) \Big]  \partial^{1+n}_x p_{m}(\mu, s_1 \vee s_2, t, v, z') \,dz',  \\
{\rm II}_{i, j}  & :=  \int_{\mathbb{R}^d} \Big[  \frac{\delta a_{i, j}}{\delta m}(t, x, [X^{s_1 \wedge s_2, \xi , (m)}_t])(z') - \frac{\delta a_{i, j}}{\delta m}(t, z,  [X^{s_1 \wedge s_2, \xi , (m)}_t])(z') \\
& \quad - (\frac{\delta a_{i, j}}{\delta m}(t, x, [X^{s_1 \wedge s_2, \xi , (m)}_t])(v) - \frac{\delta a_{i, j}}{\delta m}(t, z, [X^{s_1 \wedge s_2, \xi , (m)}_t])(v) ) \Big] \, \Delta_{s_1, s_2} \partial^{1+n}_x p_{m}(\mu, s, t, v, z')  \, dz', \\
{\rm III}_{i, j} & := \int_{(\mathbb{R}^d)^2}  \Big[ \frac{\delta a_{i, j}}{\delta m}(t, x, [X^{s_1 \vee s_2, \xi , (m)}_r])(z') - \frac{\delta a_{i, j}}{\delta m}(t, z,  [X^{s_1 \vee s_2, \xi , (m)}_t])(z') \\
& \quad - ( \frac{\delta a_{i, j}}{\delta m}(t, x, [X^{s_1 \wedge s_2, \xi , (m)}_t])(z') - \frac{\delta a_{i, j}}{\delta m}(t, z,  [X^{s_1 \wedge s_2, \xi , (m)}_t])(z') ) \Big]   \partial^{n}_v [\partial_\mu p_{m}(\mu, s_1 \vee s_2, t, x', z')](v) \, dz' \,  \mu(dx'),  \\
{\rm IV}_{i, j} &  := \int_{(\mathbb{R}^d)^2} \Big[  \frac{\delta a_{i, j}}{\delta m}(t, x, [X^{s_1 \wedge s_2, \xi , (m)}_t])(z') - \frac{\delta a_{i, j}}{\delta m}(t, z, [X^{s_1 \wedge s_2, \xi , (m)}_t])(z') \\
& -(\frac{\delta a_{i, j}}{\delta m}(t, x, [X^{s_1 \wedge s_2, \xi , (m)}_t])(x') - \frac{\delta a_{i, j}}{\delta m}(t, z, [X^{s_1 \wedge s_2, \xi , (m)}_t])(x') ) \Big]  \Delta_{s_1, s_2} \partial^{n}_v [\partial_\mu p_{m}(\mu, s, t, x', z')](v) \, dz'  \mu(dx').
\end{align*} 
 
We now quantify the contribution of each term in the above decomposition. We first establish a bound similar to \eqref{diff:time:with:holder:reg:space:drift:diff:coefficients} but with the map $[\delta a_{i, j}/\delta m]$ instead of $a_{i, j}$ or $b_i$. Let $\Theta^{(m)}_{\lambda, t}:= (1-\lambda)[X^{s_1 \vee s_2, \xi, (m)}_t] + \lambda [X^{s_1 \wedge s_2, \xi, (m)}_t]$, $\lambda \in [0,1]$. We write
 \begin{align*}
h(x)& := \frac{\delta a_{i, j}}{\delta m} (t, x, [X^{s_1 \vee s_2, \xi, (m)}_t])(z')  - \frac{\delta a_{i, j}}{\delta m}(t, x , [X^{s_1 \wedge s_2, \xi, (m)}_t])(z') \\
& = \int_0^1\int_{(\mathbb{R}^d)^2} \frac{\delta^2 a_{i, j}}{\delta m^2}(t, x, \Theta^{(m)}_{t, \lambda})(z', y') (p_m(\mu, s_1 \vee s_2, t, x', y') - p_{m}(\mu, s_1 \wedge s_2, t, x', y')) \, dy' \, \mu(dx') \, d\lambda.
\end{align*}

From the uniform $\eta$-H\"older regularity of $ [\delta^2 a_{i, j} /\delta m^2](t, ., m)(.)$, we get
\begin{align*}
\Big| \left[\frac{\delta^2 a_{i, j}}{\delta m}(t, x, \Theta^{(m)}_{\lambda, t})(z', y') -  \frac{\delta^2 a_{i, j}}{\delta m^2}(t, z, \Theta^{(m)}_{\lambda, t})(z', y')\right] & - \left[\frac{\delta^2 a_{i, j}}{\delta m^2}(t, x, \Theta^{(m)}_{\lambda, t})(z', x') -  \frac{\delta^2 a_{i, j}}{\delta m^2}(t, z, \Theta^{(m)}_{\lambda, t})(z', x')\right] \ \Big|  \\
& \leq K^{+} (|z-x|^\eta \wedge |y'-x'|^\eta \wedge 1)\\
& \leq K^{+} |z-x|^\alpha (|y'-x'|^{\eta-\alpha}\wedge 1)
\end{align*}
\noindent for any $\alpha \in [0,\eta]$, which combined with \eqref{diff:time:heat:kernel} and the space-time inequality \eqref{space:time:inequality} yields
\begin{align*}
| h(x) - h(z) |  \leq K^{+} |z-x|^\alpha |s_1-s_2|^\beta \left\{\frac{1}{(t-s_1)^{\beta+\frac{\alpha-\eta}{2}}} + \frac{1}{(t-s_2)^{\beta+\frac{\alpha-\eta}{2}}}\right\}.
\end{align*}

 
 Hence, taking $\alpha = \eta$ in the preceding inequality and using \eqref{bound:derivative:heat:kernel}, we derive
%
%
$$
|{\rm I}_{i, j}| \leq K^{+} (|z-x|^\eta \wedge 1) \frac{|s_1-s_2|^\beta}{(t-s_1 \vee s_2)^{\frac{1+n}{2}+\beta}}.
$$
Now, if $|s_1-s_2| \leq t-s_1\vee s_2$, we rather take $\alpha = 0 $. This yields
$$
|{\rm I}_{i, j}| \leq K^{+} \frac{|s_1-s_2|^\beta}{(t-s_1 \vee s_2)^{\frac{1+n-\eta}{2}+\beta}}.
$$
\noindent for any $\beta \geq \eta/2$. However, since $|s_1-s_2| \leq t-s_1\vee s_2$, the above estimate is actually valid for any $\beta \in [0,1]$.
Otherwise, if $|s_1-s_2| \geq t-s_1 \vee s_2 $, we instead write 
\begin{align*}
{\rm I}_{i, j}   & := \int_{\mathbb{R}^d} \Big[ \frac{\delta a_{i, j}}{\delta m}(t, x, [X^{s_1 \vee s_2, \xi , (m)}_t])(z') - \frac{\delta a_{i, j}}{\delta m}(t, x, [X^{s_1 \vee s_2, \xi , (m)}_t])(v) \Big]  \partial^{1+n}_x p_{m}(\mu, s_1 \vee s_2, t, v, z') \,dz' \\
&  \quad - \int_{\mathbb{R}^d} \Big[ \frac{\delta a_{i, j}}{\delta m}(t, z, [X^{s_1 \vee s_2, \xi , (m)}_t])(z') -  \frac{\delta a_{i, j}}{\delta m}(t, z, [X^{s_1 \vee s_2, \xi , (m)}_t])(v) \Big]  \partial^{1+n}_x p_{m}(\mu, s_1 \vee s_2, t, v, z') \,dz' \\
& \quad - \int_{\mathbb{R}^d} \Big[ \frac{\delta a_{i, j}}{\delta m}(t, x, [X^{s_1 \wedge s_2, \xi , (m)}_t])(z')  - \frac{\delta a_{i, j}}{\delta m}(t, x, [X^{s_1 \wedge s_2, \xi , (m)}_t])(v) \Big] \partial^{1+n}_x p_{m}(\mu, s_1 \vee s_2, t, v, z') \,dz' \\
& \quad + \int_{\mathbb{R}^d} \Big[ \frac{\delta a_{i, j}}{\delta m}(t, z, [X^{s_1 \wedge s_2, \xi , (m)}_t])(z') -  \frac{\delta a_{i, j}}{\delta m}(t, z, [X^{s_1 \wedge s_2, \xi , (m)}_t])(v) \Big]  \partial^{1+n}_x p_{m}(\mu, s_1 \vee s_2, t, v, z') \,dz'
\end{align*}
\noindent and use the uniform $\eta$-H\"older regularity of the map $ [\delta a_{i, j} / \delta m](t, x, m)(.)$ together with \eqref{bound:derivative:heat:kernel} and the space-time inequality \eqref{space:time:inequality}. We thus derive
$$
| {\rm I}_{i, j}| \leq K^{+} \frac{|s_1-s_2|^\beta}{(t-s_1 \vee s_2)^{\frac{1+n-\eta}{2}+\beta}}
$$

\noindent for any $\beta \in [0,1]$. Gathering the above estimates, we conclude
$$
| {\rm I}_{i, j}| \leq K^{+}  |s_1-s_2|^\beta \left\{ \frac{ (|z-x|^\eta \wedge 1)}{(t-s_1 \vee s_2)^{\frac{1+n}{2}+\beta}} \wedge \frac{1}{(t-s_1 \vee s_2)^{\frac{1+n-\eta}{2}+\beta}} \right\}
$$
\noindent for any $\beta \in [0,1]$.

In order to deal with ${\rm II}_{i, j}$, we use the estimates \eqref{regularity:time:estimate:v1:v2:v3:decoupling:mckean:prop:statement}, the uniform $\eta$-H\"older regularity of $ [\delta a_{i, j} / \delta m](t, ., m)(.)$ and finally the space-time inequality \eqref{space:time:inequality} so that
$$
| {\rm II}_{i, j} | \leq K_\beta |s_1-s_2|^\beta \left\{ \frac{ (|z-x|^\eta \wedge 1)}{(t-s_1 \vee s_2)^{\frac{1+n}{2}+\beta}} \wedge \frac{1}{(t-s_1 \vee s_2)^{\frac{1+n-\eta}{2}+\beta}} \right\}
$$
\noindent for any $\beta \in [0, (1+\eta)/2)$ if $n=0$ or any $\beta \in [0,\eta/2)$ if $n=1$. 

Following similar lines of reasonings as those employed to prove \eqref{diff:time:with:holder:reg:space:drift:diff:coefficients} but with the map $[\delta a_{i, j}/ \delta m]$ instead of $a_{i, j}$ or $b_i$, we obtain
\begin{align*}
\Big| \frac{\delta}{\delta m} a_{i, j}& (t, x, [X^{s_1 \vee s_2, \xi, (m)}_t])(y)  -  \frac{\delta}{\delta m}a_{i, j}(t, x, [X^{s_1 \wedge s_2, \xi, (m)}_t])(y) \\
& \quad - \Big( \frac{\delta}{\delta m} a_{i, j} (t, z, [X^{s_1 \vee s_2, \xi, (m)}_t])(y)  -  \frac{\delta}{\delta m}a_{i, j}(t, z, [X^{s_1 \wedge s_2, \xi, (m)}_t])(y) \Big) \Big|  \notag \\
& \quad \quad \leq   K^{+}_\alpha (|z-x|^\alpha \wedge 1) |s_1-s_2|^\beta \left\{\frac{1}{(t-s_1)^{\beta+\frac{\alpha-\eta}{2}}} + \frac{1}{(t-s_2)^{\beta+\frac{\alpha-\eta}{2}}}\right\}, 
\end{align*}

\noindent so that applying the above estimate with $\alpha= \eta$ or $\alpha =0$ and using the estimate \eqref{first:second:estimate:induction:decoupling:mckean} (recall that $\mathscr{C}^{n, 0}_{\infty}=\lim_{m\uparrow \infty} \mathscr{C}^{n,0}_{m} < \infty$) as well as the space-time inequality \eqref{space:time:inequality}, we get
$$
| {\rm III}_{i, j} | \leq K^{+}  |s_1-s_2|^\beta \left\{ \frac{ (|z-x|^\eta \wedge 1)}{(t-s_1 \vee s_2)^{\frac{1+n}{2}+\beta}} \wedge \frac{1}{(t-s_1 \vee s_2)^{\frac{1+n-\eta}{2}+\beta}} \right\}.
$$
In order to handle the last term ${\rm IV}_{i, j}$, we either use the uniform $\eta$-H\"older regularity of $ [\delta a_{i, j}/\delta m](t, ., m)](v)$ or the uniform $\eta$-H\"older regularity of $ [\delta a_{i, j}/\delta m](t, x, m)](.)$. We obtain
\begin{align*}
  | {\rm IV}_{i, j} | &  \leq K \left\{ (|z-x|^\eta\wedge 1) \int_{(\mathbb{R}^d)^2}  |\Delta_{s_1, s_2} \partial^{n}_v [\partial_\mu p_{m}(\mu, s, t, x', z')](v)|  \, dz'  \mu(dx') \right. \\
  & \quad \left. \wedge \, \int_{(\mathbb{R}^d)^2} (|z'-x'|^\eta \wedge 1)  |\Delta_{s_1,  s_2} \partial^{n}_v [\partial_\mu p_{m}(\mu, s, t, x', z')](v)|  \, dz'  \mu(dx') \right\}.
\end{align*}

Gathering the above estimates completes the proof of \eqref{diff:time:L:deriv:diff:diff:coeff:holder:reg}. \\

\noindent \emph{Step 5: proof of \eqref{diff:time:L:deriv:parametrix:kernel:pmp1}.}\\

We proceed as in the proof of \eqref{diff:mes:L:deriv:parametrix:kernel:pmp1}. Namely, from \eqref{deriv:mu:H:mp1} we obtain the following decomposition
$$
\Delta_{s_1, s_2} \partial^{n}_v [\partial_\mu \mH_{m+1}(\mu, s, r, t, x, z)](v) = {\rm A} + {\rm B} + {\rm C} + {\rm D} + {\rm E}, 
$$

\noindent with
\begin{align*}
{\rm A } & = - \sum_{i=1}^d \Delta_{s_1, s_2} \partial^{n}_v [ \partial_\mu [b_i(r, x, [X^{s, \xi , (m)}_r])]](v) \partial_{x_i} \widehat{p}_{m+1}(\mu, s_1 \vee s_2, r, t, x, z) \\
& \quad  - \sum_{i=1}^d \partial^{n}_v [ \partial_\mu [b_i(r, x, [X^{s_1 \wedge s_2, \xi , (m)}_r])]](v)  \Delta_{s_1, s_2}  \partial_{x_i} \widehat{p}_{m+1}(\mu, s, r, t, x, z) \\
& =: {\rm A}_1 + {\rm A}_2,
\end{align*}

\begin{align*}
{\rm B } & = - \sum_{i=1}^d \Delta_{s_1, s_2}  b_i(r, x, [X^{s, \xi , (m)}_r]) \partial^{n}_v [ \partial_\mu H^{i}_1\left(\int_r^t a(r', z, [X^{s_1 \vee s_2, \xi , (m)}_{r'}]) \, dr', z-x\right)](v) \, \widehat{p}_{m+1}(\mu, s_1 \vee s_2, r, t, x, z) \\
& \quad  - \sum_{i=1}^d b_i(r, x, [X^{s_1 \wedge s_2, \xi , (m)}_r])  \Delta_{s_1, s_2} \partial^{n}_v [ \partial_\mu H^{i}_1\left(\int_r^t a(r', z, [X^{s, \xi , (m)}_{r'}]) \, dr', z-x\right)](v) \, \widehat{p}_{m+1}(\mu, s_1 \vee s_2, r, t, x, z) \\
& \quad  - \sum_{i=1}^d b_i(r, x, [X^{s_1 \wedge s_2, \xi , (m)}_r])   \partial^{n}_v [ \partial_\mu H^{i}_1\left(\int_r^t a(r', z, [X^{s_1 \wedge s_2, \xi , (m)}_{r'}]) \, dr', z-x\right)](v) \, \Delta_{s_1, s_2} \widehat{p}_{m+1}(\mu, s, r, t, x, z) \\
& =: {\rm B}_1 + {\rm B}_2 + {\rm B}_3,
\end{align*}

\begin{align*}
{\rm C } & = \frac12 \sum_{i,j=1}^d \Delta_{s_1, s_2}  \partial^{n}_v [ \partial_\mu [a_{i, j}(r, x, [X^{s, \xi , (m)}_r]) - a_{i, j}(r, z, [X^{s, \xi, (m)}_r])] ](v)  \partial^2_{x_i, x_j} \widehat{p}_{m+1}(\mu, s_1 \vee s_2, r, t, x, z) \\
&  + \frac12 \sum_{i, j=1}^d  \partial^{n}_v [ \partial_\mu [a_{i, j}(r, x, [X^{s_1 \wedge s_2, \xi , (m)}_r]) - a_{i, j}(r, z, [X^{s_1 \wedge s_2, \xi, (m)}_r])]](v)   \Delta_{s_1, s_2}  \partial^2_{x_i, x_j}\widehat{p}_{m+1}(\mu, s, r, t, x, z) \\
& =: {\rm C}_1 + {\rm C}_2,
\end{align*}

\begin{align*}
{\rm D } & = \frac12 \sum_{i, j=1}^d \Delta_{s_1,s_2}   [a_{i, j}(r, x, [X^{s, \xi , (m)}_r]) - a_{i, j}(r, z, [X^{s, \xi, (m)}_r])]  \partial^{n}_v [ \partial_\mu H^{i, j}_2\left(\int_r^t a(r', z, [X^{s_1 \vee s_2, \xi , (m)}_{r'}]) \, dr', z-x\right)](v) \\
& \quad \times \widehat{p}_{m+1}(\mu, s_1 \vee s_2, r, t, x, z) \\
&  + \frac12 \sum_{i, j=1}^d   [a_{i, j}(r, x, [X^{s_1 \wedge s_2, \xi , (m)}_r]) - a_{i, j}(r, z, [X^{s_1 \wedge s_2, \xi, (m)}_r])]   \Delta_{s_1, s_2} \partial^{n}_v [ \partial_\mu H^{i, j}_2\left(\int_r^t a(r', z, [X^{s, \xi , (m)}_{r'}]) \, dr', z-x\right)](v) \\
& \quad \quad \times \widehat{p}_{m+1}(\mu, s_1 \vee s_2, r, t, x, z) \\
&  + \frac12 \sum_{i, j=1}^d   [a_{i, j}(r, x, [X^{s_1 \wedge s_2, \xi , (m)}_r]) - a_{i, j}(r, z, [X^{s_1 \wedge s_2, \xi, (m)}_r])]   \partial^{n}_v [ \partial_\mu H^{i, j}_2\left(\int_r^t a(r', z, [X^{s_1 \wedge s_2, \xi , (m)}_{r'}]) \, dr', z-x\right)](v) \\
& \quad \quad \times \Delta_{s_1,s_2} \widehat{p}_{m+1}(\mu, s, r, t, x, z) \\
& =: {\rm D}_1 + {\rm D}_2 + {\rm D}_3,
\end{align*}

\noindent and finally
\begin{align*}
{\rm E } & =  - \sum_{i=1}^d \Delta_{s_1,s_2}  \Big[b_i(r, x, [X^{s, \xi , (m)}_r])  H^{i}_1\left(\int_r^t a(r', z, [X^{s, \xi , (m)}_{r'}]) \, dr', z-x\right)\Big] \partial^{n}_v [ \partial_\mu \widehat{p}_{m+1}(\mu, s_1 \vee s_2, r, t, x, z)](v)\\
& + \frac12 \sum_{i, j=1}^d \Delta_{s_1, s_2}   \Big[a_{i, j}(r, x, [X^{s, \xi , (m)}_r]) - a_{i, j}(r, z, [X^{s, \xi, (m)}_r])   H^{i, j}_2\left(\int_r^t a(r', z, [X^{s, \xi , (m)}_{r'}]) \, dr', z-x\right)\Big] \\
& \quad \times \partial^{n}_v [ \partial_\mu \widehat{p}_{m+1}(\mu, s_1 \vee s_2, r, t, x, z)](v) \\
& \quad + \left\{ - \sum_{i=1}^d   b_i(r, x, [X^{s_1\wedge s_2, \xi , (m)}_r]) H^{i}_1\left(\int_r^t a(r', z, [X^{s_1 \wedge s_2, \xi , (m)}_{r'}]) \, dr', z-x\right) \right. \\
& \quad \left. + \frac12 \sum_{i, j=1}^d   [a_{i, j}(r, x, [X^{s_1 \wedge s_2, \xi , (m)}_r]) - a_{i, j}(r, z, [X^{s_1 \wedge s_2, \xi, (m)}_r])]   H^{i, j}_2\left(\int_r^t a(r', z, [X^{s_1 \wedge s_2, \xi , (m)}_{r'}]) \, dr', z-x\right) \right\}\\
& \quad \quad \times \Delta_{s_1, s_2} \partial^{n}_v [ \partial_\mu \widehat{p}_{m+1}(\mu, s, r, t, x, z)](v) \\
& =: {\rm E}_1 + {\rm E}_2 + {\rm E}_3.
\end{align*}

\noindent $\bullet $\textbf{ Estimate on ${\rm A}$:} \\

From \eqref{diff:time:cross:deriv:diff:and:deriv:coeff} and the space-time inequality \eqref{space:time:inequality}, we obtain
\begin{align*}
| {\rm A}_1 | & \leq K_\beta^{+} \left(\frac{|s_1-s_2|^\beta}{(t-r)^{\frac12} (r-s_1\vee s_2)^{\frac{1+n-\eta}{2}+\beta}} + \frac{1}{(t-r)^{\frac12}} \int_{(\rr^d)^2} (|y'-x'|^\eta \wedge 1) \, | \Delta_{s_1, s_2} \partial^{n}_v[\partial_\mu p_{m}(\mu, s, r, x', y')](v)| \, dy' \, \mu(dx') \right) \\
& \quad \times g(c(t-r), z-x).
\end{align*}

From \eqref{recursive:bound:deriv:a:or:b} combined with \eqref{first:second:estimate:induction:decoupling:mckean} (recall that $\mathscr{C}^{n, 0}_m \leq \mathscr{C}^{n,0}_\infty<\infty)$, \eqref{gaussian:bound:diff:time:hat:pm:different:time} and the space-time inequality \eqref{space:time:inequality}, we get
\begin{align*}
| {\rm A}_2 | & \leq K \frac{|s_1-s_2|^\beta}{(t-r)^{\frac12}(r-s_1\vee s_2)^{\frac{1+n-\eta}{2}+\beta}} \, g(c(t-r), z-x).
\end{align*}


Gathering the two previous estimates yields
\begin{align*}
| {\rm A} | & \leq K_\beta^{+} \left(\frac{|s_1-s_2|^\beta}{(t-r)^{\frac12} (r-s_1\vee s_2)^{\frac{1+n-\eta}{2}+\beta}} + \frac{1}{(t-r)^{\frac12}} \int_{(\rr^d)^2} (|y'-x'|^\eta \wedge 1) \, | \Delta_{s_1, s_2} \partial^{n}_v[\partial_\mu p_{m}(\mu, s, r, x', y')](v)| \, dy' \, \mu(dx') \right) \\
& \quad \times g(c(t-r), z-x).
\end{align*}

\noindent $\bullet $\textbf{ Estimate on ${\rm B}$:} \\

From \eqref{diff:time:drift:diff:coefficients}, the estimate \eqref{bound:deriv:cross:mes:H1} and the space-time inequality \eqref{space:time:inequality}, we obtain
\begin{align*}
| {\rm B}_1| \leq K \frac{|s_1-s_2|^\beta}{(t-r)^{\frac12} (r-s_1\vee s_2)^{\frac{1+n-\eta}{2}+\beta}} \, g(c(t-r), z-x).
\end{align*}

We first use the identity \eqref{deriv:mes:cross:smooth:matrix} and then the estimates \eqref{diff:time:drift:diff:coefficients}, \eqref{diff:time:cross:deriv:diff:and:deriv:coeff}, \eqref{recursive:bound:deriv:a:or:b} (together with \eqref{first:second:estimate:induction:decoupling:mckean}) so that
\begin{align*}
 | \Delta_{s_1, s_2} & \partial^{n}_v [ \partial_\mu H^{i}_1\left(\int_r^t a(r', z, [X^{s, \xi , (m)}_{r'}]) \, dr', z-x\right)](v) | \\
  &  \leq K  \left\{ |s_1-s_2|^\beta \frac{ |z-x|}{(t-r)(r-s_1\vee s_2)^{\frac{1+n-\eta}{2}+\beta}} + \frac{|z-x|}{(t-r)^{2}} \int_r^t \max_{i, j} | \Delta_{s_1 ,s_2}  \partial^{n}_{v} [\partial_\mu [a_{i, j }(r', z, [X^{s, \xi , (m)}_{r'}])]](v) |  \, dr'   \right\} \\
  & \leq K_\beta^{+} \left\{ |s_1-s_2|^\beta \frac{ |z-x|}{(t-r)(r-s_1\vee s_2)^{\frac{1+n-\eta}{2}+\beta}} \right. \\
   & \quad \left. + \frac{|z-x|}{(t-r)^{2}} \int_r^t \int_{(\rr^d)^2} (|y'-x'|^\eta \wedge 1) \, | \Delta_{s_1, s_2} \partial^{n}_v[\partial_\mu p_{m}(\mu, s, r', x', y')](v)| \, dy' \, \mu(dx') \, dr'   \right\} 
\end{align*}
\noindent which in turn, by the space-time inequality \eqref{space:time:inequality}, yields
\begin{align*}
| {\rm B}_2 | & \leq K_\beta^{+} \left\{  \frac{|s_1-s_2|^\beta}{(t-r)^{\frac12}(r-s_1\vee s_2)^{\frac{1+n-\eta}{2}+\beta}} \right. \\
   & \quad \left. + \frac{1}{(t-r)^{\frac32}} \int_r^t \int_{(\rr^d)^2} (|y'-x'|^\eta \wedge 1) \, | \Delta_{s_1, s_2} \partial^{n}_v[\partial_\mu p_{m}(\mu, s, r', x', y')](v)| \, dy' \, \mu(dx') \, dr'   \right\} \, g(c(t-r), z-x).
\end{align*}

From \eqref{gaussian:bound:diff:time:hat:pm:different:time}, the estimate \eqref{recursive:bound:deriv:a:or:b} combined again with \eqref{first:second:estimate:induction:decoupling:mckean} and the space-time inequality \eqref{space:time:inequality}, we get
\begin{align*}
| {\rm B}_3 | \leq K \frac{|s_1-s_2|^\beta}{(t-r)^{\frac12}(r-s_1 \vee s_2)^{\frac{1+n-\eta}{2}+\beta}}  \, g(c(t-r), z-x).
\end{align*}

Gathering the three previous estimates yields
\begin{align*}
| {\rm B} | & \leq K_\beta^{+} \left\{  \frac{|s_1-s_2|^\beta}{(t-r)^{\frac12}(r-s_1\vee s_2)^{\frac{1+n-\eta}{2}+\beta}} \right. \\
   & \quad \left. + \frac{1}{(t-r)^{\frac32}} \int_r^t \int_{(\rr^d)^2} (|y'-x'|^\eta \wedge 1) \, | \Delta_{s_1, s_2} \partial^{n}_v[\partial_\mu p_{m}(\mu, s, r', x', y')](v)| \, dy' \, \mu(dx') \, dr'   \right\} \, g(c(t-r), z-x).
\end{align*}

\noindent $\bullet $\textbf{ Estimate on ${\rm C}$:} \\

From \eqref{diff:time:L:deriv:diff:diff:coeff:holder:reg} and the space-time inequality \eqref{space:time:inequality}, we obtain
\begin{align*}
|{\rm C}_1| & \leq  K_\beta^{+} |s_1-s_2|^\beta \left\{ \frac{1}{(t-r)^{1-\frac{\eta}{2}}(r-s_1 \vee s_2)^{\frac{1+n}{2}+\beta}} \wedge \frac{1}{(t-r)(r-s_1 \vee s_2)^{\frac{1+n-\eta}{2} + \beta }} \right\} \\
& \quad + K_\beta^{+} \left\{ \frac{1}{(t-r)^{1-\frac{\eta}{2}}} \int_{(\mathbb{R}^d)^2} |\Delta_{s_1, s_2}\partial^{n}_v[\partial_\mu p_{m}(\mu, s, r, x', y')](v) |   \, dy' \mu(dx')  \right.  \\
& \quad \quad \left. \wedge \frac{1}{t-r}  \int_{(\mathbb{R}^d)^2} (|y'-x'|^\eta \wedge 1)  |\Delta_{s_1, s_2}\partial^{n}_v[\partial_\mu p_{m}(\mu, s, r, x', y')](v) |   \, dy' \mu(dx') \right\}\, g(c(t-r), z-x). 
\end{align*} 

The estimate \eqref{recursive:bound:deriv:mes:holder:reg:a} with $\beta' =1$ or $\beta' = 0$ together with \eqref{first:second:estimate:induction:decoupling:mckean} implies
$$
| \partial^{n}_v [ \partial_\mu [a_{i, j}(r, x, [X^{s_1 \wedge s_2, \xi , (m)}_r]) - a_{i, j}(r, z, [X^{s_1 \wedge s_2, \xi, (m)}_r])]](v) | \leq K \left\{ \frac{|z-x|^\eta}{(r-s_1 \vee s_2)^{\frac{1+n}{2}}} \wedge \frac{1}{(r-s_1 \vee s_2)^{\frac{1+n-\eta}{2}}} \right\}
$$ 
\noindent which combined with \eqref{gaussian:bound:diff:time:hat:pm:different:time} and the space-time inequality \eqref{space:time:inequality} yields
\begin{align*}
| {\rm C}_2 | & \leq K |s_1-s_2|^\beta \left\{ \frac{1}{(t-r)^{1-\frac{\eta}{2}}(r-s_1 \vee s_2)^{\frac{1+n}{2}+\beta}} \wedge \frac{1}{(t-r)(r-s_1 \vee s_2)^{\frac{1+n-\eta}{2}+\beta}} \right\}\, g(c(t-r), z-x).
\end{align*}

Hence,
\begin{align*}
|{\rm C}| & \leq  K_\beta^{+} |s_1-s_2|^\beta \left\{ \frac{1}{(t-r)^{1-\frac{\eta}{2}}(r-s_1 \vee s_2)^{\frac{1+n}{2}+\beta}} \wedge \frac{1}{(t-r)(r-s_1 \vee s_2)^{\frac{1+n-\eta}{2} + \beta }} \right\} \\
& \quad + K_\beta^{+} \left\{ \frac{1}{(t-r)^{1-\frac{\eta}{2}}} \int_{(\mathbb{R}^d)^2} |\Delta_{s_1, s_2}\partial^{n}_v[\partial_\mu p_{m}(\mu, s, r, x', y')](v) |   \, dy' \mu(dx')  \right.  \\
& \quad \quad \left. \wedge \frac{1}{t-r}  \int_{(\mathbb{R}^d)^2} (|y'-x'|^\eta \wedge 1)  |\Delta_{s_1, s_2}\partial^{n}_v[\partial_\mu p_{m}(\mu, s, r, x', y')](v) |   \, dy' \mu(dx') \right\}\, g(c(t-r), z-x). 
\end{align*}

\noindent $\bullet $\textbf{ Estimate on ${\rm D}$:} \\

We use \eqref{diff:time:with:holder:reg:space:drift:diff:coefficients} with $\alpha=\eta$, \eqref{deriv:mes:second:order:polyn:hermite} as well as the estimate \eqref{recursive:bound:deriv:a:or:b} combined again with \eqref{first:second:estimate:induction:decoupling:mckean} and finally the space-time inequality \eqref{space:time:inequality}
\begin{align*}
|{\rm D}_1 | & \leq K \frac{ |s_1-s_2|^\beta}{(t-r)^{1-\frac{\eta}{2}}(r-s_1 \vee s_2)^{\frac{1+n-\eta}{2}+\beta}} \, g(c(t-r), z-x).
\end{align*}

From the identity \eqref{deriv:mes:cross:smooth:matrix} and the estimates \eqref{diff:time:drift:diff:coefficients}, \eqref{recursive:bound:deriv:a:or:b} (combined again with \eqref{first:second:estimate:induction:decoupling:mckean}), we get
\begin{align*}
 \Big| \Delta_{s_1, s_2} & \partial^{n}_v [ \partial_\mu H^{i, j}_2\left(\int_r^t a(r', z, [X^{s, \xi , (m)}_{r'}]) \, dr', z-x\right)](v) \Big| \\
  &  \leq K |s_1-s_2|^\beta \left\{ \frac{1}{(t-r)(r-s_1\vee s_2)^{\frac{1+n-\eta}{2}+\beta}} + \frac{ |z-x|^2}{(t-r)^2(r-s_1\vee s_2)^{\frac{1+n-\eta}{2}+\beta}}  \right\} \\
  & \quad  + K \left\{  \frac{1}{(t-r)^2} + \frac{|z-x|^2}{(t-r)^{3}}\right\} \int_r^t \max_{i, j} | \Delta_{s_1 ,s_2}  \partial^{n}_{v} [\partial_\mu [a_{i, j }(r', z, [X^{s, \xi , (m)}_{r'}])]](v) |  \, dr'  
\end{align*}

\noindent which combined with \eqref{diff:time:cross:deriv:diff:and:deriv:coeff}, the uniform $\eta$-H\"older regularity of $ a(t, ., m)$ and then the space-time inequality \eqref{space:time:inequality} yields

\begin{align*}
|{\rm D}_2| & \leq K_\beta^{+} \left\{  \frac{|s_1-s_2|^\beta}{(t-r)^{1-\frac{\eta}{2}}(r-s_1\vee s_2)^{\frac{1+n-\eta}{2}+\beta}} \right. \\
& \quad \left.  +  \frac{1}{(t-r)^{2-\frac{\eta}{2}}} \int_r^t   \int_{(\mathbb{R}^d)^2} (|y'-x'|^\eta \wedge 1)  |\Delta_{s_1, s_2}\partial^{n}_v[\partial_\mu p_{m}(\mu, s, r', x', y')](v) |   \, dy' \mu(dx') \, dr'  \right\} \\
& \quad \times  \, g(c(t-r), z-x).
\end{align*}

From  \eqref{deriv:mes:second:order:polyn:hermite} with \eqref{first:second:estimate:induction:decoupling:mckean}, the uniform $\eta$-H\"older regularity of $ a(t, ., m)$, \eqref{gaussian:bound:diff:time:hat:pm:different:time} and the space-time inequality \eqref{space:time:inequality}, we get
\begin{align*}
|{\rm D}_3| & \leq K \frac{ |s_1-s_2|^\beta}{(t-r)^{1-\frac{\eta}{2}}(r-s_1 \vee s_2)^{\frac{1+n-\eta}{2}+\beta}} \, g(c(t-r), z-x).
\end{align*}

We thus conclude
\begin{align*}
|{\rm D}| & \leq K_\beta^{+} \left\{  \frac{|s_1-s_2|^\beta}{(t-r)^{1-\frac{\eta}{2}}(r-s_1\vee s_2)^{\frac{1+n-\eta}{2}+\beta}} \right. \\
& \quad \left. +  \frac{1}{(t-r)^{2-\frac{\eta}{2}}} \int_r^t   \int_{(\mathbb{R}^d)^2} (|y'-x'|^\eta \wedge 1)  |\Delta_{s_1, s_2}\partial^{n}_v[\partial_\mu p_{m}(\mu, s, r', x', y')](v) |   \, dy' \mu(dx') \, dr'  \right\} \\
& \quad \times  \, g(c(t-r), z-x).
\end{align*}

\noindent $\bullet $\textbf{ Estimate on ${\rm E}$:} \\

From \eqref{diff:time:drift:diff:coefficients}, \eqref{cross:mes:deriv:p:hat:s:r:t} (combined again with \eqref{first:second:estimate:induction:decoupling:mckean}) and the space-time inequality \eqref{space:time:inequality}, we obtain
\begin{align*}
| {\rm E}_1 | \leq K \frac{|s_1-s_2|^\beta}{(t-r)^{\frac12}(r-s_1\vee s_2)^{\frac{1+n-\eta}{2}+\beta}} \, g(c(t-r), z-x).
\end{align*}

Using \eqref{diff:time:with:holder:reg:space:drift:diff:coefficients} with $\alpha = \eta$, the mean-value theorem with \eqref{diff:time:drift:diff:coefficients}, the uniform $\eta$-H\"older regularity of $a(t, ., m)$, again \eqref{cross:mes:deriv:p:hat:s:r:t} with \eqref{first:second:estimate:induction:decoupling:mckean} and the space-time inequality \eqref{space:time:inequality}, we obtain
\begin{align*}
|{\rm E}_2|&  \leq K \frac{|s_1-s_2|^\beta}{(t-r)^{1-\frac{\eta}{2}}(r-s_1\vee s_2)^{\frac{1+n-\eta}{2}+\beta}} \, g(c(t-r), z-x).
\end{align*}

From \eqref{diff:time:L:deriv:p:hat}, the uniform boundedness of $b_i$ and the uniform $\eta$-H\"older regularity of $ a(t, ., m)$ with the space-time inequality \eqref{space:time:inequality}, we get
\begin{align*}
|{\rm E}_3| & \leq K_\beta^{+} \left\{ \frac{|s_1-s_2|^\beta}{(t-r)^{1- \frac{\eta}{2}}(r-s_1 \vee s_2)^{\frac{1+n-\eta}{2} + \beta}} \right. \\
& \quad \left.  + \frac{1}{(t-r)^{2-\frac{\eta}{2}}} \int_r^t   \int_{(\mathbb{R}^d)^2} (|y'-x'|^\eta \wedge 1)  |\Delta_{s_1, s_2}\partial^{n}_v[\partial_\mu p_{m}(\mu, s, r', x', y')](v) |   \, dy' \mu(dx') \, dr' \right\}  \\
& \quad \quad \times g(c(t-r), z-x).
\end{align*}

Gathering the three previous estimates, we deduce
\begin{align*}
|{\rm E}| & \leq K_\beta^{+} \left\{ \frac{|s_1-s_2|^\beta}{(t-r)^{1- \frac{\eta}{2}}(r-s_1 \vee s_2)^{\frac{1+n-\eta}{2} + \beta}} \right. \\
& \quad \left.  + \frac{1}{(t-r)^{2-\frac{\eta}{2}}} \int_r^t   \int_{(\mathbb{R}^d)^2} (|y'-x'|^\eta \wedge 1)  |\Delta_{s_1, s_2}\partial^{n}_v[\partial_\mu p_{m}(\mu, s, r', x', y')](v) |   \, dy' \mu(dx') \, dr' \right\}  \\
& \quad \quad \times g(c(t-r), z-x).
\end{align*}

We conclude the proof of \eqref{diff:time:L:deriv:parametrix:kernel:pmp1} by collecting the estimates on ${\rm A}$, ${\rm B}$, ${\rm C}$, ${\rm D}$ and ${\rm E}$.

\section*{Acknowledgments.}
We would like to warmly thank Fran\c{c}ois Delarue for his careful reading of the first version of this work, for pointing out an important issue in the proof of Theorem \ref{thm:martingale:problem} and for making several important comments which led to the improvement of the current manuscript. For the first author, this work has been partially supported by the ANR project ANR-15-IDEX-02.

\bibliographystyle{alpha}
\bibliography{bibli}

\def\cprime{$'$}
\begin{thebibliography}{BLPR17}

\bibitem[Ba{\~n}18]{banos2018}
D.~Ba{\~n}os.
\newblock The {B}ismut--{E}lworthy--{L}i formula for mean-field stochastic
  differential equations.
\newblock {\em Annales de l'Institut Henri Poincar{\'e}, Probabilit{\'e}s et
  Statistiques}, 54(1):220--233, 02 2018.

\bibitem[BFY17]{BENSOUSSAN20172093}
A.~Bensoussan, J.~Frehse, and S.C.P. Yam.
\newblock On the interpretation of the master equation.
\newblock {\em Stochastic Processes and their Applications}, 127(7):2093 --
  2137, 2017.

\bibitem[BLPR17]{buckdahn2017}
R.~Buckdahn, J.~Li, S.~Peng, and C.~Rainer.
\newblock Mean-field stochastic differential equations and associated pdes.
\newblock {\em Ann. Probab.}, 45(2):824--878, 03 2017.

\bibitem[Car13]{cardaliaguet}
P.~Cardaliaguet.
\newblock Notes on mean field games.
\newblock \url{https://www.ceremade.dauphine.fr/ cardaliaguet/MFG20130420.pdf},
  2013.

\bibitem[CCD14]{chassagneux:crisan:delarue}
J.-F. Chassagneux, D.~Crisan, and F.~Delarue.
\newblock A probabilistic approach to classical solutions of the master
  equation for large population equilibria.
\newblock {\em To appear in Memoirs of the AMS}, 2014.

\bibitem[CD18]{carmona2018probabilistic}
R.~Carmona and F.~Delarue.
\newblock {\em Probabilistic Theory of Mean Field Games with Applications I:
  Mean Field FBSDEs, Control, and Games}.
\newblock Probability Theory and Stochastic Modelling. Springer International
  Publishing, 2018.

\bibitem[CDLL19]{cardaliaguet:delarue:lasry:lions}
P.~Cardaliaguet, F.~Delarue, J.M. Lasry, and P.L. Lions.
\newblock {\em The Master Equation and the Convergence Problem in Mean Field
  Games: (AMS-201)}.
\newblock Annals of Mathematics Studies. Princeton University Press, 2019.

\bibitem[CdR19]{CHAUDRUDERAYNAL2019}
P.-E. Chaudru~de Raynal.
\newblock Strong well posedness of {M}c{K}ean--{V}lasov stochastic differential
  equations with {H}{\"o}lder drift.
\newblock {\em Stochastic Processes and their Applications}, 2019.

\bibitem[CM17]{Crisan2017}
D.~Crisan and E.~McMurray.
\newblock Smoothing properties of {M}c{K}ean--{V}lasov {S}{D}{E}s.
\newblock {\em Probability Theory and Related Fields}, Apr 2017.

\bibitem[DM10]{dela:meno:10}
F.~Delarue and S.~Menozzi.
\newblock Density estimates for a random noise propagating through a chain of
  differential equations.
\newblock {\em Journal of Functional Analysis}, 259(6):1577--1630, 2010.

\bibitem[FL17]{frikha:li}
N.~Frikha and L.~Li.
\newblock {W}eak uniqueness and density estimates for sdes with coefficients
  depending on some path-functionals.
\newblock {\em forthcoming for Annales de l'IHP Probability and Statistics.},
  2017.

\bibitem[Fri64]{friedman:64}
A.~Friedman.
\newblock {\em Partial differential equations of parabolic type}.
\newblock Prentice-Hall, 1964.

\bibitem[Fri11]{Friedman2011}
A.~Friedman.
\newblock {\em Stochastic Differential Equations and Applications}, pages
  75--148.
\newblock Springer Berlin Heidelberg, Berlin, Heidelberg, 2011.

\bibitem[Fri17]{hdr:report:noufel}
N.~Frikha.
\newblock Stochastic approximation, markovian perturbation of stochastic
  processes and their applications.
\newblock Hdr report, Universit\'e Paris Diderot, 2017.

\bibitem[Fun84]{Funaki1984}
T.~Funaki.
\newblock A certain class of diffusion processes associated with nonlinear
  parabolic equations.
\newblock {\em Zeitschrift f{\"u}r Wahrscheinlichkeitstheorie und Verwandte
  Gebiete}, 67(3):331--348, Oct 1984.

\bibitem[G{\"a}r88]{gartner}
J.~G{\"a}rtner.
\newblock On the {M}c{K}ean-{V}lasov {L}imit for {I}nteracting {D}iffusions.
\newblock {\em Mathematische Nachrichten}, 137(1):197--248, 1988.

\bibitem[GM92]{1992green}
M.G. Garroni and J.L. Menaldi.
\newblock {\em Green Functions for Second Order Parabolic Integro-Differential
  Problems}.
\newblock Chapman \& Hall/CRC Research Notes in Mathematics Series. Taylor \&
  Francis, 1992.

\bibitem[HHJ15]{hairer2015}
M.~Hairer, M.~Hutzenthaler, and A.~Jentzen.
\newblock Loss of regularity for {K}olmogorov equations.
\newblock {\em Ann. Probab.}, 43(2):468--527, 03 2015.

\bibitem[H{\"o}r67]{horm:67}
L.~H{\"o}rmander.
\newblock Hypoelliptic second order differential operators.
\newblock {\em Acta. Math.}, 119:147--171, 1967.

\bibitem[HvS18]{HSSzpruch:18}
W.~Hammersley, D.~\v{S}i\v{s}ka, and L.~Szpruch.
\newblock {M}c{K}ean-{V}lasov {S}{D}{E}s under {M}easure {D}ependent {L}yapunov
  {C}onditions.
\newblock Preprint Arxiv arXiv:1802.03974, 2018.

\bibitem[Jou97]{jourdain:1997}
B.~Jourdain.
\newblock Diffusions with a nonlinear irregular drift coefficient and
  probabilistic interpretation of generalized burgers' equations.
\newblock {\em ESAIM: PS}, 1:339--355, 1997.

\bibitem[Kac56]{kac1956}
M.~Kac.
\newblock Foundations of kinetic theory.
\newblock In {\em Proceedings of the Third Berkeley Symposium on Mathematical
  Statistics and Probability, Volume 3: Contributions to Astronomy and
  Physics}, pages 171--197, Berkeley, Calif., 1956. University of California
  Press.

\bibitem[KM00]{kona:mamm:00}
V.~Konakov and E.~Mammen.
\newblock Local limit theorems for transition densities of {M}arkov chains
  converging to diffusions.
\newblock {\em Prob. Th. Rel. Fields}, 117:551--587, 2000.

\bibitem[Kry99]{Krylov1999}
N.~V. Krylov.
\newblock {\em On Kolmogorov's equations for finite dimensional diffusions},
  pages 1--63.
\newblock Springer Berlin Heidelberg, Berlin, Heidelberg, 1999.

\bibitem[Lac18]{la:18}
D.~Lacker.
\newblock On a strong form of propagation of chaos for mckean-vlasov equations.
\newblock {\em Electron. Commun. Probab.}, 23:11 pp., 2018.

\bibitem[Lio14]{lecture:lions:college}
P.-L. Lions.
\newblock {Cours au coll{\`e}ge de France}.
\newblock
  \url{http://www.college-de-france.fr/site/pierre-louis-lions/seminar-2014-11-14-11h15.htm},
  2014.

\bibitem[LM16]{Li:min:2}
J.~Li and H.~Min.
\newblock Weak solutions of mean-field stochastic differential equations and
  application to zero-sum stochastic differential games.
\newblock {\em SIAM Journal on Control and Optimization}, 54(3):1826--1858,
  2016.

\bibitem[McK66]{McKean:1966}
H.~P. McKean.
\newblock A class of markov processes associated with nonlinear parabolic
  equations.
\newblock {\em Proceedings of the National Academy of Sciences of the United
  States of America}, 56(6):1907--1911, 12 1966.

\bibitem[McK67]{mckean1967propagation}
H.~P. McKean.
\newblock Propagation of chaos for a class of non-linear parabolic equations.
\newblock {\em Stochastic Differential Equations (Lecture Series in
  Differential Equations, Session 7, Catholic Univ., 1967)}, pages 41--57,
  1967.

\bibitem[MS67]{mcke:sing:67}
H.~P. McKean and I.~M. Singer.
\newblock Curvature and the eigenvalues of the {L}aplacian.
\newblock {\em J. Differential Geometry}, 1:43--69, 1967.

\bibitem[MV18]{mishura:veretenikov}
Y.~S. Mishura and A.~Y. Veretennikov.
\newblock Existence and uniqueness theorems for solutions of
  {M}c{K}ean--{V}lasov stochastic equations.
\newblock Preprint Arxiv arXiv:1603.02212, 2018.

\bibitem[Oel84]{oelschlager1984}
K.~Oelschlager.
\newblock A martingale approach to the law of large numbers for weakly
  interacting stochastic processes.
\newblock {\em Ann. Probab.}, 12(2):458--479, 05 1984.

\bibitem[RZ18]{2018arXiv180902216R}
M.~{R{\"o}ckner} and X.~{Zhang}.
\newblock {Well-posedness of distribution dependent SDEs with singular drifts}.
\newblock {\em arXiv e-prints}, page arXiv:1809.02216, Sep 2018.

\bibitem[Sch87]{scheutzow_1987}
M.~Scheutzow.
\newblock Uniqueness and non-uniqueness of solutions of vlasov-mckean
  equations.
\newblock {\em Journal of the Australian Mathematical Society. Series A. Pure
  Mathematics and Statistics}, 43(2):246--256, 1987.

\bibitem[ST85]{Shiga1985}
T.~Shiga and H.~Tanaka.
\newblock Central limit theorem for a system of markovian particles with mean
  field interactions.
\newblock {\em Zeitschrift f{\"u}r Wahrscheinlichkeitstheorie und Verwandte
  Gebiete}, 69(3):439--459, Sep 1985.

\bibitem[SV79]{stroock:varadhan}
D.~W. Stroock and S.~R.~S. Varadhan.
\newblock {\em Multidimensional diffusion processes}, volume 233 of {\em
  Grundlehren der Mathematischen Wissenschaften [Fundamental Principles of
  Mathematical Sciences]}.
\newblock Springer-Verlag, Berlin-New York, 1979.

\bibitem[Szn91]{Sznitman}
A.-S. Sznitman.
\newblock Topics in propagation of chaos.
\newblock In Paul-Louis Hennequin, editor, {\em Ecole d'Et{\'e} de
  Probabilit{\'e}s de Saint-Flour XIX --- 1989}, pages 165--251, Berlin,
  Heidelberg, 1991. Springer Berlin Heidelberg.

\bibitem[Tan78]{tanaka1978probabilistic}
H.~Tanaka.
\newblock Probabilistic treatment of the {B}oltzmann equation of {M}axwellian
  molecules.
\newblock {\em Zeitschrift f{\"u}r Wahrscheinlichkeitstheorie und Verwandte
  Gebiete}, 46(1):67--105, 1978.

\bibitem[Ver80]{veretennikov_strong_1980}
A.~Y. Veretennikov.
\newblock Strong solutions and explicit formulas for solutions of stochastic
  integral equations.
\newblock {\em Matematicheski Sbornik. Novaya Seriya}, 111(153)(3):434--452,
  480, 1980.

\end{thebibliography}
\end{document}